\documentclass[oneside,12pt]{scrbook}
\usepackage{thesis}
\usepackage[inner=3cm,outer=2cm,top=3cm,bottom=3cm]{geometry}
\usepackage[doublespacing]{setspace}
\usepackage{graphicx}
\usepackage{bussproofs}
\usepackage{cmll}
\usepackage{amsmath}
\usepackage{amssymb}
\usepackage{quiver}
\usepackage[utf8]{inputenc}
\usepackage{hyperref}

\usepackage{tikz}
\usepackage{tikzit}
\tikzstyle{box}=[fill=white, draw=black, shape=rectangle, minimum width=1cm, minimum height=2cm]
\tikzstyle{big box}=[fill=white, draw=black, shape=rectangle, minimum width=2cm, minimum height=4cm]
\tikzstyle{little box}=[fill=white, draw=black, shape=rectangle, minimum width=1cm, minimum height=1cm]
\tikzstyle{smaller box}=[fill=white, draw=black, shape=rectangle, minimum width=0.5cm, minimum height=0.5cm]
\tikzstyle{black node}=[fill=black, draw=black, shape=circle]
\tikzstyle{white node}=[fill=white, draw=black, shape=circle]
\tikzstyle{big triangle}=[fill=white, draw=black, shape=regular polygon, regular polygon sides=3, shape border rotate=30, minimum size=4cm]
\tikzstyle{triangle}=[fill=white, draw=black, shape=regular polygon, regular polygon sides=3, shape border rotate=30, minimum size=2cm]
\tikzstyle{little triangle}=[fill=white, draw=black, shape=regular polygon, regular polygon sides=3, shape border rotate=30, minimum size=1cm]
\tikzstyle{inv big triangle}=[fill=white, draw=black, shape=regular polygon, regular polygon sides=3, minimum size=4cm, shape border rotate=90]
\tikzstyle{inv little triangle}=[fill=white, draw=black, shape=regular polygon, regular polygon sides=3, shape border rotate=90, minimum size=1cm]
\tikzstyle{circle}=[fill=white, draw=black, shape=circle]

\tikzstyle{dash}=[-, dashed]
\tikzstyle{arrow}=[->]
\tikzstyle{new edge style 0}=[->, dashed]
\tikzstyle{blue edge}=[-, fill=blue]
\tikzstyle{blue arrow}=[->, fill=blue]
\tikzstyle{new edge style 1}=[-, fill=red]
\tikzstyle{new edge style 2}=[-, fill={rgb,255: red,191; green,191; blue,191}]
\tikzstyle{new edge style 3}=[-, fill={rgb,255: red,255; green,128; blue,0}]
\tikzstyle{new edge style 4}=[-, fill=green]
\tikzstyle{new edge style 5}=[-, fill=yellow]

\usepackage{tikz-cd}

\title{Bifibrations of polycategories and classical multiplicative linear logic}
\author{Nicolas Blanco}
\date{March 2023}
\titlehead{A Thesis submitted for the degree of Doctor of Philosophy}
\publishers{School of Computer Science\\University of Birmingham}

\usepackage{scrlayer-scrpage}

\ihead{}
\ohead[]{}
\ifoot{}
\ofoot[]{}

\begin{document}
\EnableBpAbbreviations

\maketitle

\frontmatter
\tableofcontents
%\listoffigures
%\listoftables

\chapter{Acknowledgements/Remerciements}

First, I would like to thank my supervisors Paul Levy and Noam Zeilberger.
I am blessed to have been advised by such wonderful people.
Paul has been incredibly helpful, in particular for all the administrative purposes.
But also by providing great advice during the writing stage and the organisation and practice of the viva.
Noam has been a fantastic advisor.
He let me explored my own research interests while providing guidance and suggesting new areas of study.
During my least productive periods, when the guilt and impostor syndrome where lurking around, he was always supportive and I never felt pressured.
When I was on the contrary prolific, he would carefully read what I wrote and provide detailed feedback.
I am also grateful to Benoit Valiron, my Master advisor for his kindness and his guidance.
The work done under his supervision greatly influenced this thesis.

I would also like to thanks the members of my jury, the chair Mart\'{i}n Escard\'{o}, the internal examiner Uday Reedy and the external examiner Richard Garner.
Thanks to their kindness, the viva has been very friendly and enjoyable.
They also provided precious comments and advice about this thesis and my research in general.

Of course, my PhD studies has not been limited to writing this thesis and defending it.
It was a journey through the wonderful land of category theory and I made a lot of invaluable friendships along the way.
First, the theory (bois + man) in Birmingham: Alex, Anna Laura, Calin, George, Todd and Tom, that made my stay there enjoyable.
Then the Mathematicians in Paris, in particular Giti, Morgan, Roman and Will, that offered me opportunities to talk about research while far away from my official lab.
Along the way I also met a lot of amazing people that I now consider friends, amongst which: Greta, Chad, Federico, R\'{e}mi and Nathan.
And a lot of others that made this journey so fun: Alexander, Anupam, Axel, Aurore, Cipriano, Cole, Erwan, Fabrizio, Gabriel, Gabriele, Ivan, Joe, Matteo, Matteo (they are two of them indeed), Ned, Ralph, Valentin, and so many others.

These past few years have not been all about Mathematics and Computer Science.
 have meet a lot of people that helped made these years unforgettable.
There are the people from my vegan group, the Cheeky Vegans in Paris, that got me to discover new restaurants and to sharpen my organisational skills: Anirudh, Clémence, Kate, Mallory...
The Paris LazyBookworms with whom I could discuss books while drinking golden and blue latte at the Unicorners: B\'{e}reng\`{e}re, Sara, Alice, Batoul, Camille, Ali, Ram, Simal, Vitalia...
My roommates with whom I have shared a place, but also tasty meals and enriching conversations: Cl\'{e}ment, Eric and Carolin.
And all the other people that I am grateful to have met this past few years: Alanna, Caroline, Katherine, Paula...

Bien entendu ces remerciements ne seraient pas complets si je ne mentionnais pas ma famille et mes ami·e·s de longue date.
Ces ami·e·s que j'ai rencontré sur les bancs du collège ou du lycée à Cannes et que je connais maintenant depuis plus de vingt ans : Th\'{e}o, BB, Le Suisse, Jerem', Bout'Bout', Benjourde, Momo, MatMouth, B-Lee, et toustes celleux qui m'ont accompagné toutes ces années durant.
Mes ami·e·s parisien·ne·s qui ont fait de Paname une deuxième maison, ma ville de cœur : Raph', Aurel', Charlot, Marie-Lou, Thomas, le Ptiot et les Belettes, Antoine, Chewbaca, Emilie, Sofiane et tant d'autres.

Et pour finir ma famille qui a toujours été là pour moi.
Ma mamie Colette que j'embrasse très fort ; mes tantes et oncles qui m'ont toujours accueilli avec amour : tatie Maryline et tonton Pierre, tonton Claude, tonton Fred et tonton Alain et tatie Corinne qui me manquent ; mes cousins, cousines avec qui j'ai grandi et qui sont ce que j'ai de plus proche d'une fratrie : Juliette et Carole et leur pitchounes respectifs qui poussent si vite, Romain, Adeline et Scotty.
J'ai aussi une très forte pensée pour mon papa, parti trop tôt, et qui, je le sais, serait fier.
Et bien sûr, ma maman à qui je dédie cette thèse.
Je n'en serais pas là aujourd'hui sans son amour et son soutien inconditionnels.
Elle est pour moi un modèle de force, de volonté et de douceur.
Elle m'a appris la bonté, la tolérance, la curiosité et la persévérance, des valeurs qui m'animent en tant qu'être humain mais aussi en tant que chercheur.

I have the chance of being born in an incredible loving family and of having met fantastic people in my life.
I am very grateful to all of them, those that I mentioned here, and also those that I forgot.

And to the reader that is considering reading this thesis, thank you for your interest.
I hope you will enjoy it.

\chapter{Abstract}

In this thesis, we develop the theory of bifibrations of polycategories.

We start by studying how to express certain categorical structures as universal properties by generalising the shape of morphism.
We call this phenomenon representability and look at different variations, namely the correspondence between representable multicategories and monoidal categories, birepresentable polycategories and $\ast$-autonomous categories, and representable virtual double categories and double categories.

We then move to introduce (bi)fibrations for these structures.
We show that it generalises representability in the sense that these structures are (bi)representable when they are (bi)fibred over the terminal one.
We show how to use this theory to lift models of logic to more refined ones.
In particular, we illustrate it by lifting the compact closed structure of the category of finite dimensional vector spaces and linear maps to the (non-compact) $\ast$-autonomous structure of the category of finite dimensional Banach spaces and contractive maps by passing to their respective polycategories.
We also give an operational reading of this example, where polylinear maps correspond to operations between systems that can act on their inputs and whose outputs can be measured/probed and where norms correspond to properties of the systems that are preserved by the operations.

Finally, we recall the B\'{e}nabou-Grothendieck correspondence linking fibrations to indexed categories.
We show how the B-G construction can be defined as a pullback of virtual double categories and we make use of fibrational properties of vdcs to get properties of this pullback.
Then we provide a polycategorical version of the B-G correspondence.  

% Glossary and acronyms

\mainmatter

\part{Introduction}

\chapter*{Introduction}
\label{ch:intro}

\section{What is a multicategory?}
The notion of multicategory was introduced by Lambek in \cite{Lambek1969}.
As recounted in \cite{Lambek1989}, he drew inspiration from his earlier work in algebra and linguistics.
In \cite{Lambek1958}, he introduced a sequent calculus, adapted from Gentzen's original sequent calculus, that he used to establish a decision procedure for grammaticality of sentences.
In this sequent calculus, there are some type formers for syntactic expressions that behave like logical connectives.
\[ B \vdash A \backslash C \Leftrightarrow A \cdot B \vdash C \Leftrightarrow A \vdash C/B \]

Lambek and Findlay noticed that a similar relationship is true for two-sided ideals $A,B,C$ of a unital ring $R$:
\[ B \subseteq A \backslash C \Leftrightarrow A \cdot B \subseteq C \Leftrightarrow A \subseteq C/B \]
where $A \cdot B := \set{\sum_{i \leq n} a_i b_i}{a_i \in A, b_i \in B}$ is the product ideal, $A \backslash C := \set{r \in R}{\forall a\in A, ar \in C}$ is the left residual quotient and $C/B := \set{r \in R}{\forall b \in B, rb \in C}$ is the right residual quotient.

They also noticed that a similar relation exists for bimodules.
Namely, that for bimodules ${\ }_{R}A_S$, ${\ }_{S}B_T$ and ${\ }_{R}C_T$ over unital rings $R,S,T$ there are natural isomorphisms
\[Hom(B,A\multimap C) \simeq Hom(A \otimes B, C) \simeq Hom(A, C \multimapinv B) \]
where $A \otimes B$ is the $(R,T)$-bimodule $A \otimes_S B$, $A \multimap C$ is $Hom_R(A,C)$ considered as a $(S,T)$-bimodule and $C \multimapinv B$ is $Hom_T(B,C)$ as a $(R,S)$-bimodule.
One way of abstracting these notions is by asking for a monoidal biclosed category, i.e. a category equipped with a tensor product functor admitting some right adjoints.
There the tensor product is given as data while the internal homs $\multimap, \multimapinv$ verify a universal property.
What Lambek realised is that Gentzen's sequent calculus provides a way of reasoning about tensors and internal homs similar to Bourbaki's method of multilinear maps.
For ease, let us focus on $(R,R)$-bimodules for a fixed ring $R$. Given bimodules $A,B,C$ there is a canonical bilinear map $A, B \to A \otimes B$ which induces a natural isomorphism between bilinear maps $A,B \to C$ and linear maps $A\otimes B \to C$.
This can be taken as a definition of the tensor product: $A \otimes B$ has the universal property of linearising bilinear maps.
Moreover, it can be extended by adjoining contexts on the left and right: there is a natural isomorphism between multilinear maps $\Gamma_1, A, B, \Gamma_2 \to C$ and multilinear maps $\Gamma_1, A \otimes B, \Gamma_2 \to C$. 
This principle can be put in parallel with the left rule for conjunction in Gentzen's sequent calculus, an analogy which led Lambek to interpret the sequent calculus in multilinear algebra.
Then, he generalised multialgebra to the abstract notion of multicategory, a mathematical structure akin to a category but where the notion of morphism is generalised to a notion of multimap whose domain is a finite list of objects.

In the general context of multicategories, one can characterise $\otimes, \multimap, \multimapinv$ universally, via the existence of multimaps $A,B \to A \otimes B$, $A, A \multimap B, B$ and $B \multimapinv A, A \to B$ inducing natural isomorphisms
\[Hom(A,B; C) \simeq Hom(B;A\multimap C) \simeq Hom(A \otimes B; C) \simeq Hom(A; C \multimapinv B) \]
An advantage of considering multicategories over monoidal categories is that the tensor product in the former is defined through a universal property rather than being given as data.
Furthermore, any monoidal category induces a multicategory called its underlying multicategory that inherits tensor products from the monoidal structure, and any monoidal biclosed category induces a multicategory with tensor products and internal homs.
However, there are also multicategories that do not possess tensor products of all objects.
So the notion of multicategory subsumes that of monoidal category.

Another use for multicategories is for defining some internal structured objects.
For example, the notion of monoid object internal to a monoidal category can be extended to the setting of multicategories.
A monoid in a multicategory \cM is an object $M$ equipped with a binary map, the multiplication, $M, M \to M$ and a nullary map, the unit $\cdot \to M$ satisfying unitality and associativity.
Interestingly, the category of multicategories is cartesian and its terminal object, called the terminal multicategory, is also the free multicategory containing a monoid.
This means that the data of a monoid internal to a multicategory \cM can equivalently be described as a functor $\one \to \cM$.
In this sense, monoids can be thought of as the generalised elements of a multicategory.
This makes multicategory theory a natural setting to study monoids.
Other structured objects can be represented internally to a multicategory, for example monoid actions, i.e. the data of an internal monoid and an object, the internal module, equipped with an action of the monoid.
The free multicategory containing a monoid action can also be easily described.

%Another interesting feature of monoidal categories is that they can be representated string diagrammatically.
%This is also true for multicategories.
%It is possible to reason string diagrammatically about multimaps in a multicategory by drawing them as triangles where the domain of the multimap are represented as strings connected to the basis of the triangle and its codomain as a string connected to its apex.
%This is illustrated in figure \ref{fig:multicat-multimap}.
%\begin{figure}
%\centering
%\tikzfig{multicat-multimap}
%\label{fig:multicat-multimap}
%\caption{A multimap $f \colon A_1, \dots, A_n \to B$}
%\end{figure}
%The notion of tensor product can also be represented graphically.
%The universal multimap $A, B \to A\otimes B$ is represented by a white node that can be think of as an introduction rule for the tensor
%\tikzfig{multicat-tensor}
%The fact that precomposition my this maps is invertible $Hom(A\otimes B; C) \simeq Hom(A,B;C)$ is then represented by a elimination rule
%\tikzfig{multicat-tensor-inv}
\vskip 0.5cm
In this thesis, I will be interested in different directions in which the notion of multicategory can be extended and how these interact together.

\section{Polycategories: the more outputs the merrier}

First, there is the notion of polycategory introduced by Szabo in \cite{Szabo1975}.
Szabo was a student of Lambek and the question that he got interested in is the following.
Can the correspondence between Gentzen's intuitionistic sequent calculus and multicategories be extended to classical sequent calculus?
In Gentzen's original sequent calculus, the sequents have multiple hypotheses and multiple conclusions.
Gentzen noticed that to get intuitionistic sequent calculus out of the classical one, one just has to restrict the sequents to those with a unique conclusion.
Szabo went backwards and tried to define a model of classical sequent calculus by considering a mathematical structure similar to categories and multicategories but where the morphisms have multiple inputs and multiple outputs.
He called these polycategories.
When considering maps with many inputs and many outputs, one is led to consider what kind of composition is appropriate for the structure considered.
For modeling logic, composition should correspond to the cut rule in logic:

\[\AXC{$\Gamma \vdash \Delta_1, A, \Delta_2$}
\AXC{$\Gamma_1, A, \Gamma_2 \vdash \Delta$}
\BIC{$\Gamma_1, \Gamma, \Gamma_2 \vdash \Delta_1, \Delta, \Delta_2$}
\DP
\]

To reflect that, in a polycategory, the notion of composition is restricted to only be allowed along one object at a time.

Polycategories offer a categorical setting in which to interpret classical calculi.
Tensor products and internal homs can be defined in a similar way as they are in a multicategory, interpreting conjunction and implication in logic.
Furthermore, other operators can be defined: a par product or cotensor product that interpret the disjunction and a dual that interpret the negation.
Both are also defined using universal properties.
A polycategory that has all these universal objects - actually there is a lot of redundancy there, and it would be enough to ask only for the existence of some of them to infer the others - corresponds to a $\ast$-autonomous category as defined by Barr in \cite{Barr1979}.
The relationship between polycategories and $\ast$-autonomous categories is analogous to the one between multicategories and monoidal categories, or more exactly between multicategories and monoidal biclosed categories.
Although, $\ast$-autonomous categories are often taken as the right notion for a categorical semantics of multiplicative linear logic (see \cite{Barr1991,Mellies2009}), I would argue that the polycategorical approach is more natural, just as multicategories offer a more natural setting for interpreting intuitionistic MLL than symmetric monoidal closed categories.

Furthermore, in a polycategory, the universal properties that characterise the different connectives are not only similar, they are all examples of a general notion of universal property.
This is the first contribution of this thesis: a unified notion of universal object in a polycategory of which all the connectives of multiplicative linear logic are instances.
More than that, these universal objects capture exactly the propositions of MLL, in the sense that any universal object can be described using tensor and par products and duals \footnote{Here and forward when talking about tensor and par products, I have generally in mind an $n$-ary version of those, in particular, their units is included}.
This led us to an ``unbiased'' notion of $\ast$-autonomous category.
Here by unbiased, we mean that we have all the connectives defined in a unified way instead of making arbitrary choices for based connectives from which other might be derived.
In the usual notion is biased in two ways, first by having some connectives as primitive ones, e.g. tensors, pars and duals, and also by choosing some specific arities for those, in general the binary connective and its unit, the nullary one.
From a precise analysis of the role of duality/negation, Cockett, Seely and coauthors derived the notion of linearly distributive categories (previously know as weakly distributive categories), see \cite{CockettSeely1997, BluteCockettSeelyTrimble1993}.
These capture the models of classical MLL without negation.
Their analysis is tightly connected to polycategories and they prove that linearly distributive categories correspond to certain polycategories.
In fact, they correspond to polycategories in which certain universal objects exist.
Another contribution of this thesis is a definition of unbiased linearly distributive category.
A linearly distributive category has two monoidal structures $(\otimes, I)$ and $(\parr, \bot)$ to interpret the conjunction and disjunction, together with some natural transformations representing their interaction, called the distributivity laws, all subjected to some coherence of course.
The distributivity laws have the following types $A \otimes (B \parr C) \to (A \otimes B) \parr C$ and $(A \parr B) \otimes C \to A \parr (B \otimes C)$.
In the unbiased case, we have $n$-ary tensor and par products, and the distributivity laws are replaced by an unbiased one with type $\otimes(\Gamma_1, \parr(\Delta_1, A, \Delta_2),\Gamma_2) \to \parr(\Delta_1, \otimes(\Gamma_1,A,\Gamma_2),\Delta_2)$.
This is precisely what is needed to define a polycategory from a (unbiased) linearly distributive category.
In fact, one can relax the monoidal structures by directing their laws and still be able to extract a polycategory from it.
This notion of lax linearly distributive category, where $\otimes$ is oplax monoidal and $\parr$ is lax monoidal is introduced and explored in this thesis.
A correspondence between lax linearly distributive categories and certain polycategories, where the tensor and par products have a weaker universal property, is considered.

As mentioned above another interesting aspect of monoidal categories and multicategories is that it is possible to define internal structured objects.
Linearly distributive categories, $\ast$-autonomous categories and polycategories offer even more latitude for that, since polymaps with multiple outputs can encode new operations, such as co-multiplication.
An important example that we will discuss in more depth in this thesis is the concept of Frobenius monoid that generalises the one in a monoidal category by asking for a monoid structure with respect to the $\otimes$ monoidal product and a comonoid structure with respect to the $\parr$ product.
Frobenius monoids for linearly distributive categories have been studied in \cite{Egger2010}.

Some examples of polycategories, $\ast$-autonomous categories and linearly distributive categories that we will consider arise from linear algebra and functional analysis.
\Vect{} and \FVect, the categories are arbitrary/finite dimensional vector spaces and linear maps are often considered as some prototypical models for linear logic.
\Vect{} is a linearly distributive category while \FVect{} is a $\ast$-autonomous category.
$\otimes$ is given by the usual tensor product of vector spaces, while $(-)^\ast$ is given by the the dual of a vector space.
However, in this model $\parr$ is also interpreted as the tensor product of vector spaces.
$\ast$-autonomous categories where tensor and par products coincide have been largely studied in the literature under the name compact closed categories, a concept introduced in \cite{Kelly1972}.
When interpreting logic the duality in these categories is too strong.
Not only are $\otimes$ and $\parr$ identified, but also if a compact closed category has finite products then these are biproducts, see \cite{Houston2008}.
Closely related categories are \Banc{} and \FBanc, the categories of arbitrary/finite dimensional Banach spaces and contractive maps, i.e., maps that do not expand the norm.
These maps inherit the linearly distributive and $\ast$-autonomous structures of \Vect{} and \FVect.
However, their $\otimes$ and $\parr$ are distinct.
This example is already treated in the seminal paper on $\ast$-autonomous categories \cite{Barr1979}.
In these categories the tensor and par products of two Banach spaces both have as an underlying vector space the tensor product of the underlying vector spaces, however the norms that are assigned to them are different.
The question of what norms can be assigned to the tensor products of Banach spaces have been tackled by Grothendieck during his PhD thesis (see \cite{Grothendieck1954}).
There are two extremal norms amongst all the well-behaved ones that can be put on a tensor product.
These correspond precisely to the ones for $\otimes$ and $\parr$.

It is natural to consider \Vect{} and \FVect{} as multicategories, but it is also possible to consider them as polycategories.
A polylinear map $f \colon A_1,\dots, A_m \to B_1, \dots, B_n$ can be defined as a linear map $A_1 \otimes \dots \otimes A_m \to B_1 \otimes \dots \otimes B_n$.
A more concrete and operational understanding of a polylinear map can be given: $f \colon A_1,\dots, A_m \to B_1, \dots, B_n$ can be thought as an operation that can be fed some states $a_i \in A_i$ (equivalently linear maps $a_i \colon \mathbb{K} \to A_i$ where $\mathbb{K}$ is the base field considered as a vector space) and that can be probed/measured by effects $\varphi_j \in B_j^\ast$ (equivalently linear maps $\varphi_j \colon B_j \to \mathbb{K}$) in a linear way.
So it assigns to any states $a_i$ and effects $\varphi_j$ a scalar $(\varphi_1,\dots,\varphi_n)f(a_1,\dots,a_m) \in \mathbb{K}$ such that $(\dots, \lambda \varphi_j + \varphi_j',\dots)f(\dots) = \lambda (\dots, \varphi_j,\dots,)f(\dots) + (\dots, \varphi_j',\dots)f(\dots)$ and $(\dots)f(\dots, \lambda a_i + a_i',\dots) = \lambda (\dots)f(\dots,a_i\dots) + (\dots)f(\dots,a_i',\dots)$.
We can extend this operational reading of polylinear maps to define contractive polylinear maps.
Under this perspective, a norm can be thought of as a way to specify a subset of vectors: its unit ball.
Then, a contractive linear map is one that preserves the unit ball, in the sense that the image of the unit ball on the domain is included in the unit ball on the codomain.
Similarly, a contractive polylinear map is one such that for any subunital states $a_i$ (i.e. elements of the unit balls of the $(A_i,||-||_{A_i})$), and any subunital effects $\varphi_j$ (i.e. elements of the unit balls of the dual spaces with the dual norm) $(\varphi_1,\dots,\varphi_n)f(a_1,\dots,a_m)$ is a subunit of $\mathbb{K}$, i.e. $|(\varphi_1,\dots,\varphi_n)f(a_1,\dots,a_m)| \leq 1$.
This defines polycategories \Banc{} and \FBanc{} which are precisely the underlying polycategories of the linearly distributive and $\ast$-autonomous categories described above.
This gives an operational reading of the injective and projective norms (the extremal norms on the tensor product mentioned above).
Closely related to this example is the causal structures construction of \cite{KissingerUijlen2017lics}.
In this article, the authors define a construction that takes a compact closed category (with some extra conditions) and produce a $\ast$-autonomous category whose objects are those of the compact closed category together with a subset of its states, the causal states, and whose maps are the one preserving the causal states.
This construction can be extended to a polycategorical one with the same operational reading as the one for Banach spaces.

\section{Functors: refinement and parametrisation}

In addition to considering individual multicategories, we can consider functors between multicategories.
There is a sense in which the theory of representable multicategories can be extended by considering functors.
This is achieved by considering opfibrations $p \colon \cE \to \cB$.
An opfibration of multicategories is a functor such that any list of objects $(A_i)_{i\leq n}$ in \cE and any morphism $f \colon p(A_1),\dots, p(A_n) \to B$ in \cB induces a universal object $\otimes_f(A_1,\dots,A_n)$ called the pushforward of the $A_i$s along $f$.
As the notation suggests, the pushforward can be understood as a parametrised version of the tensor product.
This is illustrated by the fact that if $f\colon p(A_1),\dots, p(A_n) \to \otimes(p(A_1),\dots, p(A_n))$ is the universal multimap of a tensor in \cB, $\otimes_f(A_1,\dots,A_n)$ defines a tensor product in \cE.
In particular, if the unique multimap $p \colon \cE \to \one$ into the terminal multicategory is an opfibration then $\cE$ is representable.
This characterises representable multicategories.
This correspondence between fibrations and tensors in multicategories has been worked out first by Hermida in \cite{Hermida2000,Hermida2004}.

A main contribution of this thesis, as emphasised by the title, is to introduce and study a notion of fibration of polycategories.
A bifibration of polycategories has for any finite lists of objects in \cE $\Pi, \Sigma_1, \Sigma_2$ and any polymap $f \colon p(\Pi) \to p(\Sigma_1), A, p(\Sigma_2)$, a pushforward of $\Pi$ along $f$ in context $\Sigma_1$, $\Sigma_2$.
Dually, for any finite lists of objects in \cE $\Pi_1, \Pi_2, \Sigma$ and any polymap $f \colon p(\Pi_1), A, p(\Pi_2) \to p(\Sigma)$, there is a pullback of $\Sigma$ along $f$ in context $\Pi_1$, $\Pi_2$.
Furthermore, like in the case of multicategories, the pushforwards along unary polymaps, i.e. polymaps with only one output, are parametrised tensor products, while pullbacks along counary polymaps are parametrised par product.
Maybe more surprising is the fact that all pullbacks and pushforwards can be understood as parametrised connectives for MLL.
That is, when considering functors into the terminal polycategory, the pushforwards and pullbacks are generated by $\otimes$, $\parr$ and 
$(-)^\ast$.
So polycategories that are models of MLL correspond exactly to the ones that are bifibred over $\one$.

This can be used to refine models of MLL.
Starting with a functor of polycategories $p \colon \cE \to \cB$ such that \cB{} is a model of MLL, if $p$ is a bifibration, or more generally if it has enough pullbacks and pushforwards, then we can lift the connectives of \cB to make \cE into a model of MLL.
The forgetful functor $\cU \colon \Banc \to \Vect$ (and the finite dimensional versions) offers a good illustration of this.
It is possible to define pullback and pushforward norms.
Under the operation perspective, the pullback norm is the smallest norm that can be put on the object of the domain we are pulling into that makes the polylinear map we pull along contractive.
Dually, the pushforward is the biggest norm that can be put on the object of the codomain we are pushing into that makes the polylinear map contractive.
This is in the spirit of pullback as weakest precondition and pushforward as strongest postcondition.
And it also explains why the tensor and par norms are extremal: they are defined as some pushforward and pullback norms.

\section{Virtual double categories: colouring with categories}

Virtual double categories offer yet another direction to extend multicategories.
They correspond to multicategories coloured by a category.
By coloured here, we mean that virtual double categories extend multicategories in a similar way that categories do for monoids.
A monoid can be understood as a one-object category.
The monoid is encoded in the hom-set of the object with multiplication given by composition.
A small categories then generalise monoids by letting considering sets of objects instead of a singular one.
We say that a category is a monoid coloured by a set.
Similarly part of the data of a virtual double category is a category called its vertical category.
For the trivial case where it is the terminal category, the rest of the data form a multicategory.
So a multicategory can be encoded in a one-object-one-vertical-morphism virtual double category. 
More precisely, it is well studied that monoidal categories can be horizontally categorified by bicategories, a notion introduced by B\'{e}nabou see \cite{Benabou1967}.
That is, a one-object bicategory is equivalent to a monoidal category.
One could want to perform a similar categorification for multicategories.
This will give a structure with objects, morphisms between them and 2-morphisms $\alpha \colon p_1,\dots, p_n \Rightarrow q$ of the following shape:
% https://q.uiver.app/?q=WzAsNSxbMCwxLCJBXzAiXSxbMSwwLCJBXzEiXSxbMywwLCJBX3tuLTF9Il0sWzQsMSwiQV9uIl0sWzIsMCwiXFxkb3RzIl0sWzAsMSwicF8xIiwxXSxbMiwzLCJwX24iLDFdLFswLDMsInEiLDFdLFs0LDcsIlxcYWxwaGEiLDEseyJzaG9ydGVuIjp7InRhcmdldCI6MjB9fV1d
\[\begin{tikzcd}
	& {A_1} & \dots & {A_{n-1}} \\
	{A_0} &&&& {A_n}
	\arrow["{p_1}"{description}, from=2-1, to=1-2]
	\arrow["{p_n}"{description}, from=1-4, to=2-5]
	\arrow[""{name=0, anchor=center, inner sep=0}, "q"{description}, from=2-1, to=2-5]
	\arrow["\alpha"{description}, shorten >=3pt, Rightarrow, from=1-3, to=0]
\end{tikzcd}\]
This notion can be generalised further by asking for another type of morphisms that we will call vertical, while the one above will be called horizontal, in such a way that objects and vertical morphisms form a category and that 2-morphisms have the following shape:
% https://q.uiver.app/?q=WzAsNyxbMCwwLCJBXzAiXSxbMSwwLCJBXzEiXSxbMywwLCJBX3tuLTF9Il0sWzQsMCwiQV9uIl0sWzIsMCwiXFxkb3RzIl0sWzAsMSwiQl8wIl0sWzQsMSwiQl8xIl0sWzAsMSwicF8xIiwxXSxbMiwzLCJwX24iLDFdLFs1LDYsInEiLDFdLFswLDUsImZfMCIsMV0sWzMsNiwiZl8xIiwxXSxbNCw5LCJcXGFscGhhIiwxLHsic2hvcnRlbiI6eyJ0YXJnZXQiOjIwfX1dXQ==
\[\begin{tikzcd}
	{A_0} & {A_1} & \dots & {A_{n-1}} & {A_n} \\
	{B_0} &&&& {B_1}
	\arrow["{p_1}"{description}, from=1-1, to=1-2]
	\arrow["{p_n}"{description}, from=1-4, to=1-5]
	\arrow[""{name=0, anchor=center, inner sep=0}, "q"{description}, from=2-1, to=2-5]
	\arrow["{f_0}"{description}, from=1-1, to=2-1]
	\arrow["{f_1}"{description}, from=1-5, to=2-5]
	\arrow["\alpha"{description}, shorten >=3pt, Rightarrow, from=1-3, to=0]
\end{tikzcd}\]
where $f_0$ and $f_1$ are vertical morphisms.
So instead of coloring a multicategory by a set of objects, we do it by a category of objects and vertical morphisms.
Then a multicategory is a one-object-one-vertical-morphism virtual double categories.
Similarly to how multicategories offer a great environment to internalise algebraic structures, virtual double categories offer one to internalise categorical structures.
For example, monads can be defined internally to any virtual double category by horizontally categorifying monoids internal to a multicategory.
Similarly, the tensor product gets categorified into a notion of composition of horizontal morphisms.
A virtual double category with all composites (of horizontal morphisms) is a double category.
Virtual double categories have been considered a lot, particularly in the context of formal category theory, see for example \cite{Koudenburg2019}.
It is also a good framework to define generalisation of multicategories, see \cite{CruttwellShulman2009}.
In this paper, generalised multicategories are considered, where the domain of a morphism can take other shapes than a single object or a list, e.g. a tree, a (finite) set, a multiset...
Although this subject is out of the scope of this thesis we might mention it sometimes.
Generalised multicategories are a framework flexible enough to encompass examples as different as topological spaces and virtual double categories.
The string diagram calculus for multicategories can be extended straightforwardly to one for virtual double categories by adding colors. It has been used for example in \cite{Myers2016,Myers2020}.

In this thesis, I also define a notion of (op)fibration of virtual double categories.
It has a notion of pushforward of horizontal morphisms along a 2-cell that is a parametrised version of composition of horizontal morphisms.
I use it when studying the B\'{e}nabou-Grothendieck construction.
It is often said in the literature that the Grothendieck construction can be defined as a pullback:
% https://q.uiver.app/?q=WzAsNCxbMCwyLCJcXGNCIl0sWzIsMiwiXFxDYXQiXSxbMiwwLCJcXENhdF9cXGFzdCJdLFswLDAsIlxcaW50IEYiXSxbMCwxLCJGIiwxXSxbMiwxLCJcXGNVIiwxXSxbMywyXSxbMywwXSxbMywxLCIiLDEseyJzdHlsZSI6eyJuYW1lIjoiY29ybmVyIn19XV0=
\[\begin{tikzcd}
	{\int F} && {\Cat_\ast} \\
	\\
	\cB && \Cat
	\arrow["F"{description}, from=3-1, to=3-3]
	\arrow["\cU"{description}, from=1-3, to=3-3]
	\arrow[from=1-1, to=1-3]
	\arrow[from=1-1, to=3-1]
	\arrow["\lrcorner"{anchor=center, pos=0.125}, draw=none, from=1-1, to=3-3]
\end{tikzcd}\]
where $\Cat_\ast$ is the 2-category of pointed categories.
It is not always clear however in which ambient setting this pullback takes place.
One would want it to be in a category of bicategories and pseudofunctors.
But, this will only define $\int F$ as a bicategory equipped with a pseudofunctor into \cB.
In this thesis, I answer this question by defining it in the context of virtual double categories.
I prove that for functors of virtual double categories, the pullback always exists:
% https://q.uiver.app/?q=WzAsNCxbMCwwLCJcXG1hdGhiYntBfVxcdGltZXNfe1xcbWF0aGJie0N9fVxcbWF0aGJie0J9Il0sWzAsMiwiXFxtYXRoYmJ7QX0iXSxbMiwyLCJcXG1hdGhiYntDfSJdLFsyLDAsIlxcbWF0aGJie0J9Il0sWzAsMV0sWzEsMl0sWzMsMl0sWzAsM10sWzAsMiwiIiwxLHsic3R5bGUiOnsibmFtZSI6ImNvcm5lciJ9fV1d
\[\begin{tikzcd}
	{\mathbb{A}\times_{\mathbb{C}}\mathbb{B}} && {\mathbb{B}} \\
	\\
	{\mathbb{A}} && {\mathbb{C}}
	\arrow[from=1-1, to=3-1]
	\arrow[from=3-1, to=3-3]
	\arrow[from=1-3, to=3-3]
	\arrow[from=1-1, to=1-3]
	\arrow["\lrcorner"{anchor=center, pos=0.125}, draw=none, from=1-1, to=3-3]
\end{tikzcd}\]
Then, I refine it.
The first question is what conditions do we need for $\mathbb{A}\times_{\mathbb{C}}\mathbb{B}$ to have composition of horizontal morphisms.
We want to encode the categories and 2-categories considered in the B\'{e}nabou-Grothendieck construction as virtual double categories with trivial vertical morphisms and the morphisms of the category/2-category as horizontal ones.
So the first question is what conditions do we need for $\mathbb{A}\times_{\mathbb{C}}\mathbb{B}$ to have composition of horizontal morphisms.
This is where opfibrations enter the picture: if $\mathbb{A}, \mathbb{B},\mathbb{C}$ have composite of horizontal morphisms and $\mathbb{B} \to \mathbb{C}$ is an opfibration, then $\mathbb{A}\times_{\mathbb{C}}\mathbb{B}$ has composite of horizontal morphisms (and also $\mathbb{A}\times_{\mathbb{C}}\mathbb{B} \to \mathbb{A}$ is strict, i.e. preserves composite on the nose).
Then, I give further conditions for ensuring that for $\mathbb{A}$ a category and $\mathbb{B},\mathbb{C}$ bicategories, with $\mathbb{A} \to \mathbb{C}$ a lax normal/pseudofunctor, $\mathbb{A}\times_{\mathbb{C}}\mathbb{B} \to \mathbb{A}$ is a functor of categories.
I then prove that it is the case when pulling $\Dist_\ast \to \Dist$ or $\Cat_\ast \to \Cat$ along a lax normal/pseudo functor to give the B\'{e}nabou-Grothendieck construction.

\section{Outline}

This thesis is split into three parts.

The first part is about representability, i.e. the idea that certain concepts that are usually expressed as structures on a category can be recovered through a universal property when considering more general morphisms.
This is illustrated on three examples, each treated in a different chapter.

In the first chapter, I consider multicategories.
I recall the notion of monoidal category, both the usual biased definition and an unbiased one that makes the connection to multicategories easier to establish.
Then, I define multicategories and show how any monoidal category induces an underlying multicategory. I prove that this induces a 2-fully-faithful functor from a bicategory of monoidal categories to a bicategory of multicategories.
After that, representable multicategories are considered, i.e. multicategories that have tensor products.
A string diagram for representable multicategories is introduced, and a 2-equivalence between a 2-category of representable multicategories and a 2-category of monoidal categories is established.
Finally closed multicategories are considered.
None of the material in chapter 1 is new.

The second chapter generalises this to polycategories.
I start by revisiting $\ast$-autonomous categories.
The usual definitions are first stated.
Then, I introduce a notion of unbiased $\ast$-autonomous category.
To do so, I define unbiased linearly distributive categories, and also their lax/weak versions.
This is new material.
Then, I define polycategories and show that any lax unbiased linearly distributive category induces an underlying polycategory.
This generalises the previously established fact that any linearly distributive category has an underlying polycategory.
Then, I characterise polycategories that are the underlying polycategory of a lax unbiased linearly distributive category by introducing weak tensor and par products, that have a weaker version of the universal properties of tensor and par products in a polycategory.
Then, I show how this characterisation lifts to the usual ones between two-tensor polycategories and linearly distributive categories and two-tensor polycategories with duals and $\ast$-autonomous categories.
After that, I introduce a general notion of universal object subsuming tensors, pars and duals.
I prove that birepresentable polycategories, i.e. the ones with all universal objects, are equivalent to two-tensor polycategories with duals, hence to $\ast$-autonomous categories.
This contribution was already in our paper \cite{BlancoZeilberger2020}.
Finally, I look at several examples of birepresentable polycategories, focusing in particular on Banach spaces and contractive polylinear maps exploiting the operational perspective on polycategories mentioned above.
String diagrams are considered for polycategories with weak products, with strong products, with duals and with universal objects.
It is shown how they relate to proof nets in linear logic.

The last chapter of the first part is about virtual double categories.
In this context, representability corresponds to the existence of composition for horizontal maps.
It is established that representable virtual double categories are double categories and their relation to bicategories is discussed.

The second part of this thesis is about fibrations.

The first chapter of this part considers bifibrations of polycategories.
This is the main contribution of my PhD and is explored in our paper \cite{BlancoZeilberger2020}.
An important result is that polycategories bifibred over $\one$ correspond to birepresentable polycategories, i.e. models of classical MLL.
This result has been recently generalised to full LL by Shulman in \cite{Shulman2021} via the notion of bifibred LNL-polycategory.

The second chapter is about fibrations of virtual double categories.
Two applications are considered.
First how virtual double categories  fibred over $\one$ corresponds to representable virtual double categories, i.e. double categories.
Then, how pulling back a fibration of representable virtual double categories lets us define a representable structure on the pullback.
This is used in the next part on B\'{e}nabou-Grothendieck correspondences.

The last part considers the B\'{e}nabou-Grothendieck correspondences that makes formal the idea that fibred structures correspond to indexed structures.
In the first chapter, I recall the categorical B\'{e}nabou-Grothendieck correspondences that establishes 2-equivalences between:
\begin{itemize}
\item functors $\cE \to \cB$ and lax normal functors $\cB \to \Dist$
\item fibrations $\cE \to \cB$ and pseudofunctors $\cB \to \Cat^\op$
\item opfibrations $\cE \to \cB$ and pseudofunctors $\cB \to \Cat$
\item bifibrations $\cE \to \cB$ and pseudofunctors $\cB \to \Adj$
\end{itemize}
I show that the reconstruction of the functor/fibration from the indexed data can be performed via a pullback in the category of virtual double categories by exploiting the fact that the forgetful functor $\Dist_\ast \to \Dist$ from pointed distributors (similarly for pointed categories) is a fibration of virtual double categories.
Then, in a last chapter, I consider a polycategorical version of the B\'{e}nabou-Grothendieck correspondences.
I show how it builds a bridge between different approaches for modeling classical MLL: $\ast$-autonomous categories, birepresentable polycategories, polycategories bifibred over $\one$, and Frobenius pseudomonoid in the polycategory \MVar{} of multivariable adjonctions.

\part{Representability}

\chapter{Multicategories and monoidal~closed~categories}

Monoidal closed categories are used to model intuitionistic multiplicative linear logic, linear functional programming languages and more generally any resource theory that need to accommodate higher-order processes - such as intuitionistic implication whose terms are proofs or functional programming languages where programs can take other programs in input.
This correspond to the ``closed'' part.
The monoidal part allows one to model the formation of complex systems from simpler ones.
In particular, any monoidal category comes with a monoidal product that can be used to model multivariable morphisms by considering morphisms whose inputs are joined together by the monoidal product.
An example would be the cartesian product in categories like sets and functions, topological spaces and continuous maps, vector spaces and linear maps.
Although, in the latter case, the cartesian product is the direct sum and the notion of multivariable morphism that we get that way is a multivariable function that is linear in all of its variables at once.
One might want to replace the direct sum there by the tensor product of vector spaces, so that a multivariable morphism is a multilinear map, i.e. a function linear in each of its variable independently.
This is a typical example of monoidal closed category.
A lot of those arise in practice by having in mind a notion of multivariable morphism and finding the right monoidal product to model it.
For these examples the process feels backwards.
For in this perspective, one will have to define first what is the tensor of vector spaces (constructing it explicitly, e.g. by taking a quotient of the free vector space on a product of vector spaces) to have access to multilinear maps.
In algebraic lectures, it is usually done in the other way, where a multilinear map is defined to be a multivariable function linear in each of its inputs, and the tensor product is given by its universal property of linearising multilinear maps.
Multicategories abstract this idea.
A multicategory is akin to a category but where multimorphisms have any number of inputs.
In a multicategory, one can introduce a notion of tensor product in a similar way that it is done for  vector spaces.
Then a multicategory with tensor products - called a representable multicategory - is a monoidal category, and any monoidal category arises this way.
So the two notions of representable multicategories and of monoidal categories coincide, although the perspective is different, in the former the multivariable morphisms are given as data and the tensor product is derived while in the latter it is the other way around.
Furthermore, not every multicategory is representable, so it is more general.
In particular, one can define a notion of closed multicategory without the need to talk about tensor product.

In this section we will recall the notion of monoidal closed category.
We will then carry on to define multicategories, (bi)representable ones and the 2-equivalence between these and monoidal (closed) categories.

This section mainly serves as providing background material for the rest of the thesis.
Nothing in it is new.
A treatment of monoidal categories can for example be found in Mac Lane's Categories for the Working Mathematician \cite{MacLane1998}.
A great account of multicategories is given in Leinster's Higher Operads, Higher Categories \cite{Leinster2004}.
Most of the material from this section is adapted from it.
The connection between closed categories and closed multicategories can be found in Manzyuk's paper \cite{Manzyuk2009}.

\section{Monoidal closed categories}

A monoidal category is a category where objects and morphisms can be composed together in parallel, using a so-called monoidal product.
There are multiple ways of defining it.
One can give the nullary and binary monoidal product and how they interact to form monoidal product of other arities or one can give all the monoidal products as data with some coherence laws.
These are usually referred to as the biased and unbiased definition.
Note that in the biased definition, one does not have to commit to defining nullary and binary monoidal products but could give $n$-ary ones for all n in a set $\Sigma$ containing at least 0 and one arity greater or equal to 2, but we will stick to the nullary+binary case.
The reader interested in the case for an arbitrary $\Sigma$ and the equivalence of the notions will find all the details in \cite{Leinster2004}.

\begin{definition}
A \emph{biased monoidal category} $(\cC,\otimes,I,\alpha, \lambda, \rho)$ is the data of:
\begin{itemize}
\item a category \cC
\item a functor $- \otimes - \colon \cC \times \cC \to \cC$
\item a functor $I \colon \one \to \cC$, that is an object of \cC, we will write $I$ for the object $I(\ast)$
\item a natural isomorphism $\alpha \colon (- \otimes -) \otimes - \Rightarrow - \otimes (- \otimes -)$ called the associator
\item natural isomorphisms $\lambda \colon I \otimes - \Rightarrow -$ and $\rho \colon - \otimes I \Rightarrow -$ called the left and right unitors respectively
\end{itemize}
subject to the coherence laws:
\begin{itemize}
\item the following pentagon diagram commutes for any $A,B,C,D \in \Ob{\cC}$:
% https://q.uiver.app/?q=WzAsNSxbMCwwLCIoKEEgXFxvdGltZXMgQikgXFxvdGltZXMgQykgXFxvdGltZXMgRCJdLFsyLDAsIihBIFxcb3RpbWVzIEIpIFxcb3RpbWVzIChDIFxcb3RpbWVzIEQpIl0sWzMsMSwiQSBcXG90aW1lcyAoQiBcXG90aW1lcyAoQyBcXG90aW1lcyBEKSkiXSxbMCwyLCIoQSBcXG90aW1lcyAoQiBcXG90aW1lcyBDKSkgXFxvdGltZXMgRCJdLFsyLDIsIkEgXFxvdGltZXMgKChCIFxcb3RpbWVzIEMpIFxcb3RpbWVzIEQpIl0sWzAsMSwiXFxhbHBoYV97QSBcXG90aW1lcyBCLCBDLCBEfSIsMV0sWzEsMiwiXFxhbHBoYV97QSxCLCBDXFxvdGltZXMgRH0iLDFdLFswLDMsIlxcYWxwaGFfe0EsQixDfVxcb3RpbWVzIEQiLDFdLFszLDQsIlxcYWxwaGFfe0EsIEIgXFxvdGltZXMgQywgRH0iLDFdLFs0LDIsIkEgXFxvdGltZXMgXFxhbHBoYV97QixDLER9IiwxXV0=
\[\begin{tikzcd}[ampersand replacement=\&]
	{((A \otimes B) \otimes C) \otimes D} \&\& {(A \otimes B) \otimes (C \otimes D)} \\
	\&\&\& {A \otimes (B \otimes (C \otimes D))} \\
	{(A \otimes (B \otimes C)) \otimes D} \&\& {A \otimes ((B \otimes C) \otimes D)}
	\arrow["{\alpha_{A \otimes B, C, D}}"{description}, from=1-1, to=1-3]
	\arrow["{\alpha_{A,B, C\otimes D}}"{description}, from=1-3, to=2-4]
	\arrow["{\alpha_{A,B,C}\otimes D}"{description}, from=1-1, to=3-1]
	\arrow["{\alpha_{A, B \otimes C, D}}"{description}, from=3-1, to=3-3]
	\arrow["{A \otimes \alpha_{B,C,D}}"{description}, from=3-3, to=2-4]
\end{tikzcd}\]
\item the following triangle diagram commutes for any $A,B$:
% https://q.uiver.app/?q=WzAsMyxbMCwwLCIoQSBcXG90aW1lcyBJKSBcXG90aW1lcyBCIl0sWzIsMCwiQSBcXG90aW1lcyAoSSBcXG90aW1lcyBCKSJdLFsyLDIsIkEgXFxvdGltZXMgQiJdLFswLDEsIlxcYWxwaGFfe0EsSSxCfSIsMV0sWzEsMiwiQSBcXG90aW1lcyBcXGxhbWJkYV9CIiwxXSxbMCwyLCJcXHJob19BIFxcb3RpbWVzIEIiLDFdXQ==
\[\begin{tikzcd}[ampersand replacement=\&]
	{(A \otimes I) \otimes B} \&\& {A \otimes (I \otimes B)} \\
	\\
	\&\& {A \otimes B}
	\arrow["{\alpha_{A,I,B}}"{description}, from=1-1, to=1-3]
	\arrow["{A \otimes \lambda_B}"{description}, from=1-3, to=3-3]
	\arrow["{\rho_A \otimes B}"{description}, from=1-1, to=3-3]
\end{tikzcd}\]
\end{itemize}
\end{definition}

\begin{figure}
% https://q.uiver.app/?q=WzAsNCxbMSwwLCJcXGJpZ290aW1lc19wKFxcYmlnb3RpbWVzX3tuX2l9KFxcYmlnb3RpbWVzX3ttX3tpLGp9fShhX3tpLGosa30pX3sxIFxcbGVxIGsgXFxsZXEgbV97aSxqfX0pX3sxIFxcbGVxIGogXFxsZXEgbl9pfSlfezEgXFxsZXEgaSBcXGxlcSBwfSJdLFswLDEsIlxcYmlnb3RpbWVzX3tcXHN1bV9pIG5faX0oXFxiaWdvdGltZXNfe21fe2ksan19KGFfe2ksaixrfSlfezEgXFxsZXEgayBcXGxlcSBtX3tpLGp9fSlfezEgXFxsZXEgaSBcXGxlcSBwLDEgXFxsZXEgalxcbGVxIG5faX0iXSxbMSwyLCJcXGJpZ290aW1lc197XFxzdW1faSBcXHN1bV97an0gbV97aSxqfX0oYV97aSxqLGt9KV97MSBcXGxlcSBpIFxcbGVxIHAsIDEgXFxsZXEgaiBcXGxlcSBuX2ksIDFcXGxlcSBrIFxcbGVxIG1fe2ksan19Il0sWzIsMSwiXFxiaWdvdGltZXNfcChcXGJpZ290aW1lc197XFxzdW1fan0oYV97aSxqLGt9KV97MSBcXGxlcSBqIFxcbGVxIG5faSwxIFxcbGVxIGsgXFxsZXEgbV97aSxqfX0pX3sxIFxcbGVxIGkgXFxsZXEgcH0iXSxbMCwxLCJcXGFscGhhX3soKFxcYmlnb3RpbWVzX3ttX3tpLGp9fSAoYV97aSxqLGt9KV9rKV9qKV9pfSIsMV0sWzEsMiwiXFxhbHBoYV97KChhX3tpLGosa30pX2spX3soaSxqKX19IiwxXSxbMCwzLCJcXG90aW1lc19wKFxcYWxwaGFfeygoYV97aSxqLGt9KV9rKV9qfSlfaSIsMV0sWzMsMiwiXFxhbHBoYV97KChhX3tpLGosa30pX3soaixrKX0pX2l9IiwxXV0=
\[\scalebox{0.75}{
\begin{tikzcd}[ampersand replacement=\&]
	\& {\bigotimes_p(\bigotimes_{n_i}(\bigotimes_{m_{i,j}}(a_{i,j,k})_{1 \leq k \leq m_{i,j}})_{1 \leq j \leq n_i})_{1 \leq i \leq p}} \\
	{\bigotimes_{\sum_i n_i}(\bigotimes_{m_{i,j}}(a_{i,j,k})_{1 \leq k \leq m_{i,j}})_{1 \leq i \leq p,1 \leq j\leq n_i}} \&\& {\bigotimes_p(\bigotimes_{\sum_j}(a_{i,j,k})_{1 \leq j \leq n_i,1 \leq k \leq m_{i,j}})_{1 \leq i \leq p}} \\
	\& {\bigotimes_{\sum_i \sum_{j} m_{i,j}}(a_{i,j,k})_{1 \leq i \leq p, 1 \leq j \leq n_i, 1\leq k \leq m_{i,j}}}
	\arrow["{\alpha_{((\bigotimes_{m_{i,j}} (a_{i,j,k})_k)_j)_i}}"{description}, from=1-2, to=2-1]
	\arrow["{\alpha_{((a_{i,j,k})_k)_{(i,j)}}}"{description}, from=2-1, to=3-2]
	\arrow["{\otimes_p(\alpha_{((a_{i,j,k})_k)_j})_i}"{description}, from=1-2, to=2-3]
	\arrow["{\alpha_{((a_{i,j,k})_{(j,k)})_i}}"{description}, from=2-3, to=3-2]
\end{tikzcd}}\]
% https://q.uiver.app/?q=WzAsNixbMCwwLCJcXG90aW1lc19uKGFfaSlfezEgXFxsZXEgaSBcXGxlcSBufSJdLFsyLDAsIlxcb3RpbWVzX24oXFxvdGltZXNfMShhX2kpKV97MSBcXGxlcSBpIFxcbGVxIG59Il0sWzIsMSwiXFxvdGltZXNfbihhX2kpX3sxIFxcbGVxIGkgXFxsZXEgbn0iXSxbNCwwLCJcXG90aW1lc19uKGFfaSlfezEgXFxsZXEgaSBcXGxlcSBufSJdLFs2LDAsIlxcb3RpbWVzXzEoXFxvdGltZXNfbihhX2kpX3sxIFxcbGVxIGkgXFxsZXEgbn0pIl0sWzYsMSwiXFxvdGltZXNfbihhX2kpX3sxIFxcbGVxIGkgXFxsZXEgbn0iXSxbMCwxLCJcXG90aW1lc19uKFxcaW90YV97YV9pfSlfezEgXFxsZXEgaSBcXGxlcSBufSIsMV0sWzEsMiwiXFxhbHBoYV97KChhXzEpKV97MSBcXGxlcSBpIFxcbGVxIG59fSIsMV0sWzAsMiwiIiwxLHsibGV2ZWwiOjIsInN0eWxlIjp7ImhlYWQiOnsibmFtZSI6Im5vbmUifX19XSxbMyw0LCJcXGlvdGFfe1xcb3RpbWVzX24oYV9pKV97MSBcXGxlcSBpIFxcbGVzIG59fSIsMV0sWzQsNSwiXFxhbHBoYV97KGFfaSlfezEgXFxsZXEgaSBcXGxlcSBufX0iLDFdLFszLDUsIiIsMSx7ImxldmVsIjoyLCJzdHlsZSI6eyJoZWFkIjp7Im5hbWUiOiJub25lIn19fV1d

\[\scalebox{0.75}{\begin{tikzcd}[ampersand replacement=\&]
	{\otimes_n(a_i)_{1 \leq i \leq n}} \&\& {\otimes_n(\otimes_1(a_i))_{1 \leq i \leq n}} \&\& {\otimes_n(a_i)_{1 \leq i \leq n}} \&\& {\otimes_1(\otimes_n(a_i)_{1 \leq i \leq n})} \\
	\&\& {\otimes_n(a_i)_{1 \leq i \leq n}} \&\&\&\& {\otimes_n(a_i)_{1 \leq i \leq n}}
	\arrow["{\otimes_n(\iota_{a_i})_{1 \leq i \leq n}}"{description}, from=1-1, to=1-3]
	\arrow["{\alpha_{((a_i))_{1 \leq i \leq n}}}"{description}, from=1-3, to=2-3]
	\arrow[Rightarrow, no head, from=1-1, to=2-3]
	\arrow["{\iota_{\otimes_n(a_i)_{1 \leq i \leq n}}}"{description}, from=1-5, to=1-7]
	\arrow["{\alpha_{((a_i)_{1 \leq i \leq n})}}"{description}, from=1-7, to=2-7]
	\arrow[Rightarrow, no head, from=1-5, to=2-7]
\end{tikzcd}}\]
\caption{Coherence law for an unbiased monoidal category}
\label{fig:unbiased-monoidal-coherence}
\end{figure}
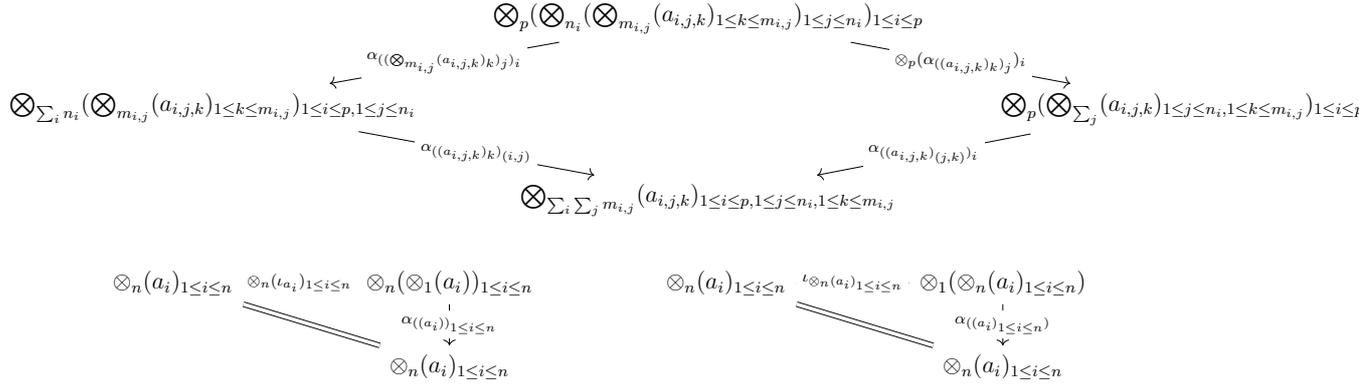

\begin{definition}
An \emph{unbiased monoidal category} $(\cC,(\bigotimes_n)_{n\in \N}, (\alpha^{n,(m_k)_{ 1 \leq k \leq n}})_{n \in \N}, \iota)$ is the data of:
\begin{itemize}
\item a category \cC
\item for each arity $n$, a functor $\bigotimes_n \colon \cC \times \dots \times \cC \to \cC$ from $n$ copies of \cC
\item for each pair of arity $m,n$, a natural isomorphism: \[\alpha^{n,(m_k)_{1 \leq k \leq n}} \colon \bigotimes_n(\bigotimes_{m_1}(-), \dots, \bigotimes_{m_n}(-)) \Rightarrow \bigotimes_{\sum_k m_k} (-)\], where we will usually drop the arity and write $\alpha$
\item a natural isomorphism $\iota \colon - \Rightarrow \otimes_1 -$
\end{itemize}
subject to the following coherence laws:
\begin{itemize}
\item % https://q.uiver.app/?q=WzAsNCxbMSwwLCJcXGJpZ290aW1lc19wXFxiaWdvdGltZXNfe25faX1cXGJpZ290aW1lc197bV97aSxqfX1hX3tpLGosa30iXSxbMCwxLCJcXGJpZ290aW1lc197XFxzdW1faSBuX2l9XFxiaWdvdGltZXNfe21fe2ksan19YV97aSxqLGt9Il0sWzEsMiwiXFxiaWdvdGltZXNfe1xcc3VtX2kgXFxzdW1fe2p9IG1fe2ksan19YV97aSxqLGt9Il0sWzIsMSwiXFxiaWdvdGltZXNfcFxcYmlnb3RpbWVzX3tcXHN1bV9qfWFfe2ksaixrfSJdLFswLDEsIlxcYWxwaGEiLDFdLFsxLDIsIlxcYWxwaGEiLDFdLFswLDMsIlxcb3RpbWVzX3BcXGFscGhhIiwxXSxbMywyLCJcXGFscGhhIiwxXV0=
\[\begin{tikzcd}
	& {\bigotimes_p\bigotimes_{n_i}\bigotimes_{m_{i,j}}a_{i,j,k}} \\
	{\bigotimes_{\sum_i n_i}\bigotimes_{m_{i,j}}a_{i,j,k}} && {\bigotimes_p\bigotimes_{\sum_j}a_{i,j,k}} \\
	& {\bigotimes_{\sum_i \sum_{j} m_{i,j}}a_{i,j,k}}
	\arrow["\alpha"{description}, from=1-2, to=2-1]
	\arrow["\alpha"{description}, from=2-1, to=3-2]
	\arrow["{\otimes_p\alpha}"{description}, from=1-2, to=2-3]
	\arrow["\alpha"{description}, from=2-3, to=3-2]
\end{tikzcd}\]
\item
% https://q.uiver.app/?q=WzAsNixbMCwwLCJcXG90aW1lc19uYV9pIl0sWzIsMCwiXFxvdGltZXNfblxcb3RpbWVzXzFhX2kiXSxbMiwxLCJcXG90aW1lc19uYV9pIl0sWzQsMCwiXFxvdGltZXNfbmFfaSJdLFs2LDAsIlxcb3RpbWVzXzFcXG90aW1lc19uYV9pIl0sWzYsMSwiXFxvdGltZXNfbmFfaSJdLFswLDEsIlxcb3RpbWVzX24gXFxpb3RhIiwxXSxbMSwyLCJcXGFscGhhIiwxXSxbMCwyLCIiLDEseyJsZXZlbCI6Miwic3R5bGUiOnsiaGVhZCI6eyJuYW1lIjoibm9uZSJ9fX1dLFszLDQsIlxcaW90YSIsMV0sWzQsNSwiXFxhbHBoYSIsMV0sWzMsNSwiIiwxLHsibGV2ZWwiOjIsInN0eWxlIjp7ImhlYWQiOnsibmFtZSI6Im5vbmUifX19XV0=
\[\begin{tikzcd}
	{\otimes_na_i} && {\otimes_n\otimes_1a_i} && {\otimes_na_i} && {\otimes_1\otimes_na_i} \\
	&& {\otimes_na_i} &&&& {\otimes_na_i}
	\arrow["{\otimes_n \iota}"{description}, from=1-1, to=1-3]
	\arrow["\alpha"{description}, from=1-3, to=2-3]
	\arrow[Rightarrow, no head, from=1-1, to=2-3]
	\arrow["\iota"{description}, from=1-5, to=1-7]
	\arrow["\alpha"{description}, from=1-7, to=2-7]
	\arrow[Rightarrow, no head, from=1-5, to=2-7]
\end{tikzcd}\]
\end{itemize}
see figure \ref{fig:unbiased-monoidal-coherence} for the full diagrams
\end{definition}

\begin{rk}
We take the convention that no copy of \cC correspond to the terminal category $\one$.
In particular, $\otimes_0 \colon \one \to \cC$.
\end{rk}

\begin{definition}
A \defin{strict} biased/unbiased monoidal category is one where the coherence isomorphisms are equalities.
\end{definition}

\begin{prop}
There is a correspondence between biased and unbiased monoidal categories.
\end{prop}

\begin{proof}
The idea is that starting from an unbiased monoidal category $(\cC, \otimes_n, \alpha_n)$ we define:
\begin{itemize}
\item $\otimes := \otimes_2$
\item $I := \otimes_0$
\item $\alpha' \colon (- \otimes -) \otimes - \xrightarrow{\alpha^{2,(2,1)}} \otimes_3(-,-,-) \xrightarrow{(\alpha^{2,(1,2)})^{-1}} - \otimes (- \otimes -)$
\item $\lambda \colon I \otimes - \xrightarrow{\alpha^{2,(0,1)}} -$
\item $\rho \colon - \otimes I \xrightarrow{\alpha^{2,(1,0)}} -$
\end{itemize}

And conversely, starting from a biased monoidal category we define recursively
\begin{itemize}
\item $\otimes_0() := I$
\item $\otimes_1(A) := A$
\item $\otimes_{n+1}(A_i)_{1 \leq i \leq n+1} := \otimes_n(A_i)_{1 \leq i \leq n} \otimes A_{n+1}$
\item $\alpha^{n,(m_i)_{1 \leq i \leq n}}$ is defined by a nested induction, first an induction on $n$
\begin{itemize}
\item $\alpha^{0} \colon \otimes_0() \to \otimes_0()$ is the identity of $\otimes_0() = I$
\item $\alpha^{1,m_1} \colon \otimes_1(\otimes_{m_1}(-)) \to \otimes_{m_1}(-)$ is given by the identity since  by definition $\otimes_1(\otimes_{m_1}(A)) = \otimes_{m_1}(A)$
\item supposing that we have $\alpha^{n,(m_i)_{1 \leq i \leq n}}$, we want to define \[\alpha^{n+1,(m_i)_{1 \leq i \leq n+1}} \colon \otimes_{n+1}(\otimes_{m_i}(A_{i,j})) \to \otimes_{\sum\limits_{i \leq n+1} m_i}(A_{i,j})\]
By definition $\otimes_{n+1}(\otimes_{m_i}(A_{i,j})) = \otimes_n(\otimes_{m_i}(A_{i,j})) \otimes\ \otimes_{m_{n+1}}(A_{n+1,j})$ so we need to define \[\alpha^{n+1,(m_i)_{1 \leq i \leq n+1}} \colon \otimes_n(\otimes_{m_i}(A_{i,j})) \otimes\ \otimes_{m_{n+1}}(A_{n+1,j}) \to \otimes_{\sum\limits_{i \leq n} m_i + m_{n+1}}(A_{i,j})\]
We proceed by induction on $m_{n+1}$:
\begin{itemize}
\item for $m_{n+1} = 0$ we want \[\alpha^{n+1,(m_i)_{1 \leq i \leq n+1}} \colon \otimes_n(\otimes_{m_i}(A_{i,j})) \otimes\ \otimes_0() \to \otimes_{\sum\limits_{i \leq n} m_i + 0}(A_{i,j})\] which by definition gives \[\alpha^{n+1,(m_i)_{1 \leq i \leq n+1}} \colon \otimes_n(\otimes_{m_i}(A_{i,j})) \otimes I \to \otimes_{\sum\limits_{i \leq n} m_i}(A_{i,j})\]
We do so by first using the right unitor and then the induction hypothesis of the first induction $\alpha^{n+1,(m_i)_{1 \leq i \leq n+1}} := \alpha^{n,(m_i)_{1 \leq i \leq n}} \circ \rho$.
\item for $m_{n+1} = 1$ we want \[\alpha^{n+1,(m_i)_{1 \leq i \leq n+1}} \colon \otimes_n(\otimes_{m_i}(A_{i,j})) \otimes\ \otimes_1(A_{n+1,j}) \to \otimes_{\sum\limits_{i \leq n} m_i + 1}(A_{i,j})\] which by definition of $\otimes_1$ and $\otimes_{\sum\limits_{i \leq n} m_i + 1}$ gives \[\alpha^{n+1,(m_i)_{1 \leq i \leq n+1}} \colon \otimes_n(\otimes_{m_i}(A_{i,j})) \otimes A_{n+1,1} \to \otimes_{\sum\limits_{i \leq n} m_i}(A_{i,j}) \otimes A_{n+1,1}\]
So we can just use \[\alpha^{n+1,(m_i)_{1 \leq i \leq n+1}} := \alpha^{n,(m_i)_{1 \leq i \leq n}} \otimes id\]
\item now, suppose that we have $\alpha^{n+1,(m_i)_{1 \leq i \leq n+1}}$ for some $m_n$ and we want to define it for $m_{n+1}$.
So we need \[\alpha^{n+1,(m_i)_{1 \leq i \leq n+1}} \colon \otimes_n(\otimes_{m_i}(A_{i,j})) \otimes\ \otimes_{m_{n+1}+1}(A_{n+1,j}) \to \otimes_{\sum\limits_{i \leq n} m_i + m_{n+1}+1}(A_{i,j})\]
By definition we need a map $\alpha^{n+1,(m_i)_{1 \leq i \leq n+1}}$ of the type \[\otimes_n(\otimes_{m_i}(A_{i,j})) \otimes (\otimes_{m_{n+1}}(A_{n+1,j}) \otimes A_{n+1,m_n+1}) \to \otimes_{\sum\limits_{i \leq n} m_i + m_{n+1}}(A_{i,j}) \otimes A_{n+1,m_n+1}\]
We can first use the inverse of the associator of the biaised monoidal category to change the order of the bracketing and then use our induction hypothesis.
\end{itemize}
\end{itemize}
\end{itemize}

Instead of proving that the coherence laws of biased and unbiased imply one another, we can first prove coherence theorems stating that any biased/unbiased monoidal category is equivalent to a strict one.
See \cite{Leinster2004} for more details.
\end{proof}

\begin{definition}
A \defin{biased left closed monoidal category} $(\cC, \otimes, I, \multimap)$ is a biased monoidal category such that the functor $A \otimes - \colon \cC \to \cC$ has a right adjoint $A \multimap -$ for any $A \in \cC$.

A \defin{biased right closed monoidal category} $(\cC, \otimes, I, \multimapinv)$ is a biased monoidal category such that the functor $- \otimes A \colon \cC \to \cC$ has a right adjoint $- \multimapinv A$ for any $A \in \cC$.

A \defin{biased biclosed monoidal category}, is a biased monoidal category that is both left and right closed.
\end{definition}

\begin{definition}
An \defin{unbiased biclosed monoidal category}, is an unbiased monoidal category, such that the functor $\otimes_{m + 1 + n}((a_i)_{1 \leq i \leq m}, -, (a_j')_{1 \leq j \leq n})$ has a right adjoint $(a_i)_{1 \leq i \leq m} \multimap - \multimapinv (a_j')_{1 \leq j \leq n}$ for any lists of objects $(a_i)_i, (a_j')j$.
\end{definition}

\begin{prop}
Biased and unbiased biclosed monoidal categories correspond.
\end{prop}
\begin{proof}
From unbiased to biased it follows from the definition of biased category obtained from an unbiased one.
For the other way around we use in addition that right adjoints compose.
\end{proof}

\begin{definition}
A \defin{(lax) monoidal functor} between biased monoidal categories $(\cC, \otimes_\cC, I_\cC)$ and $(\cD,\otimes_\cD, I_\cD)$ is a functor $F \colon \cC \to \cD$ with a morphism $F_0 \colon I_\cD \rightarrow F(I_\cC)$ and a natural transformation $F_2 \colon F(-) \otimes_\cD F(--) \Rightarrow F(- \otimes_\cC --)$ subject to coherence laws:
\begin{itemize}
\item % https://q.uiver.app/?q=WzAsNixbMSwwLCIoRihBKSBcXG90aW1lcyBGKEIpKSBcXG90aW1lcyBGKEMpIl0sWzIsMSwiRihBIFxcb3RpbWVzIEIpIFxcb3RpbWVzIEYoQykiXSxbMiwzLCJGKChBIFxcb3RpbWVzIEIpIFxcb3RpbWVzIEMpIl0sWzAsMSwiRihBKSBcXG90aW1lcyAoRihCKSBcXG90aW1lcyBGKEMpKSJdLFswLDMsIkYoQSkgXFxvdGltZXMgRihCIFxcb3RpbWVzIEMpIl0sWzEsNCwiRihBIFxcb3RpbWVzIChCIFxcb3RpbWVzIEMpKSJdLFswLDEsIkZfMihBLEIpIFxcb3RpbWVzIEYoQykiLDFdLFswLDMsIlxcYWxwaGFfe0YoQSksRihCKSxGKEMpfSIsMV0sWzMsNCwiRihBKSBcXG90aW1lcyBGXzIoQixDKSIsMV0sWzEsMiwiRl8yKEEgXFxvdGltZXMgQiwgQykiLDFdLFs0LDUsIkZfMihBLEJcXG90aW1lcyBDKSIsMV0sWzIsNSwiRihcXGFscGhhX3tBLEIsQ30pIiwxXV0=
\[\begin{tikzcd}[ampersand replacement=\&]
	\& {(F(A) \otimes F(B)) \otimes F(C)} \\
	{F(A) \otimes (F(B) \otimes F(C))} \&\& {F(A \otimes B) \otimes F(C)} \\
	\\
	{F(A) \otimes F(B \otimes C)} \&\& {F((A \otimes B) \otimes C)} \\
	\& {F(A \otimes (B \otimes C))}
	\arrow["{F_2(A,B) \otimes F(C)}"{description}, from=1-2, to=2-3]
	\arrow["{\alpha_{F(A),F(B),F(C)}}"{description}, from=1-2, to=2-1]
	\arrow["{F(A) \otimes F_2(B,C)}"{description}, from=2-1, to=4-1]
	\arrow["{F_2(A \otimes B, C)}"{description}, from=2-3, to=4-3]
	\arrow["{F_2(A,B\otimes C)}"{description}, from=4-1, to=5-2]
	\arrow["{F(\alpha_{A,B,C})}"{description}, from=4-3, to=5-2]
\end{tikzcd}\]
\item % https://q.uiver.app/?q=WzAsOCxbMCwwLCJJIFxcb3RpbWVzIEYoQSkiXSxbMiwwLCJGKEkpIFxcb3RpbWVzIEYoQSkiXSxbMiwyLCJGKEkgXFxvdGltZXMgQSkiXSxbMCwyLCJGKEEpIl0sWzQsMCwiRihBKSBcXG90aW1lcyBJIl0sWzYsMCwiRihBKSBcXG90aW1lcyBGKEkpIl0sWzYsMiwiRihBIFxcb3RpbWVzIEkpIl0sWzQsMiwiRihBKSJdLFswLDEsIkZfMCBcXG90aW1lcyBGKEEpIiwxXSxbMSwyLCJGXzIoSSxBKSIsMV0sWzAsMywiXFxsYW1iZGFfe0YoQSl9IiwxXSxbMywyLCJGKFxcbGFtYmRhX0Feey0xfSkiLDFdLFs0LDUsIkYoQSkgXFxvdGltZXMgRl8wIiwxXSxbNSw2LCJGXzIoQSxJKSIsMV0sWzcsNiwiRihcXHJob197QX1eey0xfSkiLDFdLFs0LDcsIlxccmhvX3tGKEEpfSIsMV1d
\[\begin{tikzcd}[ampersand replacement=\&]
	{I \otimes F(A)} \&\& {F(I) \otimes F(A)} \&\& {F(A) \otimes I} \&\& {F(A) \otimes F(I)} \\
	\\
	{F(A)} \&\& {F(I \otimes A)} \&\& {F(A)} \&\& {F(A \otimes I)}
	\arrow["{F_0 \otimes F(A)}"{description}, from=1-1, to=1-3]
	\arrow["{F_2(I,A)}"{description}, from=1-3, to=3-3]
	\arrow["{\lambda_{F(A)}}"{description}, from=1-1, to=3-1]
	\arrow["{F(\lambda_A^{-1})}"{description}, from=3-1, to=3-3]
	\arrow["{F(A) \otimes F_0}"{description}, from=1-5, to=1-7]
	\arrow["{F_2(A,I)}"{description}, from=1-7, to=3-7]
	\arrow["{F(\rho_{A}^{-1})}"{description}, from=3-5, to=3-7]
	\arrow["{\rho_{F(A)}}"{description}, from=1-5, to=3-5]
\end{tikzcd}\]
\end{itemize}

It is called \defin{strong} if the morphism and the natural transformation are isomorphisms and strict if they are identities.
\end{definition}

\begin{definition}
A \defin{(lax) monoidal functor} between unbiased monoidal categories $(\cC, \otimes_n^\cC)$ and $(\cD, \otimes_n^\cD)$ is a functor $F \colon \cC \to \cD$ with natural transformations $F_n \colon \otimes_n^\cD (F(-)) \Rightarrow F(\otimes_n^\cC (-))$ subject to the following coherence laws:
\begin{itemize}
\item % https://q.uiver.app/?q=WzAsNSxbMCwwLCJcXG90aW1lc19uKFxcb3RpbWVzX3ttX2l9KEYoQV97aSxqfSlfe2ogXFxsZXEgbV9pfSlfe2kgXFxsZXEgbn0iXSxbMCwyLCJcXG90aW1lc19uKEYoXFxvdGltZXNfe21faX0gKEFfe2ksan0pKSkiXSxbMiwyLCJGKFxcb3RpbWVzX24oXFxvdGltZXNfe21faX0oQV97aSxqfSkpKSJdLFs0LDEsIkYoXFxvdGltZXNfe1xcc3VtX2kgbV9pfShBX3tpLGp9KSkiXSxbMiwwLCJcXG90aW1lc197XFxzdW1faSBtX2l9IEYoQV97aSxqfSkiXSxbMCwxLCJcXG90aW1lc19uKEZfe21faX0oQV97aSxqfSkpIiwxXSxbMSwyLCJGX24oXFxvdGltZXNfe21faX0oQV97aSxqfSkpIiwxXSxbMiwzLCJGX24oXFxhbHBoYV97KEFfe2ksan0pfSkiLDFdLFswLDQsIlxcYWxwaGFfe0YoQV97aSxqfSl9IiwxXSxbNCwzLCJGX3tcXHN1bV9pIG1faX0oQV97aSxqfSkiLDFdXQ==
\[\begin{tikzcd}[ampersand replacement=\&]
	{\otimes_n(\otimes_{m_i}(F(A_{i,j})_{j \leq m_i})_{i \leq n}} \&\& {\otimes_{\sum_i m_i} F(A_{i,j})} \\
	\&\&\&\& {F(\otimes_{\sum_i m_i}(A_{i,j}))} \\
	{\otimes_n(F(\otimes_{m_i} (A_{i,j})))} \&\& {F(\otimes_n(\otimes_{m_i}(A_{i,j})))}
	\arrow["{\otimes_n(F_{m_i}(A_{i,j}))}"{description}, from=1-1, to=3-1]
	\arrow["{F_n(\otimes_{m_i}(A_{i,j}))}"{description}, from=3-1, to=3-3]
	\arrow["{F_n(\alpha_{(A_{i,j})})}"{description}, from=3-3, to=2-5]
	\arrow["{\alpha_{F(A_{i,j})}}"{description}, from=1-1, to=1-3]
	\arrow["{F_{\sum_i m_i}(A_{i,j})}"{description}, from=1-3, to=2-5]
\end{tikzcd}\]
\item % https://q.uiver.app/?q=WzAsMyxbMCwwLCJGKEEpIl0sWzIsMCwiXFxvdGltZXNfMShGKEEpKSJdLFsyLDEsIkYoXFxvdGltZXNfMShBKSkiXSxbMCwxLCJcXGlvdGFfe0YoQSl9IiwxXSxbMSwyLCJGXzEoQSkiLDFdLFswLDIsIkYoXFxpb3RhX0EpIiwxXV0=
\[\begin{tikzcd}[ampersand replacement=\&]
	{F(A)} \&\& {\otimes_1(F(A))} \\
	\&\& {F(\otimes_1(A))}
	\arrow["{\iota_{F(A)}}"{description}, from=1-1, to=1-3]
	\arrow["{F_1(A)}"{description}, from=1-3, to=2-3]
	\arrow["{F(\iota_A)}"{description}, from=1-1, to=2-3]
\end{tikzcd}\]
\end{itemize}

It is called strong/strict if the $F_n$ are isomorphisms/equalities.
\end{definition}

\begin{definition}
A \defin{monoidal transformation} $\gamma \colon (\cC, \otimes_\cC, I_\cC) \Rightarrow (\cD, \otimes_\cD, I_\cD)$ between biased monoidal functors is a natural transformation $\gamma \colon F \Rightarrow G$ such that:
% https://q.uiver.app/?q=WzAsNyxbMCwwLCJGKEEpIFxcb3RpbWVzIEYoQikiXSxbMiwwLCJGKEEgXFxvdGltZXMgQikiXSxbMCwyLCJHKEEpIFxcb3RpbWVzIEcoQikiXSxbMiwyLCJHKEEgXFxvdGltZXMgQikiXSxbNCwwLCJJIl0sWzYsMCwiRihJKSJdLFs2LDIsIkcoSSkiXSxbMCwyLCJcXGdhbW1hX0EgXFxvdGltZXMgXFxnYW1tYV9CIiwxXSxbMCwxLCJGXzIoQSxCKSIsMV0sWzEsMywiXFxnYW1tYV97QSBcXG90aW1lcyBCfSIsMV0sWzIsMywiR18yKEEsQikiLDFdLFs0LDUsIkZfMCIsMV0sWzUsNiwiXFxnYW1tYV9JIiwxXSxbNCw2LCJHXzAiLDFdXQ==
\[\begin{tikzcd}[ampersand replacement=\&]
	{F(A) \otimes F(B)} \&\& {F(A \otimes B)} \&\& I \&\& {F(I)} \\
	\\
	{G(A) \otimes G(B)} \&\& {G(A \otimes B)} \&\&\&\& {G(I)}
	\arrow["{\gamma_A \otimes \gamma_B}"{description}, from=1-1, to=3-1]
	\arrow["{F_2(A,B)}"{description}, from=1-1, to=1-3]
	\arrow["{\gamma_{A \otimes B}}"{description}, from=1-3, to=3-3]
	\arrow["{G_2(A,B)}"{description}, from=3-1, to=3-3]
	\arrow["{F_0}"{description}, from=1-5, to=1-7]
	\arrow["{\gamma_I}"{description}, from=1-7, to=3-7]
	\arrow["{G_0}"{description}, from=1-5, to=3-7]
\end{tikzcd}\]
\end{definition}

\begin{definition}
A \defin{monoidal transformation} between unbiased monoidal functors is a natural transformation $\gamma \colon F \Rightarrow G$, such that:
% https://q.uiver.app/?q=WzAsNCxbMCwwLCJcXG90aW1lc19uRihBX2kpIl0sWzIsMCwiRihcXG90aW1lc19uIEFfaSkiXSxbMiwyLCJHKFxcb3RpbWVzX24gQV9pKSJdLFswLDIsIlxcb3RpbWVzX24gRyhBX2kpIl0sWzAsMSwiRl9uKEFfaSkiLDFdLFsxLDIsIlxcZ2FtbWFfe1xcb3RpbWVzX24gQV9pfSIsMV0sWzAsMywiXFxvdGltZXNfbiBcXGdhbW1hX3tBX2l9IiwxXSxbMywyLCJHX24oQV9pKSIsMV1d
\[\begin{tikzcd}[ampersand replacement=\&]
	{\otimes_nF(A_i)} \&\& {F(\otimes_n A_i)} \\
	\\
	{\otimes_n G(A_i)} \&\& {G(\otimes_n A_i)}
	\arrow["{F_n(A_i)}"{description}, from=1-1, to=1-3]
	\arrow["{\gamma_{\otimes_n A_i}}"{description}, from=1-3, to=3-3]
	\arrow["{\otimes_n \gamma_{A_i}}"{description}, from=1-1, to=3-1]
	\arrow["{G_n(A_i)}"{description}, from=3-1, to=3-3]
\end{tikzcd}\]
\end{definition}

All of these notions form strict 2-categories.
We will write \MonCat{} ( resp. \MonCatPs{} and \MonCatSt ) for the 2-category of biased monoidal categories, monoidal functors (resp. strong/pseudo-monoidal functors and strict monoidal functors) and monoidal natural transformations.
We will write \UMonCat , \UMonCatPs{} and \UMonCatSt{} for their unbiased analogs.

\begin{prop}
There is a 2-equivalence of 2-categories
\[ \UMonCat \simeq \MonCat\]
and similarly for the strong and strict cases.
\end{prop}
\begin{proof}
See \cite{Leinster2004}.
\end{proof}

Since \UMonCat{} is a 2-category, one can defined a notion of equivalence internal to it.
We will talk about monoidal equivalence.

\begin{prop}
For unbiased monoidal categories \cC, \cD, the following are equivalent:
\begin{itemize}
\item \cC and \cD are monoidally equivalent, i.e. there are monoidal functors $F \colon \cC \to \cD$ and $G \colon \cD \to \cC$ and invertible monoidal transformations $\eta \colon \id \Rightarrow G \circ F$ and $\epsilon \colon F \circ F \Rightarrow \id$ 
\item \cC and \cD are equivalent and one of the functor of the equivalence is monoidal
\item there is a monoidal functor $F \colon \cC \to \cD$ that is full, faithful and essentially surjective on objects.
\end{itemize}
\end{prop}
\begin{proof}
%By definition the first condition implies the second.
%Futhermore, the two last conditions use the correspondence bewteen equivalence of categories and fully faithfull and surjective on objects functors.
%This can be found in MacLane for example.
%
%Now let assume that we have an equivalence of categories $F \colon \cC \to \cD$ and $G \colon \cD \to \cC$ with invertible natural transformations $\eta \colon \id \Rightarrow G \circ F$ and $\epsilon F \circ G \Rightarrow \id$ such that $(F, F_n)$ is monoidal.
%Without loss of generality we can assume that it is an adjoint equivalence.
%Let $G_n$ be % https://q.uiver.app/?q=WzAsNCxbMCwwLCJcXG90aW1lc19uIEcoQl9pKSJdLFsyLDAsIkcgXFxjaXJjIEYoXFxvdGltZXNfbiBHKEJfaSkpIl0sWzQsMCwiRyhcXG90aW1lc19uKEYgXFxjaXJjIEcoQl9pKSkpIl0sWzYsMCwiRyhcXG90aW1lc19uKEJfaSkpIl0sWzAsMSwiXFxldGFfe1xcb3RpbWVzX24oRyhCX2kpKX0iLDFdLFsxLDIsIkcoRl9uKEcoQl9pKSkpIiwxXSxbMiwzLCJHIChcXG90aW1lc19uKFxcZXBzaWxvbl97Ql9pfSkpIiwxXV0=
%\[\begin{tikzcd}[ampersand replacement=\&]
%	{\otimes_n G(B_i)} \&\& {G \circ F(\otimes_n G(B_i))} \&\& {G(\otimes_n(F \circ G(B_i)))} \&\& {G(\otimes_n(B_i))}
%	\arrow["{\eta_{\otimes_n(G(B_i))}}"{description}, from=1-1, to=1-3]
%	\arrow["{G(F_n(G(B_i)))}"{description}, from=1-3, to=1-5]
%	\arrow["{G (\otimes_n(\epsilon_{B_i}))}"{description}, from=1-5, to=1-7]
%\end{tikzcd}\]
See \cite{Leinster2004} for a proof.
\end{proof}

\begin{prop}
Every unbiased monoidal category is monoidally equivalent to a strict one.
\end{prop}
\begin{proof}
Let $(\cC, \otimes_n, \alpha_n, \iota)$ be an unbiased monoidal category.

Define $st(\cC)$ to be the category with:
\begin{itemize}
\item objects $(a_1,\dots, a_n)$ are finite lists of objects in \cC
\item morphisms $f \colon (a_1, \dots, a_m) \to (b_1, \dots, b_n)$ are morphisms $f \colon \otimes_m(a_i) \to \otimes_n(b_j)$
\item identities $\id_{(a_i)} := \otimes_n(\id_{a_i})$ 
\item composition is composition in \cC
\end{itemize}

Now define a monoidal structure on $st(\cC)$ by concatenation:
\begin{itemize}
\item $\otimes_0() := ()$
\item $\otimes_1((a_i)_i) := (a_i)_i$
\item $\otimes_n((a_{i,j})_j)_i := (a_{i,j})_{(i,j)}$ with the lexicographic order
\end{itemize}
This extends naturally to morphisms by using the $\alpha$'s.

Then $(a_i)_i = \otimes_1((a_i)_i)$ and $\otimes_n(\otimes_{m_i}((a_{i,j,k})_k)_j)_i = \otimes_n((a_{i,j,k})_{(j,k)})_i = (a_{i,j,k})_{(i,j,k)} = \otimes_{\sum_i m_i}((a_{i,j,k})_k)_{(i,j)}$.
So we can take identities as coherence morphisms, i.e. $st(\cC)$ is a strict unbiased monoidal category.

There is a functor $F \colon st(\cC) \to \cC$ defined by $F((a_i)_i) := \otimes_n(a_i)_i$ and $F(f) := f$.

$F$ is monoidal with natural transformations:
\[F_n((a_{i,j})) := \alpha_{n,(m_i)} \colon \otimes_n^\cC (F(a_{i,j})_j)_i = \otimes_n (\otimes_{m_i}(a_{i,j})_j)_i \to F(\otimes_n^{st(\cC)} (a_{i,j})_{(i,j)}) = \otimes_{\sum_i m_i}(a_{i,j})\]

The fact that it defines a weak functor follows from the coherence laws for the monoidal structure of \cC.

Since $F$ is the identity on morphisms, it is full and faithful.
Furthermore we have isomorphisms $a \simeq (a)$ so it is essentially surjective on objects.
\end{proof}

Of course, the same is true for biased monoidal categories.
This is the coherence theorem that can be found for example in Mac Lane.

\begin{definition}
A \defin{braided} monoidal category, $(\cC, \otimes, I, \gamma)$ is a monoidal category together with a natural isomorphism $\gamma_{A,B} \colon A \otimes B \to B \otimes A$ called the braiding such that the following diagrams commute:
% https://q.uiver.app/?q=WzAsMTIsWzEsMCwiKEEgXFxvdGltZXMgQikgXFxvdGltZXMgQyJdLFsyLDEsIkEgXFxvdGltZXMgKEIgXFxvdGltZXMgQykiXSxbMiwyLCIoQiBcXG90aW1lcyBDKSBcXG90aW1lcyBBIl0sWzAsMSwiKEIgXFxvdGltZXMgQSkgXFxvdGltZXMgQyJdLFswLDIsIkIgXFxvdGltZXMgKEEgXFxvdGltZXMgQykiXSxbMSwzLCJCIFxcb3RpbWVzIChDIFxcb3RpbWVzIEEpIl0sWzEsNCwiQSBcXG90aW1lcyAoQiBcXG90aW1lcyBDKSJdLFsyLDUsIihBIFxcb3RpbWVzIEIpIFxcb3RpbWVzIEMiXSxbMiw2LCJDIFxcb3RpbWVzIChBIFxcb3RpbWVzIEIpIl0sWzEsNywiKEMgXFxvdGltZXMgQSkgXFxvdGltZXMgQiJdLFswLDUsIkEgXFxvdGltZXMgKEMgXFxvdGltZXMgQikiXSxbMCw2LCIoQSBcXG90aW1lcyBDKSBcXG90aW1lcyBCIl0sWzAsMSwiXFxhbHBoYV97QSxCLEN9IiwxXSxbMSwyLCJcXGdhbW1hX3tBLEJcXG90aW1lcyBDfSIsMV0sWzAsMywiXFxnYW1tYV97QSxCfSBcXG90aW1lcyBDIiwxXSxbMyw0LCJcXGFscGhhX3tCLEEsQ30iLDFdLFs0LDUsIkIgXFxvdGltZXMgXFxnYW1tYV97QSxDfSIsMV0sWzIsNSwiXFxhbHBoYV97QixDLEF9IiwxXSxbNiw3LCJcXGFscGhhX3tBLEIsQ31eey0xfSIsMV0sWzcsOCwiXFxnYW1tYV97QSBcXG90aW1lcyBCLEN9IiwxXSxbOCw5LCJcXGFscGhhX3tDLEEsQn1eey0xfSIsMV0sWzYsMTAsIkEgXFxvdGltZXMgXFxnYW1tYV97QixDfSIsMV0sWzEwLDExLCJcXGFscGhhX3tBLEMsQn1eey0xfSIsMV0sWzExLDksIlxcZ2FtbWFfe0EsQn0gXFxvdGltZXMgQiIsMV1d
\[\begin{tikzcd}
	& {(A \otimes B) \otimes C} \\
	{(B \otimes A) \otimes C} && {A \otimes (B \otimes C)} \\
	{B \otimes (A \otimes C)} && {(B \otimes C) \otimes A} \\
	& {B \otimes (C \otimes A)} \\
	& {A \otimes (B \otimes C)} \\
	{A \otimes (C \otimes B)} && {(A \otimes B) \otimes C} \\
	{(A \otimes C) \otimes B} && {C \otimes (A \otimes B)} \\
	& {(C \otimes A) \otimes B}
	\arrow["{\alpha_{A,B,C}}"{description}, from=1-2, to=2-3]
	\arrow["{\gamma_{A,B\otimes C}}"{description}, from=2-3, to=3-3]
	\arrow["{\gamma_{A,B} \otimes C}"{description}, from=1-2, to=2-1]
	\arrow["{\alpha_{B,A,C}}"{description}, from=2-1, to=3-1]
	\arrow["{B \otimes \gamma_{A,C}}"{description}, from=3-1, to=4-2]
	\arrow["{\alpha_{B,C,A}}"{description}, from=3-3, to=4-2]
	\arrow["{\alpha_{A,B,C}^{-1}}"{description}, from=5-2, to=6-3]
	\arrow["{\gamma_{A \otimes B,C}}"{description}, from=6-3, to=7-3]
	\arrow["{\alpha_{C,A,B}^{-1}}"{description}, from=7-3, to=8-2]
	\arrow["{A \otimes \gamma_{B,C}}"{description}, from=5-2, to=6-1]
	\arrow["{\alpha_{A,C,B}^{-1}}"{description}, from=6-1, to=7-1]
	\arrow["{\gamma_{A,B} \otimes B}"{description}, from=7-1, to=8-2]
\end{tikzcd}\]
\end{definition}

\begin{definition}
A \emph{symmetric monoidal category} is a braided monoidal category where the brainding is a symmetry, i.e. $\gamma_{A,B}^{-1} = \gamma_{B,A}$ for any $A,B$.
\end{definition}

\begin{rk}
Instead of giving the same definition twice, we have chosen to talk about braided and symmetric monoidal categories without making any reference to biased or unbiased variant since it does not change anything here.

One could imagine a notion of unbiased braided and symmetric monoidal categories with $n$-ary braiding.
\end{rk}

A braided/symmetric monoidal functor is just one that respect the braiding.

\begin{definition}
A \defin{braided monoidal functor} between braided monoidal categories is a monoidal functor such that the following diagram commute:
% https://q.uiver.app/?q=WzAsNCxbMCwwLCJGKEEpIFxcb3RpbWVzIEYoQikiXSxbMiwwLCJGKEIpIFxcb3RpbWVzIEYoQSkiXSxbMiwyLCJGKEIgXFxvdGltZXMgQSkiXSxbMCwyLCJGKEEgXFxvdGltZXMgQikiXSxbMCwxLCJcXGdhbW1hX3tGKEEpLEYoQil9IiwxXSxbMSwyLCJGXzIoQixBKSIsMV0sWzAsMywiRl8yKEEsQikiLDFdLFszLDIsIkYoXFxnYW1tYV97QSxCfSkiLDFdXQ==
\[\begin{tikzcd}[ampersand replacement=\&]
	{F(A) \otimes F(B)} \&\& {F(B) \otimes F(A)} \\
	\\
	{F(A \otimes B)} \&\& {F(B \otimes A)}
	\arrow["{\gamma_{F(A),F(B)}}"{description}, from=1-1, to=1-3]
	\arrow["{F_2(B,A)}"{description}, from=1-3, to=3-3]
	\arrow["{F_2(A,B)}"{description}, from=1-1, to=3-1]
	\arrow["{F(\gamma_{A,B})}"{description}, from=3-1, to=3-3]
\end{tikzcd}\]

A \defin{symmetric monoidal functor} between symmetric monoidal categories is just a braided monoidal functor.
\end{definition}

No extra condition are required for monoidal transformations.

There are 2-categories $\mathbf{(/U)(Brd/Sym)MonCat_{/ps/st}}$ of unbiased/biased braided/symmetric monoidal categories, lax/strong/strict braided/symmetric monoidal functors and monoidal transformations.
The respective biased and unbiased ones are equivalent.
Furthermore, Mac Lane's strictification/coherence theorem
extends to this setting, i.e. the functor $\cC \to st(\cC)$ is a braided/symmetric monoidal functor when $\cC$ is braided/symmetric.
So every braided/symmetric monoidal category is braided/symmetric monoidally equivalent to a strict braided/symmetric one, where the strictness only involves the monoidal structure, the braiding need not be an equality.

\begin{rk}
For symmetric monoidal categories, left and right closure coincide.
So any symmetric monoidal category that is left/right closed is biclosed.
We will talk about symmetric monoidal closed categories when it is the case.
\end{rk}

\begin{example}
There are symmetric monoidal categories \Vect{} and \FVect{} whose objects are vector spaces (on a given field $\mathbb{K}$) and finite dimensional vector spaces respectively, morphisms are linear applications and monoidal products and units are tensor products of vector spaces and the field $\mathbb{K}$ considered as a vector space.
They are closed with $A \multimap B$ the vector space of linear maps from A to B.

There are also categories \Ban{} and \FBan{} of Banach spaces and continuous maps, and \Banc{} and \FBanc{} of Banach spaces and contractive maps.
There is a huge variety of (symmetric) monoidal products that can be put on these, even if one wants the underlying tensor product of vector spaces to be the usual one, i.e. for the forgetful functor into \Vect{} or \FVect{} to be strictly monoidal.
\end{example}

\begin{example}
The usual algebraic structures such as monoids, abelian groups, rings, etc... all forms monoidal categories with their structure-preserving maps and their usual tensor products.
Some are closed with $A \multimap B$ the set of structure-preserving maps from $A$ to $B$ with the algebraic structure inherited from the one on $B$, e.g. for commutative monoids or abelian groups.
The non-commutative ones are not.
Most of those are symmetric.

An example of a non-symmetric monoidal category involving algebraic structures is the monoidal category of $(R,R)$-bimodules over a ring $R$.
There is a group isomorphism between $A\otimes_R B$ and $B \otimes_R A$.
But since in the first case the left action is given by the action on $A$ and the right by the action on $B$ and in the second case it is reverse, the group isomorphism does not respect the actions and is not an isomorphism of bimodules. 
\end{example}

\begin{example}
Any cartesian category is monoidal with its monoidal product given by the cartesian product.
They are closed when they are as a cartesian category.
So the category of sets and functions or topological spaces and continuous maps, with the cartesian product are monoidal categories.
So are \Vect, \FVect, \Ban, \FBan{} and most of the algebraic examples with the direct sum.
This shows that a monoidal product is a structure on a category while the cartesian product is a property, meaning that there can be multiple monoidal products on a category while the cartesian product is define uniquely up to unique isomorphism.
\end{example}

\begin{example}
Models of propositional intuitionistic multiplicative linear logic are precisely symmetric monoidal closed categories.

Given the sequent calculus for IMLL, one can define its syntactic category, which is roughly given by:
\begin{itemize}
\item objects are propositions of IMLL
\item morphisms are derivation of proofs up-to $\beta$-reduction and $\eta$-expansion
\item identity is the axiom rule,
\item composition is the cut rule
\end{itemize}
It forms a symmetric monoidal closed category where the monoidal structure comes from the multiplicative conjunction $\otimes$ and its unit $1$, the internal hom comes from the implication $\multimap$ and the symmetries from the exchange law.
Furthermore, it is initial among symmetric monoidal closed categories on the given set of atomic propositions, and a model of IMLL correspond to a symmetric monoidal functor out of it.
\end{example}

\section{Multicategories}

\subsection{Definition}

As mentioned above a monoidal product is a structure on a category while cartesian product or the internal homs (the $\multimap, \multimapinv$ in a closed monoidal category) are properties.
Once a monoidal product is given on a monoidal category, one can define a notion of multimap $f \colon A_1, \dots, A_n \to B$ to be a map $f \colon A_1 \otimes \dots \otimes A_n \to B$ or $f \colon I \to B$ for the case $n=0$.
If the monoidal category is closed, then, by definition, $\cC(A_1 \otimes \dots \otimes A_n, B) \simeq \cC(A_n, A_{n-1} \multimap (\dots \multimap B))$.
For example, in the case of algebraic structures with their usual tensor products, multimaps act as maps that preserve the structure in each variable independently.
This is quite useful since it lets us talk about multilinear maps from a categorical point of view in the case of vector spaces or modules for example.
However, someone used to linear algebra might find it odd to have to define the notion of tensor products to get access to multilinear maps.
Indeed, in linear algebra the notion of multilinear map is usually more primitive and the tensor product can be define through the universal property of ``linearising multilinear maps'', i.e. the fact that multilinear maps correspond to linear maps out of the tensor product.
This is similar in the case of the other usual tensor products on algebraic structures: they let one transform multivariable functions preserving the structure in each variable independently to structure-preserving functions out of the tensor product.

So depending on the point of view one would prefer to adopt, one could start with a monoidal structure on a category and deduce a notion of multimap or one could ask for multimaps and deduce monoidal products.
Multicategories are the categorical structure needed to axiomatise the second point of view.
They are analogous to categories except that the morphisms are replaced by multimaps.

\begin{definition}
A \defin{multicategory} \cM is the data of:
\begin{itemize}
\item a collection of objects $\Ob{\cM}$
\item for any (possibly empty) finite list of objects $\Gamma$ and any object $A$, a collection of multimaps $\cM(\Gamma; A)$, we will write $f \colon \Gamma \to A$ in lieu of $f \in \cM(\Gamma;A)$
\item for any object $A$, an identity unary multimap $\id_A \colon A \to A$
\item for any multimap $f \colon \Gamma \to A$, $g \colon \Gamma_1', A, \Gamma_2' \to B$, a multimap $g \circ_{i+1} f \colon \Gamma_1', \Gamma, \Gamma_2' \to B$ where $i+1$ is the position of $A$ in the domain of $g$, i.e. the size of $\Gamma_1'$ is $|\Gamma_1'| = i$
\end{itemize}
subject to associativity, unitality and exchange law:
\begin{itemize}
\item for any $f \colon \Gamma_1, A, \Gamma_2 \to B$, $f \circ_{i + 1} \id_A = f = \id_A \circ_1 f$
\item for any $f \colon \Gamma \to A$, $g \colon \Gamma_1', A, \Gamma_2' \to B$ and $h \colon \Gamma_1'', B, \Gamma_2'' \to C$, 

$(h \circ_{j + 1} g) \circ_{i + j + 1} f = h \circ_{j + 1} (g \circ_{i + 1} f)$
\item for any $f_i \colon \Gamma_i \to A_i$ and $g \colon \Gamma_1', A_1, \Gamma_2', A_2, \Gamma_3' \to B$, 

$(g \circ_{i_1' + 1} f_1) \circ_{i_1' + i_1 + i_2' + 1} f_2 = (g \circ_{i_1' + 1 + i_2' + 1} f_2) \circ_{i_1' + 1} f_1$
\end{itemize}
\end{definition}

We will drop the index of the composition or replace it by the name of the object instead of its place when it is well-defined.

Instead of a composition in one specific component, we could have defined composition in all components at once.

\begin{definition}
A \emph{multicategory with parallel composition} \cM is the data of:
\begin{itemize}
\item a collection of objects $\Ob{\cM}$
\item for any (possibly empty) finite list of objects $\Gamma$ and any object $A$, a collection of multimaps $\cM(\Gamma; A)$, we will write $f \colon \Gamma \to A$ in lieu of $f \in \cM(\Gamma;A)$
\item for any object $A$, an identity unary multimap $\id_A \colon A \to A$
\item for any multimap $g \colon A_1, \dots, A_n \to B$ and any family of multimaps $f_i \colon \Gamma_i \to A_i$, a multimap $g \circ (f_1,\dots, f_n) \colon \Gamma_1, \dots, \Gamma_n \to B$
\end{itemize}
subject to associativity and  unitality:
\begin{itemize}
\item for any $f \colon A_1, \dots, A_n \to B$, 

$f \circ(\id_{A_1},\dots,\id_{A_n}) = f = \id_A \circ (f)$
\item for any $h \colon B_1, \dots, B_n \to C$, any $g_i \colon A_{i,1}, \dots, A_{i,m_i} \to B_i$ and any $f_{i,j} \colon \Gamma_{i,j} \to A_{i,j}$, 

$(h \circ (g_1,\dots, g_n)) \circ (f_{1,1}, \dots, f_{n,m_n}) = h \circ (g_1 \circ (f_{1,1}, \dots, f_{1,m_1}), \dots, g_n \circ (f_{n,1}, \dots, f_{n,m_n}))$
\end{itemize}
\end{definition}

We will sometimes drop $\circ$ and write $g(f_1,\dots,f_n)$ and other times we will drop the parenthesis for a list of one element $g \circ f$.

The two definitions coincide: from a partial composition we can define a parallel composition by $g \circ (f_1, \dots, f_n) := (g \circ_1 f_1) \circ_2 \dots \circ_n f_n$ which is well-defined because of associativity and exchange.
Associativity and unitality follows from associativity, unitality and exchange of the partial composition.
Conversely, from a parallel composition we can define a partial composition $g \circ_{i + 1} f := g \circ (\id_{A_1}, \dots, \id_{A_{i}}, f, \id_{A_{i + 2}}, \dots, \id_{A_n})$.
Associativity, unitality and exchange follows from associativity and unitality of the parallel composotion.

Notice that in the definition of the parallel composition from the partial one, we could choose to compose with morphisms $f_i$ in any order thanks to the exchange law.

We will represent multimaps as string diagrams.
For example parallel composition will be represented by:
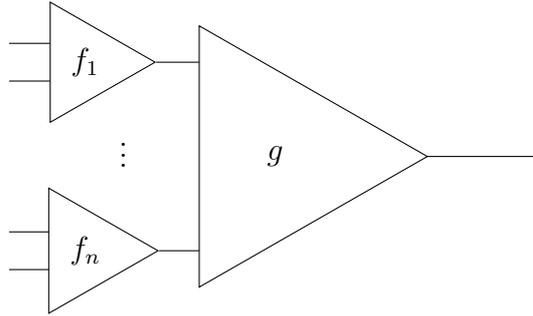
\[\input{parallel-comp.tikz}\]
while the exchange law for partial composition assert the following:
\[\input{multicat-exchange.tikz}\]

\begin{rk}
It is possible to define symmetric multicategories.
It is usually done by asking for an action of the symmetric group on the set of multimaps that is respected by composition.
Most of the theory developed in this thesis can be extended straightforwardly to the symmetric case, e.g. by asking that a notion preserves the symmetry.
\end{rk}

\begin{definition}
A \emph{functor of multicategories} $F \colon \cM \to \cN$ is the data of:
\begin{itemize}
\item for each object $A$ in \cM, an object $F(A)$ in \cN
\item for each multimap $f \colon A_1, \dots, A_n \to B$ in \cM, one $F(f) \colon F(A_1), \dots, F(A_n) \to F(B)$ in \cN
\end{itemize}
verifying functoriality:
\begin{itemize}
\item $F(\id_A) = \id_{F(A)}$
\item $F(g \circ_i f) = F(g) \circ_i F(f)$ for partial composition 
\item or $F(g \circ (f_1,\dots,f_n)) = F(g) \circ (F(f_1),\dots,F(f_n))$ for parallel composition
\end{itemize}
\end{definition}

\begin{definition}
A \emph{natural transformation} between functors of multicategories, written $F \xRightarrow{\alpha} G \colon \cC \to \cD$, is a family of multimaps $\alpha_A \colon F(A) \to G(A)$ such that for any $f \colon A_1, \dots, A_n \to B$ in \cC, we have
\[\alpha_{B} \circ F(f) = G(g) \circ (\alpha_{A_1},\dots, \alpha_{A_n})\]
Graphically,
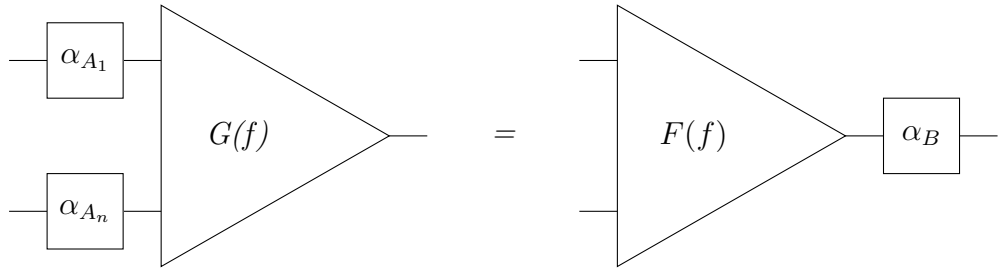
\[\input{multicat-transformation.tikz}\]
\end{definition}

We get a 2-category \MultiCat{} of multicategories, functors and natural transformations.
There is a 2-functor $\Cat \to \MultiCat$ that treats any category as a multicategory with only unary multimaps.

\subsection{Examples}

\begin{example}
There are multicategories $\Vect$, $\FVect$, $\Ban$, $\FBan$, $\Banc$ and $\FBanc$ of (finite dimensional) (normed) vector spaces and (continuous/contractive) multilinear maps.
\end{example}

\begin{example}
We can consider a lot of algebraic structures, e.g. monoids, abelian groups, rings, and form a multicategory whose maps are functions preserving the arguments in each independent variable.
\end{example}

For all of these examples, there is a monoidal category associated to it.
In fact we can do something similar for any monoidal category.

\begin{example}
Given a monoidal category $\cC$, one can define its underlying multicategory $\cU(\cC)$. 
Its objects are the objects of \cC.
Its multimaps are morphisms $f \colon \otimes_n(A_1, \dots, A_n) \to B$ in \cC.
We will see later that this always form a multicategory.
\end{example}

However, there are some multicategories that do not correspond to monoidal categories.

\begin{example}
Given a ring $R$ we can form the multicategory of (left) modules and multilinear maps between those.
This is always well-defined.
It is done by asking linearity of each variable independently.
Notice however that the order in which we use linearity of different variables shouldn't matter.
So for bilinear maps for example, we have
\[(r'r)\cdot f(a,a') = r \cdot f(a,r' \cdot a') = f(r \cdot a, r' \cdot a') = r' \cdot f(r \cdot a, a') = (r'r) \cdot f(a,a')\]
So in the non-commutative case multilinear maps ``forget'' that the ring is non-commutative \footnote{See \url{https://math.stackexchange.com/questions/1027478/correct-definition-of-bilinearmultilinear-maps-over-noncommutative-rings} for a discussion of multilinear maps over noncommutative rings}.
When the ring is commutative (so that a left module is also a right module), given two modules $A$ and $B$ and with action $\cdot_A \colon A \otimes R \to A$ and $\cdot_B \colon R \otimes B \to B$, one can define a the coequaliser of the action of $R$ on $A$ and $B$:% https://q.uiver.app/?q=WzAsMixbMCwwLCJSIFxcb3RpbWVzIEEgXFxvdGltZXMgQiJdLFszLDAsIkEgXFxvdGltZXMgQiJdLFswLDEsIlxcY2RvdF9BIFxcb3RpbWVzIEIiLDEseyJvZmZzZXQiOi0zfV0sWzAsMSwiKEEgXFxvdGltZXMgXFxjZG90X0IpIFxcY2lyYyBcXGdhbW1hX3tSLEF9IiwxLHsib2Zmc2V0IjozfV1d
% https://q.uiver.app/?q=WzAsMixbMCwwLCJBIFxcb3RpbWVzIFIgXFxvdGltZXMgQiJdLFszLDAsIkEgXFxvdGltZXMgQiJdLFswLDEsIlxcY2RvdF9BIFxcb3RpbWVzIEIiLDEseyJvZmZzZXQiOi0zfV0sWzAsMSwiQSBcXG90aW1lcyBcXGNkb3RfQiIsMSx7Im9mZnNldCI6M31dXQ==
\[\begin{tikzcd}[ampersand replacement=\&]
	{A \otimes R \otimes B} \&\&\& {A \otimes B}
	\arrow["{\cdot_A \otimes B}"{description}, shift left=3, from=1-1, to=1-4]
	\arrow["{A \otimes \cdot_B}"{description}, shift right=3, from=1-1, to=1-4]
\end{tikzcd}\]
to get a module $A \otimes_R B$.
With this tensor product, a multilinear map of modules then correspond to a module morphism out of this tensor product.

However, we have used the fact that the ring is commutative to mix left and right actions.
For a non-commutative ring, we can also define a tensor product corresponding to multilinear maps.
Then we would have $rs\cdot (a \otimes b) = r \cdot (a \otimes s \cdot b) = (r \cdot a \otimes s \cdot b) = s \cdot (r \cdot a \otimes b) = sr \cdot (a \otimes b)$ so the tensor product also ``forgets about'' the non-commutativity of $R$.

When generalising to module over monoids internal to any monoidal category the situation is worst.
If the monoidal category is not symmetric, it is not even possible to define the two actions as parallel morphisms, so we cannot take their coequaliser.
\end{example}

\begin{example}
There is a multicategory $\one$ with:
\begin{itemize}
\item a single object $\ast$
\item for any arity $n$, a unique multimap $\underline{n} \colon \ast, \dots \ast \to \ast$ from n copies of $\ast$
\end{itemize}
It is called the terminal multicategory, and it is the terminal object in category of (small) multicategories.
\end{example}

The notion of monoid internal to a monoidal category can be extend to multicategories.
As often in this thesis, we will prefer to consider an unbiased notion.

\begin{example}
Given a multicategory \cM, a monoid in \cM $(M,(m^k \colon M^k \to M)_{k \in \N})$ is an object in \cM, equipped with a family of $k$-ary multimaps $(m^k)_k$ closed under composition, i.e., $m^l(m^{k_1},\dots, m^{k_l}) = m^{\sum_i k_i}$.
We will often refer to the multimaps $m^k$ as the multiplication of the monoid.

For \cM symmetric, a monoid multimap $f \colon (M_1, (m^k_1)_k),\dots, (M_l, (m_l^k)_k) \to (N, (n^k)_k)$ is a multimap in \cM, $f \colon M_1, \dots M_l \to N$ that respects the monoid multiplication, represented graphically in \ref{fig:multicat-monoid-multimap} where the white dots are the multiplication of the monoid.
These form a symmetric multicategory $\Mon{\cM}$.
\end{example}

\begin{figure}
\centering
\input{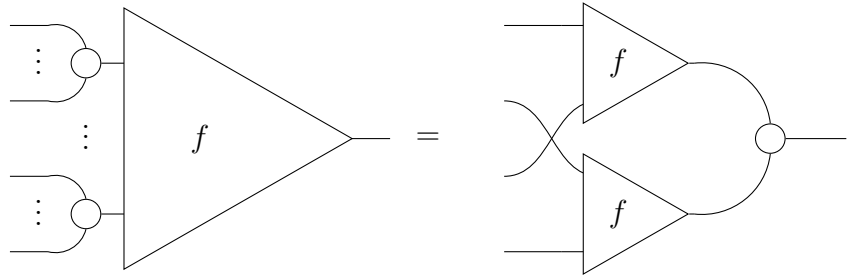}
\caption{A monoid multimap}
\label{fig:multicat-monoid-multimap}
\end{figure}

The terminal multicategory $\one$ is the free multicategory containing a monoid.
So a monoid in \cM amounts to the same data as a functor $\one \to \cM$.

Once one has defined a monoid, it is possible to talk about a module or a monoid action.

\begin{example}
Given a multicategory \cM, a monoid action $(M,C)$ consists of a monoid $(M, (m^k)_k)$ and an object $C$ in \cM equipped with a family of k-ary multimaps called the action of $M$ on $C$, $\alpha^k \colon M^k, C \to C$, compatible with the multiplication of $M$, $\alpha^l(m^{k_1},\dots,m^{k_l}) = \alpha^{\sum_i k_i}$.
$C$ is sometimes called a $M$-module or a module over $M$.

For \cM symmetric, an action multimap $(f,g) \colon (M_1,C_1), \dots, (M_l,C_l) \to (N,D)$ is a pair of a monoid multimap $f \colon M_1,\dots, M_l \to N$ and a multimap $g \colon C_1,\dots, C_l \to D$ such that the following diagram commutes:
% https://q.uiver.app/?q=WzAsNSxbMCwwLCJNXzEsQ18xLFxcZG90cyxNX2wsQ19sIl0sWzIsMCwiQ18xLFxcZG90cyxDX2wiXSxbMiwyLCJEIl0sWzAsMSwiTV8xLFxcZG90cyxNX2wsQ18xLFxcZG90cyxDX2wiXSxbMCwyLCJOLEQiXSxbMCwxLCJcXGFscGhhXzEsXFxkb3RzLFxcYWxwaGFfbCIsMV0sWzEsMiwiZyIsMV0sWzAsMywiXFxzaW1lcSIsMV0sWzMsNCwiZixnIiwxXSxbNCwyLCJcXGJldGEiLDFdXQ==
\[\begin{tikzcd}[ampersand replacement=\&]
	{M_1,C_1,\dots,M_l,C_l} \&\& {C_1,\dots,C_l} \\
	{M_1,\dots,M_l,C_1,\dots,C_l} \\
	{N,D} \&\& D
	\arrow["{\alpha_1,\dots,\alpha_l}"{description}, from=1-1, to=1-3]
	\arrow["g"{description}, from=1-3, to=3-3]
	\arrow["\simeq"{description}, from=1-1, to=2-1]
	\arrow["{f,g}"{description}, from=2-1, to=3-1]
	\arrow["\beta"{description}, from=3-1, to=3-3]
\end{tikzcd}\]

These form a symmetric multicategory.
\end{example}

\begin{example}
\Act{} is the multicategory with two objects $\ast,\star$, exactly one multimap $m^k \colon \ast^k \to \ast$ for each arity $k$ and exactly one multimap $\alpha^k \colon \ast^k,\star \to \star$ for each arity $k$.
\end{example}

\Act{} is the free multicategory containing an action.
An action internal to a multicategory \cM is the same as a functor $\Act \to \cM$.

\begin{example}
Given a monoid $M$ in a multicategory \cM, there is a multicategory $\cM/M$ whose objects are pairs $(A, \varphi_A \colon A \to M)$ of an object of $M$ and unary multimap from this object to $M$.
We will often abbreviate an object to its unary multimap part $\varphi_A$.
A multimap in $\cM/M$ $f \colon \varphi_{A_1},\dots,\varphi_{A_l} \to \varphi_{B}$ is a multimap $A_1,\dots,A_l \to B$ in \cM such that the following diagram commutes:
% https://q.uiver.app/?q=WzAsNCxbMCwwLCJBXzEsXFxkb3RzLEFfbCJdLFsyLDAsIkIiXSxbMiwyLCJNIl0sWzAsMiwiTSxcXGRvdHMsTSJdLFswLDEsImYiLDFdLFsxLDIsIlxcdmFycGhpX0IiLDFdLFswLDMsIlxcdmFycGhpX3tBXzF9LFxcZG90cyxcXHZhcnBoaV97QV9sfSIsMV0sWzMsMiwibV5sIiwxXV0=
\[\begin{tikzcd}[ampersand replacement=\&]
	{A_1,\dots,A_l} \&\& B \\
	\\
	{M,\dots,M} \&\& M
	\arrow["f"{description}, from=1-1, to=1-3]
	\arrow["{\varphi_B}"{description}, from=1-3, to=3-3]
	\arrow["{\varphi_{A_1},\dots,\varphi_{A_l}}"{description}, from=1-1, to=3-1]
	\arrow["{m^l}"{description}, from=3-1, to=3-3]
\end{tikzcd}\]
\end{example}

$\Cat/\Set$ is the multicategory of presheaves.
We will sometimes call $\cM/M$ the multicategory of $M$-presheaves in \cM.

\subsection{The underlying multicategory of a monoidal category}

\begin{prop}
There is a 2-functor $\cU \colon \UMonCat \to \MultiCat$.

It sends an unbiased monoidal category $(\cC, \otimes_n, \alpha_n, \iota)$ to the multicategory $\cU (\cC)$ with:
\begin{itemize}
\item objects those of \cC
\item multimaps $f \colon A_1, \dots, A_n \to B$, morphisms $f \colon \otimes_n(A_1,\dots, A_n) \to B$ in \cC
\item identities, $\iota_A^{-1} \colon \otimes_1(A) \to A$
\item composition of $g \colon \otimes_n B_i \to C$ and $f_i \colon \otimes_{m_i} A_{i,j} \to B_i$ given by $(g \circ \otimes_n(f_i)) \circ \alpha^{-1}_{n,(m_i)_i}$:
% https://q.uiver.app/?q=WzAsNCxbMCwwLCJcXG90aW1lc197XFxzdW1faSBtX2l9IEFfe2ksan0iXSxbMiwwLCJcXG90aW1lc19uIFxcb3RpbWVzX3ttX2l9IEFfe2ksan0iXSxbNCwwLCJcXG90aW1lc19uIEJfaSJdLFs2LDAsIkMiXSxbMCwxLCJcXGFscGhhX3soKEFfe2ksan0pX2opX2l9XnstMX0iXSxbMSwyLCJcXG90aW1lc19uIGZfaSJdLFsyLDMsImciXV0=
\[\begin{tikzcd}[ampersand replacement=\&]
	{\otimes_{\sum_i m_i} A_{i,j}} \&\& {\otimes_n \otimes_{m_i} A_{i,j}} \&\& {\otimes_n B_i} \&\& C
	\arrow["{\alpha_{((A_{i,j})_j)_i}^{-1}}", from=1-1, to=1-3]
	\arrow["{\otimes_n f_i}", from=1-3, to=1-5]
	\arrow["g", from=1-5, to=1-7]
\end{tikzcd}\]
\end{itemize}

A monoidal functor $(F, F_n) \colon \cC \to \cD$ is sent to the functor of multicategories $\cU (F)$ given by:
\begin{itemize}
\item $\cU (F)(A) = F(A)$
\item for $f \colon \otimes_n(A_i) \to B$ in $\cU (\cC)$, $\cU (F)(f) \colon \otimes_n(F(A_i)) \xrightarrow{F_n} F(\otimes_n(A_i)) \xrightarrow{F(f)} F(B)$
\end{itemize}

A monoidal transformation $\alpha \colon F \Rightarrow G$ is sent to the natural transformation 

$\cU (\alpha) \colon \cU (F) \Rightarrow \cU (G)$ given by $\cU (\alpha)_A := \alpha_A \circ \iota_A^{-1}$.
\end{prop}
\begin{proof}
First let prove that $\cU (\cC)$ is a multicategory for any monoidal category \cC.
Taking $g \colon \otimes_n B_i \to C$ and the identities, $\iota_{B_i}^{-1}$, and composing them we get
$(g \circ \otimes_n (\iota_{B_i}^{-1})) \circ \alpha_{n,(1)_i}^{-1}$ which is $g$ using the coherence law of an unbiased monoidal category:
% https://q.uiver.app/?q=WzAsNCxbMCwwLCJcXG90aW1lc19uIEJfaSJdLFswLDIsIlxcb3RpbWVzX25cXG90aW1lc18xKEJfaSkiXSxbMiwwLCJcXG90aW1lc19uIEJfaSJdLFs0LDAsIkMiXSxbMCwxLCJcXGFscGhhX3soMSlfe2kgXFxsZXEgbn19XnstMX0iLDFdLFsxLDIsIlxcb3RpbWVzX24oXFxpb3RhX3tCX2l9XnstMX0pIiwxXSxbMiwzLCJnIiwxXSxbMCwyLCIiLDEseyJsZXZlbCI6Miwic3R5bGUiOnsiaGVhZCI6eyJuYW1lIjoibm9uZSJ9fX1dXQ==
\[\begin{tikzcd}[ampersand replacement=\&]
	{\otimes_n B_i} \&\& {\otimes_n B_i} \&\& C \\
	\\
	{\otimes_n\otimes_1(B_i)}
	\arrow["{\alpha_{((B_i))_{i \leq n}}^{-1}}"{description}, from=1-1, to=3-1]
	\arrow["{\otimes_n(\iota_{B_i}^{-1})}"{description}, from=3-1, to=1-3]
	\arrow["g"{description}, from=1-3, to=1-5]
	\arrow[Rightarrow, no head, from=1-1, to=1-3]
\end{tikzcd}\]

Now composing the identity $\iota_B^{-1}$ with $f \colon \otimes_m A_j \to B$ we get $(\iota_B^{-1} \circ \otimes_1(f)) \circ \alpha_{1,(m)}^{-1}$ which gives $f$ by naturality of $\alpha^{-1}$ and coherence law of an unbiased monoidal category:
% https://q.uiver.app/?q=WzAsNSxbMCwxLCJcXG90aW1lc19tQV9qIl0sWzIsMCwiXFxvdGltZXNfMShcXG90aW1lc19tIEFfaikiXSxbNCwwLCJcXG90aW1lc18xKEIpIl0sWzYsMSwiQiJdLFsyLDEsIkIiXSxbMCwxLCJcXGFscGhhX3tcXG90aW1lc19tIEFfan1eey0xfSIsMV0sWzEsMiwiXFxvdGltZXNfMShmKSIsMV0sWzIsMywiXFxpb3RhX0Jeey0xfSIsMV0sWzAsNCwiZiIsMV0sWzQsMiwiXFxhbHBoYV97Qn1eey0xfSIsMV0sWzQsMywiIiwxLHsibGV2ZWwiOjIsInN0eWxlIjp7ImhlYWQiOnsibmFtZSI6Im5vbmUifX19XV0=
\[\begin{tikzcd}[ampersand replacement=\&]
	\&\& {\otimes_1(\otimes_m A_j)} \&\& {\otimes_1(B)} \\
	{\otimes_mA_j} \&\& B \&\&\&\& B
	\arrow["{\alpha_{\otimes_m A_j}^{-1}}"{description}, from=2-1, to=1-3]
	\arrow["{\otimes_1(f)}"{description}, from=1-3, to=1-5]
	\arrow["{\iota_B^{-1}}"{description}, from=1-5, to=2-7]
	\arrow["f"{description}, from=2-1, to=2-3]
	\arrow["{\alpha_{(B)}^{-1}}"{description}, from=2-3, to=1-5]
	\arrow[Rightarrow, no head, from=2-3, to=2-7]
\end{tikzcd}\]

Finally for associativity of the parallel composition, let take $h \colon C_i \to D$, $g_i \colon B_{i,j} \to C_i$ and $f_{i,j} \colon A_{i,j,k} \to B_{i,j}$, then the proof of the associativity is given by the following commutative diagram:

% https://q.uiver.app/?q=WzAsMTMsWzQsMCwiXFxvdGltZXNfe1xcc3VtX2kgbV9pfSBCX3tpLGp9Il0sWzEwLDAsIlxcb3RpbWVzX24gQ19pIl0sWzEyLDIsIkQiXSxbNywwLCJcXG90aW1lc19uIFxcb3RpbWVzX3ttX2l9IEJfe2ksan0iXSxbMSwwLCJcXG90aW1lc197XFxzdW1faSBtX2l9IFxcb3RpbWVzX3tsX3tpLGp9fUFfe2ksaixrfSJdLFswLDIsIlxcb3RpbWVzX3tcXHN1bV97aSxqfSBsX3tpLGp9fSBBX3tpLGosa30iXSxbMTAsNCwiXFxvdGltZXNfbiBDX2kiXSxbMSw0LCJcXG90aW1lc19uXFxvdGltZXNfe21faX1BX3tpLGosa30iXSxbMiwyLCJcXG90aW1lc19uIFxcb3RpbWVzX3ttX2l9IFxcb3RpbWVzX3tsX3tpLGp9fSBBX3tpLGosa30iXSxbMSwyLCIoMSkiXSxbMywxLCIoMikiXSxbNiwyLCIoMykiXSxbMTAsMiwiKDQpIl0sWzEsMiwiaCIsMV0sWzMsMSwiXFxvdGltZXNfbihnX2kpIiwxXSxbMCwzLCJcXGFscGhhXnstMX0iLDFdLFs0LDAsIlxcb3RpbWVzX3tcXHN1bV9pIG1faX0gKGZfe2ksan0pIiwxXSxbNSw0LCJcXGFscGhhXnstMX0iLDFdLFs3LDgsIlxcb3RpbWVzX24gKFxcYWxwaGFeey0xfSkiLDFdLFs1LDcsIlxcYWxwaGFeey0xfSIsMV0sWzYsMiwiaCIsMV0sWzcsNiwiXFxvdGltZXNfbihnX2kgXFxjaXJjIFxcb3RpbWVzX3ttX2l9KGZfe2ksan0pIFxcY2lyYyBcXGFscGhhXnstMX0pIiwxXSxbNCw4LCJcXGFscGhhXnstMX0iLDFdLFs4LDMsIlxcb3RpbWVzX25cXG90aW1lc197bV9pfShmX3tpLGp9KSIsMV0sWzMsNiwiXFxvdGltZXNfbihnX2kpIiwxXV0=
\resizebox{\hsize}{!}{
\begin{tikzcd}[ampersand replacement=\&]
	\& {\otimes_{\sum_i m_i} \otimes_{l_{i,j}}A_{i,j,k}} \&\&\& {\otimes_{\sum_i m_i} B_{i,j}} \&\&\& {\otimes_n \otimes_{m_i} B_{i,j}} \&\&\& {\otimes_n C_i} \\
	\&\&\& {(2)} \\
	{\otimes_{\sum_{i,j} l_{i,j}} A_{i,j,k}} \& {(1)} \& {\otimes_n \otimes_{m_i} \otimes_{l_{i,j}} A_{i,j,k}} \&\&\&\& {(3)} \&\&\&\& {(4)} \&\& D \\
	\\
	\& {\otimes_n\otimes_{m_i}A_{i,j,k}} \&\&\&\&\&\&\&\&\& {\otimes_n C_i}
	\arrow["h"{description}, from=1-11, to=3-13]
	\arrow["{\otimes_n(g_i)}"{description}, from=1-8, to=1-11]
	\arrow["{\alpha^{-1}}"{description}, from=1-5, to=1-8]
	\arrow["{\otimes_{\sum_i m_i} (f_{i,j})}"{description}, from=1-2, to=1-5]
	\arrow["{\alpha^{-1}}"{description}, from=3-1, to=1-2]
	\arrow["{\otimes_n (\alpha^{-1})}"{description}, from=5-2, to=3-3]
	\arrow["{\alpha^{-1}}"{description}, from=3-1, to=5-2]
	\arrow["h"{description}, from=5-11, to=3-13]
	\arrow["{\otimes_n(g_i \circ \otimes_{m_i}(f_{i,j}) \circ \alpha^{-1})}"{description}, from=5-2, to=5-11]
	\arrow["{\alpha^{-1}}"{description}, from=1-2, to=3-3]
	\arrow["{\otimes_n\otimes_{m_i}(f_{i,j})}"{description}, from=3-3, to=1-8]
	\arrow["{\otimes_n(g_i)}"{description}, from=1-8, to=5-11]
\end{tikzcd}
}

where the square (1) is a coherence law for an unbiased monoidal category, (2) is naturality of $\alpha^{-1}$, (3) is functoriality of $\otimes_n$ and (4) is an equality.

Now if $(F,F_n)$ is a monoidal functor, let prove that $\cU (F)$ is a functor of multicategories.
The image of the identity is $\cU(F)(\iota_{A}^{-1}) =  F(\iota_A^{-1}) \circ F_1(A) = \iota_{F(A)}^{-1}$ by one of the coherence law of a monoidal functor.

Then for $g \colon B_i \to C$ and $f_i \colon A_{i,j} \to B_i$, we have functorality:
% https://q.uiver.app/?q=WzAsMTEsWzIsMCwiRihcXG90aW1lc197XFxzdW1faSBtX2l9QV97aSxqfSkiXSxbMywyLCJGKFxcb3RpbWVzX24gXFxvdGltZXNfe21faX0gQV97aSxqfSkiXSxbNiwwLCJGKEMpIl0sWzYsNCwiRihcXG90aW1lc19uIEJfaSkiXSxbNCw0LCJcXG90aW1lc19uRihCX2kpIl0sWzAsMCwiXFxvdGltZXNfe1xcc3VtX2kgbV9pfSBGKEFfe2ksan0pIl0sWzIsNCwiXFxvdGltZXNfbiBGKFxcb3RpbWVzX3ttX2l9QV97aSxqfSkiXSxbMCw0LCJcXG90aW1lc19uXFxvdGltZXNfe21faX0gRihBX3tpLGp9KSJdLFsxLDEsIigxKSJdLFs0LDEsIigyKSJdLFszLDMsIigzKSJdLFswLDEsIkYoXFxhbHBoYV57LTF9KSIsMV0sWzAsMiwiRihnIFxcY2lyYyBcXG90aW1lc19uKGZfaSkgXFxjaXJjIFxcYWxwaGFeey0xfSkiLDFdLFs0LDMsIkZfbiIsMV0sWzUsMCwiRl97XFxzdW1faSBtX2l9IiwxXSxbNyw2LCJcXG90aW1lc19uIEZfe21faX0iLDFdLFs2LDQsIlxcb3RpbWVzX24gRihmX2kpIiwxXSxbNSw3LCJcXGFscGhhXnstMX0iLDFdLFszLDIsIkYoZykiLDFdLFs2LDEsIkZfbiIsMV0sWzEsMywiRihcXG90aW1lc19uKGZfaSkpIiwxXV0=

\resizebox{\hsize}{!}{
\begin{tikzcd}[ampersand replacement=\&]
	{\otimes_{\sum_i m_i} F(A_{i,j})} \&\& {F(\otimes_{\sum_i m_i}A_{i,j})} \&\&\&\& {F(C)} \\
	\& {(1)} \&\&\& {(2)} \\
	\&\&\& {F(\otimes_n \otimes_{m_i} A_{i,j})} \\
	\&\&\& {(3)} \\
	{\otimes_n\otimes_{m_i} F(A_{i,j})} \&\& {\otimes_n F(\otimes_{m_i}A_{i,j})} \&\& {\otimes_nF(B_i)} \&\& {F(\otimes_n B_i)}
	\arrow["{F(\alpha^{-1})}"{description}, from=1-3, to=3-4]
	\arrow["{F(g \circ \otimes_n(f_i) \circ \alpha^{-1})}"{description}, from=1-3, to=1-7]
	\arrow["{F_n}"{description}, from=5-5, to=5-7]
	\arrow["{F_{\sum_i m_i}}"{description}, from=1-1, to=1-3]
	\arrow["{\otimes_n F_{m_i}}"{description}, from=5-1, to=5-3]
	\arrow["{\otimes_n F(f_i)}"{description}, from=5-3, to=5-5]
	\arrow["{\alpha^{-1}}"{description}, from=1-1, to=5-1]
	\arrow["{F(g)}"{description}, from=5-7, to=1-7]
	\arrow["{F_n}"{description}, from=5-3, to=3-4]
	\arrow["{F(\otimes_n(f_i))}"{description}, from=3-4, to=5-7]
\end{tikzcd}}

where (1) is the coherence law of the monoidal functor $F$, (2) is functoriality of $F$, and (3) is naturality of $F_n$.

Then given a monoidal transformation $\gamma \colon F \Rightarrow G$ we need to prove that $\cU (\gamma)$, i.e. $\gamma \circ \iota_A^{-1}$, is a natural transformation.
It is given by the following diagram:
% https://q.uiver.app/?q=WzAsMTgsWzQsMCwiXFxvdGltZXMgXzFGKFxcb3RpbWVzX24gQV9pKSJdLFs2LDAsIlxcb3RpbWVzXzEgRihCKSJdLFsyLDAsIlxcb3RpbWVzXzEgXFxvdGltZXNfbiBGKEFfaSkiXSxbMCwwLCJcXG90aW1lc19uIEYoQV9pKSJdLFs4LDAsIkYoQikiXSxbMTAsMiwiRyhCKSJdLFs4LDQsIkcoXFxvdGltZXNfbiBBX2kpIl0sWzYsNCwiXFxvdGltZXNfbiBHKEFfaSkiXSxbMiw0LCJcXG90aW1lc19uRihBX2kpIl0sWzAsNCwiXFxvdGltZXNfblxcb3RpbWVzXzFGKEFfaSkiXSxbNSwyLCJGKFxcb3RpbWVzX24gQV9pKSJdLFsxLDEsIigxKSJdLFswLDIsIlxcYnVsbGV0Il0sWzEsMywiKDEnKSJdLFszLDIsIigyKSJdLFs1LDEsIigzKSJdLFs1LDMsIig0KSJdLFs4LDIsIig1KSJdLFswLDEsIlxcb3RpbWVzXzEgRihmKSIsMV0sWzIsMCwiXFxvdGltZXNfMUZfbiIsMV0sWzMsMiwiXFxhbHBoYV57LTF9IiwxXSxbMSw0LCJcXGlvdGFfe0YoQil9XnstMX0iLDFdLFs0LDUsIlxcZ2FtbWFfQiIsMV0sWzcsNiwiR19uIiwxXSxbOCw3LCJcXG90aW1lc19uIFxcZ2FtbWFfe0FfaX0iLDFdLFs5LDgsIlxcb3RpbWVzX24gXFxpb3RhX3tGKEFfaSl9XnstMX0iLDFdLFszLDksIlxcYWxwaGFeey0xfSIsMV0sWzYsNSwiRyhmKSIsMV0sWzAsMTAsIlxcaW90YV97RihcXG90aW1lc19uIEFfaSl9XnstMX0iLDFdLFsxMCw0LCJGKGYpIiwxXSxbMyw4LCIiLDEseyJsZXZlbCI6Miwic3R5bGUiOnsiaGVhZCI6eyJuYW1lIjoibm9uZSJ9fX1dLFsyLDgsIlxcaW90YV97XFxvdGltZXNfbkYoQV9pKV57LTF9fSIsMV0sWzgsMTAsIkZfbiIsMV0sWzEwLDYsIlxcZ2FtbWFfe1xcb3RpbWVzX24gQV9pfSIsMV1d

\resizebox{\hsize}{!}{
\begin{tikzcd}[ampersand replacement=\&]
	{\otimes_n F(A_i)} \&\& {\otimes_1 \otimes_n F(A_i)} \&\& {\otimes _1F(\otimes_n A_i)} \&\& {\otimes_1 F(B)} \&\& {F(B)} \\
	\& {(1)} \&\&\&\& {(3)} \\
	\bullet \&\&\& {(2)} \&\& {F(\otimes_n A_i)} \&\&\& {(5)} \&\& {G(B)} \\
	\& {(1')} \&\&\&\& {(4)} \\
	{\otimes_n\otimes_1F(A_i)} \&\& {\otimes_nF(A_i)} \&\&\&\& {\otimes_n G(A_i)} \&\& {G(\otimes_n A_i)}
	\arrow["{\otimes_1 F(f)}"{description}, from=1-5, to=1-7]
	\arrow["{\otimes_1F_n}"{description}, from=1-3, to=1-5]
	\arrow["{\alpha^{-1}}"{description}, from=1-1, to=1-3]
	\arrow["{\iota_{F(B)}^{-1}}"{description}, from=1-7, to=1-9]
	\arrow["{\gamma_B}"{description}, from=1-9, to=3-11]
	\arrow["{G_n}"{description}, from=5-7, to=5-9]
	\arrow["{\otimes_n \gamma_{A_i}}"{description}, from=5-3, to=5-7]
	\arrow["{\otimes_n \iota_{F(A_i)}^{-1}}"{description}, from=5-1, to=5-3]
	\arrow["{\alpha^{-1}}"{description}, from=1-1, to=5-1]
	\arrow["{G(f)}"{description}, from=5-9, to=3-11]
	\arrow["{\iota_{F(\otimes_n A_i)}^{-1}}"{description}, from=1-5, to=3-6]
	\arrow["{F(f)}"{description}, from=3-6, to=1-9]
	\arrow[Rightarrow, no head, from=1-1, to=5-3]
	\arrow["{\iota_{\otimes_nF(A_i)^{-1}}}"{description}, from=1-3, to=5-3]
	\arrow["{F_n}"{description}, from=5-3, to=3-6]
	\arrow["{\gamma_{\otimes_n A_i}}"{description}, from=3-6, to=5-9]
\end{tikzcd}}
where (1) and (1') are coherence laws of an unbiased monoidal category, (2) and (3) are naturality of $\iota^{-1}$, (4) is the coherence law of the monoidal transformation $\gamma$ and (5) is naturality of $\gamma$.

Now that we have proven that \cU is well-defined, we need to prove that it is functorial.
The identity functor is monoidal with $F_n$ the identity.
So it is clear from the definition that the identity functor is sent to the identity by \cU.
Now given two monoidal functors $(F,F_n)$ and $(G,G_n)$ we have that $\cU (G \circ F)(A) = G \circ F(A) =  \cU (G) \circ \cU (F) (A)$ for any object $A$, and for a multimap $f \colon \otimes_n A_i \to B$, we have:
% https://q.uiver.app/?q=WzAsNCxbMCwxLCJcXG90aW1lc19uIEcoRihBX2kpKSJdLFsyLDAsIkcoRihcXG90aW1lc19uIEFfaSkpIl0sWzIsMiwiRyhcXG90aW1lc19uRihBX2kpKSJdLFs0LDEsIkcoRihCKSkiXSxbMCwxLCIoRyBcXGNpcmMgRilfbiIsMV0sWzIsMywiRyhGKGYpIFxcY2lyYyBGX24pIiwxXSxbMiwxLCJHKEZfbikiLDFdLFsxLDMsIkcgXFxjaXJjIEYoZikiLDFdLFswLDIsIkdfbiIsMV1d
\[\begin{tikzcd}[ampersand replacement=\&]
	\&\& {G(F(\otimes_n A_i))} \\
	{\otimes_n G(F(A_i))} \&\&\&\& {G(F(B))} \\
	\&\& {G(\otimes_nF(A_i))}
	\arrow["{(G \circ F)_n}"{description}, from=2-1, to=1-3]
	\arrow["{G(F(f) \circ F_n)}"{description}, from=3-3, to=2-5]
	\arrow["{G(F_n)}"{description}, from=3-3, to=1-3]
	\arrow["{G \circ F(f)}"{description}, from=1-3, to=2-5]
	\arrow["{G_n}"{description}, from=2-1, to=3-3]
\end{tikzcd}\]
where the left triangle is the definition of $(G \circ F)_n$ and the right triangle is functoriality of $G$.

Finally, for functoriality of 2-cells, we have that the identity monoidal transformation is sent to the transformation with components $\iota_A^{-1}$, i.e. the identities in $\cU (\cD)$ and for the composition we use the coherence law of an unbiased monoidal category to ``cancel out'' the $\alpha^{-1}$ and one $\iota^{-1}$.
\end{proof}

\begin{prop}
$\cU \colon \UMonCat \to \MultiCat$ is 2-fully-faithful, i.e. for any pair of unbiased monoidal categories \cM and \cN  the functor \[\cU_{\cM,\cN} \colon \UMonCat(\cM,\cN) \to \MultiCat(\cU(\cM),\cU(\cN))\] is an isomorphism of categories.
\end{prop}
\begin{proof}
Let first prove that $\cU_{\cM,\cN}$ is bijective on objects.
Let $\cF \colon \cU(\cM) \to \cU(\cN)$ be a functor of multicategories.
We define a monoidal functor $F \colon \cM \to \cN$ such that $\cU(F) = \cF$ as follows.
We take $F$ to be $\cF$ on object.
On morphisms $f \colon A \to B$, we take $F(f) := \cF(f \circ \iota_{F(A)}^{-1}) \circ \iota_{F(A)}$.
Furthermore, to get $F_n$ we consider the multimap $m_{(A_i)} \colon A_1, \dots, A_n \to \otimes_n A_i$ in $\cU(\cM)$ given by $\id_{\otimes_n A_i}$ in $\cM$ and we define $F_n := \cF(m_{(A_i)})$.
It can be checked that $\cU (F) = \cF$.

To conclude we need to prove that $\cU_{\cM,\cN}$ is fully-faithful.
Given functors $F,G \colon \cM \to \cM$ the function $\cU \colon Nat(F,G) \to Nat(\cU(F),\cU(G))$ is a bijection.
It follows from the fact that $\cU(\gamma)$ is defined by $\gamma \circ \iota^{-1}$ with $\iota^{-1}$ an isomorphism.
\end{proof}

This means that $\UMonCat$ is a sub-2-category of $\MultiCat$.

\section{Representable multicategories}

\subsection{Definition}

In this section we will characterise the sub-2-category \UMonCat in \MultiCat, i.e. we will give a property of a multicategory that makes it the underlying multicategory of a monoidal category.
As mentioned in the case of vector spaces are other algebraic structure, the tensor product has the universal property of ``linearising multilinear maps''.
It is this universal property that we will generalise to any multicategory.

In the following we will assume a multicategory \cM unless stated otherwise.

\begin{definition}
A multimap $f \colon \Gamma \to B$ is \defin{universal in $B$} or just \defin{universal} if for any multimap $g \colon \Gamma_1', \Gamma, \Gamma_2' \to C$ there is a unique multimap $g/f \colon \Gamma_1', B, \Gamma_2' \to C$ such that $g = g/f \circ f$.

In other words, the functions $-\circ f \colon \cM(\Gamma_1', B, \Gamma_2';C) \to \cM(\Gamma_1', \Gamma, \Gamma_2';C)$ given by precomposition by $f$ are invertible.
\end{definition}

This unique factorisation property characterise the tensor product in a multicategory.

\begin{prop}
Given a finite list of objects $(A_i)_{1 \leq i \leq n}$, if there is a universal multimap $A_1, \dots A_n \to B$ then $B$ is unique up to unique invertible (unary) multimap.
\end{prop}
\begin{proof}
Let $f_1 \colon A_1, \dots, A_n \to B_1$ and $f_2 \colon A_1, \dots, A_n \to B_2$ be universal multimaps.
Then by the factorisation property of $f_1$ we get a unary map $f_2/f_1$ such that $f_2 = f_2/f_1 \circ f_1$ and similarly we have $f_1/f_2$ by the factorisation property of $f_2$.
But then $f_2 = f_2/f_1 \circ f_1 = f_2/f_1 \circ f_1/f_2 \circ f_2$.
And we also have that $f_2 = \id_{B_2} \circ f_2$.
By the uniqueness of prefactorisation by $f_2$ we get $f_2/f_1 \circ f_1/f_2 = \id_{B_2}$.
Similarly, $f_1/f_2 \circ f_2/f_1 = \id_{B_1}$ which gives us our isomorphism between $B_1$ and $B_2$.
\end{proof}

From now on, we will call the codomain of a universal multimap from $A_1,\dots, A_n$ their tensor product and we will write it $A_1 \otimes \dots \otimes A_n$.
We will usually denote a universal multimap $m_{(A_i)_i} \colon A_1, \dots, A_n \to A_1 \otimes \dots \otimes A_n$, and will often drop the index if there is no ambiguity.
We will represent it graphically as a white node.
So the factorisation property becomes:
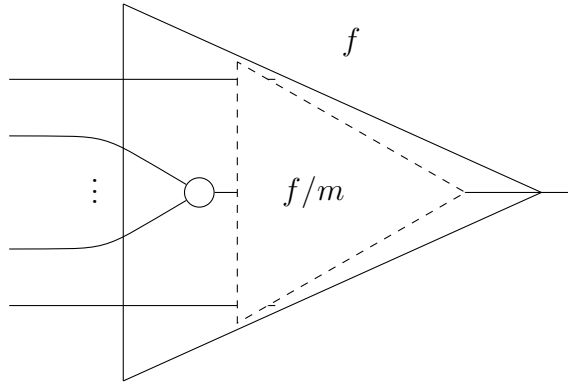
\[\input{multimap-univ1.tikz}\]

Furthermore, we will represent $f/m$ by ``precomposing'' $f$ with a black node merging the $A_i$ together like follow:
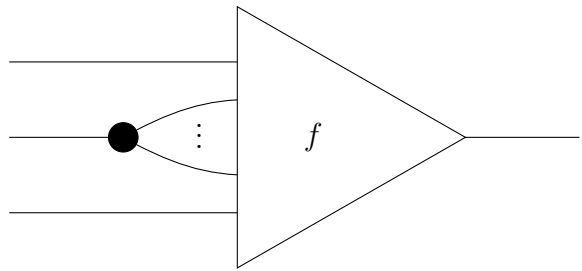
\[\input{multicat-prefact.tikz}\]

Informally, we can think of it as an operation $A_1 \otimes \dots \otimes A_n \to A_1, \dots, A_n$ that splits the tensor product into its parts.
Although this does not correspond to any actual operation in the multicategory, it will come in handy.

The factorisation property now states that the following is $f$:
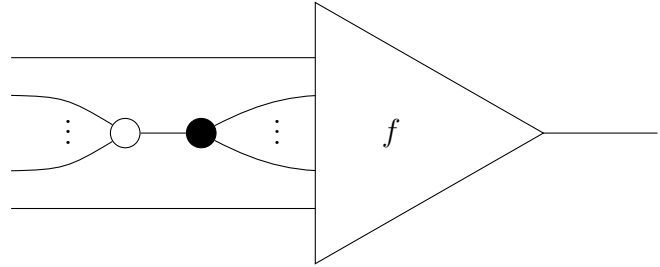
\[\input{multicat-univ.tikz}\]
i.e. that merging objects and then splitting them amounts to nothing.

Furthermore, the reverse cancellation (splitting then merging) is given by the uniqueness of factorisation.

\begin{rk}
We can also think of the white node as an introduction rule for the tensor and the black one as an elimination.
\end{rk}

\begin{prop}
The identity is a universal multimap and universal multimaps are closed under composition.
\end{prop}
\begin{proof}
For the identity, any multimap can be uniquely factorised through the identity followed by itself.

Now given universal maps $m_i \colon \Gamma_i \to A_i$ and $m \colon A_1,\dots,A_n \to B$, any multimap $f \colon \Gamma_1', \Gamma_1, \dots, \Gamma_n, \Gamma_2' \to C$ can be factored through $m_1$ to get a multimap 

$f/m_1 \colon \Gamma_1', A_1,\Gamma_2, \dots, \Gamma_n, \Gamma_n, \Gamma_2' \to B$.
This multimap can then be factorised through $m_2$ and so on, until one gets a multimap 
$((f/m_1)/\dots)/m_n \colon \Gamma_1', A_1, \dots, A_n, \Gamma_2' \to B$ which can then be factored through $m$.
So we get a multimap $(((f/m_1)/\dots)/m_n)/m$ such that $f = (((f/m_1)/\dots)/m_n)/m \circ m\circ (m_1,\dots,m_n)$. 
The uniqueness of this multimap follows from the uniqueness of the factorisation through each universal multimap.
Indeed, suppose that we can factor $f$ through $h$, so $f = h \circ  m\circ (m_1,\dots,m_n)$.
Then, by the uniqueness of the factorisation through $m_1$, we have: $h \circ  m\circ (\id_{A_1},\dots,m_n)=(((f/m_1)/\dots)/m_n)/m \circ m\circ (\id_{A_1},\dots,m_n)$.
By repeating this argument for each $m_i$ we get that $h \circ m = (((f/m_1)/\dots)/m_n)/m \circ m$.
And then we conclude that $h = (((f/m_1)/\dots)/m_n)/m$ by the uniqueness of the factorisation through $m$.
\end{proof}

In the following we will often write $g(f_1,\dots, f_n)$ instead of $g \circ (f_1,\dots,f_n)$ and $gf$ instead of $g \circ f$.
We will also write $f/(m m')$ instead of $(f/m)/m'$.
Then we have that $(f/(m m'))m'm = (f/m)m = f$.
So we can reason as if we were treating products and quotient.

\begin{definition}
A \defin{representable} multicategory is a multicategory such that for any finite list of objects $\Gamma$ there is a universal map $m_\Gamma \colon \Gamma \to \otimes \Gamma$.
\end{definition}

Under the axiom of choice it is equivalent to ask for a specific choice of universal maps.

\begin{definition}
A \defin{strict} representable multicategory is a multicategory equipped with a choice of universal maps containing the identity and closed under composition.
\end{definition}

Closure under composition means that the multimap 

$m_{\otimes \Gamma_1, \dots, \otimes \Gamma_2} (m_{A_{1,1},\dots,A_{1,p_1}}, \dots, m_{A_{n,1},\dots,A_{n,p_n}}) \colon A_{1,1}, \dots, A_{n,p_n} \to \otimes_i \otimes_j A_{i,j}$ should be the universal multimap $m_{A_{1,1},\dots, A_{n,p_n}} \colon A_{1,1}, \dots, A_{n,p_n} \to \otimes_{i,j} A_{i,j}$.
In particular, we should have that our choice of tensor enforces $\otimes_i \otimes_j A_{i,j} = \otimes_{(i,j)} A_{i,j}$ so that the corresponding monoidal category is strict.
Hence the terminology.

\begin{prop}
Every representable multicategory is equivalent to a strict one.
\end{prop}
\begin{proof}
See \cite{Hermida2000}.
\end{proof}

In the following we will always take representable multicategories to be equipped with a particular choice of universal multimaps.
So we can use the white-dot/black-dot notation.
%Furthermore, the fact that universal maps are closed under composition means that the white-dot and black-dot operators form a pseudomonoid and pseudocomonoid respectively.
%It is a strict representable multicategory iff these are a monoid and a comonoid.
%More precisely, these form unbiased pseudomonoid pseudocomonoids, i.e. equipped with $n$-ary operations closed under composition up-to-isomorphism.
As always, we can take a biased definition with a nullary operation and a binary one.

The rules for the graphical calculus are given in the following figure:
\[\input{multicat-representability.tikz}\]

\begin{rk}
It is easy to see that this graphical representation is sound, i.e. that starting from a string diagram representing a multimap and applying the rewriting rules above gives a string diagram representing the same multimap up to canonical isomorphisms, i.e. only involving isomorphisms between the tensors.

Conversely, it should be possible to prove a full completeness result, i.e. that any string diagram with exactly one output represent a (unique) multimap.
There might be a slight caveat though when dealing with the unit, where it would be not clear where the unit is being eliminated.
If it were the case we could introduce formal links that link a unit to the multimap it is eliminated in, similarly to the linking of units in proof nets.

This is left to further work.
\end{rk}

Now consider two representable multicategories \cM and \cN and a functor $F \colon \cM \to \cN$.
Then we can built a unary multimap $F_n \colon \otimes F(A_i) \to F(\otimes A_i)$ in the following way:
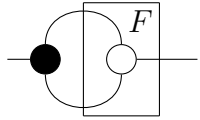
\[\input{multicat-laxfunctor.tikz}\]
where the part inside the box happens in \cM, is then turned into a multimap in \cN by $F$.
This can be translated by saying that the multimap is defined by \[F_n(A_1,\dots, A_n) := F(m_{A_1,\dots,A_n})/m_{F(A_1),\dots,F(A_n)}\]

\begin{definition}
A functor of representable multicategories is said to be \defin{strong} or \defin{pseudo} if the $F_n$ are invertible.
If the multicategories are strict and the $F_n$ are equalities then the functor is said to be \defin{strict}.
\end{definition}

Now let consider a transformation between functors of multicategories $\gamma \colon F \Rightarrow G$.
We will display the component of $\gamma$ as being in between the boxes representing $F$ and $G$.
So for example the naturality of $\gamma$ says that for any multimap $f \colon A_1, \dots, A_n \to B$ in \cM, the following holds in \cN:
\[\input{multicat-transformation-graph.tikz}\]

Now if both the functors are between representable multicategories, the following diagram commutes:
% https://q.uiver.app/?q=WzAsNCxbMCwwLCJGKEFfMSlcXG90aW1lcyBcXGRvdHMgXFxvdGltZXMgRihBX24pIl0sWzIsMCwiRihBXzEgXFxvdGltZXMgXFxkb3RzIFxcb3RpbWVzIEFfbikiXSxbMCwyLCJHKEFfMSkgXFxvdGltZXMgXFxkb3RzIFxcb3RpbWVzIEcoQV9uKSJdLFsyLDIsIkcoQV8xXFxvdGltZXMgXFxkb3RzIFxcb3RpbWVzIEFfbikiXSxbMCwxLCJGX24iLDFdLFswLDIsIlxcb3RpbWVzKFxcYWxwaGFfe0FfaX0pIiwxXSxbMiwzLCJHX24iLDFdLFsxLDMsIlxcYWxwaGFfe0FfMSBcXG90aW1lcyBcXGRvdHMgXFxvdGltZXMgQV9ufSIsMV1d
\[\begin{tikzcd}[ampersand replacement=\&]
	{F(A_1)\otimes \dots \otimes F(A_n)} \&\& {F(A_1 \otimes \dots \otimes A_n)} \\
	\\
	{G(A_1) \otimes \dots \otimes G(A_n)} \&\& {G(A_1\otimes \dots \otimes A_n)}
	\arrow["{F_n}"{description}, from=1-1, to=1-3]
	\arrow["{\otimes(\alpha_{A_i})}"{description}, from=1-1, to=3-1]
	\arrow["{G_n}"{description}, from=3-1, to=3-3]
	\arrow["{\alpha_{A_1 \otimes \dots \otimes A_n}}"{description}, from=1-3, to=3-3]
\end{tikzcd}\]
We can prove it graphically:
\[\scalebox{1}{
\input{multicat-transformation-monoidal.tikz}
}\]
where the equalities are to be read left to right, top to bottom.
The top left corner represent $G_n \circ \otimes(\alpha_{A_i})$.
Going from there to the top right corner is done by using cancellation of white and black dots.
Then, going to the bottom left corner is by functoriality of $G$.
Finally the last equality follows from naturality of the transformation.
The bottom right corner represent the multimap $\alpha_B \circ F_n$.

\subsection{Examples}

\begin{example}
We will see in the next section that the underlying multicategory of a monoidal category is always representable, and that in fact, any representable multicategory arises in this way.
\end{example}

\begin{example}
In particular, \Vect, \FVect, \Ban, \FBan, \Banc, and \FBanc{} are all representable.
The tensor product on the vector spaces is the usual tensor product.
The tensor product on the finite dimensional Banach spaces is the tensor product of the underlying vector spaces equipped with the so-called projective norm.
We will study this norm in more details in the chapter on polycategories.
For arbitrary Banach spaces, one has to take the completion of the tensor products of the vector spaces under the projective norm.
\end{example}

\begin{example}
The multicategory of modules over a ring is representable iff the ring is commutative.
\end{example}

\begin{example}
The terminal multicategory $\one$ is representable with $\otimes(\ast,\dots,\ast) = \ast$.
\end{example}

\begin{example}
For a representable symmetric multicategory \cM, $\Mon{\cM}$ is always representable.
Given monoids $(M_i,\mu_i^k)$, one can define a multiplication on $\otimes M_i$ by considering the following multimap
\[ M_1,\dots,M_l,\dots,M_1,\dots,M_l \simeq M_1, \dots, M_1, \dots, M_l,\dots, M_l \xrightarrow{\mu_1^k,\dots,\mu_l^k} M_1,\dots,M_l \xrightarrow{m_{M_i}} \otimes M_i\]
and factors it through $m_{M_i}$ $l$ times to get a multimap $\otimes M_i, \dots, \otimes M_i \to \otimes M_i$.
\end{example}

\begin{example}
The multicategory of actions over a symmetric multicategory $\cM$ is not representable.
\end{example}

It is possible to restrict the previous example to consider only actions over commutative monoids, where a commutative monoid can be defined as a monoid with an extra property or as an object in $\Mon{\Mon{\cM}}$.
Then, the category of actions over commutative monoids in \cM will be representable.

\begin{example}
\Act{} is not representable.
For example, there is no multimap $\star, \star \to \star$ so $\star \otimes \star$ cannot be defined.
\end{example}

\begin{example}
If \cM is a representable multicategory, then for a monoid $(M, \mu^k)$ in \cM, $\cM/M$ is representable where $\otimes(A_i,\varphi_{A_i})$ is defined by $\otimes A_i$ on objects and $(\mu^l(\varphi_{A_1},\dots,\varphi_{A_n}))/m_{A_i}$ on presheaves.
$m_{A_i} \colon A_1,\dots,A_l \to \otimes A_i$ defines a universal multimap by definition:
% https://q.uiver.app/?q=WzAsNCxbMCwwLCJBXzEsXFxkb3RzLEFfbCJdLFsyLDAsIkEiXSxbMiwyLCJNIl0sWzAsMiwiTSxcXGRvdHMsTSJdLFswLDEsIm1fe0FfaX0iLDFdLFsxLDIsIlxcbXVebChcXHZhcnBoaV97QV8xfSxcXGRvdHNcXHZhcnBoaV97QV9sfSkvbV97QV9pfSIsMV0sWzAsMywiXFx2YXJwaGlfe0FfMX0sXFxkb3RzLFxcdmFycGhpX3tBX2x9IiwxXSxbMywyLCJcXG11XmwiLDFdXQ==
\[\begin{tikzcd}[ampersand replacement=\&]
	{A_1,\dots,A_l} \&\& \otimes A_i \\
	\\
	{M,\dots,M} \&\& M
	\arrow["{m_{A_i}}"{description}, from=1-1, to=1-3]
	\arrow["{\mu^l(\varphi_{A_1},\dots\varphi_{A_l})/m_{A_i}}"{description}, from=1-3, to=3-3]
	\arrow["{\varphi_{A_1},\dots,\varphi_{A_l}}"{description}, from=1-1, to=3-1]
	\arrow["{\mu^l}"{description}, from=3-1, to=3-3]
\end{tikzcd}\]
\end{example}

In particular $\Cat/\Set$ is representable where $\otimes \varphi_i$ for presheaves $\varphi_i \colon A_i \to \Set$ is a presheaf on $A_1 \times \dots \times A_n$ defined pointwise.

\subsection{2-equivalence between representable multicategories and unbiased monoidal categories}

Now we are ready to prove the 2-equivalence between unbiased monoidal categories and representable multicategories.

We will write $\MultiCat_r$ for the 2-category of representable multicategories, functors and transformations.
It is a sub-2-category of $\MultiCat$.

First, let prove that $\cU \colon \UMonCat \to \MultiCat$ restricts to representable multicategories to give $\cU_r \colon \UMonCat \to \MultiCat_r$.

\begin{prop}
$\cU \colon \UMonCat \to \MultiCat$ restricts to representable multicategories to give $\cU_r \colon \UMonCat \to \MultiCat_r$.
\end{prop}
\begin{proof}
$\cU(M)$ is defined to be the multicategory with objects those of $M$ and multimaps, morphisms $f \colon \otimes_n A_i \to B$.
It is representable with $m_{\Gamma} := \id_{\otimes_n A_i} \colon \otimes_n A_i \to \otimes_n A_i$.
\end{proof}

Let \cM be a multicategory.
We will write $\cR(\cM)$ for the category with:
\begin{itemize}
\item objects, the objects of \cM
\item morphisms, the unary multimap of \cM
\end{itemize}
So it is the category that we get from forgetting about the multimaps of \cM.
The categorical structure follows directly from the multicategorical structure of \cM.

\begin{prop}
For a representable multicategory \cM, $\cR(\cM)$ is an unbiased multicategory with the monoidal structure given by:
\begin{itemize}
\item $\otimes_n A_n := \otimes A_n$ inherited from the tensor in \cM
\item for $f_i \colon A_i \to B_i$, $\otimes_n f_n := (m_{B_1, \dots, B_n} f)/m_{A_1,\dots,A_n}$, graphically:
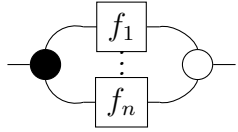
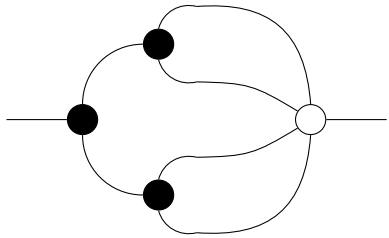
\[\input{multicat-tensor-map.tikz}\]
\item $\alpha_{n,m_i} \colon \otimes_n \otimes_{m_i} A_{i,j} \to \otimes_{\sum_i m_i} A_{i,j}$ is given by:
\[\input{multicat-tensor-alpha.tikz}\]
\item $\iota \colon A \to \otimes_1(A)$ is given by $m_A$.
\end{itemize}
\end{prop}
\begin{proof}

First let prove that the $\alpha$s and $\iota$s are natural isomorphisms.

Basically the idea is to reflect the diagram horizontally and exchange every black and white dots.
Then when composing all the black and white dots will cancel out to give an identity.

So the inverse of $\alpha$ is
\[\input{multicat-tensor-alpha-inverse.tikz}\]
and the inverse of $\iota$ is \begin{tikzpicture}
	\begin{pgfonlayer}{nodelayer}
		\node [style=black node] (0) at (0, 0) {};
		\node [style=none] (1) at (-1, 0) {};
		\node [style=none] (2) at (1, 0) {};
	\end{pgfonlayer}
	\begin{pgfonlayer}{edgelayer}
		\draw (0) to (1.center);
		\draw (2.center) to (0);
	\end{pgfonlayer}
\end{tikzpicture}
, i.e. the multimap $\id_{A}/m_{A}$.
It is easy to check that these are inverse.

Furthermore, $\alpha$ is natural.
For any $f_{i,j} \colon A_{i,j} \to  B_{i,j}$ represented by boxes in the next diagram, we have the following:
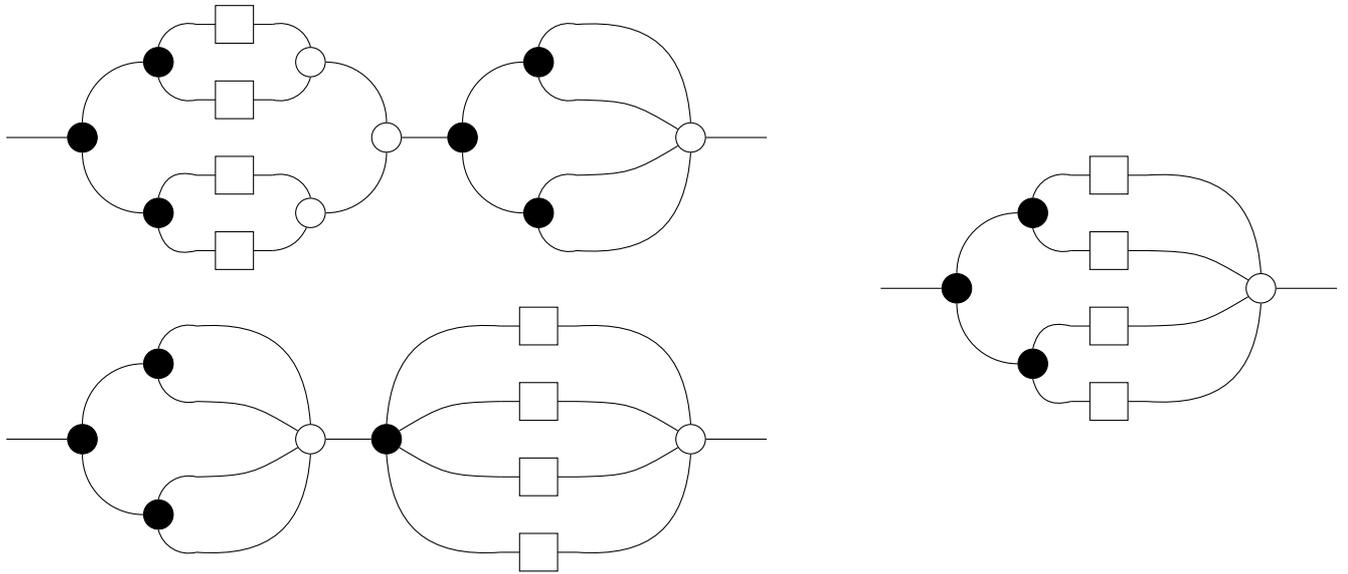
\[
\scalebox{1}{
\input{multicat-tensor-alpha-natural.tikz}
}
\]
The top and bottom left diagrams correspond to applying the $f$s followed by $\alpha$ and applying $\alpha$ followed by the $f$s respectively.
They both rewrite to the right one which can be read has: first eliminate all the tensors, then apply the $f$s and reintroduce the tensor.
Naturality of $\iota$ is similar, it says that for any $f \colon A \to B$, $\iota_B f = \otimes_1(f) \iota_A$.
But $\iota_B f = m_B f$ by definition and $\otimes_1(f) \iota_A = ((m_B f)/m_A) m_A = m_B f$.

Now, we want to prove that this defines a monoidal structure.
First, let prove that:

\scalebox{0.75}{
\begin{tikzcd}[ampersand replacement=\&]
	\& {\bigotimes_p(\bigotimes_{n_i}(\bigotimes_{m_{i,j}}(a_{i,j,k})_{1 \leq k \leq m_{i,j}})_{1 \leq j \leq n_i})_{1 \leq i \leq p}} \\
	{\bigotimes_{\sum_i n_i}(\bigotimes_{m_{i,j}}(a_{i,j,k})_{1 \leq k \leq m_{i,j}})_{1 \leq i \leq p,1 \leq j\leq n_i}} \&\& {\bigotimes_p(\bigotimes_{\sum_j}(a_{i,j,k})_{1 \leq j \leq n_i,1 \leq k \leq m_{i,j}})_{1 \leq i \leq p}} \\
	\& {\bigotimes_{\sum_i \sum_{j} m_{i,j}}(a_{i,j,k})_{1 \leq i \leq p, 1 \leq j \leq n_i, 1\leq k \leq m_{i,j}}}
	\arrow["{\alpha_{((\bigotimes_{m_{i,j}} (a_{i,j,k})_k)_j)_i}}"{description}, from=1-2, to=2-1]
	\arrow["{\alpha_{((a_{i,j,k})_k)_{(i,j)}}}"{description}, from=2-1, to=3-2]
	\arrow["{\otimes_p(\alpha_{((a_{i,j,k})_k)_j})_i}"{description}, from=1-2, to=2-3]
	\arrow["{\alpha_{((a_{i,j,k})_{(j,k)})_i}}"{description}, from=2-3, to=3-2]
\end{tikzcd}}

It is given by the following diagrams:
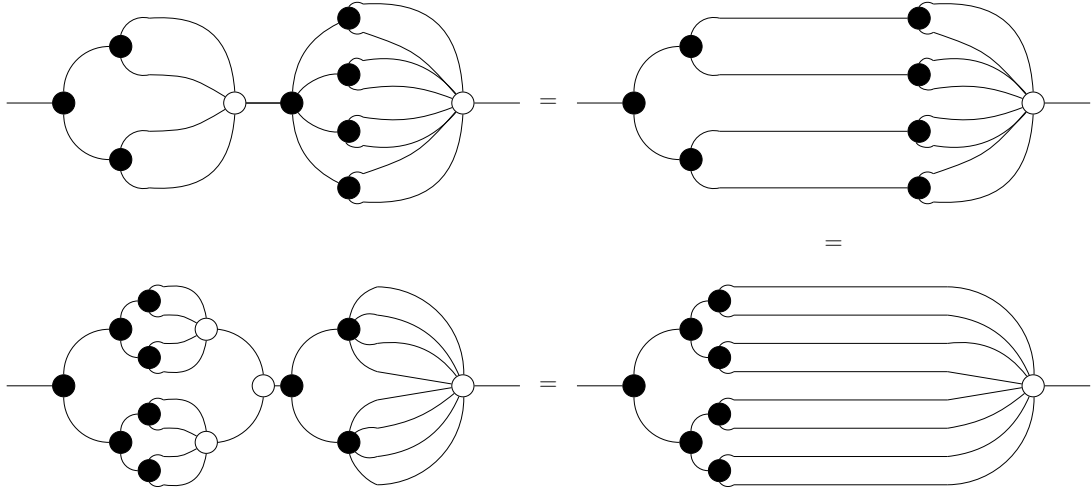
\[
\scalebox{0.75}{
\input{multicat-tensor-coherence.tikz}
}
\]

where the top left string diagram represent the multimaps going on the left of the commutative diagram and the bottom left string diagram to the multimaps going on the right.
Then, the equalities on top and bottom are canceling white and black dots while the one going from top to bottom is just translating a the black dots.

Then, let prove that
% https://q.uiver.app/?q=WzAsNixbMCwwLCJcXG90aW1lc19uKGFfaSlfezEgXFxsZXEgaSBcXGxlcSBufSJdLFsyLDAsIlxcb3RpbWVzX24oXFxvdGltZXNfMShhX2kpKV97MSBcXGxlcSBpIFxcbGVxIG59Il0sWzIsMSwiXFxvdGltZXNfbihhX2kpX3sxIFxcbGVxIGkgXFxsZXEgbn0iXSxbNCwwLCJcXG90aW1lc19uKGFfaSlfezEgXFxsZXEgaSBcXGxlcSBufSJdLFs2LDAsIlxcb3RpbWVzXzEoXFxvdGltZXNfbihhX2kpX3sxIFxcbGVxIGkgXFxsZXEgbn0pIl0sWzYsMSwiXFxvdGltZXNfbihhX2kpX3sxIFxcbGVxIGkgXFxsZXEgbn0iXSxbMCwxLCJcXG90aW1lc19uKFxcaW90YV97YV9pfSlfezEgXFxsZXEgaSBcXGxlcSBufSIsMV0sWzEsMiwiXFxhbHBoYV97KChhXzEpKV97MSBcXGxlcSBpIFxcbGVxIG59fSIsMV0sWzAsMiwiIiwxLHsibGV2ZWwiOjIsInN0eWxlIjp7ImhlYWQiOnsibmFtZSI6Im5vbmUifX19XSxbMyw0LCJcXGlvdGFfe1xcb3RpbWVzX24oYV9pKV97MSBcXGxlcSBpIFxcbGVzIG59fSIsMV0sWzQsNSwiXFxhbHBoYV97KGFfaSlfezEgXFxsZXEgaSBcXGxlcSBufX0iLDFdLFszLDUsIiIsMSx7ImxldmVsIjoyLCJzdHlsZSI6eyJoZWFkIjp7Im5hbWUiOiJub25lIn19fV1d
\[\begin{tikzcd}[ampersand replacement=\&]
	{\otimes_n(a_i)_{1 \leq i \leq n}} \&\& {\otimes_n(\otimes_1(a_i))_{1 \leq i \leq n}} \&\& {\otimes_n(a_i)_{1 \leq i \leq n}} \&\& {\otimes_1(\otimes_n(a_i)_{1 \leq i \leq n})} \\
	\&\& {\otimes_n(a_i)_{1 \leq i \leq n}} \&\&\&\& {\otimes_n(a_i)_{1 \leq i \leq n}}
	\arrow["{\otimes_n(\iota_{a_i})_{1 \leq i \leq n}}"{description}, from=1-1, to=1-3]
	\arrow["{\alpha_{((a_1))_{1 \leq i \leq n}}}"{description}, from=1-3, to=2-3]
	\arrow[Rightarrow, no head, from=1-1, to=2-3]
	\arrow["{\iota_{\otimes_n(a_i)_{1 \leq i \leq n}}}"{description}, from=1-5, to=1-7]
	\arrow["{\alpha_{(a_i)_{1 \leq i \leq n}}}"{description}, from=1-7, to=2-7]
	\arrow[Rightarrow, no head, from=1-5, to=2-7]
\end{tikzcd}\]

The proofs are given by:
\[
\scalebox{0.75}{
\input{multicat-tensor-unit-coherence.tikz}
}
\]

\end{proof}

It extends to a 2-functor.

\begin{prop}
There is a 2-functor $\cR \colon \MultiCat_r \to \UMonCat$ given by:
\begin{itemize}
\item $\cR(\cM)$ on multicategories
\item $\cR(F) := (F, F_n)$ where $F_n$ is \input{multicat-laxfunctor.tikz}
\item $\cR(\gamma) := \gamma$
\end{itemize}
\end{prop}
\begin{proof}
We have already proven above that $\cR(\cM)$ is an unbiased monoidal category and $\cR(\gamma)$ is a monoidal transformation.
All is left to prove is the coherence law for the functor $\cR(F)$.

First, let prove that
\[\begin{tikzcd}[ampersand replacement=\&]
	{\otimes_n(\otimes_{m_i}(F(A_{i,j})_{j \leq m_i})_{i \leq n}} \&\& {\otimes_{\sum_i m_i} F(A_{i,j})} \\
	\&\&\&\& {F(\otimes_{\sum_i m_i}(A_{i,j}))} \\
	{\otimes_n(F(\otimes_{m_i} (A_{i,j})))} \&\& {F(\otimes_n(\otimes_{m_i}(A_{i,j})))}
	\arrow["{\otimes_n(F_{m_i}(A_{i,j}))}"{description}, from=1-1, to=3-1]
	\arrow["{F_n(\otimes_{m_i}(A_{i,j}))}"{description}, from=3-1, to=3-3]
	\arrow["{F_n(\alpha_{(A_{i,j})})}"{description}, from=3-3, to=2-5]
	\arrow["{\alpha_{F(A_{i,j})}}"{description}, from=1-1, to=1-3]
	\arrow["{F_{\sum_i m_i}(A_{i,j})}"{description}, from=1-3, to=2-5]
\end{tikzcd}\]

It is given by the following rewriting where the top row represent the multimap going on top of the commutative diagram and the bottom one going on the bottom.
We use the usual white-black cancellation and functoriality of $F$.
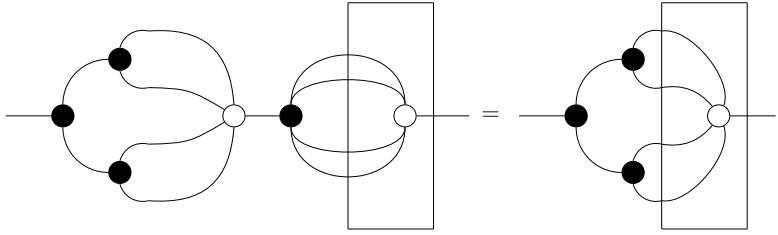
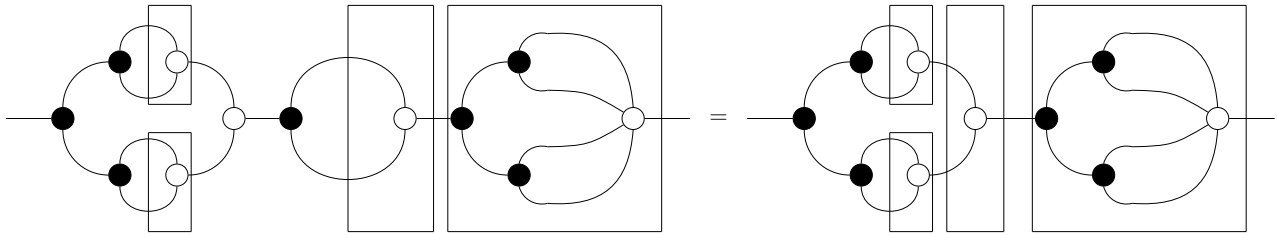
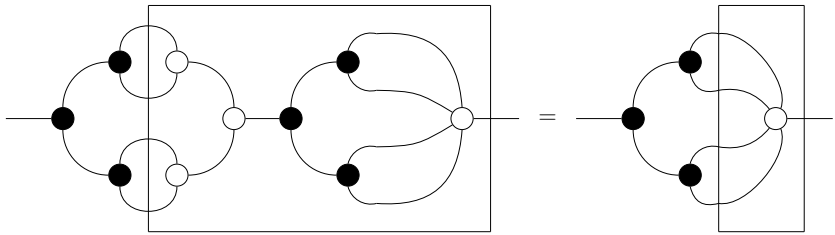
\[
\scalebox{0.75}{
\input{multicat-laxfunctor-coherence.tikz}
}
\]
Now let prove that

% https://q.uiver.app/?q=WzAsMyxbMCwwLCJGKEEpIl0sWzIsMCwiXFxvdGltZXNfMShGKEEpKSJdLFsyLDEsIkYoXFxvdGltZXNfMShBKSkiXSxbMCwxLCJcXGlvdGFfe0YoQSl9IiwxXSxbMSwyLCJGXzEoQSkiLDFdLFswLDIsIkYoXFxpb3RhX0EpIiwxXV0=
\[\begin{tikzcd}[ampersand replacement=\&]
	{F(A)} \&\& {\otimes_1(F(A))} \\
	\&\& {F(\otimes_1(A))}
	\arrow["{\iota_{F(A)}}"{description}, from=1-1, to=1-3]
	\arrow["{F_1(A)}"{description}, from=1-3, to=2-3]
	\arrow["{F(\iota_A)}"{description}, from=1-1, to=2-3]
\end{tikzcd}\]

It is given by the following rewriting:
\input{multicat-laxfunctor-unit.tikz}

Furthermore, we need to prove that $\cR$ is 2-functorial, which follows from the fact that it is the identity on functors and transformations.
The only non-trivial thing is that the $F_n$ match those of the identity and composite, which can be checked easily.

\end{proof}

\begin{prop}
There is an equivalence of 2-categories $\UMonCat \simeq \MultiCat_r$.
\end{prop}
\begin{proof}
Let prove that $\cU_r$ is a 2-equivalence.
We have already proven that \cU is 2-fully faithful.
Its restriction $\cU_r$ is also since we didn't ask for extra conditions on functors and transformations of representable multicategories.

Now let prove that it is essentially surjective on objects.
Given a representable multicategory \cM, we consider the unbiased monoidal category $\cR(\cM)$.
Then $\cU_r\cR(\cM)$ is given by:
\begin{itemize}
\item objects, those of \cM
\item multimaps $f \colon A_1,\dots,A_n \to B$, unary multimaps $f \colon \otimes(A_i) \to B$ in \cM
\item identities, $\iota_A^{-1} = \id_A/m_{A}$
\item composition, of $g \colon \otimes(B_i) \to C$ and $f_i \colon \otimes A_{i,j} \to B_i$ is given by:
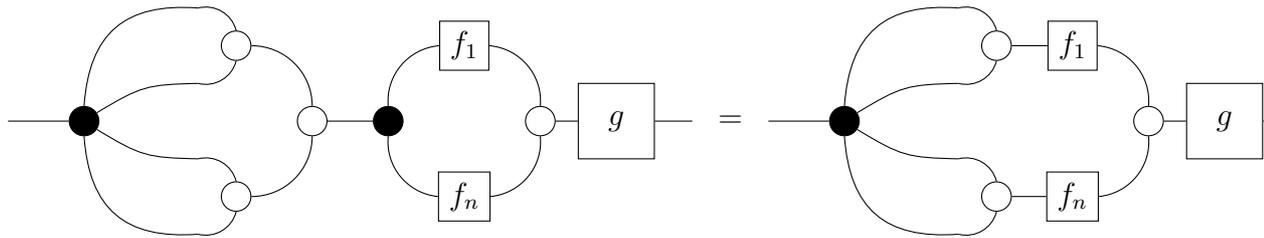
\[\input{multicat-ur-composition.tikz}\]
\end{itemize}

Now we have a functor of multicategories $-/m \colon \cM \to \cU_r\cR(\cM)$ that:
\begin{itemize}
\item is the identity on objects
\item sends $f \colon A_1, \dots, A_n \to B$ to $f/m_{(A_i)_i}$
\end{itemize}
It sends the identity to $\id_A/m_A$ which is the identity of $\cU_r\cR(\cM)$.
Given $g \colon B_i \to C$ and $f_i \colon A_{i,j} \to B_i$ we have that:
\[\input{multicat-ur-equiv-composition.tikz}\]
which is functoriality.

Furthermore, we have a functor $- \cdot m \colon \cU_r\cR(\cM) \to \cM$ that:
\begin{itemize}
\item is the identity on objects
\item sends $f \colon A_1 \otimes \dots \otimes A_n \to B$ to $f m_{(A_i)_i}$
\end{itemize}
One can check that it is functorial and that it is inverse to $-/m$.
So we have $\cM \simeq \cU_r\cR(\cM)$ and $\cU_r$ is essentially surjective on objects.
\end{proof}

\section{Closed Multicategories}

One could define a closed representable multicategory by asking for adjoints to the tensor product in a similar way to what is done with monoidal categories.
But instead, it is possible to introduce a notion of internal hom in a multicategory that does not assume the existence of tensors.
It is defined by a universal property.
In a monoidal category $M$, $A \multimap B$ is defined by the existence of a natural isomorphism $M(A \otimes B, C) \to M(A, B \multimap C)$ but in a multicategory \cM we could instead use bilinear maps $\cM(A,B; C) \simeq \cM(A; B \multimap C)$.
Of course, we can define an $n$-ary notion directly.
We will see that it amounts to the existence of a multimap having a universal property.

\subsection{Definition}

\begin{definition}
A multimap $g \colon \Gamma_1', A, \Gamma_2' \to B$ is said to be \defin{universal in A}, written $g \colon \Gamma_1', \focin{A}, \Gamma_2' \to B$ if for any multimap $f \colon \Gamma_1', \Gamma, \Gamma_2' \to B$ there is a unique multimap $g\backslash f \colon \Gamma \to A$ such that $f = g \circ g\backslash f$. 
\end{definition}

This unique factorisation property characterises the input $A$ uniquely up-to unique invertible multimap.
The proof is similar to the one for universality in the output.

We will write $\minthom{\Gamma_1}{A}{\Gamma_2}$ for the part of the domain where a multimap is universal: $\ev_{\Gamma_1;\Gamma_2} \colon \Gamma_1, \focin{\minthom{\Gamma_1}{A}{\Gamma_2}}, \Gamma_2 \to A$.
We will denote a factorisation of $f$ by 

$\lambda_{\Gamma_1;\Gamma_2}. f := \ev_{\Gamma_1;\Gamma_2}\backslash f$.
We will often not write the indexes.
So we have $\ev(\lambda.f) = f$.
We will also write $\Gamma_1 \multimap A$ and $A \multimapinv \Gamma_2$ when one of the contexts is empty.

\begin{definition}
A multicategory is said to be \defin{closed} if for any finite lists of objects $\Gamma_1, \Gamma_2$ and any object $A$ there is universal multimap $\ev_{\Gamma_1;\Gamma_2} \colon \Gamma_1, \focin{\minthom{\Gamma_1}{A}{\Gamma_2}}, \Gamma_2 \to A$.
\end{definition}

\begin{definition}
A \defin{birepresentable multicategory} is a multicategory that is both representable and closed.
\end{definition}

We will also talk about representable closed multicategories.
They have all universal multimaps, both in all of the inputs and in the output.

We will also use a graphical representation for internal homs (universal inputs) and evaluation multimaps (universal multimaps).
The evaluation multimap will be represented by:
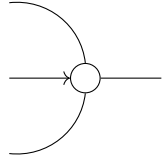
\[\input{multicat-evaluation.tikz}\]
where the tipped arrow is the input in which the multimap is universal.
This can be thought as an elimination rule for the internal hom.
Then we will introduce a dual introduction rule:
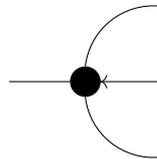
\[\input{multicat-hom-introduction.tikz}\]
that will let us define $\lambda.f$ by:
\[\input{multicat-lambda.tikz}\]
The factorisation property states that:
\[\scalebox{0.75}{\input{multicat-hom-universal.tikz}}\]
Furthermore, universal multimaps compose through their universal inputs, in particular $\minthom{\Gamma_2}{\minthom{\Gamma_1}{A}{\Gamma_4}}{\Gamma_3} \simeq \minthom{\Gamma_1,\Gamma_2}{A}{\Gamma_3,\Gamma_4}$.
So we get the following equations:
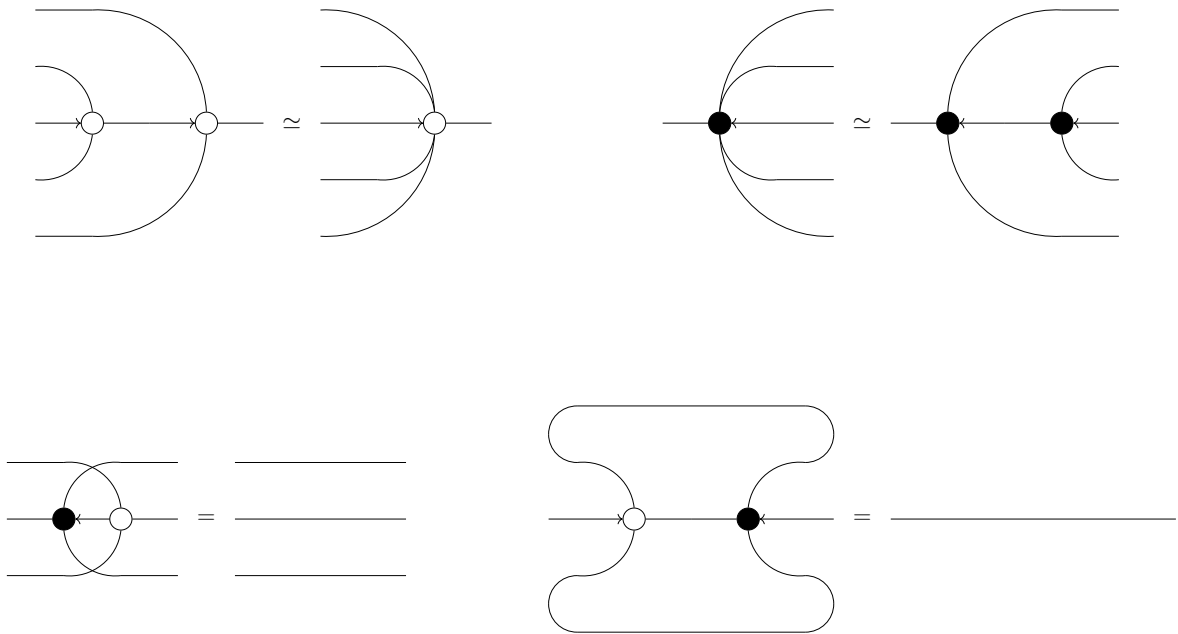
\[\scalebox{0.75}{\input{multicat-hom-equations.tikz}}\]

Now let $F \colon \cM \to \cN$ be a functor of closed multicategories.
Then we have a family of maps $F(\minthom{\Gamma_1}{A}{\Gamma_2}) \to \minthom{F(\Gamma_1)}{F(A)}{F(\Gamma_2)}$ given by:
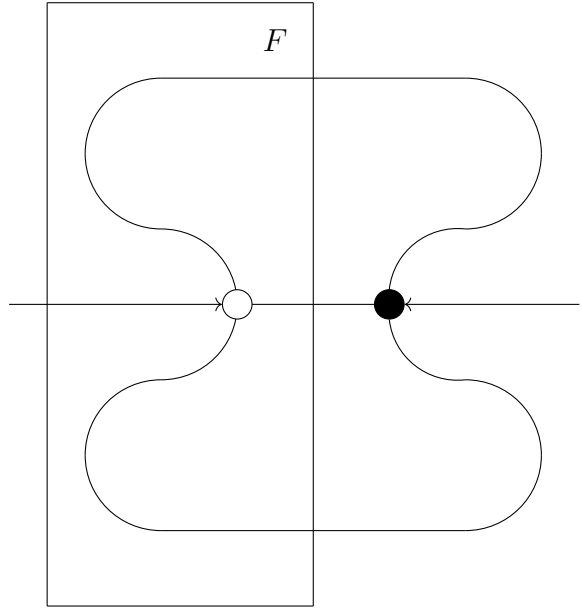
\[\input{multicat-hom-functor.tikz}\]

Furthermore these are subject to the following coherence law:

\[\scalebox{0.5}{
% https://q.uiver.app/?q=WzAsNSxbMSwwLCJGKFxcbWludGhvbXtcXEdhbW1hXzEsXFxHYW1tYV8yfXtBfXtcXEdhbW1hXzMsXFxHYW1tYV80fSkiXSxbMCwxLCJGKFxcbWludGhvbXtcXEdhbW1hXzJ9e1xcbWludGhvbXtcXEdhbW1hXzF9e0F9e1xcR2FtbWFfNH19e1xcR2FtbWFfM30pIl0sWzAsMiwiXFxtaW50aG9te0YoXFxHYW1tYV8yKX17RihcXG1pbnRob217XFxHYW1tYV8xfXtBfXtcXEdhbW1hXzR9KX17RihcXEdhbW1hXzMpfSJdLFsxLDMsIlxcbWludGhvbXtGKFxcR2FtbWFfMil9e1xcbWludGhvbXtGKFxcR2FtbWFfMSl9e0YoQSl9e1xcRihcXEdhbW1hXzQpfX17RihcXEdhbW1hXzMpfSJdLFsyLDEsIlxcbWludGhvbXtGKFxcR2FtbWFfMSksRihcXEdhbW1hXzIpfXtGKEEpfXtGKFxcR2FtbWFfMyksRihcXEdhbW1hXzQpfSJdLFswLDFdLFsxLDJdLFsyLDNdLFswLDRdLFs0LDNdXQ==
\begin{tikzcd}[ampersand replacement=\&]
	\& {F(\minthom{\Gamma_1,\Gamma_2}{A}{\Gamma_3,\Gamma_4})} \\
	{F(\minthom{\Gamma_2}{\minthom{\Gamma_1}{A}{\Gamma_4}}{\Gamma_3})} \&\& {\minthom{F(\Gamma_1),F(\Gamma_2)}{F(A)}{F(\Gamma_3),F(\Gamma_4)}} \\
	{\minthom{F(\Gamma_2)}{F(\minthom{\Gamma_1}{A}{\Gamma_4})}{F(\Gamma_3)}} \\
	\& {\minthom{F(\Gamma_2)}{\minthom{F(\Gamma_1)}{F(A)}{F(\Gamma_4)}}{F(\Gamma_3)}}
	\arrow[from=1-2, to=2-1]
	\arrow[from=2-1, to=3-1]
	\arrow[from=3-1, to=4-2]
	\arrow[from=1-2, to=2-3]
	\arrow[from=2-3, to=4-2]
\end{tikzcd}
}\]

With the proof given by the following rewriting:
\[
\scalebox{0.5}{\input{multicat-hom-functor-coherence.tikz}}
\]
\begin{prop}
The correspondence between representable multicategories and unbiased monoidal categories restrict to a correspondence between birepresentable multicategories and unbiased monoidal biclosed categories.
\end{prop}
\begin{proof}
We need to extend $\minthom{\Gamma_1}{-}{\Gamma_2}$ to a functor in $\cR(\cM)$ and prove that it is a right adjoint to $\otimes(\Gamma_1,A,\Gamma_2)$.

Given a morphism $f \colon A \to B$ in $\cR(\cM)$, i.e. a unary multimap in \cM, we define $\minthom{\Gamma_1}{f}{\Gamma_2} \colon \minthom{\Gamma_1}{A}{\Gamma_2} \to \minthom{\Gamma_1}{B}{\Gamma_2}$ by:
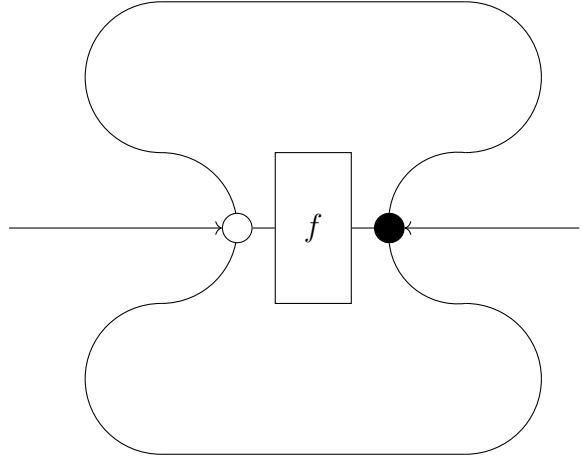
\[\input{multicat-hom-f.tikz}\]
It can be checked that this assignment is functorial.

Now let prove that there is a natural isomorphism \[\cR(\cM)(\otimes(\Gamma_1,A,\Gamma_2),B) \simeq \cR(\cM)(A,\minthom{\Gamma_1}{B}{\Gamma_2})\]
The function from left to right acts as follows:
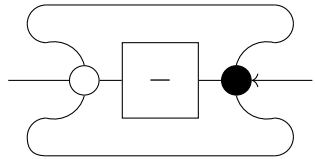
\[\input{multicat-tensor-hom.tikz}\]
It is given by $\lambda_{\Gamma_1;\Gamma_2}.(- m_{\Gamma_1,A,\Gamma_2})$.
In the other direction the function is given by:
\[\input{multicat-hom-tensor.tikz}\]
i.e. $(\ev_{\Gamma_1;\Gamma_2} -)/m_{\Gamma_1,A,\Gamma_2}$.
The proof that it is an isomorphism is given in one direction by:
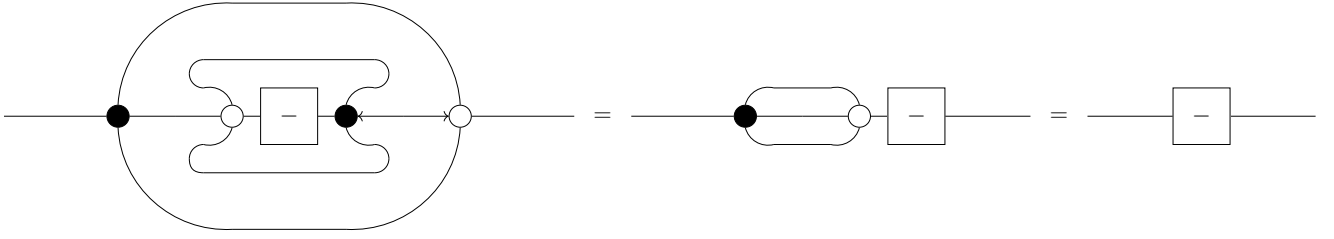
\[\scalebox{0.75}{\input{multicat-tensor-hom-tensor.tikz}}\]
and in the other by:
\[\lambda_{\Gamma_1;\Gamma_2}.((\ev_{\Gamma_1;\Gamma_2} -/m_{\Gamma_1,A,\Gamma_2}) m_{\Gamma_1,A,\Gamma_2}) = \lambda_{\Gamma_1;\Gamma_2}.(\ev_{\Gamma_1;\Gamma_2}-) = -\]

Furthermore, it can be check that those are natural in $A$ and $B$.
\end{proof}

\begin{rk}
In the literature people will prefer to call biclosed what we called closed here.
As far as we are aware, only notions of closed categories, i.e. left or right closed, that do not assume the existence of a monoidal structure appear explicitly in the literature.
These have been developed by Eilenberg and Kelly in \cite{EilenbergKelly1966}.
In \cite{Manzyuk2009} it has been proven that it is equivalent to a closed multicategory with a unit (i.e. a nullary tensor product).
Then, one can also define a notion of closed functor that corresponds to the notion of functor of closed multicategories.
They are equipped with a map $F(\minthom{\Gamma_1}{A}{\Gamma_2}) \to \minthom{F(\Gamma_1)}{F(A)}{F(\Gamma_2)}$ with some coherence laws.

A monoidal functor between monoidal biclosed categories is automatically closed.
Giving a morphism $F(\minthom{\Gamma_1}{A}{\Gamma_2}) \to \minthom{F(\Gamma_1)}{F(A)}{F(\Gamma_2)}$ is equivalent to giving a morphism $F(\Gamma_1) \otimes F(\minthom{\Gamma_1}{A}{\Gamma_2}) \otimes F(\Gamma_2) \to F(A)$ by the definition of an adjunction.
Since $F$ is monoidal we have a morphim \[F(\Gamma_1) \otimes F(\minthom{\Gamma_1}{A}{\Gamma_2}) \otimes F(\Gamma_2) \to F(\Gamma_1  \otimes \minthom{\Gamma_1}{A}{\Gamma_2} \otimes \Gamma_2)\] and another one \[F(\Gamma_1  \otimes \minthom{\Gamma_1}{A}{\Gamma_2} \otimes \Gamma_2) \to F(A)\] by $F(\ev)$ where $\ev$ is obtained by considering the identity $\id_{\Gamma_1 \otimes A \otimes \Gamma_2}$ and using the property of $\minthom{\Gamma_1}{-}{\Gamma_2}$ being a right adjoint.
\end{rk}

\begin{rk}
The correspondence between representable closed multicategories and monoidal biclosed categories extend to a 2-equivalence.
\end{rk}

\chapter{Polycategories and $\ast$-autonomous~categories}

The equivalence between representable closed multicategories and monoidal biclosed categories and its extension to the symmetric case gives two approaches to modeling multiplicative intuitionistic linear logic and its non-commutative counterpart.
In this section, we will see how this can be extended to classical multiplicative linear logic by a 2-equivalence between birepresentable polycategories and $\ast$-autonomous categories.

One way of stating the difference between intuitionistic and classic logic is in the latter the double negation elimination holds whereas it doesn't in the former.
This leads to one definition of $\ast$-autonomous category as a monoidal biclosed category with a fixed object $\bot$ inducing functors $\ldual{(-)} := - \multimap \bot$ and $\rdual{(-)} := \bot \multimapinv -$ such that the canonical morphisms $A \to \ldual{(\rdual{A})}$ and $A \to \rdual{(\ldual{A})}$ obtained by currying the evaluation maps are invertible.
There are other equivalent definitions of $\ast$-autonomous categories that are more suitable to relate to polycategories.
We will first recall all those definitions and prove that they are equivalent.

Another approach in comparing intuitionistic and classical logic is through their sequent calculi.
Gentzen introduced a sequent calculus for classical logic with sequents of the form $\Gamma \vdash \Delta$ where $\Gamma$ and $\Delta$ are finite lists of formulae.
He then showed that intuitionistic logic can be recovered by forcing the list on the right hand of the sequent to contain exactly one formula.
Polycategories model this approach.
They are similar to multicategories but their morphisms have multiple inputs and multiple outputs.
Furthermore, the composition in polycategories can only be performed along one object, mimicking the cut rule in sequent calculus:

\AXC{$\Gamma \xrightarrow{f} \Delta_1, A, \Delta_2$}
\AXC{$\Gamma_1', A, \Gamma_2' \xrightarrow{g} \Delta'$}
\BIC{$\Gamma_1', \Gamma, \Gamma_2' \xrightarrow{g \circ f} \Delta_1, \Delta', \Delta_2$}
\DP

We will define polycategories and how the interpretation of the connectives are given by the existence of objects with universal properties while they are modeled by structures in $\ast$-autonomous categories.

Then we will prove the equivalence between $\ast$-autonomous categories and those polycategories with all connectives.

There is a long-standing tradition of using $\ast$-autonomous categories for models of MLL.
It goes back to Barr's work (\cite{Barr1979,Barr1991,Barr1995}.
A detailed account is given in Mellies's textbook \cite{Mellies2009}.
Polycategories have been introduced with applications to sequent calculus in mind by Szabo in \cite{Szabo1975}.
The connection to $\ast$-autonomous categories via the introduction of linearly distributive categories and two-tensor polycategories has been developed in a series of papers by Cockett, Seely and collaborators \cite{BluteCockettSeelyTrimble1993,CockettSeely1997,CockettSeely1999}.
This section consists of a review of this background material together with some original works, namely (lax) unbiased linearly distributive categories, weak two-tensor polycategories, birepresentable polycategories and the polycategory of Banach spaces and contractive polylinear maps.

\section{$\ast$-autonomous categories}
In this section we present several definitions of $\ast$-autonomous categories and prove them equivalent.
Our $\ast$-autonomous categories will be assumed to be non-symmetric in general which differs from the usual definitions.

\subsection{As a monoidal biclosed category with a dualising object}

Since monoidal biclosed categories model (non-commutative) multiplicative intuitionistic linear logic and $\ast$-autonomous categories ought to model (non-commutative) multiplicative classical linear logic, one might hope that understanding how to go from intuitionistic to classical logic will help in defining $\ast$-autonomous categories.
It is indeed the case.
In this section we will see how to define a $\ast$-autonomous category by asking for a double negation elimination principle.

The negations can be defined by introducing an object $\bot$ that will act as a multiplicative falsehood and define left and right negation by $\rdual{(-)}:= \bot \multimapinv -$ and $\ldual{(-)} := - \multimap \bot$.

\begin{definition}
In a monoidal biclosed category $M$, a \defin{dualising object} $\bot$ is an object such that the two morphisms 
$\delta_A := \lambda_{A \multimap \bot}^r. \ev_{A}^l \colon A \longrightarrow \bot \multimapinv (A \multimap \bot)$
and
$\delta_A' := \lambda_{\bot \multimapinv A}^l. \ev_{A}^r \colon A \longrightarrow (\bot \multimapinv A) \multimap \bot$
are invertible for any $A$.
\end{definition}

With the definition for negations given above this means that $\ldual{(\rdual{A})} \simeq A \simeq \rdual{(\ldual{A})}$

\begin{definition}
A \defin{$\ast$-autonomous category} is a monoidal biclosed category equipped with a dualising object $\bot$.
\end{definition}

\begin{prop}
In a $\ast$-autonomous category we have that $\ldual{I} \simeq \bot \simeq \rdual{I}$
\end{prop}
\begin{proof}
We use the isomorphisms $I \otimes \bot \simeq \bot$ and $\bot \otimes I \simeq \bot$ and then we curry those.
\end{proof}

\begin{rk}
This definition works for both a biased and an unbiased notion of monoidal category.
In both case the dualising object only needs the possibility to curry in one object.
\end{rk}

In a $\ast$-autonomous category, one can define another monoidal structure corresponding to the multiplicative disjunction.
It is defined by a de Morgan duality.

We will call a $\ast$-autonomous category biased/unbiased if the monoidal category is.

\begin{prop}
In a $\ast$-autonomous category, $\ldual{(-)}$ and $\rdual{(-)}$ extend to functors \[\ldual{(-)},\rdual{(-)} \colon \cC^\op \to \cC\] making \cC self-dual.
\end{prop}
\begin{proof}
Given $f \colon A \to B$, $\ldual{f} \colon \ldual{B} \to \ldual{A}$ is defined by currying \[A \otimes (B \multimap \bot) \xrightarrow[]{A \otimes (f \multimap \bot)} A \otimes (A \multimap \bot) \xrightarrow{ev_A^l} \bot\]
It is functorial and it defines a duality thanks to the definition of a dualising object.
\end{proof}

\begin{definition}
For an unbiased $\ast$-autonomous category $(M, \otimes_n, \alpha, \iota, \bot)$, we define:
\begin{itemize}
\item $\parr_n(-) := \rdual{(\otimes_n\ldual{(-)})}$
\item $\alpha_{n,(m_k)_{1 \leq k \leq n}}'(a_{i,j})  \colon  \parr_n(\parr_{m_k}(a_{i,j})) \to \parr_{\sum_k m_k}(a_{i,j})$ by
\begin{align*}
\parr_n(\parr_{m_k}(a_{i,j})) = \rdual{(\otimes_n(\ldual{(\rdual{(\otimes_{m_k}(\ldual{a_{i,j}}))}))})} &\xrightarrow{\rdual{(\otimes_n((\delta'_{\otimes_{m_k}(\ldual{a_{i,j}})})^{-1}))}} \rdual{(\otimes_n\otimes_{m_k}(\ldual{a_{i,j}}))}\\
&\xrightarrow{\rdual{(\alpha_{n,(m_k)_{1\leq k \leq n}}(\ldual{a_{i,j}})^{-1})}} \rdual{\otimes_{\sum_k m_k}(\ldual{a_{i,j}})} = \parr_{\sum_k m_k}(a_{i,j})
\end{align*}
\item $\iota_A' \colon A \xrightarrow{\delta_A'} \rdual{(\ldual{A})} \xrightarrow{\rdual{\iota_{\ldual{A}}}} \rdual{(\otimes_1 \ldual{A})} = \parr_1 A$
\end{itemize}
\end{definition} 

\begin{prop}
For an unbiased $\ast$-autonomous category $(M, \otimes_n, \alpha, \iota, \bot)$, $(M, \parr_n, \alpha', \iota')$ defines another unbiased monoidal structure on $M$.
\end{prop}
\begin{proof}
Let us first prove that $\alpha'$ and $\iota'$ are natural transformations.

\resizebox{\hsize}{!}{% https://q.uiver.app/?q=WzAsMTAsWzEsMSwiXFxyZHVhbHsoXFxvdGltZXNfblxcbGR1YWx7XFxyZHVhbHsoXFxvdGltZXNfe21fa30oXFxsZHVhbHthX3tpLGp9fSkpfX0pfSJdLFs5LDEsIlxccmR1YWx7KFxcb3RpbWVzX3tcXHN1bV9rIG1fa30oXFxsZHVhbHthX3tpLGp9fSkpfSJdLFsxLDUsIlxccmR1YWx7KFxcb3RpbWVzX25cXGxkdWFse1xccmR1YWx7KFxcb3RpbWVzX3ttX2t9KFxcbGR1YWx7Yl97aSxqfX0pKX19KX0iXSxbOSw1LCJcXHJkdWFseyhcXG90aW1lc197XFxzdW1fayBtX2t9KFxcbGR1YWx7Yl97aSxqfX0pKX0iXSxbNSwxLCJcXHJkdWFseyhcXG90aW1lc19uXFxvdGltZXNfe21fa30oXFxsZHVhbHthX3tpLGp9fSkpfSJdLFs1LDUsIlxccmR1YWx7KFxcb3RpbWVzX25cXG90aW1lc197bV9rfShcXGxkdWFse2Jfe2ksan19KSl9Il0sWzAsMCwiXFxwYXJyX25cXHBhcnJfe21fa30oYV97aSxqfSkiXSxbMTAsMCwiXFxwYXJyX3tcXHN1bV9rIG1fa30oYV97aSxqfSkiXSxbMCw2LCJcXHBhcnJfblxccGFycl97bV9rfShiX3tpLGp9KSJdLFsxMCw2LCJcXHBhcnJfe1xcc3VtX2sgbV9rfShiX3tpLGp9KSJdLFswLDIsIlxccmR1YWx7KFxcb3RpbWVzX25cXGxkdWFse1xccmR1YWx7KFxcb3RpbWVzX3ttX2t9KFxcbGR1YWx7Zl97aSxqfX0pKX19KX0iLDFdLFsxLDMsIlxccmR1YWx7KFxcb3RpbWVzX3tcXHN1bV9rIG1fa30oXFxsZHVhbHtmX3tpLGp9fSkpfSIsMV0sWzAsNCwiXFxyZHVhbHsoXFxvdGltZXNfbiAoXFxkZWx0YSdfe1xcb3RpbWVzX3ttX2t9KFxcbGR1YWx7YV97aSxqfX0pfSleey0xfSl9IiwxXSxbNCwxLCJcXHJkdWFseyhcXGFscGhhX3tuLG1fa30oXFxsZHVhbHthX3tpLGp9fSkpfSIsMV0sWzIsNSwiXFxyZHVhbHsoXFxvdGltZXNfbihcXGRlbHRhJ197XFxvdGltZXNfe21fa30oXFxsZHVhbHtiX3tpLGp9fSl9KV57LTF9KX0iLDFdLFs1LDMsIlxccmR1YWx7KFxcYWxwaGFfe24sbV9rfShcXGxkdWFse2Jfe2ksan19KSl9IiwxXSxbNCw1LCJcXHJkdWFseyhcXG90aW1lc19uXFxvdGltZXNfe21fa30oXFxsZHVhbHtmX3tpLGp9fSkpfSIsMV0sWzYsMCwiIiwxLHsibGV2ZWwiOjIsInN0eWxlIjp7ImhlYWQiOnsibmFtZSI6Im5vbmUifX19XSxbNywxLCIiLDEseyJsZXZlbCI6Miwic3R5bGUiOnsiaGVhZCI6eyJuYW1lIjoibm9uZSJ9fX1dLFs2LDcsIlxcYWxwaGFfe24sbV9rfScoYV97aSxqfSkiLDFdLFs4LDIsIiIsMSx7ImxldmVsIjoyLCJzdHlsZSI6eyJoZWFkIjp7Im5hbWUiOiJub25lIn19fV0sWzYsOCwiXFxwYXJyX25cXHBhcnJfe21fa30oZl97aSxqfSkiLDFdLFs5LDMsIiIsMSx7ImxldmVsIjoyLCJzdHlsZSI6eyJoZWFkIjp7Im5hbWUiOiJub25lIn19fV0sWzgsOSwiXFxhbHBoYV97bixtX2t9JyhiX3tpLGp9KSIsMV0sWzcsOSwiXFxwYXJyX3tcXHN1bV9rIG1fa30oZl97aSxqfSkiLDFdXQ==
\begin{tikzcd}[ampersand replacement=\&]
	{\parr_n\parr_{m_k}(a_{i,j})} \&\&\&\&\&\&\&\&\&\& {\parr_{\sum_k m_k}(a_{i,j})} \\
	\& {\rdual{(\otimes_n\ldual{(\rdual{(\otimes_{m_k}(\ldual{a_{i,j}}))})})}} \&\&\&\& {\rdual{(\otimes_n\otimes_{m_k}(\ldual{a_{i,j}}))}} \&\&\&\& {\rdual{(\otimes_{\sum_k m_k}(\ldual{a_{i,j}}))}} \\
	\\
	\\
	\\
	\& {\rdual{(\otimes_n\ldual{(\rdual{(\otimes_{m_k}(\ldual{b_{i,j}}))})})}} \&\&\&\& {\rdual{(\otimes_n\otimes_{m_k}(\ldual{b_{i,j}}))}} \&\&\&\& {\rdual{(\otimes_{\sum_k m_k}(\ldual{b_{i,j}}))}} \\
	{\parr_n\parr_{m_k}(b_{i,j})} \&\&\&\&\&\&\&\&\&\& {\parr_{\sum_k m_k}(b_{i,j})}
	\arrow["{\rdual{(\otimes_n\ldual{(\rdual{(\otimes_{m_k}(\ldual{f_{i,j}}))})})}}"{description}, from=2-2, to=6-2]
	\arrow["{\rdual{(\otimes_{\sum_k m_k}(\ldual{f_{i,j}}))}}"{description}, from=2-10, to=6-10]
	\arrow["{\rdual{(\otimes_n (\delta'_{\otimes_{m_k}(\ldual{a_{i,j}})})^{-1})}}"{description}, from=2-2, to=2-6]
	\arrow["{\rdual{(\alpha_{n,m_k}(\ldual{a_{i,j}}))}}"{description}, from=2-6, to=2-10]
	\arrow["{\rdual{(\otimes_n(\delta'_{\otimes_{m_k}(\ldual{b_{i,j}})})^{-1})}}"{description}, from=6-2, to=6-6]
	\arrow["{\rdual{(\alpha_{n,m_k}(\ldual{b_{i,j}}))}}"{description}, from=6-6, to=6-10]
	\arrow["{\rdual{(\otimes_n\otimes_{m_k}(\ldual{f_{i,j}}))}}"{description}, from=2-6, to=6-6]
	\arrow[Rightarrow, no head, from=1-1, to=2-2]
	\arrow[Rightarrow, no head, from=1-11, to=2-10]
	\arrow["{\alpha_{n,m_k}'(a_{i,j})}"{description}, from=1-1, to=1-11]
	\arrow[Rightarrow, no head, from=7-1, to=6-2]
	\arrow["{\parr_n\parr_{m_k}(f_{i,j})}"{description}, from=1-1, to=7-1]
	\arrow[Rightarrow, no head, from=7-11, to=6-10]
	\arrow["{\alpha_{n,m_k}'(b_{i,j})}"{description}, from=7-1, to=7-11]
	\arrow["{\parr_{\sum_k m_k}(f_{i,j})}"{description}, from=1-11, to=7-11]
\end{tikzcd}}

The outer diagram states naturality of $\alpha'$.
The left square commutes by naturality of $\delta'^{-1}$ while the right one commutes by naturality of $\alpha$.

% https://q.uiver.app/?q=WzAsMTAsWzEsMSwiQSJdLFszLDEsIlxccmR1YWx7KFxcbGR1YWx7QX0pfSJdLFs1LDEsIlxccmR1YWx7KFxcb3RpbWVzXzFcXGxkdWFse0F9KX0iXSxbMSwzLCJCIl0sWzMsMywiXFxyZHVhbHsoXFxsZHVhbHtCfSl9Il0sWzUsMywiXFxyZHVhbHsoXFxvdGltZXNfMVxcbGR1YWx7Qn0pfSJdLFswLDAsIkEiXSxbMCw0LCJCIl0sWzYsMCwiXFxwYXJyXzEgQSJdLFs2LDQsIlxccGFycl8xIEIiXSxbMCwxLCJcXGRlbHRhX0EnIiwxXSxbMSwyLCJcXHJkdWFseyhcXGlvdGFfe1xcbGR1YWx7QX19KX0iLDFdLFswLDMsImYiLDFdLFszLDQsIlxcZGVsdGFfQiciLDFdLFs0LDUsIlxccmR1YWx7KFxcaW90YV97XFxsZHVhbHtCfX0pfSIsMV0sWzIsNSwiXFxyZHVhbHsoXFxvdGltZXNfMSBcXGxkdWFse2Z9KX0iLDFdLFsxLDQsIlxccmR1YWx7KFxcbGR1YWx7Zn0pfSIsMV0sWzYsMCwiIiwxLHsibGV2ZWwiOjIsInN0eWxlIjp7ImhlYWQiOnsibmFtZSI6Im5vbmUifX19XSxbNiw4LCJcXGlvdGFfQSciLDFdLFs4LDksIlxccGFycl8xIGYiLDFdLFs3LDksIlxcaW90YV9CJyIsMV0sWzYsNywiZiIsMV0sWzcsMywiIiwxLHsibGV2ZWwiOjIsInN0eWxlIjp7ImhlYWQiOnsibmFtZSI6Im5vbmUifX19XSxbOSw1LCIiLDEseyJsZXZlbCI6Miwic3R5bGUiOnsiaGVhZCI6eyJuYW1lIjoibm9uZSJ9fX1dLFs4LDIsIiIsMSx7ImxldmVsIjoyLCJzdHlsZSI6eyJoZWFkIjp7Im5hbWUiOiJub25lIn19fV1d
\[\begin{tikzcd}[ampersand replacement=\&]
	A \&\&\&\&\&\& {\parr_1 A} \\
	\& A \&\& {\rdual{(\ldual{A})}} \&\& {\rdual{(\otimes_1\ldual{A})}} \\
	\\
	\& B \&\& {\rdual{(\ldual{B})}} \&\& {\rdual{(\otimes_1\ldual{B})}} \\
	B \&\&\&\&\&\& {\parr_1 B}
	\arrow["{\delta_A'}"{description}, from=2-2, to=2-4]
	\arrow["{\rdual{(\iota_{\ldual{A}})}}"{description}, from=2-4, to=2-6]
	\arrow["f"{description}, from=2-2, to=4-2]
	\arrow["{\delta_B'}"{description}, from=4-2, to=4-4]
	\arrow["{\rdual{(\iota_{\ldual{B}})}}"{description}, from=4-4, to=4-6]
	\arrow["{\rdual{(\otimes_1 \ldual{f})}}"{description}, from=2-6, to=4-6]
	\arrow["{\rdual{(\ldual{f})}}"{description}, from=2-4, to=4-4]
	\arrow[Rightarrow, no head, from=1-1, to=2-2]
	\arrow["{\iota_A'}"{description}, from=1-1, to=1-7]
	\arrow["{\parr_1 f}"{description}, from=1-7, to=5-7]
	\arrow["{\iota_B'}"{description}, from=5-1, to=5-7]
	\arrow["f"{description}, from=1-1, to=5-1]
	\arrow[Rightarrow, no head, from=5-1, to=4-2]
	\arrow[Rightarrow, no head, from=5-7, to=4-6]
	\arrow[Rightarrow, no head, from=1-7, to=2-6]
\end{tikzcd}\]

The outer rectangle corresponds to naturality of $\iota'$ and we get the inner ones by naturality of $\delta'$ and $\iota$ respectively.

Now let us prove that $\alpha'$ and $\iota'$ verify the coherence laws of an unbiased monoidal category.

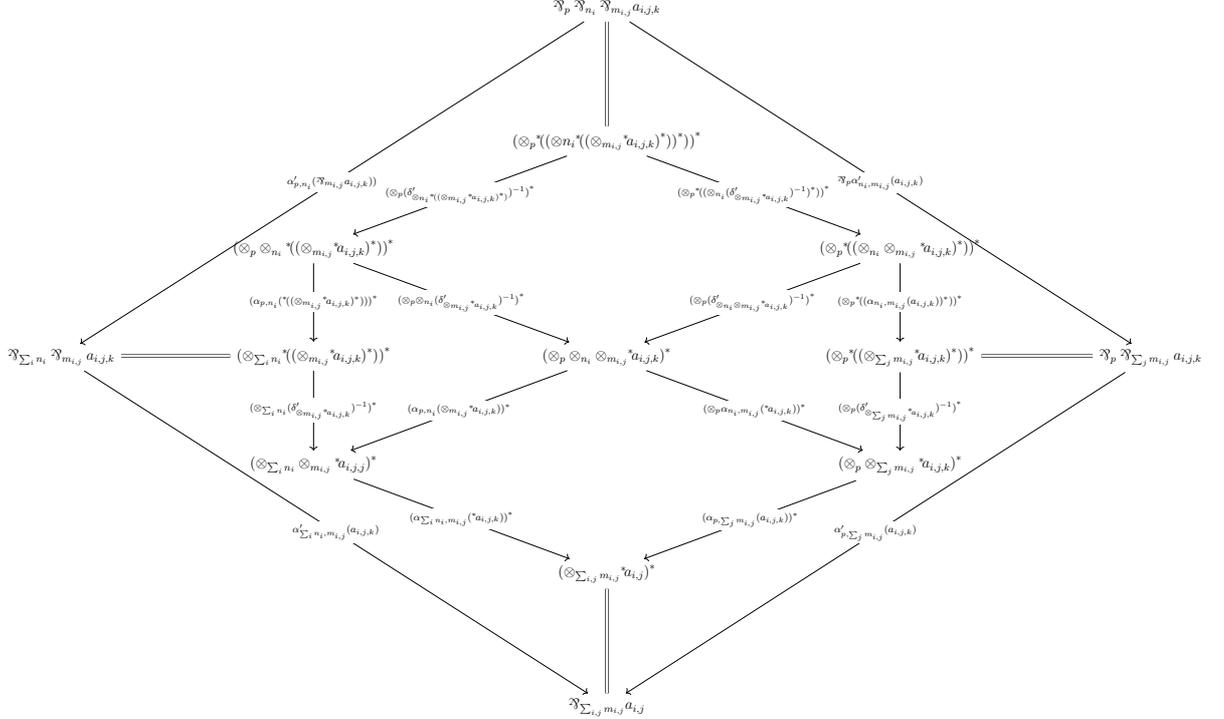
\resizebox{\hsize}{!}{
\begin{tikzcd}[ampersand replacement=\&]
	\&\&\&\&\&\& {\parr_p \parr_{n_i} \parr_{m_{i,j}} a_{i,j,k}} \\
	\\
	\\
	\\
	\&\&\&\&\&\& {\rdual{(\otimes_p\ldual{(\rdual{(\otimes{n_i} \ldual{(\rdual{(\otimes_{m_{i,j}} \ldual{a_{i,j,k}})})})})})}} \\
	\\
	\\
	\&\&\& {\rdual{(\otimes_p \otimes_{n_i} \ldual{(\rdual{(\otimes_{m_{i,j}} \ldual{a_{i,j,k}})})})}} \&\&\&\&\&\& {\rdual{(\otimes_p \ldual{(\rdual{(\otimes_{n_i} \otimes_{m_{i,j}} \ldual{a_{i,j,k}})})})}} \\
	\\
	\\
	{\parr_{\sum_i n_i} \parr_{m_{i,j}} a_{i,j,k}} \&\&\& {\rdual{(\otimes_{\sum_i n_i} \ldual{(\rdual{(\otimes_{m_{i,j}} \ldual{a_{i,j,k}})})})}} \&\&\& {\rdual{(\otimes_p \otimes_{n_i}\otimes_{m_{i,j}}\ldual{a_{i,j,k}})}} \&\&\& {\rdual{(\otimes_p \ldual{(\rdual{(\otimes_{\sum_j m_{i,j}} \ldual{a_{i,j,k}})})})}} \&\&\& {\parr_p \parr_{\sum_j m_{i,j}} a_{i,j,k}} \\
	\\
	\\
	\&\&\& {\rdual{(\otimes_{\sum_i n_i} \otimes_{m_{i,j}} \ldual{a_{i,j,j}})}} \&\&\&\&\&\& {\rdual{(\otimes_p \otimes_{\sum_j m_{i,j}}\ldual{a_{i,j,k}})}} \\
	\\
	\\
	\&\&\&\&\&\& {\rdual{(\otimes_{\sum_{i,j}m_{i,j}} \ldual{a_{i,j}})}} \\
	\\
	\\
	\\
	\&\&\&\&\&\& {\parr_{\sum_{i,j}m_{i,j}} a_{i,j}}
	\arrow["{\rdual{(\otimes_p (\delta'_{\otimes_{n_i}\ldual{(\rdual{(\otimes_{m_{i,j}} \ldual{a_{i,j,k}})})}})^{-1})}}"{description}, from=5-7, to=8-4]
	\arrow["{\rdual{(\alpha_{p,n_i} (\ldual{(\rdual{(\otimes_{m_{i,j}} \ldual{a_{i,j,k}})})}))}}"{description}, from=8-4, to=11-4]
	\arrow["{\rdual{(\otimes_{\sum_i n_i} (\delta'_{\otimes_{m_{i,j}} \ldual{a_{i,j,k}}})^{-1})}}"{description}, from=11-4, to=14-4]
	\arrow["{\rdual{(\alpha_{\sum_i n_i, m_{i,j}}(\ldual{a_{i,j,k}}))}}"{description}, from=14-4, to=17-7]
	\arrow["{\rdual{(\otimes_p \ldual{(\rdual{(\otimes_{n_i} (\delta_{\otimes_{m_{i,j}} \ldual{a_{i,j,k}}}')^{-1})})})}}"{description}, from=5-7, to=8-10]
	\arrow["{\rdual{(\otimes_p \ldual{(\rdual{(\alpha_{n_i,m_{i,j}}(a_{i,j,k}))})})}}"{description}, from=8-10, to=11-10]
	\arrow["{\rdual{(\otimes_p (\delta_{\otimes_{\sum_j m_{i,j}}\ldual{a_{i,j,k}}}')^{-1})}}"{description}, from=11-10, to=14-10]
	\arrow["{\rdual{(\alpha_{p,\sum_j m_{i,j}}(a_{i,j,k}))}}"{description}, from=14-10, to=17-7]
	\arrow["{\rdual{(\otimes_p (\delta'_{\otimes_{n_i}\otimes_{m_{i,j}}\ldual{a_{i,j,k}}})^{-1})}}"{description}, from=8-10, to=11-7]
	\arrow["{\rdual{(\otimes_p \alpha_{n_i,m_{i,j}}(\ldual{a_{i,j,k}}))}}"{description}, from=11-7, to=14-10]
	\arrow["{\rdual{(\otimes_p \otimes_{n_i} (\delta'_{{\otimes_{m_{i,j}}\ldual{a_{i,j,k}}}})^{-1})}}"{description}, from=8-4, to=11-7]
	\arrow["{\rdual{(\alpha_{p,n_i}(\otimes_{m_{i,j}}\ldual{a_{i,j,k}}))}}"{description}, from=11-7, to=14-4]
	\arrow["{\alpha_{\sum_i n_i, m_{i,j}}'(a_{i,j,k})}"{description}, from=11-1, to=21-7]
	\arrow["{\parr_p \alpha_{n_i,m_{i,j}}'(a_{i,j,k})}"{description}, from=1-7, to=11-13]
	\arrow["{\alpha_{p,n_i}'(\parr_{m_{i,j}} a_{i,j,k}))}"{description}, from=1-7, to=11-1]
	\arrow[Rightarrow, no head, from=1-7, to=5-7]
	\arrow[Rightarrow, no head, from=11-1, to=11-4]
	\arrow["{\alpha'_{p,\sum_j m_{i,j}}(a_{i,j,k})}"{description}, from=11-13, to=21-7]
	\arrow[Rightarrow, no head, from=11-13, to=11-10]
	\arrow[Rightarrow, no head, from=17-7, to=21-7]
\end{tikzcd}
}

The outer diamond represent the coherence law for $\parr$.
The outer parts commute by definition.
The top diamond and the two triangles commute by naturality of $(\delta')^{-1}$ and $\alpha$ while the bottom diamond commutes by the coherence law of $\otimes$.

Another coherence law is the following:

% https://q.uiver.app/?q=WzAsOSxbNCwxLCJcXHJkdWFseyhcXG90aW1lc19uIFxcbGR1YWx7YV9pfSl9Il0sWzEwLDEsIlxccmR1YWx7KFxcb3RpbWVzX24gXFxsZHVhbHtcXHJkdWFseyhcXG90aW1lc18xIFxcbGR1YWx7YV9pfSl9fSl9Il0sWzEwLDUsIlxccmR1YWx7KFxcb3RpbWVzX24gXFxsZHVhbHthX2l9KX0iXSxbNywxLCJcXHJkdWFseyhcXG90aW1lc19uIFxcbGR1YWx7XFxyZHVhbHsoXFxsZHVhbHthX2l9KX19KX0iXSxbMTAsMywiXFxyZHVhbHsoXFxvdGltZXNfbiBcXG90aW1lc18xIFxcbGR1YWx7YV9pfSl9Il0sWzcsMywiXFxyZHVhbHsoXFxvdGltZXNfbiBcXGxkdWFse2FfaX0pfSJdLFswLDAsIlxccGFycl9uIGFfaSJdLFsxMSw3LCJcXHBhcnJfbiBhX2kiXSxbMTEsMCwiXFxwYXJyX24gXFxwYXJyXzEgYV9pIl0sWzAsMywiXFxyZHVhbHsoXFxvdGltZXNfbiBcXGxkdWFse1xcZGVsdGEnX3thX2l9fSl9IiwxXSxbMywxLCJcXHJkdWFseyhcXG90aW1lc19uIFxcbGR1YWx7XFxyZHVhbHsoXFxpb3RhX3tcXGxkdWFse2FfaX19KX19KX0iLDFdLFsxLDQsIlxccmR1YWx7KFxcb3RpbWVzX24gKFxcZGVsdGEnX3tcXG90aW1lc18xIFxcbGR1YWx7YV9pfX0pXnstMX0pfSIsMV0sWzQsMiwiXFxyZHVhbHsoXFxhbHBoYV97biwxfShcXGxkdWFse2FfaX0pKX0iLDFdLFszLDUsIlxccmR1YWx7KFxcb3RpbWVzX24gKFxcZGVsdGEnX3tcXGxkdWFse2FfaX19KV57LTF9KX0iLDFdLFs1LDQsIlxccmR1YWx7XFxvdGltZXNfbiBcXGlvdGFfe1xcbGR1YWx7YV9pfX19IiwxXSxbMCw1LCIiLDEseyJsZXZlbCI6Miwic3R5bGUiOnsiaGVhZCI6eyJuYW1lIjoibm9uZSJ9fX1dLFs1LDIsIiIsMSx7ImxldmVsIjoyLCJzdHlsZSI6eyJoZWFkIjp7Im5hbWUiOiJub25lIn19fV0sWzYsOCwiXFxwYXJyX24gXFxpb3RhJ197YV9pfSIsMV0sWzgsNywiXFxhbHBoYSdfe24sMX0oYV9pKSIsMV0sWzYsNywiIiwxLHsibGV2ZWwiOjIsInN0eWxlIjp7ImhlYWQiOnsibmFtZSI6Im5vbmUifX19XSxbNiwwLCIiLDEseyJsZXZlbCI6Miwic3R5bGUiOnsiaGVhZCI6eyJuYW1lIjoibm9uZSJ9fX1dLFs4LDEsIiIsMSx7ImxldmVsIjoyLCJzdHlsZSI6eyJoZWFkIjp7Im5hbWUiOiJub25lIn19fV0sWzIsNywiIiwxLHsibGV2ZWwiOjIsInN0eWxlIjp7ImhlYWQiOnsibmFtZSI6Im5vbmUifX19XV0=
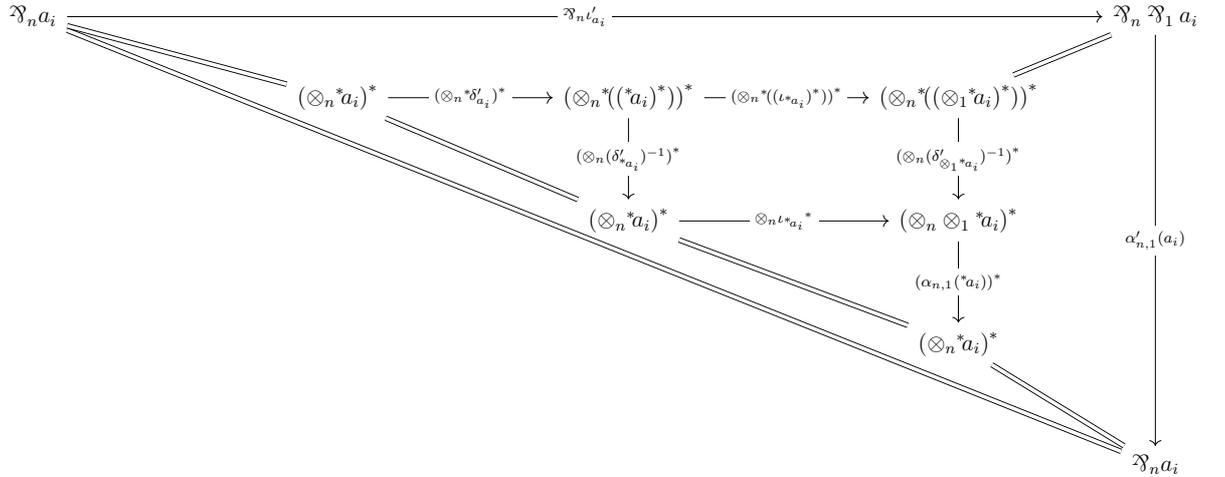
\resizebox{\hsize}{!}{
\begin{tikzcd}[ampersand replacement=\&]
	{\parr_n a_i} \&\&\&\&\&\&\&\&\&\&\& {\parr_n \parr_1 a_i} \\
	\&\&\&\& {\rdual{(\otimes_n \ldual{a_i})}} \&\&\& {\rdual{(\otimes_n \ldual{(\rdual{(\ldual{a_i})})})}} \&\&\& {\rdual{(\otimes_n \ldual{(\rdual{(\otimes_1 \ldual{a_i})})})}} \\
	\\
	\&\&\&\&\&\&\& {\rdual{(\otimes_n \ldual{a_i})}} \&\&\& {\rdual{(\otimes_n \otimes_1 \ldual{a_i})}} \\
	\\
	\&\&\&\&\&\&\&\&\&\& {\rdual{(\otimes_n \ldual{a_i})}} \\
	\\
	\&\&\&\&\&\&\&\&\&\&\& {\parr_n a_i}
	\arrow["{\rdual{(\otimes_n \ldual{\delta'_{a_i}})}}"{description}, from=2-5, to=2-8]
	\arrow["{\rdual{(\otimes_n \ldual{(\rdual{(\iota_{\ldual{a_i}})})})}}"{description}, from=2-8, to=2-11]
	\arrow["{\rdual{(\otimes_n (\delta'_{\otimes_1 \ldual{a_i}})^{-1})}}"{description}, from=2-11, to=4-11]
	\arrow["{\rdual{(\alpha_{n,1}(\ldual{a_i}))}}"{description}, from=4-11, to=6-11]
	\arrow["{\rdual{(\otimes_n (\delta'_{\ldual{a_i}})^{-1})}}"{description}, from=2-8, to=4-8]
	\arrow["{\rdual{\otimes_n \iota_{\ldual{a_i}}}}"{description}, from=4-8, to=4-11]
	\arrow[Rightarrow, no head, from=2-5, to=4-8]
	\arrow[Rightarrow, no head, from=4-8, to=6-11]
	\arrow["{\parr_n \iota'_{a_i}}"{description}, from=1-1, to=1-12]
	\arrow["{\alpha'_{n,1}(a_i)}"{description}, from=1-12, to=8-12]
	\arrow[Rightarrow, no head, from=1-1, to=8-12]
	\arrow[Rightarrow, no head, from=1-1, to=2-5]
	\arrow[Rightarrow, no head, from=1-12, to=2-11]
	\arrow[Rightarrow, no head, from=6-11, to=8-12]
\end{tikzcd}} 

The left triangle is just invertibility and the rectangle is naturality of $(\delta')^{-1}$ while the bottom triangle is the equivalent coherence law for $\otimes$.

Similarly the last coherence law is obtained by the similar one for $\otimes$ and naturality.

\end{proof}

\begin{prop}
The duality induced by the negation extends to a monoidal one \[% https://q.uiver.app/?q=WzAsMixbMCwwLCIoXFxjQywgXFxvdGltZXMsIEkpIl0sWzIsMCwiKFxcY0Nee1xcb3B9LCBcXHBhcnIsIFxcYm90KSJdLFswLDEsIlxcbGR1YWx7KC0pfSIsMSx7ImN1cnZlIjotM31dLFsxLDAsIlxccmR1YWx7KC0pfSIsMSx7ImN1cnZlIjotM31dLFsyLDMsIlxcc2ltZXEiLDEseyJzaG9ydGVuIjp7InNvdXJjZSI6MjAsInRhcmdldCI6MjB9LCJzdHlsZSI6eyJib2R5Ijp7Im5hbWUiOiJub25lIn0sImhlYWQiOnsibmFtZSI6Im5vbmUifX19XV0=
\begin{tikzcd}[ampersand replacement=\&]
	{(\cC, \otimes, I)} \&\& {(\cC^{\op}, \parr, \bot)}
	\arrow[""{name=0, anchor=center, inner sep=0}, "{\rdual{(-)}}"{description}, curve={height=-18pt}, from=1-1, to=1-3]
	\arrow[""{name=1, anchor=center, inner sep=0}, "{\ldual{(-)}}"{description}, curve={height=-18pt}, from=1-3, to=1-1]
	\arrow["\simeq"{description}, Rightarrow, draw=none, from=0, to=1]
\end{tikzcd}\]
\end{prop}
\begin{proof}
Let us prove that $\rdual{(-)}$ is strong monoidal.
We have \[\rdual{(-)}_n \colon \parr_n \rdual{(-)} = \rdual{(\otimes_n \ldual{(\rdual{(-)})})} \simeq \rdual{(\otimes_n -)}\]
Then the coherence diagram for monoidal functors follows from naturality of the duality.
\end{proof}

\begin{rk}
The previous propositions can be derived from a more general fact: if \cC is an unbiased monoidal category and there is an equivalence of categories $\cC \simeq \cD$ then \cD has a transported structure of unbiased monoidal category.
This is what happens here with the equivalence $\rdual{(-)} \colon \cC \to \cC^\op$ and then noting that the unbiased monoidal structures on \cC and $\cC^\op$ correspond \footnote{Thanks to Richard Garner for suggesting this explanation}.
\end{rk}

\begin{prop}
The following are equivalent:
\begin{itemize}
\item $(\cC, \otimes, I, \bot)$ is a (biased) $\ast$-autonomous category
\item there is a monoidal duality  % https://q.uiver.app/?q=WzAsMixbMCwwLCIoXFxjQywgXFxvdGltZXMsIEkpIl0sWzIsMCwiKFxcY0Nee1xcb3B9LCBcXHBhcnIsIFxcYm90KSJdLFswLDEsIlxcbGR1YWx7KC0pfSIsMSx7ImN1cnZlIjotM31dLFsxLDAsIlxccmR1YWx7KC0pfSIsMSx7ImN1cnZlIjotM31dLFsyLDMsIlxcc2ltZXEiLDEseyJzaG9ydGVuIjp7InNvdXJjZSI6MjAsInRhcmdldCI6MjB9LCJzdHlsZSI6eyJib2R5Ijp7Im5hbWUiOiJub25lIn0sImhlYWQiOnsibmFtZSI6Im5vbmUifX19XV0=
\begin{tikzcd}[ampersand replacement=\&]
	{(\cC, \otimes, I)} \&\& {(\cC^{\op}, \parr, \bot)}
	\arrow[""{name=0, anchor=center, inner sep=0}, "{\rdual{(-)}}"{description}, curve={height=-18pt}, from=1-1, to=1-3]
	\arrow[""{name=1, anchor=center, inner sep=0}, "{\ldual{(-)}}"{description}, curve={height=-18pt}, from=1-3, to=1-1]
	\arrow["\simeq"{description}, Rightarrow, draw=none, from=0, to=1]
\end{tikzcd}

s.t. $\cC(A \otimes B, C) \simeq \cC(B, \ldual{A} \parr C)$ and $\cC(C, A \parr B) \simeq \cC(C \otimes \rdual{B},A)$ naturally
\end{itemize}
\end{prop}

\subsection{As a linearly distributive category with duals}

The notion of linearly distributive category has been introduced by Cockett, Seely and collaborators in a series of papers under the term weakly distributive categories.
From a logical point of view, the introduction of $\ast$-autonomous categories as a monoidal biclosed category with dualising objects can be understand as generating the logic from the conjunction $\otimes$, the implications $\multimap$ and $\multimapinv$ and the false value $\bot$.
From this can be deduced the negations $\ldual{-} := - \multimap \bot$ and $\rdual{-} := \bot \multimapinv$.
The perspective of linearly distributive categories with duals, is to introduce the logic from the conjunction $\otimes$, the disjunction $\parr$ and the negations $\ldual{-}$ and $\rdual{-}$.
Linearly distributive categories can be understood as models of multiplicative linear logic without negation.

Involved in their definition are two monoidal structures, that interact through so-called distributivity laws.

The distributivity laws are what is needed to interpret the cut rule.
For example, if one wants to interpret the following cut rule:
\AXC{$A,B \vdash C,D$}
\AXC{$D,E \vdash F, G$}
\BIC{$A,B,E \vdash C,F,G$}
\DP
The interpretation of the premises are given by morphisms $f \colon \den{A} \otimes \den{B} \to \den{C} \parr \den{D}$ and $g \colon \den{D} \otimes \den{E} \to \den{F} \parr \den{G}$.
To obtain a morphism $\den{A} \otimes \den{B} \otimes \den{E} \to \den{C} \parr \den{F} \parr \den{G}$, one have to proceed as follow:
\[(\den{A} \otimes \den{B}) \otimes \den{E} \xrightarrow{f \otimes \den{E}} (\den{C} \parr \den{D}) \otimes \den{E} \xrightarrow{\delta^R} \den{C} \parr (\den{D} \otimes \den{E}) \xrightarrow{\den{C} \parr g} \den{C} \parr (\den{F} \parr \den{G})\]
This $\delta^R$ is one of the two distributivity laws needed to define a linearly distributive category (see \cite{CockettSeely1997}).

\begin{definition}
A \defin{linearly distributive category} is a category equipped with two biased monoidal structures $(\cC, \otimes, I, \alpha, \lambda, \rho)$ and $(\cC, \parr, \bot, \alpha', \lambda', \rho')$ and natural transformations
\[ \delta^L \colon A \otimes (B \parr C) \to A \otimes (B \parr C)\]
\[\delta^R \colon (A \parr B) \otimes C \to A \parr (B \otimes C)\]
satisfying six pentagons equations and four triangle equations that we omit here, making all the monoidal structures works together, see \cite{CockettSeely1997}.
\end{definition}

\subsubsection{Left-biased lax linearly distributive category}

Instead of spelling out all the diagrams, let us give a more general version.

\begin{definition}
A \defin{lax left-biased linearly distributive category} is a category $\cC$ equipped with functors $\otimes, \parr \colon \cC \times \cC \to \cC$, objects $I, \bot$ and natural transformations given by the following families of morphisms:
\begin{itemize}
\item $\alpha_{A,B,C} \colon (A \otimes B) \otimes C \to A \otimes (B \otimes C)$, $\lambda_A \colon A \to I \otimes A$ and $\rho_A \colon A \to A \otimes I$
\item $\alpha_{A,B,C}' \colon A \parr (B \parr C) \to (A \parr B) \parr C$, $\lambda_A' \colon \bot \parr A \to A$ and $\rho_A' \colon A \parr \bot \to A$
\item $\delta^L_{A,B,C} \colon A \otimes (B \parr C) \to (A \otimes B) \parr C$
\item $\delta^R_{A,B,C} \colon (A \parr B) \otimes C \to A \parr ( B \otimes C)$
\end{itemize}
subject to the following diagrams:
\begin{itemize}
\item eight pentagons, one for each choice of connective $\ominus, \oslash, \odot \in \{\otimes, \parr\}$:
\[ %https://q.uiver.app/?q=WzAsNSxbMCwwLCIoKEEgXFxvbWludXMgQikgXFxvc2xhc2ggQykgXFxvZG90IEQiXSxbMiwwLCIoQSBcXG9taW51cyBCKSBcXG9zbGFzaCAoQyBcXG9kb3QgRCkiXSxbMywxLCJBIFxcb21pbnVzIChCIFxcb3NsYXNoIChDIFxcb2RvdCBEKSkiXSxbMCwyLCIoQSBcXG9taW51cyAoQiBcXG9zbGFzaCBDKSkgXFxvZG90IEQiXSxbMiwyLCJBIFxcb21pbnVzICgoQiBcXG9zbGFzaCBDKSBcXG9kb3QgRCkiXSxbMCwxLCIiLDEseyJzdHlsZSI6eyJoZWFkIjp7Im5hbWUiOiJub25lIn19fV0sWzEsMiwiIiwxLHsic3R5bGUiOnsiaGVhZCI6eyJuYW1lIjoibm9uZSJ9fX1dLFswLDMsIiIsMSx7InN0eWxlIjp7ImhlYWQiOnsibmFtZSI6Im5vbmUifX19XSxbMyw0LCIiLDEseyJzdHlsZSI6eyJoZWFkIjp7Im5hbWUiOiJub25lIn19fV0sWzQsMiwiIiwxLHsic3R5bGUiOnsiaGVhZCI6eyJuYW1lIjoibm9uZSJ9fX1dXQ==
\begin{tikzcd}[ampersand replacement=\&]
	{((A \ominus B) \oslash C) \odot D} \&\& {(A \ominus B) \oslash (C \odot D)} \\
	\&\&\& {A \ominus (B \oslash (C \odot D))} \\
	{(A \ominus (B \oslash C)) \odot D} \&\& {A \ominus ((B \oslash C) \odot D)}
	\arrow[no head, from=1-1, to=1-3]
	\arrow[no head, from=1-3, to=2-4]
	\arrow[no head, from=1-1, to=3-1]
	\arrow[no head, from=3-1, to=3-3]
	\arrow[no head, from=3-3, to=2-4]
\end{tikzcd}\]
\item ten triangles, four of the left type, two of the middle one and four of the right one, corresponding to the choices of $\ominus, \star \in \{\otimes, \parr\}$ and where $U$ is always the unit of $\star$
% https://q.uiver.app/?q=WzAsOSxbMCwwLCIoVSBcXHN0YXIgQSkgXFxvbWludXMgQiJdLFsxLDEsIkEgXFxvbWludXMgQiJdLFswLDEsIlUgXFxzdGFyIChBIFxcb21pbnVzIEIpIl0sWzIsMCwiKEEgXFxzdGFyIFUpIFxcc3RhciBCIl0sWzIsMSwiQSBcXHN0YXIgKFUgXFxzdGFyIEIpIl0sWzMsMSwiQSBcXHN0YXIgQiJdLFs0LDAsIihBXFxvbWludXMgQikgXFxzdGFyIFUiXSxbNCwxLCJBIFxcb21pbnVzIChCIFxcc3RhciBVKSJdLFs1LDEsIkEgXFxvbWludXMgQiJdLFswLDEsIiIsMSx7InN0eWxlIjp7ImhlYWQiOnsibmFtZSI6Im5vbmUifX19XSxbMCwyLCIiLDEseyJzdHlsZSI6eyJoZWFkIjp7Im5hbWUiOiJub25lIn19fV0sWzIsMSwiIiwxLHsic3R5bGUiOnsiaGVhZCI6eyJuYW1lIjoibm9uZSJ9fX1dLFszLDQsIiIsMSx7InN0eWxlIjp7ImhlYWQiOnsibmFtZSI6Im5vbmUifX19XSxbMyw1LCIiLDEseyJzdHlsZSI6eyJoZWFkIjp7Im5hbWUiOiJub25lIn19fV0sWzQsNSwiIiwxLHsic3R5bGUiOnsiaGVhZCI6eyJuYW1lIjoibm9uZSJ9fX1dLFs2LDcsIiIsMSx7InN0eWxlIjp7ImhlYWQiOnsibmFtZSI6Im5vbmUifX19XSxbNiw4LCIiLDEseyJzdHlsZSI6eyJoZWFkIjp7Im5hbWUiOiJub25lIn19fV0sWzcsOCwiIiwxLHsic3R5bGUiOnsiaGVhZCI6eyJuYW1lIjoibm9uZSJ9fX1dXQ==
\[\begin{tikzcd}[ampersand replacement=\&]
	{(U \star A) \ominus B} \&\& {(A \star U) \star B} \&\& {(A\ominus B) \star U} \\
	{U \star (A \ominus B)} \& {A \ominus B} \& {A \star (U \star B)} \& {A \star B} \& {A \ominus (B \star U)} \& {A \ominus B}
	\arrow[no head, from=1-1, to=2-2]
	\arrow[no head, from=1-1, to=2-1]
	\arrow[no head, from=2-1, to=2-2]
	\arrow[no head, from=1-3, to=2-3]
	\arrow[no head, from=1-3, to=2-4]
	\arrow[no head, from=2-3, to=2-4]
	\arrow[no head, from=1-5, to=2-5]
	\arrow[no head, from=1-5, to=2-6]
	\arrow[no head, from=2-5, to=2-6]
\end{tikzcd}\]
\item and finally two globes corresponding to the choice of monoidal structure $(\star, U)$:
% https://q.uiver.app/?q=WzAsMixbMCwwLCJVIFxcc3RhciBVIl0sWzIsMCwiVSJdLFswLDEsIiIsMSx7ImN1cnZlIjotMiwic3R5bGUiOnsiaGVhZCI6eyJuYW1lIjoibm9uZSJ9fX1dLFswLDEsIiIsMSx7ImN1cnZlIjoyLCJzdHlsZSI6eyJoZWFkIjp7Im5hbWUiOiJub25lIn19fV1d
\[\begin{tikzcd}[ampersand replacement=\&]
	{U \star U} \&\& U
	\arrow[curve={height=-12pt}, no head, from=1-1, to=1-3]
	\arrow[curve={height=12pt}, no head, from=1-1, to=1-3]
\end{tikzcd}\]
\end{itemize}
where the directions of the arrows in the diagram are uniquely determined by the choice of monoidal products.
\end{definition}

\begin{rk}
	There is also a notion of lax right-biased linearly distributive category where the associators $\alpha$ and $\alpha'$ go in the opposite direction. The distributivity maps are unchanged.
\end{rk}

\begin{rk}
Each lax monoidal structure comes equipped with a pentagon diagram, three triangle ones and one globular one.
This is similar to the original definition of a monoidal category in Mac Lane.
In the case of a monoidal category where the natural transformations are invertible, it has been proven by Kelly in \cite{Kelly1964} that it is enough to consider only the usual pentagon diagram plus two triangle diagrams.
It is not true in the lax case though.
\end{rk}

\begin{prop}
A linearly distributive category is a lax linearly distributive category where the associators and unitors for the monoidal structures are invertible.
\end{prop}
\begin{proof}
By definition if we take monoidal category to mean the original version by Mac Lane.
If we prefer the more recent version, the proof of the equivalence of the two definitions, relies solely on the equivalence between the two definitions of monoidal categories.
\end{proof}

\begin{definition}
In a lax linearly distributive category a left dual $\ldual{A}$ and a right dual $\rdual{A}$ of an object $A$ is an object equipped with morphisms:
\begin{itemize}
\item $\lcup_A \colon I \to \ldual{A} \parr A$ and $\lcap_A \colon A \otimes \ldual{A} \to \bot$
\item $\rcup_A \colon I \to A \parr \rdual{A}$ and $\rcap_A \colon \rdual{A} \otimes A \to \bot$
\end{itemize}
satisfying the coherence laws:
% https://q.uiver.app/?q=WzAsNSxbMCwwLCIoQSBcXG9taW51cyBBJykgXFxvbWludXMnIEEiXSxbMCwyLCJBIFxcb21pbnVzIChBJyBcXG9taW51cycgQSkiXSxbMSwwLCJVJyBcXG9taW51cycgQSJdLFsyLDEsIkEiXSxbMSwyLCJBIFxcb21pbnVzIFUiXSxbMCwxLCIiLDEseyJzdHlsZSI6eyJoZWFkIjp7Im5hbWUiOiJub25lIn19fV0sWzAsMiwiIiwxLHsic3R5bGUiOnsiaGVhZCI6eyJuYW1lIjoibm9uZSJ9fX1dLFsyLDMsIiIsMSx7InN0eWxlIjp7ImhlYWQiOnsibmFtZSI6Im5vbmUifX19XSxbMSw0LCIiLDEseyJzdHlsZSI6eyJoZWFkIjp7Im5hbWUiOiJub25lIn19fV0sWzQsMywiIiwxLHsic3R5bGUiOnsiaGVhZCI6eyJuYW1lIjoibm9uZSJ9fX1dXQ==
\[\begin{tikzcd}[ampersand replacement=\&]
	{(A \ominus A') \ominus' A} \& {U' \ominus' A} \\
	\&\& A \\
	{A \ominus (A' \ominus' A)} \& {A \ominus U}
	\arrow[no head, from=1-1, to=3-1]
	\arrow[no head, from=1-1, to=1-2]
	\arrow[no head, from=1-2, to=2-3]
	\arrow[no head, from=3-1, to=3-2]
	\arrow[no head, from=3-2, to=2-3]
\end{tikzcd}\]
where $(\ominus, U)$ and $(\ominus', U')$ are dual monoidal structures and the combinations possible are $(A,\ominus, A') \in \{(A, \otimes, \ldual{A}),(\ldual{A},\parr, A)\}$ for the left dual and $(A,\ominus, A') \in \{(\rdual{A}, \otimes, A),(A, \parr, \rdual{A})\}$ for the right dual.
\end{definition}

\begin{prop}
$\ast$-autonomous categories correspond to linearly distributive categories with left and right duals for all objects. 
\end{prop}
\begin{proof}
See \cite{CockettSeely1997} for a proof.

\end{proof}

\subsubsection{Lax unbiased linearly distributive category}

For it to be an interesting notion, we would like for any unbiased lax linearly distributive category to define a polycategory.
As mentioned in the introduction of this chapter, the composition in a polycategory is defined for polymaps of the form $f \colon \Gamma \to \Delta_1, A, \Delta_2$ and $g \colon \Gamma_1', A, \Gamma_2' \to \Delta'$ to give a polymap $g \circ f \colon \Gamma_1', \Gamma, \Gamma_2' \to \Delta_1, \Delta', \Delta_2$.
When defining a polycategory from a lax unbiased linearly distributive category, polymaps will correspond to morphisms $ \otimes(\Gamma) \to \parr(\Delta)$.
Composition of two polymaps $f \colon \otimes(\Gamma) \to \parr(\Delta_1, A, \Delta_2)$ and $g \colon \otimes(\Gamma_1', A, \Gamma_2') \to \parr(\Delta')$ will be defined by:

\begin{align*}
\otimes(\Gamma_1', \Gamma, \Gamma_2') &\xrightarrow{\text{oplax}} \otimes(\Gamma_1', \otimes \Gamma, \Gamma_2')\\ 
&\xrightarrow{f} \otimes(\Gamma_1', \parr(\Delta_1, A, \Delta_2), \Gamma_2')\\ 
&\xrightarrow{\text{dist.}} \parr(\Delta_1, \otimes(\Gamma_1', A, \Gamma_2'), \Delta_2) \\
&\xrightarrow{g} \parr(\Delta_1, \parr \Delta', \Delta_2) \\
&\xrightarrow{\text{lax}} \parr(\Delta_1, \Delta', \Delta_2)
\end{align*}

We learn several things from this.
First, that in order to define composition we need for $(\cC, \otimes_n)$ to be oplax monoidal and $(\cC, \parr_n)$ to be lax.
Also we can get the general form of the distributivity law $\otimes(\Gamma_1, \parr(\Delta_1, A, \Delta_2), \Gamma_2) \to \parr(\Delta_1, \otimes(\Gamma_1, A, \Gamma_2), \Delta_2)$.
However, this general form exchanges the position of $\Gamma_i$ and $\Delta_i$ which should not be possible in the non-symmetric case.
In fact, when defining composition of a non-symmetric polycategory, we constrain the composition to only be defined when either $\Gamma_1$ is empty or $\Delta_1$ is and similarly either $\Gamma_2$ is empty or $\Delta_2$ is.
We will ask for the same constraint when defining distributivity for (non-symmetric) lax unbiased linearly distributive categories.
This boils down to requiring four distributivity laws:
\begin{itemize}
\item $\otimes(\Gamma, \parr(A, \Delta)) \to \parr(\otimes(\Gamma,A), \Delta))$
\item $\otimes(\parr(\Delta,A), \Gamma) \to \parr(\Delta, \otimes(A,\Gamma))$
\item $\otimes(\Gamma_1, \parr(A), \Gamma_2) \to \parr(\otimes(\Gamma_1, A, \Gamma_2))$
\item $\otimes(\parr(\Delta_1, A, \Delta_2)) \to \parr(\Delta_1, \otimes(A), \Delta_2)$
\end{itemize}

In the case where the lists are taken to be singletons the first two gives the usual $A \otimes (B \parr C) \to (A \otimes B) \parr C$ and $(A \parr B) \otimes C \to A \parr (B \otimes C)$.

For the definitions of lax and oplax unbiased monoidal categories, we take the definition of unbiased monoidal categories and replace natural isomorphisms by natural transformations.

\begin{definition}
A \defin{lax unbiased linearly distributive category} (luldc for short) is a category \cC equipped with
\begin{itemize}
\item an oplax monoidal structure 

$(\cC, (\otimes_n)_{n \in \N}, (\alpha_{n,(m_k)_{1 \leq k \leq n}})_{n \in \N}, \eta)$, i.e.
\begin{itemize}
\item for each arity $n$, a functor $\otimes_n \colon \cC \times \dots \times \cC \to \cC$ from $n$ copies of \cC
\item for each $n,(m_k)$, natural transformations $\alpha_{n, (m_k)} \colon\otimes_{\sum_k m_k (-)} \Rightarrow \otimes_n(\otimes_{m_1}(-),\dots, \otimes_{m_k}(-))$
\item a natural transformation $\eta \colon \otimes_1(-) \Rightarrow -$
\end{itemize}
subject to the coherence law:
% https://q.uiver.app/?q=WzAsOCxbMSwwLCJcXG90aW1lc19wXFxvdGltZXNfe25faX1cXG90aW1lc197bV97aSxqfX1BX3tpLGosa30iXSxbMCwxLCJcXG90aW1lc197XFxzdW1faSBuX2l9IFxcb3RpbWVzX3ttX3tpLGp9fSBBX3tpLGosa30iXSxbMSwyLCJcXG90aW1lc197XFxzdW1fe2ksan0gbV97aSxqfX0gQV97aSxqLGt9Il0sWzIsMSwiXFxvdGltZXNfcCBcXG90aW1lc197XFxzdW1faiBtX3tpLGp9fSBBX3tpLGosa30iXSxbMSw0LCJcXG90aW1lc19uIEFfaSJdLFsyLDQsIlxcb3RpbWVzXzFcXG90aW1lc19uQV9pIl0sWzEsNiwiXFxvdGltZXNfbkFfaSJdLFswLDQsIlxcb3RpbWVzX24gXFxvdGltZXNfMSBBX2kiXSxbMSwwLCJcXGFscGhhX3twLChuX2kpX2l9IiwxXSxbMiwxLCJcXGFscGhhX3tcXHN1bV9pIG5faSwgKG1fe2ksan0pX2p9IiwxXSxbMiwzLCJcXGFscGhhX3twLChcXHN1bV9qIG1fe2ksan0pX2l9IiwxXSxbMywwLCJcXG90aW1lc19wIFxcYWxwaGFfe25faSwobV97aSxqfSlfan0iLDFdLFs1LDQsIlxcZXRhIiwxXSxbNCw2LCIiLDEseyJsZXZlbCI6Miwic3R5bGUiOnsiaGVhZCI6eyJuYW1lIjoibm9uZSJ9fX1dLFs2LDUsIlxcYWxwaGFfezEsKG4pfSIsMV0sWzcsNCwiXFxvdGltZXNfbiBcXGV0YSIsMV0sWzYsNywiXFxhbHBoYV97biwoMSl9IiwxXV0=
\[\begin{tikzcd}[ampersand replacement=\&]
	\& {\otimes_p\otimes_{n_i}\otimes_{m_{i,j}}A_{i,j,k}} \\
	{\otimes_{\sum_i n_i} \otimes_{m_{i,j}} A_{i,j,k}} \&\& {\otimes_p \otimes_{\sum_j m_{i,j}} A_{i,j,k}} \\
	\& {\otimes_{\sum_{i,j} m_{i,j}} A_{i,j,k}} \\
	\\
	{\otimes_n \otimes_1 A_i} \& {\otimes_n A_i} \& {\otimes_1\otimes_nA_i} \\
	\\
	\& {\otimes_nA_i}
	\arrow["{\alpha_{p,(n_i)_i}}"{description}, from=2-1, to=1-2]
	\arrow["{\alpha_{\sum_i n_i, (m_{i,j})_j}}"{description}, from=3-2, to=2-1]
	\arrow["{\alpha_{p,(\sum_j m_{i,j})_i}}"{description}, from=3-2, to=2-3]
	\arrow["{\otimes_p \alpha_{n_i,(m_{i,j})_j}}"{description}, from=2-3, to=1-2]
	\arrow["\eta"{description}, from=5-3, to=5-2]
	\arrow[Rightarrow, no head, from=5-2, to=7-2]
	\arrow["{\alpha_{1,(n)}}"{description}, from=7-2, to=5-3]
	\arrow["{\otimes_n \eta}"{description}, from=5-1, to=5-2]
	\arrow["{\alpha_{n,(1)}}"{description}, from=7-2, to=5-1]
\end{tikzcd}\]
\item a lax monoidal structure $(\cC, (\parr_n)_{n\in \N}, (\gamma_{(n,(m_k)_{1 \leq k \leq n}})_{n \in \N}, \iota)$, i.e.
\begin{itemize}
\item for each arity $n$, a functor $\parr_n \colon \cC \times \dots \times \cC \to \cC$
\item for each $n, (m_k)$, natural transformations $\gamma_{n,(m_k)} \colon \parr_n(\parr_{m_1}(-), \dots, \parr_{m_n}(-)) \Rightarrow \parr_{\sum_k m_k}(-)$
\item a natural transformation $\iota \colon \parr_1(-) \Rightarrow -$
\end{itemize}
subject to the coherence laws:
% https://q.uiver.app/?q=WzAsOCxbMSwwLCJcXHBhcnJfcFxccGFycl97bl9pfVxccGFycl97bV97aSxqfX1BX3tpLGosa30iXSxbMCwxLCJcXHBhcnJfe1xcc3VtX2kgbl9pfSBcXHBhcnJfe21fe2ksan19IEFfe2ksaixrfSJdLFsxLDIsIlxcZ2FtbWFfe1xcc3VtX3tpLGp9IG1fe2ksan19IEFfe2ksaixrfSJdLFsyLDEsIlxccGFycl9wIFxccGFycl97XFxzdW1faiBtX3tpLGp9fSBBX3tpLGosa30iXSxbMSw0LCJcXHBhcnJfbiBBX2kiXSxbMiw0LCJcXHBhcnJfMVxccGFycl9uQV9pIl0sWzEsNiwiXFxvdGltZXNfbkFfaSJdLFswLDQsIlxccGFycl9uIFxccGFycl8xIEFfaSJdLFswLDEsIlxcZ2FtbWFfe3AsKG5faSlfaX0iLDFdLFsxLDIsIlxcZ2FtbWFfe1xcc3VtX2kgbl9pLCAobV97aSxqfSlfan0iLDFdLFszLDIsIlxcZ2FtbWFfe3AsKFxcc3VtX2ogbV97aSxqfSlfaX0iLDFdLFswLDMsIlxccGFycl9wIFxcZ2FtbWFfe25faSwobV97aSxqfSlfan0iLDFdLFs0LDUsIlxcaW90YSIsMV0sWzQsNiwiIiwxLHsibGV2ZWwiOjIsInN0eWxlIjp7ImhlYWQiOnsibmFtZSI6Im5vbmUifX19XSxbNSw2LCJcXGdhbW1hX3sxLChuKX0iLDFdLFs0LDcsIlxccGFycl9uIFxcaW90YSIsMV0sWzcsNiwiXFxnYW1tYV97biwoMSl9IiwxXV0=
\[
\begin{tikzcd}[ampersand replacement=\&]
	\& {\parr_p\parr_{n_i}\parr_{m_{i,j}}A_{i,j,k}} \\
	{\parr_{\sum_i n_i} \parr_{m_{i,j}} A_{i,j,k}} \&\& {\parr_p \parr_{\sum_j m_{i,j}} A_{i,j,k}} \\
	\& {\gamma_{\sum_{i,j} m_{i,j}} A_{i,j,k}} \\
	\\
	{\parr_n \parr_1 A_i} \& {\parr_n A_i} \& {\parr_1\parr_nA_i} \\
	\\
	\& {\otimes_nA_i}
	\arrow["{\gamma_{p,(n_i)_i}}"{description}, from=1-2, to=2-1]
	\arrow["{\gamma_{\sum_i n_i, (m_{i,j})_j}}"{description}, from=2-1, to=3-2]
	\arrow["{\gamma_{p,(\sum_j m_{i,j})_i}}"{description}, from=2-3, to=3-2]
	\arrow["{\parr_p \gamma_{n_i,(m_{i,j})_j}}"{description}, from=1-2, to=2-3]
	\arrow["\iota"{description}, from=5-2, to=5-3]
	\arrow[Rightarrow, no head, from=5-2, to=7-2]
	\arrow["{\gamma_{1,(n)}}"{description}, from=5-3, to=7-2]
	\arrow["{\parr_n \iota}"{description}, from=5-2, to=5-1]
	\arrow["{\gamma_{n,(1)}}"{description}, from=5-1, to=7-2]
\end{tikzcd}\]
\item for each $m_1, n_1, n_2, m_1$ such that $m_i = 0$ or $n_i = 0$, a natural transformation $\delta_{m_1, n_1, n_2, m_2} \colon \otimes_{m_1 + 1 + m_2}(\Gamma_1, \parr_{n_1 + 1 + n_2}(\Delta_1, A, \Delta_2) ,\Gamma_2) \to \parr_{n_1 + 1 + n_2}(\Delta_1, \otimes_{m_1 + 1 + m_2}(\Gamma_1, A, \Gamma_2), \Delta_2)$
\end{itemize}
satisfying the following coherence laws:
% https://q.uiver.app/?q=WzAsNixbMSwwLCJcXG90aW1lc19wKFxcR2FtbWFfMSwgXFxwYXJyX3tuXzEgKyAxfShcXERlbHRhXzEsQSksIFxcR2FtbWFfMiwgXFxwYXJyX3sxICsgbl8yfShCLCBcXERlbHRhXzIpLCBcXEdhbW1hXzMpIl0sWzAsMSwiXFxwYXJyX3tuXzErMX0oXFxEZWx0YV8xLCBcXG90aW1lc19wKFxcR2FtbWFfMSwgQSwgXFxHYW1tYV8yLCBcXHBhcnJfezEgKyBuXzJ9KEIsXFxEZWx0YV8yKSwgXFxHYW1tYV8zKSkiXSxbMCwzLCJcXHBhcnJfe25fMSsxfShcXERlbHRhXzEsXFxwYXJyX3sxK25fMn0oXFxvdGltZXNfcChcXEdhbW1hXzEsIEEsIFxcR2FtbWFfMiwgQiwgXFxHYW1tYV8zKSxcXERlbHRhXzIpKSJdLFsxLDQsIlxccGFycl97bl8xKzErbl8yfShcXERlbHRhXzEsIFxcb3RpbWVzX3AoXFxHYW1tYV8xLEEsXFxHYW1tYV8yLEIsIFxcR2FtbWFfMyksIFxcRGVsdGFfMikiXSxbMiwxLCJcXHBhcnJfezErbl8yfShcXG90aW1lc19wKFxcR2FtbWFfMSwgXFxwYXJyX3tuXzEgKyAxfShcXERlbHRhXzEsQSksXFxHYW1tYV8yLEIsXFxHYW1tYV8zKSxcXERlbHRhXzIpIl0sWzIsMywiXFxwYXJyX3sxK25fMn0oXFxwYXJyX3tuXzErMX0oXFxEZWx0YV8xLCBcXG90aW1lc19wKFxcR2FtbWFfMSwgQSwgXFxHYW1tYV8yLCBCLCBcXEdhbW1hXzMpKSxcXERlbHRhXzIpIl0sWzAsMSwiXFxkZWx0YV97bV8xLG5fMSwwLG1fMiArIDEgKyBtXzN9IiwxXSxbMSwyLCJcXHBhcnJfe25fMSArIDF9KFxcRGVsdGFfMSxcXGRlbHRhX3ttXzErMSttXzIsMCxuXzIsbV8zfSkiLDFdLFsyLDMsIlxcZ2FtbWEnX3tuXzEsMStuXzIsMH0iLDFdLFswLDQsIlxcZGVsdGFfe21fMSArIDEgKyBtXzIsMCxuXzIsbV8yfSIsMV0sWzQsNSwiXFxwYXJyX3sxK25fMn0oXFxkZWx0YV97bV8xLG5fMSwwLG1fMisxK21fM30sXFxEZWx0YV8yKSIsMV0sWzUsMywiXFxnYW1tYSdfezAsbl8xKzEsbl8yfSIsMV1d
\begin{itemize}
\item
\[\resizebox{\hsize}{!}{\begin{tikzcd}[ampersand replacement=\&]
	\& {\otimes_p(\Gamma_1, \parr_{n_1 + 1}(\Delta_1,A), \Gamma_2, \parr_{1 + n_2}(B, \Delta_2), \Gamma_3)} \\
	{\parr_{n_1+1}(\Delta_1, \otimes_p(\Gamma_1, A, \Gamma_2, \parr_{1 + n_2}(B,\Delta_2), \Gamma_3))} \&\& {\parr_{1+n_2}(\otimes_p(\Gamma_1, \parr_{n_1 + 1}(\Delta_1,A),\Gamma_2,B,\Gamma_3),\Delta_2)} \\
	\\
	{\parr_{n_1+1}(\Delta_1,\parr_{1+n_2}(\otimes_p(\Gamma_1, A, \Gamma_2, B, \Gamma_3),\Delta_2))} \&\& {\parr_{1+n_2}(\parr_{n_1+1}(\Delta_1, \otimes_p(\Gamma_1, A, \Gamma_2, B, \Gamma_3)),\Delta_2)} \\
	\& {\parr_{n_1+1+n_2}(\Delta_1, \otimes_p(\Gamma_1,A,\Gamma_2,B, \Gamma_3), \Delta_2)}
	\arrow["{\delta_{m_1,n_1,0,m_2 + 1 + m_3}}"{description}, from=1-2, to=2-1]
	\arrow["{\parr_{n_1 + 1}(\Delta_1,\delta_{m_1+1+m_2,0,n_2,m_3})}"{description}, from=2-1, to=4-1]
	\arrow["{\gamma'_{n_1,1+n_2,0}}"{description}, from=4-1, to=5-2]
	\arrow["{\delta_{m_1 + 1 + m_2,0,n_2,m_2}}"{description}, from=1-2, to=2-3]
	\arrow["{\parr_{1+n_2}(\delta_{m_1,n_1,0,m_2+1+m_3},\Delta_2)}"{description}, from=2-3, to=4-3]
	\arrow["{\gamma'_{0,n_1+1,n_2}}"{description}, from=4-3, to=5-2]
\end{tikzcd}}\]
where $p:= m_1 + 1 + m_2 + 1 + m_3$ and with $\gamma'_{n_1,n_2,n_3} \colon \parr_{n_1+1+n_3}(-, \parr_{n_2}(-),-) \Rightarrow \parr_{n_1 + n_2 + n_3}(-)$ defined by:
% https://q.uiver.app/?q=WzAsMyxbMCwwLCJcXHBhcnJfe25fMSsxK25fM30oQV8xLFxcZG90cyxBX3tuXzF9LCBcXHBhcnJfe25fMn0oXFxEZWx0YV8yKSxCXzEsXFxkb3RzLEJfe25fM30pIl0sWzAsMiwiXFxwYXJyX3tuXzErMStuXzN9KFxccGFycl8xKEFfMSksXFxkb3RzLFxccGFycl8xKEFfe25fMX0pLCBcXHBhcnJfe25fMn0oXFxEZWx0YV8yKSxcXHBhcnJfMShCXzEpLFxcZG90cyxcXHBhcnJfMShCX3tuXzN9KSkiXSxbMywyLCJcXHBhcnJfe25fMStuXzIrbl8zfShBXzEsXFxkb3RzLEFfe25fMX0sIFxcRGVsdGFfMixCXzEsXFxkb3RzLEJfe25fM30pIl0sWzAsMSwiXFxwYXJyKFxcZXRhX3tBXzF9LFxcZG90cyxcXGV0YV97QV9ufSwgXFxwYXJyKFxcRGVsdGFfMiksIFxcZXRhX3tCXzF9LFxcZG90c1xcZXRhX3tCX3tuXzJ9fSkiLDFdLFsxLDIsIlxcZ2FtbWFfe25fMSsxK25fMywoMSxcXGRvdHMsMSxuXzIsMSxcXGRvdHMsMSl9IiwxXSxbMCwyLCJcXGdhbW1hJ197bl8xLG5fMixuXzN9IiwxXV0=
\[\resizebox{\hsize}{!}{\begin{tikzcd}[ampersand replacement=\&]
	{\parr_{n_1+1+n_3}(A_1,\dots,A_{n_1}, \parr_{n_2}(\Delta_2),B_1,\dots,B_{n_3})} \\
	\\
	{\parr_{n_1+1+n_3}(\parr_1(A_1),\dots,\parr_1(A_{n_1}), \parr_{n_2}(\Delta_2),\parr_1(B_1),\dots,\parr_1(B_{n_3}))} \&\&\& {\parr_{n_1+n_2+n_3}(A_1,\dots,A_{n_1}, \Delta_2,B_1,\dots,B_{n_3})}
	\arrow["{\parr(\eta_{A_1},\dots,\eta_{A_n}, \parr(\Delta_2), \eta_{B_1},\dots\eta_{B_{n_2}})}"{description}, from=1-1, to=3-1]
	\arrow["{\gamma}"{description}, from=3-1, to=3-4]
	\arrow["{\gamma'_{n_1,n_2,n_3}}"{description}, from=1-1, to=3-4]
\end{tikzcd}}\]

% https://q.uiver.app/?q=WzAsNSxbMCwwLCJcXG90aW1lcyhcXEdhbW1hXzEsXFxkb3RzLFxcR2FtbWFfcCxcXEdhbW1hLFxccGFycihcXERlbHRhLEEsXFxEZWx0YScpLFxcR2FtbWEnLCBcXEdhbW1hXzEnLFxcZG90cyxcXEdhbW1hJ197cCd9KSJdLFswLDQsIlxcb3RpbWVzKFxcb3RpbWVzXFxHYW1tYV8xLFxcZG90cyxcXG90aW1lc1xcR2FtbWFfcCxcXG90aW1lcyhcXEdhbW1hLFxccGFycihcXERlbHRhLEEsXFxEZWx0YScpLFxcR2FtbWEnKSxcXG90aW1lc1xcR2FtbWEnXzEsXFxkb3RzLFxcb3RpbWVzXFxHYW1tYSdfe3AnfSkiXSxbMyw0LCJcXG90aW1lcyhcXG90aW1lc1xcR2FtbWFfMSxcXGRvdHMsXFxvdGltZXNcXEdhbW1hX3AsXFxwYXJyKFxcRGVsdGEsXFxvdGltZXMoXFxHYW1tYSxBLFxcR2FtbWEnKSxcXERlbHRhJyksXFxvdGltZXNcXEdhbW1hJ18xLFxcZG90cyxcXG90aW1lc1xcR2FtbWEnX3twJ30pIl0sWzUsMiwiXFxwYXJyKFxcRGVsdGEsXFxvdGltZXMoXFxvdGltZXNcXEdhbW1hXzEsXFxkb3RzLFxcb3RpbWVzXFxHYW1tYV9wLFxcb3RpbWVzKFxcR2FtbWEsQSxcXEdhbW1hJyksXFxvdGltZXNcXEdhbW1hXzEnLFxcZG90cyxcXG90aW1lc1xcR2FtbWEnX3twJ30pLFxcRGVsdGEnKSJdLFszLDAsIlxccGFycihcXERlbHRhLFxcb3RpbWVzKFxcR2FtbWFfMSxcXGRvdHMsXFxHYW1tYV9wLFxcR2FtbWEsQSxcXEdhbW1hJyxcXEdhbW1hJ18xLFxcZG90cyxcXEdhbW1hJ197cCd9KSxcXERlbHRhJykiXSxbMCwxLCJcXGFscGhhX3twKzErcCcsKG1fMSxcXGRvdHMsbV9wLG0rMSttJyxtJ18xLFxcZG90cyxtJ197cCd9KX0iLDFdLFsxLDIsIlxcb3RpbWVzKC0sXFxkZWx0YV97bSxuLG4nLG0nfSwtKSIsMV0sWzIsMywiXFxkZWx0YV97cCxuLG4nLHAnfSIsMV0sWzAsNCwiXFxkZWx0YV97XFxzdW1fayBtX2sgKyBtLG4sbicsbScgKyBcXHN1bV9rIG0nX2t9IiwxXSxbNCwzLCJcXHBhcnIoXFxEZWx0YSxcXGFscGhhX3twKzErcCcsKG1fMSxcXGRvdHMsbV9wLG0rMSttJyxtJ18xLFxcZG90cyxtJ197cCd9KX0sXFxEZWx0YScpIiwxXV0=
\item
% https://q.uiver.app/?q=WzAsNSxbMCwwLCJcXG90aW1lcyhcXEdhbW1hXzEsXFxkb3RzLFxcR2FtbWFfcCxcXEdhbW1hLFxccGFycihcXERlbHRhLEEsXFxEZWx0YScpLFxcR2FtbWEnLCBcXEdhbW1hXzEnLFxcZG90cyxcXEdhbW1hJ197cCd9KSJdLFswLDIsIlxcb3RpbWVzKFxcb3RpbWVzXFxHYW1tYV8xLFxcZG90cyxcXG90aW1lc1xcR2FtbWFfcCxcXG90aW1lcyhcXEdhbW1hLFxccGFycihcXERlbHRhLEEsXFxEZWx0YScpLFxcR2FtbWEnKSxcXG90aW1lc1xcR2FtbWEnXzEsXFxkb3RzLFxcb3RpbWVzXFxHYW1tYSdfe3AnfSkiXSxbMiwyLCJcXG90aW1lcyhcXG90aW1lc1xcR2FtbWFfMSxcXGRvdHMsXFxvdGltZXNcXEdhbW1hX3AsXFxwYXJyKFxcRGVsdGEsXFxvdGltZXMoXFxHYW1tYSxBLFxcR2FtbWEnKSxcXERlbHRhJyksXFxvdGltZXNcXEdhbW1hJ18xLFxcZG90cyxcXG90aW1lc1xcR2FtbWEnX3twJ30pIl0sWzMsMSwiXFxwYXJyKFxcRGVsdGEsXFxvdGltZXMoXFxvdGltZXNcXEdhbW1hXzEsXFxkb3RzLFxcb3RpbWVzXFxHYW1tYV9wLFxcb3RpbWVzKFxcR2FtbWEsQSxcXEdhbW1hJyksXFxvdGltZXNcXEdhbW1hXzEnLFxcZG90cyxcXG90aW1lc1xcR2FtbWEnX3twJ30pLFxcRGVsdGEnKSJdLFsyLDAsIlxccGFycihcXERlbHRhLFxcb3RpbWVzKFxcR2FtbWFfMSxcXGRvdHMsXFxHYW1tYV9wLFxcR2FtbWEsQSxcXEdhbW1hJyxcXEdhbW1hJ18xLFxcZG90cyxcXEdhbW1hJ197cCd9KSxcXERlbHRhJykiXSxbMCwxLCJcXGFscGhhIiwxXSxbMSwyLCJcXG90aW1lcygtLFxcZGVsdGEsLSkiLDFdLFsyLDMsIlxcZGVsdGEiLDFdLFswLDQsIlxcZGVsdGEiLDFdLFs0LDMsIlxccGFycigtLFxcYWxwaGEsLSkiLDFdXQ==
\[\resizebox{\hsize}{!}{\begin{tikzcd}[ampersand replacement=\&]
	{\otimes(\Gamma_1,\dots,\Gamma_p,\Gamma,\parr(\Delta,A,\Delta'),\Gamma', \Gamma_1',\dots,\Gamma'_{p'})} \&\& {\parr(\Delta,\otimes(\Gamma_1,\dots,\Gamma_p,\Gamma,A,\Gamma',\Gamma'_1,\dots,\Gamma'_{p'}),\Delta')} \\
	\&\&\& {\parr(\Delta,\otimes(\otimes\Gamma_1,\dots,\otimes\Gamma_p,\otimes(\Gamma,A,\Gamma'),\otimes\Gamma_1',\dots,\otimes\Gamma'_{p'}),\Delta')} \\
	{\otimes(\otimes\Gamma_1,\dots,\otimes\Gamma_p,\otimes(\Gamma,\parr(\Delta,A,\Delta'),\Gamma'),\otimes\Gamma'_1,\dots,\otimes\Gamma'_{p'})} \&\& {\otimes(\otimes\Gamma_1,\dots,\otimes\Gamma_p,\parr(\Delta,\otimes(\Gamma,A,\Gamma'),\Delta'),\otimes\Gamma'_1,\dots,\otimes\Gamma'_{p'})}
	\arrow["\alpha"{description}, from=1-1, to=3-1]
	\arrow["{\otimes(-,\delta,-)}"{description}, from=3-1, to=3-3]
	\arrow["\delta"{description}, from=3-3, to=2-4]
	\arrow["\delta"{description}, from=1-1, to=1-3]
	\arrow["{\parr(-,\alpha,-)}"{description}, from=1-3, to=2-4]
\end{tikzcd}}\]
\item
% https://q.uiver.app/?q=WzAsNixbMiwwLCJcXG90aW1lc18xKFxccGFycl97bl8xKzErbl8yfShcXERlbHRhXzEsQSxcXERlbHRhXzIpKSJdLFsyLDIsIlxccGFycl97bl8xKzErbl8yfShcXERlbHRhXzEsXFxvdGltZXNfMShBKSxcXERlbHRhXzIpIl0sWzAsMiwiXFxwYXJyX3tuXzErMStuXzJ9KFxcRGVsdGFfMSxBLFxcRGVsdGFfMikiXSxbMywyLCJcXG90aW1lc197bV8xKzErbV8yfShcXEdhbW1hXzEsXFxwYXJyXzEoQSksXFxHYW1tYV8yKSJdLFs1LDIsIlxccGFycl8xKFxcb3RpbWVzX3ttXzErMSttXzJ9KFxcR2FtbWFfMSxBLFxcR2FtbWFfMikpIl0sWzMsMCwiXFxvdGltZXNfe21fMSsxK21fMn0oXFxHYW1tYV8xLEEsXFxHYW1tYV8yKSJdLFswLDEsIlxcZGVsdGEiLDFdLFsxLDIsIlxccGFycigtLFxcZXRhLC0pIiwxXSxbMCwyLCJcXGV0YSIsMV0sWzMsNCwiXFxkZWx0YSIsMV0sWzUsNCwiXFxpb3RhIiwxXSxbNSwzLCJcXG90aW1lcygtLFxcaW90YSwtKSIsMV1d
\[\resizebox{\hsize}{!}{\begin{tikzcd}[ampersand replacement=\&]
	\&\& {\otimes_1(\parr_{n_1+1+n_2}(\Delta_1,A,\Delta_2))} \& {\otimes_{m_1+1+m_2}(\Gamma_1,A,\Gamma_2)} \\
	\\
	{\parr_{n_1+1+n_2}(\Delta_1,A,\Delta_2)} \&\& {\parr_{n_1+1+n_2}(\Delta_1,\otimes_1(A),\Delta_2)} \& {\otimes_{m_1+1+m_2}(\Gamma_1,\parr_1(A),\Gamma_2)} \&\& {\parr_1(\otimes_{m_1+1+m_2}(\Gamma_1,A,\Gamma_2))}
	\arrow["\delta"{description}, from=1-3, to=3-3]
	\arrow["{\parr(-,\eta,-)}"{description}, from=3-3, to=3-1]
	\arrow["\eta"{description}, from=1-3, to=3-1]
	\arrow["\delta"{description}, from=3-4, to=3-6]
	\arrow["\iota"{description}, from=1-4, to=3-6]
	\arrow["{\otimes(-,\iota,-)}"{description}, from=1-4, to=3-4]
\end{tikzcd}}\]
\end{itemize}
\end{definition}

\begin{definition}
A \defin{lax normal unbiased linearly distributive category} is a luldc where the natural transformations $\eta \colon \otimes_1(-) \Rightarrow -$ and $\iota \colon - \Rightarrow \parr_1(-)$ are natural isomorphisms.
\end{definition}

\begin{definition}
A \defin{unbiased linearly distributive category} is a luldc that is normal and such that the associators $\alpha$ and $\gamma$ are all isomorphisms.
\end{definition}

In other words an unbiased ldc is a luldc where $\otimes_n$ and $\parr_n$ define monoidal structures.

\begin{prop}
There is a correspondence between unbiased and biased ldc.
\end{prop}
\begin{proof}
The correspondence between the underlying monoidal categories follows from the correspondence between unbiased and biased monoidal categories.
So all that is needed is to exhibit the connection between the distributivity laws.
The proof is left to the reader.
\end{proof}

\begin{definition}
In a luldc $(\cC, \otimes_n, \parr_n)$, a \defin{left dual} of an object $A$, is an object $\ldual{A}$ equipped with morphisms $\lcup_A \colon \otimes_0 \to \parr_2(\ldual{A}, A)$ and $\lcap_A \colon \otimes_2(A,\ldual{A}) \to \parr_0$ such that the following coherence laws hold:
% https://q.uiver.app/?q=WzAsOSxbMCwwLCJcXG90aW1lc18xXFxsZHVhbHtBfSJdLFsyLDAsIlxcb3RpbWVzXzIoXFxvdGltZXNfMCxcXG90aW1lc18xXFxsZHVhbHtBfSkiXSxbNywwLCJcXG90aW1lc18yKFxccGFycl8yKFxcbGR1YWx7QX0sQSksXFxsZHVhbHtBfSkiXSxbMTAsMCwiXFxwYXJyXzIoXFxsZHVhbHtBfSxcXG90aW1lc18yKEEsXFxsZHVhbHtBfSkpIl0sWzEwLDIsIlxccGFycl8yKFxcbGR1YWx7QX0sXFxwYXJyXzApIl0sWzEwLDQsIlxccGFycl8yKFxccGFycl8xIFxcbGR1YWx7QX0sXFxwYXJyXzApIl0sWzEwLDYsIlxccGFycl8xXFxsZHVhbHtBfSJdLFs1LDAsIlxcb3RpbWVzXzIoXFxvdGltZXNfMCxcXGxkdWFse0F9KSJdLFswLDYsIlxcbGR1YWx7QX0iXSxbMCwxLCJcXGFscGhhX3syLCgwLDEpfSIsMV0sWzIsMywiXFxkZWx0YV97MCwxLDAsMX0iLDFdLFs0LDUsIlxccGFycl8yKFxcaW90YSwtKSIsMV0sWzUsNiwiXFxnYW1tYV97MiwoMSwwKX0iLDFdLFsxLDcsIlxcb3RpbWVzXzIoLSxcXGV0YSkiLDFdLFs3LDIsIlxcb3RpbWVzXzIoXFxsY3VwX0EsLSkiLDFdLFszLDQsIlxccGFycl8yKC0sXFxsY2FwX0EpIiwxXSxbMCw4LCJcXGV0YSIsMV0sWzgsNiwiXFxpb3RhIiwxXV0=
\[\resizebox{\hsize}{!}{\begin{tikzcd}[ampersand replacement=\&]
	{\otimes_1\ldual{A}} \&\& {\otimes_2(\otimes_0,\otimes_1\ldual{A})} \&\&\& {\otimes_2(\otimes_0,\ldual{A})} \&\& {\otimes_2(\parr_2(\ldual{A},A),\ldual{A})} \&\&\& {\parr_2(\ldual{A},\otimes_2(A,\ldual{A}))} \\
	\\
	\&\&\&\&\&\&\&\&\&\& {\parr_2(\ldual{A},\parr_0)} \\
	\\
	\&\&\&\&\&\&\&\&\&\& {\parr_2(\parr_1 \ldual{A},\parr_0)} \\
	\\
	{\ldual{A}} \&\&\&\&\&\&\&\&\&\& {\parr_1\ldual{A}}
	\arrow["{\alpha_{2,(0,1)}}"{description}, from=1-1, to=1-3]
	\arrow["{\delta_{0,1,0,1}}"{description}, from=1-8, to=1-11]
	\arrow["{\parr_2(\iota,-)}"{description}, from=3-11, to=5-11]
	\arrow["{\gamma_{2,(1,0)}}"{description}, from=5-11, to=7-11]
	\arrow["{\otimes_2(-,\eta)}"{description}, from=1-3, to=1-6]
	\arrow["{\otimes_2(\lcup_A,-)}"{description}, from=1-6, to=1-8]
	\arrow["{\parr_2(-,\lcap_A)}"{description}, from=1-11, to=3-11]
	\arrow["\eta"{description}, from=1-1, to=7-1]
	\arrow["\iota"{description}, from=7-1, to=7-11]
\end{tikzcd}}\]

% https://q.uiver.app/?q=WzAsOSxbMCwwLCJcXG90aW1lc18xIEEiXSxbMiwwLCJcXG90aW1lc18yKFxcb3RpbWVzXzEgQSwgXFxvdGltZXNfMCkiXSxbNywwLCJcXG90aW1lc18yKEEsXFxwYXJyXzIoXFxsZHVhbHtBfSxBKSkiXSxbMTAsMCwiXFxwYXJyXzIoXFxvdGltZXNfMihcXEEsIFxcbGR1YWx7QX0pLCBBKSJdLFsxMCwyLCJcXHBhcnJfMihcXHBhcnJfMCxBKSJdLFsxMCw0LCJcXHBhcnJfMihcXHBhcnIwLFxccGFycl8xIEEpIl0sWzEwLDYsIlxccGFycl8xIEEiXSxbNSwwLCJcXG90aW1lc18yKEEsIFxcb3RpbWVzXzApIl0sWzAsNiwiQSJdLFswLDEsIlxcYWxwaGFfezIsKDEsMCl9IiwxXSxbMiwzLCJcXGRlbHRhX3sxLDAsMSwwfSIsMV0sWzQsNSwiXFxwYXJyXzIoLSxcXGlvdGEpIiwxXSxbNSw2LCJcXGdhbW1hX3syLCgwLDEpfSIsMV0sWzEsNywiXFxvdGltZXNfMihcXGV0YSwtKSIsMV0sWzcsMiwiXFxvdGltZXNfMigtLFxcbGN1cF9BKSIsMV0sWzMsNCwiXFxwYXJyXzIoXFxsY2FwX0EsLSkiLDFdLFswLDgsIlxcZXRhIiwxXSxbOCw2LCJcXGlvdGEiLDFdXQ==
\[\resizebox{\hsize}{!}{\begin{tikzcd}[ampersand replacement=\&]
	{\otimes_1 A} \&\& {\otimes_2(\otimes_1 A, \otimes_0)} \&\&\& {\otimes_2(A, \otimes_0)} \&\& {\otimes_2(A,\parr_2(\ldual{A},A))} \&\&\& {\parr_2(\otimes_2(A, \ldual{A}), A)} \\
	\\
	\&\&\&\&\&\&\&\&\&\& {\parr_2(\parr_0,A)} \\
	\\
	\&\&\&\&\&\&\&\&\&\& {\parr_2(\parr_0,\parr_1 A)} \\
	\\
	A \&\&\&\&\&\&\&\&\&\& {\parr_1 A}
	\arrow["{\alpha_{2,(1,0)}}"{description}, from=1-1, to=1-3]
	\arrow["{\delta_{1,0,1,0}}"{description}, from=1-8, to=1-11]
	\arrow["{\parr_2(-,\iota)}"{description}, from=3-11, to=5-11]
	\arrow["{\gamma_{2,(0,1)}}"{description}, from=5-11, to=7-11]
	\arrow["{\otimes_2(\eta,-)}"{description}, from=1-3, to=1-6]
	\arrow["{\otimes_2(-,\lcup_A)}"{description}, from=1-6, to=1-8]
	\arrow["{\parr_2(\lcap_A,-)}"{description}, from=1-11, to=3-11]
	\arrow["\eta"{description}, from=1-1, to=7-1]
	\arrow["\iota"{description}, from=7-1, to=7-11]
\end{tikzcd}}\]

A \defin{right dual} to an object $A$ is an object $\rdual{A}$ equipped with morphisms $\otimes_0 \to \parr_2(A, \rdual{A})$ and $\otimes_2(\rdual{A},A) \to \parr_0$ with coherence laws symmetric to the ones for left duals.
\end{definition}

\begin{rk}
The coherence laws are a refinement of the so-called snake identities for duals in a compact close category.
\end{rk}

Left and right duals are unique up to unique isomorphisms.
The proof is similar to the usual one, for example in compact closed categories.
We will see a version of it in the case of polycategories later.

\begin{definition}
A \defin{lax unbiased $\ast$-autonomous category} is a lax unbiased linearly distributive category with all left and right duals.
\end{definition}

Lax unbiased linearly distributive categories are more general than their biased analogs.
In fact, when dealing with duals, lax unbiased linearly distributive categories are too general.
This is because to use the cup and cap we want to be able to introduce/eliminate the tensor/parr unit, i.e., we rely on the morphisms $A \to I \otimes A$, $\bot \parr A \to A$ and their right versions.
However, these do not exist in a general lax unbiased linearly distributive category.
The best one can hope for are $\otimes_1 A \to \otimes_2(\otimes_0, A)$ that we get by $\otimes_1 A \xrightarrow{\alpha} \otimes_2(\otimes_0, \otimes_1 A) \xrightarrow{\otimes_2(-,\eta)} \otimes_2(\otimes_0, A)$ and similarly $\parr_2(\parr_0, A) \to \parr_1 A$.
So most of the usual morphisms that we get in a $\ast$-autonomous category will exist in a lax unbiased $\ast$-autonomous category only when decorated with $\otimes_1$ and $\parr_1$.
For example, from a morphism $f \colon A \to B$ we only get a morphism $\ldual{f} \colon \otimes_1 B \to \parr_1 A$.
Similarly, we only have morphisms $\otimes_1 A \to \parr_1 \ldual{(\rdual{A})}$ and $\otimes_1 \ldual{(\rdual{A})} \to \parr_1 A$, to name a few.
This is also why in the coherence law for the duals we had to replace the single $A$ by $\otimes_1 A$.
So in order to get a  notion extending the familiar theory of $\ast$-autonomous category, from now on, we will consider lax normal $\ast$-autonomous categories, i.e. ones where $\eta$ and $\iota$ are invertible giving $\otimes_1 A \simeq A \simeq \parr_1 A$.

\begin{prop}
In a lax normal $\ast$-autonomous category, $\ldual{-}$ and $\rdual{-}$ extend to contravariant functors.
\end{prop}
\begin{proof}
Given $f \colon B \to A$, we define \[\begin{aligned}
\ldual{f} \colon \ldual{A} &\xrightarrow{\eta^{-1}} \otimes_1 \ldual{A}\\
&\xrightarrow{\alpha} \otimes_2(\otimes_0, \otimes_1 \ldual{A}) \\
&\xrightarrow{\otimes_2(-,\eta)} \otimes_2(\otimes_0, \ldual{A}) \\
&\xrightarrow{\otimes_2(\lcup_B, -)} \otimes_2(\parr_2(\ldual{B}, B), \ldual{A}) \\
&\xrightarrow{\otimes_2(\parr_2(-,f),-)} \otimes_2(\parr_2(\ldual{B}, A), \ldual{A}) \\
&\xrightarrow{\delta} \parr_2(\ldual{B}, \otimes_2(A, \ldual{A})) \\
&\xrightarrow{\parr_2(-,\lcap_A)} \parr_2(\ldual{B}, \parr_0) \\
&\xrightarrow{\parr_2(\iota,-)} \parr_2(\parr_1 \ldual{B}, \parr_0) \\
&\xrightarrow{\gamma} \parr_1 \ldual{B} \\
&\xrightarrow{\iota^{-1}} \ldual{B}
\end{aligned}\]

This is functorial.

$\ldual{\id_A}$ is by definition:
% https://q.uiver.app/?q=WzAsMTIsWzAsMCwiXFxsZHVhbHtBfSJdLFsyLDAsIlxcb3RpbWVzXzEgXFxsZHVhbHtBfSJdLFs0LDAsIlxcb3RpbWVzXzIoXFxvdGltZXNfMCwgXFxvdGltZXNfMVxcbGR1YWx7QX0pIl0sWzYsMCwiXFxvdGltZXNfMihcXG90aW1lc18wLFxcbGR1YWx7QX0pIl0sWzgsMCwiXFxvdGltZXNfMihcXHBhcnJfMihcXGxkdWFse0F9LEEpLFxcbGR1YWx7QX0pIl0sWzgsMiwiXFxvdGltZXNfMihcXHBhcnJfMihcXGxkdWFse0F9LEEpLFxcbGR1YWx7QX0pIl0sWzgsNCwiXFxwYXJyXzIoXFxsZHVhbHtBfSxcXG90aW1lc18yKEEsXFxsZHVhbHtBfSkpIl0sWzgsNiwiXFxwYXJyXzIoXFxsZHVhbHtBfSxcXHBhcnJfMCkiXSxbOCw4LCJcXHBhcnJfMihcXHBhcnJfMCBcXGxkdWFse0F9LFxccGFycl8wKSJdLFs4LDEwLCJcXHBhcnJfMVxcbGR1YWx7QX0iXSxbOCwxMiwiXFxsZHVhbHtBfSJdLFsyLDEwLCJcXGxkdWFse0F9Il0sWzAsMSwiXFxldGFeey0xfSIsMV0sWzEsMiwiXFxhbHBoYSIsMV0sWzIsMywiXFxvdGltZXNfMigtLFxcZXRhKSIsMV0sWzMsNCwiXFxvdGltZXNfMihcXGxjdXBfQSwtKSIsMV0sWzQsNSwiXFxvdGltZXNfMihcXHBhcnJfMigtLFxcaWRfQSksLSkiLDFdLFs1LDYsIlxcZGVsdGEiLDFdLFs2LDcsIlxccGFycl8yKC0sXFxsY2FwX0EpIiwxXSxbNyw4LCJcXHBhcnJfMihcXGlvdGEsLSkiLDFdLFs4LDksIlxcZ2FtbWEiLDFdLFs5LDEwLCJcXGlvdGFeezF9IiwxXSxbMSwxMSwiXFxldGEiLDFdLFsxMSw5LCJcXGlvdGEiLDFdLFswLDExLCIiLDEseyJsZXZlbCI6Miwic3R5bGUiOnsiaGVhZCI6eyJuYW1lIjoibm9uZSJ9fX1dLFsxMSwxMCwiIiwxLHsibGV2ZWwiOjIsInN0eWxlIjp7ImhlYWQiOnsibmFtZSI6Im5vbmUifX19XV0=
\[\resizebox{\hsize}{!}{\begin{tikzcd}[ampersand replacement=\&]
	{\ldual{A}} \&\& {\otimes_1 \ldual{A}} \&\& {\otimes_2(\otimes_0, \otimes_1\ldual{A})} \&\& {\otimes_2(\otimes_0,\ldual{A})} \&\& {\otimes_2(\parr_2(\ldual{A},A),\ldual{A})} \\
	\\
	\&\&\&\&\&\&\&\& {\otimes_2(\parr_2(\ldual{A},A),\ldual{A})} \\
	\\
	\&\&\&\&\&\&\&\& {\parr_2(\ldual{A},\otimes_2(A,\ldual{A}))} \\
	\\
	\&\&\&\&\&\&\&\& {\parr_2(\ldual{A},\parr_0)} \\
	\\
	\&\&\&\&\&\&\&\& {\parr_2(\parr_0 \ldual{A},\parr_0)} \\
	\\
	\&\& {\ldual{A}} \&\&\&\&\&\& {\parr_1\ldual{A}} \\
	\\
	\&\&\&\&\&\&\&\& {\ldual{A}}
	\arrow["{\eta^{-1}}"{description}, from=1-1, to=1-3]
	\arrow["\alpha"{description}, from=1-3, to=1-5]
	\arrow["{\otimes_2(-,\eta)}"{description}, from=1-5, to=1-7]
	\arrow["{\otimes_2(\lcup_A,-)}"{description}, from=1-7, to=1-9]
	\arrow["{\otimes_2(\parr_2(-,\id_A),-)}"{description}, from=1-9, to=3-9]
	\arrow["\delta"{description}, from=3-9, to=5-9]
	\arrow["{\parr_2(-,\lcap_A)}"{description}, from=5-9, to=7-9]
	\arrow["{\parr_2(\iota,-)}"{description}, from=7-9, to=9-9]
	\arrow["\gamma"{description}, from=9-9, to=11-9]
	\arrow["{\iota^{1}}"{description}, from=11-9, to=13-9]
	\arrow["\eta"{description}, from=1-3, to=11-3]
	\arrow["\iota"{description}, from=11-3, to=11-9]
	\arrow[Rightarrow, no head, from=1-1, to=11-3]
	\arrow[Rightarrow, no head, from=11-3, to=13-9]
\end{tikzcd}}\]

To prove that $\ldual{(f \circ g)} = \
\ldual{g} \circ \ldual{f}$ for $f \colon B \to A$ and $g \colon C \to B$ we need to prove that the outer rectangle commute:

\[\scalebox{0.7}{\begin{tikzcd}[ampersand replacement=\&]
	{\ldual{A}} \& {\otimes_2(\otimes_0,\ldual{A})} \& {\otimes_2(\parr_2(\ldual{B},B),\ldual{A})} \&\& {\otimes_2(\parr_2(\ldual{B},A),\ldual{A})} \\
	\\
	\&\&\&\& {\parr_2(\ldual{B},\otimes_2(A,\ldual{A}))} \\
	{\otimes_2(\otimes_0,\ldual{A})} \\
	\&\&\&\& {\parr_2(\ldual{B},\parr_0)} \\
	\&\& {\otimes_2(\parr_2(\ldual{C},\otimes_2(B,\otimes_0)),\ldual{A})} \\
	\&\&\&\& {\ldual{B}} \\
	{\otimes_2(\parr_2(\ldual{C},C),\ldual{A})} \&\&\& {\otimes_2(\parr_2(\ldual{C},\otimes_2(B,\parr_2(\ldual{B},B))),\ldual{A})} \\
	\&\&\&\& {\otimes_2(\otimes_0,\ldual{B})} \\
	\& {\otimes_2(\parr_2(\ldual{C},B),\ldual{A})} \\
	\&\&\&\& {\otimes_2(\parr_2(\ldual{C},C),\ldual{B})} \\
	{\otimes_2(\parr_2(\ldual{C},A),\ldual{A})} \&\&\& {\otimes_2(\parr_2(\ldual{C},\parr_2(\otimes_2(B,\ldual{B}),B)),\ldual{A})} \\
	\&\&\&\& {\otimes_2(\parr_2(\ldual{C},B),\ldual{B})} \\
	\&\& {\otimes_2(\parr_2(\ldual{C},\parr_2(\parr_0,B)),\ldual{A})} \\
	\&\&\&\& {\parr_2(\ldual{C},\otimes_2(B,\ldual{B}))} \\
	\\
	\&\&\&\& {\parr_2(\ldual{C},\parr_0)} \\
	\\
	{\parr_2(\ldual{C},\otimes_2(A,\ldual{A}))} \&\& {\parr_2(\ldual{C},\parr_0)} \&\& {\ldual{C}}
	\arrow["\lambda"{description}, from=1-1, to=1-2]
	\arrow["{\otimes_2(\lcup_B,-)}"{description}, from=1-2, to=1-3]
	\arrow["{\otimes_2(\parr_2(-,f),-)}"{description}, from=1-3, to=1-5]
	\arrow["\delta"{description}, from=1-5, to=3-5]
	\arrow["{\parr_2(-,\lcap_A)}"{description}, from=3-5, to=5-5]
	\arrow["{\rho'}"{description}, from=5-5, to=7-5]
	\arrow["\lambda"{description}, from=7-5, to=9-5]
	\arrow["{\otimes_2(\lcup_C,-)}"{description}, from=9-5, to=11-5]
	\arrow["{\otimes_2(\parr_2(-,g),-)}"{description}, from=11-5, to=13-5]
	\arrow["\delta"{description}, from=13-5, to=15-5]
	\arrow["{\parr_2(-,\lcap_B)}"{description}, from=15-5, to=17-5]
	\arrow["{\rho'}"{description}, from=17-5, to=19-5]
	\arrow["\lambda"{description}, from=1-1, to=4-1]
	\arrow["{\otimes_2(\lcup_C,-)}"{description}, from=4-1, to=8-1]
	\arrow["{\otimes_2(\parr_2(-,f \circ g),-)}"{description}, from=8-1, to=12-1]
	\arrow["\delta"{description}, from=12-1, to=19-1]
	\arrow["{\parr_2(-,\lcap_A)}"{description}, from=19-1, to=19-3]
	\arrow["{\rho'}"{description}, from=19-3, to=19-5]
	\arrow["{\otimes_2(\parr_2(-,\rho),-)}"{description}, from=10-2, to=6-3]
	\arrow["{\otimes_2(\parr_2(-,\otimes_2(-,\lcup_B)),-)}"{description}, from=6-3, to=8-4]
	\arrow["{\otimes_2(\parr_2(-,\delta),-)}"{description}, from=8-4, to=12-4]
	\arrow["{\otimes_2(\parr_2(-,\parr_2(\lcap_B,-)),-)}"{description}, from=12-4, to=14-3]
	\arrow["{\otimes_2(\parr_2(-,\lambda'),-)}"{description}, from=14-3, to=10-2]
	\arrow["{\otimes_2(\parr_2(-,f),-)}"{description}, from=10-2, to=12-1]
	\arrow["{\otimes_2(\parr_2(-,g),-)}"{description}, from=8-1, to=10-2]
\end{tikzcd}}\]

This is a huge diagram chasing where we go from the topmost path to the interior path by using naturality of the different natural transformations, functoriality of $\otimes_2$ and $\parr_2$ and the coherence laws of a luldc relating $\delta$ and $\eta$/$\iota$.
Then we use the coherence law of a lax $\ast$-autonomous category to get rid of the internal pentagon.

In term of string diagrams the first part would correspond to sliding $f,g$ and the cups and caps and the second part to using the snake identity, in the following way.

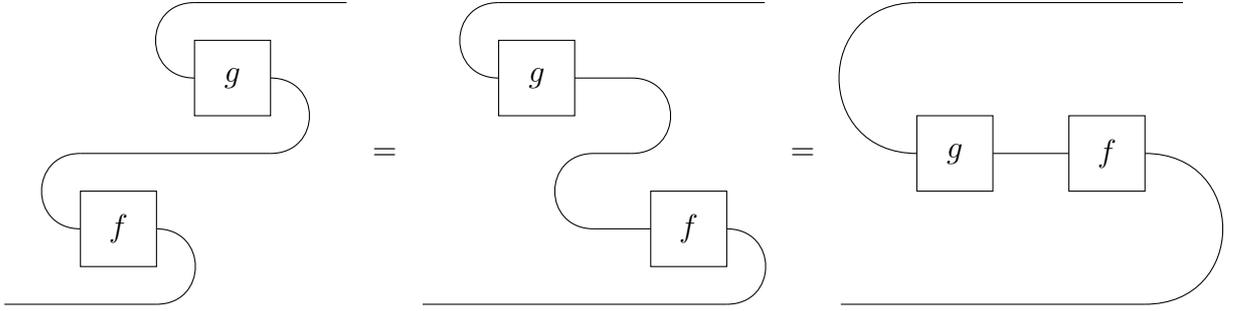
\center\input{lax-autonomous-functorial.tikz} 

The proof for functoriality of $\rdual{(-)}$ is similar.
\end{proof}

There are different ways a functor can interact with lax/oplax monoidal structures.
The idea will be to express that the functor preserves the monoidal product.
However this preservation can be defined on the nose, up to iso, or even in a specific direction.
The different definitions are the same as the one for usual monoidal categories.
In particular, the direction of the preservation of the monoidal product is unrelated to that of the laws the monoidal structure.
So it is possible to define oplax monoidal functors between lax monoidal categories.

\begin{definition}
A \defin{lax monoidal functor} between lax monoidal categories \[(F, F_n) \colon (\cC,\otimes_n^{\cC}, \alpha^{\cC},\iota^{\cC}) \to (\cD, \otimes_n^{\cD}, \alpha^{\cD}, \iota^{\cD})\] is a functor $F \colon \cC \to \cD$ together with natural transformations \[F_n \colon \otimes_n^{\cD}(F(-)) \Rightarrow F(\otimes_n^{\cC}(-))\] such that the following diagrams commute:
% https://q.uiver.app/?q=WzAsNSxbMCwwLCJcXG90aW1lc19uXFxvdGltZXNfe21faX1GKEFfe2ksan0pIl0sWzIsMCwiXFxvdGltZXNfe1xcc3VtX2kgbV9pfUYoQV97aSxqfSkiXSxbNCwxLCJGKFxcb3RpbWVzX3tcXHN1bV9pIG1faX1BX3tpLGp9KSJdLFswLDIsIlxcb3RpbWVzX25GKFxcb3RpbWVzX3ttX2l9QV97aSxqfSkiXSxbMiwyLCJGKFxcb3RpbWVzX25cXG90aW1lc197bV9pfUFfe2ksan0pIl0sWzAsMSwiXFxhbHBoYSIsMV0sWzEsMiwiRl97XFxzdW1faSBtX2l9IiwxXSxbMCwzLCJcXG90aW1lc19uIEZfe21faX0iLDFdLFszLDQsIkZfbiIsMV0sWzQsMiwiRihcXGFscGhhKSIsMV1d
\[\begin{tikzcd}[ampersand replacement=\&]
	{\otimes_n\otimes_{m_i}F(A_{i,j})} \&\& {\otimes_{\sum_i m_i}F(A_{i,j})} \\
	\&\&\&\& {F(\otimes_{\sum_i m_i}A_{i,j})} \\
	{\otimes_nF(\otimes_{m_i}A_{i,j})} \&\& {F(\otimes_n\otimes_{m_i}A_{i,j})}
	\arrow["\alpha"{description}, from=1-1, to=1-3]
	\arrow["{F_{\sum_i m_i}}"{description}, from=1-3, to=2-5]
	\arrow["{\otimes_n F_{m_i}}"{description}, from=1-1, to=3-1]
	\arrow["{F_n}"{description}, from=3-1, to=3-3]
	\arrow["{F(\alpha)}"{description}, from=3-3, to=2-5]
\end{tikzcd}\]
% https://q.uiver.app/?q=WzAsMyxbMCwwLCJGKEEpIl0sWzIsMCwiXFxvdGltZXNfMUYoQSkiXSxbMiwyLCJGKFxcb3RpbWVzXzEgQSkiXSxbMCwxLCJcXGlvdGEiLDFdLFsxLDIsIkZfMSIsMV0sWzAsMiwiRihcXGlvdGEpIiwxXV0=
\[\begin{tikzcd}[ampersand replacement=\&]
	{F(A)} \&\& {\otimes_1F(A)} \\
	\\
	\&\& {F(\otimes_1 A)}
	\arrow["\iota"{description}, from=1-1, to=1-3]
	\arrow["{F_1}"{description}, from=1-3, to=3-3]
	\arrow["{F(\iota)}"{description}, from=1-1, to=3-3]
\end{tikzcd}\]

A \defin{oplax monoidal functor} between lax monoidal categories is defined similarly except that the natural transformations $F_n$ go in the reverse direction, and so the corresponding arrows in the coherence laws are reversed.

\defin{Lax/oplax monoidal functors} between oplax monoidal categories are defined similarly with the natural transformations $\alpha$ and $\iota$ going in the opposite direction.

A \defin{strong monoidal functor} between lax/oplax monoidal categories is a lax monoidal functor where the $F_n$ are isomorphisms (and so it is also an oplax monoidal functor).

A \defin{strict monoidal functor} is one where they are identities.
\end{definition}

\begin{definition}
A \defin{monoidal transformation} between lax monoidal functors between lax monoidal categories is a natural transformation $\gamma \colon F \Rightarrow G$, such that:
% https://q.uiver.app/?q=WzAsNCxbMCwwLCJcXG90aW1lc19uIEYoQV9pKSJdLFsyLDAsIkYoXFxvdGltZXNfbiBBX2kpIl0sWzIsMiwiRyhcXG90aW1lc19uIEFfaSkiXSxbMCwyLCJcXG90aW1lc19uIEcoQV9pKSJdLFswLDEsIkZfbiIsMV0sWzEsMiwiXFxnYW1tYV97XFxvdGltZXNfbiBBX2l9IiwxXSxbMCwzLCJcXG90aW1lc19uXFxnYW1tYV97QV9pfSIsMV0sWzMsMiwiR19uIiwxXV0=
\[\begin{tikzcd}[ampersand replacement=\&]
	{\otimes_n F(A_i)} \&\& {F(\otimes_n A_i)} \\
	\\
	{\otimes_n G(A_i)} \&\& {G(\otimes_n A_i)}
	\arrow["{F_n}"{description}, from=1-1, to=1-3]
	\arrow["{\gamma_{\otimes_n A_i}}"{description}, from=1-3, to=3-3]
	\arrow["{\otimes_n\gamma_{A_i}}"{description}, from=1-1, to=3-1]
	\arrow["{G_n}"{description}, from=3-1, to=3-3]
\end{tikzcd}\]

Similarly a monoidal transformation between oplax monoidal functors is a natural transformation such that the previous square commutes with the horizontal arrows going from right to left.
\end{definition}

There are 2-categories consisting of lax/oplax monoidal categories, lax/oplax/strong/strict monoidal functors and monoidal transformations.

When defining functors between luldcs in addition to respecting the monoidal structures, we will want the functor to interact well with the distributivity law.
This will consist of an hexagonal commutative diagram as follows:
% https://q.uiver.app/?q=WzAsNixbMSwwLCJcXG90aW1lcyhGKFxcR2FtbWFfMSksRihcXHBhcnIoXFxEZWx0YV8xLEEsXFxEZWx0YV8yKSksRihcXEdhbW1hXzIpKSJdLFswLDEsIlxcb3RpbWVzKEYoXFxHYW1tYV8xKSxcXHBhcnIoRihcXERlbHRhXzEpLEYoQSksRihcXERlbHRhXzIpKSxGKFxcR2FtbWFfMikpIl0sWzAsMywiXFxwYXJyKEYoXFxEZWx0YV8xKSxcXG90aW1lcyhGKFxcR2FtbWFfMSksRihBKSxGKFxcR2FtbWFfMikpLEYoXFxEZWx0YV8yKSkiXSxbMSw0LCJcXHBhcnIoRihcXERlbHRhXzEpLEYoXFxvdGltZXMoXFxHYW1tYV8xLEEsXFxHYW1tYV8yKSksRihcXERlbHRhXzIpKSJdLFsyLDEsIkYoXFxvdGltZXMoXFxHYW1tYV8xLFxccGFycihcXERlbHRhXzEsQSxcXERlbHRhXzIpLFxcR2FtbWFfMikpIl0sWzIsMywiRihcXHBhcnIoXFxEZWx0YV8xLFxcb3RpbWVzKFxcR2FtbWFfMSxBLFxcR2FtbWFfMiksXFxEZWx0YV8yKSkiXSxbMCwxLCIiLDEseyJzdHlsZSI6eyJoZWFkIjp7Im5hbWUiOiJub25lIn19fV0sWzEsMiwiXFxkZWx0YSIsMV0sWzIsMywiIiwxLHsic3R5bGUiOnsiaGVhZCI6eyJuYW1lIjoibm9uZSJ9fX1dLFswLDQsIiIsMSx7InN0eWxlIjp7ImhlYWQiOnsibmFtZSI6Im5vbmUifX19XSxbNCw1LCJGKFxcZGVsdGEpIiwxXSxbNSwzLCIiLDEseyJzdHlsZSI6eyJoZWFkIjp7Im5hbWUiOiJub25lIn19fV1d
\[\resizebox{\hsize}{!}{\begin{tikzcd}[ampersand replacement=\&]
	\& {\otimes(F(\Gamma_1),F(\parr(\Delta_1,A,\Delta_2)),F(\Gamma_2))} \\
	{\otimes(F(\Gamma_1),\parr(F(\Delta_1),F(A),F(\Delta_2)),F(\Gamma_2))} \&\& {F(\otimes(\Gamma_1,\parr(\Delta_1,A,\Delta_2),\Gamma_2))} \\
	\\
	{\parr(F(\Delta_1),\otimes(F(\Gamma_1),F(A),F(\Gamma_2)),F(\Delta_2))} \&\& {F(\parr(\Delta_1,\otimes(\Gamma_1,A,\Gamma_2),\Delta_2))} \\
	\& {\parr(F(\Delta_1),F(\otimes(\Gamma_1,A,\Gamma_2)),F(\Delta_2))}
	\arrow[no head, from=1-2, to=2-1]
	\arrow["\delta"{description}, from=2-1, to=4-1]
	\arrow[no head, from=4-1, to=5-2]
	\arrow[no head, from=1-2, to=2-3]
	\arrow["{F(\delta)}"{description}, from=2-3, to=4-3]
	\arrow[no head, from=4-3, to=5-2]
\end{tikzcd}}\]
where the unmarked arrow correspond to some lax/oplax conditions for the functor with respect to both monoidal structures.
Out of the four possible choices, only three are valid, since the $\otimes$-oplax $\parr$-lax functor does not make two parallel paths:

% https://q.uiver.app/?q=WzAsMzQsWzMsMSwiXFxidWxsZXQiXSxbMiwyLCJcXGJ1bGxldCJdLFsyLDQsIlxcYnVsbGV0Il0sWzMsNSwiXFxidWxsZXQiXSxbNCwyLCJcXGJ1bGxldCJdLFs0LDQsIlxcYnVsbGV0Il0sWzgsMiwiXFxidWxsZXQiXSxbOSwxLCJcXGJ1bGxldCJdLFsxMCwyLCJcXGJ1bGxldCJdLFs4LDQsIlxcYnVsbGV0Il0sWzEwLDQsIlxcYnVsbGV0Il0sWzksNSwiXFxidWxsZXQiXSxbMywzLCJcXGNoZWNrbWFyayJdLFs5LDMsIlxcdGltZXMiXSxbMiw3LCJcXGJ1bGxldCJdLFszLDYsIlxcYnVsbGV0Il0sWzQsNywiXFxidWxsZXQiXSxbMiw5LCJcXGJ1bGxldCJdLFs0LDksIlxcYnVsbGV0Il0sWzMsMTAsIlxcYnVsbGV0Il0sWzMsOCwiXFxjaGVja21hcmsiXSxbOSw2LCJcXGJ1bGxldCJdLFs4LDcsIlxcYnVsbGV0Il0sWzgsOSwiXFxidWxsZXQiXSxbOSwxMCwiXFxidWxsZXQiXSxbMTAsOSwiXFxidWxsZXQiXSxbMTAsNywiXFxidWxsZXQiXSxbOSw4LCJcXGNoZWNrbWFyayJdLFswLDAsIlxcb3RpbWVzL1xccGFyciJdLFswLDMsIlxcdGV4dHtsYXgvb3BsYXh9Il0sWzYsMywiXFx0ZXh0e29wbGF4L2xheH0iXSxbMCw4LCJcXHRleHR7bGF4L2xheH0iXSxbNiw4LCJcXHRleHR7b3BsYXgvb3BsYXh9Il0sWzYsMCwiXFxvdGltZXMvXFxwYXJyIl0sWzAsMV0sWzEsMl0sWzIsM10sWzAsNF0sWzQsNV0sWzUsM10sWzYsN10sWzgsN10sWzYsOV0sWzgsMTBdLFsxMSw5XSxbMTEsMTBdLFsxNCwxNV0sWzE1LDE2XSxbMTQsMTddLFsxNiwxOF0sWzE3LDE5XSxbMTksMThdLFsyMSwyMl0sWzIyLDIzXSxbMjQsMjNdLFsyNSwyNF0sWzI2LDIxXSxbMjYsMjVdXQ==
\[\begin{tikzcd}[ampersand replacement=\&]
	{\otimes/\parr} \&\&\&\&\&\& {\otimes/\parr} \\
	\&\&\& \bullet \&\&\&\&\&\& \bullet \\
	\&\& \bullet \&\& \bullet \&\&\&\& \bullet \&\& \bullet \\
	{\text{lax/oplax}} \&\&\& \checkmark \&\&\& {\text{oplax/lax}} \&\&\& \times \\
	\&\& \bullet \&\& \bullet \&\&\&\& \bullet \&\& \bullet \\
	\&\&\& \bullet \&\&\&\&\&\& \bullet \\
	\&\&\& \bullet \&\&\&\&\&\& \bullet \\
	\&\& \bullet \&\& \bullet \&\&\&\& \bullet \&\& \bullet \\
	{\text{lax/lax}} \&\&\& \checkmark \&\&\& {\text{oplax/oplax}} \&\&\& \checkmark \\
	\&\& \bullet \&\& \bullet \&\&\&\& \bullet \&\& \bullet \\
	\&\&\& \bullet \&\&\&\&\&\& \bullet
	\arrow[from=2-4, to=3-3]
	\arrow[from=3-3, to=5-3]
	\arrow[from=5-3, to=6-4]
	\arrow[from=2-4, to=3-5]
	\arrow[from=3-5, to=5-5]
	\arrow[from=5-5, to=6-4]
	\arrow[from=3-9, to=2-10]
	\arrow[from=3-11, to=2-10]
	\arrow[from=3-9, to=5-9]
	\arrow[from=3-11, to=5-11]
	\arrow[from=6-10, to=5-9]
	\arrow[from=6-10, to=5-11]
	\arrow[from=8-3, to=7-4]
	\arrow[from=7-4, to=8-5]
	\arrow[from=8-3, to=10-3]
	\arrow[from=8-5, to=10-5]
	\arrow[from=10-3, to=11-4]
	\arrow[from=11-4, to=10-5]
	\arrow[from=7-10, to=8-9]
	\arrow[from=8-9, to=10-9]
	\arrow[from=11-10, to=10-9]
	\arrow[from=10-11, to=11-10]
	\arrow[from=8-11, to=7-10]
	\arrow[from=8-11, to=10-11]
\end{tikzcd}\]

The choice lax/oplax will be the one that will correspond to functor of polycategories.
This comes from the universal properties of $\otimes$ and $\parr$ in a polycategory.
For the functor we will have a polymap $\Gamma \to \otimes \Gamma$.
Applying the functor $F$ this will give a polymap $F(\Gamma) \to F(\otimes \Gamma)$.
Finally the universal property of $\otimes$ will let us build a unique functor $\otimes F(\Gamma) \to F(\otimes \Gamma)$.
For the $\parr$ it will go in the other direction.

However, this notion of functor between luldcs is sometimes too restrictive, as argued in \cite{CockettSeely1999}.
In general, one might want a pair of functors, one interacting with each monoidal structure.
Let $F_\otimes$ be lax/oplax with respect to $\otimes$ and $F_\parr$ be lax/oplax with respect to $\parr$.
Then we will have two hexagons similar to the one above.
First:
% https://q.uiver.app/?q=WzAsNixbMSwwLCJcXG90aW1lcyhGX1xcb3RpbWVzKFxcR2FtbWFfMSksRl9cXHBhcnIoXFxwYXJyKFxcRGVsdGFfMSxBLFxcRGVsdGFfMikpLEZfXFxvdGltZXMoXFxHYW1tYV8yKSkiXSxbMCwxLCJcXG90aW1lcyhGX1xcb3RpbWVzKFxcR2FtbWFfMSksXFxwYXJyKEZfXFxwYXJyKFxcRGVsdGFfMSksRl9cXG90aW1lcyhBKSxGX1xccGFycihcXERlbHRhXzIpKSxGX1xcb3RpbWVzKFxcR2FtbWFfMikpIl0sWzAsMywiXFxwYXJyKEZfXFxwYXJyKFxcRGVsdGFfMSksXFxvdGltZXMoRl9cXG90aW1lcyhcXEdhbW1hXzEpLEZfXFxvdGltZXMoQSksRl9cXG90aW1lcyhcXEdhbW1hXzIpKSxGX1xccGFycihcXERlbHRhXzIpKSJdLFsxLDQsIlxccGFycihGX1xccGFycihcXERlbHRhXzEpLEZfXFxvdGltZXMoXFxvdGltZXMoXFxHYW1tYV8xLEEsXFxHYW1tYV8yKSksRl9cXHBhcnIoXFxEZWx0YV8yKSkiXSxbMiwxLCJGX1xcb3RpbWVzKFxcb3RpbWVzKFxcR2FtbWFfMSxcXHBhcnIoXFxEZWx0YV8xLEEsXFxEZWx0YV8yKSxcXEdhbW1hXzIpKSJdLFsyLDMsIkZfXFxvdGltZXMoXFxwYXJyKFxcRGVsdGFfMSxcXG90aW1lcyhcXEdhbW1hXzEsQSxcXEdhbW1hXzIpLFxcRGVsdGFfMikpIl0sWzAsMSwiIiwxLHsic3R5bGUiOnsiaGVhZCI6eyJuYW1lIjoibm9uZSJ9fX1dLFsxLDIsIlxcZGVsdGEiLDFdLFsyLDMsIiIsMSx7InN0eWxlIjp7ImhlYWQiOnsibmFtZSI6Im5vbmUifX19XSxbMCw0LCIiLDEseyJzdHlsZSI6eyJoZWFkIjp7Im5hbWUiOiJub25lIn19fV0sWzQsNSwiRl9cXG90aW1lcyhcXGRlbHRhKSIsMV0sWzUsMywiIiwxLHsic3R5bGUiOnsiaGVhZCI6eyJuYW1lIjoibm9uZSJ9fX1dXQ==
\[\resizebox{\hsize}{!}{\begin{tikzcd}[ampersand replacement=\&]
	\& {\otimes(F_\otimes(\Gamma_1),F_\otimes(\parr(\Delta_1,A,\Delta_2)),F_\otimes(\Gamma_2))} \\
	{\otimes(F_\otimes(\Gamma_1),\parr(F_\parr(\Delta_1),F_\otimes(A),F_\parr(\Delta_2)),F_\otimes(\Gamma_2))} \&\& {F_\otimes(\otimes(\Gamma_1,\parr(\Delta_1,A,\Delta_2),\Gamma_2))} \\
	\\
	{\parr(F_\parr(\Delta_1),\otimes(F_\otimes(\Gamma_1),F_\otimes(A),F_\otimes(\Gamma_2)),F_\parr(\Delta_2))} \&\& {F_\otimes(\parr(\Delta_1,\otimes(\Gamma_1,A,\Gamma_2),\Delta_2))} \\
	\& {\parr(F_\parr(\Delta_1),F_\otimes(\otimes(\Gamma_1,A,\Gamma_2)),F_\parr(\Delta_2))}
	\arrow[no head, from=1-2, to=2-1]
	\arrow["\delta"{description}, from=2-1, to=4-1]
	\arrow[no head, from=4-1, to=5-2]
	\arrow[no head, from=1-2, to=2-3]
	\arrow["{F_\otimes(\delta)}"{description}, from=2-3, to=4-3]
	\arrow[no head, from=4-3, to=5-2]
\end{tikzcd}}\]
The top-left and bottom-right arrow correspond to some strength-like rule connecting $F_\otimes$ and $\parr$ relating $F_\otimes(\parr(\Delta_1,A,\Delta_2))$ and $\parr(F_\parr(\Delta_1),F_\otimes(A), F_\parr(\Delta_2))$.
We will talk of a $\parr$-costrength when we have a natural transformation:
$F_\otimes(\parr(\Delta_1,A,\Delta_2)) \to \parr(F_\parr(\Delta_1),F_\otimes(A), F_\parr(\Delta_2))$ and a $\parr$-strength when it goes the other way.
The top-right and bottom-left ones correpond to some lax/oplax condition on $F_\otimes$.
Like for the interaction between $\otimes$-laxity and $\parr$-laxity only three choices makes sense: $\otimes$-lax with $\parr$-costrength, $\otimes$-lax with $\parr$-strength and $\otimes$-oplax with $\parr$-costrength.

We also have the following hexagon:
% https://q.uiver.app/?q=WzAsNixbMSwwLCJcXG90aW1lcyhGX1xcb3RpbWVzKFxcR2FtbWFfMSksRl9cXHBhcnIoXFxwYXJyKFxcRGVsdGFfMSxBLFxcRGVsdGFfMikpLEZfXFxvdGltZXMoXFxHYW1tYV8yKSkiXSxbMCwxLCJcXG90aW1lcyhGX1xcb3RpbWVzKFxcR2FtbWFfMSksXFxwYXJyKEZfXFxwYXJyKFxcRGVsdGFfMSksRl9cXHBhcnIoQSksRl9cXHBhcnIoXFxEZWx0YV8yKSksRl9cXG90aW1lcyhcXEdhbW1hXzIpKSJdLFswLDMsIlxccGFycihGX1xccGFycihcXERlbHRhXzEpLFxcb3RpbWVzKEZfXFxvdGltZXMoXFxHYW1tYV8xKSxGX1xccGFycihBKSxGX1xcb3RpbWVzKFxcR2FtbWFfMikpLEZfXFxwYXJyKFxcRGVsdGFfMikpIl0sWzEsNCwiXFxwYXJyKEZfXFxwYXJyKFxcRGVsdGFfMSksRl9cXHBhcnIoXFxvdGltZXMoXFxHYW1tYV8xLEEsXFxHYW1tYV8yKSksRl9cXHBhcnIoXFxEZWx0YV8yKSkiXSxbMiwxLCJGX1xccGFycihcXG90aW1lcyhcXEdhbW1hXzEsXFxwYXJyKFxcRGVsdGFfMSxBLFxcRGVsdGFfMiksXFxHYW1tYV8yKSkiXSxbMiwzLCJGX1xccGFycihcXHBhcnIoXFxEZWx0YV8xLFxcb3RpbWVzKFxcR2FtbWFfMSxBLFxcR2FtbWFfMiksXFxEZWx0YV8yKSkiXSxbMCwxLCIiLDEseyJzdHlsZSI6eyJoZWFkIjp7Im5hbWUiOiJub25lIn19fV0sWzEsMiwiXFxkZWx0YSIsMV0sWzIsMywiIiwxLHsic3R5bGUiOnsiaGVhZCI6eyJuYW1lIjoibm9uZSJ9fX1dLFswLDQsIiIsMSx7InN0eWxlIjp7ImhlYWQiOnsibmFtZSI6Im5vbmUifX19XSxbNCw1LCJGX1xccGFycihcXGRlbHRhKSIsMV0sWzUsMywiIiwxLHsic3R5bGUiOnsiaGVhZCI6eyJuYW1lIjoibm9uZSJ9fX1dXQ==
\[\resizebox{\hsize}{!}{\begin{tikzcd}[ampersand replacement=\&]
	\& {\otimes(F_\otimes(\Gamma_1),F_\parr(\parr(\Delta_1,A,\Delta_2)),F_\otimes(\Gamma_2))} \\
	{\otimes(F_\otimes(\Gamma_1),\parr(F_\parr(\Delta_1),F_\parr(A),F_\parr(\Delta_2)),F_\otimes(\Gamma_2))} \&\& {F_\parr(\otimes(\Gamma_1,\parr(\Delta_1,A,\Delta_2),\Gamma_2))} \\
	\\
	{\parr(F_\parr(\Delta_1),\otimes(F_\otimes(\Gamma_1),F_\parr(A),F_\otimes(\Gamma_2)),F_\parr(\Delta_2))} \&\& {F_\parr(\parr(\Delta_1,\otimes(\Gamma_1,A,\Gamma_2),\Delta_2))} \\
	\& {\parr(F_\parr(\Delta_1),F_\parr(\otimes(\Gamma_1,A,\Gamma_2)),F_\parr(\Delta_2))}
	\arrow[no head, from=1-2, to=2-1]
	\arrow["\delta"{description}, from=2-1, to=4-1]
	\arrow[no head, from=4-1, to=5-2]
	\arrow[no head, from=1-2, to=2-3]
	\arrow["{F_\parr(\delta)}"{description}, from=2-3, to=4-3]
	\arrow[no head, from=4-3, to=5-2]
\end{tikzcd}}\]
Where this time the top-left and bottom-right arrow correspond to $\parr$-laxity condition and the top-right and bottom-left to $\otimes$-strength condition.
The possible arrangement this time are: $\parr$-oplax with $\otimes$-strength, $\parr$-lax with $\otimes$-strength and $\parr$-oplax with $\otimes$-costrength.

Putting everything together, out of the 16 possibilities of $\otimes$/$\parr$-laxity/strength, 7 are ruled out.

The one that we will be interested in and that correspond to the one proposed in \cite{CockettSeely1999} is the $\otimes$-lax, $\parr$-oplax with $\otimes$-strengh and $\parr$-costrength.

\begin{definition}
A \defin{linear functor} between luldc is a pair of functors $F_\otimes, F_\parr \colon \cC \to \cD$ such that $F_\otimes$ is $\otimes$-lax monoidal, $F_\parr$ is $\parr$-oplax monoidal, equipped with a $\otimes$-strength natural transformation $\nu_\parr \otimes(F_\otimes(\Gamma_1), F_\parr(A),F_\otimes(\Gamma_2)) \to F_\parr(\otimes(\Gamma_1,A,\Gamma_2))$ and a $\parr$-costrength natural transformation $\nu_\otimes \colon F_\otimes(\parr(\Delta_1,A,\Delta_2)) \to \parr(F_\parr(\Delta_1),F_\otimes(A),F_\parr(\Delta_2))$, such that the following diagrams commute:
% https://q.uiver.app/?q=WzAsNixbMSwwLCJcXG90aW1lcyhGX1xcb3RpbWVzKFxcR2FtbWFfMSksRl9cXG90aW1lcyhcXHBhcnIoXFxEZWx0YV8xLEEsXFxEZWx0YV8yKSksRl9cXG90aW1lcyhcXEdhbW1hXzIpKSJdLFswLDEsIlxcb3RpbWVzKEZfXFxvdGltZXMoXFxHYW1tYV8xKSxcXHBhcnIoRl9cXHBhcnIoXFxEZWx0YV8xKSxGX1xcb3RpbWVzKEEpLEZfXFxwYXJyKFxcRGVsdGFfMikpLEZfXFxvdGltZXMoXFxHYW1tYV8yKSkiXSxbMCwzLCJcXHBhcnIoRl9cXHBhcnIoXFxEZWx0YV8xKSxcXG90aW1lcyhGX1xcb3RpbWVzKFxcR2FtbWFfMSksRl9cXG90aW1lcyhBKSxGX1xcb3RpbWVzKFxcR2FtbWFfMikpLEZfXFxwYXJyKFxcRGVsdGFfMikpIl0sWzEsNCwiXFxwYXJyKEZfXFxwYXJyKFxcRGVsdGFfMSksRl9cXG90aW1lcyhcXG90aW1lcyhcXEdhbW1hXzEsQSxcXEdhbW1hXzIpKSxGX1xccGFycihcXERlbHRhXzIpKSJdLFsyLDEsIkZfXFxvdGltZXMoXFxvdGltZXMoXFxHYW1tYV8xLFxccGFycihcXERlbHRhXzEsQSxcXERlbHRhXzIpLFxcR2FtbWFfMikpIl0sWzIsMywiRl9cXG90aW1lcyhcXHBhcnIoXFxEZWx0YV8xLFxcb3RpbWVzKFxcR2FtbWFfMSxBLFxcR2FtbWFfMiksXFxEZWx0YV8yKSkiXSxbMCwxLCJcXG90aW1lcygtLFxcbnVfXFxvdGltZXMsLSkiLDFdLFsxLDIsIlxcZGVsdGEiLDFdLFsyLDMsIlxccGFycigtLChGX1xcb3RpbWVzKV97bV8xKzErbV8yfSwtKSIsMV0sWzAsNCwiKEZfXFxvdGltZXMpX3ttXzErMSttXzJ9IiwxXSxbNCw1LCJGX1xcb3RpbWVzKFxcZGVsdGEpIiwxXSxbNSwzLCJcXG51X1xcb3RpbWVzIiwxXV0=
\[\resizebox{\hsize}{!}{\begin{tikzcd}[ampersand replacement=\&]
	\& {\otimes(F_\otimes(\Gamma_1),F_\otimes(\parr(\Delta_1,A,\Delta_2)),F_\otimes(\Gamma_2))} \\
	{\otimes(F_\otimes(\Gamma_1),\parr(F_\parr(\Delta_1),F_\otimes(A),F_\parr(\Delta_2)),F_\otimes(\Gamma_2))} \&\& {F_\otimes(\otimes(\Gamma_1,\parr(\Delta_1,A,\Delta_2),\Gamma_2))} \\
	\\
	{\parr(F_\parr(\Delta_1),\otimes(F_\otimes(\Gamma_1),F_\otimes(A),F_\otimes(\Gamma_2)),F_\parr(\Delta_2))} \&\& {F_\otimes(\parr(\Delta_1,\otimes(\Gamma_1,A,\Gamma_2),\Delta_2))} \\
	\& {\parr(F_\parr(\Delta_1),F_\otimes(\otimes(\Gamma_1,A,\Gamma_2)),F_\parr(\Delta_2))}
	\arrow["{\otimes(-,\nu_\otimes,-)}"{description}, from=1-2, to=2-1]
	\arrow["\delta"{description}, from=2-1, to=4-1]
	\arrow["{\parr(-,(F_\otimes)_{m_1+1+m_2},-)}"{description}, from=4-1, to=5-2]
	\arrow["{(F_\otimes)_{m_1+1+m_2}}"{description}, from=1-2, to=2-3]
	\arrow["{F_\otimes(\delta)}"{description}, from=2-3, to=4-3]
	\arrow["{\nu_\otimes}"{description}, from=4-3, to=5-2]
\end{tikzcd}}\]

% https://q.uiver.app/?q=WzAsMyxbMCwwLCJGX1xcb3RpbWVzKEEpIl0sWzIsMCwiRl9cXG90aW1lcyhcXHBhcnJfMSBBKSJdLFsyLDIsIlxccGFycl8xIEZfXFxvdGltZXMoQSkiXSxbMCwxLCJGX1xcb3RpbWVzKFxcaW90YV9BKSIsMV0sWzEsMiwiXFxudV9cXG90aW1lcyIsMV0sWzAsMiwiXFxpb3RhX3tGX1xcb3RpbWVzKEEpfSIsMV1d
\[\begin{tikzcd}[ampersand replacement=\&]
	{F_\otimes(A)} \&\& {F_\otimes(\parr_1 A)} \\
	\\
	\&\& {\parr_1 F_\otimes(A)}
	\arrow["{F_\otimes(\iota_A)}"{description}, from=1-1, to=1-3]
	\arrow["{\nu_\otimes}"{description}, from=1-3, to=3-3]
	\arrow["{\iota_{F_\otimes(A)}}"{description}, from=1-1, to=3-3]
\end{tikzcd}\]

% https://q.uiver.app/?q=WzAsNixbMSwwLCJcXG90aW1lcyhGX1xcb3RpbWVzKFxcR2FtbWFfMSksRl9cXHBhcnIoXFxwYXJyKFxcRGVsdGFfMSxBLFxcRGVsdGFfMikpLEZfXFxvdGltZXMoXFxHYW1tYV8yKSkiXSxbMCwxLCJcXG90aW1lcyhGX1xcb3RpbWVzKFxcR2FtbWFfMSksXFxwYXJyKEZfXFxwYXJyKFxcRGVsdGFfMSksRl9cXHBhcnIoQSksRl9cXHBhcnIoXFxEZWx0YV8yKSksRl9cXG90aW1lcyhcXEdhbW1hXzIpKSJdLFswLDMsIlxccGFycihGX1xccGFycihcXERlbHRhXzEpLFxcb3RpbWVzKEZfXFxvdGltZXMoXFxHYW1tYV8xKSxGX1xccGFycihBKSxGX1xcb3RpbWVzKFxcR2FtbWFfMikpLEZfXFxwYXJyKFxcRGVsdGFfMikpIl0sWzEsNCwiXFxwYXJyKEZfXFxwYXJyKFxcRGVsdGFfMSksRl9cXHBhcnIoXFxvdGltZXMoXFxHYW1tYV8xLEEsXFxHYW1tYV8yKSksRl9cXHBhcnIoXFxEZWx0YV8yKSkiXSxbMiwxLCJGX1xccGFycihcXG90aW1lcyhcXEdhbW1hXzEsXFxwYXJyKFxcRGVsdGFfMSxBLFxcRGVsdGFfMiksXFxHYW1tYV8yKSkiXSxbMiwzLCJGX1xccGFycihcXHBhcnIoXFxEZWx0YV8xLFxcb3RpbWVzKFxcR2FtbWFfMSxBLFxcR2FtbWFfMiksXFxEZWx0YV8yKSkiXSxbMCwxLCJcXG90aW1lcygtLChGX1xccGFycilfe25fMSsxK25fMn0sLSkiLDFdLFsxLDIsIlxcZGVsdGEiLDFdLFsyLDMsIlxccGFycigtLFxcbnVfXFxwYXJyLC0pIiwxXSxbMCw0LCJcXG51X1xccGFyciIsMV0sWzQsNSwiRl9cXHBhcnIoXFxkZWx0YSkiLDFdLFs1LDMsIihGX1xccGFycilfe25fMSsxK25fMn0iLDFdXQ==
\[\resizebox{\hsize}{!}{\begin{tikzcd}[ampersand replacement=\&]
	\& {\otimes(F_\otimes(\Gamma_1),F_\parr(\parr(\Delta_1,A,\Delta_2)),F_\otimes(\Gamma_2))} \\
	{\otimes(F_\otimes(\Gamma_1),\parr(F_\parr(\Delta_1),F_\parr(A),F_\parr(\Delta_2)),F_\otimes(\Gamma_2))} \&\& {F_\parr(\otimes(\Gamma_1,\parr(\Delta_1,A,\Delta_2),\Gamma_2))} \\
	\\
	{\parr(F_\parr(\Delta_1),\otimes(F_\otimes(\Gamma_1),F_\parr(A),F_\otimes(\Gamma_2)),F_\parr(\Delta_2))} \&\& {F_\parr(\parr(\Delta_1,\otimes(\Gamma_1,A,\Gamma_2),\Delta_2))} \\
	\& {\parr(F_\parr(\Delta_1),F_\parr(\otimes(\Gamma_1,A,\Gamma_2)),F_\parr(\Delta_2))}
	\arrow["{\otimes(-,(F_\parr)_{n_1+1+n_2},-)}"{description}, from=1-2, to=2-1]
	\arrow["\delta"{description}, from=2-1, to=4-1]
	\arrow["{\parr(-,\nu_\parr,-)}"{description}, from=4-1, to=5-2]
	\arrow["{\nu_\parr}"{description}, from=1-2, to=2-3]
	\arrow["{F_\parr(\delta)}"{description}, from=2-3, to=4-3]
	\arrow["{(F_\parr)_{n_1+1+n_2}}"{description}, from=4-3, to=5-2]
\end{tikzcd}}\]

% https://q.uiver.app/?q=WzAsMyxbMCwwLCJGX1xccGFycihBKSJdLFsyLDAsIlxcb3RpbWVzXzFGX1xccGFycihBKSJdLFsyLDIsIkZfXFxwYXJyKFxcb3RpbWVzXzEgQSkiXSxbMSwwLCJcXGV0YV97Rl9cXHBhcnIoQSl9IiwxXSxbMiwxLCJcXG51X1xccGFyciIsMV0sWzIsMCwiRl9cXHBhcnIoXFxldGFfQSkiLDFdXQ==
\[\begin{tikzcd}[ampersand replacement=\&]
	{F_\parr(A)} \&\& {\otimes_1F_\parr(A)} \\
	\\
	\&\& {F_\parr(\otimes_1 A)}
	\arrow["{\eta_{F_\parr(A)}}"{description}, from=1-3, to=1-1]
	\arrow["{\nu_\parr}"{description}, from=3-3, to=1-3]
	\arrow["{F_\parr(\eta_A)}"{description}, from=3-3, to=1-1]
\end{tikzcd}\]

\end{definition}

\begin{rk}
This strength and co-strength and their coherence laws are generalisations of the ones in \cite{CockettSeely1999}.
It should be possible to extend the correspondence between unbiased linearly distributive categories and biased ones to an equivalence of categories where the morphisms are linear functors.
\end{rk}

As mentioned above, these are more general than what we would get by considering functors between the underlying polycategories.
On the nlab, the latter are called Frobenius linear functors.

\begin{definition}
A \defin{Frobenius linear functor} between luldc is a linear functor such that $F_\otimes = F_\parr$ and the $\otimes$-strength is laxity $\nu_\parr = (F_\otimes)_{m_1+1+m_2}$ and the $\parr$-costrength is oplaxity $\nu_\otimes = (F_\parr)_{n_1+1+n_2}$.
\end{definition}

\begin{definition}
A \defin{linear transformation} between linear functors consists of a pair of monoidal transformations $\alpha_\otimes \colon F_\otimes \Rightarrow G_\otimes$ and $\alpha_\parr \colon F_\parr \Rightarrow G_\parr$ such that the following diagrams commute:
% https://q.uiver.app/?q=WzAsMTAsWzAsMCwiXFxvdGltZXMoRl9cXG90aW1lcyhcXEdhbW1hXzEpLEdfXFxwYXJyKEEpLEZfXFxvdGltZXMoXFxHYW1tYV8yKSkiXSxbMiwwLCJcXG90aW1lcyhGX1xcb3RpbWVzKFxcR2FtbWFfMSksRl9cXHBhcnIoQSksRl9cXG90aW1lcyhcXEdhbW1hXzIpKSJdLFsyLDIsIkZfXFxwYXJyKFxcb3RpbWVzKFxcR2FtbWFfMSxBLFxcR2FtbWFfMikpIl0sWzEsMywiR19cXHBhcnIoXFxvdGltZXMoXFxHYW1tYV8xLEEsXFxHYW1tYV8yKSkiXSxbMCwyLCJcXG90aW1lcyhHX1xcb3RpbWVzKFxcR2FtbWFfMSksR19cXHBhcnIoQSksR19cXG90aW1lcyhcXEdhbW1hXzIpKSJdLFsxLDQsIkZfXFxvdGltZXMoXFxwYXJyKFxcRGVsdGFfMSxBLFxcRGVsdGFfMikpIl0sWzAsNSwiXFxwYXJyKEZfXFxwYXJyKFxcRGVsdGFfMSksRl9cXG90aW1lcyhBKSxGX1xccGFycihcXERlbHRhXzIpKSJdLFswLDcsIlxccGFycihGX1xccGFycihcXERlbHRhXzEpLEdfXFxvdGltZXMoQSksRl9cXHBhcnIoXFxEZWx0YV8yKSkiXSxbMiw1LCJHX1xcb3RpbWVzKFxccGFycihcXERlbHRhXzEsQSxcXERlbHRhXzIpKSJdLFsyLDcsIlxccGFycihHX1xccGFycihcXERlbHRhXzEpLEdfXFxvdGltZXMoQSksR19cXHBhcnIoXFxEZWx0YV8yKSkiXSxbMCwxLCJcXG90aW1lcygtLFxcYWxwaGFfXFxwYXJyLC0pIiwxXSxbMSwyLCJcXG51X1xccGFyciIsMV0sWzIsMywiXFxhbHBoYV9cXHBhcnIiLDFdLFswLDQsIlxcb3RpbWVzKFxcYWxwaGFfXFxvdGltZXMsLSxcXGFscGhhX1xcb3RpbWVzKSIsMV0sWzQsMywiXFxudV9cXHBhcnIiLDFdLFs1LDYsIlxcbnVfXFxvdGltZXMiLDFdLFs2LDcsIlxccGFycigtLFxcYWxwaGFfXFxvdGltZXMsLSkiLDFdLFs1LDgsIlxcYWxwaGFfXFxvdGltZXMiLDFdLFs4LDksIlxcbnVfXFxvdGltZXMiLDFdLFs5LDcsIlxccGFycihcXGFscGhhX1xccGFyciwtLFxcYWxwaGFfXFxwYXJyKSIsMV1d
\[\resizebox{\hsize}{!}{\begin{tikzcd}[ampersand replacement=\&]
	{\otimes(F_\otimes(\Gamma_1),G_\parr(A),F_\otimes(\Gamma_2))} \&\& {\otimes(F_\otimes(\Gamma_1),F_\parr(A),F_\otimes(\Gamma_2))} \\
	\\
	{\otimes(G_\otimes(\Gamma_1),G_\parr(A),G_\otimes(\Gamma_2))} \&\& {F_\parr(\otimes(\Gamma_1,A,\Gamma_2))} \\
	\& {G_\parr(\otimes(\Gamma_1,A,\Gamma_2))} \\
	\& {F_\otimes(\parr(\Delta_1,A,\Delta_2))} \\
	{\parr(F_\parr(\Delta_1),F_\otimes(A),F_\parr(\Delta_2))} \&\& {G_\otimes(\parr(\Delta_1,A,\Delta_2))} \\
	\\
	{\parr(F_\parr(\Delta_1),G_\otimes(A),F_\parr(\Delta_2))} \&\& {\parr(G_\parr(\Delta_1),G_\otimes(A),G_\parr(\Delta_2))}
	\arrow["{\otimes(-,\alpha_\parr,-)}"{description}, from=1-3, to=1-1]
	\arrow["{\nu_\parr}"{description}, from=1-3, to=3-3]
	\arrow["{\alpha_\parr}"{description}, from=3-3, to=4-2]
	\arrow["{\otimes(\alpha_\otimes,-,\alpha_\otimes)}"{description}, from=1-1, to=3-1]
	\arrow["{\nu_\parr}"{description}, from=3-1, to=4-2]
	\arrow["{\nu_\otimes}"{description}, from=5-2, to=6-1]
	\arrow["{\parr(-,\alpha_\otimes,-)}"{description}, from=6-1, to=8-1]
	\arrow["{\alpha_\otimes}"{description}, from=5-2, to=6-3]
	\arrow["{\nu_\otimes}"{description}, from=6-3, to=8-3]
	\arrow["{\parr(\alpha_\parr,-,\alpha_\parr)}"{description}, from=8-1, to=8-3]
\end{tikzcd}}\]
\end{definition}

Notice that if $F$ is a Frobenius functor and $\alpha_\otimes$ is a natural transformation that is monoidal with respect to $\otimes$ and $\parr$ then $(\alpha,\alpha)$ is a linear transformation.
The two pentagons commute by monoidality:

\begin{rk}
In \cite{CockettSeely1999}, the notion of linear transformation is taken to be a pair of $\alpha_\otimes \colon F_\otimes \Rightarrow G_\otimes$ and $\alpha_\parr \colon G_\parr \Rightarrow F_\parr$ satisfying similar coherence conditions as above.
The reverse direction of $\alpha_\parr$ is justified by the definition of $F_\parr$ and $\alpha_\parr$ as $\ldual{F(\rdual{(-)})}$ and $\ldual{\alpha_{\rdual{(-)}}}$ when $F$ and $\alpha$ are monoidal functors/transformations between $\ast$-autonomous categories.
However with this definition, the equivalence between the category of linearly distributive categories and Frobenius functors and the category of two-tensor polycategories and functors does not extend to a 2-equivalence.
\end{rk}

From now on we will write \LULDClin (\LNULDClin), and \LULDCFrob (\LNULDCFrob) for the 2-categories of lax (normal) unbiased linearly distributive categories with linear/Frobenius functors and linear transformations.

\subsection{Comments}

Some things that I haven't mentioned or proven in this section.

\begin{prop}
A linear functor $(F_\otimes, F_\parr)$ between $\ast$-autonomous categories is equivalent to the data of a monoidal functor $F_\otimes$ such that $\ldual{F_\otimes(\rdual{A})} \simeq \rdual{F_\otimes(\ldual{A})}$.
\end{prop}
In the latter case we define $F_\parr(-) := \ldual{F_\otimes(\rdual{(-)})}$.

This gives a situation where the following diagram commutes in the 2-category of oplax monoidal categories, lax monoidal functors and monoidal transformations:
% https://q.uiver.app/?q=WzAsNCxbMCwwLCIoXFxjQyxcXG90aW1lc19uKSJdLFsxLDAsIihcXGNDXntcXG9wfSxcXHBhcnJfbl57XFxvcH0pIl0sWzAsMywiKFxcY0QsXFxvdGltZXNfbikiXSxbMSwzLCIoXFxjRF57XFxvcH0sXFxwYXJyX25ee1xcb3B9KSJdLFswLDIsIkZfXFxvdGltZXMiLDFdLFsxLDMsIkZfXFxwYXJyIiwxXSxbMSwwLCJcXHJkdWFseygtKX0iLDEseyJjdXJ2ZSI6LTN9XSxbMCwxLCJcXGxkdWFseygtKX0iLDEseyJjdXJ2ZSI6LTN9XSxbMiwzLCJcXGxkdWFseygtKX0iLDEseyJjdXJ2ZSI6LTN9XSxbMywyLCJcXHJkdWFseygtKX0iLDEseyJjdXJ2ZSI6LTR9XSxbNyw2LCJcXHNpbWVxIiwxLHsic2hvcnRlbiI6eyJzb3VyY2UiOjIwLCJ0YXJnZXQiOjIwfX1dLFs4LDksIlxcc2ltZXEiLDEseyJzaG9ydGVuIjp7InNvdXJjZSI6MjAsInRhcmdldCI6MjB9fV1d
\[\begin{tikzcd}[ampersand replacement=\&]
	{(\cC,\otimes_n)} \& {(\cC^{\op},\parr_n^{\op})} \\
	\\
	\\
	{(\cD,\otimes_n)} \& {(\cD^{\op},\parr_n^{\op})}
	\arrow["{F_\otimes}"{description}, from=1-1, to=4-1]
	\arrow["{F_\parr}"{description}, from=1-2, to=4-2]
	\arrow[""{name=0, anchor=center, inner sep=0}, "{\rdual{(-)}}"{description}, curve={height=-18pt}, from=1-2, to=1-1]
	\arrow[""{name=1, anchor=center, inner sep=0}, "{\ldual{(-)}}"{description}, curve={height=-18pt}, from=1-1, to=1-2]
	\arrow[""{name=2, anchor=center, inner sep=0}, "{\ldual{(-)}}"{description}, curve={height=-18pt}, from=4-1, to=4-2]
	\arrow[""{name=3, anchor=center, inner sep=0}, "{\rdual{(-)}}"{description}, curve={height=-24pt}, from=4-2, to=4-1]
	\arrow["\simeq"{description}, shorten <=5pt, shorten >=5pt, Rightarrow, from=1, to=0]
	\arrow["\simeq"{description}, shorten <=6pt, shorten >=6pt, Rightarrow, from=2, to=3]
\end{tikzcd}\]

When considering functors between $\ast$-autonomous categories we also have that:
\begin{prop}
A Frobenius linear functor between $\ast$-autonomous categories is equivalent to the data of a monoidal functor $F$ that preserves the dual $F(\ldual{-}) \simeq \ldual{F(-)}$ and $F(\rdual{-}) \simeq \rdual{F(-)}$. 
\end{prop}

It would be interesting to consider how this can be relaxed.
I don't think that similar statements can be made in the full generality of (Frobenius) linear functors between lax unbiased linearly distributive categories.
However, it could be the case that it is true when restricted to the lax normal case.

\section{Polycategories}

There are several different definitions of ``polycategory'' in the literature.
We will consider the following definition of (non-symmetric) polycategory due to Cockett and Seely \cite{CockettSeely1997}, which differs slightly from Szabo's original definition \cite{Szabo1975} in imposing a planarity condition on composition.
The ideas in this thesis may be transferred in an almost straightforward way to the setting of symmetric polycategories (cf.~\cite{Hyland2002,Shulman2020}), but we work with planar polycategories for the sake of greater generality.

\begin{definition}
  A polycategory $\mathcal{P}$ consists of:
  \begin{itemize}
  \item a collection of objects $Ob(\mathcal{P})$
  \item for any pair of finite lists of objects $\Gamma$ and $\Delta$, a set $\mathcal{P}(\Gamma;\Delta)$ of polymaps from $\Gamma$ to $\Delta$ denoted $f \colon \Gamma \to \Delta$ (we refer to objects in $\Gamma$ as inputs of $f$, and to objects in $\Delta$ as outputs)
  \item for every object $A$, an identity polymap $\id_A : A \to A$
  \item for any pair of polymaps $f \colon \Gamma \to \Delta_1 , A, \Delta_2$ and $g \colon \Gamma_1', A, \Gamma_2' \to \Delta'$
    satisfying the restriction that [either $\Delta_1$ or $\Gamma_1'$ is empty] and [either $\Delta_2$ or $\Gamma_2'$ is empty],
    a polymap $g \circ_A f \colon \Gamma_1', \Gamma, \Gamma_2' \to \Delta_1, \Delta', \Delta_2$
  \end{itemize}
  subject to appropriate unitality, associativity, and interchange laws whenever these make sense:
  \begin{align}
    \id_A \circ_A f &= f \label{eqn:unit1}\\
    f \circ_A \id_A &= f \label{eqn:unit2} \\
    (h \circ_B g) \circ_A f &= h \circ_B (g \circ_A f) \label{eqn:assoc} \\
    (h \circ_B g) \circ_A f &= (h \circ_A f) \circ_B g \label{eqn:int1} \\
    h \circ_B (g \circ_A f) &= g \circ_A (h \circ_B f) \label{eqn:int2}
  \end{align}
\end{definition}

\begin{remark}
  The notation $\circ_A$ for the composition can be ambiguous when there are multiple copies of the same object.
  This can be dealt with more carefully by indexing or labelling each input and output of a polymap.
  However, we will stick with the more relaxed (albeit less precise) notation in this thesis, since it will never lead to ambiguity in the examples.
\end{remark}
\begin{remark}
  We will sometimes find it useful to represent polymaps by string diagrams.
  In this diagrammatic syntax, the composition operation may be depicted schematically as follows:
  \begin{center}
    {\scalebox{0.7}{\input{comp.tikz}}}
  \end{center}
  The restriction on the composition operation that either $\Delta_1$ or $\Gamma'_1$ is empty and that either $\Delta_2$ or $\Gamma'_2$ is empty is called a ``planarity'' condition, since in the picture above it means that there are actually no crossing wires.
  In general, the string diagram of a polymap corresponds to a planar tree with the edges oriented from left to right, and the polycategory axioms correspond to natural isotopies between diagrams.
  For example, the interchange law \eqref{eqn:int1} states that when composing along two different inputs, the order should not matter:
  \begin{center}\scalebox{0.7}{\input{comm.tikz}}\end{center}
  This justifies drawing the two polymaps $f$ and $g$ above on the same level, as we will sometimes do in examples.
\end{remark}

\begin{definition}
Given polycategories \cP and \cQ , a functor of polycategories $F$ is the data of:
\begin{itemize}
\item for each $A$ in \cP , an object $F(A)$ in \cQ
\item for each $f \colon A_1,\dots,A_m \to B_1, \dots, B_n$ in \cP , a polymap \[F(f) \colon F(A_1),\dots, F(A_m) \to F(B_1), \dots, F(B_n)\] in \cQ
\end{itemize}
such that 
\begin{itemize}
\item $F(\id_A) = \id_{F(A)}$
\item $F(g \circ_A f) = F(g) \circ_{F(A)} F(f)$
\end{itemize}
\end{definition}

\begin{definition}
Given functors of polycategories $F, G \colon \cP \to \cQ$, a natural transformation $\alpha \colon F \Rightarrow G$ is a family of polymaps $\alpha_A \colon F(A) \to G(A)$ such that the following diagram commutes:
% https://q.uiver.app/?q=WzAsNCxbMCwwLCJGKEFfMSksXFxkb3RzLCBGKEFfbSkiXSxbMiwwLCJGKEJfMSksXFxkb3RzLEYoQl9uKSJdLFsyLDIsIkcoQl8xKSxcXGRvdHMsIEcoQl9uKSJdLFswLDIsIkcoQV8xKSxcXGRvdHMsRyhBX20pIl0sWzAsMSwiRihmKSIsMV0sWzEsMiwiXFxnYW1tYV97Ql8xfSxcXGRvdHMsXFxnYW1tYV97Ql9ufSIsMV0sWzAsMywiXFxnYW1tYV97QV8xfSwgXFxkb3RzLCBcXGdhbW1hX3tBX219IiwxXSxbMywyLCJHKGYpIiwxXV0=
\[\begin{tikzcd}[ampersand replacement=\&]
	{F(A_1),\dots, F(A_m)} \&\& {F(B_1),\dots,F(B_n)} \\
	\\
	{G(A_1),\dots,G(A_m)} \&\& {G(B_1),\dots, G(B_n)}
	\arrow["{F(f)}"{description}, from=1-1, to=1-3]
	\arrow["{\gamma_{B_1},\dots,\gamma_{B_n}}"{description}, from=1-3, to=3-3]
	\arrow["{\gamma_{A_1}, \dots, \gamma_{A_m}}"{description}, from=1-1, to=3-1]
	\arrow["{G(f)}"{description}, from=3-1, to=3-3]
\end{tikzcd}\]
\end{definition}

\begin{prop}
There is a (strict) 2-category $\PolyCat$ of polycategories, functors and natural transformations.
\end{prop}

\section{Two-tensor polycategories with duals}

Any lax unbiased linearly distributive category has an underlying polycategory.
In this section I will characterise the polycategories that arise this way.
Furthermore, I will describe those that are the underlying polycategory of (nonlax) linearly distributive categories.

\subsection{Underlying polycategory of a luldc}

Let $(\cC,\otimes_n,\parr_n)$ be a luldc.

We define a polycategory $\cP(\cC)$ with:
\begin{itemize}
\item objects, those of \cC
\item polymaps $f \colon A_1,\dots,A_m \to B_1, \dots, B_n$, maps $\otimes_m A_i \to \parr_n B_j$ in \cC
\item $\id_A \colon A \to A$ in $\cP(\cC)$, $\otimes_1 A \xrightarrow{\eta_A} A \xrightarrow{\iota_A} \parr_1 A$ in \cC
\item given $f \colon \otimes \Gamma \to \parr(\Delta_1,A,\Delta_2)$ and $f \colon \otimes(\Gamma_1, A, \Gamma_2) \to \parr \Delta$, we define $g\circ_A f$ by:

\begin{align*}
\otimes(\Gamma_1, \Gamma, \Gamma_2) &\xrightarrow{\alpha'} \otimes(\Gamma_1, \otimes \Gamma, \Gamma_2)\\
&\xrightarrow{\otimes(\Gamma_1,f,\Gamma_2)} \otimes(\Gamma_1, \parr(\Delta_1,A,\Delta_2), \Gamma_2)\\
&\xrightarrow{\delta} \parr(\Delta_1,\otimes(\Gamma_1,A,\Gamma_2),\Delta_2)\\
&\xrightarrow{\parr(\Delta_1,g,\Delta_2)} \parr(\Delta_1,\parr \Delta, \Delta_2)\\
&\xrightarrow{\gamma'} \parr(\Delta_1,\Delta,\Delta_2)
\end{align*}
\end{itemize}

Notice that the constraints on composition of polymaps and the existence of the distributivity law coincide.

\begin{prop}
For any luldc \cC, $\cP(\cC)$ forms a polycategory.
\end{prop}
\begin{proof}
First, let us prove that the identity is unital.
The following diagram commutes:
% https://q.uiver.app/?q=WzAsOCxbMCwxLCJcXG90aW1lcyhcXEdhbW1hXzEsQSxcXEdhbW1hXzIpIl0sWzEsMCwiXFxvdGltZXMoXFxHYW1tYV8xLFxcb3RpbWVzXzFBLFxcR2FtbWFfMikiXSxbMiwxLCJcXG90aW1lcyhcXEdhbW1hXzEsQSxcXEdhbW1hXzIpIl0sWzMsMCwiXFxvdGltZXMoXFxHYW1tYV8xLFxccGFycl8xIEEsXFxHYW1tYV8yKSJdLFs0LDEsIlxccGFycl8xKFxcb3RpbWVzKFxcR2FtbWFfMSxBLFxcR2FtbWFfMikpIl0sWzQsMywiXFxwYXJyXzEoXFxwYXJyIFxcRGVsdGEpIl0sWzQsNSwiXFxwYXJyIFxcRGVsdGEiXSxbMiwzLCJcXHBhcnIgXFxEZWx0YSJdLFswLDEsIlxcYWxwaGEnIiwxXSxbMSwyLCJcXG90aW1lcygtLFxcZXRhX0EsLSkiLDFdLFsyLDMsIlxcb3RpbWVzKC0sXFxpb3RhX0EsLSkiLDFdLFszLDQsIlxcZGVsdGEiLDFdLFs0LDUsIlxccGFycl8xKC0sZywtKSIsMV0sWzUsNiwiXFxnYW1tYSciLDFdLFswLDIsIiIsMSx7ImxldmVsIjoyLCJzdHlsZSI6eyJoZWFkIjp7Im5hbWUiOiJub25lIn19fV0sWzIsNCwiXFxpb3RhX3tcXG90aW1lcyhcXEdhbW1hXzEsQSxcXEdhbW1hXzIpfSIsMV0sWzIsNywiZyIsMV0sWzcsNSwiXFxpb3RhX3tcXHBhcnIgXFxEZWx0YX0iLDFdLFs3LDYsIiIsMSx7ImxldmVsIjoyLCJzdHlsZSI6eyJoZWFkIjp7Im5hbWUiOiJub25lIn19fV1d
\resizebox{\hsize}{!}{
\begin{tikzcd}[ampersand replacement=\&]
	\& {\otimes(\Gamma_1,\otimes_1A,\Gamma_2)} \&\& {\otimes(\Gamma_1,\parr_1 A,\Gamma_2)} \\
	{\otimes(\Gamma_1,A,\Gamma_2)} \&\& {\otimes(\Gamma_1,A,\Gamma_2)} \&\& {\parr_1(\otimes(\Gamma_1,A,\Gamma_2))} \\
	\\
	\&\& {\parr \Delta} \&\& {\parr_1(\parr \Delta)} \\
	\\
	\&\&\&\& {\parr \Delta}
	\arrow["{\alpha'}"{description}, from=2-1, to=1-2]
	\arrow["{\otimes(-,\eta_A,-)}"{description}, from=1-2, to=2-3]
	\arrow["{\otimes(-,\iota_A,-)}"{description}, from=2-3, to=1-4]
	\arrow["\delta"{description}, from=1-4, to=2-5]
	\arrow["{\parr_1(-,g,-)}"{description}, from=2-5, to=4-5]
	\arrow["{\gamma'}"{description}, from=4-5, to=6-5]
	\arrow[Rightarrow, no head, from=2-1, to=2-3]
	\arrow["{\iota_{\otimes(\Gamma_1,A,\Gamma_2)}}"{description}, from=2-3, to=2-5]
	\arrow["g"{description}, from=2-3, to=4-3]
	\arrow["{\iota_{\parr \Delta}}"{description}, from=4-3, to=4-5]
	\arrow[Rightarrow, no head, from=4-3, to=6-5]
\end{tikzcd}}
where the top left and bottom right triangles commutes by definition of 
$\alpha'$/$\gamma'$ and their coherence laws.
The other triangle is one of the coherence laws for a luldc, while the square is by naturality of $\iota$.
Similarly the following diagram commutes:
% https://q.uiver.app/?q=WzAsOCxbMCwwLCJcXG90aW1lcyBcXEdhbW1hIl0sWzIsMCwiXFxvdGltZXNfMVxcb3RpbWVzXFxHYW1tYSJdLFs0LDAsIlxcb3RpbWVzXzFcXHBhcnIoXFxEZWx0YV8xLEEsXFxEZWx0YV8yKSJdLFs1LDEsIlxccGFycihcXERlbHRhXzEsXFxvdGltZXNfMSBBLFxcRGVsdGFfMikiXSxbNCwyLCJcXHBhcnIoXFxEZWx0YV8xLEEsXFxEZWx0YV8yKSJdLFs1LDMsIlxccGFycihcXERlbHRhXzEsXFxwYXJyXzEgQSwgXFxEZWx0YV8yKSJdLFs0LDQsIlxccGFycihcXERlbHRhXzEsQSxcXERlbHRhXzIpIl0sWzIsMiwiXFxvdGltZXNcXEdhbW1hIl0sWzAsMSwiXFxhbHBoYSciLDFdLFsxLDIsIlxcb3RpbWVzXzFmIiwxXSxbMiwzLCJcXGRlbHRhIiwxXSxbMyw0LCJcXHBhcnIoLSxcXGV0YV9BLC0pIiwxXSxbNCw1LCJcXHBhcnIoLSxcXGlvdGFfQSwtKSIsMV0sWzUsNiwiXFxnYW1tYSciLDFdLFs0LDYsIiIsMSx7ImxldmVsIjoyLCJzdHlsZSI6eyJoZWFkIjp7Im5hbWUiOiJub25lIn19fV0sWzIsNCwiXFxldGFfe1xccGFycihcXERlbHRhXzEsQSxcXERlbHRhXzIpfSIsMV0sWzEsNywiXFxldGFfe1xcb3RpbWVzXFxHYW1tYX0iLDFdLFs3LDQsImYiLDFdLFswLDcsIiIsMSx7ImxldmVsIjoyLCJzdHlsZSI6eyJoZWFkIjp7Im5hbWUiOiJub25lIn19fV1d
\[\begin{tikzcd}[ampersand replacement=\&]
	{\otimes \Gamma} \&\& {\otimes_1\otimes\Gamma} \&\& {\otimes_1\parr(\Delta_1,A,\Delta_2)} \\
	\&\&\&\&\& {\parr(\Delta_1,\otimes_1 A,\Delta_2)} \\
	\&\& \otimes\Gamma \&\& {\parr(\Delta_1,A,\Delta_2)} \\
	\&\&\&\&\& {\parr(\Delta_1,\parr_1 A, \Delta_2)} \\
	\&\&\&\& {\parr(\Delta_1,A,\Delta_2)}
	\arrow["{\alpha'}"{description}, from=1-1, to=1-3]
	\arrow["{\otimes_1f}"{description}, from=1-3, to=1-5]
	\arrow["\delta"{description}, from=1-5, to=2-6]
	\arrow["{\parr(-,\eta_A,-)}"{description}, from=2-6, to=3-5]
	\arrow["{\parr(-,\iota_A,-)}"{description}, from=3-5, to=4-6]
	\arrow["{\gamma'}"{description}, from=4-6, to=5-5]
	\arrow[Rightarrow, no head, from=3-5, to=5-5]
	\arrow["{\eta_{\parr(\Delta_1,A,\Delta_2)}}"{description}, from=1-5, to=3-5]
	\arrow["{\eta_{\otimes\Gamma}}"{description}, from=1-3, to=3-3]
	\arrow["f"{description}, from=3-3, to=3-5]
	\arrow[Rightarrow, no head, from=1-1, to=3-3]
\end{tikzcd}\]
Now we want to prove associativity of composition:
\[\resizebox{\hsize}{!}{\begin{tikzcd}[ampersand replacement=\&]
	{\otimes(\Gamma_1'',\Gamma_1',\Gamma,\Gamma_2',\Gamma_2'')} \&\& {\otimes(\Gamma_1'',\Gamma_1',\otimes\Gamma,\Gamma_2',\Gamma_2'')} \&\& {\otimes(\Gamma_1'',\Gamma_1',\parr(\Delta_1,A,\Delta_2),\Gamma_2',\Gamma_2'')} \\
	\\
	{\otimes(\Gamma_1'',\otimes(\Gamma_1',\Gamma,\Gamma_2'),\Gamma_2'')} \&\&\&\& {\parr(\Delta_1,\otimes(\Gamma_1'',\Gamma_1',A,\Gamma_2',\Gamma_2''),\Delta_2)} \\
	\\
	{\otimes(\Gamma_1'',\otimes(\Gamma_1',\otimes\Gamma,\Gamma_2'),\Gamma_2'')} \&\&\&\& {\parr(\Delta_1,\otimes(\Gamma_1'',\otimes(\Gamma_1',A,\Gamma_2'),\Gamma_2''),\Delta_2)} \\
	\\
	{\otimes(\Gamma_1'',\otimes(\Gamma_1',\parr(\Delta_1,A,\Delta_2),\Gamma_2'),\Gamma_2'')} \&\&\&\& {\parr(\Delta_1,\otimes(\Gamma_1'',\parr(\Delta_1',B,\Delta_2'),\Gamma_2''),\Delta_2)} \\
	\\
	{\otimes(\Gamma_1'',\parr(\Delta_1,\otimes(\Gamma_1',A,\Gamma_2'),\Delta_2),\Gamma_2'')} \&\&\&\& {\parr(\Delta_1,\parr(\Delta_1',\otimes(\Gamma_1'',B,\Gamma_2''),\Delta_2'),\Delta_2)} \\
	\\
	{\otimes(\Gamma_1'',\parr(\Delta_1,\parr(\Delta_1',B,\Delta_2'),\Delta_2),\Gamma_2'')} \&\&\&\& {\parr(\Delta_1,\parr(\Delta_1',\parr \Delta'',\Delta_2'),\Delta_2)} \\
	\\
	{\otimes(\Gamma_1'',\parr(\Delta_1,\Delta_1',B,\Delta_2',\Delta_2),\Gamma_2'')} \&\&\&\& {\parr(\Delta_1,\parr(\Delta_1',\Delta'',\Delta_2'),\Delta_2)} \\
	\\
	{\parr(\Delta_1,\Delta_1',\otimes(\Gamma_1'',B,\Gamma_2''),\Delta_2',\Delta_2)} \&\& {\parr(\Delta_1,\Delta_1',\parr \Delta,\Delta_2',\Delta_2)} \&\& {\parr(\Delta_1,\Delta_1',\Delta'',\Delta_2',\Delta_2)}
	\arrow["{\alpha'}"{description}, from=1-1, to=1-3]
	\arrow["{\otimes(-,f,-)}"{description}, from=1-3, to=1-5]
	\arrow["\delta"{description}, from=1-5, to=3-5]
	\arrow["{\parr(-,\alpha',-)}"{description}, from=3-5, to=5-5]
	\arrow["{\parr(-,\otimes(-,g,-),-)}"{description}, from=5-5, to=7-5]
	\arrow["{\parr(-,\delta,-)}"{description}, from=7-5, to=9-5]
	\arrow["{\parr(-,\parr(-,h,-),-)}"{description}, from=9-5, to=11-5]
	\arrow["{\parr(-,\gamma',-)}"{description}, from=11-5, to=13-5]
	\arrow["{\gamma'}"{description}, from=13-5, to=15-5]
	\arrow["{\alpha'}"{description}, from=1-1, to=3-1]
	\arrow["{\otimes(-,\alpha',-)}"{description}, from=3-1, to=5-1]
	\arrow["{\otimes(-,\otimes(-,f,-),-)}"{description}, from=5-1, to=7-1]
	\arrow["{\otimes(-,\delta,-)}"{description}, from=7-1, to=9-1]
	\arrow["{\otimes(-,\parr(-,g,-),-)}"{description}, from=9-1, to=11-1]
	\arrow["{\otimes(-,\gamma',-)}"{description}, from=11-1, to=13-1]
	\arrow["\delta"{description}, from=13-1, to=15-1]
	\arrow["{\parr(-,h,-)}"{description}, from=15-1, to=15-3]
	\arrow["{\gamma'}"{description}, from=15-3, to=15-5]
	\arrow["{\gamma'}"{description}, from=11-5, to=15-3]
	\arrow["{\gamma'}"{description}, from=9-5, to=15-1]
	\arrow["\delta"{description}, from=11-1, to=7-5]
	\arrow["\delta"{description}, from=9-1, to=5-5]
	\arrow["{\alpha'}"{description}, from=1-5, to=7-1]
	\arrow["{\alpha'}"{description}, from=1-3, to=5-1]
\end{tikzcd}}\]
where the laws involving $\alpha'$ and $\gamma'$ follow from similar ones for $\alpha$ and $\gamma$, namely from top-left to bottom-right: coherence law of an oplax monoidal structure, naturality, coherence law of a luldc, naturality, coherence law of a luldc, naturality and coherence law of a lax monoidal structure.

Now given $g \colon \otimes(\Gamma_1',A_1,\Gamma_2',A_2,\Gamma_3') \to \parr \Delta'$ and $f_1 \colon \otimes \Gamma_1 \to \parr(\Delta_1,A_1)$ and $f_2 \colon \otimes \Gamma_2 \to \parr(A_2,\Delta_2)$ we have $(g \circ f_1) \circ f_2 = (g \circ f_2) \circ f_1 \colon \otimes(\Gamma_1',\Gamma_1,\Gamma_2',\Gamma_2,\Gamma_3') \to \parr(\Delta_1,\Delta',\Delta_2)$.
This is given by the following commuting diagram:
\[\resizebox{\hsize}{!}{\begin{tikzcd}[ampersand replacement=\&]
	{\otimes(\Gamma_1',\Gamma_1,\Gamma_2',\Gamma_2,\Gamma_3')} \&\& {\otimes(\Gamma_1',\otimes \Gamma_1,\Gamma_2',\Gamma_2,\Gamma_3')} \&\& {\otimes(\Gamma_1',\parr(\Delta_1,A_1),\Gamma_2',\Gamma_2,\Gamma_3')} \&\& {\parr(\Delta_1,\otimes(\Gamma_1',A_1,\Gamma_2',\Gamma_2,\Gamma_3'))} \\
	\& {\text{coh. of oplax mon. cat.}} \&\& {\text{nat. of } \alpha'} \\
	{\otimes(\Gamma_1',\Gamma_1,\Gamma_2',\otimes \Gamma_2,\Gamma_3')} \&\& {\otimes(\Gamma_1',\otimes\Gamma_1,\Gamma_2',\otimes\Gamma_2,\Gamma_3')} \&\& {\text{nat. of } \delta} \&\& {\parr(\Delta_1,\otimes(\Gamma_1',A_1,\Gamma_2',\otimes\Gamma_2,\Gamma_3'))} \\
	\& {\otimes(\Gamma_1',\otimes\Gamma_1,\Gamma_2',\parr(A_2,\Delta_2),\Gamma_3')} \& {\text{funct. of } \otimes} \& {\otimes(\Gamma_1',\parr(A_1,\Delta_1),\Gamma_2',\otimes\Gamma_2,\Gamma_3')} \\
	{\otimes(\Gamma_1',\Gamma_1,\Gamma_2',\parr(A_2,\Delta_2),\Gamma_3')} \&\& {\otimes(\Gamma_1',\parr(\Delta_1,A_1),\Gamma_2',\parr(A_2,\Delta_2),\Gamma_3')} \&\&\&\& {\parr(\Delta_1,\otimes(\Gamma_1',A_1,\Gamma_2',\parr(A_2,\Delta_2),\Gamma_3'))} \\
	\& {\text{nat. of } \delta} \& {\text{coh. of luldc}} \\
	{\parr(\otimes(\Gamma_1',\Gamma_1,\Gamma_2',A_2,\Gamma_3'),\Delta_2)} \&\& {\parr(\Delta_1,\otimes(\Gamma_1',A_1,\Gamma_2',A_2,\Gamma_3'),\Delta_2)} \&\&\&\& {\parr(\Delta_1,\parr(\otimes(\Gamma_1',A_1,\Gamma_2',A_2,\Gamma_3'),\Delta_2))} \\
	\&\&\&\& {\text{nat. of } \gamma'} \\
	{\parr(\otimes(\Gamma_1',\otimes\Gamma_1,\Gamma_2',A_2,\Gamma_3'),\Delta_2)} \&\& {\parr(\Delta_1,\parr \Delta',\Delta_2)} \&\&\&\& {\parr(\Delta_1,\parr(\parr \Delta',\Delta_2))} \\
	\&\&\&\& {\text{coh. of lax mon. cat.}} \\
	{\parr(\otimes(\Gamma_1',\parr(\Delta_1,A_1),\Gamma_2',A_2,\Gamma_3'),\Delta_2)} \& {\text{nat. of } \gamma'} \&\&\&\&\& {\parr(\Delta_1,\parr(\Delta',\Delta_2))} \\
	\&\&\& {\text{coh. of lax mon. cat.}} \\
	{\parr(\parr(\Delta_1,\otimes(\Gamma_1',A_1,\Gamma_2',A_2,\Gamma_3')),\Delta_2)} \&\& {\parr(\parr(\Delta_1,\parr \Delta'),\Delta_2)} \&\& {\parr(\parr(\Delta_1,\Delta'),\Delta_2)} \&\& {\parr(\Delta_1,\Delta',\Delta_2)}
	\arrow["{\alpha'}"{description}, from=1-1, to=1-3]
	\arrow["{\otimes(-,f_1,-,-,-)}"{description}, from=1-3, to=1-5]
	\arrow["\delta"{description}, from=1-5, to=1-7]
	\arrow["{\parr(-,\alpha')}"{description}, from=1-7, to=3-7]
	\arrow["{\parr(-,\otimes(-,-,-,f_2,-))}"{description}, from=3-7, to=5-7]
	\arrow["{\parr(-,\delta)}"{description}, from=5-7, to=7-7]
	\arrow["{\parr(-,\parr(g,-))}"{description}, from=7-7, to=9-7]
	\arrow["{\parr(-,\gamma')}"{description}, from=9-7, to=11-7]
	\arrow["{\gamma'}"{description}, from=11-7, to=13-7]
	\arrow["{\alpha'}"{description}, from=1-1, to=3-1]
	\arrow["{\otimes(-,-,-,f_2,-)}"{description}, from=3-1, to=5-1]
	\arrow["\delta"{description}, from=5-1, to=7-1]
	\arrow["{\parr(\alpha',-)}"{description}, from=7-1, to=9-1]
	\arrow["{\parr(\otimes(-,f_1,-,-,-),-)}"{description}, from=9-1, to=11-1]
	\arrow["{\parr(\delta,-)}"{description}, from=11-1, to=13-1]
	\arrow["{\parr(\parr(-,g),-)}"{description}, from=13-1, to=13-3]
	\arrow["{\parr(\gamma',-)}"{description}, from=13-3, to=13-5]
	\arrow["{\gamma'}"{description}, from=13-5, to=13-7]
	\arrow["{\gamma'}"{description}, from=13-3, to=9-3]
	\arrow["{\gamma'}"{description}, from=9-7, to=9-3]
	\arrow["{\gamma'}"{description}, from=9-3, to=13-7]
	\arrow["{\parr(-,g,-)}"{description}, from=7-3, to=9-3]
	\arrow["{\gamma'}"{description}, from=13-1, to=7-3]
	\arrow["{\gamma'}"{description}, from=7-7, to=7-3]
	\arrow["\delta"{description}, from=5-3, to=11-1]
	\arrow["\delta"{description}, from=5-3, to=5-7]
	\arrow["{\otimes(-,f_1,-,-,-)}"{description}, from=4-2, to=5-3]
	\arrow["{\otimes(-,-,-,f_2,-)}"{description}, from=3-3, to=4-2]
	\arrow["{\alpha'}"{description}, from=5-1, to=4-2]
	\arrow["{\alpha'}"{description}, from=3-1, to=3-3]
	\arrow["{\alpha'}"{description}, from=1-3, to=3-3]
	\arrow["{\otimes(-,f_1,-,-,-)}"{description}, from=3-3, to=4-4]
	\arrow["{\otimes(-,-,-,f_2,-)}"{description}, from=4-4, to=5-3]
	\arrow["{\alpha'}"{description}, from=1-5, to=4-4]
\end{tikzcd}}\]

Similarly for composing on two different outputs:
\[\resizebox{\hsize}{!}{\begin{tikzcd}[ampersand replacement=\&]
	{\otimes(\Gamma_1',\Gamma,\Gamma_2')} \&\& {\otimes(\otimes (\Gamma_1', \Gamma), \Gamma_2')} \&\& {\otimes(\otimes(\Gamma_1',\otimes\Gamma),\Gamma_2')} \&\& {\otimes(\otimes(\Gamma_1',\parr(\Delta_1,A_1,\Delta_2,A_2,\Delta_3)),\Gamma_2')} \&\& {\otimes(\parr(\Delta_1,\otimes(\Gamma_1',A_1),\Delta_2,A_2,\Delta_3),\Gamma_2')} \\
	\\
	{\otimes(\Gamma_1',\otimes(\Gamma,\Gamma_2'))} \\
	\\
	{\otimes(\Gamma_1',\otimes(\otimes \Gamma, \Gamma_2'))} \&\&\&\& {\otimes(\Gamma_1',\otimes\Gamma,\Gamma_2')} \&\&\&\& {\otimes(\parr(\Delta_1,\parr \Delta_1',\Delta_2,A_2,\Delta_3),\Gamma_2')} \\
	\\
	{\otimes(\Gamma_1',\otimes(\parr(\Delta_1,A_1,\Delta_2,A_2,\Delta_3),\Gamma_2'))} \&\&\&\& {\otimes(\Gamma_1',\parr(\Delta_1,A_1,\Delta_2,A_2,\Delta_3),\Gamma_2')} \&\&\&\& {\otimes(\parr(\Delta_1,\Delta_1',\Delta_2,A_2,\Delta_3),\Gamma_2')} \\
	\\
	{\otimes(\Gamma_1',\parr(\Delta_1,A_1,\Delta_2,\otimes(A_2,\Gamma_2'),\Delta_3))} \&\&\&\& {\parr(\Delta_1,\otimes(\Gamma_1',A_1),\Delta_2,\otimes(A_2,\Gamma_2'),\Delta_3)} \&\& {\parr(\Delta_1,\parr\Delta_1',\Delta_2,\otimes(A_2,\Gamma_2'),\Delta_3)} \&\& {\parr(\Delta_1,\Delta_1',\Delta_2,\otimes(A_2,\Gamma_2'),\Delta_3)} \\
	\\
	{\otimes(\Gamma_1',\parr(\Delta_1,A_1,\Delta_2,\parr \Delta_2', \Delta_3))} \&\&\&\& {\parr(\Delta_1,\otimes(\Gamma_1',A_1),\Delta_2,\parr\Delta_2',\Delta_3)} \&\& {\parr(\Delta_1,\parr\Delta_1',\Delta_2,\parr\Delta_2',\Delta_3)} \&\& {\parr(\Delta_1,\Delta_1',\Delta_2,\parr \Delta_2',\Delta_3)} \\
	\\
	{\otimes(\Gamma_1',\parr(\Delta_1,A_1,\Delta_2,\Delta_2',\Delta_3))} \&\&\&\& {\parr(\Delta_1,\otimes(\Gamma_1',A_1),\Delta_2,\Delta_2',\Delta_3)} \&\& {\parr(\Delta_1,\parr \Delta_1',\Delta_2,\Delta_2',\Delta_3)} \&\& {\parr(\Delta_1,\Delta_1',\Delta_2,\Delta_2',\Delta_3)}
	\arrow["{\alpha'}"{description}, from=1-1, to=1-3]
	\arrow["{\otimes(\alpha',-)}"{description}, from=1-3, to=1-5]
	\arrow["{\otimes(\otimes(-,f),-)}"{description}, from=1-5, to=1-7]
	\arrow["{\otimes(\delta,-)}"{description}, from=1-7, to=1-9]
	\arrow["{\otimes(\parr(-,g_1,-,-,-),-)}"{description}, from=1-9, to=5-9]
	\arrow["{\otimes(\gamma',-)}"{description}, from=5-9, to=7-9]
	\arrow["\delta"{description}, from=7-9, to=9-9]
	\arrow["{\parr(-,-,-,g_2,-)}"{description}, from=9-9, to=11-9]
	\arrow["{\gamma'}"{description}, from=11-9, to=13-9]
	\arrow["{\alpha'}"{description}, from=1-1, to=3-1]
	\arrow["{\otimes(-,\alpha')}"{description}, from=3-1, to=5-1]
	\arrow["{\otimes(-,\otimes(f,-))}"{description}, from=5-1, to=7-1]
	\arrow["{\otimes(-,\delta)}"{description}, from=7-1, to=9-1]
	\arrow["{\otimes(-,\parr(-,-,-,g_2,-))}"{description}, from=9-1, to=11-1]
	\arrow["{\otimes(-,\gamma')}"{description}, from=11-1, to=13-1]
	\arrow["\delta"{description}, from=13-1, to=13-5]
	\arrow["{\parr(-,g_1,-,-,-)}"{description}, from=13-5, to=13-7]
	\arrow["{\gamma'}"{description}, from=13-7, to=13-9]
	\arrow["{\alpha'}"{description}, from=1-1, to=5-5]
	\arrow["{\alpha'}"{description}, from=5-5, to=5-1]
	\arrow["{\alpha'}"{description}, from=5-5, to=1-5]
	\arrow["{\otimes(-,f,-)}"{description}, from=5-5, to=7-5]
	\arrow["{\alpha'}"{description}, from=7-5, to=7-1]
	\arrow["{\alpha'}"{description}, from=7-5, to=1-7]
	\arrow["{\gamma'}"{description}, from=11-7, to=13-7]
	\arrow["{\gamma'}"{description}, from=11-7, to=11-9]
	\arrow["{\parr(-,-,-,g_2,-)}"{description}, from=9-5, to=11-5]
	\arrow["{\parr(-,g_1,-,-,-)}"{description}, from=11-5, to=11-7]
	\arrow["{\parr(-,g_1,-,-,-)}"{description}, from=9-5, to=9-7]
	\arrow["{\parr(-,-,-,g_2,-)}"{description}, from=9-7, to=11-7]
	\arrow["{\gamma'}"{description}, from=9-7, to=9-9]
	\arrow["{\gamma'}"{description}, from=11-5, to=13-5]
	\arrow["\delta"{description}, from=1-9, to=9-5]
	\arrow["\delta"{description}, from=9-1, to=9-5]
\end{tikzcd}}\]
\end{proof}

We can extend \cP to a 2-functor $\cP \colon \LULDCFrob \to \PolyCat$.
It is defined on functors by:
\begin{itemize}
\item $\cP(F)(A) := F(A)$
\item $\cP(F)(f) := \otimes_m F(A_i) \xrightarrow{F_m} F(\otimes_m A_i) \xrightarrow{F(f)} F(\parr_n B_j) \xrightarrow{F^n} \parr_n F(B_j)$ where $F_m$ comes from the $\otimes$-lax structure of $F$ and $F^n$ from its $\parr$-oplax structure.
\end{itemize}

\begin{prop}
For a Frobenius functor $F \colon \cC \to \cD$, $\cP(F) \colon \cP(\cC) \to \cP(\cD)$ is a functor of polycategories.
\end{prop}
\begin{proof}
First it preserves identity:
% https://q.uiver.app/?q=WzAsNSxbMCwxLCJcXG90aW1lc18xRihBKSJdLFsxLDAsIkYoXFxvdGltZXNfMSBBKSJdLFsyLDEsIkYoQSkiXSxbMywyLCJGKFxccGFycl8xIEEpIl0sWzIsMywiXFxwYXJyXzFGKEEpIl0sWzAsMSwiRl8xIiwxXSxbMSwyLCJGKFxcZXRhX0EpIiwxXSxbMiwzLCJGKFxcaW90YV9BKSIsMV0sWzMsNCwiRl4xIiwxXSxbMCwyLCJcXGV0YV97RihBKX0iLDFdLFsyLDQsIlxcaW90YV97RihBKX0iLDFdXQ==
\[\begin{tikzcd}[ampersand replacement=\&]
	\& {F(\otimes_1 A)} \\
	{\otimes_1F(A)} \&\& {F(A)} \\
	\&\&\& {F(\parr_1 A)} \\
	\&\& {\parr_1F(A)}
	\arrow["{F_1}"{description}, from=2-1, to=1-2]
	\arrow["{F(\eta_A)}"{description}, from=1-2, to=2-3]
	\arrow["{F(\iota_A)}"{description}, from=2-3, to=3-4]
	\arrow["{F^1}"{description}, from=3-4, to=4-3]
	\arrow["{\eta_{F(A)}}"{description}, from=2-1, to=2-3]
	\arrow["{\iota_{F(A)}}"{description}, from=2-3, to=4-3]
\end{tikzcd}\]
Furthermore it preserves composition:
\[\resizebox{\hsize}{!}{\begin{tikzcd}[ampersand replacement=\&]
	{\otimes (F(\Gamma_1),F(\Gamma),F(\Gamma_2))} \&\& {\otimes(F(\Gamma_1),\otimes F(\Gamma),F(\Gamma_2))} \&\& {\otimes(F(\Gamma_1),F(\otimes \Gamma),F(\Gamma_2))} \&\& {\otimes(F(\Gamma_1),F(\parr(\Delta_1,A,\Delta_2)),F(\Gamma_2))} \&\& {\otimes(F(\Gamma_1),\parr(F(\Delta_1),F(A),F(\Delta_2)),F(\Gamma_2))} \\
	\& {\text{coh. for lax mon. funct.}} \\
	{F(\otimes(\Gamma_1,\Gamma,\Gamma_2))} \&\&\&\&\&\&\&\& {\parr(F(\Delta_1),\otimes(F(\Gamma_1),F(A),F(\Gamma_2)),F(\Delta_2))} \\
	\&\& {\text{nat of } F_{m_1+1+m_2}} \\
	{F(\otimes(\Gamma_1,\otimes\Gamma,\Gamma_2))} \&\&\&\&\& {\text{coh. for linear functor}} \&\&\& {\parr(F(\Delta_1),F(\otimes(\Gamma_1,A,\Gamma_2)),F(\Delta_2))} \\
	\\
	{F(\otimes(\Gamma_1,\parr(\Delta_1,A,\Delta_2),\Gamma_2))} \&\&\&\&\&\&\&\& {\parr(F(\Delta_1),F(\parr \Delta),F(\Delta_2))} \\
	\\
	\&\&\&\&\& {\text{nat. of } F^{n_1+1+n_2}} \&\&\& {\parr(F(\Delta_1),\parr F(\Delta),F(\Delta_2))} \\
	\&\&\&\&\&\&\& {\text{coh. for oplax mon. funct.}} \\
	{F(\parr(\Delta_1,\otimes(\Gamma_1,A,\Gamma_2),\Delta_2))} \&\&\&\& {F(\parr(\Delta_1,\parr\Delta,\Delta_2))} \&\& {F(\parr(\Delta_1,\Delta,\Delta_2))} \&\& {\parr(F(\Delta_1),F(\Delta),F(\Delta_2))}
	\arrow["{\alpha'}"{description}, from=1-1, to=1-3]
	\arrow["{\otimes(-,F_m,-)}"{description}, from=1-3, to=1-5]
	\arrow["{\otimes(-,F(f),-)}"{description}, from=1-5, to=1-7]
	\arrow["{\otimes(-,F^{n_1+1+n_2},-)}"{description}, from=1-7, to=1-9]
	\arrow["\delta"{description}, from=1-9, to=3-9]
	\arrow["{\parr(F(\Gamma_1),F_{m_1+1+m_2},F(\Gamma_2))}"{description}, from=3-9, to=5-9]
	\arrow["{\parr(-,F(g),-)}"{description}, from=5-9, to=7-9]
	\arrow["{\parr(-,F^n,-)}"{description}, from=7-9, to=9-9]
	\arrow["{\gamma'}"{description}, from=9-9, to=11-9]
	\arrow["{F_{m_1+m+m_2}}"{description}, from=1-1, to=3-1]
	\arrow["{F(\alpha')}"{description}, from=3-1, to=5-1]
	\arrow["{F(\otimes(-,f,-))}"{description}, from=5-1, to=7-1]
	\arrow["{F(\delta)}"{description}, from=7-1, to=11-1]
	\arrow["{F(\parr(-,g,-))}"{description}, from=11-1, to=11-5]
	\arrow["{F(\gamma')}"{description}, from=11-5, to=11-7]
	\arrow["{F^{n_1+n+n_2}}"{description}, from=11-7, to=11-9]
	\arrow["{F^{n_1+1+n_2}}"{description}, from=11-5, to=7-9]
	\arrow["{F^{n_1+1+n_2}}"{description}, from=11-1, to=5-9]
	\arrow["{F_{m_1+1+m_2}}"{description}, from=1-5, to=5-1]
	\arrow["{F_{m_1+1+m_2}}"{description}, from=1-7, to=7-1]
\end{tikzcd}}\]
\end{proof}

On transformations it is defined as $\cP(\alpha)_A := \otimes_1 F(A) \xrightarrow{\eta_A} F(A) \xrightarrow{\alpha_A} G(A) \xrightarrow{\iota_A} \parr_1 G(A)$

\begin{prop}
If $\alpha \colon F \Rightarrow G$ is a linear transformation between Frobenius functors then $\cP(\alpha)\colon \cP(F) \to \cP(G)$ is a natural transformation between functors of polycategories.
\end{prop}
\begin{proof}
We need to prove that for any $f \colon A_1,\dots, A_m \to B_1, \dots, B_n$ in $\cP(\cD)$, \[(\cP(\alpha)_{B_1},\dots,\cP(\alpha)_{B_n}) \circ \cP(F)(f) = \cP(G)(f) \circ (\cP(\alpha)_{A_1},\dots,\cP(\alpha)_{A_n})\]
First, let prove that the polymap $(\cP(\alpha)_{B_1},\dots,\cP(\alpha)_{B_n}) \circ \cP(F)(f)$ in $\cP(\cD)$ is the map $\parr(\alpha_{B_1},\dots,\alpha_{B_n}) \circ F^n \circ F(f) \circ F_m$ in \cD .
The proof is given in a figure on the next page.
It uses the coherence laws of a luldc, functoriality of $\parr$ and naturality of $\eta$.
\begin{figure}
\rotatebox{90}{
\resizebox{\vsize}{!}{
\begin{tikzcd}[ampersand replacement=\&]
	{\otimes(F(A_1),\dots,F(A_m))} \&\& {\otimes\otimes(F(A_1),\dots,F(A_m))} \& \dots \& {\otimes^{n-j+1}(F(A_1),\dots,F(A_m))} \&\& {\otimes^{n-j+2}(F(A_1),\dots,F(A_m))} \& \dots \& {\otimes^n(F(A_1),\dots,F(A_m))} \&\& {\otimes^{n+1}(F(A_1),\dots,F(A_m))} \&\& {\otimes^n F(\otimes(A_1,\dots,A_m))} \&\& {\otimes^nF(\parr(B_1,\dots,B_n))} \&\& {\otimes^n\parr(F(B_1),\dots,F(B_n))} \\
	\&\&\&\&\&\&\&\&\&\&\&\&\&\&\&\&\& {\otimes^{n-1}\parr(\otimes F(B_1),\dots,F(B_n))} \\
	\&\&\&\&\&\&\&\&\&\& {\otimes^n(F(A_1),\dots,F(A_m))} \&\&\&\&\& {\otimes^{n-1}\parr(F(B_1),\dots,F(B_n))} \& {\otimes^{n-1}\parr(F(B_1),\dots,F(B_n))} \\
	\&\&\&\&\&\&\&\&\&\& \vdots \&\&\&\&\& \vdots \\
	\&\&\&\&\&\&\&\&\&\& {\otimes^{n-j+2}(F(A_1),\dots,F(A_m))} \&\&\&\&\& {\otimes^{n-j+1}\parr(F(B_1),\dots,F(B_n))} \& {\otimes^{n-1}\parr(G(B_1),\dots,F(B_n))} \\
	\&\&\&\&\&\&\&\&\&\&\&\&\&\&\&\&\& {\otimes^{n-1}\parr(\parr G(B_1),\dots,F(B_n))} \\
	\&\&\&\&\&\&\&\&\&\& {\otimes^{n-j+1}(F(A_1),\dots,F(A_m))} \&\&\&\&\& {\otimes^{n-j}\parr(F(B_1),\dots,F(B_n))} \& {\otimes^{n-1}\parr(G(B_1),\dots,F(B_n))} \\
	\&\&\&\&\&\&\&\&\&\& \vdots \&\&\&\&\& \vdots \& \vdots \\
	{F(\otimes(A_1,\dots,A_m))} \&\&\&\&\&\&\&\&\&\& {\otimes\otimes(F(A_1),\dots,F(A_m))} \&\&\&\&\& {\otimes\parr(F(B_1),\dots,F(B_n))} \& {\otimes^{n-j+1}\parr(G(B_1),\dots,G(B_{j-1}),F(B_j),F(B_{j+1}),\dots,F(B_n))} \\
	\&\&\&\&\&\&\&\&\&\&\&\&\&\&\&\&\& {\otimes^{n-j}\parr(G(B_1),\dots,G(B_{j-1}),\otimes F(B_j),F(B_{j+1}),\dots,F(B_n))} \\
	\&\&\&\&\&\&\&\&\&\&\&\&\&\&\&\& {\otimes^{n-j}\parr(G(B_1),\dots,G(B_{j-1}),F(B_j),F(B_{j+1}),\dots,F(B_n))} \\
	\&\&\&\&\&\&\&\&\&\& {\otimes(F(A_1),\dots,F(A_m))} \&\& {F(\otimes(A_1,\dots,A_m))} \&\& {F(\parr(B_1,\dots,B_n))} \& {\parr(F(B_1),\dots,F(B_n))} \\
	\&\&\&\&\&\&\&\&\&\&\&\&\&\&\&\& {\otimes^{n-j}\parr(G(B_1),\dots,G(B_{j-1}),G(B_j),F(B_{j+1}),\dots,F(B_n))} \\
	\&\&\&\&\&\&\&\&\&\&\&\&\&\&\&\&\& {\otimes^{n-j}\parr(G(B_1),\dots,G(B_{j-1}),\parr G(B_j),F(B_{j+1}),\dots,F(B_n))} \\
	{F(\parr(B_1,B_n))} \&\&\&\&\&\&\&\&\&\&\&\&\&\&\& {\parr(G(B_1),\dots,F(B_n))} \& {\otimes^{n-j}\parr(G(B_1),\dots,G(B_{j-1}),G(B_j),F(B_{j+1}),\dots,F(B_n))} \\
	\&\&\&\&\&\&\&\&\&\&\&\&\&\&\& \vdots \& \vdots \\
	\&\&\&\&\&\&\&\&\&\&\&\&\&\&\& {\parr(G(B_1),\dots,G(B_{j-1}),F(B_j),F(B_{j+1}),\dots,F(B_n))} \& {\otimes \parr (G(B_1),\dots,F(B_n))} \\
	\&\&\&\&\&\&\&\&\&\&\&\&\&\&\&\&\& {\parr(G(B_1),\dots,\otimes F(B_n))} \\
	\&\&\&\&\&\&\&\&\&\&\&\&\&\&\& {\parr(G(B_1),\dots,G(B_{j-1}),G(B_j),F(B_{j+1}),\dots,F(B_n))} \& {\parr(G(B_1),\dots,F(B_n))} \\
	\&\&\&\&\&\&\&\&\&\&\&\&\&\&\& \vdots \\
	\&\&\&\&\&\&\&\&\&\&\&\&\&\&\& {\parr(G(B_1),\dots,F(B_n))} \& {\parr(G(B_1),\dots,G(B_n))} \\
	\&\&\&\&\&\&\&\&\&\&\&\&\&\&\&\&\& {\parr(G(B_1),\dots,\parr G(B_n))} \\
	{\parr(F(B_1),\dots,F(B_n))} \&\&\&\&\&\&\&\&\&\&\&\&\&\&\&\& {\parr(G(B_1),\dots,G(B_n))}
	\arrow["\alpha"{description}, from=1-1, to=1-3]
	\arrow["{\otimes^{n-j}\alpha}"{description}, from=1-5, to=1-7]
	\arrow["{\otimes^{n-1}\alpha}"{description}, from=1-9, to=1-11]
	\arrow["{\otimes^n F_m}"{description}, from=1-11, to=1-13]
	\arrow["{\otimes^nF(f)}"{description}, from=1-13, to=1-15]
	\arrow["{\otimes^nF^n}"{description}, from=1-15, to=1-17]
	\arrow["{\otimes^{n-1} \delta}"{description}, from=1-17, to=2-18]
	\arrow["{\otimes^{n-1}\parr(\eta_{F(B_1)},-)}"{description}, from=2-18, to=3-17]
	\arrow["{\otimes^{n-1}(\alpha_{B_1},-)}"{description}, from=3-17, to=5-17]
	\arrow["{\otimes^{n-1}(\iota_{G(B_1)},-)}"{description}, from=5-17, to=6-18]
	\arrow["{\otimes^{n-1} \gamma'}"{description}, from=6-18, to=7-17]
	\arrow["{\otimes^{n-j}\delta}"{description}, from=9-17, to=10-18]
	\arrow["{\otimes^{n-j}(-,\eta_{F(B_j)},-)}"{description}, from=10-18, to=11-17]
	\arrow["{\otimes^{n-j}(-,\alpha_{B_j},-)}"{description}, from=11-17, to=13-17]
	\arrow["{\otimes^{n-j}(-,\iota_{G(B_j)},-)}"{description}, from=13-17, to=14-18]
	\arrow["{\otimes^{n-j}\gamma'}"{description}, from=14-18, to=15-17]
	\arrow["\delta"{description}, from=17-17, to=18-18]
	\arrow["{\parr(-,\eta_{F(B_n)})}"{description}, from=18-18, to=19-17]
	\arrow["{\parr(-,\alpha_{B_n})}"{description}, from=19-17, to=21-17]
	\arrow["{\parr(-,\iota_{G(B_n)})}"{description}, from=21-17, to=22-18]
	\arrow["{\gamma'}"{description}, from=22-18, to=23-17]
	\arrow[Rightarrow, no head, from=21-17, to=23-17]
	\arrow["{\eta_{\parr(G(B_1),\dots,F(B_n))}}"{description}, from=17-17, to=19-17]
	\arrow[Rightarrow, no head, from=13-17, to=15-17]
	\arrow["{\otimes^{n-j}\eta_{\parr(G(B_1),\dots,G(B_{j-1}),F(B_j),F(B_{j+1}),\dots,F(B_n))}}"{description}, from=9-17, to=11-17]
	\arrow[Rightarrow, no head, from=5-17, to=7-17]
	\arrow["{\otimes^{n-1}\eta_{\parr(F(B_1),\dots,F(B_n))}}"{description}, from=1-17, to=3-17]
	\arrow["{\otimes^{n-1}\eta_{\parr(F(B_1),\dots,F(B_n))}}"{description}, from=1-17, to=3-16]
	\arrow[Rightarrow, no head, from=3-16, to=3-17]
	\arrow["{\otimes^{n-j}\eta_{\parr(F(B_1),\dots,F(B_n))}}"{description}, from=5-16, to=7-16]
	\arrow["{\eta_{\parr(F(B_1),\dots,F(B_n))}}"{description}, from=9-16, to=12-16]
	\arrow["{\parr(\alpha_{B_1},-)}"{description}, from=12-16, to=15-16]
	\arrow["{\parr(-,\alpha_{B_j},-)}"{description}, from=17-16, to=19-16]
	\arrow["{\parr(-,\alpha_{B_n})}"{description}, from=21-16, to=23-17]
	\arrow[Rightarrow, no head, from=21-16, to=21-17]
	\arrow["{\otimes^{n-1}\eta_{\otimes(F(A_1),\dots,F(A_m))}}"{description}, from=1-11, to=3-11]
	\arrow["{\otimes^{n-j}\eta_{\otimes(F(A_1),\dots,F(A_m))}}"{description}, from=5-11, to=7-11]
	\arrow["{\eta_{\otimes(F(A_1),\dots,F(A_m))}}"{description}, from=9-11, to=12-11]
	\arrow["{F_m}"{description}, from=12-11, to=12-13]
	\arrow["{F(f)}"{description}, from=12-13, to=12-15]
	\arrow["{F^n}"{description}, from=12-15, to=12-16]
	\arrow[Rightarrow, no head, from=1-9, to=3-11]
	\arrow[Rightarrow, no head, from=1-7, to=5-11]
	\arrow[Rightarrow, no head, from=1-5, to=7-11]
	\arrow[Rightarrow, no head, from=1-3, to=9-11]
	\arrow[Rightarrow, no head, from=1-1, to=12-11]
	\arrow["{F_m}"{description}, from=1-1, to=9-1]
	\arrow["{F(f)}"{description}, from=9-1, to=15-1]
	\arrow["{F^n}"{description}, from=15-1, to=23-1]
	\arrow["{\parr(\alpha_1,\dots,\alpha_n)}"{description}, from=23-1, to=23-17]
	\arrow[Rightarrow, no head, from=12-16, to=23-1]
\end{tikzcd}
}
}
\end{figure}

Similarly, one can prove that the polymap $\cP(G)(f) \circ (\cP(\alpha)_{A_1},\dots,\cP(\alpha_{A_m})$ in $\cP(\cD)$ is the map $G^n \circ G(f) \circ G_m \circ \otimes(\alpha_{A_1},\dots,\alpha_{A_m})$ in \cD.

Finally the following diagram commutes:
% https://q.uiver.app/?q=WzAsOCxbMCwwLCJcXG90aW1lcyhGKEFfMSksXFxkb3RzLEYoQV9tKSkiXSxbMiwwLCJGKFxcb3RpbWVzKEFfMSxcXGRvdHMsQV9tKSkiXSxbNCwwLCJGKFxccGFycihCXzEsXFxkb3RzLEJfbikpIl0sWzYsMCwiXFxwYXJyKEYoQl8xKSxcXGRvdHMsRihCX24pKSJdLFs2LDIsIlxccGFycihHKEJfMSksXFxkb3RzLEcoQl9uKSkiXSxbMCwyLCJcXG90aW1lcyhHKEFfMSksXFxkb3RzLEcoQV9tKSkiXSxbMiwyLCJHKFxcb3RpbWVzKEFfMSxcXGRvdHMsQV9tKSkiXSxbNCwyLCJHKFxccGFycihCXzEsXFxkb3RzLEJfbikpIl0sWzAsMSwiRl9tIiwxXSxbMSwyLCJGKGYpIiwxXSxbMiwzLCJGXm4iLDFdLFszLDQsIlxccGFycihcXGFscGhhX3tCXzF9LFxcZG90cyxcXGFscGhhX3tCX259KSIsMV0sWzAsNSwiXFxvdGltZXMoXFxhbHBoYV97QV8xfSxcXGRvdHMsXFxhbHBoYV97QV9tfSkiLDFdLFs1LDYsIkdfbSIsMV0sWzEsNiwiXFxhbHBoYV97XFxvdGltZXMoQV8xLFxcZG90cyxBX20pfSIsMV0sWzYsNywiRyhmKSIsMV0sWzIsNywiXFxhbHBoYV97XFxwYXJyKEFfMSxcXGRvdHMsQV9uKX0iLDFdLFs3LDQsIkdebiIsMV1d
\[\resizebox{\hsize}{!}{\begin{tikzcd}[ampersand replacement=\&]
	{\otimes(F(A_1),\dots,F(A_m))} \&\& {F(\otimes(A_1,\dots,A_m))} \&\& {F(\parr(B_1,\dots,B_n))} \&\& {\parr(F(B_1),\dots,F(B_n))} \\
	\\
	{\otimes(G(A_1),\dots,G(A_m))} \&\& {G(\otimes(A_1,\dots,A_m))} \&\& {G(\parr(B_1,\dots,B_n))} \&\& {\parr(G(B_1),\dots,G(B_n))}
	\arrow["{F_m}"{description}, from=1-1, to=1-3]
	\arrow["{F(f)}"{description}, from=1-3, to=1-5]
	\arrow["{F^n}"{description}, from=1-5, to=1-7]
	\arrow["{\parr(\alpha_{B_1},\dots,\alpha_{B_n})}"{description}, from=1-7, to=3-7]
	\arrow["{\otimes(\alpha_{A_1},\dots,\alpha_{A_m})}"{description}, from=1-1, to=3-1]
	\arrow["{G_m}"{description}, from=3-1, to=3-3]
	\arrow["{\alpha_{\otimes(A_1,\dots,A_m)}}"{description}, from=1-3, to=3-3]
	\arrow["{G(f)}"{description}, from=3-3, to=3-5]
	\arrow["{\alpha_{\parr(A_1,\dots,A_n)}}"{description}, from=1-5, to=3-5]
	\arrow["{G^n}"{description}, from=3-5, to=3-7]
\end{tikzcd}}\]
where the left and right squares commute because $\alpha$ is a monoidal transformation with respect to both monoidal structures and the middle one by naturality.
\end{proof}

\begin{prop}
$\cP \colon \LULDCFrob \to \PolyCat$ is a 2-functor.
\end{prop}
\begin{proof}
We have already proven that it is well-defined.
We have that $\cP(\id_{\cC})(A) = A$ and $\cP(\id_{\cC})(f) = f$ by definition.
Furthermore, $\cP(G \circ F)(A) = (G \circ F)(A) = (\cP(G) \circ \cP(F))(A)$.
The equality is also true on morphisms.
% https://q.uiver.app/?q=WzAsNixbMCwxLCJcXG90aW1lc19tIEcoRihBX2kpKSJdLFsxLDAsIkcoXFxvdGltZXNfbUYoQV9pKSkiXSxbMiwxLCJHKEYoXFxvdGltZXNfbSBBX2kpKSJdLFs0LDEsIkcoRihcXHBhcnJfbiBCX2opKSJdLFs1LDAsIkcoXFxwYXJyX25GKEJfaikpIl0sWzYsMSwiXFxwYXJyX25HKEYoQl9qKSkiXSxbMCwxLCJHX20iLDFdLFsxLDIsIkcoRl9tKSIsMV0sWzIsMywiRyhGKGYpKSIsMV0sWzMsNCwiRyhGXm4pIiwxXSxbNCw1LCJHXm4iLDFdLFswLDIsIihHXFxjaXJjIEYpX20iLDFdLFszLDUsIihHXFxjaXJjIEYpXm4iLDFdXQ==
\[\resizebox{\hsize}{!}{\begin{tikzcd}[ampersand replacement=\&]
	\& {G(\otimes_mF(A_i))} \&\&\&\& {G(\parr_nF(B_j))} \\
	{\otimes_m G(F(A_i))} \&\& {G(F(\otimes_m A_i))} \&\& {G(F(\parr_n B_j))} \&\& {\parr_nG(F(B_j))}
	\arrow["{G_m}"{description}, from=2-1, to=1-2]
	\arrow["{G(F_m)}"{description}, from=1-2, to=2-3]
	\arrow["{G(F(f))}"{description}, from=2-3, to=2-5]
	\arrow["{G(F^n)}"{description}, from=2-5, to=1-6]
	\arrow["{G^n}"{description}, from=1-6, to=2-7]
	\arrow["{(G\circ F)_m}"{description}, from=2-1, to=2-3]
	\arrow["{(G\circ F)^n}"{description}, from=2-5, to=2-7]
\end{tikzcd}}\]

Finally, for the identity linear transformation $\cP(\id_F)_A = \iota_{F(A)} \circ \eta_{F(A)}$ which is the identity in $\cP(\cD)$ and composition of linear transformations is given by:
% https://q.uiver.app/?q=WzAsMTMsWzEsMCwiXFxvdGltZXNfMVxcb3RpbWVzXzFGKEEpIl0sWzIsMSwiXFxvdGltZXNfMUYoQSkiXSxbNCwxLCJcXG90aW1lc18xRyhBKSJdLFswLDEsIlxcb3RpbWVzXzFGKEEpIl0sWzYsMSwiXFxvdGltZXNfMVxccGFycl8xRyhBKSJdLFs2LDMsIlxccGFycl8xXFxvdGltZXNfMUcoQSkiXSxbNiw1LCJcXHBhcnJfMUcoQSkiXSxbNiw3LCJcXHBhcnJfMUgoQSkiXSxbNyw4LCJcXHBhcnJfMVxccGFycl8xSChBKSJdLFs2LDksIlxccGFycl8xIEgoQSkiXSxbMCwzLCJGKEEpIl0sWzAsNSwiRyhBKSJdLFswLDksIkgoQSkiXSxbMCwxLCJcXG90aW1lc18xXFxldGFfe0YoQSl9IiwxXSxbMSwyLCJcXG90aW1lc18xXFxhbHBoYSIsMV0sWzMsMCwiXFxhbHBoYSIsMV0sWzIsNCwiXFxvdGltZXNfMVxcaW90YV97RyhBKX0iLDFdLFs0LDUsIlxcZGVsdGEiLDFdLFs1LDYsIlxccGFycl8xXFxldGFfe0coQSl9IiwxXSxbNiw3LCJcXHBhcnJfMVxcYmV0YV9BIiwxXSxbNyw4LCJcXHBhcnJfMVxcaW90YV97SChBKX0iLDFdLFs4LDksIlxcZ2FtbWEiLDFdLFs3LDksIiIsMSx7ImxldmVsIjoyLCJzdHlsZSI6eyJoZWFkIjp7Im5hbWUiOiJub25lIn19fV0sWzMsMSwiIiwxLHsibGV2ZWwiOjIsInN0eWxlIjp7ImhlYWQiOnsibmFtZSI6Im5vbmUifX19XSxbMywxMCwiXFxldGFfe0YoQSl9IiwxXSxbMTAsMTEsIlxcYWxwaGFfQSIsMV0sWzExLDEyLCJcXGJldGFfQSIsMV0sWzEyLDksIlxcaW90YV97SChBKX0iLDFdLFsyLDUsIlxcaW90YV97XFxvdGltZXNfMUcoQSl9IiwxXSxbMiwxMSwiXFxldGFfe0coQSl9IiwxXSxbMTEsNiwiXFxpb3RhX3tHKEEpfSIsMV1d
\[\resizebox{\hsize}{!}{
\begin{tikzcd}[ampersand replacement=\&]
	\& {\otimes_1\otimes_1F(A)} \\
	{\otimes_1F(A)} \&\& {\otimes_1F(A)} \&\& {\otimes_1G(A)} \&\& {\otimes_1\parr_1G(A)} \\
	\\
	{F(A)} \&\&\&\&\&\& {\parr_1\otimes_1G(A)} \\
	\\
	{G(A)} \&\&\&\&\&\& {\parr_1G(A)} \\
	\\
	\&\&\&\&\&\& {\parr_1H(A)} \\
	\&\&\&\&\&\&\& {\parr_1\parr_1H(A)} \\
	{H(A)} \&\&\&\&\&\& {\parr_1 H(A)}
	\arrow["{\otimes_1\eta_{F(A)}}"{description}, from=1-2, to=2-3]
	\arrow["{\otimes_1\alpha}"{description}, from=2-3, to=2-5]
	\arrow["\alpha"{description}, from=2-1, to=1-2]
	\arrow["{\otimes_1\iota_{G(A)}}"{description}, from=2-5, to=2-7]
	\arrow["\delta"{description}, from=2-7, to=4-7]
	\arrow["{\parr_1\eta_{G(A)}}"{description}, from=4-7, to=6-7]
	\arrow["{\parr_1\beta_A}"{description}, from=6-7, to=8-7]
	\arrow["{\parr_1\iota_{H(A)}}"{description}, from=8-7, to=9-8]
	\arrow["\gamma"{description}, from=9-8, to=10-7]
	\arrow[Rightarrow, no head, from=8-7, to=10-7]
	\arrow[Rightarrow, no head, from=2-1, to=2-3]
	\arrow["{\eta_{F(A)}}"{description}, from=2-1, to=4-1]
	\arrow["{\alpha_A}"{description}, from=4-1, to=6-1]
	\arrow["{\beta_A}"{description}, from=6-1, to=10-1]
	\arrow["{\iota_{H(A)}}"{description}, from=10-1, to=10-7]
	\arrow["{\iota_{\otimes_1G(A)}}"{description}, from=2-5, to=4-7]
	\arrow["{\eta_{G(A)}}"{description}, from=2-5, to=6-1]
	\arrow["{\iota_{G(A)}}"{description}, from=6-1, to=6-7]
\end{tikzcd}}\]
\end{proof}

\subsection{Weak two-tensor polycategories with duals}

In this section we will introduce the notion of weak two-tensor polycategories with duals.
These are categories equipped with objects, the tensors, pars, and duals, having some weak universal property.
We will prove that it characterises the polycategories that are the underlying polycategories of a lax normal linearly distributive category.
As the name suggests, these correspond to a weakening of the notion of two-tensor polycategories with duals.

\begin{definition}
Let $\Gamma,\Delta$ be lists of objects in a polycategory \cP .

A (weak) tensor product of $\Gamma$ is an object $\otimes \Gamma$ equipped with a polymap $m_{\Gamma} \colon \Gamma \to \otimes \Gamma$ such that for any polymap $f \colon \Gamma \to \Delta$ there is a unique polymap $f/m_{\Gamma} \colon \otimes \Gamma \to A$ such that $f = (f/m_{\Gamma}) \circ m_{\Gamma}$.

A (weak) par product of $\Delta$ is an object $\parr \Delta$ equipped with a polymap $w_{\Delta} \parr \Delta \to \Delta$ such that for any polymap $f \colon \Gamma \to \Delta$ there is a unique $w_{\Delta}\backslash f$ such that $f = w_\Delta \circ w_\Delta \backslash f$.
\end{definition}
 
When working with polymaps in weak two-tensor polycategories we will often write $gf$ for $g \circ f$, so we have that $f = (f/m_\Gamma) m_\Gamma$ for example. 
 
\begin{prop}
Tensor and par products are unique up to unique isomorphism.
\end{prop} 
\begin{proof}
Suppose that we have $\otimes \Gamma$ and $(\otimes \Gamma)'$ with weakly universal polymaps $m_\Gamma \colon \Gamma \to \otimes \Gamma$ and $m_\Gamma' \colon \Gamma \to (\otimes \Gamma)'$.
We can use the weakly universality of each polymap to factor the other through it giving polymaps $m_\Gamma/(m_\Gamma') \colon (\otimes \Gamma)' \to \otimes \Gamma$ and $m_\Gamma'/m_\Gamma \colon \otimes \Gamma \to (\otimes \Gamma)'$.
These polymaps are inverse to each other.
Indeed we have that
\[m_\Gamma = (m_\Gamma/(m_\Gamma)') m_\Gamma' = (m_\Gamma/(m_\Gamma)') (m_\Gamma'/m_\Gamma) m_\Gamma\]
But then the polymap $m_\Gamma$ factors through $m_\Gamma$ by $m_\Gamma/(m_\Gamma)' m_\Gamma'/m_\Gamma$ and $\id_{\otimes \Gamma}$.
By uniqueness of the factorisation through $m_\Gamma$ we have $m_\Gamma/(m_\Gamma)' m_\Gamma'/m_\Gamma = \id_{\otimes \Gamma}$.

\noindent Similarly $m_\Gamma'/m_\Gamma m_\Gamma/(m_\Gamma') = \id_{(\otimes \Gamma)'}$.

The proof for $\parr \Delta$ is similar.
\end{proof}
 
Given a list of lists of objects in a polycategory \cP , $(A_{i,j})_{1 \leq i \leq p, 1 \leq j \leq n_i}$, we have a polymap $\otimes A_{i,j} \to \otimes \otimes A_{i,j}$ from the tensor of the concatenation of the lists to the tensor of the tensors of the individual list.
To build it first consider the following weakly universal polymaps:
\begin{itemize}
\item for $1 \leq i \leq p$, $m_{(A_{i,j})_j} \colon A_{i,1},\dots,A_{i,n_i} \to \otimes A_{i,j}$
\item $m_{(\otimes A_{i,j})_i} \colon \otimes A_{1,j},\dots, \otimes A_{p,j} \to \otimes \otimes A_{i,j}$
\end{itemize}
We can compose them to get
\[ m_{(\otimes A_{i,j})_i}(m_{(A_{1,j})_j},\dots,m_{(A_{p,j})_j}) \colon A_{1,1},\dots, A_{1,n_1},\dots, A_{p,1},\dots, A_{p,n_p} \to \otimes \otimes A_{i,j}\]
Finally we can factor through
\[m_{(A_{i,j})_{(i,j)}} \colon A_{1,1},\dots, A_{1,n_1},\dots, A_{p,1},\dots, A_{p,n_p} \to \otimes A_{i,j}\] to get
\[m_{(\otimes A_{i,j})_i}(m_{(A_{1,j})_j},\dots,m_{(A_{p,j})_j})/m_{(A_{i,j})_{(i,j)}} \colon \otimes A_{i,j} \to \otimes \otimes A_{i,j}\]

Similarly we get a polymap $\parr \parr A_{i,j} \to \parr A_{i,j}$.

These polymaps are not in general invertible.
In particular, weakly universal polymaps don't compose in general.
If it were the case, since tensors/pars are unique up to unique isomorphisms we would have $\otimes \otimes A_{i,j} \simeq \otimes A_{i,j}$ which is not the case.

\begin{definition}
A weak two-tensor polycategory is a polycategory that has all weak tensors and pars. 
\end{definition}

\begin{figure}
\input{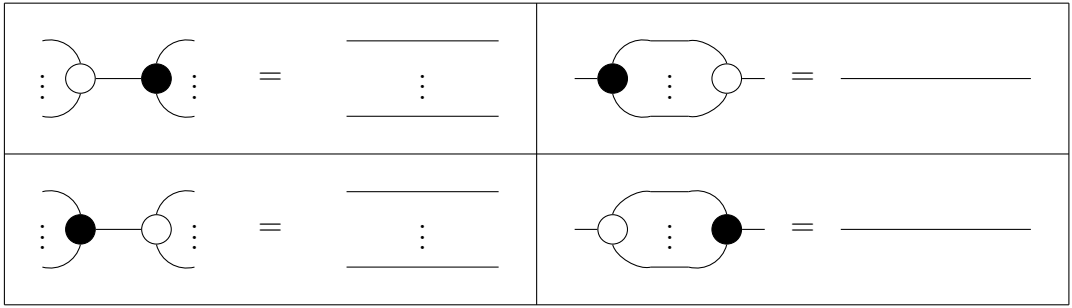}
\caption{Graphical calculus for a weak two tensor polycategory}
\label{fig:polycat-weak-two-tensor}
\end{figure}

There is a graphical calculus for weak two-tensor polycategories that extend the one for representable polycategories.
We have a white multiplicative node that represent the multiplication polymap $m_\Gamma$ and correspond to the introduction of the tensor.
It is oplax monoidal.
There is a black comultiplicative node.
Precomposing a polymap $f$ by it represent the polymap $f/m_\Gamma$, so it is only possible to precompose all the inputs of $f$ at once by a black node.
It corresponds to elimination of the tensor and is lax comonoidal.
Then there are black multiplicative and white comultiplicative nodes corresponding to the par.
These are lax and oplax respectively, and the black one can only by use to postcompose all the output of a polymap.
Furthermore, black and white nodes cancel each other.
The rules are given in figure \ref{fig:polycat-weak-two-tensor}.

Given a weak two-tensor polycategory \cP , one can define a lnuldc that consists of unary co-unary maps in \cP.
We define $\cR(\cP)$ as the category with:
\begin{itemize}
\item objects those of \cP
\item maps,unary co-unary polymaps $f \colon A \to B$ in \cP
\item $\otimes_n (A_i)_i := \otimes (A_i)_i$
\item $\otimes_n(f_i)_i := m_{(B_i)}(f_1\dots f_n)/m_{(A_i)}$ which is represented graphically
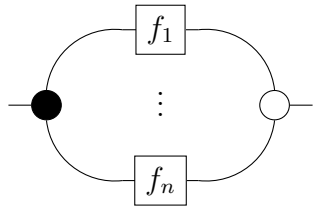
\[\input{polycat-tensor-map.tikz}\]
\item $\parr_n (A_i)_i := \parr (A_i)$
\item $\parr_n (f_i)_i := w_{(B_i)}\backslash(f_1 \dots f_n) w_{(A_i)}$ which is represented graphically by \[\input{polycat-par-map.tikz}\]
\item $\alpha_{(A_{i,j})} := m_{(\otimes (A_{1,j}),\dots,\otimes (A_{n,j}))}(m_{(A_{1,j})},\dots,m_{(A_{n,j})})/m_{(A_{i,j})}$, graphically:
\[\input{polycat-tensor-alpha.tikz}\]
\item $\eta_A := \id_A/m_A \colon \otimes A \to A$
\item $\gamma_{(A_{i,j})}$ is given by:
\[\input{polycat-par-gamma.tikz}\]
\item $\iota_A := w_A \backslash \id_A \colon A \to \parr A$
\item the distributivity law $\delta \colon \otimes(\Gamma_1, \parr(\Delta_1,A,\Delta_2),\Gamma_2) \to \parr(\Delta_1,\otimes(\Gamma_1,A,\Gamma_2),\Delta_2)$ is given graphically by the following figure:
\input{polycat-delta.tikz}
Notice that it involves crossing of wires which is not possible in general.
However the restriction on the contexts in which $\delta$ is defined exactly prevents these crossing to happen.
For example when we take $\Gamma_1$ and $\Delta_2$ to be empty we get the following map
\input{polycat-delta-partial.tikz}
which represents $w_{\Delta_1,\otimes(A,\Gamma_1)}\backslash m_{A,\Gamma_2}w_{\Delta_1,A}/m_{\parr(\Delta_1,A),\Gamma_2}$
\end{itemize}

\begin{prop}
For any weak two-tensor polycategory \cP , $\cR(\cP)$ is a lax normal unbiased linearly distributive category.
\end{prop}
\begin{proof}
First, it is a category: unitality of the identities and associativity of composition follow directly from those in the polycategory.

Now let us prove that it is oplax $\otimes$-monoidal.
First we want to prove that the following diagram commutes:

% https://q.uiver.app/?q=WzAsNCxbMSwyLCJcXG90aW1lc19wXFxvdGltZXNfe25faX1cXG90aW1lc197bV97aSxqfX1BX3tpLGosa30iXSxbMCwxLCJcXG90aW1lc197XFxzdW1faSBuX2l9XFxvdGltZXNfe21fe2ksan19QV97aSxqLGt9Il0sWzEsMCwiXFxvdGltZXNfe1xcc3VtX3tpLGp9IG1fe2ksan19QV97aSxqLGt9Il0sWzIsMSwiXFxvdGltZXNfcCBcXG90aW1lc197XFxzdW1faiBtX3tpLGp9fUFfe2ksaixrfSJdLFsxLDAsIlxcYWxwaGEiLDFdLFsyLDEsIlxcYWxwaGEiLDFdLFszLDAsIlxcb3RpbWVzX3AgXFxhbHBoYSIsMV0sWzIsMywiXFxhbHBoYSIsMV1d
\[\begin{tikzcd}[ampersand replacement=\&]
	\& {\otimes_{\sum_{i,j} m_{i,j}}A_{i,j,k}} \\
	{\otimes_{\sum_i n_i}\otimes_{m_{i,j}}A_{i,j,k}} \&\& {\otimes_p \otimes_{\sum_j m_{i,j}}A_{i,j,k}} \\
	\& {\otimes_p\otimes_{n_i}\otimes_{m_{i,j}}A_{i,j,k}}
	\arrow["\alpha"{description}, from=2-1, to=3-2]
	\arrow["\alpha"{description}, from=1-2, to=2-1]
	\arrow["{\otimes_p \alpha}"{description}, from=2-3, to=3-2]
	\arrow["\alpha"{description}, from=1-2, to=2-3]
\end{tikzcd}\]

We will prove that both paths give polymaps equal to \[m_{(\otimes_{n_i}\otimes_{m_{i,j}} A_{i,j,k})_i}(m_{(\otimes_{m_{1,j} A_{1,j,k}}(m_{(A_{1,1,k})_k},\dots,m_{(A_{1,n_1,k})_k}))_j},\dots,m_{(\otimes_{m_{p,j} A_{p,j,k}}(m_{(A_{p,1,k})_k},\dots,m_{(A_{p,n_p,k})_k}))_j})/m_{(A_{i,j,k})_{i,j,k}}\]

which can be understood as eliminating the tensor product in $\otimes_{\sum_{i,j} m_{i,j}} A_{i,j,k}$ then introducing first the tensor products $\otimes_{m_{i,j}} A_{i,j,k}$ followed by the tensor products $\otimes_{n_i}\otimes_{m_{i,j}} A_{i,j,k}$ and finally the tensor product $\otimes_p \otimes_{n_i} \otimes_{m_{i,j}} A_{i,j,k}$

\begin{figure}
\centering
\resizebox{\hsize}{!}{\input{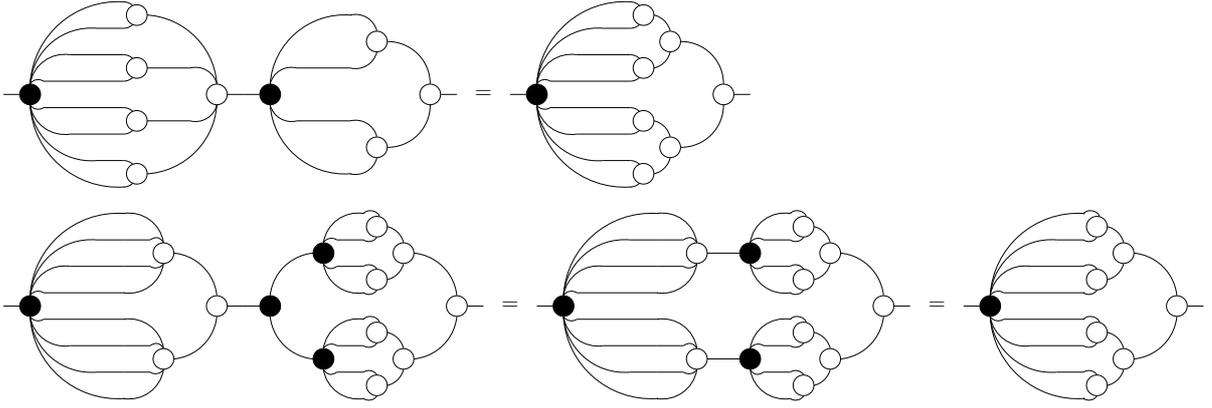}}
\caption{Coherence for $\alpha$}
\label{fig:polycat-tensor-alpha-coherence}
\end{figure}

The proof is given graphically in figure \ref{fig:polycat-tensor-alpha-coherence}.
Then we need to prove that the following diagram commutes:

% https://q.uiver.app/?q=WzAsNCxbMCwwLCJcXG90aW1lc19uIEFfaSJdLFsyLDAsIlxcb3RpbWVzX25cXG90aW1lc18xQV9pIl0sWzIsMiwiXFxvdGltZXNfbiBBX2kiXSxbMCwyLCJcXG90aW1lc18xXFxvdGltZXNfbkFfaSJdLFswLDEsIlxcYWxwaGEiLDFdLFsxLDIsIlxcb3RpbWVzX25cXGV0YV97QV9pfSIsMV0sWzAsMiwiIiwxLHsibGV2ZWwiOjIsInN0eWxlIjp7ImhlYWQiOnsibmFtZSI6Im5vbmUifX19XSxbMCwzLCJcXGFscGhhIiwxXSxbMywyLCJcXGV0YV97XFxvdGltZXNfbkFfaX0iLDFdXQ==
\[\begin{tikzcd}[ampersand replacement=\&]
	{\otimes_n A_i} \&\& {\otimes_n\otimes_1A_i} \\
	\\
	{\otimes_1\otimes_nA_i} \&\& {\otimes_n A_i}
	\arrow["\alpha"{description}, from=1-1, to=1-3]
	\arrow["{\otimes_n\eta_{A_i}}"{description}, from=1-3, to=3-3]
	\arrow[Rightarrow, no head, from=1-1, to=3-3]
	\arrow["\alpha"{description}, from=1-1, to=3-1]
	\arrow["{\eta_{\otimes_nA_i}}"{description}, from=3-1, to=3-3]
\end{tikzcd}\]

\begin{figure}
\centering
\resizebox{\hsize}{!}{\input{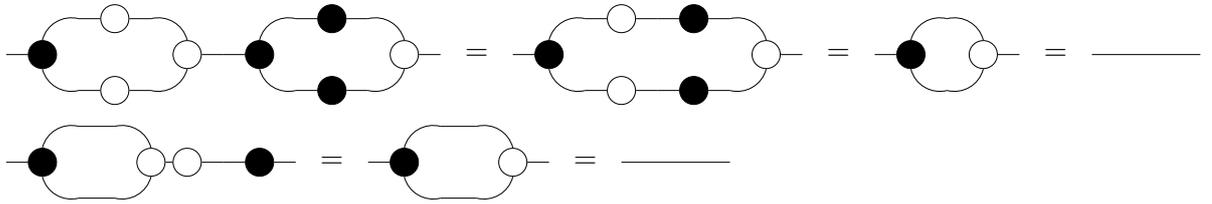}}
\caption{Coherence for $\eta$}
\label{fig:polycat-tensor-eta-coherence}
\end{figure}

The proof is given in figure \ref{fig:polycat-tensor-eta-coherence}.

The proofs of the coherence laws for $\gamma$ and $\iota$ are obtained by doing a horizontal reflection of each figure.

Finally we want to prove coherence for the distributivity law.
First, we shall prove that the following diagram commutes:
% https://q.uiver.app/?q=WzAsNixbMSwwLCJcXG90aW1lcyhcXEdhbW1hXzEsXFxwYXJyKFxcRGVsdGFfMSxBXzEpLFxcR2FtbWFfMixcXHBhcnIoQV8yLFxcRGVsdGFfMiksXFxHYW1tYV8zKSJdLFswLDEsIlxccGFycihcXERlbHRhXzEsXFxvdGltZXMoXFxHYW1tYV8xLEFfMSxcXEdhbW1hXzIsXFxwYXJyKEFfMixcXERlbHRhXzIpLFxcR2FtbWFfMykpIl0sWzAsMywiXFxwYXJyKFxcRGVsdGFfMSxcXHBhcnIoXFxvdGltZXMoXFxHYW1tYV8xLEFfMSxcXEdhbW1hXzIsQV8yLFxcR2FtbWFfMyksXFxEZWx0YV8yKSkiXSxbMSw0LCJcXHBhcnIoXFxEZWx0YV8xLFxcb3RpbWVzKFxcR2FtbWFfMSxBXzEsXFxHYW1tYV8yLEFfMixcXEdhbW1hXzMpLFxcRGVsdGFfMikiXSxbMiwxLCJcXHBhcnIoXFxvdGltZXMoXFxHYW1tYV8xLFxccGFycihcXERlbHRhXzEsQV8xKSxcXEdhbW1hXzIsQV8yLFxcR2FtbWFfMyksXFxEZWx0YV8yKSJdLFsyLDMsIlxccGFycihcXHBhcnIoXFxEZWx0YV8xLFxcb3RpbWVzKFxcR2FtbWFfMSxBXzEsXFxHYW1tYV8yLEFfMixcXEdhbW1hXzMpKSxcXERlbHRhXzIpIl0sWzAsMSwiXFxkZWx0YSIsMV0sWzEsMiwiXFxwYXJyKC0sXFxkZWx0YSkiLDFdLFsyLDMsIlxcZ2FtbWEnIiwxXSxbMCw0LCJcXGRlbHRhIiwxXSxbNCw1LCJcXHBhcnIoXFxkZWx0YSwtKSIsMV0sWzUsMywiXFxnYW1tYSciLDFdXQ==
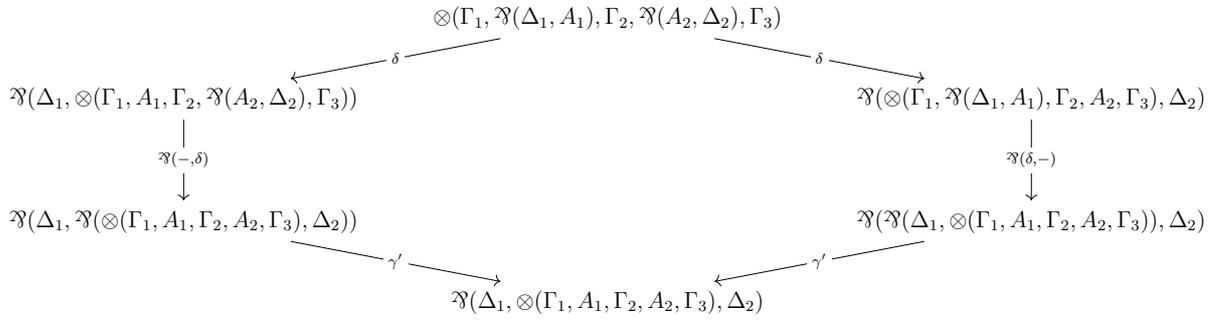
\[\resizebox{\hsize}{!}{\begin{tikzcd}[ampersand replacement=\&]
	\& {\otimes(\Gamma_1,\parr(\Delta_1,A_1),\Gamma_2,\parr(A_2,\Delta_2),\Gamma_3)} \\
	{\parr(\Delta_1,\otimes(\Gamma_1,A_1,\Gamma_2,\parr(A_2,\Delta_2),\Gamma_3))} \&\& {\parr(\otimes(\Gamma_1,\parr(\Delta_1,A_1),\Gamma_2,A_2,\Gamma_3),\Delta_2)} \\
	\\
	{\parr(\Delta_1,\parr(\otimes(\Gamma_1,A_1,\Gamma_2,A_2,\Gamma_3),\Delta_2))} \&\& {\parr(\parr(\Delta_1,\otimes(\Gamma_1,A_1,\Gamma_2,A_2,\Gamma_3)),\Delta_2)} \\
	\& {\parr(\Delta_1,\otimes(\Gamma_1,A_1,\Gamma_2,A_2,\Gamma_3),\Delta_2)}
	\arrow["\delta"{description}, from=1-2, to=2-1]
	\arrow["{\parr(-,\delta)}"{description}, from=2-1, to=4-1]
	\arrow["{\gamma'}"{description}, from=4-1, to=5-2]
	\arrow["\delta"{description}, from=1-2, to=2-3]
	\arrow["{\parr(\delta,-)}"{description}, from=2-3, to=4-3]
	\arrow["{\gamma'}"{description}, from=4-3, to=5-2]
\end{tikzcd}}\]

\begin{figure}
\centering
\resizebox{\hsize}{!}{\input{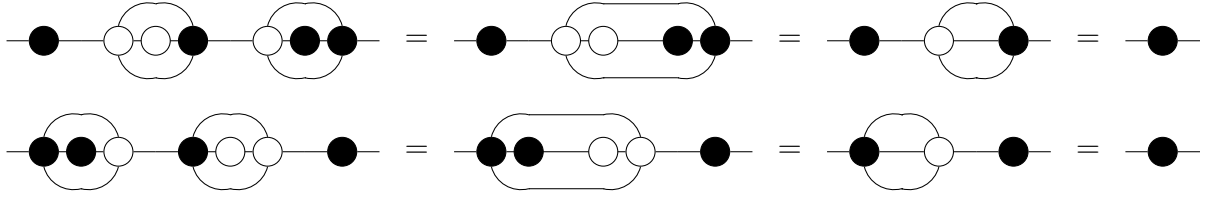}}
\caption{Coherence for $\delta$}
\label{fig:polycat-tensor-par-par-coherence}
\end{figure}

The proof is given in figure \ref{fig:polycat-tensor-par-par-coherence}.
Notice that the symmetry between the two paths means that the two corresponding digrams are a vertical reflection of one another.

Then we want to prove that:
% https://q.uiver.app/?q=WzAsNSxbMCwwLCJcXG90aW1lcyhcXEdhbW1hXzEsXFxkb3RzLFxcR2FtbWFfcCxcXEdhbW1hLFxccGFycihcXERlbHRhLEEsXFxEZWx0YScpLFxcR2FtbWEnLFxcR2FtbWFfMScsXFxkb3RzLFxcR2FtbWFfe3F9JykiXSxbMCwyLCJcXHBhcnIoXFxEZWx0YSwgXFxvdGltZXMoXFxHYW1tYV8xLFxcZG90cyxcXEdhbW1hX3AsQSxcXEdhbW1hJyxcXEdhbW1hXzEnLFxcZG90cyxcXEdhbW1hX3EnKSxcXERlbHRhJykiXSxbMSw0LCJcXHBhcnIoXFxEZWx0YSxcXG90aW1lcyhcXG90aW1lc1xcR2FtbWFfMSxcXGRvdHMsXFxvdGltZXNcXEdhbW1hX3AsXFxvdGltZXMoXFxHYW1tYSxBLFxcR2FtbWEnKSxcXG90aW1lc1xcR2FtbWFfMScsXFxkb3RzLFxcb3RpbWVzXFxHYW1tYV9xJyksXFxEZWx0YScpIl0sWzIsMCwiXFxvdGltZXMoXFxvdGltZXNcXEdhbW1hXzEsXFxkb3RzLFxcb3RpbWVzXFxHYW1tYV9wLFxcb3RpbWVzKFxcR2FtbWEsXFxwYXJyKFxcRGVsdGEsQSxcXERlbHRhJyksXFxHYW1tYScpLFxcb3RpbWVzIFxcR2FtbWFfMScsXFxkb3RzLFxcb3RpbWVzIFxcR2FtbWFfcScpIl0sWzIsMiwiXFxvdGltZXMoXFxvdGltZXNcXEdhbW1hXzEsXFxkb3RzLFxcb3RpbWVzXFxHYW1tYV9wLFxccGFycihcXERlbHRhLFxcb3RpbWVzKFxcR2FtbWEsQSxcXEdhbW1hJyksXFxEZWx0YScpLFxcb3RpbWVzXFxHYW1tYV8xJyxcXGRvdHMsXFxvdGltZXNcXEdhbW1hX3EnKSJdLFswLDEsIlxcZGVsdGEiLDFdLFsxLDIsIlxccGFycigtLFxcYWxwaGEsLSkiLDFdLFswLDMsIlxcYWxwaGEiLDFdLFszLDQsIlxcb3RpbWVzKC0sXFxkZWx0YSwtKSIsMV0sWzQsMiwiXFxkZWx0YSIsMV1d
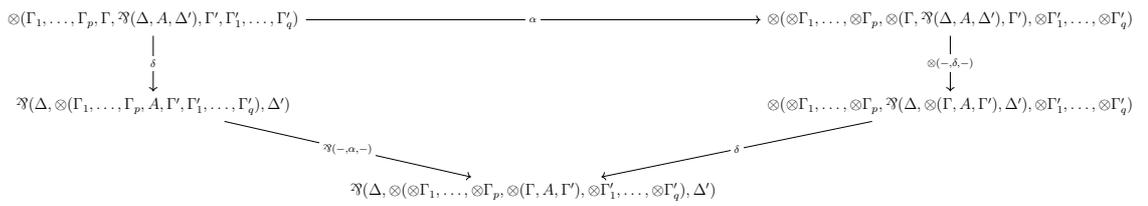
\[\scalebox{0.5}{\begin{tikzcd}[ampersand replacement=\&]
	{\otimes(\Gamma_1,\dots,\Gamma_p,\Gamma,\parr(\Delta,A,\Delta'),\Gamma',\Gamma_1',\dots,\Gamma_{q}')} \&\& {\otimes(\otimes\Gamma_1,\dots,\otimes\Gamma_p,\otimes(\Gamma,\parr(\Delta,A,\Delta'),\Gamma'),\otimes \Gamma_1',\dots,\otimes \Gamma_q')} \\
	\\
	{\parr(\Delta, \otimes(\Gamma_1,\dots,\Gamma_p,A,\Gamma',\Gamma_1',\dots,\Gamma_q'),\Delta')} \&\& {\otimes(\otimes\Gamma_1,\dots,\otimes\Gamma_p,\parr(\Delta,\otimes(\Gamma,A,\Gamma'),\Delta'),\otimes\Gamma_1',\dots,\otimes\Gamma_q')} \\
	\\
	\& {\parr(\Delta,\otimes(\otimes\Gamma_1,\dots,\otimes\Gamma_p,\otimes(\Gamma,A,\Gamma'),\otimes\Gamma_1',\dots,\otimes\Gamma_q'),\Delta')}
	\arrow["\delta"{description}, from=1-1, to=3-1]
	\arrow["{\parr(-,\alpha,-)}"{description}, from=3-1, to=5-2]
	\arrow["\alpha"{description}, from=1-1, to=1-3]
	\arrow["{\otimes(-,\delta,-)}"{description}, from=1-3, to=3-3]
	\arrow["\delta"{description}, from=3-3, to=5-2]
\end{tikzcd}}\]

\begin{figure}
\centering
\resizebox{\hsize}{!}{\input{polycat-tensor-par-coherence.tikz}}
\caption{Coherence for $\delta$}
\label{fig:polycat-tensor-par-coherence}
\end{figure}

It is proven in figure \ref{fig:polycat-tensor-par-coherence}.

Last, we want to prove that:
% https://q.uiver.app/?q=WzAsNixbMCwwLCJcXG90aW1lc1xccGFycihcXERlbHRhXzEsQSxcXERlbHRhXzIpIl0sWzIsMCwiXFxwYXJyKFxcRGVsdGFfMSxcXG90aW1lcyBBLFxcRGVsdGFfMikiXSxbMiwyLCJcXHBhcnIoXFxEZWx0YV8xLEEsXFxEZWx0YV8yKSJdLFs0LDIsIlxcb3RpbWVzKFxcR2FtbWFfMSxcXHBhcnIgQSwgXFxHYW1tYV8yKSJdLFs2LDAsIlxccGFyciBcXG90aW1lcyhcXEdhbW1hXzEsQSxcXEdhbW1hXzIpIl0sWzQsMCwiXFxvdGltZXMoXFxHYW1tYV8xLEEsXFxHYW1tYV8yKSJdLFswLDEsIlxcZGVsdGEiLDFdLFsxLDIsIlxccGFycigtLFxcZXRhLC0pIiwxXSxbMCwyLCJcXGV0YSIsMV0sWzMsNCwiXFxkZWx0YSIsMV0sWzUsNCwiXFxpb3RhIiwxXSxbNSwzLCJcXG90aW1lcygtLFxcaW90YSwtKSIsMV1d
\[\begin{tikzcd}[ampersand replacement=\&]
	{\otimes\parr(\Delta_1,A,\Delta_2)} \&\& {\parr(\Delta_1,\otimes A,\Delta_2)} \&\& {\otimes(\Gamma_1,A,\Gamma_2)} \&\& {\parr \otimes(\Gamma_1,A,\Gamma_2)} \\
	\\
	\&\& {\parr(\Delta_1,A,\Delta_2)} \&\& {\otimes(\Gamma_1,\parr A, \Gamma_2)}
	\arrow["\delta"{description}, from=1-1, to=1-3]
	\arrow["{\parr(-,\eta,-)}"{description}, from=1-3, to=3-3]
	\arrow["\eta"{description}, from=1-1, to=3-3]
	\arrow["\delta"{description}, from=3-5, to=1-7]
	\arrow["\iota"{description}, from=1-5, to=1-7]
	\arrow["{\otimes(-,\iota,-)}"{description}, from=1-5, to=3-5]
\end{tikzcd}\]

\begin{figure}
\centering
\resizebox{\hsize}{!}{\input{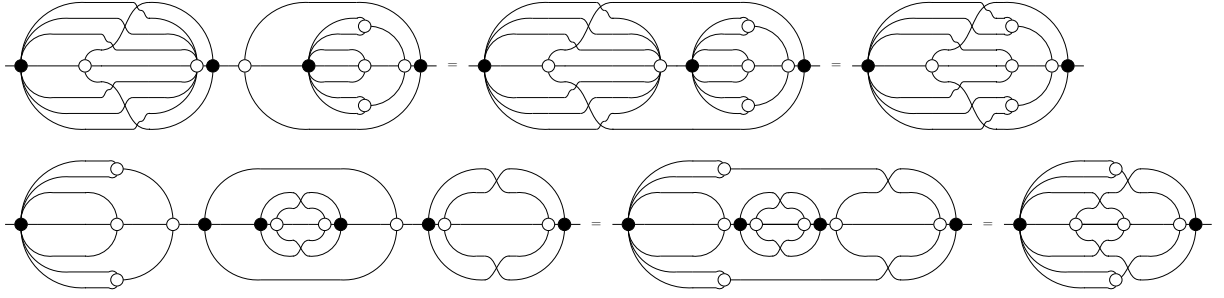}}
\caption{Coherence for $\delta$}
\label{fig:polycat-tensor-par-coherence-unit}
\end{figure}

The proof is given by figure \ref{fig:polycat-tensor-par-coherence-unit}.

This concludes the proof that $\cR(\cP)$ is a lax unbiased linearly distributive category.

For the fact that it is normal, we have that $\eta_A := \id_A/m_A$ has an inverse: $m_A$.
We have by definition that $\id_A/m_A \circ m_A = \id_A$.
Furthermore, \[(m_A \circ \id_A/m_A) \circ m_A = m_A \circ \id_A = m_A = \id_{\otimes A} \circ m_A\]
But by the unicity of prefactorisation by $m_A$, $(m_A \circ \id_A) = \id_{\otimes A}$.
Similarly, $\iota_A$ has for inverse $w_A$.

\end{proof}

Now we can extend \cR to functors to get a Frobenius functor.
Given a functor $F \colon \cP \to \cQ$ we define $\cR(F) \colon \cR(\cP) \to \cR(\cQ)$
\begin{itemize}
\item it is the restriction of $F$ to the category of unary counary polymaps in \cP
\item $\cR(F)_n \colon \otimes_n F(-) \Rightarrow F(\otimes_n -)$ is given on components $\Gamma$ by $F(m_\Gamma)/m_{F(\Gamma)}$.
Graphically it will be represented as follows where a functor is represented as a box:
\input{polycat-lax-monoidal.tikz}
\item $\cR(F)^n \colon F(\parr_n -) \Rightarrow \parr_n F(-)$ is given dually by \input{polycat-oplax-monoidal.tikz}
\end{itemize}

\begin{prop}
Given a functor of polycategories $F \colon \cP \to \cQ$, $\cR(F) \colon \cR(\cP) \to \cR(\cQ)$ is a Frobenius functor.
\end{prop}

\begin{proof}
First let prove that $F$ is lax $\otimes$-monoidal.
We first want to prove that the following diagram commutes:
% https://q.uiver.app/?q=WzAsNSxbMCwyLCJcXG90aW1lc19wXFxvdGltZXNfe25faX0gRihBX3tpLGp9KSJdLFsyLDIsIlxcb3RpbWVzX3AgRihcXG90aW1lc197bl9pfSBBX3tpLGp9KSJdLFszLDEsIkYoXFxvdGltZXNfcCBcXG90aW1lc197bl9pfSBBX3tpLGp9KSJdLFsyLDAsIkYoXFxvdGltZXNfe1xcc3VtX2kgbl9pfSBBX3tpLGp9KSJdLFswLDAsIlxcb3RpbWVzX3tcXHN1bV9pIG5faX1GKEFfe2ksan0pIl0sWzAsMSwiXFxvdGltZXNfcCBGX3tuX2l9IiwxXSxbMSwyLCJGX3AiLDFdLFszLDIsIkYoXFxhbHBoYSkiLDFdLFs0LDAsIlxcYWxwaGEiLDFdLFs0LDMsIkZfe1xcc3VtX2kgbl9pfSIsMV1d
\[\begin{tikzcd}[ampersand replacement=\&]
	{\otimes_{\sum_i n_i}F(A_{i,j})} \&\& {F(\otimes_{\sum_i n_i} A_{i,j})} \\
	\&\&\& {F(\otimes_p \otimes_{n_i} A_{i,j})} \\
	{\otimes_p\otimes_{n_i} F(A_{i,j})} \&\& {\otimes_p F(\otimes_{n_i} A_{i,j})}
	\arrow["{\otimes_p F_{n_i}}"{description}, from=3-1, to=3-3]
	\arrow["{F_p}"{description}, from=3-3, to=2-4]
	\arrow["{F(\alpha)}"{description}, from=1-3, to=2-4]
	\arrow["\alpha"{description}, from=1-1, to=3-1]
	\arrow["{F_{\sum_i n_i}}"{description}, from=1-1, to=1-3]
\end{tikzcd}\]

\begin{figure}
\centering
\resizebox{\hsize}{!}{\input{polycat-lax-monoidal-coherence.tikz}}
\caption{Coherence for $\otimes$-lax monoidality}
\label{fig:polycat-lax-monoidal-coherence}
\end{figure}

The proof is given in figure \ref{fig:polycat-lax-monoidal-coherence}.

Then we want to prove that the following diagram commutes:
% https://q.uiver.app/?q=WzAsMyxbMCwwLCJcXG90aW1lcyBGKEEpIl0sWzIsMCwiRihBKSJdLFswLDEsIkYoXFxvdGltZXMgQSkiXSxbMCwxLCJcXGV0YSIsMV0sWzAsMiwiRl8xIiwxXSxbMiwxLCJGKFxcZXRhKSIsMV1d
\[\begin{tikzcd}[ampersand replacement=\&]
	{\otimes F(A)} \&\& {F(A)} \\
	{F(\otimes A)}
	\arrow["\eta"{description}, from=1-1, to=1-3]
	\arrow["{F_1}"{description}, from=1-1, to=2-1]
	\arrow["{F(\eta)}"{description}, from=2-1, to=1-3]
\end{tikzcd}\]

\begin{figure}
\centering
\input{polycat-lax-monoidal-coherence-unit.tikz}
\caption{Coherence for $\otimes$-lax monoidality}
\label{fig:polycat-lax-monoidal-coherence-unit}
\end{figure}

It is given in figure \ref{fig:polycat-lax-monoidal-coherence-unit}.

This ends the proof of lax $\otimes$-monoidality of $\cR(F)$.
The proofs for oplax $\parr$-monoidality of $\cR(F)$ are dual.

\end{proof}

Finally given a natural transformation $\beta \colon F \Rightarrow G$ between functors of polycategories, we define a transformation $\cR(\beta) \colon \cR(F) \Rightarrow \cR(G)$ given by $\cR(\beta) := \beta$.

\begin{prop}
Any natural transformation $\beta \colon F \Rightarrow G$ between functors of polycategories induces a linear transformation between Frobenius functors $\alpha \colon \cR(F) \Rightarrow \cR(G)$.
\end{prop}

\begin{proof}
The naturality of $\alpha$ followed directly from its naturality as a transformation between functors of polycategories.

To prove that it is $\otimes$-monoidal we want to show that the following diagram commutes:
% https://q.uiver.app/?q=WzAsNCxbMCwwLCJcXG90aW1lcyBGKEFfaSkiXSxbMiwwLCJGKFxcb3RpbWVzIEFfaSkiXSxbMiwyLCJHKFxcb3RpbWVzIEFfaSkiXSxbMCwyLCJcXG90aW1lcyBHKEFfaSkiXSxbMCwxLCJGX24iLDFdLFsxLDIsIlxcYmV0YV97XFxvdGltZXMgQV9pfSIsMV0sWzAsMywiXFxvdGltZXNcXGJldGFfe0FfaX0iLDFdLFszLDIsIkdfbiIsMV1d
\[\begin{tikzcd}[ampersand replacement=\&]
	{\otimes F(A_i)} \&\& {F(\otimes A_i)} \\
	\\
	{\otimes G(A_i)} \&\& {G(\otimes A_i)}
	\arrow["{F_n}"{description}, from=1-1, to=1-3]
	\arrow["{\beta_{\otimes A_i}}"{description}, from=1-3, to=3-3]
	\arrow["{\otimes\beta_{A_i}}"{description}, from=1-1, to=3-1]
	\arrow["{G_n}"{description}, from=3-1, to=3-3]
\end{tikzcd}\]

The proof is given by the following commuting diagram:
% https://q.uiver.app/?q=WzAsNixbMSwxLCJGKEFfMSksXFxkb3RzLEYoQV9uKSJdLFsyLDAsIkYoXFxvdGltZXMgQV9pKSJdLFsyLDMsIkcoXFxvdGltZXMgQV9pKSJdLFsxLDIsIkcoQV8xKSxcXGRvdHMsRyhBX24pIl0sWzAsMCwiXFxvdGltZXMgRihBX2kpIl0sWzAsMywiXFxvdGltZXMgRyhBX2kpIl0sWzAsMSwiRihtX3soQV9pKV9pfSkiLDFdLFsxLDIsIlxcYmV0YV97XFxvdGltZXMgQV9pfSIsMV0sWzAsMywiXFxiZXRhX3tBXzF9LFxcZG90cyxcXGJldGFfe0Ffbn0iLDFdLFszLDIsIkcobV97KEFfaSlfaX0pIiwxXSxbNCwxLCJGX24iLDFdLFswLDQsIm1feyhGKEFfaSkpXyl9IiwxXSxbMyw1LCJtX3soRyhBX2kpKV9pfSIsMV0sWzUsMiwiR19uIiwxXSxbNCw1LCJcXG90aW1lcyBcXGJldGFfe0FfaX0iLDFdXQ==
\[\begin{tikzcd}[ampersand replacement=\&]
	{\otimes F(A_i)} \&\& {F(\otimes A_i)} \\
	\& {F(A_1),\dots,F(A_n)} \\
	\& {G(A_1),\dots,G(A_n)} \\
	{\otimes G(A_i)} \&\& {G(\otimes A_i)}
	\arrow["{F(m_{(A_i)_i})}"{description}, from=2-2, to=1-3]
	\arrow["{\beta_{\otimes A_i}}"{description}, from=1-3, to=4-3]
	\arrow["{\beta_{A_1},\dots,\beta_{A_n}}"{description}, from=2-2, to=3-2]
	\arrow["{G(m_{(A_i)_i})}"{description}, from=3-2, to=4-3]
	\arrow["{F_n}"{description}, from=1-1, to=1-3]
	\arrow["{m_{(F(A_i))_)}}"{description}, from=2-2, to=1-1]
	\arrow["{m_{(G(A_i))_i}}"{description}, from=3-2, to=4-1]
	\arrow["{G_n}"{description}, from=4-1, to=4-3]
	\arrow["{\otimes \beta_{A_i}}"{description}, from=1-1, to=4-1]
\end{tikzcd}\]

Dually, $F$ is $\parr$-monoidal.

\end{proof}

\begin{prop}
There is a 2-functor $\cR \colon \PolyCat_{wtt} \to \LNULDCFrob$.
\end{prop}
\begin{proof}
We have proven that it is well defined.
We now just need to prove the functoriality conditions.

$\cR(\id_{\cP}) = \id_{\cR(\cP)}$ by definition.
Furthermore, $\cR(\id_{\cP})_n \colon \otimes_n - \Rightarrow \otimes -$ is given on $\Gamma$ by $\cR(\id_{\cP})(m_\Gamma)/(m_{\cR(\id_{\cP})(\Gamma)}) = m_\Gamma/m_\Gamma = \id_{\otimes \Gamma}$.
Similarly, $\cR(\id_{\cP})^n_\Gamma = \id_{\parr \Gamma}$.
So we get that $\cR(\id_{\cP})$ is the identity in $\LNULDCFrob$.

$\cR(G \circ F) = G \circ F = \cR(G) \circ \cR(F)$.
Furthermore, $\cR(G \circ F)_n = \cR(G)_n \circ \cR(F)_n$ is given by the following figure:

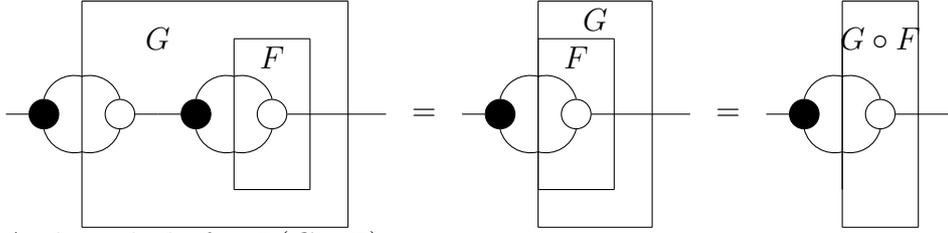
\input{polycat-lax-monoidal-functoriality.tikz}

And similarly for $\cR(G \circ F)^n$.

Since $\cR$ is defined as the identity on transformations and there is no extra structure on them, 2-functoriality follows directly.
\end{proof}

\begin{prop}
There is a 2-equivalence of 2-categories \[\LNULDCFrob \simeq \PolyCat_{wtt}\]
\end{prop}
\begin{proof}
First, notice that the 2-functor $\cP \colon \LNULDCFrob \to \PolyCat$ factors through $\PolyCat_{wtt}$.

Furthermore, its restriction $\cP_{wtt}$ is essentially surjective on objects since for any weak two-tensor polycategory $\cQ$, $\cQ \simeq \cP(\cR(\cQ))$.
Indeed, we have that $\cP(\cR(\cQ))$ has the same objects as \cQ and polymaps $f \colon \Gamma \to \Delta$ in $\cP(\cR(\cQ))$ are polymaps $f \colon \otimes \Gamma \to \parr \Delta$ in \cQ.
Given a polymap in \cQ $f \colon \Gamma \to \Delta$, one gets one in $\cP(\cR(\cQ))$ by factorisation $w_\Delta \backslash f / m_\Gamma$.
In the other direction, given a polymap in $\cP(\cR(\cQ))$ $g \colon \otimes \Gamma \to \parr \Delta$, one gets one in \cQ $w_\Delta \circ g \circ m_\Gamma$.
The factorisation properties and its unicity ensure that these operations are inverse.

Now let us prove that it is fully-faithful, i.e. that $\cP$ induces an isomorphism of categories $\cP_{\cC,\cD} \colon\LNULDCFrob(\cC, \cD) \simeq \PolyCat_{wtt}(\cP(\cC),\cP(\cD))$.
Notice that we cannot directly takes $\cR_{\cP(\cC),\cP(\cD)}$ as an inverse to $\cP_{\cC,\cD}$ because we only have an isomorphism $\cC \simeq \cR(\cP(\cC))$ and not an equality.
This is because a morphism $f \colon A \to B$ in $\cC$ is sent to $\cR(\cP(f)) \colon \otimes A \to \parr B$.
Instead we consider the functor $\cU \colon \PolyCat_{wtt}(\cP(\cC),\cP(\cD)) \to \LNULDCFrob(\cC, \cD)$ that sends a functor of polycategories $G \colon \cP(\cC) \to \cP(\cD)$ to the Frobenius functor $\cU \colon \cC \to \cD$ that acts like $G$ on objects and such that $\cU(f) = \iota_{G(B)}^{-1} \circ f \circ \eta_{G(A)}^{-1}$.
Similarly, its takes a natural transformation $\alpha \colon G \Rightarrow G'$ to one $\iota^{-1}_{G'(-)} \circ \alpha \circ \eta^{-1}_{G(-)}$.
It can be checked that those define Frobenius functors and linear transformations.
Furthermore, it is easy to verify that this is inverse to $\cP_{\cC,\cD}$ since the latter acts on morphisms via $\iota \circ - \circ \eta$ and the former via $\iota^{-1} \circ - \circ \eta^{-1}$.  
\end{proof}

\subsection{Two-tensor polycategories}

Now we will consider how the 2-equivalence above reduces to a 2-equivalence involving (non-lax) linearly distributive categories.
One option is simply to ask for the universal polymaps defining the weak tensor products and par products to be closed under composition.
This is equivalent to ask for a stronger universal property given below.

\begin{definition}
Let $\Gamma,\Delta$ be lists of objects in a polycategory \cP .

A (strong) tensor product of $\Gamma$ is an object $\otimes \Gamma$ equipped with a polymap $m_{\Gamma} \colon \Gamma \to \otimes \Gamma$ such that for any polymap $f \colon \Gamma_1, \Gamma, \Gamma_2 \to \Delta$ there is a unique polymap $f/m_{\Gamma} \colon \otimes \Gamma \to A$ such that $f = f/m_{\Gamma} \circ m_{\Gamma}$.

A (strong) par product of $\Delta$ is an object $\parr \Delta$ equipped with a polymap $w_{\Delta} \parr \Delta \to \Delta$ such that for any polymap $f \colon \Gamma \to \Delta_1, \Delta, \Delta$ there is a unique $w_{\Delta}\backslash f$ such that $f = w_\Delta \circ w_\Delta \backslash f$.
\end{definition}

\begin{definition}
A two-tensor polycategory is a polycategory that has all strong tensors and pars. 
\end{definition}

Any strong tensor/par product is a weak one by restricting the universal property to $\Gamma_i = \emptyset$/$\Delta_i = \emptyset$.

From now on, we will talk of the tensor/par product and only specify weak/strong when it is not clear from the context.

Now, given a lax linearly distributive category $\cC$, let us consider $\cP(\cC)$.
As a reminder, polymaps $f \colon \Gamma \to \Delta$ in $\cP(\cC)$ correspond to polymaps $f \colon \otimes \Gamma \to \parr \Delta$ in \cC. 
We have that $\cP(\cC)$ has weak tensors given by the ones in $\cC$ with universal map $m_\Gamma := \iota_{\otimes \Gamma} \colon \otimes \Gamma \to \parr \otimes \Gamma$.
Then for a polymap $f \colon \otimes \Gamma \to \parr \Delta$, \[f/m_\Gamma := \otimes \otimes \Gamma \xrightarrow{\eta_{\otimes \Gamma}} \otimes \Gamma \xrightarrow{f} \parr  \Delta\]

Now if \cC is a linearly distributive category, given a polymap $f \colon \Gamma_1,\Gamma,\Gamma_2 \to \Delta$ we can define \[f/m_\Gamma \colon \otimes (\Gamma_1, \otimes \Gamma, \Gamma_2) \xrightarrow{\alpha'^{-1}} \otimes(\Gamma_1, \Gamma, \Gamma_2) \xrightarrow{f} \parr \Delta\]
It makes $\otimes$ a strong tensor product.
Similarly for $\parr$.

So we can restrict $\cP$ to $\LDCFrob \to \PolyCat_{tt}$ from linearly distributive categories to two-tensor polycategories.

Now given a weak two-tensor polycategory \cP, we consider a category $\cR(\cP)$ by restricting to unary co-unary polymaps.
It is lax linearly distributive with $\alpha \colon \otimes A_{i,j} \to \otimes \otimes A_{i,j}$ given by
\input{polycat-tensor-alpha.tikz}

If $\otimes$ is a strong tensor product we can define an inverse to $\alpha$ by:
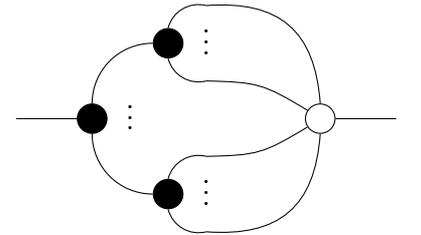
\input{polycat-tensor-alpha-inv.tikz}

This is not possible with the weak universal property since precomposing by a black dot means factorising through the multiplication.
In the weak case it is only possible to do so for the whole domain of a polymap at once.
But to define the inverse we need to do it multiple times on a partition of the domain of $m_{(A_{i,j})}$, the white dot.
The fact that it provides an inverse is clear by the usual rules that white and black dots cancel each other out.

Similarly, the $\gamma \colon \parr \parr A_{i,j} \to \parr A_{i,j}$ are invertible when the par product is strong.

So, we can restrict $\cR$ to $\PolyCat_{tt} \to \LDCFrob$.

\begin{prop}
The 2-equivalence $\LNULDCFrob \simeq \PolyCat_{wtt}$ restricts to a 2-equivalence $\LDCFrob \simeq \PolyCat_{tt}$
\end{prop}

\subsection{Duals}

In this section we recall the notion of duals in a polycategory.
The theory of weak two-tensor polycategories with duals is still to be explored.
In particular, the question of a potential equivalence with lax normal linearly distributive category with duals.
More generally, the theory of lax $\ast$-autonomous categories has still a lot of unanswered challenges.
For example, is it enough to consider the usual notion of duals or should they be also relaxed?
Does the equivalence between two-tensor polycategories with duals and birepresentable polycategories that we will see in the next section have a lax version?
In which case, what is the lax notion of $\ast$-autonomous category related to it?
All of this is left to further work.

From now on, we will consider everything to be strong.
We will see that the 2-equivalence between linearly distributive categories and two-tensor polycategories restricts to one between $\ast$-autonomous categories and two-tensor polycategories with duals.

\begin{definition}
  In a polycategory, a right dual of an object $A$ is an object $\rdual{A}$ equipped with polymaps $\rcup_A : \cdot \to A,\rdual{A}$ and $\rcap_A : \rdual{A},A \to \cdot$ such that $\rcup_A \circ_{\rdual{A}} \rcap_A = \id_A$ and $\rcap_A \circ_A \rcup_A = \id_{\rdual{A}}$.
  A left dual of $A$ is an object $\ldual{A}$ equipped with polymaps $\lcup_A : \cdot \to \ldual{A},A$ and $\lcap_A : A,\ldual{A} \to \cdot$ such that $\lcup_A \circ_{\ldual{A}} \lcap_A = \id_A$ and $\lcap_A \circ_A \lcup_A = \id_{\ldual{A}}$.
\end{definition}
\begin{definition}
  A polycategory is said to have duals
  if every object has a right and a left dual.
\end{definition}
Note that this definition may be simplified in the case of a symmetric polycategory because left and right duals coincide in that case, although following Cockett and Seely we have chosen to consider the more general situation.

We will write $\PolyCat^\ast$ for the sub-2-category of $\PolyCat$ consisting of polycategories with duals and $\PolyCat^\ast_{tt}$ for the one consisting of two-tensor polycategories with duals.

Now consider \cC a $\ast$-autonomous category.
One way to think about $\cC$ is as a linearly distributive category with duals.
The duals in $\cC$ come equipped with cups and caps given for example in the left dual case by $\lcup_A \colon \otimes_0 \to \parr_2(\ldual{A},A)$ and $\lcap_A \colon \otimes_2(A,\ldual{A}) \to \parr_0$.
These induce polymaps $\lcup_A \colon \cdot \to \ldual{A},A$ and $\lcap_A \colon A,\ldual{A} \to \cdot$ in $\cP(\cC)$ that exhibit $\ldual{A}$ as a left dual in $\cP(\cC)$.

Dually, if $\cP$ is a two-tensor polycategory with duals, the polymaps $\lcup_A \colon \cdot \to \ldual{A},A$ and $\lcap_A \colon A,\ldual{A} \to \cdot$ induce maps $w_{(\ldual{A},A)} \backslash \lcup_A / m_{()} \colon \otimes_0 \to \parr_2(\ldual{A},A)$ and $w_{()} \backslash \lcap_A / m_{(A,\ldual{A})} \colon \otimes_2(A,\ldual{A}) \to \parr_0$ in $\cR(\cP)$.

\begin{prop}
The 2-equivalence $\LDCFrob \simeq \PolyCat_{tt}$ restricts to a 2-equivalence $\ast-\mathbf{Aut} \simeq \PolyCat_{tt}^\ast$.
\end{prop}

\section{Birepresentable polycategories}

In this section, we will introduce a notion of birepresentability for polycategories.
The idea is to have a general notion of object equipped with an universal polymap that encompasses all the connectives defined above.
To do that we will introduce the notion of a polymap being universal in one of its inputs or one of its outputs.
When restricted to unary polymaps universal in their only input and co-unary polymaps universal in their only input we get the polymap $w$ and $m$ defining $\parr$ and $\otimes$.
Maybe more surprisingly, when we restrict to polymaps with two inputs and no output or no input and two outputs we get the cups and cups defining the duals.
So any birepresentable polycategory is a two-tensor polycategory with duals.
We will prove that the converse holds and that any universal polymap can be decomposed into instances of $m,w,\lcup,\lcap,\rcup,\rcap$.

\begin{definition}
  A polymap $u : \Gamma \to \Delta_1, A, \Delta_2$ is said to be universal in the output $A$ (or out-universal for short, or simply universal when there is no ambiguity), written $u : \Gamma \to \Delta_1, \focout{A}, \Delta_2$ if for any polymap $h : \Gamma_1, \Gamma, \Gamma_2 \to \Delta_1, \Delta, \Delta_2$ such that $\Gamma_i = \emptyset$ or $\Delta_i = \emptyset$, there is a unique polymap $h/u : \Gamma_1, A, \Gamma_2 \to \Delta$ such that $h = h/u \circ_A u$.
  
  Dually, a  polymap $n : \Gamma_1, A, \Gamma_2 \to \Delta$ is universal in the input $A$ (or in-universal), written $n : \Gamma_1, \focin{A}, \Gamma_2 \to \Delta$ if for any polymap $h : \Gamma_1, \Gamma, \Gamma_2 \to \Delta_1,\Delta, \Delta_2$ such that $\Gamma_i = \emptyset$ or $\Delta_i = \emptyset$ there is a unique polymap $n \backslash h : \Gamma \to \Delta_1, A, \Delta_2$ such that $h = n \circ_A n \backslash h$.
\end{definition}
\begin{remark}
  By extension, we say that $A$ is an out-universal object (resp.~ in-universal object) with respect to the surrounding context $\Gamma\to \Delta_1,\textvisiblespace,\Delta_2$ (resp.~$\Gamma_1,\textvisiblespace,\Gamma_2 \to \Delta$) if there is an out-universal polymap $\Gamma \to \Delta_1,\focout{A},\Delta_2$ (resp.~in-universal polymap $\Gamma_1,\focin{A},\Gamma_2 \to \Delta$).
  For a fixed surrounding context, in-universal and out-universal objects are unique up to unique isomorphism.
\end{remark}

\begin{definition}
  A polycategory is said to be birepresentable if it has all in-universal and out-universal objects, that is, if for any $\Gamma$, $\Delta_1$, $\Delta_2$ there is an object $A$ equipped with an out-universal polymap $\Gamma \to \Delta_1,\focout{A},\Delta_2$, and similarly, for any $\Gamma_1$, $\Gamma_2$, $\Delta$ there is an object $A$ equipped with an in-universal polymap $\Gamma_1,\focin{A},\Gamma_2 \to \Delta$.
\end{definition}

In \cite{BlancoZeilberger2020} we called such polycategories $\ast$-representable.
The term birepresentable has been introduced by Mike Shulman in \cite{Shulman2021}.

These universal polymaps are generalisations of the universal polymaps for $\otimes$ and $\parr$.
In Section~\ref{ch:bifib-poly}, we will see that these concepts are special cases of more general fibrational concepts.
Like strong universal multimaps in a multicategory, both in-universal and out-universal polymaps are closed under composition in an appropriate sense.
\begin{prop}
  \label{prop:univ_comp}
  In-universal polymaps compose, in the sense that if $f : \Gamma_1, \focin{A}, \Gamma_2 \to \Delta_1, B, \Delta_2$ and $g : \Gamma_1', \focin{B}, \Gamma_2' \to \Delta'$, then $g \circ_B f : \Gamma_1', \Gamma_1, \focin{A}, \Gamma_2, \Gamma_2' \to \Delta_1, \Delta', \Delta_2$.
  Similarly, out-universal maps compose in the sense that if $f : \Gamma \to \Delta_1, \focout{B}, \Delta_2$ and $g : \Gamma_1', B, \Gamma_2' \to \Delta_1', \focout{C}, \Delta_2'$, then $g \circ_B f : \Gamma_1', \Gamma, \Gamma_2' \to \Delta_1, \Delta'_1, \focout{C}, \Delta'_2, \Delta_2$.
\end{prop}
\begin{proof}
  Given $h \colon \Gamma_1', \Gamma_1, \Gamma'', \Gamma_1, \Gamma_2' \to \Delta_1'',\Delta_1,\Delta', \Delta_2, \Delta_2''$, we take $(gf)\backslash h := f \backslash (g \backslash h) \colon \Gamma'' \to \Delta_1'', A , \Delta_2''$.
  First we have that \[gf ((gf)\backslash h) = gf (f \backslash (g \backslash h)) = g (g \backslash h) = h\]  
  Furthermore, let $k \colon \Gamma'' \to \Delta_1'',A,\Delta''$ be a polymap such that \[gfk = h = gf((gf)\backslash h)\]
  Then by uniqueness of the factorisation through $g$ we have that $fk = f((gf)\backslash h)$ and by uniqueness of the factorisation through $f$ that $k = (gf)\backslash h$.
\end{proof}

An immediate consequence of these definitions is that tensor products can be considered as out-universal objects, and par products as in-universal objects.
\begin{prop}
  \label{prop:tensorpar}
  An object $\bigotimes\Gamma$ equipped with a polymap $m : \Gamma \to \bigotimes\Gamma$ is a tensor product of $\Gamma$ iff $m$ is out-universal (in its unique output).
  Dually, an object $\bigparr \Delta$ equipped with a polymap $w : \bigparr \Delta \to \Delta$ is a par product of $\Delta$ iff $w$ is in-universal (in its unique input).
\end{prop}

Somewhat more surprisingly, duals can also be characterised as either in-universal or out-universal objects.

\begin{proposition}
  \label{prop:dual}
  Let $A$ and $\rdual{A}$ be objects of a polycategory $\mathcal{P}$. The following are equivalent:
  \begin{enumerate}
  \item there is an out-universal map $\rcup_A : \cdot \to A,\focout{\rdual{A}}$
  \item there is an in-universal map $\rcap_A : \focin{\rdual{A}},A \to \cdot$
  \item there is an out-universal map $\rcup_A : \cdot \to \focout{A},\rdual{A}$
  \item there is an in-universal map $\rcap_A : \rdual{A},\focin{A} \to \cdot$
  \item $\rdual{A}$ is the right dual of $A$
  \end{enumerate}
\end{proposition}

\begin{proof}
  Let prove that $1 \Leftrightarrow 5$.
  The others - $2 \Leftrightarrow 5$ and so on - are proved similarly.
  
  First, let suppose that we have an universal polymap $\rcup_A \colon \cdot \to A, \focout{\rdual{A}}$.
  We can define $\rcap_A := \id_A / \rcup_A \colon \rdual{A}, A \to \cdot$.
  By definition we have that \[\rcap_A \circ_{\rdual{A}} \rcup_A = \id_A/\rcup_A \circ_{\rdual{A}} \rcup_A = \id_A\]
  Furthermore, by postcomposing $\rcap_A$ we have that \[\rcap_A \circ_A \rcap_A \circ_{\rdual{A}} \rcup_A = \rcup_A = \id_{\rdual{A}} \circ_{\rdual{A}} \rcup_A\].
  But since factorisation through $\rcup_A$ in $\rdual{A}$ is unique, we have that \[\rcup_A \circ_A \rcap_A = \id_{\rdual{A}}\]
  
  Conversely, suppose that $\rdual{A}$ is a right dual.
  Consider a polymap $f \colon \Gamma \to A,\Delta$.
  We define $f/\rcup_A := \rcap_A \circ_A f \colon \rdual{A}, \Gamma \to \Delta$.
  Then we have that \[f/\rcup_A \circ_{\rdual{A}} \rcup_A = \rcap_A \circ_A f \circ_{\rdual{A}} \rcup_A = \rcap_A \circ_{\rdual{A}} \rcup_A \circ_A f = \id_A \circ_A f = f\]

Furthermore, assume that $g \colon \rdual{A}, \Gamma \to \Delta$ is such that $g \circ_{\rdual{A}} \rcup_A = f$.
Then, \[\rcap_A \circ_A f = \rcap_A \circ_A g \circ_{\rdual{A}} \rcup_A = \rcap_A \circ_{\rdual{A}} \rcup_A \circ_A g = \id_A g = g\]
So $f/\rcup_A$ is uniquely determined which concludes the proof.  
\end{proof}

There is of course a similar result for left duals.

\begin{theorem}  $\mathcal{P}$ is a two-tensor polycategory with duals iff it is birepresentable.
\end{theorem}
\begin{proof}
  The right to left direction follows by propositions \ref{prop:tensorpar} and \ref{prop:dual}.
  For the left to right direction we want to construct in-universal and out-universal objects for any contexts just using $\otimes$, $\parr$ and $\ast$.
  Given contexts $\Gamma,\Delta_1,\Delta_2$ consider the object $A := \otimes(\rdual{\Delta}_1, \Gamma, \ldual{\Delta_2})$ where $\rdual{\Delta}_1 := \rdual{B}_{1,n_1}, ..., \rdual{B}_{1,1}$ for $\Delta_1 = B_{1,1},...,B_{1,n_1}$ and similarly for $\ldual{\Delta}_2$.
  This object comes with the following polymap, which is a composition of universal polymaps along their universal objects. So by proposition \ref{prop:univ_comp}, it is universal.
  
  \begin{center}
    \input{polycat-birepresentable-out.tikz}
  \end{center}
  
  Similarly given $\Gamma_1,\Gamma_2,\Delta$ the object $A := \parr(\ldual{\Gamma}_1, \Delta, \rdual{\Gamma}_2)$ is in-universal with in-universal polymap.
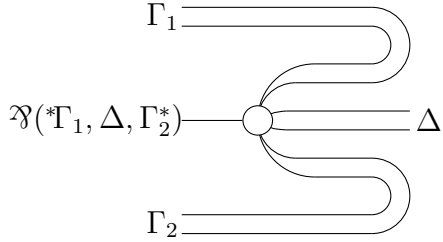
\begin{center}
    \input{polycat-birepresentable-in.tikz}
  \end{center}
\end{proof}

There is a string diagram calculus for birepresentable polycategories that extends the one of two-tensor polycategories with duals and the one of monoidal biclosed categories.
The idea is to represent a universal polymap as a white spider with an arrow indicating the object it is universal in.
Then, its factorisation property is given by a dual black spider with the universal object marked.
The same color spiders compose along marked objects.
Furthermore, there are two cancellation possibilities for a white and a black spider with mirrored types: either they are connecting along their marked arrows or they are connecting along all the non-marked ones.
The equations for the graphical calculus are given in figure \ref{fig:polycat-birepresentable-equations}.
To understand how to use the black spider to factor, consider the universal polymap $u \colon \Gamma_1, \focin{A}, \Gamma_2 \to \Delta$ and a polymap $f \colon \Gamma_1, \Gamma, \Gamma_2 \to \Delta_1, \Delta, \Delta$, then the polymap $u\backslash f \colon \Gamma \to \Delta_1, A, \Delta_2$ is given by:
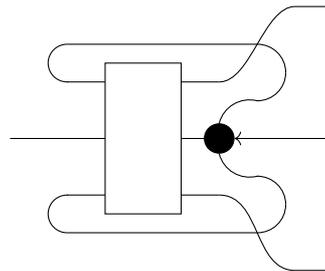
\input{polycat-birepresentable-factorisation.tikz}

\begin{figure}
\centering
\resizebox{\hsize}{!}{\input{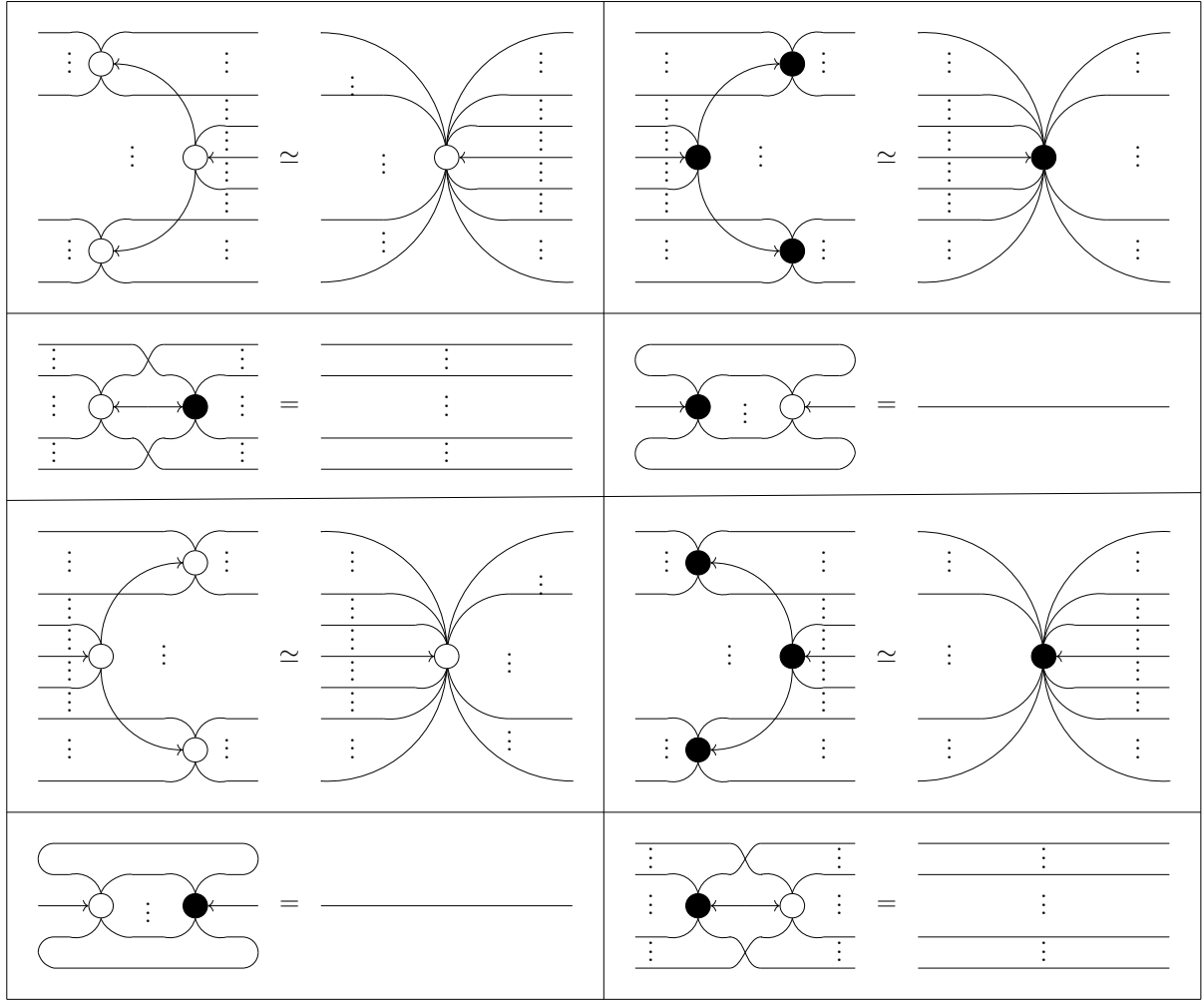}}
\caption{Graphical calculus for a birepresentable polycategory}
\label{fig:polycat-birepresentable-equations}
\end{figure}

In a birepresentable polycategory \cP, any polymap $A_1,\dots, A_m \to B_1, \dots, B_n$ is equivalent to a polymap $\cdot \to \ldual{A_m},\dots, \ldual{A_1}, B_1, \dots, B_n$ by precomposing with cups.
This gives a one-sided presentation of the maps in the polycategory.
If \cP is the syntactic polycategory of a two-sided sequent calculus for MLL, restricting to co-nullary polymaps corresponds to considering a one-sided presentation of the sequent calculus.

Furthermore, consider the polymaps $\ldual{A}, A \to \cdot$, $\ldual{B}, B \to \cdot$, $A, B \to A\otimes B$ and $A\otimes B, \rdual{(A\otimes B)} \to \cdot$ universal in $A, B$, $A \otimes B$  and $\rdual{(A\otimes B)}$ respectively.
Since universal maps compose, we get a polymap $\rdual{(A\otimes B)} \to \ldual{B}, \ldual{A}$ universal in $\rdual{A\otimes B}$.
This means that in the graphical calculus, all the white dots - except the $(0,1)$-ones - can be generated by only $(2,1)$-white dots introducing the tensor, and cups and caps.
Similarly, all the black dots - except the $(0,1)$-ones - can be generated by only $(2,1)$-black dots introducing the par, and cups and caps.
If we rename the cup ``axiom'' and the cap ``cut'', the one-sided graphical calculus generated by axiom, cut, (binary) tensor introduction and (binary) par introduction is the graphical calculus known in linear logic as proof-nets for MLL without units.
More precisely, these correspond to proof structures, and not all proof structures represent a proof.
The ones that do are called proof nets and there are criterion to identify these.
It would be an interesting question to adapt these criterion to the graphical calculus of polycategories to make it complete.

\section{Examples}

\begin{example}
  Any linearly distributive category $\mathcal{C}$ gives a polycategory $\mathcal{P}(\mathcal{C})$ called its underlying polycategory.
  It has the same objects as $\mathcal{C}$ and a polymap $f : A_1,...,A_m \to B_1,...,B_n$ in $\mathcal{P}(\mathcal{C})$ is a map $f : A_1 \otimes ... \otimes A_m \to B_1 \parr ... \parr B_n$ in $\mathcal{C}$.
\end{example}

\begin{example}\label{ex:polymon}
  In particular any monoidal category gives rises to a polycategory with the same objects and with polymaps $f : A_1 \otimes ... \otimes A_m \to B_1 \otimes ... \otimes B_n$.
\end{example}

\begin{example}
  The terminal polycategory $\one$ has one object $\ast$ and a unique arrow $\spider mn : \ast^m \to \ast^n$ for every arity $m$ and co-arity $n$.
  Although this example is trivial, we will see that it plays an important role when looking at bifibrations.
\end{example}

\begin{example}
  Any category induces a polycategory with only unary-counary maps.
  Conversely any polycategory has an underlying category obtained by forgetting about the non-unary-counary maps. 
\end{example}

\begin{example}
  From any multicategory $\mathcal{M}$ we can define two polycategories $\mathcal{M}^{+}$ and $\mathcal{M}^{-}$ that have the same objects as $\mathcal{M}$.
  The polymaps of $\mathcal{M}^+$ have always exactly one output and correspond to multimaps in $\mathcal{M}$ while the polymaps in $\mathcal{M}^-$ have always exactly one input and correspond to multimaps in $\mathcal{M}$ reversed.
  Conversely from any polycategory we get two multicategories by restricting to polymaps with exactly one output and (reversed) polymaps with exactly one input.
\end{example}

\begin{example}
  There are polycategories $\mathbf{Vect}$ and $\mathbf{FVect}$ of vector spaces (resp.~finite dimensional vector spaces) and polylinear maps.
  Both of these can be seen as the underlying polycategories of monoidal categories of vector spaces and linear maps.
  $\mathbf{FVect}$ is a representable polycategory with duals while $\mathbf{Vect}$ is representable but does not have duals in general.
  In fact the vector spaces that admit a dual are precisely the finite dimensional ones.
\end{example}

\begin{example}
  Free polycategories give examples of polycategories which are \emph{not} representable.
  Let a ``poly-signature'' $\Sigma$ consist of a collection of types, together with for any finite lists of types $\Gamma$ and $\Delta$, a set of operations $\Sigma(\Gamma;\Delta)$.
  The free polycategory generated by $\Sigma$, denoted $\mathcal{P}(\Sigma)$, has types as objects, and polymaps given by planar oriented trees with a boundary of free edges, whose nodes are labelled by operations and whose edges are labelled by types subject to the constraints specified by the signature.
  For example, here is a depiction of the composite polymap  $f \circ_A (g \circ_B f) : A, B, B \to A, B$ in the free polycategory generated by the signature containing a pair of types $A$ and $B$ and a pair of operations $f : A,B \to B$ and $g : B \to A,A$ (in the diagram, the edges are implicitly oriented from left to right):
  \begin{center}
    \scalebox{0.7}{\input{free.tikz}}
  \end{center}
  In general, composition is performed by grafting two trees along an edge, while the identity on a type $A$ is given by the trivial tree with no nodes and one oriented edge labelled $A$.
  Observe this polycategory is not representable, for example there is no polymap $A,A \to A \otimes A$.
\end{example}

\begin{example}
  A one-object multicategory is commonly referred to as an operad,
  while a one-object polycategory is also known as a dioperad \cite{Gan2003}.
  For any polycategory $\mathcal{P}$ and any object $A \in \mathcal{P}$ there is a dioperad called the endomorphism dioperad of $A$, denoted $End_{\mathcal{P}}(A)$, defined as the full subpolycategory of $\mathcal{P}$ containing only the object $A$.
  It has one object and its polymaps correspond to polymaps $A,...,A \to A,...,A$ in $\mathcal{P}$. 
\end{example}

\section{Example of Banach spaces}
\label{sec:Banach1}

In this section, we focus on Banach spaces.
Although the use of polycategories is new most of the results are standard.
The question of what norms can be assigned to the tensor product of Banach spaces, and more generally of topological vector spaces, was the focus of Grothendieck's PhD thesis, see \cite{Grothendieck1954}.
The standard theory of Banach spaces can be found in \cite{Ryan2002}.
We will only consider finite dimensional Banach spaces.
Some of the results might be extended to the general case, although not all of them.
For example, a vector space has duals in the polycategorical sense iff it is finite dimensional.
Furthermore, in the treatment of this example we often rely on the fact that every finite-dimensional normed vector space is complete.
This allows us to skip the subtleties about completeness.

We fix a field $\mathbb{K}$ that is either $\mathbb{R}$ or $\mathbb{C}$.
Given a ($\mathbb{K}$-)vector space $A$, we denote $\rdual{A}$ the vector space of functionals over $A$, i.e., linear maps $A \to \mathbb{K}$.

\begin{definition}
A \defin{polylinear map} $f \colon A_1,\dots, A_m \to B_1, \dots B_n$ between vectors spaces, is a functional assignment of a scalar $(\varphi_1,\dots, \varphi_n)f(a_1,\dots,a_m) \in \mathbb{K}$ for each $a_i \in A_i$ and each $\varphi_j \in \rdual{B_j}$ that is linear in each variable, i.e.,
\begin{itemize}
\item $(\dots)f(\dots, \lambda a_i + a_i', \dots) = \lambda (\dots)f(\dots, a_i, \dots) + (\dots)f(\dots, a_i', \dots)$
\item $(\dots, \lambda \varphi_i + \varphi_i', \dots)f(\dots) = \lambda (\dots, \varphi_i, \dots )f(\dots) + (\dots, \varphi_i', \dots)f(\dots)$
\end{itemize}
\end{definition}

Consider two polylinear maps $f \colon \Gamma \to \Delta_1, A, \Delta_2$ and $g \colon \Gamma_1, A, \Gamma_2 \to \Delta$ and lists of vectors and functionals $\overrightarrow{a_i} \in \Gamma_i$, $\overrightarrow{a} \in \Gamma$, $\overrightarrow{\varphi_j} \in \Delta_j$, and $\overrightarrow{\varphi} \in \Delta$.
The expression $(\overrightarrow{\varphi})g(\overrightarrow{a_1},-,\overrightarrow{a_2})$ defines a functional on $A$.
Indeed, for any $x \in A$ it assigns a scalar $(\overrightarrow{\varphi})g(\overrightarrow{a_1},x,\overrightarrow{a_2})$ in a linear way.
On the other hand, $(\overrightarrow{\varphi_1},-,\overrightarrow{\varphi_2})f(\overrightarrow{x})$ assigns a scalar to any functional on $A$.
So we get a scalar $(\overrightarrow{\varphi_1},(\overrightarrow{\varphi})g(\overrightarrow{a_1},-,\overrightarrow{a_2}),\overrightarrow{\varphi_2})f(\overrightarrow{x})$ that we will note $(\overrightarrow{\varphi_1},\overrightarrow{\varphi},\overrightarrow{\varphi_2})g\circ f(\overrightarrow{a_1},\overrightarrow{a},\overrightarrow{a_2})$.
This assignment is linear in each of the variables since $f$ and $g$ are.
This defines a polylinear map $g \circ f$, their composition.

Furthermore, the polylinear map $\id_A \colon A \to A$ is defined by $(\varphi)\id_A(a) := \varphi(a)$.

\begin{rk}
Technically, we should only define the composition in a planar setting, i.e. with the constraint that both input and output have one empty side.
That would not capture all the structure of the polylinear composition, but it would be enough for everything we do in this thesis: defining the connectives and extending to Banach spaces via the fibrational structure.
Instead of constraining the composition to a planar one, one could also take into account the symmetries and consider symmetric polycategories.
\end{rk}

\begin{prop}
There is a polycategory $\FVect$ of finite dimensional vector spaces and polylinear maps.
It has all tensors and pars both given by $\otimes$ the tensor product of vector spaces.
It also has all duals given by the dual space $\rdual{A}$, so it is birepresentable.
\end{prop}
\begin{proof}
$m_{A_i} \colon A_1,\dots,A_n \to A_1 \otimes \dots \otimes A_n$ is defined by \[(\varphi)m(a_1,\dots,a_n) := \varphi(a_1 \otimes \dots \otimes a_n)\]
Any polymap $f \colon \Gamma_1, A_1, \dots, A_n, \Gamma_2 \to \Delta$ uniquely factorises into a polymap \\$f/m_{A_i} \colon \Gamma_1, A_1 \otimes \dots \otimes A_n, \Gamma_2 \to \Delta$ by \[(\overrightarrow{\psi})f/m(\overrightarrow{x_1},\sum_i a_1^i \otimes \dots \otimes a_n^i,\overrightarrow{x_2}) := \sum_i (\overrightarrow{\psi})f(\overrightarrow{x_1},a_1^i,\dots,a_n^i,\overrightarrow{x_2})\]

$w_{B_j} \colon B_1 \otimes \dots \otimes B_n \to B_1, \dots, B_n$ is defined by
\[ (\varphi_1,\dots,\varphi_n)w(\sum_i b_1^i \otimes \dots \otimes b_n^i) := \sum_i \varphi_1(b_1^i)\dots\varphi_n(b_n^i) \]
Any polymap $f \colon \Gamma \to \Delta_1, B_1,\dots, B_n, \Delta_2$ uniquely factorises through $w\backslash f$ defined by:
\[(\overrightarrow{\psi_1},\sum_i \varphi_1^i \otimes \dots \otimes \varphi_n^i, \overrightarrow{\psi_2})w \backslash f(\overrightarrow{x}) := \sum_i (\overrightarrow{\psi_1},\varphi_1,\dots,\varphi_n,\overrightarrow{\psi_2})f(\overrightarrow{x})\]
For the duals, the map $\rcup_A \colon \cdot \to A, \rdual{A}$ is given by $(\varphi, \check{a})\rcup_A = \check{a}(\varphi)$ where $(\varphi,\check{a}) \in \rdual{A} \otimes A^{\ast\ast}$.
While the map $\rcap_A \colon \rdual{A}, A \to \cdot$ is given by $\rcap_A(\varphi,a) = \varphi(a)$.
We can check that these verify the snake identities.
\[ (\varphi)(\rcap_A \circ_{\rdual{A}} \rcup_A)(a) = (\varphi, \rcap_A(-,a)) \rcup_A = \rcap_A(\varphi,a) = \varphi(a) \]
\[ (\check{a})(\rcap_A \circ_A \rcup_A)(\varphi) = (\rcap_A(\varphi,-), \check{a})\rcup_A = (\varphi,\check{a})\rcup_A = \check{a}(\varphi) \]
\end{proof}

Given a finite dimensional Banach space $(A,\|-\|_A)$ one can define a finite dimensional Banach space $(\rdual{A},\|-\|_{\rdual{A}})$ where $\|-\|_{\rdual{A}}$ is the dual norm defined on functionals by:
\[ \|\varphi\| := \sup\limits_{\|a\|_A \leq 1} |\varphi(a)|\]

Continuous linear maps between Banach spaces correspond to bounded maps.
This can be generalised to polylinear maps.

\begin{definition}
  A polylinear map $f : A_1,...,A_m \to B_1,...,B_n$ between normed vector spaces $(A_i,\|-\|_{A_i})$ and $(B_j,\|-\|_{B_j})$ is \defin{bounded} if \[\exists K, \forall a_i \in A_i, \forall \varphi_j \in B_j^\ast,\, |(\varphi_1, ..., \varphi_n)f(a_1, ..., a_m)| \leq K \prod\limits_{i,j}\|a_i\|_{A_i}\|\varphi_j\|_{B_j^\ast}\].
\end{definition}

\begin{prop}
  A unary polymap $f : A \to B$ is bounded if it is bounded as a linear map.
\end{prop}

We write $\|f\|_{A \multimap B}$ for the smallest such $K$.
It defines a norm on $A \to B$ the vector space of linear maps from $A$ to $B$ and $f$ is contractive when its norm is smaller than 1.

\begin{definition}
  A polylinear map  $f : A_1,...,A_m \to B_1,...,B_n$ between normed vector spaces $(A_i,\|-\|_{A_i})$ and $(B_j,\|-\|_{B_j})$ is \defin{contractive} if \[\forall a_i \in A_i, \forall \varphi_j \in B_j^\ast,\ |(\varphi_1, ..., \varphi_n) f(a_1, ..., a_m)| \leq \prod\limits_{i,j}\|a_i\|_{A_i}\|\varphi_j\|_{B_j^\ast}\]
\end{definition}

\begin{definition}
  There are polycategories:
  \begin{itemize}
  \item $\FBan$ of finite dimensional Banach spaces and bounded polylinear maps
  \item $\FBanc$ of finite dimensional Banach spaces and contractive polylinear maps
  \end{itemize}
\end{definition}

For objects in any of those polycategories to be isomorphic they need to be isomorphic as vector spaces.
$(A,\|-\|)$ and $(A,\|-\|')$ are isomorphic in $\FBan$ if $\exists K,K', \forall a \in A,\ K \|a\| \leq \|a\|' \leq K' \|a\|$.
Such norms are called equivalent.
On the other hand, two Banach spaces are isomorphic in $\FBanc$ if their norms are equal in the sense that they are equal on every vector.
Since all the norms on a given finite dimensional vector space are equivalent, \FBan{} is not an interesting polycategory to study.

\begin{prop}
  $\FBan$ is equivalent to $\FVect$.
\end{prop}

On the other hand, $\FBanc$ is a $\ast$-representable polycategory that does not come from a compact closed category.
It is one of the examples of $\ast$-autonomous categories described in Barr's original paper \cite{Barr1979}.
In this article, the $\ast$-autonomous structure of the cateogry of finite dimensional Banach spaces and contractive linear maps is proved by using a characterisation of a $\ast$-autonomous category as a symmetric monoidal closed category where the canonical maps $A \to A^{\ast\ast}$ are isomorphisms.
In particular, the induced norm for the par is never discussed.
We did not find any reference in the literature linking it to the well-known injective norm in the theory of Banach spaces.

\begin{definition}
  Let $(A_i,\|-\|_{A_i})_i$ be Banach spaces.
  The \defin{projective norm} $\|-\|_{\otimes A_i}$ and the \defin{injective norm} $\|-\|_{\parr A_i}$ are the norms defined on the vector space $A_1 \otimes \dots \otimes A_n$ by the following formulae:
\[
  \|u\|_{\otimes A_i} := \inf\limits_{u = \sum\limits_j a_1^j \otimes \dots \otimes a_n^j}\sum\limits_j\|a_1^j\|_{A_1}\dots\|a_n^j\|_{A_n}
  \qquad
  \|u\|_{\parr A_i} := \sup\limits_{\|\varphi_i\|_{A_i^\ast}\leq 1}|(\varphi_1 \otimes \dots \otimes \varphi_n)(u)|
\]
\end{definition}
These norms are known to be extremal among the set of well-behaved norms that one can put on the tensor.

\begin{definition}
    For Banach spaces $(A_i,\|-\|_{A_i})$, a norm $\|-\|$ on $A_1 \otimes \dots \otimes A_n$ is a \defin{crossnorm} if 
\begin{itemize}
\item $\forall a_i \in A_i,\ \|a_1 \otimes \dots \otimes a_n\| \leq \|a_1\|_{A_1}\dots\|a_n\|_{A_n}$ 
\item $\forall \varphi_j \in A_j^\ast, \ \|\varphi_1 \otimes \dots \otimes \varphi_n \|' \leq \|\varphi_1\|_{A_1^\ast}\dots\|\varphi_n\|_{A_n^\ast}$ where $\|-\|'$ is the dual norm associated to $\|-\|$
\end{itemize}
\end{definition}

\begin{remark}
  It is equivalent to ask for equalities in the definition.
  A proof can by found in \cite{Ryan2002}.
\end{remark}

\begin{prop}
  A norm is a crossnorm iff it makes $m \colon A_1,\dots,A_n \to A_1 \otimes \dots \otimes A_n$ and $w \colon A_1 \otimes \dots \otimes A_n \to A_1,\dots,A_n$ contractive.
\end{prop}
\begin{proof}
Suppose that $\|-\|$ is a crossnorm.
Then we want to prove that for any $\|a_i\|_{A_i} \leq 1$ and any $\|\varphi\|' \leq 1$, \[|(\varphi)m(a_1,\dots,a_n)| := |\varphi(a_1\otimes\dots\otimes a_n)| \leq 1\]
By definition $\|\varphi\|' := \sup\limits_{\|u\| \leq 1} |\varphi(u)| \leq 1$.
Since $\|-\|$ is contractive we have that 
\[\|a_1 \otimes \dots \otimes a_n\| \leq \|a_1\|_{A_1}\dots\|a_n\|_{A_n} \leq 1\]
So $|\varphi(a_1\otimes\dots\otimes a_n)| \leq 1$, proving that $m$ is contractive.

Now we want to prove that $w$ is contractive.
That is for any $\|\varphi_i\|_{\rdual{A_i}} \leq 1$ and any $\|u\| \leq 1$, \[|(\varphi_1,\dots,\varphi_n)w(u)| \leq 1\]
By definition we have that $(\varphi_1,\dots,\varphi_n)w(u) = (\varphi_1 \otimes \dots \otimes \varphi_n)(u)$.
Since $\|-\|$ is contractive we have
\[\|\varphi_1 \otimes \dots \otimes \varphi_n\|' \leq 1\]
which means by definition of $\|-\|'$ and since $\|u\| \leq 1$ that \[|(\varphi_1 \otimes \dots \otimes \varphi_n)(u)| \leq 1\]

Now suppose that $m$ and $w$ are contractive.
Since $m$ is contractive we have that for any $\|a_i'\|_{A_i} \leq 1$ and any $\|\varphi\|' \leq 1$, $|\varphi(a_1 \otimes \dots \otimes a_n)| \leq 1$.
Now let us fix $a_i \neq 0 \in A_i$ and consider the vectors $\frac{a_i}{\|a_i\|_{A_i}}$.
We have that by the property of a norm that \[\|\frac{a_i}{\|a_i\|_{A_i}}\|_{A_i} = \frac{\|a_i\|_{A_i}}{\|a_i\|_{A_i}} = 1\]
Furthermore, as a corollary of the Hahn-Banach theorem (see \cite{Ryan2002}) there exists a functional $\|\varphi\|' \leq 1$ such that \[\varphi(\frac{a_1}{\|a_1\|_{A_1}}\otimes\dots\otimes\frac{a_n}{\|a_n\|_{a_n}}) = \|\frac{a_1}{\|a_1\|_{A_1}}\otimes\dots\otimes\frac{a_n}{\|a_n\|_{a_n}}\|\]
By the property of a norm
\[\|\frac{a_1}{\|a_1\|_{A_1}}\otimes\dots\otimes\frac{a_n}{\|a_n\|_{a_n}}\| = \frac{\|a_1 \otimes \dots \otimes a_n\|}{\|a_1\|_{A_1}\dots\|a_n\|_{A_n}}\]
and since $m$ is contractive
\[\frac{\|a_1 \otimes \dots \otimes a_n\|}{\|a_1\|_{A_1}\dots\|a_n\|_{A_n}} \leq 1\]
So $\|a_1 \otimes \dots \otimes a_n\| \leq \|a_1\|_{A_1} \dots \|a_n\|_{A_n}$.
If at least one of the $a_i$ is $0$ then both sides equal $0$ and the inequality holds.
This proves one property of the crossnorm.

For the other property, we want to prove that for any $\varphi_i \in \rdual{A_i}$ we have \[\|\varphi_1 \otimes \dots \otimes \varphi_n\|' \leq \|\varphi_1\|_{\rdual{A_1}} \dots \|\varphi_n\|_{\rdual{A_n}}\]
If at least one of the $\|\varphi_i\|_{\rdual{A_i}} = 0$ then by definition of a norm $\varphi_i =0$ and the left side is also $0$.
Otherwise the above inequality is equivalent to
\[ \|\frac{\varphi_1}{\|\varphi_1\|_{\rdual{A_1}}} \otimes \dots \otimes \frac{\varphi_n}{\|\varphi_n\|_{\rdual{A_n}}}\|' \leq 1\]
By definition of $\|-\|'$ this is equivalent to prove that for any $\|u\| \leq 1$,
\[ |\frac{\varphi_1}{\|\varphi_1\|_{\rdual{A_1}}} \otimes \dots \otimes \frac{\varphi_n}{\|\varphi_n\|_{\rdual{A_n}}}(u)| \leq 1\]
which is true since $\|\frac{\varphi_i}{\|\varphi_i\|_{\rdual{A_i}}}\|_{\rdual{A_i}} \leq 1$ and $w$ is contractive.
\end{proof}

The injective and projective norms are crossnorms.
The following property of the injective and projective crossnorm made us consider the injective crossnorm as a potential candidate for interpreting the par, and was one of our original motivations for studying the notion of bifibration of polycategories that we will develop later in this thesis.

\begin{prop}
  Let $\|-\|$ be a crossnorm then for any $u\in A_1 \otimes \dots \otimes A_n$ we have \[\|u\|_{A_1 \parr \dots \parr A_n} \leq \|u\| \leq \|u\|_{A_1 \otimes \dots \otimes A_n}\]
\end{prop}
\begin{proof}
Consider a decomposition $u = \sum_j a_1^j \otimes \dots \otimes a_n^j$ then by properties of a norm we have
\[\|u\| = \|\sum_j a_1^j \otimes \dots \otimes a_n^j\| \leq \sum_j \|a_1^j \otimes \dots \otimes a_n^j\|\]
Now since $\|-\|$ is a crossnorm,
\[\sum_j \|a_1^j \otimes \dots \otimes a_n^j\| \leq \sum_j \|a_1^j\|_{A_1}\dots\|a_n^j\|_{A_n}\]
Since this is true for any decomposition, this is true on the infimum and $\|u\| \leq \|u\|_{A_1 \otimes \dots \otimes A_n}$.

Now consider functionals $\|\varphi_i\|_{\rdual{A_i}} \leq 1$.
Since $\|-\|$ is a crossnorm, \[\|\varphi_1 \otimes \dots \otimes \varphi_n\|' \leq \|\varphi_1\|_{\rdual{A_1}}\dots \|\varphi_n\|_{\rdual{A_n}} \leq 1 \]
By definition of $\|-\|'$ this means that
\[\sup\limits_{\|w\| \leq 1} |(\varphi_1 \otimes \dots \otimes \varphi_n)(w)| \leq 1 \]
But then since $\|\frac{u}{\|u\|}\| \leq 1$,
\[|(\varphi_1 \otimes \dots \otimes \varphi_n)(\frac{u}{\|u\|})| \leq 1\]
And so, \[|(\varphi_1 \otimes \dots \otimes \varphi_n)(u)| \leq \|u\|\]
Since this is true for any $\|\varphi_i\|_{\rdual{A_i}}$ it is true on the sup and $\|u\|_{A_1 \parr \dots \parr A_n} \leq \|u\|$.
\end{proof}

\begin{theorem}
  $\FBanc$ is a $\ast$-representable polycategory with tensor, par and duality defined above.
\end{theorem}
\begin{proof}
Let us first prove that the projective norm defines a tensor product.
Given a contractive polylinear map $f \colon \Gamma_1, A_1, \dots, A_n, \Gamma_2 \to \Delta$ it factors uniquely as a polylinear map $f/m \colon \Gamma_1, A_1 \otimes \dots \otimes A_n, \Gamma_2 \to \Delta$.
It then suffices to show that $f/m$ is contractive.
That is we want to prove that
\[|(\overrightarrow{\psi})f/m(\overrightarrow{x_1},u,\overrightarrow{x_2})| \leq 1\]
when all the vectors and functionals considered are of norm lesser than 1.
Consider a decomposition $u = \sum_j a_1^j \otimes \dots \otimes a_n^j$, then, by definition,
\[|(\overrightarrow{\psi})f/m(\overrightarrow{x_1},\sum_j a_1^j \otimes \dots \otimes a_n^j,\overrightarrow{x_2})| = |\sum_j(\overrightarrow{\psi})f(\overrightarrow{x_1},a_1^j,\dots,a_n^j,\overrightarrow{x_2})|\]
Using the properties of the norm we get
\[|(\overrightarrow{\psi})f/m(\overrightarrow{x_1},\sum_j a_1^j \otimes \dots \otimes a_n^j,\overrightarrow{x_2})| \leq \sum_j\|a_1^j\|_{A_1}\dots\|a_n^j\|_{A_n}|(\overrightarrow{\psi})f(\overrightarrow{x_1},\frac{a_1^j}{\|a_1^j\|_{A_1}},\dots,\frac{a_n^j}{\|a_n^j\|_{A_n}},\overrightarrow{x_2})|\]
Now since $\frac{a_i^j}{\|a_i^j\|_{A_i}} \leq 1$ and $f$ is contractive we get:
\[|(\overrightarrow{\psi})f/m(\overrightarrow{x_1},\sum_j a_1^j \otimes \dots \otimes a_n^j,\overrightarrow{x_2})| \leq \sum_j\|a_1^j\|_{A_1}\dots\|a_n^j\|_{A_n}\]
Since this holds for any decomposition of $u$, by definition of the projective norm:
\[|(\overrightarrow{\psi})f/m(\overrightarrow{x_1},\sum_j a_1^j \otimes \dots \otimes a_n^j,\overrightarrow{x_2})| \leq \|u\|_{A_1 \otimes \dots \otimes A_n} \leq 1\]
As usual, we assumed that all the vectors are non-zero.
If one is $0$ both sides are too and the inequality holds.

The proof that $\parr$ is universal is similar.
For the dual, any polylinear map $f \colon \Gamma \to A, \Delta$ factorises uniquely into a polylinear map $f/\rcap_A \colon \rdual{A}, \Gamma \to \Delta$ where \[(\overrightarrow{\psi})f/\rcap_A(\varphi, \overrightarrow{x}) = (\varphi,\overrightarrow{\psi})f(\overrightarrow{x})\]
From the definition it is straightforward that if $f$ is contractive then $f/\rcap_A$ is too.
\end{proof}

\begin{remark}
  More than just a model of classical MLL, $\FBanc$ is a model of classical MALL.
  The additive connectives are given by the vector space $A \oplus B$ with the norms $\|(a,b)\|_1 := \sum\limits_i \|a\|_A + \|b\|_B$ and $\|(a,b)\|_\infty := \max(\|a\|_A,\|b\|_B)$.
  These norms are extremal among the $p$-norms.
\end{remark}

After introducing bifibrations of polycategories in next chapter, we will reexamined the example of \FBanc{} from a fibrational view in section \ref{sec:Banach2}.

\chapter{Bicategories, double categories and virtual double categories}

This section is a review of bicategories, double categories and virtual double categories.
Nothing in it is new.
These concepts were introduced in the sixties and seventies, bicategories by B\'{e}nabou \cite{Benabou1967}, double categories by Ehresmann \cite{Ehresmann1963} and virtual double categories by Burroni \cite{Burroni1971} under the name multicat\'{e}gories.
A recent textbook by Johnson and Yau \cite{JohnsonYau2020} offers an extensive review of the theory of bicategories and of double categories.
Modern accounts of virtual double categories are given in \cite{Leinster2004} and in \cite{CruttwellShulman2009}.

\section{2-categories as categories enriched in $\Cat$}

In order to generalise the notion of category, one can reformulate it as follows.
A (locally small) category is given by
\begin{itemize}
\item a collection of objects $\Ob{\cC}$
\item for any two objects $A,B$ a set $\cC(A,B)$ of morphisms
\item for any object $A$, an identity function $\id_A \colon \one \to \cC(A,A)$ from the terminal set
\item for any objects $A,B,C$, a composition function $- \circ - \colon \cC(B,C) \times \cC(A,B) \to \cC(A,C)$
\end{itemize}
such that the following diagrams commute:
% https://q.uiver.app/?q=WzAsNSxbMCwwLCIoXFxjQyhDLEQpIFxcdGltZXMgXFxjQyhCLEMpKSBcXHRpbWVzIFxcY0MoQSxCKSJdLFsyLDAsIlxcY0MoQixEKSBcXHRpbWVzIFxcY0MoQSxCKSJdLFs0LDEsIlxcY0MoQSxEKSJdLFswLDIsIlxcY0MoQyxEKSBcXHRpbWVzIChcXGNDKEIsQykgXFx0aW1lcyBcXGNDKEEsQikpIl0sWzIsMiwiXFxjQyhDLEQpIFxcdGltZXMgXFxjQyhBLEMpIl0sWzAsMSwiLSBcXGNpcmMgLSBcXHRpbWVzIDFfe1xcY0MoQSxCKX0iXSxbMSwyLCItXFxjaXJjLSJdLFswLDMsIlxcc2ltZXEiLDJdLFszLDQsIjFfe1xcY0MoQyxEKX1cXHRpbWVzIC0gXFxjaXJjIC0iLDJdLFs0LDIsIi1cXGNpcmMtIiwyXV0=
\[\begin{tikzcd}
	{(\cC(C,D) \times \cC(B,C)) \times \cC(A,B)} && {\cC(B,D) \times \cC(A,B)} \\
	&&&& {\cC(A,D)} \\
	{\cC(C,D) \times (\cC(B,C) \times \cC(A,B))} && {\cC(C,D) \times \cC(A,C)}
	\arrow["{- \circ - \times 1_{\cC(A,B)}}", from=1-1, to=1-3]
	\arrow["{-\circ-}", from=1-3, to=2-5]
	\arrow["\simeq"', from=1-1, to=3-1]
	\arrow["{1_{\cC(C,D)}\times - \circ -}"', from=3-1, to=3-3]
	\arrow["{-\circ-}"', from=3-3, to=2-5]
\end{tikzcd}\]
% https://q.uiver.app/?q=WzAsMyxbMCwwLCJcXGNDKEEsQikgXFx0aW1lcyBcXG9uZSJdLFsyLDAsIlxcY0MoQSxCKSBcXHRpbWVzIFxcY0MoQSxBKSJdLFsyLDEsIlxcY0MoQSxCKSJdLFswLDEsIjFfe1xcY0MoQSxCKSB9IFxcdGltZXMgXFxpZF9BIl0sWzEsMiwiLVxcY2lyYy0iXSxbMCwyLCJcXHNpbWVxIiwyXV0=
\[\begin{tikzcd}
	{\cC(A,B) \times \one} && {\cC(A,B) \times \cC(A,A)} \\
	&& {\cC(A,B)}
	\arrow["{1_{\cC(A,B) } \times \id_A}", from=1-1, to=1-3]
	\arrow["{-\circ-}", from=1-3, to=2-3]
	\arrow["\simeq"', from=1-1, to=2-3]
\end{tikzcd}
% https://q.uiver.app/?q=WzAsMyxbMCwwLCJcXG9uZSBcXHRpbWVzIFxcY0MoQSxCKSJdLFsyLDAsIlxcY0MoQixCKSBcXHRpbWVzIFxcY0MoQSxCKSJdLFsyLDEsIlxcY0MoQSxCKSJdLFswLDEsIlxcaWRfQiBcXHRpbWVzIDFfe1xcY0MoQSxCKX0iXSxbMSwyLCItXFxjaXJjLSJdLFswLDIsIlxcc2ltZXEiLDJdXQ==
\begin{tikzcd}
	{\one \times \cC(A,B)} && {\cC(B,B) \times \cC(A,B)} \\
	&& {\cC(A,B)}
	\arrow["{\id_B \times 1_{\cC(A,B)}}", from=1-1, to=1-3]
	\arrow["{-\circ-}", from=1-3, to=2-3]
	\arrow["\simeq"', from=1-1, to=2-3]
\end{tikzcd}\]

From this definition one can replace the set of morphisms $\cC(A,B)$ and the identity and composition functions by objects and morphisms in another category $\cV$.
In order to state this definition one sufficient condition on $\cV$ is that it is monoidal.
This gives the notion of category enriched in $\cV$, or $\cV$-category.
In particular, one can consider categories enriched in $\Cat$.
These are called $2$-categories.
If we unravel the definition we get the following.

\begin{definition}
A \defin{2-category} \cC is given by
\begin{itemize}
\item a collection of object $\Ob{\cC}$
\item for any object $A,B$, a category $\cC(A,B)$ of morphisms, i.e.
\begin{itemize}
\item a collection of morphisms $f \colon A \to B$
\item for any pair of morphisms $f, g \colon A \to B$, a collection of 2-morphisms $\cC(f,g)$ written $\alpha \colon f \Rightarrow g$
\item an identity 2-morphism $\id_f \colon f \Rightarrow f$
\item for any $f, g, h \colon A \to B$ and any 2-morphisms $\alpha \colon f \Rightarrow g$ and $\beta \colon g \to h$ a vertical composition $\beta \bullet \alpha \colon f \Rightarrow h$
\end{itemize}
such that
\begin{itemize}
\item $\alpha \bullet \id_f = \alpha = \id_g \bullet \alpha$
\item $(\gamma \bullet \beta) \bullet \alpha = \gamma \bullet (\beta \bullet \alpha)$
\end{itemize}
\item for any $A$, an identity functor $\id_A \colon \one \to \cC(A,A)$, i.e.
\begin{itemize}
\item a morphism $\id_A \colon A \to A$
\end{itemize}
\item a composition functor $\circ \colon \cC(B,C) \times \cC(A,B) \to \cC(A,C)$, i.e.
\begin{itemize}
\item for any morphisms $g \colon B \to C$ and $f \colon A \to B$, a morphism $g \circ f \colon A \to C$
\item for any morphisms $g, g' \colon B \to C$ and $f, f' \colon A \to B$, and any 2-morphisms $\alpha \colon g \Rightarrow g'$ and $\beta \colon g \Rightarrow f'$, a horizontal composition $\beta \circ \alpha \colon g \circ f \Rightarrow g' \circ f'$
\end{itemize}
such that
\begin{itemize}
\item $\id_g \circ \id_f = \id_{g \circ f}$
\item $(\beta' \bullet \beta) \circ (\alpha' \bullet \alpha) = (\beta' \circ \alpha') \bullet (\beta \circ \alpha)$
\end{itemize}
such that the following diagrams commutes:
% https://q.uiver.app/?q=WzAsNSxbMCwwLCIoXFxjQyhDLEQpIFxcdGltZXMgXFxjQyhCLEMpKSBcXHRpbWVzIFxcY0MoQSxCKSJdLFsyLDAsIlxcY0MoQixEKSBcXHRpbWVzIFxcY0MoQSxCKSJdLFs0LDEsIlxcY0MoQSxEKSJdLFswLDIsIlxcY0MoQyxEKSBcXHRpbWVzIChcXGNDKEIsQykgXFx0aW1lcyBcXGNDKEEsQikpIl0sWzIsMiwiXFxjQyhDLEQpIFxcdGltZXMgXFxjQyhBLEMpIl0sWzAsMSwiLSBcXGNpcmMgLSBcXHRpbWVzIDFfe1xcY0MoQSxCKX0iXSxbMSwyLCItXFxjaXJjLSJdLFswLDMsIlxcc2ltZXEiLDJdLFszLDQsIjFfe1xcY0MoQyxEKX1cXHRpbWVzIC0gXFxjaXJjIC0iLDJdLFs0LDIsIi1cXGNpcmMtIiwyXV0=
\item \[\begin{tikzcd}
	{(\cC(C,D) \times \cC(B,C)) \times \cC(A,B)} && {\cC(B,D) \times \cC(A,B)} \\
	&&&& {\cC(A,D)} \\
	{\cC(C,D) \times (\cC(B,C) \times \cC(A,B))} && {\cC(C,D) \times \cC(A,C)}
	\arrow["{- \circ - \times 1_{\cC(A,B)}}", from=1-1, to=1-3]
	\arrow["{-\circ-}", from=1-3, to=2-5]
	\arrow["\simeq"', from=1-1, to=3-1]
	\arrow["{1_{\cC(C,D)}\times - \circ -}"', from=3-1, to=3-3]
	\arrow["{-\circ-}"', from=3-3, to=2-5]
\end{tikzcd}\]
i.e.
\begin{itemize}
\item $(h \circ g) \circ f = h \circ (g \circ f)$
\item $(\gamma \circ \beta) \circ \alpha = \gamma \circ (\beta \circ \alpha)$
\end{itemize}
% https://q.uiver.app/?q=WzAsMyxbMCwwLCJcXGNDKEEsQikgXFx0aW1lcyBcXG9uZSJdLFsyLDAsIlxcY0MoQSxCKSBcXHRpbWVzIFxcY0MoQSxBKSJdLFsyLDEsIlxcY0MoQSxCKSJdLFswLDEsIjFfe1xcY0MoQSxCKSB9IFxcdGltZXMgXFxpZF9BIl0sWzEsMiwiLVxcY2lyYy0iXSxbMCwyLCJcXHNpbWVxIiwyXV0=
\item \[\begin{tikzcd}
	{\cC(A,B) \times \one} && {\cC(A,B) \times \cC(A,A)} \\
	&& {\cC(A,B)}
	\arrow["{1_{\cC(A,B) } \times \id_A}", from=1-1, to=1-3]
	\arrow["{-\circ-}", from=1-3, to=2-3]
	\arrow["\simeq"', from=1-1, to=2-3]
\end{tikzcd}
% https://q.uiver.app/?q=WzAsMyxbMCwwLCJcXG9uZSBcXHRpbWVzIFxcY0MoQSxCKSJdLFsyLDAsIlxcY0MoQixCKSBcXHRpbWVzIFxcY0MoQSxCKSJdLFsyLDEsIlxcY0MoQSxCKSJdLFswLDEsIlxcaWRfQiBcXHRpbWVzIDFfe1xcY0MoQSxCKX0iXSxbMSwyLCItXFxjaXJjLSJdLFswLDIsIlxcc2ltZXEiLDJdXQ==
\begin{tikzcd}
	{\one \times \cC(A,B)} && {\cC(B,B) \times \cC(A,B)} \\
	&& {\cC(A,B)}
	\arrow["{\id_B \times 1_{\cC(A,B)}}", from=1-1, to=1-3]
	\arrow["{-\circ-}", from=1-3, to=2-3]
	\arrow["\simeq"', from=1-1, to=2-3]
\end{tikzcd}\]
i.e.
\begin{itemize}
\item $f \circ \id_A = f = \id_B \circ f$
\end{itemize}
\end{itemize}
\end{definition}

\begin{definition}
A 2-category \cC is \defin{locally small} if for any $A,B \in \Ob{\cC}$, $\cC(A,B)$ is small.
\end{definition}

\begin{example}
The archetypal 2-category is $\Cat$, the 2-category of categories, functors and natural transformations.
\end{example}

\begin{example}
Any category can be seen as a 2-category whose category of morphisms is discrete.
\end{example}

\begin{example}
There is a 2-category $\Poset$ whose objects are posets, morphisms are monotone maps functions and there is a unique 2-morphism $\alpha \colon f \Rightarrow g$ between functions $f,g \colon A \to B$ iff $f \leq g$, where the order is defined pointwise: $f \leq g$ iff for all $x\in A$, $f(x) \leq g(x)$.
\end{example}

\begin{example}
There is a 2-category $\Rel$ whose objects are sets, morphisms are relations and there is a unique 2-morphism $\alpha \colon R \Rightarrow S$ iff $R \subseteq S$ as subsets of $A\times B$, or equivalently iff for any $(a,b)\in A \times B$, $a R b \Rightarrow a S b$.
\end{example}

A category can be seen as an many-object generalisation of a monoid.
That is, any monoid $M$ gives a one-object category $\cB(M)$ called the delooping of $M$.
And conversely, any one-object (locally small) category is a monoid, since for any object $A$ of a category, the set of its endomorphisms forms a monoid under composition.
A similar relationship holds for 2-categories and strict monoidal categories.

\begin{example}
For any object $A$ in a 2-category \cC, $(\cC(A,A),\circ,\id_A)$ forms a strict monoidal category, i.e. a monoidal category where associativity and unitality holds strictly.
Reciprocally, any strict monoidal category gives a one-object 2-category called its delooping.
\end{example}

A way to categorify the notion of relation between sets is to consider distributors between categories.
A distributor $p \colon A \xto B$ between categories $A$ and $B$ is a functor $p \colon A^\op \times B \to \Set$.
Distributors compose via a coend formula.
One might expect that small categories and distributors for a 2-category $\Dist$ with 2-morphisms given by natural transformations between distributors.

However, it is not the case.
Let us fix a category $A$ and consider the endo-distributors on $A$.
It defines a category $\Dist_A := \Cat(A^\op \times A, \Set)$ whose:
\begin{itemize}
\item objects are distributors $P \colon A \xto A$
\item morphisms are natural transformations
\end{itemize}
Now the identity and composition in $\Dist$ are given by the homfunctor $\id_A := A(-,-) \colon A \xto A$ and given $P \colon A \xto B$ and $P' \colon B \to C$, by the coend $(P' \circ P)(a,c) = \int^{b \in B} P(a,b) \times P'(b,c)$.
This makes $\Dist_A$ into a monoidal category.
However, it is not strict since the coend is a colimit defined up to natural isomorphism.
So $\Dist$ does not form a 2-category.

\section{Bicategories}

In order to get examples such as $\Dist$ where the identity morphisms and composition of morphisms does not hold on the nose but up to 2-isomorphism, one can relax the condition that the pentagon and triangle diagrams above commute and ask for 2-isomorphisms between those.

\begin{definition}
A \defin{bicategory} \cC is given by the data of:
\begin{itemize}
\item a collection of objects $\Ob{\cC}$
\item for any objects $A,B$, a category $\cC(A,B)$
\item identity functors $\id_A \colon \one \to \cC(A,A)$
\item horizontal composition functors $-\circ_{A,B,C} - \colon \cC(B,C) \times \cC(A,B) \to \cC(A,C)$
\item natural isomorphisms $\mathrm{A}_{A,B,C,D} \colon (-\circ_{A,B,D}-)(- \circ_{B,C,D} - ) \times 1_{\cC(A,B)} \simeq (- \circ_{A,C,D} -)(1_{\cC(C,D)} \times (- \circ_{A,B,C} -)\alpha_{\cC(C,D),\cC(B,C),\cC(A,B)}$, i.e. for any morphisms $f \colon A \to B$, $g \colon B \to C$ and $h \colon C \to D$ an invertible 2-morphism $A_{f,g,h} \colon (h \circ g) \circ f \Rightarrow h \circ (g \circ f)$ such that for any 2-morphisms $\alpha \colon f \Rightarrow f'$, $\beta \colon g \Rightarrow g'$ and $\gamma \colon h \Rightarrow h'$ the following diagram commutes:
% https://q.uiver.app/?q=WzAsNCxbMCwwLCIoaCBcXGNpcmMgZykgXFxjaXJjIGYiXSxbMiwwLCJoIFxcY2lyYyAoZyBcXGNpcmMgZikiXSxbMiwyLCJoJyBcXGNpcmMgKGcnIFxcY2lyYyBmJykiXSxbMCwyLCIoaCcgXFxjaXJjIGcnKSBcXGNpcmMgZiciXSxbMCwxLCJcXG1hdGhybXtBfV97ZixnLGh9IiwwLHsibGV2ZWwiOjJ9XSxbMSwyLCJcXGdhbW1hIFxcY2lyYyAoXFxiZXRhIFxcY2lyYyBcXGFscGhhKSIsMCx7ImxldmVsIjoyfV0sWzAsMywiKFxcZ2FtbWEgXFxjaXJjIFxcYmV0YSkgXFxjaXJjIFxcYWxwaGEiLDIseyJsZXZlbCI6Mn1dLFszLDIsIkFfe2YnLGcnLGgnfSIsMix7ImxldmVsIjoyfV1d
\[\begin{tikzcd}
	{(h \circ g) \circ f} && {h \circ (g \circ f)} \\
	\\
	{(h' \circ g') \circ f'} && {h' \circ (g' \circ f')}
	\arrow["{\mathrm{A}_{f,g,h}}", Rightarrow, from=1-1, to=1-3]
	\arrow["{\gamma \circ (\beta \circ \alpha)}", Rightarrow, from=1-3, to=3-3]
	\arrow["{(\gamma \circ \beta) \circ \alpha}"', Rightarrow, from=1-1, to=3-1]
	\arrow["{A_{f',g',h'}}"', Rightarrow, from=3-1, to=3-3]
\end{tikzcd}\]
\item natural isomorphisms $\mathrm{R}_{A,B} \colon (- \circ_{A,A,B} -)(1_{\cC(A,B)} \times \id_A) \to \rho_{\cC(A,B)}$ and $\Lambda_{A,B} \colon (-\circ_{A,B,B} -)(\id_B \times 1_{\cC(A,B)}) \to \lambda_{\cC(A,B)}$, i.e. invertible 2-morphisms $R_f \colon f \circ \id_A \Rightarrow f$ and $\Lambda_f \colon \id_B \circ f \Rightarrow f$ such that for any 2-morphism $\alpha \colon f \to f'$ the following diagrams commute:
% https://q.uiver.app/?q=WzAsNSxbMCwwLCJmIFxcY2lyYyBcXGlkX0EiXSxbMiwwLCJmIl0sWzIsMiwiZiciXSxbMCwyLCJmJyBcXGNpcmMgXFxpZF9BIl0sWzIsMV0sWzAsMSwiXFxtYXRocm17Un1fe2Z9IiwwLHsibGV2ZWwiOjJ9XSxbMSwyLCJcXGFscGhhIiwwLHsibGV2ZWwiOjJ9XSxbMCwzLCJcXGFscGhhIFxcY2lyYyBcXGlkX0EiLDIseyJsZXZlbCI6Mn1dLFszLDIsIlxcbWF0aHJte1J9X3tmJ30iLDIseyJsZXZlbCI6Mn1dXQ==
\[\begin{tikzcd}
	{f \circ \id_A} && f \\
	&& {} \\
	{f' \circ \id_A} && {f'}
	\arrow["{\mathrm{R}_{f}}", Rightarrow, from=1-1, to=1-3]
	\arrow["\alpha", Rightarrow, from=1-3, to=3-3]
	\arrow["{\alpha \circ \id_A}"', Rightarrow, from=1-1, to=3-1]
	\arrow["{\mathrm{R}_{f'}}"', Rightarrow, from=3-1, to=3-3]
\end{tikzcd}
% https://q.uiver.app/?q=WzAsNSxbMCwwLCJcXGlkX0IgXFxjaXJjIGYiXSxbMiwwLCJmIl0sWzIsMiwiZiciXSxbMCwyLCJcXGlkX0IgXFxjaXJjIGYnIl0sWzIsMV0sWzAsMSwiXFxMYW1iZGFfZiIsMCx7ImxldmVsIjoyfV0sWzEsMiwiXFxhbHBoYSIsMCx7ImxldmVsIjoyfV0sWzAsMywiXFxpZF9CIFxcY2lyYyBcXGFscGhhIiwyLHsibGV2ZWwiOjJ9XSxbMywyLCJcXExhbWJkYV97Zid9IiwyLHsibGV2ZWwiOjJ9XV0=
\begin{tikzcd}
	{\id_B \circ f} && f \\
	&& {} \\
	{\id_B \circ f'} && {f'}
	\arrow["{\Lambda_f}", Rightarrow, from=1-1, to=1-3]
	\arrow["\alpha", Rightarrow, from=1-3, to=3-3]
	\arrow["{\id_B \circ \alpha}"', Rightarrow, from=1-1, to=3-1]
	\arrow["{\Lambda_{f'}}"', Rightarrow, from=3-1, to=3-3]
\end{tikzcd}\]

where the composition of functors $F,G$ is written as concatenation $GF$ and $\alpha, \rho, \lambda$ are the associator and unitors of the monoidal category $(\Cat,\times, \one)$ 
\end{itemize}
such that the following diagram commutes
% https://q.uiver.app/?q=WzAsMyxbMCwwLCIoZyBcXGNpcmMgXFxpZF9CKSBcXGNpcmMgZiJdLFsyLDAsImcgXFxjaXJjIChcXGlkX0IgXFxjaXJjIGYpIl0sWzIsMiwiZyBcXGNpcmMgZiJdLFswLDEsIkFfe2YsXFxpZF9CLGd9IiwwLHsibGV2ZWwiOjJ9XSxbMSwyLCJcXGlkX2cgXFxjaXJjIFxcTGFtYmRhX2YiLDAseyJsZXZlbCI6Mn1dLFswLDIsIlxcbWF0aHJte1J9X2cgXFxjaXJjIFxcaWRfZiIsMix7ImxldmVsIjoyfV1d
\[\begin{tikzcd}
	{(g \circ \id_B) \circ f} && {g \circ (\id_B \circ f)} \\
	\\
	&& {g \circ f}
	\arrow["{A_{f,\id_B,g}}", Rightarrow, from=1-1, to=1-3]
	\arrow["{\id_g \circ \Lambda_f}", Rightarrow, from=1-3, to=3-3]
	\arrow["{\mathrm{R}_g \circ \id_f}"', Rightarrow, from=1-1, to=3-3]
\end{tikzcd}\]
% https://q.uiver.app/?q=WzAsNSxbMCwwLCIoKGkgXFxjaXJjIGgpIFxcY2lyYyBnKSBcXGNpcmMgZiJdLFsyLDAsIihpIFxcY2lyYyBoKSBcXGNpcmMgKGcgXFxjaXJjIGYpIl0sWzQsMSwiaSBcXGNpcmMgKGggXFxjaXJjIChnIFxcY2lyYyBmKSkiXSxbMCwyLCIoaSBcXGNpcmMgKGggXFxjaXJjIGcpKSBcXGNpcmMgZiJdLFsyLDIsImkgXFxjaXJjICgoaCBcXGNpcmMgZykgXFxjaXJjIGYpIl0sWzAsMSwiXFxtYXRocm17QX1fe2YsZyxpXFxjaXJjIGh9IiwwLHsibGV2ZWwiOjJ9XSxbMSwyLCJcXG1hdGhybXtBfV97Z1xcY2lyYyBmLGgsaX0iLDAseyJsZXZlbCI6Mn1dLFswLDMsIlxcbWF0aHJte0F9X3tnLGgsaX0gXFxjaXJjIFxcaWRfZiIsMix7ImxldmVsIjoyLCJzdHlsZSI6eyJoZWFkIjp7Im5hbWUiOiJub25lIn19fV0sWzMsNCwiXFxtYXRocm17QX1fe2YsaFxcY2lyYyBnLGl9IiwyLHsibGV2ZWwiOjJ9XSxbNCwyLCJcXGlkX2kgXFxjaXJjXFxtYXRocm17QX1fe2YsZyxofSIsMix7ImxldmVsIjoyfV1d
\[\begin{tikzcd}
	{((i \circ h) \circ g) \circ f} && {(i \circ h) \circ (g \circ f)} \\
	&&&& {i \circ (h \circ (g \circ f))} \\
	{(i \circ (h \circ g)) \circ f} && {i \circ ((h \circ g) \circ f)}
	\arrow["{\mathrm{A}_{f,g,i\circ h}}", Rightarrow, from=1-1, to=1-3]
	\arrow["{\mathrm{A}_{g\circ f,h,i}}", Rightarrow, from=1-3, to=2-5]
	\arrow["{\mathrm{A}_{g,h,i} \circ \id_f}"', Rightarrow, no head, from=1-1, to=3-1]
	\arrow["{\mathrm{A}_{f,h\circ g,i}}"', Rightarrow, from=3-1, to=3-3]
	\arrow["{\id_i \circ\mathrm{A}_{f,g,h}}"', Rightarrow, from=3-3, to=2-5]
\end{tikzcd}\]
\end{definition}

\begin{example}
Any strict 2-category is a bicategory for which the natural isomorphisms $\mathrm{A}$, $\mathrm{R}$ and $\Lambda$ are all equalities.
Reciprocally, if all of those natural isomorphisms are equalities then the bicategory is a strict 2-category.
\end{example}

\begin{example}
Monoidal categories correspond to one-object bicategories through their delooping.
\end{example}

\begin{example}
$\Dist$ is a bicategory.
\end{example}

\begin{definition}
A \defin{(lax) functor} between bicategories $F \colon \cC \to \cD$ is the data of:
\begin{itemize}
\item an assignment for each object $A \in \Ob{\cC}$ of an object $F(A) \in \Ob{\cD}$
\item for each object $A,B$ a functor $F \colon \cC(A,B) \to \cD(F(A),F(B))$
\item for each object $A,B,C$ a natural transformation:
% https://q.uiver.app/?q=WzAsNCxbMCwwLCJcXGNDKEIsQykgXFx0aW1lcyBcXGNDKEEsQikiXSxbMiwwLCJcXGNEKEYoQiksRihDKSkgXFx0aW1lcyBcXGNEKEYoQSksRihCKSkiXSxbMCwyLCJcXGNDKEEsQykiXSxbMiwyLCJcXGNEKEYoQSksIEYoQykpIl0sWzAsMSwiRiBcXHRpbWVzIEYiXSxbMSwzLCItXFxjaXJjLSJdLFswLDIsIi1cXGNpcmMtIiwyXSxbMiwzLCJGIiwyXSxbNCw3LCJGXzIiLDAseyJzaG9ydGVuIjp7InNvdXJjZSI6MjAsInRhcmdldCI6MjB9fV1d
\[\begin{tikzcd}
	{\cC(B,C) \times \cC(A,B)} && {\cD(F(B),F(C)) \times \cD(F(A),F(B))} \\
	\\
	{\cC(A,C)} && {\cD(F(A), F(C))}
	\arrow[""{name=0, anchor=center, inner sep=0}, "{F \times F}", from=1-1, to=1-3]
	\arrow["{-\circ-}", from=1-3, to=3-3]
	\arrow["{-\circ-}"', from=1-1, to=3-1]
	\arrow[""{name=1, anchor=center, inner sep=0}, "F"', from=3-1, to=3-3]
	\arrow["{F_2}", shorten <=9pt, shorten >=9pt, Rightarrow, from=0, to=1]
\end{tikzcd}\]
i.e.
a family of 2-morphisms $F_2(f,g) \colon F(g) \circ F(f) \Rightarrow F(g \circ f)$ such that for any 2-morphisms $\alpha \colon f \Rightarrow f'$ and $\beta \colon g \Rightarrow g'$, the following square commutes:
% https://q.uiver.app/?q=WzAsNCxbMCwwLCJGKGcpXFxjaXJjIEYoZikiXSxbMiwwLCJGKGcgXFxjaXJjIGYpIl0sWzIsMiwiRihnJyBcXGNpcmMgZicpIl0sWzAsMiwiRihnJykgXFxjaXJjIEYoZicpIl0sWzAsMSwiRl8yKGYsZykiLDAseyJsZXZlbCI6Mn1dLFsxLDIsIkYoXFxiZXRhIFxcY2lyYyBcXGFscGhhKSIsMCx7ImxldmVsIjoyfV0sWzAsMywiRihcXGJldGEpXFxjaXJjIEYoXFxhbHBoYSkiLDIseyJsZXZlbCI6Mn1dLFszLDIsIkZfMihmJyxnJykiLDIseyJsZXZlbCI6Mn1dXQ==
\[\begin{tikzcd}
	{F(g)\circ F(f)} && {F(g \circ f)} \\
	\\
	{F(g') \circ F(f')} && {F(g' \circ f')}
	\arrow["{F_2(f,g)}", Rightarrow, from=1-1, to=1-3]
	\arrow["{F(\beta \circ \alpha)}", Rightarrow, from=1-3, to=3-3]
	\arrow["{F(\beta)\circ F(\alpha)}"', Rightarrow, from=1-1, to=3-1]
	\arrow["{F_2(f',g')}"', Rightarrow, from=3-1, to=3-3]
\end{tikzcd}\]
\item for any object A, a natural transformation:
% https://q.uiver.app/?q=WzAsMyxbMCwwLCJcXG9uZSJdLFsyLDAsIlxcY0QoRihBKSxGKEEpKSJdLFsxLDIsIlxcY0MoQSxBKSJdLFswLDIsIlxcaWRfQSIsMl0sWzIsMSwiRiIsMl0sWzAsMSwiXFxpZF97RihBKX0iXSxbNSwyLCJGXzAiLDAseyJzaG9ydGVuIjp7InNvdXJjZSI6MjB9fV1d
\[\begin{tikzcd}
	\one && {\cD(F(A),F(A))} \\
	\\
	& {\cC(A,A)}
	\arrow["{\id_A}"', from=1-1, to=3-2]
	\arrow["F"', from=3-2, to=1-3]
	\arrow[""{name=0, anchor=center, inner sep=0}, "{\id_{F(A)}}", from=1-1, to=1-3]
	\arrow["{F_0}", shorten <=7pt, Rightarrow, from=0, to=3-2]
\end{tikzcd}\]
i.e. a 2-morphism $F_0 \colon \id_{F(A)} \Rightarrow F(\id_A)$ 
\end{itemize}  
such that the following diagrams commute:
\begin{itemize}
\item % https://q.uiver.app/?q=WzAsNixbMSwwLCIoRihoKSBcXGNpcmMgRihnKSkgXFxjaXJjIEYoZikiXSxbMywxLCJGKGggXFxjaXJjIGcpIFxcY2lyYyBGKGYpIl0sWzMsMiwiRigoaCBcXGNpcmMgZykgXFxjaXJjIGYpIl0sWzAsMSwiRihoKSBcXGNpcmMgKEYoZykgXFxjaXJjIEYoZikpIl0sWzAsMiwiRihoKSBcXGNpcmMgRihnIFxcY2lyYyBmKSJdLFsxLDMsIkYoaCBcXGNpcmMgKGcgXFxjaXJjIGYpKSJdLFswLDEsIkZfMihnLGgpIFxcY2lyYyBcXGlkX3tGKGYpfSIsMCx7ImxldmVsIjoyfV0sWzEsMiwiRl8yKGYsaCBcXGNpcmMgZykiLDAseyJsZXZlbCI6Mn1dLFswLDMsIlxcbWF0aHJte0F9X3tGKGYpLEYoZyksRihoKX0iLDIseyJsZXZlbCI6Mn1dLFszLDQsIlxcaWRfe0YoaCl9XFxjaXJjIEZfMihmLGcpIiwyLHsibGV2ZWwiOjJ9XSxbNCw1LCJGXzIoZ1xcY2lyYyBmLGgpIiwyLHsibGV2ZWwiOjJ9XSxbMiw1LCJGKFxcbWF0aHJte0F9X3tmLGcsaH0pIiwwLHsibGV2ZWwiOjJ9XV0=
\[\begin{tikzcd}
	& {(F(h) \circ F(g)) \circ F(f)} \\
	{F(h) \circ (F(g) \circ F(f))} &&& {F(h \circ g) \circ F(f)} \\
	{F(h) \circ F(g \circ f)} &&& {F((h \circ g) \circ f)} \\
	& {F(h \circ (g \circ f))}
	\arrow["{F_2(g,h) \circ \id_{F(f)}}", Rightarrow, from=1-2, to=2-4]
	\arrow["{F_2(f,h \circ g)}", Rightarrow, from=2-4, to=3-4]
	\arrow["{\mathrm{A}_{F(f),F(g),F(h)}}"', Rightarrow, from=1-2, to=2-1]
	\arrow["{\id_{F(h)}\circ F_2(f,g)}"', Rightarrow, from=2-1, to=3-1]
	\arrow["{F_2(g\circ f,h)}"', Rightarrow, from=3-1, to=4-2]
	\arrow["{F(\mathrm{A}_{f,g,h})}", Rightarrow, from=3-4, to=4-2]
\end{tikzcd}\]
\item % https://q.uiver.app/?q=WzAsOCxbMCwwLCJGKGYpIFxcY2lyYyBcXGlkX3tGKEEpfSJdLFsyLDAsIkYoZikgXFxjaXJjIEYoXFxpZF9BKSJdLFsyLDIsIkYoZiBcXGNpcmMgXFxpZF9BKSJdLFswLDIsIkYoZikiXSxbNCwwLCJcXGlkX3tGKEIpfVxcY2lyYyBGKGYpIl0sWzYsMCwiRihcXGlkX0IpIFxcY2lyYyBGKGYpIl0sWzYsMiwiRihcXGlkX0IgXFxjaXJjIGYpIl0sWzQsMiwiRihmKSJdLFsxLDIsIkZfMihcXGlkX0EsZikiXSxbMCwxLCJcXGlkX3tGKGYpfVxcY2lyYyBGXzAiXSxbMCwzLCJcXG1hdGhybXtSfV97RihmKX0iLDJdLFsyLDMsIkYoUl9mKSJdLFs0LDUsIkZfMCBcXGNpcmMgXFxpZF97RihmKX0iXSxbNSw2LCJGXzIoZixcXGlkX0IpIl0sWzYsNywiRihcXExhbWJkYV9mKSJdLFs0LDcsIlxcTGFtYmRhX3tGKGYpfSIsMl1d
\[\begin{tikzcd}[column sep=small]
	{F(f) \circ \id_{F(A)}} && {F(f) \circ F(\id_A)} && {\id_{F(B)}\circ F(f)} && {F(\id_B) \circ F(f)} \\
	\\
	{F(f)} && {F(f \circ \id_A)} && {F(f)} && {F(\id_B \circ f)}
	\arrow["{F_2(\id_A,f)}", from=1-3, to=3-3]
	\arrow["{\id_{F(f)}\circ F_0}", from=1-1, to=1-3]
	\arrow["{\mathrm{R}_{F(f)}}"', from=1-1, to=3-1]
	\arrow["{F(R_f)}", from=3-3, to=3-1]
	\arrow["{F_0 \circ \id_{F(f)}}", from=1-5, to=1-7]
	\arrow["{F_2(f,\id_B)}", from=1-7, to=3-7]
	\arrow["{F(\Lambda_f)}", from=3-7, to=3-5]
	\arrow["{\Lambda_{F(f)}}"', from=1-5, to=3-5]
\end{tikzcd}\]
\end{itemize}
\end{definition}

\begin{definition}
A functor $F$ between bicategories is called \defin{a pseudofunctor} if the $F_2$ and $F_0$ are natural isomorphisms and \defin{a strict functor} if they are identities.
\end{definition}

\begin{definition}

A \defin{(lax) natural transformation} $\alpha \colon F \Rightarrow G$ between parallel functors $F,G \colon \cC \to \cD$ is given by a family of morphisms $\alpha_A \colon F(A) \to G(A)$ and of 2-morphisms
% https://q.uiver.app/?q=WzAsNCxbMCwwLCJGKEEpIl0sWzAsMiwiRihCKSJdLFsyLDAsIkcoQSkiXSxbMiwyLCJHKEIpIl0sWzAsMSwiRihmKSIsMl0sWzEsMywiXFxhbHBoYV9CIiwyXSxbMCwyLCJcXGFscGhhX0EiXSxbMiwzLCJHKGYpIl0sWzYsNSwiXFxhbHBoYV9mIiwyLHsic2hvcnRlbiI6eyJzb3VyY2UiOjIwLCJ0YXJnZXQiOjIwfX1dXQ==
\[\begin{tikzcd}
	{F(A)} && {G(A)} \\
	\\
	{F(B)} && {G(B)}
	\arrow["{F(f)}"', from=1-1, to=3-1]
	\arrow[""{name=0, anchor=center, inner sep=0}, "{\alpha_B}"', from=3-1, to=3-3]
	\arrow[""{name=1, anchor=center, inner sep=0}, "{\alpha_A}", from=1-1, to=1-3]
	\arrow["{G(f)}", from=1-3, to=3-3]
	\arrow["{\alpha_f}"', shorten <=9pt, shorten >=9pt, Rightarrow, from=1, to=0]
\end{tikzcd}\]
such that the equalities holds:
% https://q.uiver.app/?q=WzAsOSxbMCwwLCJGKEEpIl0sWzAsNCwiRihBKSJdLFs0LDAsIkcoQSkiXSxbNCw0LCJHKEEpIl0sWzYsMiwiPSJdLFs4LDAsIkYoQSkiXSxbMTIsMCwiRyhBKSJdLFsxMiw0LCJHKEEpIl0sWzgsNCwiRihBKSJdLFswLDEsIkYoXFxpZF9BKSIsMl0sWzEsMywiXFxhbHBoYV9BIiwyXSxbMCwyLCJcXGFscGhhX0EiXSxbMiwzLCJHKFxcaWRfQSkiLDJdLFsyLDMsIlxcaWRfe0coQSl9IiwwLHsiY3VydmUiOi01fV0sWzUsNiwiXFxhbHBoYV9BIl0sWzUsOCwiRihcXGlkX0EpIiwyLHsiY3VydmUiOjV9XSxbOCw3LCJcXGFscGhhX0EiLDJdLFs2LDcsIlxcaWRfe0coQSl9Il0sWzUsOCwiXFxpZF97RihBKX0iXSxbNSw3LCJcXGFscGhhX0EiLDFdLFsxMSwxMCwiXFxhbHBoYV97XFxpZF9BfSIsMix7InNob3J0ZW4iOnsic291cmNlIjoyMCwidGFyZ2V0IjoyMH19XSxbMTMsMTIsIkdfMCIsMCx7InNob3J0ZW4iOnsic291cmNlIjoyMCwidGFyZ2V0IjoyMH19XSxbMTgsMTUsIkZfMCIsMCx7InNob3J0ZW4iOnsic291cmNlIjoyMCwidGFyZ2V0IjoyMH19XSxbNiwxOSwiXFxMYW1iZGFfe1xcYWxwaGFfQX0iLDEseyJzaG9ydGVuIjp7InRhcmdldCI6MjB9fV0sWzE5LDgsIlxcbWF0aHJte1J9X3tcXGFscGhhX0F9IiwxLHsic2hvcnRlbiI6eyJzb3VyY2UiOjIwfX1dXQ==
\[\begin{tikzcd}[column sep = small]
	{F(A)} &&&& {G(A)} &&&& {F(A)} &&&& {G(A)} \\
	\\
	&&&&&& {=} \\
	\\
	{F(A)} &&&& {G(A)} &&&& {F(A)} &&&& {G(A)}
	\arrow["{F(\id_A)}"', from=1-1, to=5-1]
	\arrow[""{name=0, anchor=center, inner sep=0}, "{\alpha_A}"', from=5-1, to=5-5]
	\arrow[""{name=1, anchor=center, inner sep=0}, "{\alpha_A}", from=1-1, to=1-5]
	\arrow[""{name=2, anchor=center, inner sep=0}, "{G(\id_A)}"', from=1-5, to=5-5]
	\arrow[""{name=3, anchor=center, inner sep=0}, "{\id_{G(A)}}", curve={height=-30pt}, from=1-5, to=5-5]
	\arrow["{\alpha_A}", from=1-9, to=1-13]
	\arrow[""{name=4, anchor=center, inner sep=0}, "{F(\id_A)}"', curve={height=30pt}, from=1-9, to=5-9]
	\arrow["{\alpha_A}"', from=5-9, to=5-13]
	\arrow["{\id_{G(A)}}", from=1-13, to=5-13]
	\arrow[""{name=5, anchor=center, inner sep=0}, "{\id_{F(A)}}", from=1-9, to=5-9]
	\arrow[""{name=6, anchor=center, inner sep=0}, "{\alpha_A}"{description}, from=1-9, to=5-13]
	\arrow["{\alpha_{\id_A}}"', shorten <=17pt, shorten >=17pt, Rightarrow, from=1, to=0]
	\arrow["{G_0}", shorten <=6pt, shorten >=6pt, Rightarrow, from=3, to=2]
	\arrow["{F_0}", shorten <=6pt, shorten >=6pt, Rightarrow, from=5, to=4]
	\arrow["{\Lambda_{\alpha_A}}"{description}, shorten >=13pt, Rightarrow, from=1-13, to=6]
	\arrow["{\mathrm{R}_{\alpha_A}}"{description}, shorten <=13pt, Rightarrow, from=6, to=5-9]
\end{tikzcd}\]

% https://q.uiver.app/?q=WzAsMTIsWzAsMCwiRihBKSJdLFsxLDIsIkYoQikiXSxbMiwwLCJHKEEpIl0sWzMsMiwiRyhCKSJdLFswLDQsIkYoQykiXSxbMiw0LCJHKEMpIl0sWzQsMiwiPSJdLFs2LDAsIkYoQSkiXSxbNiw0LCJGKEMpIl0sWzgsMCwiRyhBKSJdLFs4LDQsIkcoQykiXSxbOSwyLCJHKEIpIl0sWzAsMSwiRihmKSIsMV0sWzEsMywiXFxhbHBoYV9CIiwxXSxbMCwyLCJcXGFscGhhX0EiXSxbMiwzLCJHKGYpIl0sWzEsNCwiRihnKSIsMV0sWzMsNSwiRyhnKSJdLFs0LDUsIlxcYWxwaGFfQyIsMl0sWzAsNCwiRihnXFxjaXJjIGYpIiwyLHsiY3VydmUiOjV9XSxbNyw4LCJGKGdcXGNpcmMgZikiLDIseyJjdXJ2ZSI6NX1dLFs5LDEwLCJHKGdcXGNpcmMgZikiLDEseyJsYWJlbF9wb3NpdGlvbiI6NDAsImN1cnZlIjo1fV0sWzksMTEsIkcoZikiXSxbMTEsMTAsIkcoZykiXSxbNyw5LCJcXGFscGhhX0EiXSxbOCwxMCwiXFxhbHBoYV9DIiwyXSxbMTQsMTMsIlxcYWxwaGFfe2Z9IiwyLHsic2hvcnRlbiI6eyJzb3VyY2UiOjIwLCJ0YXJnZXQiOjIwfX1dLFsxMywxOCwiXFxhbHBoYV9nIiwyLHsic2hvcnRlbiI6eyJzb3VyY2UiOjIwLCJ0YXJnZXQiOjIwfX1dLFsxMSwyMSwiR18yIiwwLHsic2hvcnRlbiI6eyJ0YXJnZXQiOjIwfX1dLFsyNCwyNSwiXFxhbHBoYV97Z1xcY2lyYyBmfSIsMSx7ImN1cnZlIjo1LCJzaG9ydGVuIjp7InNvdXJjZSI6MjAsInRhcmdldCI6MjB9fV0sWzEsMTksIkZfMiIsMCx7InNob3J0ZW4iOnsidGFyZ2V0IjoyMH19XV0=
\resizebox{\hsize}{!}{
\begin{tikzcd}[ampersand replacement=\&]
	{F(A)} \&\& {G(A)} \&\&\&\& {F(A)} \&\& {G(A)} \\
	\\
	\& {F(B)} \&\& {G(B)} \& {=} \&\&\&\&\& {G(B)} \\
	\\
	{F(C)} \&\& {G(C)} \&\&\&\& {F(C)} \&\& {G(C)}
	\arrow["{F(f)}"{description}, from=1-1, to=3-2]
	\arrow[""{name=0, anchor=center, inner sep=0}, "{\alpha_B}"{description}, from=3-2, to=3-4]
	\arrow[""{name=1, anchor=center, inner sep=0}, "{\alpha_A}", from=1-1, to=1-3]
	\arrow["{G(f)}", from=1-3, to=3-4]
	\arrow["{F(g)}"{description}, from=3-2, to=5-1]
	\arrow["{G(g)}", from=3-4, to=5-3]
	\arrow[""{name=2, anchor=center, inner sep=0}, "{\alpha_C}"', from=5-1, to=5-3]
	\arrow[""{name=3, anchor=center, inner sep=0}, "{F(g\circ f)}"', curve={height=30pt}, from=1-1, to=5-1]
	\arrow["{F(g\circ f)}"', curve={height=30pt}, from=1-7, to=5-7]
	\arrow[""{name=4, anchor=center, inner sep=0}, "{G(g\circ f)}"{description, pos=0.4}, curve={height=30pt}, from=1-9, to=5-9]
	\arrow["{G(f)}", from=1-9, to=3-10]
	\arrow["{G(g)}", from=3-10, to=5-9]
	\arrow[""{name=5, anchor=center, inner sep=0}, "{\alpha_A}", from=1-7, to=1-9]
	\arrow[""{name=6, anchor=center, inner sep=0}, "{\alpha_C}"', from=5-7, to=5-9]
	\arrow["{\alpha_{f}}"', shorten <=10pt, shorten >=10pt, Rightarrow, from=1, to=0]
	\arrow["{\alpha_g}"', shorten <=10pt, shorten >=10pt, Rightarrow, from=0, to=2]
	\arrow["{G_2}", shorten >=11pt, Rightarrow, from=3-10, to=4]
	\arrow["{\alpha_{g\circ f}}"{description}, curve={height=30pt}, shorten <=19pt, shorten >=19pt, Rightarrow, from=5, to=6]
	\arrow["{F_2}", shorten >=11pt, Rightarrow, from=3-2, to=3]
\end{tikzcd}}

% https://q.uiver.app/?q=WzAsOSxbMCwwLCJGKEEpIl0sWzIsMCwiRihCKSJdLFswLDIsIkcoQSkiXSxbMiwyLCJHKEIpIl0sWzMsMSwiPSJdLFs0LDAsIkYoQSkiXSxbNCwyLCJHKEEpIl0sWzYsMiwiRyhCKSJdLFs2LDAsIkYoQikiXSxbMCwxLCJGKGYpIiwxLHsiY3VydmUiOjN9XSxbMCwxLCJGKGcpIiwxLHsiY3VydmUiOi0zfV0sWzAsMiwiXFxhbHBoYV9BIiwxXSxbMiwzLCJHKGYpIiwxXSxbMSwzLCJcXGFscGhhX0IiLDFdLFs2LDcsIkcoZykiLDEseyJjdXJ2ZSI6LTN9XSxbNSw4LCJGKGcpIiwxXSxbOCw3LCJcXGFscGhhX0IiLDFdLFs1LDYsIlxcYWxwaGFfQSIsMV0sWzYsNywiRyhmKSIsMSx7ImN1cnZlIjozfV0sWzksMTAsIkYoXFx0aGV0YSkiLDEseyJzaG9ydGVuIjp7InNvdXJjZSI6MjAsInRhcmdldCI6MjB9fV0sWzExLDEzLCJcXGFscGhhX2YiLDEseyJzaG9ydGVuIjp7InNvdXJjZSI6MjAsInRhcmdldCI6MjB9fV0sWzE3LDE2LCJcXGFscGhhX2ciLDEseyJzaG9ydGVuIjp7InNvdXJjZSI6MjAsInRhcmdldCI6MjB9fV0sWzE4LDE0LCJHKFxcdGhldGEpIiwxLHsic2hvcnRlbiI6eyJzb3VyY2UiOjIwLCJ0YXJnZXQiOjIwfX1dXQ==
\[\begin{tikzcd}[ampersand replacement=\&]
	{F(A)} \&\& {F(B)} \&\& {F(A)} \&\& {F(B)} \\
	\&\&\& {=} \\
	{G(A)} \&\& {G(B)} \&\& {G(A)} \&\& {G(B)}
	\arrow[""{name=0, anchor=center, inner sep=0}, "{F(f)}"{description}, curve={height=18pt}, from=1-1, to=1-3]
	\arrow[""{name=1, anchor=center, inner sep=0}, "{F(g)}"{description}, curve={height=-18pt}, from=1-1, to=1-3]
	\arrow[""{name=2, anchor=center, inner sep=0}, "{\alpha_A}"{description}, from=1-1, to=3-1]
	\arrow["{G(f)}"{description}, from=3-1, to=3-3]
	\arrow[""{name=3, anchor=center, inner sep=0}, "{\alpha_B}"{description}, from=1-3, to=3-3]
	\arrow[""{name=4, anchor=center, inner sep=0}, "{G(g)}"{description}, curve={height=-18pt}, from=3-5, to=3-7]
	\arrow["{F(g)}"{description}, from=1-5, to=1-7]
	\arrow[""{name=5, anchor=center, inner sep=0}, "{\alpha_B}"{description}, from=1-7, to=3-7]
	\arrow[""{name=6, anchor=center, inner sep=0}, "{\alpha_A}"{description}, from=1-5, to=3-5]
	\arrow[""{name=7, anchor=center, inner sep=0}, "{G(f)}"{description}, curve={height=18pt}, from=3-5, to=3-7]
	\arrow["{F(\theta)}"{description}, shorten <=5pt, shorten >=5pt, Rightarrow, from=0, to=1]
	\arrow["{\alpha_f}"{description}, shorten <=13pt, shorten >=13pt, Rightarrow, from=2, to=3]
	\arrow["{\alpha_g}"{description}, shorten <=13pt, shorten >=13pt, Rightarrow, from=6, to=5]
	\arrow["{G(\theta)}"{description}, shorten <=5pt, shorten >=5pt, Rightarrow, from=7, to=4]
\end{tikzcd}\]
\end{definition}

\begin{definition}
A \defin{pseudo-natural transformation} is a natural transformation where all the $\alpha_f$ are invertible.
A strict natural transformation is one where they are equalities.
\end{definition}

\begin{rk}
Although pseudo-natural transformations can be defined between arbitrary functors, strict natural transformations only make sense between strict functors.
\end{rk}

\section{Virtual double category}

\begin{definition}
A \defin{virtual double category}, or \defin{vdc}, \vdC is the data of:
\begin{itemize}
\item a category $\vdV(\vdC)$ called its vertical category, whose morphisms are called vertical morphisms of \vdC
\item for any objects $A_0, A_1 \in \Ob{\vdC} := \Ob{\vdV(\vdC)}$, a collection of horizontal morphisms $\vdH(\vdC)$
\item for any (possibly empty) finite chain of horizontal morphisms $p_i \colon A_{i-1} \to A_i$, any horizontal morphism $q \colon B_0 \to B_1$ and any vertical morphisms $f \colon A_0 \to B_0$ and $g \colon A_n \to B_1$, a collection of cells:
% https://q.uiver.app/?q=WzAsNyxbMCwwLCJBXzAiXSxbMiwwLCJBXzEiXSxbMywwLCJcXGRvdHMiXSxbNCwwLCJBX3tuLTF9Il0sWzYsMCwiQV9uIl0sWzAsMiwiQl8wIl0sWzYsMiwiQl8xIl0sWzAsMSwicF8xIiwxXSxbMCw1LCJmIiwxXSxbNCw2LCJnIiwxXSxbMyw0LCJwX24iLDFdLFs1LDYsInEiLDFdLFsyLDExLCJcXGFscGhhIiwxLHsic2hvcnRlbiI6eyJ0YXJnZXQiOjIwfX1dXQ==
\[\begin{tikzcd}
	{A_0} && {A_1} & \dots & {A_{n-1}} && {A_n} \\
	\\
	{B_0} &&&&&& {B_1}
	\arrow["{p_1}"{description}, from=1-1, to=1-3]
	\arrow["f"{description}, from=1-1, to=3-1]
	\arrow["g"{description}, from=1-7, to=3-7]
	\arrow["{p_n}"{description}, from=1-5, to=1-7]
	\arrow[""{name=0, anchor=center, inner sep=0}, "q"{description}, from=3-1, to=3-7]
	\arrow["\alpha"{description}, shorten >=7pt, Rightarrow, from=1-4, to=0]
\end{tikzcd}\]
\item for any horizontal morphism, an identity cell:
% https://q.uiver.app/?q=WzAsNCxbMCwwLCJBXzAiXSxbMCwyLCJBXzAiXSxbMiwwLCJCXzAiXSxbMiwyLCJCXzAiXSxbMCwxLCIiLDEseyJsZXZlbCI6Miwic3R5bGUiOnsiaGVhZCI6eyJuYW1lIjoibm9uZSJ9fX1dLFsyLDMsIiIsMSx7ImxldmVsIjoyLCJzdHlsZSI6eyJoZWFkIjp7Im5hbWUiOiJub25lIn19fV0sWzAsMiwicCIsMV0sWzEsMywicCIsMV0sWzYsNywiXFxpZF9QIiwxLHsic2hvcnRlbiI6eyJzb3VyY2UiOjIwLCJ0YXJnZXQiOjIwfX1dXQ==
\[\begin{tikzcd}
	{A_0} && {B_0} \\
	\\
	{A_0} && {B_0}
	\arrow[Rightarrow, no head, from=1-1, to=3-1]
	\arrow[Rightarrow, no head, from=1-3, to=3-3]
	\arrow[""{name=0, anchor=center, inner sep=0}, "p"{description}, from=1-1, to=1-3]
	\arrow[""{name=1, anchor=center, inner sep=0}, "p"{description}, from=3-1, to=3-3]
	\arrow["{\id_P}"{description}, shorten <=9pt, shorten >=9pt, Rightarrow, from=0, to=1]
\end{tikzcd}\]
\item for any composable cells, $\alpha_i$, $\beta$, i.e. of the following form:

\resizebox{\hsize}{!}{
% https://q.uiver.app/?q=WzAsMTksWzAsMCwiQV8wIl0sWzIsMCwiQV8xIl0sWzQsMCwiQV97bV8xLTF9Il0sWzYsMCwiQV97bV8xfSJdLFszLDAsIlxcZG90cyJdLFs3LDAsIlxcZG90cyJdLFs4LDAsIkFfe21fe24tMX19Il0sWzEwLDAsIkFfe21fe24tMX0rMX0iXSxbMTEsMCwiXFxkb3RzIl0sWzEyLDAsIkFfe21fbi0xfSJdLFsxNCwwLCJBX3ttX259Il0sWzAsMiwiQl8wIl0sWzYsMiwiQl8xIl0sWzgsMiwiQl97bi0xfSJdLFsxNCwyLCJCX24iXSxbNywxLCJcXGRvdHMiXSxbNywyLCJcXGRvdHMiXSxbMCw0LCJDXzAiXSxbMTQsNCwiQ18xIl0sWzAsMSwicF8xIiwxXSxbMiwzLCJwX3ttXzF9IiwxXSxbNiw3LCJwX3ttX3tuLTF9KzF9IiwxXSxbOSwxMCwicF97bV9ufSIsMV0sWzExLDEyLCJxXzEiLDFdLFsxMywxNCwicV9uIiwxXSxbMCwxMSwiZl8wIiwxXSxbMywxMiwiZl8xIiwxXSxbNiwxMywiZl97bl8xfSIsMV0sWzEwLDE0LCJmX24iLDFdLFsxNywxOCwiciIsMV0sWzExLDE3LCJnXzAiLDFdLFsxNCwxOCwiZ19uIiwxXSxbNCwyMywiXFxhbHBoYV8xIiwxLHsic2hvcnRlbiI6eyJ0YXJnZXQiOjIwfX1dLFs4LDI0LCJcXGFscGhhX24iLDEseyJzaG9ydGVuIjp7InRhcmdldCI6MjB9fV0sWzE2LDI5LCJcXGJldGEiLDEseyJzaG9ydGVuIjp7InRhcmdldCI6MjB9fV1d
\begin{tikzcd}[ampersand replacement=\&]
	{A_0} \&\& {A_1} \& \dots \& {A_{m_1-1}} \&\& {A_{m_1}} \& \dots \& {A_{m_{n-1}}} \&\& {A_{m_{n-1}+1}} \& \dots \& {A_{m_n-1}} \&\& {A_{m_n}} \\
	\&\&\&\&\&\&\& \dots \\
	{B_0} \&\&\&\&\&\& {B_1} \& \dots \& {B_{n-1}} \&\&\&\&\&\& {B_n} \\
	\\
	{C_0} \&\&\&\&\&\&\&\&\&\&\&\&\&\& {C_1}
	\arrow["{p_1}"{description}, from=1-1, to=1-3]
	\arrow["{p_{m_1}}"{description}, from=1-5, to=1-7]
	\arrow["{p_{m_{n-1}+1}}"{description}, from=1-9, to=1-11]
	\arrow["{p_{m_n}}"{description}, from=1-13, to=1-15]
	\arrow[""{name=0, anchor=center, inner sep=0}, "{q_1}"{description}, from=3-1, to=3-7]
	\arrow[""{name=1, anchor=center, inner sep=0}, "{q_n}"{description}, from=3-9, to=3-15]
	\arrow["{f_0}"{description}, from=1-1, to=3-1]
	\arrow["{f_1}"{description}, from=1-7, to=3-7]
	\arrow["{f_{n_1}}"{description}, from=1-9, to=3-9]
	\arrow["{f_n}"{description}, from=1-15, to=3-15]
	\arrow[""{name=2, anchor=center, inner sep=0}, "r"{description}, from=5-1, to=5-15]
	\arrow["{g_0}"{description}, from=3-1, to=5-1]
	\arrow["{g_n}"{description}, from=3-15, to=5-15]
	\arrow["{\alpha_1}"{description}, shorten >=7pt, Rightarrow, from=1-4, to=0]
	\arrow["{\alpha_n}"{description}, shorten >=7pt, Rightarrow, from=1-12, to=1]
	\arrow["\beta"{description}, shorten >=8pt, Rightarrow, from=3-8, to=2]
\end{tikzcd}
}

a cell $\beta(\alpha_1,\dots,\alpha_n)$:

% https://q.uiver.app/?q=WzAsMTUsWzAsMCwiQV8wIl0sWzIsMCwiQV8xIl0sWzQsMCwiQV97bV8xLTF9Il0sWzYsMCwiQV97bV8xfSJdLFszLDAsIlxcZG90cyJdLFs3LDAsIlxcZG90cyJdLFs4LDAsIkFfe21fe24tMX19Il0sWzEwLDAsIkFfe21fe24tMX0rMX0iXSxbMTEsMCwiXFxkb3RzIl0sWzEyLDAsIkFfe21fbi0xfSJdLFsxNCwwLCJBX3ttX259Il0sWzAsMiwiQl8wIl0sWzE0LDIsIkJfbiJdLFswLDQsIkNfMCJdLFsxNCw0LCJDXzEiXSxbMCwxLCJwXzEiLDFdLFsyLDMsInBfe21fMX0iLDFdLFs2LDcsInBfe21fe24tMX0rMX0iLDFdLFs5LDEwLCJwX3ttX259IiwxXSxbMCwxMSwiZl8wIiwxXSxbMTAsMTIsImZfbiIsMV0sWzEzLDE0LCJyIiwxXSxbMTEsMTMsImdfMCIsMV0sWzEyLDE0LCJnX24iLDFdLFs1LDIxLCJcXGJldGEoXFxhbHBoYV8xLFxcZG90cyxcXGFscGhhX24pIiwxLHsic2hvcnRlbiI6eyJ0YXJnZXQiOjIwfX1dXQ==
\resizebox{0.9\hsize}{!}{
\begin{tikzcd}[ampersand replacement=\&]
	{A_0} \&\& {A_1} \& \dots \& {A_{m_1-1}} \&\& {A_{m_1}} \& \dots \& {A_{m_{n-1}}} \&\& {A_{m_{n-1}+1}} \& \dots \& {A_{m_n-1}} \&\& {A_{m_n}} \\
	\\
	{B_0} \&\&\&\&\&\&\&\&\&\&\&\&\&\& {B_n} \\
	\\
	{C_0} \&\&\&\&\&\&\&\&\&\&\&\&\&\& {C_1}
	\arrow["{p_1}"{description}, from=1-1, to=1-3]
	\arrow["{p_{m_1}}"{description}, from=1-5, to=1-7]
	\arrow["{p_{m_{n-1}+1}}"{description}, from=1-9, to=1-11]
	\arrow["{p_{m_n}}"{description}, from=1-13, to=1-15]
	\arrow["{f_0}"{description}, from=1-1, to=3-1]
	\arrow["{f_n}"{description}, from=1-15, to=3-15]
	\arrow[""{name=0, anchor=center, inner sep=0}, "r"{description}, from=5-1, to=5-15]
	\arrow["{g_0}"{description}, from=3-1, to=5-1]
	\arrow["{g_n}"{description}, from=3-15, to=5-15]
	\arrow["{\beta(\alpha_1,\dots,\alpha_n)}"{description}, shorten >=16pt, Rightarrow, from=1-8, to=0]
\end{tikzcd}
}
\end{itemize}
such that
\begin{itemize}
\item $\alpha(\id_{p_1},\dots,\id_{p_n}) = \alpha = \id_q(\alpha)$
\item $\gamma(\beta_1,\dots,\beta_n)(\alpha_{1,1},\dots, \alpha_{n,m_n}) = \gamma(\beta_1(\alpha_{1,1},\dots,\alpha_{1,m_1}),\dots,\beta_n(\alpha_{n,1},\dots,\alpha_{n,m_n})$
\end{itemize}
\end{definition}

\begin{example}
From any bicategory \cC we can define a vdc with:
\begin{itemize}
\item vertical category the discrete category with objects those of \cC
\item horizontal morphisms the 1-morphisms of \cC
\item cells $\alpha \colon p_1,\dots,p_n \Rightarrow q$, the 2-morphisms $\alpha \colon p_n \circ \dots \circ p_1 \Rightarrow q$ in \cC 
\end{itemize}

Not every vdc with a discrete vertical category defines a bicategory though.
This is because in a vdc, horizontal morphisms do not generally compose.
When it is the case, we call it a (pseudo) double category.
\end{example}

\begin{example}
From any multicategory \cM we can define a vdc called its delooping $\cB(\cM)$ with:
\begin{itemize}
\item vertical category the terminal category $\one$
\item horizontal morphisms, the objects of \cM
\item cells, the multimap in \cM
\end{itemize}
This vdc is a double category iff \cM is representable.
\end{example}

In particular, a monoidal category is a vdc either by considering the underlying vdc of its delooping bicategory or by considering the delooping vdc of its underlying multicategory.
Both definitions coincide.

\begin{example}
There is a vdc $\Dist$ with:
\begin{itemize}
\item vertical category $\Cat$
\item horizontal morphisms, distributors
\item cells, natural transformations
\end{itemize}
\end{example}

\begin{example}
The terminal vdc $\one$ has:
\begin{itemize}
\item for vertical category the terminal category $\one$
\item one horizontal arrow $\ast \to \ast$
\item for each arity $n \geq 0$ one cell $\underline{n}$ from n copies of the arrow to it:

% https://q.uiver.app/?q=WzAsNyxbMCwwLCJcXGFzdCJdLFsyLDAsIlxcYXN0Il0sWzMsMCwiXFxkb3RzIl0sWzQsMCwiXFxhc3QiXSxbNiwwLCJcXGFzdCJdLFswLDIsIlxcYXN0Il0sWzYsMiwiXFxhc3QiXSxbMCwxXSxbMCw1LCIiLDEseyJsZXZlbCI6Miwic3R5bGUiOnsiaGVhZCI6eyJuYW1lIjoibm9uZSJ9fX1dLFs0LDYsIiIsMSx7ImxldmVsIjoyLCJzdHlsZSI6eyJoZWFkIjp7Im5hbWUiOiJub25lIn19fV0sWzMsNF0sWzUsNl0sWzIsMTEsIlxcdW5kZXJsaW5le259IiwxXV0=
\begin{tikzcd}[ampersand replacement=\&]
	\ast \&\& \ast \& \dots \& \ast \&\& \ast \\
	\\
	\ast \&\&\&\&\&\& \ast
	\arrow[from=1-1, to=1-3]
	\arrow[Rightarrow, no head, from=1-1, to=3-1]
	\arrow[Rightarrow, no head, from=1-7, to=3-7]
	\arrow[from=1-5, to=1-7]
	\arrow[""{name=0, anchor=center, inner sep=0}, from=3-1, to=3-7]
	\arrow["{\underline{n}}"{description}, Rightarrow, from=1-4, to=0]
\end{tikzcd}
\end{itemize}
\end{example}

\begin{example}
There is a vdc $\Rel$ with:
\begin{itemize}
\item vertical category $\Set$
\item horizontal morphisms relations
\item a unique cell
% https://q.uiver.app/?q=WzAsNyxbMCwwLCJBXzAiXSxbMiwwLCJBXzEiXSxbMywwLCJcXGRvdHMiXSxbNCwwLCJBX3tuLTF9Il0sWzYsMCwiQV9uIl0sWzAsMiwiQl8wIl0sWzYsMiwiQl8xIl0sWzAsMSwicl8xIiwxXSxbMyw0LCJyX24iLDFdLFs1LDYsInMiLDFdLFswLDUsImYiLDFdLFs0LDYsImciLDFdLFsyLDksIiIsMSx7InNob3J0ZW4iOnsidGFyZ2V0IjoyMH19XV0=
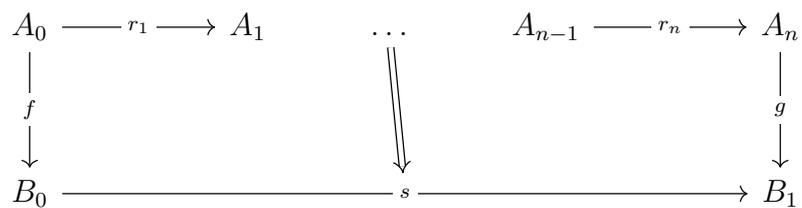
\[\begin{tikzcd}[ampersand replacement=\&]
	{A_0} \&\& {A_1} \& \dots \& {A_{n-1}} \&\& {A_n} \\
	\\
	{B_0} \&\&\&\&\&\& {B_1}
	\arrow["{r_1}"{description}, from=1-1, to=1-3]
	\arrow["{r_n}"{description}, from=1-5, to=1-7]
	\arrow[""{name=0, anchor=center, inner sep=0}, "s"{description}, from=3-1, to=3-7]
	\arrow["f"{description}, from=1-1, to=3-1]
	\arrow["g"{description}, from=1-7, to=3-7]
	\arrow[shorten >=7pt, Rightarrow, from=1-4, to=0]
\end{tikzcd}\]
iff for any $(a_0, a_n)\in A_0 \times A_n$ if there exists $a_i$ for $0 < i < n$ with $r_i(a_{i-1},a_i)$ then $s(f(a_0),g(a_n))$.
\end{itemize}
\end{example}

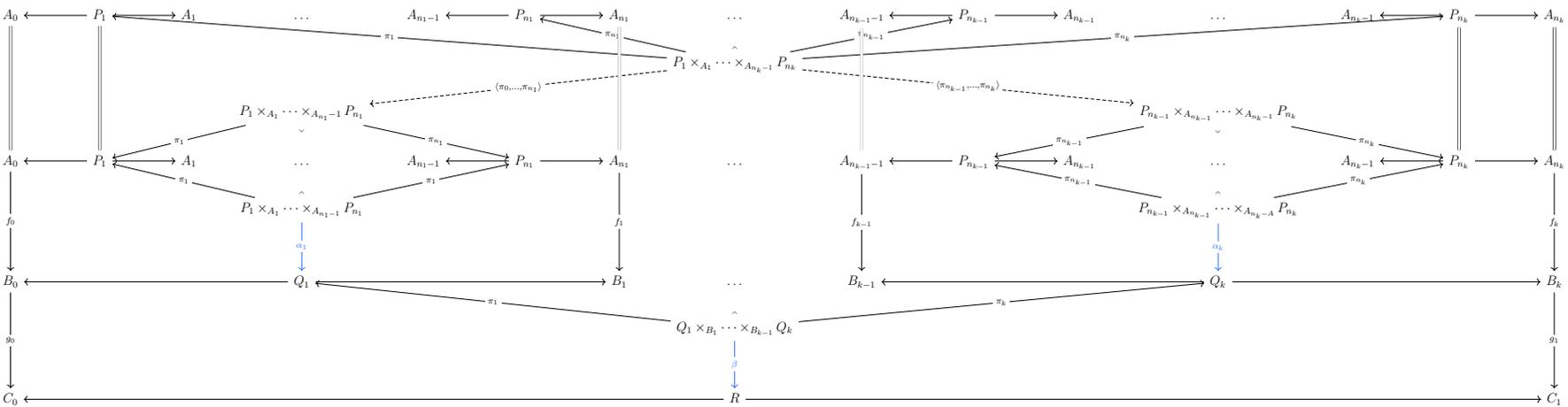
\begin{figure}
\rotatebox{90}{
\resizebox{\vsize}{!}{
\begin{tikzcd}[ampersand replacement=\&]
	{A_0} \&\& {P_1} \&\& {A_1} \& \dots \& {A_{n_1-1}} \&\& {P_{n_1}} \&\& {A_{n_1}} \& \dots \& {A_{n_{k-1}-1}} \&\& {P_{n_{k-1}}} \&\& {A_{n_{k-1}}} \& \dots \& {A_{n_k-1}} \&\& {P_{n_k}} \&\& {A_{n_k}} \\
	\&\&\&\&\&\&\&\&\&\&\& {P_1 \times_{A_1}\dots\times_{A_{n_k-1}}P_{n_k}} \\
	\&\&\&\&\& {P_1\times_{A_1}\dots\times_{A_{n_1}-1}P_{n_1}} \&\&\&\&\&\&\&\&\&\&\&\& {P_{n_{k-1}}\times_{A_{n_{k-1}}}\dots\times_{A_{n_k-1}}P_{n_k}} \\
	{A_0} \&\& {P_1} \&\& {A_1} \& \dots \& {A_{n_1-1}} \&\& {P_{n_1}} \&\& {A_{n_1}} \& \dots \& {A_{n_{k-1}-1}} \&\& {P_{n_{k-1}}} \&\& {A_{n_{k-1}}} \& \dots \& {A_{n_k-1}} \&\& {P_{n_k}} \&\& {A_{n_k}} \\
	\&\&\&\&\& {P_1 \times_{A_1} \dots \times_{A_{n_1-1}}P_{n_1}} \&\&\&\&\&\&\&\&\&\&\&\& {P_{n_{k-1}}\times_{A_{n_{k-1}}}\dots\times_{A_{n_k-A}}P_{n_k}} \\
	\\
	{B_0} \&\&\&\&\& {Q_1} \&\&\&\&\& {B_1} \& \dots \& {B_{k-1}} \&\&\&\&\& {Q_k} \&\&\&\&\& {B_k} \\
	\&\&\&\&\&\&\&\&\&\&\& {Q_1\times_{B_1}\dots\times_{B_{k-1}}Q_k} \\
	\\
	{C_0} \&\&\&\&\&\&\&\&\&\&\& R \&\&\&\&\&\&\&\&\&\&\& {C_1}
	\arrow[from=1-3, to=1-1]
	\arrow[from=1-3, to=1-5]
	\arrow[from=1-9, to=1-7]
	\arrow[from=1-9, to=1-11]
	\arrow[from=1-15, to=1-13]
	\arrow[from=1-15, to=1-17]
	\arrow[from=1-21, to=1-19]
	\arrow[from=1-21, to=1-23]
	\arrow[Rightarrow, no head, from=1-1, to=4-1]
	\arrow[Rightarrow, no head, from=1-23, to=4-23]
	\arrow["{\pi_1}"{description}, from=2-12, to=1-3]
	\arrow["{\pi_{n_1}}"{description}, from=2-12, to=1-9]
	\arrow["{\pi_{n_{k-1}}}"{description}, from=2-12, to=1-15]
	\arrow["{\pi_{n_k}}"{description}, from=2-12, to=1-21]
	\arrow["\lrcorner"{anchor=center, pos=0.125, rotate=135}, draw=none, from=2-12, to=1-12]
	\arrow[from=4-3, to=4-1]
	\arrow[from=4-3, to=4-5]
	\arrow[from=4-9, to=4-7]
	\arrow[from=4-9, to=4-11]
	\arrow[from=4-15, to=4-13]
	\arrow[from=4-15, to=4-17]
	\arrow[from=4-21, to=4-19]
	\arrow[from=4-21, to=4-23]
	\arrow["{\pi_1}"{description}, from=3-6, to=4-3]
	\arrow["{\pi_{n_1}}"{description}, from=3-6, to=4-9]
	\arrow["\lrcorner"{anchor=center, pos=0.125, rotate=-45}, draw=none, from=3-6, to=4-6]
	\arrow["{\pi_{n_{k-1}}}"{description}, from=3-18, to=4-15]
	\arrow["{\pi_{n_k}}"{description}, from=3-18, to=4-21]
	\arrow["\lrcorner"{anchor=center, pos=0.125, rotate=-45}, draw=none, from=3-18, to=4-18]
	\arrow["{\langle \pi_0,\dots,\pi_{n_1}\rangle}"{description}, dashed, from=2-12, to=3-6]
	\arrow["{\langle\pi_{n_{k-1}},\dots,\pi_{n_k}\rangle}"{description}, dashed, from=2-12, to=3-18]
	\arrow["{\pi_1}"{description}, from=5-6, to=4-3]
	\arrow["{\pi_1}"{description}, from=5-6, to=4-9]
	\arrow["{\pi_{n_{k-1}}}"{description}, from=5-18, to=4-15]
	\arrow["{\pi_{n_k}}"{description}, from=5-18, to=4-21]
	\arrow[from=7-6, to=7-1]
	\arrow[from=7-6, to=7-11]
	\arrow[from=7-18, to=7-13]
	\arrow[from=7-18, to=7-23]
	\arrow["{f_0}"{description}, from=4-1, to=7-1]
	\arrow["{\alpha_1}"{description}, color={rgb,255:red,51;green,119;blue,255}, from=5-6, to=7-6]
	\arrow["\lrcorner"{anchor=center, pos=0.125, rotate=135}, draw=none, from=5-6, to=4-6]
	\arrow["{\alpha_k}"{description}, color={rgb,255:red,51;green,119;blue,255}, from=5-18, to=7-18]
	\arrow["\lrcorner"{anchor=center, pos=0.125, rotate=135}, draw=none, from=5-18, to=4-18]
	\arrow["{f_k}"{description}, from=4-23, to=7-23]
	\arrow["{\pi_1}"{description}, from=8-12, to=7-6]
	\arrow["{\pi_k}"{description}, from=8-12, to=7-18]
	\arrow["\lrcorner"{anchor=center, pos=0.125, rotate=135}, draw=none, from=8-12, to=7-12]
	\arrow[from=10-12, to=10-1]
	\arrow[from=10-12, to=10-23]
	\arrow["{g_0}"{description}, from=7-1, to=10-1]
	\arrow["{g_1}"{description}, from=7-23, to=10-23]
	\arrow["\beta"{description}, color={rgb,255:red,51;green,119;blue,255}, from=8-12, to=10-12]
	\arrow[Rightarrow, no head, from=1-3, to=4-3]
	\arrow[color={rgb,255:red,179;green,179;blue,179}, Rightarrow, no head, from=1-11, to=4-11]
	\arrow[color={rgb,255:red,179;green,179;blue,179}, Rightarrow, no head, from=1-13, to=4-13]
	\arrow[Rightarrow, no head, from=1-21, to=4-21]
	\arrow["{f_1}"{description}, from=4-11, to=7-11]
	\arrow["{f_{k-1}}"{description}, from=4-13, to=7-13]
\end{tikzcd}
}
}
\caption{Composition in the vdc $\mathbf{Span}$}
\label{fig:span_comp}
\end{figure}

\begin{example}
For any category with pullbacks \cC, we denote by $\Span(\cC)$ the virtual category with:
\begin{itemize}
\item vertical category \cC
\item horizontal morphisms $P \colon A \to B$ are spans % https://q.uiver.app/?q=WzAsMyxbMiwwLCJQIl0sWzAsMCwiQSJdLFs0LDAsIkIiXSxbMCwxLCJwX0EiLDFdLFswLDIsInBfQiIsMV1d
\begin{tikzcd}[ampersand replacement=\&]
	A \&\& P \&\& B
	\arrow["{p_A}"{description}, from=1-3, to=1-1]
	\arrow["{p_B}"{description}, from=1-3, to=1-5]
\end{tikzcd}
\item cells are morphisms from the pullback of the $P_i$ to $Q$ making the following diagram commute:
\scalebox{0.7}{
% https://q.uiver.app/?q=WzAsMTEsWzIsMCwiUF8xIl0sWzAsMCwiQV8wIl0sWzQsMCwiQV8xIl0sWzUsMCwiXFxkb3RzIl0sWzYsMCwiQV97bi0xfSJdLFs4LDAsIlBfbiJdLFsxMCwwLCJBX24iXSxbNSwyLCJQXzEgXFx0aW1lc197QV8xfSAuLi4gXFx0aW1lc197QV97bi0xfX0gUF9uIl0sWzUsNCwiUSJdLFswLDQsIkJfMCJdLFsxMCw0LCJCXzEiXSxbMCwxXSxbMCwyXSxbNSw0XSxbNSw2XSxbNywwXSxbNyw1XSxbNywzLCIiLDEseyJzdHlsZSI6eyJuYW1lIjoiY29ybmVyIn19XSxbOCw5XSxbOCwxMF0sWzEsOSwiZl8wIiwxXSxbNiwxMCwiZl9uIiwxXSxbNyw4LCJcXGFscGhhIiwxXV0=
\begin{tikzcd}[ampersand replacement=\&]
	{A_0} \&\& {P_1} \&\& {A_1} \& \dots \& {A_{n-1}} \&\& {P_n} \&\& {A_n} \\
	\\
	\&\&\&\&\& {P_1 \times_{A_1} ... \times_{A_{n-1}} P_n} \\
	\\
	{B_0} \&\&\&\&\& Q \&\&\&\&\& {B_1}
	\arrow[from=1-3, to=1-1]
	\arrow[from=1-3, to=1-5]
	\arrow[from=1-9, to=1-7]
	\arrow[from=1-9, to=1-11]
	\arrow[from=3-6, to=1-3]
	\arrow[from=3-6, to=1-9]
	\arrow["\lrcorner"{anchor=center, pos=0.125, rotate=135}, draw=none, from=3-6, to=1-6]
	\arrow[from=5-6, to=5-1]
	\arrow[from=5-6, to=5-11]
	\arrow["{f_0}"{description}, from=1-1, to=5-1]
	\arrow["{f_n}"{description}, from=1-11, to=5-11]
	\arrow["\alpha"{description}, from=3-6, to=5-6]
\end{tikzcd}
}
\item the identity cell is given by the identity cell in \cC:
\[
% https://q.uiver.app/?q=WzAsNixbMCwwLCJBXzAiXSxbMiwwLCJQIl0sWzQsMCwiQV8xIl0sWzAsMiwiQV8wIl0sWzIsMiwiUCJdLFs0LDIsIkFfMSJdLFsxLDBdLFsxLDJdLFswLDMsIiIsMSx7ImxldmVsIjoyLCJzdHlsZSI6eyJoZWFkIjp7Im5hbWUiOiJub25lIn19fV0sWzEsNCwiIiwxLHsibGV2ZWwiOjIsInN0eWxlIjp7ImhlYWQiOnsibmFtZSI6Im5vbmUifX19XSxbMiw1LCIiLDEseyJsZXZlbCI6Miwic3R5bGUiOnsiaGVhZCI6eyJuYW1lIjoibm9uZSJ9fX1dLFs0LDNdLFs0LDVdXQ==
\begin{tikzcd}[ampersand replacement=\&]
	{A_0} \&\& P \&\& {A_1} \\
	\\
	{A_0} \&\& P \&\& {A_1}
	\arrow[from=1-3, to=1-1]
	\arrow[from=1-3, to=1-5]
	\arrow[Rightarrow, no head, from=1-1, to=3-1]
	\arrow[Rightarrow, no head, from=1-3, to=3-3]
	\arrow[Rightarrow, no head, from=1-5, to=3-5]
	\arrow[from=3-3, to=3-1]
	\arrow[from=3-3, to=3-5]
\end{tikzcd}
\]

\item the composition of $(f_{i-1},\alpha_i,f_i)$ and $(g_0,\beta,g_1)$ is given by \[(g_0 \circ f_0, \beta \circ \langle \alpha_i \circ \langle \pi_{j}\rangle_{n_i \leq j \leq n_{i+1}}\rangle_{1\leq i \leq k-1},g_1\circ f_k)\] where $\pi_i$ is the projection of the pullback to $A_i$ and $\langle f_i \rangle$ is the unique factorisation corresponding to the universal property of the pullback, as represented in figure \ref{fig:span_comp} where the dotted lines corresponds to the morphisms obtained by the universal property of the pullback and the blue ones corresponds to the cells being composed.
\end{itemize}

This is a virtual double category.
One can check that \[\beta \circ \langle \id_{A_i} \langle \pi_j \rangle_{i \leq j \leq i} \rangle_{1 \leq i \leq k} = \beta \circ \langle \pi_i \rangle_{1 \leq i \leq k} = \beta \circ \id_{P_1 \times_{A_1} \dots \times_{A_k} P_k} = \beta \]
and \[\id_Q \circ \langle \beta \circ \langle \pi_j \rangle_{1 \leq j \leq k} \rangle_{1 \leq i \leq 1} = \id_Q \circ \beta = \beta\]
The associativity follows from the property of pullbacks.
\end{example}

\begin{example}
There is a vdc \Ring{} whose:
\begin{itemize}
\item objects are rings
\item vertical morphisms are ring homomorphisms
\item horizontal morphisms are bimodules
\item cells are balanced multimorphisms with the bimodule structure on the target obtained through restriction along the vertical morphisms
% https://q.uiver.app/?q=WzAsNyxbMCwwLCJBXzAiXSxbMiwwLCJBXzEiXSxbNCwwLCJBX3tuLTF9Il0sWzMsMCwiXFxkb3RzIl0sWzYsMCwiQV9uIl0sWzAsMiwiQl8wIl0sWzYsMiwiQl8xIl0sWzAsMSwiTV8xIiwxXSxbMiw0LCJNX24iLDFdLFs1LDYsIk4iLDFdLFswLDUsImZfMCIsMV0sWzQsNiwiZl8xIiwxXSxbMyw5LCJcXGFscGhhIiwxLHsic2hvcnRlbiI6eyJ0YXJnZXQiOjIwfX1dXQ==
\[\begin{tikzcd}[ampersand replacement=\&]
	{A_0} \&\& {A_1} \& \dots \& {A_{n-1}} \&\& {A_n} \\
	\\
	{B_0} \&\&\&\&\&\& {B_1}
	\arrow["{M_1}"{description}, from=1-1, to=1-3]
	\arrow["{M_n}"{description}, from=1-5, to=1-7]
	\arrow[""{name=0, anchor=center, inner sep=0}, "N"{description}, from=3-1, to=3-7]
	\arrow["{f_0}"{description}, from=1-1, to=3-1]
	\arrow["{f_1}"{description}, from=1-7, to=3-7]
	\arrow["\alpha"{description}, shorten >=7pt, Rightarrow, from=1-4, to=0]
\end{tikzcd}\]
\end{itemize}
Let look at this example in more detail.
A ring $(A,+,\times,0,1)$ is an abelian group $(A,+,0)$ together with a monoid structure $(A,\times,1)$ compatible with the abelian structure, i.e. $a \times (b+c) = (a \times b) + (a \times c)$.
A ring homomorphism is a group homomorphism that is also a monoid homomorphism.
An $(A,B)$-bimodule $(M,+,0,\ls_A\cdot, \cdot_B)$ is an abelian group $(M,+,0)$ equipped with a left action by $A$ and a right action by $B$, i.e. $a \ls_A\cdot (m + m') = a \ls_A\cdot m + a \ls_A\cdot m'$, $(a + a') \ls_A\cdot m = a \ls_A\cdot m + a'\ls_A\cdot m$, $((a_1 \times a_2) \ls_A\cdot m = a_1 \ls_A\cdot (a_2 \ls_A\cdot m)$, $1 \ls_A\cdot m = m$ and similarly for $\cdot_B$.
A balanced multimorphism $\alpha \colon M_1, \dots, M_n \Rightarrow N$ from $(A_{i-1},A_i)$-bimodules to a $(A_0,A_n)$-bimodule is a multivariable function $\alpha \colon M_1 \times \dots \times M_n \to N$ which is additive in each variable and balanced:
\begin{itemize}
\item $\alpha(m_1,\dots, m_i + m_i', \dots m_n) = \alpha(m_1,\dots,m_i,\dots,m_n) + \alpha(m_1,\dots,m_i',\dots,m_n)$
\item $\alpha(m_1, \dots, m_{i-1} \cdot_{A_i} a_i, m_i, \dots, m_n) = \alpha(m_1, \dots, m_{i-1}, a_i \ls_{A_i}\cdot m_i, m_n)$
\item $a_0 \ls_{A_0}\cdot \alpha(m_1,\dots,m_n) = \alpha(a_0 \ls_{A_0}\cdot m_1,\dots, m_n)$
\item $\alpha(m_1,\dots,m_n \cdot_{A_n} a_n) = \alpha(m_1,\dots,m_n) \cdot_{A_n} a_n$
\end{itemize}
Given a $(B_0,B_1)$-bimodule $N$ and ring homomorphisms $f_i \colon A_i \to B_i$ one can define a $(A_0,A_1)$-bimodule structure on $N$ by $a_0 \ls_{A_0}\cdot m := f_0(a_0) \ls_{B_0}\cdot m$ and similarly for the right action.
One can check that it defines an action using the fact that $f_0$ is a homomorphism and $\ls_{B_0}\cdot$ an action.
\end{example}

\begin{example}
There is a vdc $\mathbf{Mon}$ whose
\begin{itemize}
\item objects are monoids
\item vertical morphisms are monoid homomorphisms
\item horizon morphisms are bimodules of monoids, i.e. sets equipped with a left action by one monoid and a right action by the other
\item cells are balanced multimorphims of bimodules, i.e. multivariable functions such that the image of the left action of a monoid and of its right action are equated (see the balanced multimorphisms of bimodules in \Ring)
\end{itemize}
\end{example}

\begin{definition}
The examples above can be extended to monoids internal to any monoidal category \cV. 
\end{definition}

As mentioned above, double categories are virtual double categories in which we can also compose horizontal morphisms (and where we have an horizontal identity).
An explicit description of a double category can be given.
Then, it is possible to characterise the vdcs that are the underlying vdc of a double category universally, in a similar way that we can characterise monoidal categories in multicategories.
Due to lack of space, we will not recall the definition of a double category but instead directly give their characterisation in vdc.
In \cite{CruttwellShulman2009} the universal cells that characterise composition are called opcartesian.
We will prefer the term universal.
First because we want to stay coherent with the rest of this manuscript.
Also because we want to reserve this term for cartesian cells relatively to a functor of vdc in order to define pushfibrations of vdcs.
We will see that universal cells are a special case of opcartesian one (relatively to the functor into the terminal vdc).

\begin{definition}
In a vdc \vdC, a cell

% https://q.uiver.app/?q=WzAsNyxbMCwwLCJBXzAiXSxbMiwwLCJBXzEiXSxbMywwLCJcXGRvdHMiXSxbNCwwLCJBX3tuLTF9Il0sWzYsMCwiQV9uIl0sWzAsMiwiQV8wIl0sWzYsMiwiQV9uIl0sWzAsMSwicF8xIiwxXSxbMCw1LCIiLDEseyJsZXZlbCI6Miwic3R5bGUiOnsiaGVhZCI6eyJuYW1lIjoibm9uZSJ9fX1dLFs0LDYsIiIsMSx7ImxldmVsIjoyLCJzdHlsZSI6eyJoZWFkIjp7Im5hbWUiOiJub25lIn19fV0sWzMsNCwicF9uIiwxXSxbNSw2LCJxIiwxXSxbMiwxMSwiXFxhbHBoYSIsMSx7InNob3J0ZW4iOnsidGFyZ2V0IjoyMH19XV0=
\begin{tikzcd}[ampersand replacement=\&]
	{A_0} \&\& {A_1} \& \dots \& {A_{n-1}} \&\& {A_n} \\
	\\
	{A_0} \&\&\&\&\&\& {A_n}
	\arrow["{p_1}"{description}, from=1-1, to=1-3]
	\arrow[Rightarrow, no head, from=1-1, to=3-1]
	\arrow[Rightarrow, no head, from=1-7, to=3-7]
	\arrow["{p_n}"{description}, from=1-5, to=1-7]
	\arrow[""{name=0, anchor=center, inner sep=0}, "q"{description}, from=3-1, to=3-7]
	\arrow["\alpha"{description}, shorten >=7pt, Rightarrow, from=1-4, to=0]
\end{tikzcd}
is \defin{universal} if for any cell

% https://q.uiver.app/?q=WzAsMTUsWzYsMCwiQV8wIl0sWzgsMCwiQV8xIl0sWzksMCwiXFxkb3RzIl0sWzEwLDAsIkFfe24tMX0iXSxbMTIsMCwiQV9uIl0sWzAsMiwiQl8wIl0sWzE4LDIsIkJfMSJdLFs0LDAsIkNfe20tMX0iXSxbMCwwLCJDXzAiXSxbMiwwLCJDXzIiXSxbMywwLCJcXGRvdHMiXSxbMTQsMCwiRF8xIl0sWzE1LDAsIlxcZG90cyJdLFsxNiwwLCJEX3trLTF9Il0sWzE4LDAsIkRfayJdLFswLDEsInBfMSIsMV0sWzMsNCwicF9uIiwxXSxbNSw2LCJxIiwxXSxbNywwLCJyX20iLDFdLFs4LDksInJfMSIsMV0sWzQsMTEsInNfMSIsMV0sWzEzLDE0LCJzX2siLDFdLFsxNCw2LCJnIiwxXSxbOCw1LCJmIiwxXSxbMiwxNywiXFxiZXRhIiwxLHsic2hvcnRlbiI6eyJ0YXJnZXQiOjIwfX1dXQ==
\resizebox{\hsize}{!}{
\begin{tikzcd}[ampersand replacement=\&]
	{C_0} \&\& {C_2} \& \dots \& {C_{m-1}} \&\& {A_0} \&\& {A_1} \& \dots \& {A_{n-1}} \&\& {A_n} \&\& {D_1} \& \dots \& {D_{k-1}} \&\& {D_k} \\
	\\
	{B_0} \&\&\&\&\&\&\&\&\&\&\&\&\&\&\&\&\&\& {B_1}
	\arrow["{p_1}"{description}, from=1-7, to=1-9]
	\arrow["{p_n}"{description}, from=1-11, to=1-13]
	\arrow[""{name=0, anchor=center, inner sep=0}, "q"{description}, from=3-1, to=3-19]
	\arrow["{r_m}"{description}, from=1-5, to=1-7]
	\arrow["{r_1}"{description}, from=1-1, to=1-3]
	\arrow["{s_1}"{description}, from=1-13, to=1-15]
	\arrow["{s_k}"{description}, from=1-17, to=1-19]
	\arrow["g"{description}, from=1-19, to=3-19]
	\arrow["f"{description}, from=1-1, to=3-1]
	\arrow["\beta"{description}, shorten >=7pt, Rightarrow, from=1-10, to=0]
\end{tikzcd}}
there is a unique factorisation:

\resizebox{\hsize}{!}{
% https://q.uiver.app/?q=WzAsMjUsWzYsMCwiQV8wIl0sWzgsMCwiQV8xIl0sWzksMCwiXFxkb3RzIl0sWzEwLDAsIkFfe24tMX0iXSxbMTIsMCwiQV9uIl0sWzAsNCwiQl8wIl0sWzE4LDQsIkJfMSJdLFs0LDAsIkNfe20tMX0iXSxbMCwwLCJDXzAiXSxbMiwwLCJDXzEiXSxbMywwLCJcXGRvdHMiXSxbMTQsMCwiRF8xIl0sWzE1LDAsIlxcZG90cyJdLFsxNiwwLCJEX3trLTF9Il0sWzE4LDAsIkRfayJdLFs2LDIsIkFfMCJdLFsxMiwyLCJBX24iXSxbMCwyLCJDXzAiXSxbMTgsMiwiRF9rIl0sWzIsMiwiQ18xIl0sWzMsMiwiXFxkb3RzIl0sWzQsMiwiQ197bS0xfSJdLFsxNCwyLCJEXzEiXSxbMTUsMiwiXFxkb3RzIl0sWzE2LDIsIkRfe2stMX0iXSxbMCwxLCJwXzEiLDFdLFszLDQsInBfbiIsMV0sWzUsNiwicSIsMV0sWzcsMCwicl9tIiwxXSxbOCw5LCJyXzEiLDFdLFs0LDExLCJzXzEiLDFdLFsxMywxNCwic19rIiwxXSxbMCwxNSwiIiwxLHsibGV2ZWwiOjIsInN0eWxlIjp7ImhlYWQiOnsibmFtZSI6Im5vbmUifX19XSxbNCwxNiwiIiwxLHsibGV2ZWwiOjIsInN0eWxlIjp7ImhlYWQiOnsibmFtZSI6Im5vbmUifX19XSxbMTQsMTgsIiIsMSx7ImxldmVsIjoyLCJzdHlsZSI6eyJoZWFkIjp7Im5hbWUiOiJub25lIn19fV0sWzgsMTcsIiIsMSx7ImxldmVsIjoyLCJzdHlsZSI6eyJoZWFkIjp7Im5hbWUiOiJub25lIn19fV0sWzE3LDUsImYiLDFdLFsxOCw2LCJnIiwxXSxbMTcsMTksInJfMSIsMV0sWzcsMjEsIiIsMSx7ImxldmVsIjoyLCJzdHlsZSI6eyJoZWFkIjp7Im5hbWUiOiJub25lIn19fV0sWzksMTksIiIsMSx7ImxldmVsIjoyLCJzdHlsZSI6eyJoZWFkIjp7Im5hbWUiOiJub25lIn19fV0sWzIxLDE1LCJyX20iLDFdLFsxNSwxNiwicSIsMV0sWzE2LDIyLCJzXzEiLDFdLFsyNCwxOCwic19LIiwxXSxbMTMsMjQsIiIsMSx7ImxldmVsIjoyLCJzdHlsZSI6eyJoZWFkIjp7Im5hbWUiOiJub25lIn19fV0sWzExLDIyLCIiLDEseyJsZXZlbCI6Miwic3R5bGUiOnsiaGVhZCI6eyJuYW1lIjoibm9uZSJ9fX1dLFsyLDQyLCJcXGFscGhhIiwxLHsic2hvcnRlbiI6eyJ0YXJnZXQiOjIwfX1dLFs0MiwyNywiXFxiZXRhL1xcYWxwaGEiLDEseyJzaG9ydGVuIjp7InNvdXJjZSI6MjAsInRhcmdldCI6MjB9fV1d
\begin{tikzcd}[ampersand replacement=\&]
	{C_0} \&\& {C_1} \& \dots \& {C_{m-1}} \&\& {A_0} \&\& {A_1} \& \dots \& {A_{n-1}} \&\& {A_n} \&\& {D_1} \& \dots \& {D_{k-1}} \&\& {D_k} \\
	\\
	{C_0} \&\& {C_1} \& \dots \& {C_{m-1}} \&\& {A_0} \&\&\&\&\&\& {A_n} \&\& {D_1} \& \dots \& {D_{k-1}} \&\& {D_k} \\
	\\
	{B_0} \&\&\&\&\&\&\&\&\&\&\&\&\&\&\&\&\&\& {B_1}
	\arrow["{p_1}"{description}, from=1-7, to=1-9]
	\arrow["{p_n}"{description}, from=1-11, to=1-13]
	\arrow[""{name=0, anchor=center, inner sep=0}, "q"{description}, from=5-1, to=5-19]
	\arrow["{r_m}"{description}, from=1-5, to=1-7]
	\arrow["{r_1}"{description}, from=1-1, to=1-3]
	\arrow["{s_1}"{description}, from=1-13, to=1-15]
	\arrow["{s_k}"{description}, from=1-17, to=1-19]
	\arrow[Rightarrow, no head, from=1-7, to=3-7]
	\arrow[Rightarrow, no head, from=1-13, to=3-13]
	\arrow[Rightarrow, no head, from=1-19, to=3-19]
	\arrow[Rightarrow, no head, from=1-1, to=3-1]
	\arrow["f"{description}, from=3-1, to=5-1]
	\arrow["g"{description}, from=3-19, to=5-19]
	\arrow["{r_1}"{description}, from=3-1, to=3-3]
	\arrow[Rightarrow, no head, from=1-5, to=3-5]
	\arrow[Rightarrow, no head, from=1-3, to=3-3]
	\arrow["{r_m}"{description}, from=3-5, to=3-7]
	\arrow[""{name=1, anchor=center, inner sep=0}, "q"{description}, from=3-7, to=3-13]
	\arrow["{s_1}"{description}, from=3-13, to=3-15]
	\arrow["{s_K}"{description}, from=3-17, to=3-19]
	\arrow[Rightarrow, no head, from=1-17, to=3-17]
	\arrow[Rightarrow, no head, from=1-15, to=3-15]
	\arrow["\alpha"{description}, shorten >=7pt, Rightarrow, from=1-10, to=1]
	\arrow["{\beta/\alpha}"{description}, shorten <=9pt, shorten >=9pt, Rightarrow, from=1, to=0]
\end{tikzcd}
}
We call $q$ the composite of $p_1,\dots,p_n$ and we write it $p_n\bullet\dots\bullet p_n$.
\end{definition}

\begin{prop}
The composite of horizontal cells is unique up to unique invertible cell.
\end{prop}
\begin{proof}
Take two universal cells:

\resizebox{\hsize}{!}{
% https://q.uiver.app/?q=WzAsNyxbMCwwLCJBXzAiXSxbMiwwLCJBXzEiXSxbMywwLCJcXGRvdHMiXSxbNiwwLCJBX24iXSxbMCwyLCJBXzAiXSxbNiwyLCJBX24iXSxbNCwwLCJBX3tuLTF9Il0sWzAsMSwicF8xIiwxXSxbMCw0LCIiLDEseyJsZXZlbCI6Miwic3R5bGUiOnsiaGVhZCI6eyJuYW1lIjoibm9uZSJ9fX1dLFszLDUsIiIsMSx7ImxldmVsIjoyLCJzdHlsZSI6eyJoZWFkIjp7Im5hbWUiOiJub25lIn19fV0sWzQsNSwicSIsMV0sWzYsMywicF9uIiwxXSxbMiwxMCwiXFxhbHBoYSIsMSx7InNob3J0ZW4iOnsidGFyZ2V0IjoyMH19XV0=
\begin{tikzcd}[ampersand replacement=\&]
	{A_0} \&\& {A_1} \& \dots \& {A_{n-1}} \&\& {A_n} \\
	\\
	{A_0} \&\&\&\&\&\& {A_n}
	\arrow["{p_1}"{description}, from=1-1, to=1-3]
	\arrow[Rightarrow, no head, from=1-1, to=3-1]
	\arrow[Rightarrow, no head, from=1-7, to=3-7]
	\arrow[""{name=0, anchor=center, inner sep=0}, "q"{description}, from=3-1, to=3-7]
	\arrow["{p_n}"{description}, from=1-5, to=1-7]
	\arrow["\alpha"{description}, shorten >=7pt, Rightarrow, from=1-4, to=0]
\end{tikzcd}
}
and
\resizebox{1\hsize}{!}{
% https://q.uiver.app/?q=WzAsNyxbMCwwLCJBXzAiXSxbMiwwLCJBXzEiXSxbMywwLCJcXGRvdHMiXSxbNiwwLCJBX24iXSxbMCwyLCJBXzAiXSxbNiwyLCJBX24iXSxbNCwwLCJBX3tuLTF9Il0sWzAsMSwicF8xIiwxXSxbMCw0LCIiLDEseyJsZXZlbCI6Miwic3R5bGUiOnsiaGVhZCI6eyJuYW1lIjoibm9uZSJ9fX1dLFszLDUsIiIsMSx7ImxldmVsIjoyLCJzdHlsZSI6eyJoZWFkIjp7Im5hbWUiOiJub25lIn19fV0sWzQsNSwicSciLDFdLFs2LDMsInBfbiIsMV0sWzIsMTAsIlxcYWxwaGEnIiwxLHsic2hvcnRlbiI6eyJ0YXJnZXQiOjIwfX1dXQ==
\begin{tikzcd}[ampersand replacement=\&]
	{A_0} \&\& {A_1} \& \dots \& {A_{n-1}} \&\& {A_n} \\
	\\
	{A_0} \&\&\&\&\&\& {A_n}
	\arrow["{p_1}"{description}, from=1-1, to=1-3]
	\arrow[Rightarrow, no head, from=1-1, to=3-1]
	\arrow[Rightarrow, no head, from=1-7, to=3-7]
	\arrow[""{name=0, anchor=center, inner sep=0}, "{q'}"{description}, from=3-1, to=3-7]
	\arrow["{p_n}"{description}, from=1-5, to=1-7]
	\arrow["{\alpha'}"{description}, shorten >=7pt, Rightarrow, from=1-4, to=0]
\end{tikzcd}
}

Since $\alpha'$ is universal we can factorise $\alpha$ through it:

\resizebox{\hsize}{!}{
% https://q.uiver.app/?q=WzAsMTcsWzgsMCwiQV8wIl0sWzEwLDAsIkFfMSJdLFsxMSwwLCJcXGRvdHMiXSxbMTQsMCwiQV9uIl0sWzgsMiwiQV8wIl0sWzE0LDIsIkFfbiJdLFsxMiwwLCJBX3tuLTF9Il0sWzgsNCwiQV8wIl0sWzE0LDQsIkFfbiJdLFswLDAsIkFfMCJdLFsyLDAsIkFfMSJdLFszLDAsIlxcZG90cyJdLFs0LDAsIkFfe24tMX0iXSxbNiwwLCJBX24iXSxbMCw0LCJBXzAiXSxbNiw0LCJBX24iXSxbNywyLCI9Il0sWzAsMSwicF8xIiwxXSxbMCw0LCIiLDEseyJsZXZlbCI6Miwic3R5bGUiOnsiaGVhZCI6eyJuYW1lIjoibm9uZSJ9fX1dLFszLDUsIiIsMSx7ImxldmVsIjoyLCJzdHlsZSI6eyJoZWFkIjp7Im5hbWUiOiJub25lIn19fV0sWzQsNSwicSciLDFdLFs2LDMsInBfbiIsMV0sWzQsNywiIiwxLHsibGV2ZWwiOjIsInN0eWxlIjp7ImhlYWQiOnsibmFtZSI6Im5vbmUifX19XSxbNSw4LCIiLDEseyJsZXZlbCI6Miwic3R5bGUiOnsiaGVhZCI6eyJuYW1lIjoibm9uZSJ9fX1dLFs3LDgsInEiLDFdLFs5LDEwLCJwXzEiLDFdLFsxMiwxMywicF9uIiwxXSxbOSwxNCwiIiwxLHsibGV2ZWwiOjIsInN0eWxlIjp7ImhlYWQiOnsibmFtZSI6Im5vbmUifX19XSxbMTMsMTUsIiIsMSx7ImxldmVsIjoyLCJzdHlsZSI6eyJoZWFkIjp7Im5hbWUiOiJub25lIn19fV0sWzE0LDE1LCJxIiwxXSxbMiwyMCwiXFxhbHBoYSciLDEseyJzaG9ydGVuIjp7InRhcmdldCI6MjB9fV0sWzIwLDI0LCJcXGFscGhhL1xcYWxwaGEnIiwxLHsic2hvcnRlbiI6eyJzb3VyY2UiOjIwLCJ0YXJnZXQiOjIwfX1dLFsxMSwyOSwiXFxhbHBoYSIsMSx7InNob3J0ZW4iOnsidGFyZ2V0IjoyMH19XV0=
\begin{tikzcd}[ampersand replacement=\&]
	{A_0} \&\& {A_1} \& \dots \& {A_{n-1}} \&\& {A_n} \&\& {A_0} \&\& {A_1} \& \dots \& {A_{n-1}} \&\& {A_n} \\
	\\
	\&\&\&\&\&\&\& {=} \& {A_0} \&\&\&\&\&\& {A_n} \\
	\\
	{A_0} \&\&\&\&\&\& {A_n} \&\& {A_0} \&\&\&\&\&\& {A_n}
	\arrow["{p_1}"{description}, from=1-9, to=1-11]
	\arrow[Rightarrow, no head, from=1-9, to=3-9]
	\arrow[Rightarrow, no head, from=1-15, to=3-15]
	\arrow[""{name=0, anchor=center, inner sep=0}, "{q'}"{description}, from=3-9, to=3-15]
	\arrow["{p_n}"{description}, from=1-13, to=1-15]
	\arrow[Rightarrow, no head, from=3-9, to=5-9]
	\arrow[Rightarrow, no head, from=3-15, to=5-15]
	\arrow[""{name=1, anchor=center, inner sep=0}, "q"{description}, from=5-9, to=5-15]
	\arrow["{p_1}"{description}, from=1-1, to=1-3]
	\arrow["{p_n}"{description}, from=1-5, to=1-7]
	\arrow[Rightarrow, no head, from=1-1, to=5-1]
	\arrow[Rightarrow, no head, from=1-7, to=5-7]
	\arrow[""{name=2, anchor=center, inner sep=0}, "q"{description}, from=5-1, to=5-7]
	\arrow["{\alpha'}"{description}, shorten >=7pt, Rightarrow, from=1-12, to=0]
	\arrow["{\alpha/\alpha'}"{description}, shorten <=9pt, shorten >=9pt, Rightarrow, from=0, to=1]
	\arrow["\alpha"{description}, shorten >=16pt, Rightarrow, from=1-4, to=2]
\end{tikzcd}
}

and similarly, $\alpha'$ can be factorised through $\alpha$.

But then, we have:

\resizebox{\hsize}{!}{
\begin{tikzcd}[ampersand replacement=\&]
	{A_0} \&\& {A_1} \& \dots \& {A_{n-1}} \&\& {A_n} \&\& {A_0} \&\& {A_1} \& \dots \& {A_{n-1}} \&\& {A_n} \\
	\\
	{A_0} \&\&\&\&\&\& {A_n} \\
	\&\&\&\&\&\&\& {=} \\
	\\
	\\
	{A_0} \&\&\&\&\&\& {A_n} \&\& {A_0} \&\&\&\&\&\& {A_n} \\
	\\
	{A_0} \&\& {A_1} \& \dots \& {A_{n-1}} \&\& {A_n} \&\& {A_0} \&\& {A_1} \& \dots \& {A_{n-1}} \&\& {A_n} \\
	\\
	\&\&\&\&\&\&\&\& {A_0} \&\&\&\&\&\& {A_n} \\
	\&\&\&\&\&\&\& {=} \\
	{A_0} \&\&\&\&\&\& {A_n} \&\& {A_0} \&\&\&\&\&\& {A_n} \\
	\\
	{A_0} \&\&\&\&\&\& {A_n} \&\& {A_0} \&\&\&\&\&\& {A_n}
	\arrow["{p_1}"{description}, from=9-1, to=9-3]
	\arrow[Rightarrow, no head, from=9-1, to=13-1]
	\arrow[Rightarrow, no head, from=9-7, to=13-7]
	\arrow[""{name=0, anchor=center, inner sep=0}, "{q'}"{description}, from=13-1, to=13-7]
	\arrow["{p_n}"{description}, from=9-5, to=9-7]
	\arrow[Rightarrow, no head, from=13-1, to=15-1]
	\arrow[Rightarrow, no head, from=13-7, to=15-7]
	\arrow[""{name=1, anchor=center, inner sep=0}, "q"{description}, from=15-1, to=15-7]
	\arrow["{p_1}"{description}, from=9-9, to=9-11]
	\arrow["{p_n}"{description}, from=9-13, to=9-15]
	\arrow[Rightarrow, no head, from=9-9, to=11-9]
	\arrow[Rightarrow, no head, from=11-9, to=13-9]
	\arrow[Rightarrow, no head, from=13-9, to=15-9]
	\arrow[""{name=2, anchor=center, inner sep=0}, "q"{description}, from=15-9, to=15-15]
	\arrow[Rightarrow, no head, from=9-15, to=11-15]
	\arrow[Rightarrow, no head, from=11-15, to=13-15]
	\arrow[Rightarrow, no head, from=13-15, to=15-15]
	\arrow[""{name=3, anchor=center, inner sep=0}, "{q'}"{description}, from=13-9, to=13-15]
	\arrow[""{name=4, anchor=center, inner sep=0}, "q"{description}, from=11-9, to=11-15]
	\arrow[""{name=5, anchor=center, inner sep=0}, "q"{description}, from=7-9, to=7-15]
	\arrow[Rightarrow, no head, from=1-9, to=7-9]
	\arrow[Rightarrow, no head, from=1-15, to=7-15]
	\arrow["{p_1}"{description}, from=1-9, to=1-11]
	\arrow["{p_n}"{description}, from=1-13, to=1-15]
	\arrow["{p_n}"{description}, from=1-5, to=1-7]
	\arrow["{p_1}"{description}, from=1-1, to=1-3]
	\arrow[""{name=6, anchor=center, inner sep=0}, "q"{description}, from=7-1, to=7-7]
	\arrow[Rightarrow, no head, from=1-1, to=3-1]
	\arrow[Rightarrow, no head, from=3-1, to=7-1]
	\arrow[Rightarrow, no head, from=1-7, to=3-7]
	\arrow[Rightarrow, no head, from=3-7, to=7-7]
	\arrow[""{name=7, anchor=center, inner sep=0}, "q"{description}, from=3-1, to=3-7]
	\arrow["{\alpha'}"{description}, shorten >=16pt, Rightarrow, from=9-4, to=0]
	\arrow["{\alpha/\alpha'}"{description}, shorten <=9pt, shorten >=9pt, Rightarrow, from=0, to=1]
	\arrow["{\alpha/\alpha'}"{description}, shorten <=9pt, shorten >=9pt, Rightarrow, from=3, to=2]
	\arrow["{\alpha'/\alpha}"{description}, shorten <=9pt, shorten >=9pt, Rightarrow, from=4, to=3]
	\arrow["\alpha"{description}, shorten >=7pt, Rightarrow, from=9-12, to=4]
	\arrow["\alpha"{description}, shorten >=12pt, Rightarrow, from=1-12, to=5]
	\arrow["\alpha"{description}, shorten >=7pt, Rightarrow, from=1-4, to=7]
	\arrow["{\id_q}"{description}, shorten <=17pt, shorten >=17pt, Rightarrow, from=7, to=6]
\end{tikzcd}
}

which by unicity of the factorisation through $\alpha$ give $\alpha/\alpha'\circ\alpha'/\alpha = \id_q$.
And similarly $\alpha/\alpha'\circ\alpha'/\alpha = \id_{q'}$.
\end{proof}

Furthermore, universal cells compose.

\begin{prop}
Given universal cells

% https://q.uiver.app/?q=WzAsNyxbMCwwLCJcXGJ1bGxldCJdLFszLDAsIlxcZG90cyJdLFsyLDAsIlxcYnVsbGV0Il0sWzQsMCwiXFxidWxsZXQiXSxbNiwwLCJcXGJ1bGxldCJdLFswLDIsIlxcYnVsbGV0Il0sWzYsMiwiXFxidWxsZXQiXSxbMCwyLCJwX3tpLDF9IiwxXSxbMyw0LCJwX3tpLG1faX0iLDFdLFs1LDYsInFfaSIsMV0sWzAsNSwiIiwxLHsibGV2ZWwiOjIsInN0eWxlIjp7ImhlYWQiOnsibmFtZSI6Im5vbmUifX19XSxbNCw2LCIiLDEseyJsZXZlbCI6Miwic3R5bGUiOnsiaGVhZCI6eyJuYW1lIjoibm9uZSJ9fX1dLFsxLDksIlxcYWxwaGFfaSIsMSx7InNob3J0ZW4iOnsidGFyZ2V0IjoyMH19XV0=
\begin{tikzcd}[ampersand replacement=\&]
	\bullet \&\& \bullet \& \dots \& \bullet \&\& \bullet \\
	\\
	\bullet \&\&\&\&\&\& \bullet
	\arrow["{p_{i,1}}"{description}, from=1-1, to=1-3]
	\arrow["{p_{i,m_i}}"{description}, from=1-5, to=1-7]
	\arrow[""{name=0, anchor=center, inner sep=0}, "{q_i}"{description}, from=3-1, to=3-7]
	\arrow[Rightarrow, no head, from=1-1, to=3-1]
	\arrow[Rightarrow, no head, from=1-7, to=3-7]
	\arrow["{\alpha_i}"{description}, shorten >=7pt, Rightarrow, from=1-4, to=0]
\end{tikzcd}

and

% https://q.uiver.app/?q=WzAsNyxbMCwwLCJcXGJ1bGxldCJdLFszLDAsIlxcZG90cyJdLFsyLDAsIlxcYnVsbGV0Il0sWzQsMCwiXFxidWxsZXQiXSxbNiwwLCJcXGJ1bGxldCJdLFswLDIsIlxcYnVsbGV0Il0sWzYsMiwiXFxidWxsZXQiXSxbMCwyLCJxXzEiLDFdLFszLDQsInFfbiIsMV0sWzUsNiwidCIsMV0sWzAsNSwiIiwxLHsibGV2ZWwiOjIsInN0eWxlIjp7ImhlYWQiOnsibmFtZSI6Im5vbmUifX19XSxbNCw2LCIiLDEseyJsZXZlbCI6Miwic3R5bGUiOnsiaGVhZCI6eyJuYW1lIjoibm9uZSJ9fX1dLFsxLDksIlxcYmV0YSIsMSx7InNob3J0ZW4iOnsidGFyZ2V0IjoyMH19XV0=
\begin{tikzcd}[ampersand replacement=\&]
	\bullet \&\& \bullet \& \dots \& \bullet \&\& \bullet \\
	\\
	\bullet \&\&\&\&\&\& \bullet
	\arrow["{q_1}"{description}, from=1-1, to=1-3]
	\arrow["{q_n}"{description}, from=1-5, to=1-7]
	\arrow[""{name=0, anchor=center, inner sep=0}, "t"{description}, from=3-1, to=3-7]
	\arrow[Rightarrow, no head, from=1-1, to=3-1]
	\arrow[Rightarrow, no head, from=1-7, to=3-7]
	\arrow["\beta"{description}, shorten >=7pt, Rightarrow, from=1-4, to=0]
\end{tikzcd}

then

% https://q.uiver.app/?q=WzAsMTgsWzAsMiwiXFxidWxsZXQiXSxbNywyLCJcXGRvdHMiXSxbNiwyLCJcXGJ1bGxldCJdLFs4LDIsIlxcYnVsbGV0Il0sWzE0LDIsIlxcYnVsbGV0Il0sWzAsNCwiXFxidWxsZXQiXSxbMTQsNCwiXFxidWxsZXQiXSxbMCwwLCJcXGJ1bGxldCJdLFsyLDAsIlxcYnVsbGV0Il0sWzMsMCwiXFxkb3RzIl0sWzQsMCwiXFxidWxsZXQiXSxbNiwwLCJcXGJ1bGxldCJdLFs3LDAsIlxcZG90cyJdLFs4LDAsIlxcYnVsbGV0Il0sWzEwLDAsIlxcYnVsbGV0Il0sWzExLDAsIlxcZG90cyJdLFsxMiwwLCJcXGJ1bGxldCJdLFsxNCwwLCJcXGJ1bGxldCJdLFswLDIsInFfMSIsMV0sWzMsNCwicV9uIiwxXSxbNSw2LCJ0IiwxXSxbMCw1LCIiLDEseyJsZXZlbCI6Miwic3R5bGUiOnsiaGVhZCI6eyJuYW1lIjoibm9uZSJ9fX1dLFs0LDYsIiIsMSx7ImxldmVsIjoyLCJzdHlsZSI6eyJoZWFkIjp7Im5hbWUiOiJub25lIn19fV0sWzcsOCwicF97MSwxfSIsMV0sWzEwLDExLCJwX3sxLG1fMX0iLDFdLFsxMywxNCwicF97biwxfSIsMV0sWzE2LDE3LCJwX3tuLG1fbn0iLDFdLFs3LDAsIiIsMSx7ImxldmVsIjoyLCJzdHlsZSI6eyJoZWFkIjp7Im5hbWUiOiJub25lIn19fV0sWzExLDIsIiIsMSx7ImxldmVsIjoyLCJzdHlsZSI6eyJoZWFkIjp7Im5hbWUiOiJub25lIn19fV0sWzEzLDMsIiIsMSx7ImxldmVsIjoyLCJzdHlsZSI6eyJoZWFkIjp7Im5hbWUiOiJub25lIn19fV0sWzE3LDQsIiIsMSx7ImxldmVsIjoyLCJzdHlsZSI6eyJoZWFkIjp7Im5hbWUiOiJub25lIn19fV0sWzEsMjAsIlxcYmV0YSIsMSx7InNob3J0ZW4iOnsidGFyZ2V0IjoyMH19XSxbMTUsMTksIlxcYWxwaGFfbiIsMSx7InNob3J0ZW4iOnsidGFyZ2V0IjoyMH19XSxbOSwxOCwiXFxhbHBoYV8xIiwxLHsic2hvcnRlbiI6eyJ0YXJnZXQiOjIwfX1dXQ==
\scalebox{0.7}{\begin{tikzcd}[ampersand replacement=\&]
	\bullet \&\& \bullet \& \dots \& \bullet \&\& \bullet \& \dots \& \bullet \&\& \bullet \& \dots \& \bullet \&\& \bullet \\
	\\
	\bullet \&\&\&\&\&\& \bullet \& \dots \& \bullet \&\&\&\&\&\& \bullet \\
	\\
	\bullet \&\&\&\&\&\&\&\&\&\&\&\&\&\& \bullet
	\arrow[""{name=0, anchor=center, inner sep=0}, "{q_1}"{description}, from=3-1, to=3-7]
	\arrow[""{name=1, anchor=center, inner sep=0}, "{q_n}"{description}, from=3-9, to=3-15]
	\arrow[""{name=2, anchor=center, inner sep=0}, "t"{description}, from=5-1, to=5-15]
	\arrow[Rightarrow, no head, from=3-1, to=5-1]
	\arrow[Rightarrow, no head, from=3-15, to=5-15]
	\arrow["{p_{1,1}}"{description}, from=1-1, to=1-3]
	\arrow["{p_{1,m_1}}"{description}, from=1-5, to=1-7]
	\arrow["{p_{n,1}}"{description}, from=1-9, to=1-11]
	\arrow["{p_{n,m_n}}"{description}, from=1-13, to=1-15]
	\arrow[Rightarrow, no head, from=1-1, to=3-1]
	\arrow[Rightarrow, no head, from=1-7, to=3-7]
	\arrow[Rightarrow, no head, from=1-9, to=3-9]
	\arrow[Rightarrow, no head, from=1-15, to=3-15]
	\arrow["\beta"{description}, shorten >=7pt, Rightarrow, from=3-8, to=2]
	\arrow["{\alpha_n}"{description}, shorten >=7pt, Rightarrow, from=1-12, to=1]
	\arrow["{\alpha_1}"{description}, shorten >=7pt, Rightarrow, from=1-4, to=0]
\end{tikzcd}}

is universal.
\end{prop}
\begin{proof}
First, we can using associativity and unitality we have:
\begin{align*}
\beta(\alpha_1,\alpha_2,\dots,\alpha_n) &= \beta(\id_{q_1}(\alpha_1),\alpha_2(\id_{p_{2,1}},\dots,\id_{p_{2,m_2}}),\dots,\alpha_n(\id_{p_{n,1}},\dots,\id_{p_{n,m_n}}))\\
&= \beta(\id_{q_1},\alpha_2,\dots,\alpha_n)(\alpha_1,\id_{p_{2,1}},\dots,\id_{p_{n,m_n}})\\
&\vdots \\
&= \beta(\id_{p_{1,1}},\dots,\id_{p_{n-1,m_{n-1}}}, \alpha_n)\dots(\alpha_1,\id_{p_{2,1}},\dots,\id_{p_{n,m_n}})
\end{align*}
i.e. we can arrange for vertical composition of cells to only have a single non-identity cell at each layer:

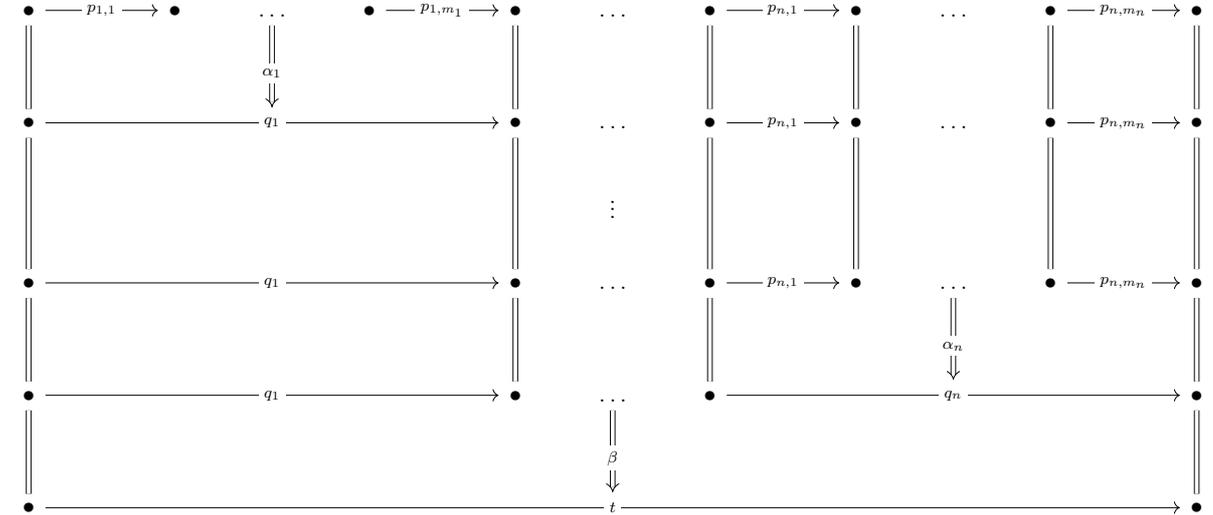
\resizebox{\hsize}{!}{
\begin{tikzcd}[ampersand replacement=\&]
	\bullet \&\& \bullet \& \dots \& \bullet \&\& \bullet \& \dots \& \bullet \&\& \bullet \& \dots \& \bullet \&\& \bullet \\
	\\
	\bullet \&\&\&\&\&\& \bullet \& \dots \& \bullet \&\& \bullet \& \dots \& \bullet \&\& \bullet \\
	\&\&\&\&\&\&\& \vdots \\
	\bullet \&\&\&\&\&\& \bullet \& \dots \& \bullet \&\& \bullet \& \dots \& \bullet \&\& \bullet \\
	\\
	\bullet \&\&\&\&\&\& \bullet \& \dots \& \bullet \&\&\&\&\&\& \bullet \\
	\\
	\bullet \&\&\&\&\&\&\&\&\&\&\&\&\&\& \bullet
	\arrow["{p_{1,1}}"{description}, from=1-1, to=1-3]
	\arrow["{p_{1,m_1}}"{description}, from=1-5, to=1-7]
	\arrow["{p_{n,1}}"{description}, from=1-9, to=1-11]
	\arrow["{p_{n,m_n}}"{description}, from=1-13, to=1-15]
	\arrow[Rightarrow, no head, from=1-1, to=3-1]
	\arrow[Rightarrow, no head, from=1-7, to=3-7]
	\arrow[""{name=0, anchor=center, inner sep=0}, "{q_1}"{description}, from=3-1, to=3-7]
	\arrow[Rightarrow, no head, from=1-9, to=3-9]
	\arrow["{p_{n,1}}"{description}, from=3-9, to=3-11]
	\arrow["{p_{n,m_n}}"{description}, from=3-13, to=3-15]
	\arrow[Rightarrow, no head, from=1-15, to=3-15]
	\arrow[Rightarrow, no head, from=1-13, to=3-13]
	\arrow[Rightarrow, no head, from=1-11, to=3-11]
	\arrow["{q_1}"{description}, from=5-1, to=5-7]
	\arrow["{p_{n,1}}"{description}, from=5-9, to=5-11]
	\arrow["{p_{n,m_n}}"{description}, from=5-13, to=5-15]
	\arrow[Rightarrow, no head, from=3-15, to=5-15]
	\arrow[Rightarrow, no head, from=3-13, to=5-13]
	\arrow[Rightarrow, no head, from=3-11, to=5-11]
	\arrow[Rightarrow, no head, from=3-9, to=5-9]
	\arrow[Rightarrow, no head, from=3-7, to=5-7]
	\arrow[Rightarrow, no head, from=3-1, to=5-1]
	\arrow["{q_1}"{description}, from=7-1, to=7-7]
	\arrow[""{name=1, anchor=center, inner sep=0}, "{q_n}"{description}, from=7-9, to=7-15]
	\arrow[Rightarrow, no head, from=5-1, to=7-1]
	\arrow[Rightarrow, no head, from=5-7, to=7-7]
	\arrow[Rightarrow, no head, from=5-9, to=7-9]
	\arrow[Rightarrow, no head, from=5-15, to=7-15]
	\arrow[""{name=2, anchor=center, inner sep=0}, "t"{description}, from=9-1, to=9-15]
	\arrow[Rightarrow, no head, from=7-1, to=9-1]
	\arrow[Rightarrow, no head, from=7-15, to=9-15]
	\arrow["{\alpha_1}"{description}, shorten >=7pt, Rightarrow, from=1-4, to=0]
	\arrow["{\alpha_n}"{description}, shorten >=7pt, Rightarrow, from=5-12, to=1]
	\arrow["\beta"{description}, shorten >=7pt, Rightarrow, from=7-8, to=2]
\end{tikzcd}
}

Now given a cell:

\resizebox{\hsize}{!}{
% https://q.uiver.app/?q=WzAsMjEsWzYsMCwiXFxidWxsZXQiXSxbOCwwLCJcXGJ1bGxldCJdLFs5LDAsIlxcZG90cyJdLFsxMCwwLCJcXGJ1bGxldCJdLFsxMiwwLCJcXGJ1bGxldCJdLFsxMywwLCJcXGRvdHMiXSxbMTQsMCwiXFxidWxsZXQiXSxbMTYsMCwiXFxidWxsZXQiXSxbMTcsMCwiXFxkb3RzIl0sWzE4LDAsIlxcYnVsbGV0Il0sWzIwLDAsIlxcYnVsbGV0Il0sWzAsMiwiXFxidWxsZXQiXSxbMjYsMiwiXFxidWxsZXQiXSxbNCwwLCJcXGJ1bGxldCJdLFszLDAsIlxcZG90cyJdLFswLDAsIlxcYnVsbGV0Il0sWzIsMCwiXFxidWxsZXQiXSxbMjIsMCwiXFxidWxsZXQiXSxbMjMsMCwiXFxkb3RzIl0sWzI0LDAsIlxcYnVsbGV0Il0sWzI2LDAsIlxcYnVsbGV0Il0sWzAsMSwicF97MSwxfSIsMV0sWzMsNCwicF97MSxtXzF9IiwxXSxbNiw3LCJwX3tuLDF9IiwxXSxbOSwxMCwicF97bixtX259IiwxXSxbMTEsMTIsInQiLDFdLFsxMywwLCJyX2siLDFdLFsxNSwxNiwicl8xIiwxXSxbMTAsMTcsInNfMSIsMV0sWzE5LDIwLCJzX2oiLDFdLFsxNSwxMSwiZiIsMV0sWzIwLDEyLCJnIiwxXSxbNSwyNSwiXFxnYW1tYSIsMSx7InNob3J0ZW4iOnsidGFyZ2V0IjoyMH19XV0=
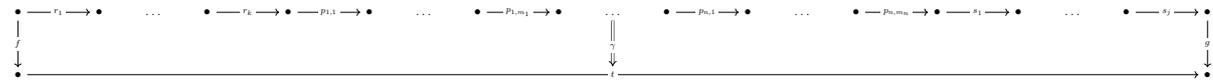
\begin{tikzcd}[ampersand replacement=\&]
	\bullet \&\& \bullet \& \dots \& \bullet \&\& \bullet \&\& \bullet \& \dots \& \bullet \&\& \bullet \& \dots \& \bullet \&\& \bullet \& \dots \& \bullet \&\& \bullet \&\& \bullet \& \dots \& \bullet \&\& \bullet \\
	\\
	\bullet \&\&\&\&\&\&\&\&\&\&\&\&\&\&\&\&\&\&\&\&\&\&\&\&\&\& \bullet
	\arrow["{p_{1,1}}"{description}, from=1-7, to=1-9]
	\arrow["{p_{1,m_1}}"{description}, from=1-11, to=1-13]
	\arrow["{p_{n,1}}"{description}, from=1-15, to=1-17]
	\arrow["{p_{n,m_n}}"{description}, from=1-19, to=1-21]
	\arrow[""{name=0, anchor=center, inner sep=0}, "t"{description}, from=3-1, to=3-27]
	\arrow["{r_k}"{description}, from=1-5, to=1-7]
	\arrow["{r_1}"{description}, from=1-1, to=1-3]
	\arrow["{s_1}"{description}, from=1-21, to=1-23]
	\arrow["{s_j}"{description}, from=1-25, to=1-27]
	\arrow["f"{description}, from=1-1, to=3-1]
	\arrow["g"{description}, from=1-27, to=3-27]
	\arrow["\gamma"{description}, shorten >=7pt, Rightarrow, from=1-14, to=0]
\end{tikzcd}
}

since $\alpha_1$ is universal we can factor $\gamma$ through it:

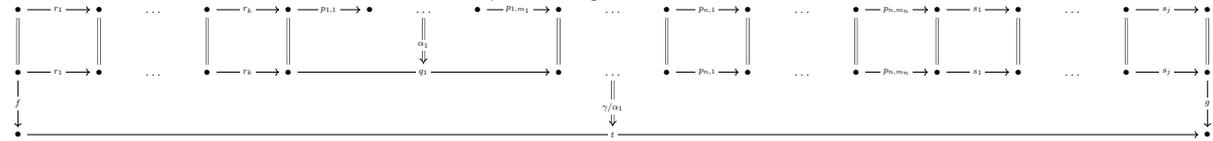
\resizebox{\hsize}{!}{
\begin{tikzcd}[ampersand replacement=\&]
	\bullet \&\& \bullet \& \dots \& \bullet \&\& \bullet \&\& \bullet \& \dots \& \bullet \&\& \bullet \& \dots \& \bullet \&\& \bullet \& \dots \& \bullet \&\& \bullet \&\& \bullet \& \dots \& \bullet \&\& \bullet \\
	\\
	\bullet \&\& \bullet \& \dots \& \bullet \&\& \bullet \&\&\&\&\&\& \bullet \& \dots \& \bullet \&\& \bullet \& \dots \& \bullet \&\& \bullet \&\& \bullet \& \dots \& \bullet \&\& \bullet \\
	\\
	\bullet \&\&\&\&\&\&\&\&\&\&\&\&\&\&\&\&\&\&\&\&\&\&\&\&\&\& \bullet
	\arrow["{p_{1,1}}"{description}, from=1-7, to=1-9]
	\arrow["{p_{1,m_1}}"{description}, from=1-11, to=1-13]
	\arrow["{p_{n,1}}"{description}, from=1-15, to=1-17]
	\arrow["{p_{n,m_n}}"{description}, from=1-19, to=1-21]
	\arrow[""{name=0, anchor=center, inner sep=0}, "t"{description}, from=5-1, to=5-27]
	\arrow["{r_k}"{description}, from=1-5, to=1-7]
	\arrow["{r_1}"{description}, from=1-1, to=1-3]
	\arrow["{s_1}"{description}, from=1-21, to=1-23]
	\arrow["{s_j}"{description}, from=1-25, to=1-27]
	\arrow["{r_1}"{description}, from=3-1, to=3-3]
	\arrow["{r_k}"{description}, from=3-5, to=3-7]
	\arrow[""{name=1, anchor=center, inner sep=0}, "{q_1}"{description}, from=3-7, to=3-13]
	\arrow["{p_{n,1}}"{description}, from=3-15, to=3-17]
	\arrow["{p_{n,m_n}}"{description}, from=3-19, to=3-21]
	\arrow["{s_1}"{description}, from=3-21, to=3-23]
	\arrow["{s_j}"{description}, from=3-25, to=3-27]
	\arrow[Rightarrow, no head, from=1-1, to=3-1]
	\arrow[Rightarrow, no head, from=1-3, to=3-3]
	\arrow[Rightarrow, no head, from=1-5, to=3-5]
	\arrow[Rightarrow, no head, from=1-13, to=3-13]
	\arrow[Rightarrow, no head, from=1-15, to=3-15]
	\arrow[Rightarrow, no head, from=1-17, to=3-17]
	\arrow[Rightarrow, no head, from=1-19, to=3-19]
	\arrow[Rightarrow, no head, from=1-7, to=3-7]
	\arrow[Rightarrow, no head, from=1-21, to=3-21]
	\arrow[Rightarrow, no head, from=1-23, to=3-23]
	\arrow[Rightarrow, no head, from=1-25, to=3-25]
	\arrow[Rightarrow, no head, from=1-27, to=3-27]
	\arrow["f"{description}, from=3-1, to=5-1]
	\arrow["g"{description}, from=3-27, to=5-27]
	\arrow["{\alpha_1}"{description}, shorten >=7pt, Rightarrow, from=1-10, to=1]
	\arrow["{\gamma/\alpha_1}"{description}, shorten >=7pt, Rightarrow, from=3-14, to=0]
\end{tikzcd}
}

But then since $\alpha_2$ is universal we can factorise $\gamma/\alpha_1$ by it and so on.
Finally, we get a unique cell $(((\gamma/\alpha_1)/\dots)/\alpha_n)/\beta$ such that $\gamma$ factorise through it and $\beta(\alpha_1,\dots,\alpha_n)$:

\resizebox{\hsize}{!}{
\begin{tikzcd}[ampersand replacement=\&]
	\bullet \&\& \bullet \& \dots \& \bullet \&\& \bullet \&\& \bullet \& \dots \& \bullet \&\& \bullet \& \dots \& \bullet \&\& \bullet \& \dots \& \bullet \&\& \bullet \&\& \bullet \& \dots \& \bullet \&\& \bullet \\
	\\
	\bullet \&\& \bullet \& \dots \& \bullet \&\& \bullet \&\&\&\&\&\&\&\&\&\&\&\&\&\& \bullet \&\& \bullet \& \dots \& \bullet \&\& \bullet \\
	\\
	\bullet \&\&\&\&\&\&\&\&\&\&\&\&\&\&\&\&\&\&\&\&\&\&\&\&\&\& \bullet
	\arrow["{p_{1,1}}"{description}, from=1-7, to=1-9]
	\arrow["{p_{1,m_1}}"{description}, from=1-11, to=1-13]
	\arrow["{p_{n,1}}"{description}, from=1-15, to=1-17]
	\arrow["{p_{n,m_n}}"{description}, from=1-19, to=1-21]
	\arrow[""{name=0, anchor=center, inner sep=0}, "t"{description}, from=5-1, to=5-27]
	\arrow["{r_k}"{description}, from=1-5, to=1-7]
	\arrow["{r_1}"{description}, from=1-1, to=1-3]
	\arrow["{s_1}"{description}, from=1-21, to=1-23]
	\arrow["{s_j}"{description}, from=1-25, to=1-27]
	\arrow["{r_1}"{description}, from=3-1, to=3-3]
	\arrow["{r_k}"{description}, from=3-5, to=3-7]
	\arrow["{s_1}"{description}, from=3-21, to=3-23]
	\arrow["{s_j}"{description}, from=3-25, to=3-27]
	\arrow[Rightarrow, no head, from=1-1, to=3-1]
	\arrow[Rightarrow, no head, from=1-3, to=3-3]
	\arrow[Rightarrow, no head, from=1-5, to=3-5]
	\arrow[Rightarrow, no head, from=1-7, to=3-7]
	\arrow[Rightarrow, no head, from=1-21, to=3-21]
	\arrow[Rightarrow, no head, from=1-23, to=3-23]
	\arrow[Rightarrow, no head, from=1-25, to=3-25]
	\arrow[Rightarrow, no head, from=1-27, to=3-27]
	\arrow["f"{description}, from=3-1, to=5-1]
	\arrow["g"{description}, from=3-27, to=5-27]
	\arrow[""{name=1, anchor=center, inner sep=0}, "s"{description}, from=3-7, to=3-21]
	\arrow["{\beta(\alpha_1,\dots,\alpha_n)}"{description}, shorten >=7pt, Rightarrow, from=1-14, to=1]
	\arrow["{(((\gamma/\alpha_1)/\dots)/\alpha_n)/\beta}"{description}, shorten <=9pt, shorten >=9pt, Rightarrow, from=1, to=0]
\end{tikzcd}
}
\end{proof}

We also have that the identity cell is universal.

\begin{prop}
The identity cell $1_p$ is universal.
\end{prop}
\begin{proof}
It follows directly from unitality of the identity cell.
\end{proof}

\begin{rk}
To gain some space we will not write the objects of the cells when it is not necessary.
\end{rk}

We will write $1_A$ for the 0-composite, i.e. the codomain of a universal cell with empty source:

% https://q.uiver.app/?q=WzAsMyxbMCwyLCJBIl0sWzIsMiwiQSJdLFsxLDAsIkEiXSxbMCwxLCIxX0EiLDFdLFsyLDAsIiIsMSx7ImxldmVsIjoyLCJzdHlsZSI6eyJoZWFkIjp7Im5hbWUiOiJub25lIn19fV0sWzIsMSwiIiwxLHsibGV2ZWwiOjIsInN0eWxlIjp7ImhlYWQiOnsibmFtZSI6Im5vbmUifX19XSxbMiwzLCIiLDEseyJzaG9ydGVuIjp7InRhcmdldCI6MjB9fV1d
\begin{tikzcd}[ampersand replacement=\&]
	\& A \\
	\\
	A \&\& A
	\arrow[""{name=0, anchor=center, inner sep=0}, "{1_A}"{description}, from=3-1, to=3-3]
	\arrow[Rightarrow, no head, from=1-2, to=3-1]
	\arrow[Rightarrow, no head, from=1-2, to=3-3]
	\arrow[shorten >=7pt, Rightarrow, from=1-2, to=0]
\end{tikzcd}

\begin{cor}
For any horizontal morphisms $f,g,h$ we have invertible cells:
\begin{itemize}
\item $f \bullet 1_A \simeq f \bullet 1_B$
\item $(h \bullet g) \bullet f \simeq h \bullet (g \bullet f)$.
\end{itemize}
Those cells satisfy the coherence law of a double category.
\end{cor}

\begin{proof}
Since universal cells compose and composites are unique up to unique invertible cell we get the cells.
The coherence law follow from unicity of the factorisation by a universal cell.
\end{proof}

\begin{definition}
A vdc that admits a composite for any chain of horizontal morphisms is said to be \defin{representable}.
\end{definition}

Given a vdc \vdC, we write $\vdH(\vdC)$ for the data of:
\begin{itemize}
\item the objects of \vdC
\item the horizontal morphisms
\item the cells whose vertical morphisms are identities
\end{itemize}

\begin{prop}
For a vdc \vdC, if \vdC is representable then $\vdH(\vdC)$ forms a bicategory.
\end{prop}

\begin{example}
The underlying vdc of a bicategory is representable with $\bullet$ given by horizontal composition.
In fact, the vertical category of a representable vdc is discrete iff it is the underlying vdc of a bicategory.
\end{example}

\begin{example}
The delooping of a multicategory is representable iff the multicategory is.
Horizontal composition is given by tensor product.
\end{example}

\begin{example}
$\Dist$ is representable with horizontal composition given by composition of distributors, i.e. by a coend.
\end{example}

\begin{example}
Recall that the terminal vdc has exactly one object $\ast$, one vertical morphism $\id_\ast$, one horizontal morphism $\ast \to \ast$ and one cell for each arity:
\[% https://q.uiver.app/?q=WzAsNyxbMCwwLCJcXGFzdCJdLFsyLDAsIlxcYXN0Il0sWzMsMCwiXFxkb3RzIl0sWzQsMCwiXFxhc3QiXSxbNiwwLCJcXGFzdCJdLFswLDIsIlxcYXN0Il0sWzYsMiwiXFxhc3QiXSxbMCwxXSxbMyw0XSxbMCw1LCIiLDEseyJsZXZlbCI6Miwic3R5bGUiOnsiaGVhZCI6eyJuYW1lIjoibm9uZSJ9fX1dLFs0LDYsIiIsMSx7ImxldmVsIjoyLCJzdHlsZSI6eyJoZWFkIjp7Im5hbWUiOiJub25lIn19fV0sWzUsNl0sWzIsMTEsIlxcdW5kZXJsaW5le259IiwxLHsic2hvcnRlbiI6eyJ0YXJnZXQiOjIwfX1dXQ==
\begin{tikzcd}[ampersand replacement=\&]
	\ast \&\& \ast \& \dots \& \ast \&\& \ast \\
	\\
	\ast \&\&\&\&\&\& \ast
	\arrow[from=1-1, to=1-3]
	\arrow[from=1-5, to=1-7]
	\arrow[Rightarrow, no head, from=1-1, to=3-1]
	\arrow[Rightarrow, no head, from=1-7, to=3-7]
	\arrow[""{name=0, anchor=center, inner sep=0}, from=3-1, to=3-7]
	\arrow["{\underline{n}}"{description}, shorten >=7pt, Rightarrow, from=1-4, to=0]
\end{tikzcd}
\]
It is representable with $1_\ast := \ast \to \ast$ and $(\ast \to \ast)\bullet \dots \bullet (\ast \to \ast) := \ast \to \ast$.
In particular, its vertical category is $\one$ the terminal category and its horizontal bicategory is $\one$ the terminal bicategory.
\end{example}

\begin{example}
\Rel{} is representable where the horizontal composition is the relation given by $a (R_1 \bullet R_0) c$ iff $\exists b,\ a R_0 b \wedge b R_1 c$.
\end{example}

\begin{example}
For a category with pullbacks \cC, $\Span(\cC)$ is representable with horizontal composition given by pullback.
\end{example}

\begin{example}
\Ring{} is representable with composition of bimodules given by their tensor product.
By definition, the tensor product of bimodules is characterised by the universal property of ``linearising'' balanced multimorphisms, i.e. it is equipped with a balanced multimorphism $M_1, \dots, M_n \to M_1 \otimes_{A_1} \dots \otimes_{A_{n-1}} M_n$ that gives an equivalence between balanced multimorphisms and balanced (linear) morphisms out of the tensor product of the inputs.

Let us construct the tensor product of the $(A_{i-1},A_i)$-bimodules $M_i$.
In the following we will use the correspondence between abelian groups and $\Z$-modules.
First let $\langle M_i \rangle$ be the free $\Z$-module on $M_1 \times \dots \times M_n$.
So its elements consist of formal sums $\sum_i k_i (m_{i,1},\dots, m_{i_n})$ with $k_i \in \Z$ and $m_{i,j} \in M_j$.
We can quotient it by the ideal generated by elements of the form:
\begin{itemize}
\item $(m_1, \dots, m_i + m_i', \dots, m_n) - (m_1,\dots,m_i, \dots, m_n) - (m_1,\dots,m_i',\dots,m_n)$
\item $(m_1, \dots, m_{i-1} \cdot a_i, \dots m_n) - (m_1,\dots,m_{i-1}, a_i \cdot m_i, m_n)$
\item $(m_1,\dots,m_{i-1},0,m_{i+1},\dots,m_n)$
\end{itemize}
The abelian group that we get is the underlying abelian group of the tensor product $M_1 \otimes_{A_1} \dots \otimes_{A_n} M_n$.
We write its elements $\sum\limits_i m_{i,1} \otimes \dots, \otimes m_{i,n}$.
Notice that for a multivariable function $f \colon M_1 \times \dots \times M_n \to N$, we get a unique $\Z$-module morphism out of $\langle M_i \rangle$ whose kernel include $I$ iff $f$ is balanced and additive.
So for any balanced additive map we have the following universal property:
% https://q.uiver.app/?q=WzAsNCxbMCwwLCJNXzEgXFx0aW1lcyBcXGRvdHMgXFx0aW1lcyBNX24iXSxbMSwxLCJHIl0sWzAsMSwiXFxsYW5nbGUgTV9pXFxyYW5nbGUiXSxbMCwyLCJNXzFcXG90aW1lc197QV8xfVxcZG90c1xcb3RpbWVzX3tBX3tuLTF9fU1fbiJdLFswLDEsIlxcdGV4dHtiYWwuICsgYWRkLn0iLDFdLFswLDIsIiIsMSx7InN0eWxlIjp7InRhaWwiOnsibmFtZSI6Imhvb2siLCJzaWRlIjoidG9wIn19fV0sWzIsMSwiIiwxLHsic3R5bGUiOnsiYm9keSI6eyJuYW1lIjoiZGFzaGVkIn19fV0sWzIsMywiIiwxLHsic3R5bGUiOnsidGFpbCI6eyJuYW1lIjoiaG9vayIsInNpZGUiOiJ0b3AifX19XSxbMywxLCIiLDEseyJzdHlsZSI6eyJib2R5Ijp7Im5hbWUiOiJkYXNoZWQifX19XV0=
\[\begin{tikzcd}[ampersand replacement=\&]
	{M_1 \times \dots \times M_n} \\
	{\langle M_i\rangle} \& G \\
	{M_1\otimes_{A_1}\dots\otimes_{A_{n-1}}M_n}
	\arrow["{\text{bal. + add.}}"{description}, from=1-1, to=2-2]
	\arrow[hook, from=1-1, to=2-1]
	\arrow[dashed, from=2-1, to=2-2]
	\arrow[hook, from=2-1, to=3-1]
	\arrow[dashed, from=3-1, to=2-2]
\end{tikzcd}\]
Following from the universal property of free module and quotient module.

We still need to define the $(A_0,A_n)$-bimodule structure on $M_1 \otimes_{A_1} \dots \otimes_{A_{n-1}} M_n$.
The left action is inherited by the left action on $M_1$ in the following way:
\[a_0 \cdot \sum\limits_i m_{i,1} \otimes \dots \otimes m_{i,n} := \sum\limits_i (a_0 \cdot m_{i,1}) \otimes \dots \otimes m_{i_n}\] and similarly for the right action of $A_n$
This is well-defined: suppose that there are two representations of an element of the tensor products $\sum\limits_i m_{i,1} \otimes \dots \otimes m_{i,n} = \sum\limits_j m'_{j,1} \otimes \dots \otimes m_{j,n}$ we want to prove that \[\sum\limits_i (a_0 \cdot m_{i,1}) \otimes \dots \otimes m_{i,n} = \sum\limits_j (a_0 \cdot m'_{j,1}) \otimes \dots \otimes m_{j,n}\] for any $a_0$.
This can be done by considering the action of $A_0$ on $M_1 \times \dots \times M_n$ defined by $(a_0 \cdot m_1,\dots, m_n)$, extend it to $\langle M_i \rangle$ and then prove that it sends generators of the ideal for balanced additiveness to other generators.

In particular, for $R$ a commutative ring and considering a $R$-module as a $(R,R)$-bimodule, this gives the usual tensor product of modules and balanced multimorphisms gives multilinear maps.
\end{example}

\begin{rk}
There is a subtlety about the previous example.
In the format above it does not quite work for the case $n=0$:
% https://q.uiver.app/?q=WzAsMyxbMSwwLCJBXzAiXSxbMCwxLCJBXzAiXSxbMiwxLCJBXzAiXSxbMCwxLCIiLDEseyJsZXZlbCI6Miwic3R5bGUiOnsiaGVhZCI6eyJuYW1lIjoibm9uZSJ9fX1dLFswLDIsIiIsMSx7ImxldmVsIjoyLCJzdHlsZSI6eyJoZWFkIjp7Im5hbWUiOiJub25lIn19fV0sWzEsMiwiMV97QV8wfSIsMV1d
\[\begin{tikzcd}[ampersand replacement=\&]
	\& {A_0} \\
	{A_0} \&\& {A_0}
	\arrow[Rightarrow, no head, from=1-2, to=2-1]
	\arrow[Rightarrow, no head, from=1-2, to=2-3]
	\arrow["{1_{A_0}}"{description}, from=2-1, to=2-3]
\end{tikzcd}\]
Indeed, when performing the construction above for $n=0$, we get that $1_{A_0}$ is \Z{} equipped with a trivial action by $A_0$: $a_0\bullet z \bullet a_0' = z$.
This forgets everything about $A_0$.
What we would want instead is for $1_{A_0}$ to be $A_0$ considered as an $(A_0,A_0)$-bimodule by multiplication.
We can remedy that in two ways.
One would be to simply separate the case: give the $n$-ary tensor product as above for $n > 0$, and take $A_0$ as a bimodule for $n=0$.
A more uniform approach is to define the $n$-ary composition to be $A_0 \otimes_{A_0} M_1 \otimes_{A_1} \dots \otimes_{A_{n-1}} M_n \otimes_{A_n} A_n$ where $A_0$ and $A_n$ are considered as bimodules with multiplication.
It can then be checked that the $(A,A)$-bimodule $A$ is a unit for $\otimes_A$.
So in the case $n>0$ we get \[A_0 \otimes_{A_0} M_1 \otimes_{A_1} \dots \otimes_{A_{n-1}} M_n \otimes_{A_n} A_n \simeq M_1 \otimes_{A_1} \dots \otimes_{A_{n-1}} M_n\] and in the case $n= 0$ we get \[A_0 \otimes_{A_0} A_0 \simeq A_0\].

This subtlety also arises for $\Span(\cC)$ and $\Dist$.
We can also fix it in two ways, by considering the nullary case independently or by composing by $1_{A}$ on both part.
For $\Span(\cC)$ it will be $A_0 \times_{A_0} P_0 \times_{A_1} \dots \times_{A_{n-1}} P_n \times_{A_n} A_n$ and for $\Dist$, \[\int^{a_i} A_0(-,a_0) \times P_1(a_0,a_1) \times \dots \times P_n(a_{n-1},a_n) \times A_n(a_n,-)\]
\end{rk}

\begin{example}
For any category with reflexive coequalisers preserved by tensor product in each variable, we can define a tensor product of bimodules (of internal monoids) by the reflexive coequaliser:
% https://q.uiver.app/?q=WzAsMyxbMCwwLCJNXzEgXFxvdGltZXMgQV8xIFxcb3RpbWVzIE1fMiJdLFsyLDAsIk1fMSBcXG90aW1lcyBNXzIiXSxbNCwwLCJNXzEgXFxvdGltZXNfe0FfMX0gTV8yIl0sWzAsMSwiXFxjZG90X3tBXzF9IFxcb3RpbWVzIE1fMiIsMCx7Im9mZnNldCI6LTJ9XSxbMCwxLCJNXzEgXFxvdGltZXMge31fe0FfMX1cXGNkb3QiLDIseyJvZmZzZXQiOjJ9XSxbMSwyXV0=
\[\begin{tikzcd}[ampersand replacement=\&]
	{M_1 \otimes A_1 \otimes M_2} \&\& {M_1 \otimes M_2} \&\& {M_1 \otimes_{A_1} M_2}
	\arrow["{\cdot_{A_1} \otimes M_2}", shift left=2, from=1-1, to=1-3]
	\arrow["{M_1 \otimes {}_{A_1}\cdot}"', shift right=2, from=1-1, to=1-3]
	\arrow[from=1-3, to=1-5]
\end{tikzcd}\]
\end{example}

A functor of vdcs is an assignment on objects, vertical and horizontal morphisms, and cells that preserves all the structure involved.

\begin{definition}
A functor of vdcs $F \colon \vdC \to \vdD$ is given by:
\begin{itemize}
\item a functor $F \colon \vdV(\vdC) \to \vdV(\vdD)$ between the vertical categories
\item for any horizontal morphism $p \colon A_0 \to A_1$ in \vdC, one $F(p) \colon F(A_0) \to F(A_1)$ in \vdD
\item for any cell

% https://q.uiver.app/?q=WzAsNyxbMCwwLCJBXzAiXSxbMiwwLCJBXzEiXSxbMywwLCJcXGRvdHMiXSxbNCwwLCJBX3tuLTF9Il0sWzYsMCwiQV9uIl0sWzAsMiwiQl8wIl0sWzYsMiwiQl8xIl0sWzAsMSwicF8xIiwxXSxbMyw0LCJwX24iLDFdLFswLDUsImYiLDFdLFs0LDYsImciLDFdLFs1LDYsInEiLDFdLFsyLDExLCJcXGFscGhhIiwxLHsic2hvcnRlbiI6eyJ0YXJnZXQiOjIwfX1dXQ==
\begin{tikzcd}[ampersand replacement=\&]
	{A_0} \&\& {A_1} \& \dots \& {A_{n-1}} \&\& {A_n} \\
	\\
	{B_0} \&\&\&\&\&\& {B_1}
	\arrow["{p_1}"{description}, from=1-1, to=1-3]
	\arrow["{p_n}"{description}, from=1-5, to=1-7]
	\arrow["f"{description}, from=1-1, to=3-1]
	\arrow["g"{description}, from=1-7, to=3-7]
	\arrow[""{name=0, anchor=center, inner sep=0}, "q"{description}, from=3-1, to=3-7]
	\arrow["\alpha"{description}, shorten >=7pt, Rightarrow, from=1-4, to=0]
\end{tikzcd}

in \vdC, one in \vdD:

% https://q.uiver.app/?q=WzAsNyxbMCwwLCJGKEFfMCkiXSxbMiwwLCJGKEFfMSkiXSxbMywwLCJcXGRvdHMiXSxbNCwwLCJGKEFfe24tMX0pIl0sWzYsMCwiRihBX24pIl0sWzAsMiwiRihCXzApIl0sWzYsMiwiRihCXzEpIl0sWzAsMSwiRihwXzEpIiwxXSxbMyw0LCJGKHBfbikiLDFdLFswLDUsIkYoZikiLDFdLFs0LDYsIkYoZykiLDFdLFs1LDYsIkYocSkiLDFdLFsyLDExLCJGKFxcYWxwaGEpIiwxLHsic2hvcnRlbiI6eyJ0YXJnZXQiOjIwfX1dXQ==
\begin{tikzcd}[ampersand replacement=\&]
	{F(A_0)} \&\& {F(A_1)} \& \dots \& {F(A_{n-1})} \&\& {F(A_n)} \\
	\\
	{F(B_0)} \&\&\&\&\&\& {F(B_1)}
	\arrow["{F(p_1)}"{description}, from=1-1, to=1-3]
	\arrow["{F(p_n)}"{description}, from=1-5, to=1-7]
	\arrow["{F(f)}"{description}, from=1-1, to=3-1]
	\arrow["{F(g)}"{description}, from=1-7, to=3-7]
	\arrow[""{name=0, anchor=center, inner sep=0}, "{F(q)}"{description}, from=3-1, to=3-7]
	\arrow["{F(\alpha)}"{description}, shorten >=7pt, Rightarrow, from=1-4, to=0]
\end{tikzcd}
\end{itemize}
such that
\begin{itemize}
\item $F(\id_p) = \id_{F(p)}$
\item $F(\beta(\alpha_1,\dots,\alpha_n)) = F(\beta)(F(\alpha_1),\dots,F(\alpha_n))$
\end{itemize}
\end{definition}

Consider a functor of representable vdcs $F \colon \vdC \to \vdD$.
Then we can consider the image of a universal cell from \vdC:

% https://q.uiver.app/?q=WzAsNyxbMCwwLCJGKEFfMCkiXSxbMiwwLCJGKEFfMSkiXSxbMywwLCJcXGRvdHMiXSxbNCwwLCJGKEFfe24tMX0pIl0sWzYsMCwiRihBX24pIl0sWzAsMiwiRihBXzApIl0sWzYsMiwiRihBX24pIl0sWzAsMSwiRihwXzEpIiwxXSxbMyw0LCJGKHBfbikiLDFdLFs1LDYsIkYocF9uIFxcYnVsbGV0XFxkb3RzXFxidWxsZXQgcF8xKSIsMV0sWzQsNiwiIiwxLHsibGV2ZWwiOjIsInN0eWxlIjp7ImhlYWQiOnsibmFtZSI6Im5vbmUifX19XSxbMCw1LCIiLDEseyJsZXZlbCI6Miwic3R5bGUiOnsiaGVhZCI6eyJuYW1lIjoibm9uZSJ9fX1dXQ==
\begin{tikzcd}[ampersand replacement=\&]
	{F(A_0)} \&\& {F(A_1)} \& \dots \& {F(A_{n-1})} \&\& {F(A_n)} \\
	\\
	{F(A_0)} \&\&\&\&\&\& {F(A_n)}
	\arrow["{F(p_1)}"{description}, from=1-1, to=1-3]
	\arrow["{F(p_n)}"{description}, from=1-5, to=1-7]
	\arrow["{F(p_n \bullet\dots\bullet p_1)}"{description}, from=3-1, to=3-7]
	\arrow[Rightarrow, no head, from=1-7, to=3-7]
	\arrow[Rightarrow, no head, from=1-1, to=3-1]
\end{tikzcd}

Then it can be factored uniquely through the universal cell in \vdD:

% https://q.uiver.app/?q=WzAsOSxbMCwwLCJGKEFfMCkiXSxbMiwwLCJGKEFfMSkiXSxbMywwLCJcXGRvdHMiXSxbNCwwLCJGKEFfe24tMX0pIl0sWzYsMCwiRihBX24pIl0sWzAsNCwiRihBXzApIl0sWzYsNCwiRihBX24pIl0sWzAsMiwiRihBXzApIl0sWzYsMiwiRihBX24pIl0sWzAsMSwiRihwXzEpIiwxXSxbMyw0LCJGKHBfbikiLDFdLFs1LDYsIkYocF9uIFxcYnVsbGV0XFxkb3RzXFxidWxsZXQgcF8xKSIsMV0sWzAsNywiIiwxLHsibGV2ZWwiOjIsInN0eWxlIjp7ImhlYWQiOnsibmFtZSI6Im5vbmUifX19XSxbNCw4LCIiLDEseyJsZXZlbCI6Miwic3R5bGUiOnsiaGVhZCI6eyJuYW1lIjoibm9uZSJ9fX1dLFs3LDUsIiIsMSx7ImxldmVsIjoyLCJzdHlsZSI6eyJoZWFkIjp7Im5hbWUiOiJub25lIn19fV0sWzgsNiwiIiwxLHsibGV2ZWwiOjIsInN0eWxlIjp7ImhlYWQiOnsibmFtZSI6Im5vbmUifX19XSxbNyw4LCJGKHBfbilcXGJ1bGxldFxcZG90c1xcYnVsbGV0IEYocF8xKSIsMV0sWzE2LDExLCIiLDEseyJzaG9ydGVuIjp7InNvdXJjZSI6MjAsInRhcmdldCI6MjB9fV0sWzIsMTYsIiIsMSx7InNob3J0ZW4iOnsidGFyZ2V0IjoyMH19XV0=
\begin{tikzcd}[ampersand replacement=\&]
	{F(A_0)} \&\& {F(A_1)} \& \dots \& {F(A_{n-1})} \&\& {F(A_n)} \\
	\\
	{F(A_0)} \&\&\&\&\&\& {F(A_n)} \\
	\\
	{F(A_0)} \&\&\&\&\&\& {F(A_n)}
	\arrow["{F(p_1)}"{description}, from=1-1, to=1-3]
	\arrow["{F(p_n)}"{description}, from=1-5, to=1-7]
	\arrow[""{name=0, anchor=center, inner sep=0}, "{F(p_n \bullet\dots\bullet p_1)}"{description}, from=5-1, to=5-7]
	\arrow[Rightarrow, no head, from=1-1, to=3-1]
	\arrow[Rightarrow, no head, from=1-7, to=3-7]
	\arrow[Rightarrow, no head, from=3-1, to=5-1]
	\arrow[Rightarrow, no head, from=3-7, to=5-7]
	\arrow[""{name=1, anchor=center, inner sep=0}, "{F(p_n)\bullet\dots\bullet F(p_1)}"{description}, from=3-1, to=3-7]
	\arrow[shorten <=9pt, shorten >=9pt, Rightarrow, from=1, to=0]
	\arrow[shorten >=7pt, Rightarrow, from=1-4, to=1]
\end{tikzcd}

These cells obtained by factorisation verify coherence laws because of the unicity of said factorisation.
This correspond to the definition of (lax) functor of double categories.
If these cells are invertible/identities then we talk of pseudo/strict functor.
If these cells are identities only for the units $1$, i.e. the nullary composite, we say that the functor is normal. 

\begin{example}
Functors between the underlying vdc of bicategories correspond to (lax) functors of bicategories.
They are pseudo/strict as functors of vdcs iff they are as functors of bicategories.
Similarly for monoidal categories.
\end{example}

\section{Monoids and modules}

Given a vdc \vdC, we can form a new vdc \Mon{\vdC} whose objects are monoids in \vdC.
This generalises the construction that produce a monoidal category of monoids internal to a monoidal category.
Some examples of interest for us arise in this way.

\begin{definition}
Given a vdc \vdC, a monoid in \vdC is an object $M$ and a horizontal morphisms $m \colon M \to M$ equipped with cells:
% https://q.uiver.app/?q=WzAsOCxbMCwyLCJNIl0sWzIsMiwiTSJdLFsxLDAsIk0iXSxbNSwwLCJNIl0sWzcsMCwiTSJdLFs5LDAsIk0iXSxbNSwyLCJNIl0sWzksMiwiTSJdLFsyLDAsIiIsMSx7ImxldmVsIjoyLCJzdHlsZSI6eyJoZWFkIjp7Im5hbWUiOiJub25lIn19fV0sWzIsMSwiIiwxLHsibGV2ZWwiOjIsInN0eWxlIjp7ImhlYWQiOnsibmFtZSI6Im5vbmUifX19XSxbMCwxLCJtIiwxXSxbNiw3LCJtIiwxXSxbMyw2LCIiLDEseyJsZXZlbCI6Miwic3R5bGUiOnsiaGVhZCI6eyJuYW1lIjoibm9uZSJ9fX1dLFs1LDcsIiIsMSx7ImxldmVsIjoyLCJzdHlsZSI6eyJoZWFkIjp7Im5hbWUiOiJub25lIn19fV0sWzMsNCwibSIsMV0sWzQsNSwibSIsMV0sWzIsMTAsIm1fMCIsMSx7InNob3J0ZW4iOnsidGFyZ2V0IjoyMH19XSxbNCwxMSwibV8yIiwxLHsic2hvcnRlbiI6eyJ0YXJnZXQiOjIwfX1dXQ==
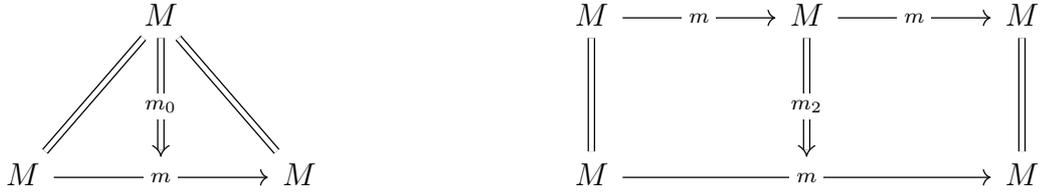
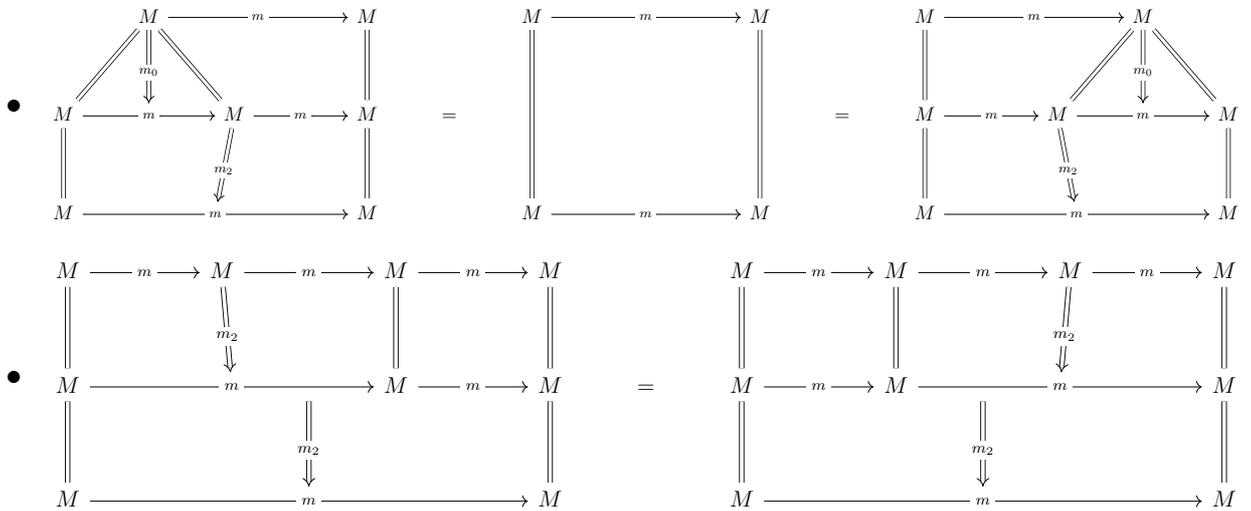
\[\begin{tikzcd}[ampersand replacement=\&]
	\& M \&\&\&\& M \&\& M \&\& M \\
	\\
	M \&\& M \&\&\& M \&\&\&\& M
	\arrow[Rightarrow, no head, from=1-2, to=3-1]
	\arrow[Rightarrow, no head, from=1-2, to=3-3]
	\arrow[""{name=0, anchor=center, inner sep=0}, "m"{description}, from=3-1, to=3-3]
	\arrow[""{name=1, anchor=center, inner sep=0}, "m"{description}, from=3-6, to=3-10]
	\arrow[Rightarrow, no head, from=1-6, to=3-6]
	\arrow[Rightarrow, no head, from=1-10, to=3-10]
	\arrow["m"{description}, from=1-6, to=1-8]
	\arrow["m"{description}, from=1-8, to=1-10]
	\arrow["{m_0}"{description}, shorten >=7pt, Rightarrow, from=1-2, to=0]
	\arrow["{m_2}"{description}, shorten >=7pt, Rightarrow, from=1-8, to=1]
\end{tikzcd}\]
satisfying:
\begin{itemize}
\item
\resizebox{\hsize}{!}{
\begin{tikzcd}[ampersand replacement=\&]
	\& M \&\&\& M \&\& M \&\&\&\& M \&\& M \&\&\& M \\
	\\
	M \&\& M \&\& M \& {=} \&\&\&\&\&\& {=} \& M \&\& M \&\& M \\
	\\
	M \&\&\&\& M \&\& M \&\&\&\& M \&\& M \&\&\&\& M
	\arrow[""{name=0, anchor=center, inner sep=0}, "m"{description}, from=5-1, to=5-5]
	\arrow[Rightarrow, no head, from=3-1, to=5-1]
	\arrow[Rightarrow, no head, from=3-5, to=5-5]
	\arrow[""{name=1, anchor=center, inner sep=0}, "m"{description}, from=3-1, to=3-3]
	\arrow["m"{description}, from=3-3, to=3-5]
	\arrow[Rightarrow, no head, from=1-2, to=3-3]
	\arrow[Rightarrow, no head, from=1-2, to=3-1]
	\arrow[Rightarrow, no head, from=1-5, to=3-5]
	\arrow["m"{description}, from=1-2, to=1-5]
	\arrow[Rightarrow, no head, from=1-7, to=5-7]
	\arrow["m"{description}, from=5-7, to=5-11]
	\arrow["m"{description}, from=1-7, to=1-11]
	\arrow[Rightarrow, no head, from=1-11, to=5-11]
	\arrow["m"{description}, from=3-13, to=3-15]
	\arrow[""{name=2, anchor=center, inner sep=0}, "m"{description}, from=3-15, to=3-17]
	\arrow[Rightarrow, no head, from=3-13, to=5-13]
	\arrow[""{name=3, anchor=center, inner sep=0}, "m"{description}, from=5-13, to=5-17]
	\arrow[Rightarrow, no head, from=3-17, to=5-17]
	\arrow[Rightarrow, no head, from=1-16, to=3-15]
	\arrow[Rightarrow, no head, from=1-16, to=3-17]
	\arrow[Rightarrow, no head, from=1-13, to=3-13]
	\arrow["m"{description}, from=1-13, to=1-16]
	\arrow["{m_2}"{description}, shorten >=7pt, Rightarrow, from=3-3, to=0]
	\arrow["{m_0}"{description}, shorten >=7pt, Rightarrow, from=1-2, to=1]
	\arrow["{m_0}"{description}, shorten >=7pt, Rightarrow, from=1-16, to=2]
	\arrow["{m_2}"{description}, shorten >=7pt, Rightarrow, from=3-15, to=3]
\end{tikzcd}
}
\item
\resizebox{\hsize}{!}{
% https://q.uiver.app/?q=WzAsMjEsWzAsMCwiTSJdLFsyLDAsIk0iXSxbNCwwLCJNIl0sWzYsMCwiTSJdLFswLDIsIk0iXSxbNCwyLCJNIl0sWzYsMiwiTSJdLFswLDQsIk0iXSxbNiw0LCJNIl0sWzcsMiwiPSJdLFs4LDAsIk0iXSxbMTAsMCwiTSJdLFsxMiwwLCJNIl0sWzE0LDAsIk0iXSxbOCwyLCJNIl0sWzEwLDIsIk0iXSxbMTQsMiwiTSJdLFs4LDQsIk0iXSxbMTQsNCwiTSJdLFszLDJdLFsxMSwyXSxbMCwxLCJtIiwxXSxbMSwyLCJtIiwxXSxbMiwzLCJtIiwxXSxbNCw1LCJtIiwxXSxbNSw2LCJtIiwxXSxbNyw4LCJtIiwxXSxbMCw0LCIiLDEseyJsZXZlbCI6Miwic3R5bGUiOnsiaGVhZCI6eyJuYW1lIjoibm9uZSJ9fX1dLFs0LDcsIiIsMSx7ImxldmVsIjoyLCJzdHlsZSI6eyJoZWFkIjp7Im5hbWUiOiJub25lIn19fV0sWzIsNSwiIiwxLHsibGV2ZWwiOjIsInN0eWxlIjp7ImhlYWQiOnsibmFtZSI6Im5vbmUifX19XSxbMyw2LCIiLDEseyJsZXZlbCI6Miwic3R5bGUiOnsiaGVhZCI6eyJuYW1lIjoibm9uZSJ9fX1dLFs2LDgsIiIsMSx7ImxldmVsIjoyLCJzdHlsZSI6eyJoZWFkIjp7Im5hbWUiOiJub25lIn19fV0sWzEwLDExLCJtIiwxXSxbMTEsMTIsIm0iLDFdLFsxMiwxMywibSIsMV0sWzEwLDE0LCIiLDEseyJsZXZlbCI6Miwic3R5bGUiOnsiaGVhZCI6eyJuYW1lIjoibm9uZSJ9fX1dLFsxNCwxNSwibSIsMV0sWzExLDE1LCIiLDEseyJsZXZlbCI6Miwic3R5bGUiOnsiaGVhZCI6eyJuYW1lIjoibm9uZSJ9fX1dLFsxMywxNiwiIiwxLHsibGV2ZWwiOjIsInN0eWxlIjp7ImhlYWQiOnsibmFtZSI6Im5vbmUifX19XSxbMTUsMTYsIm0iLDFdLFsxNCwxNywiIiwxLHsibGV2ZWwiOjIsInN0eWxlIjp7ImhlYWQiOnsibmFtZSI6Im5vbmUifX19XSxbMTcsMTgsIm0iLDFdLFsxNiwxOCwiIiwxLHsibGV2ZWwiOjIsInN0eWxlIjp7ImhlYWQiOnsibmFtZSI6Im5vbmUifX19XSxbMSwyNCwibV8yIiwxLHsic2hvcnRlbiI6eyJ0YXJnZXQiOjIwfX1dLFsxOSwyNiwibV8yIiwxLHsic2hvcnRlbiI6eyJ0YXJnZXQiOjIwfX1dLFsxMiwzOSwibV8yIiwxLHsic2hvcnRlbiI6eyJ0YXJnZXQiOjIwfX1dLFsyMCw0MSwibV8yIiwxLHsic2hvcnRlbiI6eyJ0YXJnZXQiOjIwfX1dXQ==
\begin{tikzcd}[ampersand replacement=\&]
	M \&\& M \&\& M \&\& M \&\& M \&\& M \&\& M \&\& M \\
	\\
	M \&\&\& {} \& M \&\& M \& {=} \& M \&\& M \& {} \&\&\& M \\
	\\
	M \&\&\&\&\&\& M \&\& M \&\&\&\&\&\& M
	\arrow["m"{description}, from=1-1, to=1-3]
	\arrow["m"{description}, from=1-3, to=1-5]
	\arrow["m"{description}, from=1-5, to=1-7]
	\arrow[""{name=0, anchor=center, inner sep=0}, "m"{description}, from=3-1, to=3-5]
	\arrow["m"{description}, from=3-5, to=3-7]
	\arrow[""{name=1, anchor=center, inner sep=0}, "m"{description}, from=5-1, to=5-7]
	\arrow[Rightarrow, no head, from=1-1, to=3-1]
	\arrow[Rightarrow, no head, from=3-1, to=5-1]
	\arrow[Rightarrow, no head, from=1-5, to=3-5]
	\arrow[Rightarrow, no head, from=1-7, to=3-7]
	\arrow[Rightarrow, no head, from=3-7, to=5-7]
	\arrow["m"{description}, from=1-9, to=1-11]
	\arrow["m"{description}, from=1-11, to=1-13]
	\arrow["m"{description}, from=1-13, to=1-15]
	\arrow[Rightarrow, no head, from=1-9, to=3-9]
	\arrow["m"{description}, from=3-9, to=3-11]
	\arrow[Rightarrow, no head, from=1-11, to=3-11]
	\arrow[Rightarrow, no head, from=1-15, to=3-15]
	\arrow[""{name=2, anchor=center, inner sep=0}, "m"{description}, from=3-11, to=3-15]
	\arrow[Rightarrow, no head, from=3-9, to=5-9]
	\arrow[""{name=3, anchor=center, inner sep=0}, "m"{description}, from=5-9, to=5-15]
	\arrow[Rightarrow, no head, from=3-15, to=5-15]
	\arrow["{m_2}"{description}, shorten >=7pt, Rightarrow, from=1-3, to=0]
	\arrow["{m_2}"{description}, shorten >=7pt, Rightarrow, from=3-4, to=1]
	\arrow["{m_2}"{description}, shorten >=7pt, Rightarrow, from=1-13, to=2]
	\arrow["{m_2}"{description}, shorten >=7pt, Rightarrow, from=3-12, to=3]
\end{tikzcd}
}
\end{itemize}
\end{definition}

Alternatively, one could ask for an unbiased notion.

\begin{definition}
An \defin{unbiased monoid} in a vdc \vdC is an object $M$ and an horizontal morphism $m \colon M \to M$ equipped with for each arity $n \geq 0$ a cell from $n$ copies of $m$ to itself:

% https://q.uiver.app/?q=WzAsNyxbMCwwLCJNIl0sWzIsMCwiTSJdLFszLDAsIlxcZG90cyJdLFs0LDAsIk0iXSxbNiwwLCJNIl0sWzAsMiwiTSJdLFs2LDIsIk0iXSxbMCwxLCJtIiwxXSxbMyw0LCJtIiwxXSxbNSw2LCJtIiwxXSxbMCw1LCIiLDEseyJsZXZlbCI6Miwic3R5bGUiOnsiaGVhZCI6eyJuYW1lIjoibm9uZSJ9fX1dLFs0LDYsIiIsMSx7ImxldmVsIjoyLCJzdHlsZSI6eyJoZWFkIjp7Im5hbWUiOiJub25lIn19fV0sWzIsOSwibV9uIiwxLHsic2hvcnRlbiI6eyJ0YXJnZXQiOjIwfX1dXQ==
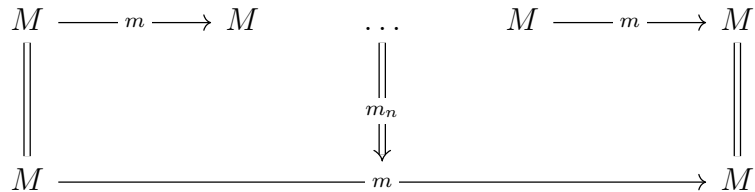
\[\begin{tikzcd}[ampersand replacement=\&]
	M \&\& M \& \dots \& M \&\& M \\
	\\
	M \&\&\&\&\&\& M
	\arrow["m"{description}, from=1-1, to=1-3]
	\arrow["m"{description}, from=1-5, to=1-7]
	\arrow[""{name=0, anchor=center, inner sep=0}, "m"{description}, from=3-1, to=3-7]
	\arrow[Rightarrow, no head, from=1-1, to=3-1]
	\arrow[Rightarrow, no head, from=1-7, to=3-7]
	\arrow["{m_n}"{description}, shorten >=7pt, Rightarrow, from=1-4, to=0]
\end{tikzcd}\]
such that $m_1 = \id_m$ and they compose, i.e. for any arities $n,n'$:

\resizebox{\hsize}{!}{
% https://q.uiver.app/?q=WzAsMjYsWzMsMCwiXFxkb3RzIl0sWzAsMiwiTSJdLFs2LDIsIk0iXSxbOCwwLCJNIl0sWzksMCwiXFxkb3RzIl0sWzEwLDAsIk0iXSxbMTIsMCwiTSJdLFsxMiwyLCJNIl0sWzAsNCwiTSJdLFsxMiw0LCJNIl0sWzE0LDAsIk0iXSxbMTYsMCwiTSJdLFsxNywwLCJcXGRvdHMiXSxbMTgsMCwiTSJdLFsyMCwwLCJNIl0sWzIyLDAsIk0iXSxbMjMsMCwiXFxkb3RzIl0sWzI0LDAsIk0iXSxbMjYsMCwiTSJdLFsxNCw0LCJNIl0sWzI2LDQsIk0iXSxbMTMsMiwiPSJdLFs2LDAsIk0iXSxbNCwwLCJNIl0sWzIsMCwiTSJdLFswLDAsIk0iXSxbNSw2LCJtIiwxXSxbNiw3LCIiLDEseyJsZXZlbCI6Miwic3R5bGUiOnsiaGVhZCI6eyJuYW1lIjoibm9uZSJ9fX1dLFs4LDksIm0iLDFdLFsxLDgsIiIsMSx7ImxldmVsIjoyLCJzdHlsZSI6eyJoZWFkIjp7Im5hbWUiOiJub25lIn19fV0sWzcsOSwiIiwxLHsibGV2ZWwiOjIsInN0eWxlIjp7ImhlYWQiOnsibmFtZSI6Im5vbmUifX19XSxbMTAsMTEsIm0iLDFdLFsxMywxNCwibSIsMV0sWzE0LDE1LCJtIiwxXSxbMTcsMTgsIm0iLDFdLFsxMCwxOSwiIiwxLHsibGV2ZWwiOjIsInN0eWxlIjp7ImhlYWQiOnsibmFtZSI6Im5vbmUifX19XSxbMTgsMjAsIiIsMSx7ImxldmVsIjoyLCJzdHlsZSI6eyJoZWFkIjp7Im5hbWUiOiJub25lIn19fV0sWzE5LDIwLCJtIiwxXSxbMjIsMiwiIiwxLHsibGV2ZWwiOjIsInN0eWxlIjp7ImhlYWQiOnsibmFtZSI6Im5vbmUifX19XSxbMjIsMywibSIsMV0sWzIzLDIyLCJtIiwxXSxbMjUsMSwiIiwxLHsibGV2ZWwiOjIsInN0eWxlIjp7ImhlYWQiOnsibmFtZSI6Im5vbmUifX19XSxbMjUsMjQsIm0iLDFdLFsxLDIsIm0iLDFdLFsyLDcsIm0iLDFdLFsxNCwzNywibV97bituJ30iLDEseyJzaG9ydGVuIjp7InRhcmdldCI6MjB9fV0sWzIsMjgsIm1fMiIsMSx7InNob3J0ZW4iOnsidGFyZ2V0IjoyMH19XSxbMCw0MywibV9uIiwxLHsic2hvcnRlbiI6eyJ0YXJnZXQiOjIwfX1dLFs0LDQ0LCJtX3tuJ30iLDEseyJzaG9ydGVuIjp7InRhcmdldCI6MjB9fV1d
\begin{tikzcd}[ampersand replacement=\&]
	M \&\& M \& \dots \& M \&\& M \&\& M \& \dots \& M \&\& M \&\& M \&\& M \& \dots \& M \&\& M \&\& M \& \dots \& M \&\& M \\
	\\
	M \&\&\&\&\&\& M \&\&\&\&\&\& M \& {=} \\
	\\
	M \&\&\&\&\&\&\&\&\&\&\&\& M \&\& M \&\&\&\&\&\&\&\&\&\&\&\& M
	\arrow["m"{description}, from=1-11, to=1-13]
	\arrow[Rightarrow, no head, from=1-13, to=3-13]
	\arrow[""{name=0, anchor=center, inner sep=0}, "m"{description}, from=5-1, to=5-13]
	\arrow[Rightarrow, no head, from=3-1, to=5-1]
	\arrow[Rightarrow, no head, from=3-13, to=5-13]
	\arrow["m"{description}, from=1-15, to=1-17]
	\arrow["m"{description}, from=1-19, to=1-21]
	\arrow["m"{description}, from=1-21, to=1-23]
	\arrow["m"{description}, from=1-25, to=1-27]
	\arrow[Rightarrow, no head, from=1-15, to=5-15]
	\arrow[Rightarrow, no head, from=1-27, to=5-27]
	\arrow[""{name=1, anchor=center, inner sep=0}, "m"{description}, from=5-15, to=5-27]
	\arrow[Rightarrow, no head, from=1-7, to=3-7]
	\arrow["m"{description}, from=1-7, to=1-9]
	\arrow["m"{description}, from=1-5, to=1-7]
	\arrow[Rightarrow, no head, from=1-1, to=3-1]
	\arrow["m"{description}, from=1-1, to=1-3]
	\arrow[""{name=2, anchor=center, inner sep=0}, "m"{description}, from=3-1, to=3-7]
	\arrow[""{name=3, anchor=center, inner sep=0}, "m"{description}, from=3-7, to=3-13]
	\arrow["{m_{n+n'}}"{description}, shorten >=16pt, Rightarrow, from=1-21, to=1]
	\arrow["{m_2}"{description}, shorten >=7pt, Rightarrow, from=3-7, to=0]
	\arrow["{m_n}"{description}, shorten >=7pt, Rightarrow, from=1-4, to=2]
	\arrow["{m_{n'}}"{description}, shorten >=7pt, Rightarrow, from=1-10, to=3]
\end{tikzcd}
}

\end{definition}

\begin{prop}
The biased and unbiased notions of monoid coincide.
\end{prop}
\begin{proof}
If $m$ is an unbiased monoid then we have \[m_2(\id_m,m_0) = m_2(m_1,m_0) = m_1 = \id_m\]
and similarly for $m_2(m_0,\id_m) = \id_m$
\[m_2(m_2,\id_m) = m_2(m_2,m_1) = m_3 = m_2(m_1,m_2) = m_2(\id_m,m_2)\]
So $m$ is a biased monoid.

Conversely, if $m$ is a monoid then we take $m_0 := m_0$ and $m_{n+1} := m_2(\id_m,m_n)$.
Notice that we have $m_1 = \id_m$ and $m_2 = m_2$ by the rules of a monoid.

Then let prove that $m_2(m_n,m_{n'}) = m_{n+n'}$ by induction:
\begin{itemize}
\item for $n=0$, 
\[m_2(m_0,m_{n'}) = (m_2(m_0,\id_m))(m_{n'}) = \id_m(m_{n'}) = m_{n'}\]
where the second equality use the fact that $m$ is a monoid
\item now suppose that it is true for $n$,
\begin{align*}
m_2(m_{n+1},m_{n'}) &= m_2(m_2(\id_m,m_n),m_{n'}))\\
&= m_2(m_2,\id_m)(\id_m,m_n,m_{n'}) \\
&= m_2(\id_m,m_2)(\id_m,m_n,m_{n'}) \\&
= m_2(\id_m,m_2(m_n,m_{n'})) \\&
= m_2(\id_m,m_{n+n'}) \\&
= m_{n+n'+1}
\end{align*}
\end{itemize}

\begin{definition}
A \defin{morphism} of monoids $(f,F) \colon (M,m) \to (M',m')$ is a pair of a vertical morphisms $f \colon M \to M'$ and a cell $F \colon m \Rightarrow m'$ with horizontal domain and codomain $f$ that preserves the multiplication cells of the monoids, i.e.:

\resizebox{\hsize}{!}{
% https://q.uiver.app/?q=WzAsMjIsWzgsMCwiTSJdLFsxMCwwLCJNIl0sWzExLDAsIlxcZG90cyJdLFsxMiwwLCJNIl0sWzE0LDAsIk0iXSxbOCwyLCJNJyJdLFsxMCwyLCJNJyJdLFsxMSwyLCJcXGRvdHMiXSxbMTIsMiwiTSciXSxbMTQsMiwibSciXSxbOCw0LCJNJyJdLFsxNCw0LCJNJyJdLFs3LDIsIj0iXSxbNCwwLCJNIl0sWzYsMCwiTSJdLFszLDAsIlxcZG90cyJdLFswLDAsIk0iXSxbMiwwLCJNIl0sWzAsMiwiTSJdLFs2LDIsIk0iXSxbMCw0LCJNJyJdLFs2LDQsIk0nIl0sWzAsMSwibSIsMV0sWzMsNCwibSIsMV0sWzUsNiwibSciLDFdLFs4LDksIm0nIiwxXSxbMTAsMTEsIm0nIiwxXSxbMCw1LCJmIiwxXSxbMSw2LCJmIiwxXSxbMyw4LCJmIiwxXSxbNCw5LCJmIiwxXSxbNSwxMCwiIiwxLHsibGV2ZWwiOjIsInN0eWxlIjp7ImhlYWQiOnsibmFtZSI6Im5vbmUifX19XSxbOSwxMSwiIiwxLHsibGV2ZWwiOjIsInN0eWxlIjp7ImhlYWQiOnsibmFtZSI6Im5vbmUifX19XSxbMTMsMTQsIm0iLDFdLFsxNiwxNywibSIsMV0sWzE4LDE5LCJtIiwxXSxbMjAsMjEsIm0nIiwxXSxbMTYsMTgsIiIsMSx7ImxldmVsIjoyLCJzdHlsZSI6eyJoZWFkIjp7Im5hbWUiOiJub25lIn19fV0sWzE0LDE5LCIiLDEseyJsZXZlbCI6Miwic3R5bGUiOnsiaGVhZCI6eyJuYW1lIjoibm9uZSJ9fX1dLFsxOCwyMCwiZiIsMV0sWzE5LDIxLCJmIiwxXSxbMjIsMjQsIkYiLDEseyJzaG9ydGVuIjp7InNvdXJjZSI6MjAsInRhcmdldCI6MjB9fV0sWzIzLDI1LCJGIiwxLHsic2hvcnRlbiI6eyJzb3VyY2UiOjIwLCJ0YXJnZXQiOjIwfX1dLFs3LDI2LCJtX24nIiwxLHsic2hvcnRlbiI6eyJ0YXJnZXQiOjIwfX1dLFsxNSwzNSwibV9uIiwxLHsic2hvcnRlbiI6eyJ0YXJnZXQiOjIwfX1dLFszNSwzNiwiRiIsMSx7InNob3J0ZW4iOnsic291cmNlIjoyMCwidGFyZ2V0IjoyMH19XV0=
\begin{tikzcd}[ampersand replacement=\&]
	M \&\& M \& \dots \& M \&\& M \&\& M \&\& M \& \dots \& M \&\& M \\
	\\
	M \&\&\&\&\&\& M \& {=} \& {M'} \&\& {M'} \& \dots \& {M'} \&\& {m'} \\
	\\
	{M'} \&\&\&\&\&\& {M'} \&\& {M'} \&\&\&\&\&\& {M'}
	\arrow[""{name=0, anchor=center, inner sep=0}, "m"{description}, from=1-9, to=1-11]
	\arrow[""{name=1, anchor=center, inner sep=0}, "m"{description}, from=1-13, to=1-15]
	\arrow[""{name=2, anchor=center, inner sep=0}, "{m'}"{description}, from=3-9, to=3-11]
	\arrow[""{name=3, anchor=center, inner sep=0}, "{m'}"{description}, from=3-13, to=3-15]
	\arrow[""{name=4, anchor=center, inner sep=0}, "{m'}"{description}, from=5-9, to=5-15]
	\arrow["f"{description}, from=1-9, to=3-9]
	\arrow["f"{description}, from=1-11, to=3-11]
	\arrow["f"{description}, from=1-13, to=3-13]
	\arrow["f"{description}, from=1-15, to=3-15]
	\arrow[Rightarrow, no head, from=3-9, to=5-9]
	\arrow[Rightarrow, no head, from=3-15, to=5-15]
	\arrow["m"{description}, from=1-5, to=1-7]
	\arrow["m"{description}, from=1-1, to=1-3]
	\arrow[""{name=5, anchor=center, inner sep=0}, "m"{description}, from=3-1, to=3-7]
	\arrow[""{name=6, anchor=center, inner sep=0}, "{m'}"{description}, from=5-1, to=5-7]
	\arrow[Rightarrow, no head, from=1-1, to=3-1]
	\arrow[Rightarrow, no head, from=1-7, to=3-7]
	\arrow["f"{description}, from=3-1, to=5-1]
	\arrow["f"{description}, from=3-7, to=5-7]
	\arrow["F"{description}, shorten <=9pt, shorten >=9pt, Rightarrow, from=0, to=2]
	\arrow["F"{description}, shorten <=9pt, shorten >=9pt, Rightarrow, from=1, to=3]
	\arrow["{m_n'}"{description}, shorten >=7pt, Rightarrow, from=3-12, to=4]
	\arrow["{m_n}"{description}, shorten >=7pt, Rightarrow, from=1-4, to=5]
	\arrow["F"{description}, shorten <=9pt, shorten >=9pt, Rightarrow, from=5, to=6]
\end{tikzcd}
}
\end{definition}
\end{proof}

\begin{definition}
A \defin{module} from $(M,m)$ to $(N,n)$ or $(M,N)$-module is an horizontal morphism $p \colon M \to N$ with a left action by $m$ and a right action by $n$, i.e. cells:
% https://q.uiver.app/?q=WzAsMTAsWzAsMCwiTSJdLFsyLDAsIk0iXSxbNCwwLCJOIl0sWzAsMiwiTSJdLFs0LDIsIk4iXSxbMCwzLCJNIl0sWzIsMywiTiJdLFs0LDMsIk4iXSxbMCw1LCJNIl0sWzQsNSwiTiJdLFswLDEsIm0iLDFdLFszLDQsInAiLDFdLFswLDMsIiIsMSx7ImxldmVsIjoyLCJzdHlsZSI6eyJoZWFkIjp7Im5hbWUiOiJub25lIn19fV0sWzIsNCwiIiwxLHsibGV2ZWwiOjIsInN0eWxlIjp7ImhlYWQiOnsibmFtZSI6Im5vbmUifX19XSxbNSw2LCJwIiwxXSxbNiw3LCJuIiwxXSxbOCw5LCJwIiwxXSxbNSw4LCIiLDEseyJsZXZlbCI6Miwic3R5bGUiOnsiaGVhZCI6eyJuYW1lIjoibm9uZSJ9fX1dLFs3LDksIiIsMSx7ImxldmVsIjoyLCJzdHlsZSI6eyJoZWFkIjp7Im5hbWUiOiJub25lIn19fV0sWzEsMiwicCIsMV0sWzEsMTEsIlxcY2RvdF9NIiwxLHsic2hvcnRlbiI6eyJ0YXJnZXQiOjIwfX1dLFs2LDE2LCJcXGNkb3RfTiIsMSx7InNob3J0ZW4iOnsidGFyZ2V0IjoyMH19XV0=
\[\begin{tikzcd}[ampersand replacement=\&]
	M \&\& M \&\& N \\
	\\
	M \&\&\&\& N \\
	M \&\& N \&\& N \\
	\\
	M \&\&\&\& N
	\arrow["m"{description}, from=1-1, to=1-3]
	\arrow[""{name=0, anchor=center, inner sep=0}, "p"{description}, from=3-1, to=3-5]
	\arrow[Rightarrow, no head, from=1-1, to=3-1]
	\arrow[Rightarrow, no head, from=1-5, to=3-5]
	\arrow["p"{description}, from=4-1, to=4-3]
	\arrow["n"{description}, from=4-3, to=4-5]
	\arrow[""{name=1, anchor=center, inner sep=0}, "p"{description}, from=6-1, to=6-5]
	\arrow[Rightarrow, no head, from=4-1, to=6-1]
	\arrow[Rightarrow, no head, from=4-5, to=6-5]
	\arrow["p"{description}, from=1-3, to=1-5]
	\arrow["{\cdot_m}"{description}, shorten >=7pt, Rightarrow, from=1-3, to=0]
	\arrow["{\cdot_n}"{description}, shorten >=7pt, Rightarrow, from=4-3, to=1]
\end{tikzcd}\]
such that applying $\cdot$ to a multiplication consists of applying it multiple times, i.e.:

\resizebox{\hsize}{!}{
\begin{tikzcd}[ampersand replacement=\&]
	M \&\& M \& \dots \& M \&\& M \&\& N \&\& M \&\& M \& \dots \& M \&\& M \&\& N \\
	\\
	\&\&\&\&\&\&\&\&\&\& M \&\& M \&\& M \&\&\&\& N \\
	M \& {} \&\&\& {} \&\& M \&\& N \& {=} \&\& \vdots \&\& \vdots \&\&\& \vdots \\
	\&\&\&\&\&\&\&\&\&\& M \&\& M \&\& {} \&\&\&\& N \\
	\\
	M \&\&\&\&\&\&\&\& N \&\& M \&\&\&\&\&\&\&\& N
	\arrow[""{name=0, anchor=center, inner sep=0}, "m"{description}, from=4-1, to=4-7]
	\arrow["p"{description}, from=4-7, to=4-9]
	\arrow[Rightarrow, no head, from=4-1, to=7-1]
	\arrow[Rightarrow, no head, from=4-9, to=7-9]
	\arrow[""{name=1, anchor=center, inner sep=0}, "p"{description}, from=7-1, to=7-9]
	\arrow["m"{description}, from=1-5, to=1-7]
	\arrow["m"{description}, from=1-1, to=1-3]
	\arrow["p"{description}, from=1-7, to=1-9]
	\arrow[Rightarrow, no head, from=1-9, to=4-9]
	\arrow[Rightarrow, no head, from=1-7, to=4-7]
	\arrow[Rightarrow, no head, from=1-1, to=4-1]
	\arrow["m"{description}, from=1-11, to=1-13]
	\arrow["m"{description}, from=1-15, to=1-17]
	\arrow["p"{description}, from=1-17, to=1-19]
	\arrow[Rightarrow, no head, from=1-15, to=3-15]
	\arrow[Rightarrow, no head, from=1-19, to=3-19]
	\arrow[""{name=2, anchor=center, inner sep=0}, "p"{description}, from=3-15, to=3-19]
	\arrow[Rightarrow, no head, from=1-11, to=3-11]
	\arrow[Rightarrow, no head, from=1-13, to=3-13]
	\arrow["m"{description}, from=3-11, to=3-13]
	\arrow[""{name=3, anchor=center, inner sep=0}, "p"{description}, from=7-11, to=7-19]
	\arrow[Rightarrow, no head, from=5-11, to=7-11]
	\arrow[Rightarrow, no head, from=5-19, to=7-19]
	\arrow["m"{description}, from=5-11, to=5-13]
	\arrow["p"{description}, from=5-13, to=5-19]
	\arrow[Rightarrow, no head, from=3-11, to=5-11]
	\arrow[Rightarrow, no head, from=3-19, to=5-19]
	\arrow["{m_k}"{description}, shorten >=12pt, Rightarrow, from=1-4, to=0]
	\arrow["{\cdot_m}"{description}, shorten >=12pt, Rightarrow, from=4-5, to=1]
	\arrow["{\cdot_m}"{description}, shorten >=7pt, Rightarrow, from=1-17, to=2]
	\arrow["{\cdot_m}"{description}, shorten >=7pt, Rightarrow, from=5-15, to=3]
\end{tikzcd}
}
and similarly for $\cdot_n$
\end{definition}

\begin{definition}
Given modules $p_i$ from $(M_{i-1},m_{i-1})$ to $(M_i,m_i)$ and $q$ from $(N_0,n_0)$ to $(N_1,n_1)$, a \defin{module multimorphism} $\alpha \colon p_1,\dots, p_n \Rightarrow q$ is a cell in \vdC that
\begin{itemize}
\item respects the action externally, i.e.:

\resizebox{\hsize}{!}{
\begin{tikzcd}[ampersand replacement=\&]
	\&\&\&\&\&\&\&\&\&\&\&\&\&\& {} \\
	\\
	\\
	\\
	\\
	\\
	\\
	\\
	M \&\& {M_0} \&\& {M_1} \& \dots \& {M_{k-1}} \&\& {M_k} \&\&\&\& {M_0} \&\& {M_0} \&\& {M_1} \& \dots \& {M_{k-1}} \&\& {M_k} \\
	\\
	{M_0} \&\&\&\& {M_1} \& \dots \& {M_{k-1}} \&\& {M_k} \&\& {=} \&\& {N_0} \&\& {N_0} \&\& {} \&\&\&\& {N_1} \\
	\\
	{N_0} \&\&\&\&\&\&\&\& {N_1} \&\&\&\& {N_0} \&\&\&\&\&\&\&\& {N_1}
	\arrow[""{name=0, anchor=center, inner sep=0}, "{m_0}"{description}, from=9-13, to=9-15]
	\arrow["{p_1}"{description}, from=9-15, to=9-17]
	\arrow["f"{description}, from=9-13, to=11-13]
	\arrow["f"{description}, from=9-15, to=11-15]
	\arrow[""{name=1, anchor=center, inner sep=0}, "{n_0}"{description}, from=11-13, to=11-15]
	\arrow[""{name=2, anchor=center, inner sep=0}, "q"{description}, from=13-13, to=13-21]
	\arrow[Rightarrow, no head, from=11-13, to=13-13]
	\arrow["{p_1}"{description}, from=9-3, to=9-5]
	\arrow["m"{description}, from=9-1, to=9-3]
	\arrow["{p_k}"{description}, from=9-19, to=9-21]
	\arrow["g"{description}, from=9-21, to=11-21]
	\arrow[""{name=3, anchor=center, inner sep=0}, "q"{description}, from=11-15, to=11-21]
	\arrow[Rightarrow, no head, from=11-21, to=13-21]
	\arrow[Rightarrow, no head, from=9-1, to=11-1]
	\arrow["{p_k}"{description}, from=9-7, to=9-9]
	\arrow[Rightarrow, no head, from=9-5, to=11-5]
	\arrow[""{name=4, anchor=center, inner sep=0}, "{p_1}"{description}, from=11-1, to=11-5]
	\arrow["{p_k}"{description}, from=11-7, to=11-9]
	\arrow[Rightarrow, no head, from=9-7, to=11-7]
	\arrow[Rightarrow, no head, from=9-9, to=11-9]
	\arrow[""{name=5, anchor=center, inner sep=0}, "q"{description}, from=13-1, to=13-9]
	\arrow["f"{description}, from=11-1, to=13-1]
	\arrow["g"{description}, from=11-9, to=13-9]
	\arrow["F"{description}, shorten <=9pt, shorten >=9pt, Rightarrow, from=0, to=1]
	\arrow["\alpha"{description}, shorten >=7pt, Rightarrow, from=9-18, to=3]
	\arrow["{\cdot_n}"{description}, shorten >=7pt, Rightarrow, from=11-17, to=2]
	\arrow["{\cdot_m}"{description}, shorten >=7pt, Rightarrow, from=9-3, to=4]
	\arrow["\alpha"{description}, shorten >=7pt, Rightarrow, from=11-5, to=5]
\end{tikzcd}
}
and similarly on the right.
\item is balanced, i.e. respects the action internally:

\resizebox{\hsize}{!}{
\begin{tikzcd}[ampersand replacement=\&]
	{M_0} \&\& {M_1} \& \dots \& {M_{i-1}} \&\& {M_i} \&\& {M_i} \&\& {M_{i+1}} \& \dots \& {M_{k-1}} \&\& {M_k} \\
	\\
	{M_0} \&\& {M_1} \& \dots \& {M_{i-1}} \&\&\&\& {M_i} \&\& {M_{i+1}} \& \dots \& {M_{k-1}} \&\& {M_k} \\
	\\
	{N_0} \&\&\&\&\&\&\&\&\&\&\&\&\&\& {N_1} \\
	\&\&\&\&\&\&\& {=} \\
	{M_0} \&\& {M_1} \& \dots \& {M_{i-1}} \&\& {M_i} \&\& {M_i} \&\& {M_{i+1}} \& \dots \& {M_{k-1}} \&\& {M_k} \\
	\\
	{M_0} \&\& {M_1} \& \dots \& {M_{i-1}} \&\& {M_i} \&\&\&\& {M_{i+1}} \& \dots \& {M_{k-1}} \&\& {M_k} \\
	\\
	{N_0} \&\&\&\&\&\&\&\&\&\&\&\&\&\& {N_1}
	\arrow["{p_1}"{description}, from=1-1, to=1-3]
	\arrow["{p_i}"{description}, from=1-5, to=1-7]
	\arrow["{p_k}"{description}, from=1-13, to=1-15]
	\arrow["{m_i}"{description}, from=1-7, to=1-9]
	\arrow["{p_{i+1}}"{description}, from=1-9, to=1-11]
	\arrow[Rightarrow, no head, from=1-1, to=3-1]
	\arrow["{p_1}"{description}, from=3-1, to=3-3]
	\arrow[Rightarrow, no head, from=1-3, to=3-3]
	\arrow[""{name=0, anchor=center, inner sep=0}, "{p_i}"{description}, from=3-5, to=3-9]
	\arrow[Rightarrow, no head, from=1-5, to=3-5]
	\arrow[Rightarrow, no head, from=1-9, to=3-9]
	\arrow[Rightarrow, no head, from=1-11, to=3-11]
	\arrow["{p_{i+1}}"{description}, from=3-9, to=3-11]
	\arrow[Rightarrow, no head, from=1-13, to=3-13]
	\arrow[Rightarrow, no head, from=1-15, to=3-15]
	\arrow["{p_k}"{description}, from=3-13, to=3-15]
	\arrow["f"{description}, from=3-1, to=5-1]
	\arrow[""{name=1, anchor=center, inner sep=0}, "q"{description}, from=5-1, to=5-15]
	\arrow["g"{description}, from=3-15, to=5-15]
	\arrow["{p_1}"{description}, from=7-1, to=7-3]
	\arrow["{p_{i}}"{description}, from=7-5, to=7-7]
	\arrow["{m_i}"{description}, from=7-7, to=7-9]
	\arrow["{p_{i+1}}"{description}, from=7-9, to=7-11]
	\arrow["{p_k}"{description}, from=7-13, to=7-15]
	\arrow[Rightarrow, no head, from=7-1, to=9-1]
	\arrow[Rightarrow, no head, from=7-3, to=9-3]
	\arrow[Rightarrow, no head, from=7-5, to=9-5]
	\arrow[Rightarrow, no head, from=7-7, to=9-7]
	\arrow[Rightarrow, no head, from=7-11, to=9-11]
	\arrow[Rightarrow, no head, from=7-13, to=9-13]
	\arrow[Rightarrow, no head, from=7-15, to=9-15]
	\arrow["{p_1}"{description}, from=9-1, to=9-3]
	\arrow["{p_i}"{description}, from=9-5, to=9-7]
	\arrow[""{name=2, anchor=center, inner sep=0}, "{p_{i+1}}"{description}, from=9-7, to=9-11]
	\arrow["{p_k}"{description}, from=9-13, to=9-15]
	\arrow["f"{description}, from=9-1, to=11-1]
	\arrow["g"{description}, from=9-15, to=11-15]
	\arrow[""{name=3, anchor=center, inner sep=0}, "q"{description}, from=11-1, to=11-15]
	\arrow["{\cdot_{m_i}}"{description}, shorten >=7pt, Rightarrow, from=1-7, to=0]
	\arrow["\alpha"{description}, shorten >=9pt, Rightarrow, from=3-9, to=1]
	\arrow["{\cdot_{m_i}}"{description}, shorten >=7pt, Rightarrow, from=7-9, to=2]
	\arrow["\alpha"{description}, shorten >=9pt, Rightarrow, from=9-7, to=3]
\end{tikzcd}
}
\end{itemize}
\end{definition}

\begin{prop}
Monoids and monoid morphisms in \vdC form a category with identity and composition inherited from \vdC.
\end{prop}
\begin{proof}
These form a category by unitality and associativity of vertical morphisms and cells in \vdC.
\end{proof}

\begin{prop}
The identity cell is a module multimorphism and these compose.
\end{prop}
\begin{proof}
\begin{figure}
\resizebox{\hsize}{!}{
\begin{tikzcd}[ampersand replacement=\&]
	{M_0} \&\& {M_0} \&\& {M_1} \& \dots \& {M_{k_1-1}} \&\& {M_{k_1}} \& \dots \& {M_{k_{l-1}}} \&\& {M_{k_{l-1}+1}} \& \dots \& {M_{k_l-1}} \&\& {M_{k_l}} \\
	\\
	{M_0} \&\&\&\& {M_1} \& \dots \& {M_{k_1-1}} \&\& {M_{k_l}} \& \dots \& {M_{k_{l-1}}} \&\& {M_{k_{l-1}+1}} \& \dots \& {M_{k_l-1}} \&\& {M_{k_l}} \\
	\\
	{N_0} \&\&\&\&\&\&\&\& {N_1} \& \dots \& {N_{l-1}} \&\&\&\&\&\& {N_l} \\
	\\
	{P_0} \&\&\&\&\&\&\&\&\&\&\&\&\&\&\&\& {P_1} \\
	\\
	\&\&\&\&\&\&\&\& {=} \\
	\\
	{M_0} \&\& {M_0} \&\& {M_1} \& \dots \& {M_{k_1-1}} \&\& {M_{k_1}} \& \dots \& {M_{k_{l-1}}} \&\& {M_{k_{l-1}+1}} \& \dots \& {M_{k_l-1}} \&\& {M_{k_l}} \\
	\\
	{N_0} \&\& {N_0} \&\& {} \&\&\&\& {N_1} \& \dots \& {N_{l-1}} \&\&\&\&\&\& {N_l} \\
	\\
	{N_0} \&\&\&\&\&\&\&\& {N_1} \& \dots \& {N_{l-1}} \&\&\&\&\&\& {N_l} \\
	\\
	{P_0} \&\&\&\&\&\&\&\&\&\&\&\&\&\&\&\& {P_1} \\
	\\
	\&\&\&\&\&\&\&\& {=} \\
	\\
	{M_0} \&\& {M_0} \&\& {M_1} \& \dots \& {M_{k_1-1}} \&\& {M_{k_1}} \& \dots \& {M_{k_{l-1}}} \&\& {M_{k_{k-1}+1}} \& \dots \& {M_{k_l-1}} \&\& {M_{k_l}} \\
	\\
	{N_0} \&\& {N_0} \&\&\&\&\&\& {N_1} \& \dots \& {N_{l-1}} \&\&\&\&\&\& {N_l} \\
	\\
	{P_0} \&\& {P_0} \&\&\&\&\&\& {} \&\&\&\&\&\&\&\& {P_1} \\
	\\
	{P_0} \&\&\&\&\&\&\&\&\&\&\&\&\&\&\&\& {P_1}
	\arrow[""{name=0, anchor=center, inner sep=0}, "{p_1}"{description}, from=3-1, to=3-5]
	\arrow["{p_{k_1}}"{description}, from=3-7, to=3-9]
	\arrow["{p_{k_{l-1}+1}}"{description}, from=3-11, to=3-13]
	\arrow["{p_{k_l}}"{description}, from=3-15, to=3-17]
	\arrow[""{name=1, anchor=center, inner sep=0}, "{q_0}"{description}, from=5-1, to=5-9]
	\arrow[""{name=2, anchor=center, inner sep=0}, "{q_l}"{description}, from=5-11, to=5-17]
	\arrow["{f_0}"{description}, from=3-1, to=5-1]
	\arrow["{f_1}"{description}, from=3-9, to=5-9]
	\arrow["{f_{l-1}}"{description}, from=3-11, to=5-11]
	\arrow["{f_l}"{description}, from=3-17, to=5-17]
	\arrow[""{name=3, anchor=center, inner sep=0}, "r"{description}, from=7-1, to=7-17]
	\arrow["{g_0}"{description}, from=5-1, to=7-1]
	\arrow["{g_1}"{description}, from=5-17, to=7-17]
	\arrow["{p_1}"{description}, from=1-3, to=1-5]
	\arrow["{p_{k_1}}"{description}, from=1-7, to=1-9]
	\arrow["{p_{k_{l-1}+1}}"{description}, from=1-11, to=1-13]
	\arrow["{p_{k_l}}"{description}, from=1-15, to=1-17]
	\arrow[Rightarrow, no head, from=1-1, to=3-1]
	\arrow[Rightarrow, no head, from=1-5, to=3-5]
	\arrow[Rightarrow, no head, from=1-7, to=3-7]
	\arrow[Rightarrow, no head, from=1-9, to=3-9]
	\arrow[Rightarrow, no head, from=1-11, to=3-11]
	\arrow[Rightarrow, no head, from=1-13, to=3-13]
	\arrow[Rightarrow, no head, from=1-15, to=3-15]
	\arrow[Rightarrow, no head, from=1-17, to=3-17]
	\arrow["{m_0}"{description}, from=1-1, to=1-3]
	\arrow[""{name=4, anchor=center, inner sep=0}, "{m_0}"{description}, from=21-1, to=21-3]
	\arrow["{p_1}"{description}, from=21-3, to=21-5]
	\arrow["{p_{k_1}}"{description}, from=21-7, to=21-9]
	\arrow["{p_{k_{l-1}+1}}"{description}, from=21-11, to=21-13]
	\arrow["{p_{k_l}}"{description}, from=21-15, to=21-17]
	\arrow["{f_0}"{description}, from=21-1, to=23-1]
	\arrow["{f_0}"{description}, from=21-3, to=23-3]
	\arrow[""{name=5, anchor=center, inner sep=0}, "{q_1}"{description}, from=23-3, to=23-9]
	\arrow[""{name=6, anchor=center, inner sep=0}, "{n_0}"{description}, from=23-1, to=23-3]
	\arrow["{f_1}"{description}, from=21-9, to=23-9]
	\arrow["{f_{l-1}}"{description}, from=21-11, to=23-11]
	\arrow[""{name=7, anchor=center, inner sep=0}, "{q_l}"{description}, from=23-11, to=23-17]
	\arrow["{f_l}"{description}, from=21-17, to=23-17]
	\arrow["{g_0}"{description}, from=23-1, to=25-1]
	\arrow["{g_1}"{description}, from=23-3, to=25-3]
	\arrow[""{name=8, anchor=center, inner sep=0}, "{p_0}"{description}, from=25-1, to=25-3]
	\arrow[""{name=9, anchor=center, inner sep=0}, "r"{description}, from=25-3, to=25-17]
	\arrow["{g_1}"{description}, from=23-17, to=25-17]
	\arrow[""{name=10, anchor=center, inner sep=0}, "r"{description}, from=27-1, to=27-17]
	\arrow[Rightarrow, no head, from=25-1, to=27-1]
	\arrow[Rightarrow, no head, from=25-17, to=27-17]
	\arrow["{p_{k_1}}"{description}, from=11-7, to=11-9]
	\arrow["{p_{k_{l-1}+1}}"{description}, from=11-11, to=11-13]
	\arrow["{p_{k_l}}"{description}, from=11-15, to=11-17]
	\arrow[""{name=11, anchor=center, inner sep=0}, "{m_0}"{description}, from=11-1, to=11-3]
	\arrow["{p_1}"{description}, from=11-3, to=11-5]
	\arrow["{f_0}"{description}, from=11-1, to=13-1]
	\arrow["{f_0}"{description}, from=11-3, to=13-3]
	\arrow[""{name=12, anchor=center, inner sep=0}, "{n_0}"{description}, from=13-1, to=13-3]
	\arrow[""{name=13, anchor=center, inner sep=0}, "{q_1}"{description}, from=13-3, to=13-9]
	\arrow["{f_1}"{description}, from=11-9, to=13-9]
	\arrow[Rightarrow, no head, from=13-1, to=15-1]
	\arrow[""{name=14, anchor=center, inner sep=0}, "{q_0}"{description}, from=15-1, to=15-9]
	\arrow[Rightarrow, no head, from=13-9, to=15-9]
	\arrow[""{name=15, anchor=center, inner sep=0}, "{q_l}"{description}, from=13-11, to=13-17]
	\arrow["{f_l}"{description}, from=11-17, to=13-17]
	\arrow["{f_{l-1}}"{description}, from=11-11, to=13-11]
	\arrow[Rightarrow, no head, from=13-11, to=15-11]
	\arrow["{q_l}"{description}, from=15-11, to=15-17]
	\arrow[Rightarrow, no head, from=13-17, to=15-17]
	\arrow["{g_0}"{description}, from=15-1, to=17-1]
	\arrow[""{name=16, anchor=center, inner sep=0}, "r"{description}, from=17-1, to=17-17]
	\arrow["{g_1}"{description}, from=15-17, to=17-17]
	\arrow["{\alpha_l}"{description}, shorten >=7pt, Rightarrow, from=3-14, to=2]
	\arrow["{F_0}"{description}, shorten <=9pt, shorten >=9pt, Rightarrow, from=4, to=6]
	\arrow["{G_0}"{description}, shorten <=9pt, shorten >=9pt, Rightarrow, from=6, to=8]
	\arrow["{\alpha_1}"{description}, shorten >=7pt, Rightarrow, from=21-6, to=5]
	\arrow["{\alpha_l}"{description}, shorten >=7pt, Rightarrow, from=21-14, to=7]
	\arrow["\beta"{description}, shorten >=8pt, Rightarrow, from=23-10, to=9]
	\arrow["{\cdot_{p_0}}"{description}, shorten >=8pt, Rightarrow, from=25-9, to=10]
	\arrow["{F_0}"{description}, shorten <=9pt, shorten >=9pt, Rightarrow, from=11, to=12]
	\arrow["{\alpha_1}"{description}, shorten >=7pt, Rightarrow, from=11-6, to=13]
	\arrow["{\cdot_{n_0}}"{description}, shorten >=7pt, Rightarrow, from=13-5, to=14]
	\arrow["{\alpha_l}"{description}, shorten >=7pt, Rightarrow, from=11-14, to=15]
	\arrow["\beta"{description}, shorten >=8pt, Rightarrow, from=15-9, to=16]
	\arrow["\beta"{description}, shorten >=8pt, Rightarrow, from=5-9, to=3]
	\arrow["{\alpha_1}"{description}, shorten >=7pt, Rightarrow, from=3-5, to=1]
	\arrow["{\cdot_{m_0}}"{description}, shorten >=7pt, Rightarrow, from=1-3, to=0]
\end{tikzcd}
}
\caption{Left action on the composition}
\label{fig:lact_comp}
\end{figure}
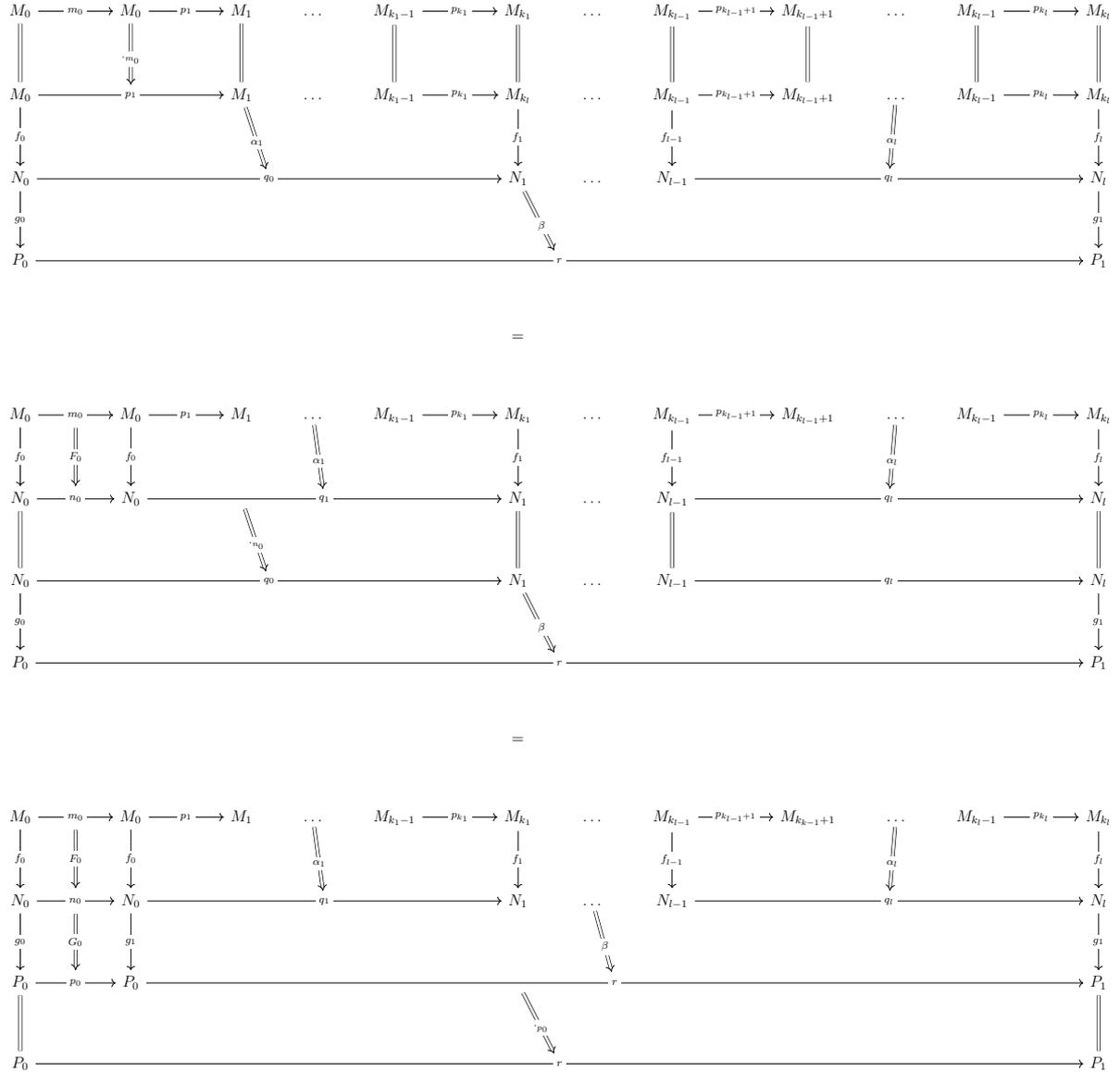
For the identity it follows from unitality and for composition it uses associativity.
For example the proof that the composite respect left action is given in figure \ref{fig:lact_comp}.
\end{proof}

So we can define a vdc \Mon{\vdC}.

\begin{definition}
\Mon{\vdC} is the virtual double category with:
\begin{itemize}
\item objects, monoids
\item vertical morphisms, monoid morphisms,
\item horizontal morphisms, modules
\item cells, module multimorphisms
\end{itemize}
\end{definition}

\begin{example}
Let \cC be a multicategory considered as a vdc with one object.
Then we get the usual notion of monoids, bimodules and their morphisms internal to \cC.
In particular for \Ab{} and \Set{} we get \Ring{} and $\mathbf{Mon}$.
\end{example}

\begin{example}
$\Mon{\Span} = \Dist$.
\end{example}

\part{Fibrations}

Fibrations have been introduced by Grothendieck in the context of of descent theory in \cite{Grothendieck1960} under the name cat\'{e}gories fibr\'{e}es, fibered categories.
Since then, they have been used in diverse contexts, for example to define categorical models of dependent type (see \cite{Jacobs2001}).
In this thesis, we will take the point-of-view of considering functors as type refinement systems
advertised in \cite{MelliesZeilberger2015}.

A fibration of categories is a functor $p \colon \cE \to \cB$ such that for any morphism $f \colon A \to B$ in \cB and any object $S$ in \cE such that $p(S) = B$ one can compute an object $R$ in \cE, called the pullback of $R$ along $f$, such that $p(R) = A$ equipped with a map $\varphi \colon R \to S$ such that $p(\varphi) = f$ satisfying some universal property.
One intuition for understanding fibrations is given by the example $\cU \colon \Sub \to \Set$ where $\Sub$ is the category whose objects $(A, R \subseteq A)$ are sets equipped with a subset and morphisms $f \colon (A,R) \to (B,S)$ are functions such that $f(R) \subseteq S$, and $\cU$ is the functor forgetting the subset.
Then, given a function $f \colon A \to B$ and a subset $S \subseteq B$, the pullback of $S$ along $f$ is the inverse image $f^{-1}(S)$.
Using \Set{} as a denotational semantics for a programming language and a subset of (the interpretation of) a type as a property of this type, the pullback can be use to interpret the weakest precondition of Hoare logic (see \cite{MelliesZeilberger2015} for more on this).
Similarly, an opfibration is a functor $p \colon \cE \to \cB$ such that for any map $f \colon A \to B$ in \cB and any object $R \in \cE$ with $p(R) = A$ there is a universal choice of an object $S$ called the pushforward of $R$ along $f$.
It corresponds to the image $f(R)$ in $\cU \colon \Sub \to \Set$ and can be used to interpret the strongest postcondition in Hoare logic.

In a different direction, Hermida
noticed that the universal property of a pushforward is similar to the universal property of the tensor product in a multicategory.
To formalise this connection, he defines a notion of opfibration of multicategories in \cite{Hermida2004}.
He then proves that tensor products are a particular case of pushforwards, namely the pushwards for the unique functor $\cM \to \one$ into the terminal multicategory.
Under this reading, the pushforward of a family of objects $(R_i)_i$ along a multimap $f \colon A_1, \dots, A_n \to B$ is akin to a tensor product parametrised by $f$, hence in this thesis we will take the convention of denoting it by $\otimes_f(R_1,\dots,R_n)$.

In the first chapter of this part, we will define a notion of bifibration of polycategories.
We will show that pullbacks and pullforwards correspond to parametrised version of the universal objects introduced earlier.
Namely, when considering the unique functor $\cP \to \one$, they will correspond exactly to the universal objects.
In other words, a polycategory \cP is birepresentable iff $\cP \to \one$ is a bifibration.

The notion of bifibration can also be leveraged to lift a model of MLL to a more refined one.
That is, given a bifibration of polycategories $p \colon \cE \to \cB$, if \cB is birepresentable then so is \cE.
The tensor in \cE is given by pushing along the universal map of the tensor in \cB, and its par by pulling along the universal map of the par.
We will illustrate this by recovering the birepresentability of \FBanc{} from the bifibrational property of the forgetful functor $\cU \colon \FBanc \to \FVect$.
In particular, this explains the connection between the projective and injective norms being the norms of $\otimes$ and $\parr$, and the fact that these are the largest and smallest crossnorm.

Finally, in a second chapter we will define a notion of opfibration of virtual double categories.
In this context, it will be a functor of vdcs $p \colon \vdE \to \vdC$ where lists of composable horizontal maps can be pushed along cells.
This will provide a parametrised notion of composition of horizontal maps.
Opfibrations of vdcs will be used in the next part on the Grothendieck construction.

As far as the author is aware, everything in this part is original material.
It has been highly influenced by Mellies and Zeilberger's perspective on fibrations (see \cite{MelliesZeilberger2015,MelliesZeilberger2016}) as well as Hermida's work on fibrations of multicategories \cite{Hermida2004}.

\chapter{Bifibrations of polycategories}
\label{ch:bifib-poly}

In this section we introduce a notion of bifibration of polycategories, and prove that a polycategory is a representable polycategory with duals just in case it is bifibred over $\one$.
We find it convenient to begin by adapting some terminological and notational conventions from the study of type refinement systems \cite{MelliesZeilberger2015,MelliesZeilberger2016}.

\section{Definitions}

\begin{definition}
  A \defin{poly-refinement system} is defined as a functor of polycategories $p : \mathcal{E} \to \mathcal{B}$.
  Explicitly, $p$ sends objects $R \in \mathcal{E}$ to objects $p(R) \in \mathcal{B}$ and polymaps $\psi : R_1,...,R_m \to S_1,..., S_n$ in $\mathcal{E}$ to polymaps $p(f) : p(R_1),...,p(R_m) \to p(S_1),...,p(S_n)$ in $\mathcal{B}$ in such a way that identities and composition are preserved strictly.
  % $p(\id_R) = \id_{p(R)}$ and $p(g \circ_A f) = p(g) \circ_{p(A)} p(f)$.   
  We write $R \refs A$ (pronounced ``$R$ refines $A$'') to indicate that $p(R) = A$, and extend this to lists of objects in the obvious way, writing $\Pi \refs \Gamma$ to indicate that $\Pi = R_1,\dots,R_n$ and $\Gamma = A_1,\dots,A_n$ for some $R_1\refs A_1, \dots, R_n \refs A_n$.
  Finally, we write $\psi : \Pi \seq{f} \Sigma$ to indicate that $\psi$ is a polymap $\Pi \to \Sigma$ in $\mathcal{E}$ such that $p(\psi) = f$, with the implied constraint that $f : \Gamma \to \Delta$ where $\Pi \refs \Gamma$ and $\Sigma \refs \Delta$.
  
\end{definition}

\begin{remark}
  We will draw poly-refinement systems vertically. The top diagram will be in $\mathcal{E}$ and the bottom one in $\mathcal{B}$ with objects and polymaps directly above their image, e.g. preservation of composition is given by:
  \begin{center}\scalebox{0.7}{\input{func.tikz}}\end{center}
\end{remark}

\begin{definition}\label{defn:cartesian}
  Fix $p : \mathcal{E} \to \mathcal{B}$ a poly-refinement system and
  $\psi : \Pi_1, R, \Pi_2 \seq{g} \Sigma$ a polymap in $\mathcal{E}$
  with $R \refs A$.
  $\psi$ is \defin{in-cartesian in $R$} (relative to $p$),
  written $\psi : \Pi_1, \focin{R}, \Pi_2 \seq{g} \Sigma$,
  if for any polymap $\xi : \Pi_1, \Pi, \Pi_2 \seq{g \circ_A f} \Sigma_1,\Sigma,\Sigma_2$ there exists a unique polymap $\psi \backslash \xi : \Pi \seq{f} \Sigma_1, R, \Sigma_2$ such that $\xi = \psi \circ_R (\psi \backslash \xi)$.

  Dually, $\varphi : \Pi \seq{f} \Sigma_1, S, \Sigma_2$, with $S \refs B$, is \defin{out-cartesian in $S$},
  written $\varphi : \Pi \seq{f} \Sigma_1, \focout{S}, \Sigma_2$,
  if for any polymap $\xi : \Pi_1, \Pi, \Pi_2 \seq{g \circ_B f} \Sigma_1,\Sigma,\Sigma_2$ there is a unique polymap $\xi/\varphi : \Pi_1, S, \Pi_2 \seq{g} \Sigma$ such that $\xi = (\xi/\varphi) \circ_S \varphi$.

  Graphically, the definitions are summarised by the following diagram:

  % \begin{figure}
  %   \centering
  %   \scalebox{0.7}{\tikzfig{cart}}
  %   \caption{Cartesian polymaps}
  %   \label{fig:cart}
  % \end{figure}
  \begin{equation}\label{diag:cartesian}\tag{$\dagger$}
    \scalebox{0.7}{\input{cart.tikz}}
  \end{equation}
\end{definition}

\begin{prop}
  \label{prop:cart_comp}
  In-cartesian polymaps compose, in the sense that if $\varphi : \Pi_1, \focin{R}, \Pi_2 \seq{g} \Sigma_1, S, \Sigma_2$ and $\psi : \Pi_1', \focin{S}, \Pi_2' \seq{f} \Sigma'$ then $\psi \circ_S \varphi : \Pi'_1, \Pi_1, \focin{R}, \Pi_2, \Pi_2' \seq{g \circ_B f} \Sigma_1,\Sigma',\Sigma_2$.
  Similarly, out-cartesian maps compose in the sense that if $\varphi : \Pi \seq{g} \Sigma_1, \focout{S}, \Sigma_2$ and $\psi : \Pi_1', S, \Pi_2' \seq{f} \Sigma_1', \focout{T}, \Sigma_2'$ then $\psi \circ_S \varphi : \Pi'_1, \Pi, \Pi_2' \seq{g \circ_B f} \Sigma_1,\Sigma'_1,\focout{T},\Sigma'_2,\Sigma_2$.
\end{prop}

 \begin{proof}
   Given $\xi : \Pi_1', \Pi_1, \Pi'', \Pi_2, \Pi_2' \seq{p(\psi \circ_S \varphi) \circ_A d} \Sigma_1'', \Sigma_1, \Sigma_1', T, \Sigma_2', \Sigma_2, \Sigma_2'' $ .
   By functoriality of $p$ and associativity of composition we have \[p(\xi) = (p(\psi)\circ_B p(\varphi)) \circ_A d = p(\psi)\circ_B (p(\varphi) \circ_A d)\]
   Since $\psi$ is in-cartesian in $S$ there is a unique polymap $\epsilon$ such that $p(\epsilon) = p(\varphi) \circ_A d$ and $\xi = \psi \circ_S \epsilon$.
   Now since $\varphi$ is in-cartesian an $R$ there is a unique polymap $\delta$ such that $p(\delta) = d$ and $\epsilon = \varphi \circ_R \delta$.
   Putting these two results together there is a unique polymap $\delta$ such that $p(\delta) = d$ and $\xi = \psi \circ_S (\varphi \circ_R \delta) = (\psi \circ_S \varphi) \circ_R \delta$.
   Which proves that $\psi \circ_S \varphi$ is in-cartesian in $R$.
 \end{proof}

\begin{definition}
  A poly-refinement system $p : \mathcal{E} \to \mathcal{B}$ is said to be a \defin{pull-fibration} if for any $f : \Gamma_1, A, \Gamma_2 \to \Delta$ in $\mathcal{B}$ and any $\Pi_1 \refs \Gamma_1$, $\Pi_2 \refs \Gamma_2$, and $\Sigma \refs \Delta$ there is an object $\pull{f}{\Pi_1}{\Pi_2}{\Sigma} \refs A$ together with an in-cartesian polymap $\Pi_1,\focin{\pull{f}{\Pi_1}{\Pi_2}{\Sigma}},\Pi_2 \seq{f} \Sigma$.
  Dually, $p$ is said to be a \defin{push-fibration} if for any $f : \Gamma \to \Delta_1, B, \Delta_2$ in $\mathcal{B}$ and any $\Pi\refs \Gamma$, $\Sigma_1 \refs \Delta_1$, and $\Sigma_2 \refs \Delta_2$ there is an object $\push{f}{\Pi}{\Sigma_1}{\Sigma_2} \refs B$ together with an out-cartesian polymap $\Pi \seq{f} \Sigma_1, \focout{\push{f}{\Pi}{\Sigma_1}{\Sigma_2}}, \Sigma_2$.
  Finally, $p$ is said to be a \defin{bifibration} if it is both a pull-fibration and a push-fibration.  
\end{definition}

\begin{remark}
  When pulling along a map $f : A \to \Delta$ with only one input, we will write $\pullone{f}{\Sigma}$ as shorthand for $\pull{f}{}{}{\Sigma}$.
  Similarly when pushing along a map $f : \Gamma \to A$, we will write $\pushone{f}{\Gamma}$ for $\push{f}{\Gamma}{}{}$.
\end{remark}

\section{$\ast$-autonomous categories as bifibrations~of~polycategories}

\begin{prop}\label{prop:univiffcart}
  Let $\mathcal{P}$ be a polycategory.
  A polymap $u : \Gamma \to \Delta_1,A,\Delta_2$ (resp.~$u : \Gamma_1,A,\Gamma_2 \to \Delta$) is out-universal (resp.~in-universal) in $A$ iff it is out-cartesian (resp.~in-cartesian) with respect to the unique functor $\mathcal{P} \to \one$ into the terminal polycategory.
\end{prop}
\begin{proof}
Consider the definition of cartesian polymaps in \ref{diag:cartesian}.
If the bottom polycategory is the terminal one, then there is only one choice of polymap of each arity and the bottom part of the diagram as no information.
So in-cartesian and out-cartesian in this case only amounts to the unique factorisation property displayed on the top of the diagram.
This is exactly the factorisation property of universal polymaps. 
\end{proof}

\begin{proposition}\label{prop:*repiffbifib}
  $\mathcal{P}$ is a birepresentable polycategory iff $\mathcal{P} \to \one$ is a bifibration of polycategories.
\end{proposition}

So we get from the equivalence between birepresentable polycategories and $\ast$-autonomous categories.

\begin{theorem}
  \label{th:bif1}
  There is an equivalence between $\ast$-autonomous categories and bifibrations over the terminal polycategory $\one$.
\end{theorem}

We also expect that this result may be stated more precisely as an equivalence of 2-categories, but we leave this to future work.

One application of Theorem~\ref{th:bif1} is that it provides a way of decomposing a $\ast$-autonomous structure on a category, using elementary facts about cartesian polymaps.

\begin{prop}\label{prop:compcartfunc}
  For $p : \mathcal{P} \to \mathcal{E}$ and $q : \mathcal{E} \to \mathcal{B}$ poly-refinement systems and $\psi : \Pi_1, R, \Pi_2 \seq{g} \Sigma$ a polymap in $\mathcal{P}$,
if $\psi$ is $p$-in-cartesian in $R\refs A$ and $g$ is $q$-in-cartesian in $A \refs X$ then $\psi$ is $q \circ p$-in-cartesian in $R \refs X$.
\end{prop}
\begin{proof}
Consider a polymap $\xi$ in \cP such that $(q \circ p)(\xi) = (q \circ p)(\psi) \circ d$.
Then $q(p(\xi)) = q(p(\psi)) \circ d = q(g) \circ d$.
Since $g$ is $q$-in-cartesian there is a unique polymap $g \backslash p(\xi)$ with $q(g \backslash p(\xi)) = d$ such that \[p(\xi) = g \circ (g \backslash p(\xi)) = p(\psi) \circ (g \backslash p(\xi))\]
Now since $\psi$ is $p$-in-cartesian there is a unique polymap $\psi \backslash \xi$ with $p(\psi \backslash \xi) = g \backslash p(\xi)$ such that \[\xi = \psi \circ (\psi \backslash \xi)\]
But then, \[(q\circ p)(\psi \backslash \xi) = q(g \backslash p(\xi)) = g\]
So $\psi$ is $(q \circ p)$-in-cartesian.
\end{proof}

\begin{remark}
  Similarly, a $p$-out-cartesian polymap over a $q$-out-cartesian polymap is $(q \circ p)$-out-cartesian.
\end{remark}

\begin{prop}
  Let $p : \mathcal{E} \to \mathcal{B}$ be a poly-refinement system, and suppose that $\mathcal{B}$ is $\ast$-representable. 
  If $\mathcal{E}$ has all in-cartesian liftings of in-universal polymaps and all out-cartesian liftings of out-universal polymaps then $\mathcal{E}$ is a $\ast$-representable polycategory.
\end{prop}

\begin{proof}
  By Propositions~\ref{prop:univiffcart} and \ref{prop:compcartfunc}.
\end{proof}

\section{Frobenius monoids}
\label{sect:frob}

\begin{definition}
  In a polycategory $\mathcal{P}$ a \defin{Frobenius monoid} is an object $A$ equipped with a polymap $\overline{(m,n)_A} :A^m \to A^n$ for each $m,n \in \mathbb{N}$ such that $\overline{(1,1)_A} = \id_A$ and these polymaps are stable under composition.
\end{definition}

\begin{prop}
  Equivalently a Frobenius monoid in $\mathcal{P}$ is a functor $F : \one \to \mathcal{P}$.
\end{prop}

\begin{proof}
  The Frobenius monoid corresponds to $F(\ast)$ and the polymaps $\overline{(m,n)_{F(\ast)}}$ to $F(\overline{(m,n)})$.
  The properties needed on the polymaps are exactly functoriality of $F$.
\end{proof}

\begin{remark}
  For $\mathcal{P}$ representable with $\otimes = \parr$, this reduces to the unbiased definition of a Frobenius monoid in a monoidal category.
\end{remark}

\begin{definition}
  Given a poly-refinement system $p : \mathcal{E} \to \mathcal{B}$ and a Frobenius monoid $A$ in $\mathcal{B}$ the \defin{polyfiber} of $p$ over $A$, noted $p^{-1}(A)$ is the subcategory of $\mathcal{E}$ whose objects and polymaps are sent by $p$ to $A$ and the $\overline{(m,n)_A}$. 
\end{definition}

\begin{prop}
  $p^{-1}(A)$ is equivalent to the following pullback:\\
  
  \begin{tikzcd}[sep = huge]
    p^{-1}(A) \ar[r, hookrightarrow] \ar[d,"!"'] \arrow[dr, phantom, "\scalebox{1.5}{$\lrcorner$}" , very near start, color=black] & \mathcal{E} \ar[d,"p"]\\
    \one \ar[r,"A"'] & \mathcal{B}
  \end{tikzcd}
  \ \ where $A : \one \to \mathcal{B}$ is the functor associated to the object $A$.
\end{prop}

We have that bifibrations are stable under pullback.

\begin{prop}
  Given a poly-refinement system $p : \mathcal{E} \to \mathcal{B}$ and a functor $s : \mathcal{B}' \to \mathcal{B}$, let $\mathcal{E} \times_{\mathcal{B}} \mathcal{B}'$ be the pullback.
  
   \begin{tikzcd}[sep = huge]
    \mathcal{E} \times_{\mathcal{B}} \mathcal{B}' \ar[r,"\pi_1"] \ar[d,"\pi_2"'] \arrow[dr, phantom, "\scalebox{1.5}{$\lrcorner$}" , very near start, color=black] & \mathcal{E} \ar[d,"p"]\\
    \mathcal{B}' \ar[r,"s"'] & \mathcal{B}
  \end{tikzcd} 

  For a polymap $f: \Gamma_1, A, \Gamma_2 \to \Delta$ in $\mathcal{B}'$ and lists of objects $\Pi_1, \Pi_2,\Sigma$ in $\mathcal{E} \times_{\mathcal{B}} \mathcal{B}'$ lying over $\Gamma_1,\Gamma_2$ and $\Delta$, if there is a pullback $\pull{s(f)}{\pi_1(\Pi_1)}{\pi_1(\Pi_2)}{\pi_1(\Sigma)}$ in $\mathcal{E}$ then there is a pullback $\pull{f}{\Pi_1}{\Pi_2}{\Sigma}$ in $\mathcal{E} \times_{\mathcal{B}} \mathcal{B}'$.
\end{prop}
\begin{proof}
  $\mathcal{E} \times_{\mathcal{B}} \mathcal{B}'$ is the polycategory whose objects are pairs of objects $(E,B')$ of $\mathcal{E}$ and $\mathcal{B}$ such that $p(E) = s(B')$ and whose polymaps are pairs of polymaps $(f,b')$ such that $p(f) = s(b')$.
  
  Let us consider a polymap $f: \Gamma_1, A, \Gamma_2 \to \Delta$ in $\mathcal{B}'$ and lists of objects $(\Pi_1,\Gamma_1), (\Pi_2,\Gamma_2),(\Sigma,\Delta)$ in $\mathcal{E} \times_{\mathcal{B}} \mathcal{B}'$. 
  
From the pullback $\pull{s(f)}{\Pi_1}{\Pi_2}{\Sigma}$ in $\mathcal{E}$ with in-cartesian polymap $\varphi \colon \Pi_1, \pull{s(f)}{\Pi_1}{\Pi_2}{\Sigma}, \Pi_2 \to \Sigma$ we get a pullback $\pull{f}{(\Pi_1,\Gamma_1)}{(\Pi_2,\Gamma_2)}{(\Sigma,\Delta)} := (\pull{s(f)}{\Pi_1}{\Pi_2}{\Sigma},A)$ with in-cartesian polymap $(\varphi,f)$. 
To prove that $(\varphi,f)$ is incartesian we need to show that any polymap $(\psi,f \circ h)$ can be decomposed uniquely as \[(\psi,f \circ h) = (\varphi, f) \circ (\varphi,f)\backslash(\psi,f \circ h)\]
We can take $(\varphi,f)\backslash(\psi,f \circ h) := (\varphi\backslash \psi, h)$ using that $\varphi$ is incartesian and functoriality to prove that $p(\varphi \backslash \psi) = s(h)$.
\end{proof}

\begin{remark}
  Similarly if the pushforward exists in $\mathcal{E}$ it exists in $\mathcal{E} \times_{\mathcal{B}} \mathcal{B'}$.
\end{remark}

In particular, if we can push and pull along the polymap defining a Frobenius monoid, then the fibre over the Frobenius monoid is birepresentable.

\begin{cor}
  \label{cor:frob}
  Given a poly-refinement system $p : \mathcal{E} \to \mathcal{B}$ and a Frobenius monoid $(A,\{\overline{(m,n)_A}\})$ in $\mathcal{B}$ if all in-cartesian and out-cartesian liftings of $\overline{(m,n)_A}$ exist then $p^{-1}(A)$ is birepresentable.
\end{cor}

\section{Examples}
\label{sec:bifib:examples}

\begin{example}
  Let $\mathcal{E}$ and $\mathcal{B}$ be ordinary categories considered as degenerate polycategories with only unary co-unary maps (i.e., polymaps of arity and co-arity 1), and let $p : \mathcal{E} \to \mathcal{B}$ be an ordinary (strict) functor.
  Then $p$ is a pull-fibration, push-fibration or bifibration just in case it is an ordinary (Grothendieck) fibration, opfibration or bifibration.
  Similarly, if $\mathcal{E}$ and $\mathcal{B}$ are multicategories considered as polycategories with only co-unary maps, then $p$ is a push-fibration just in case it is a covariant fibration of multicategories in the sense of Hermida \cite{Hermida2004}, and more generally the polycategorical notions of pullback and pushforward coincide with the multicategorical ones described in \cite{Hoermann2017,LicataShulmanRiley2017}.
\end{example}

\begin{example}
  The forgetful functor $\Cat_\ast \to \Cat$ from the category of pointed (small) categories to the category of (small) categories is an opfibration of 2-categories.
  The pushforward of $(\mathcal{A},A)$ along $F : \mathcal{A} \to \mathcal{B}$ is $(\mathcal{B},F(A))$.
  Similarly the forgetful functor $\Adj_\ast \to \Adj$ of pointed adjunctions is a bifibration of 2-categories.
  Here a pointed adjunction between pointed categories $(\mathcal{A},A)$ and $(\mathcal{B},B)$ consist of an adjunction $F \dashv G : \mathcal{A} \to \mathcal{B}$ and a morphism $f : F(A) \to B$ in $\mathcal{B}$ - or equivalently of a morphism $g : A \to G(B)$ in $\mathcal{A}$.
  The pushforward is given by the image by $F$ while the pullback is given by the image of $G$.
  While working on the polycategorical Grothendieck correspondences we will define the 2-polycategory of multivariable adjunction $\MVar$.
  It also has a pointed variant $\MVar_\ast$.
  The forgetful functor induced is a bifibration of 2-polycategories.  
\end{example}

\begin{example}
Consider a multicategories \cE and \cB as polycategories with only co-unary polymaps.
A push-fibration then correspond to an opfibration of multicategories as defined in \cite{Hermida2004}.
Then, a multicategory \cE is representable iff it is opfibred over the terminal multicategory $\one$.
We also get a notion of fibration of multicategories from by asking for a pull-fibration of the polycategories.
In it a pullback $\pull{f}{\Gamma_1}{\Gamma_2}{A}$ can be seen as a parametrised notion of internal hom $\Gamma_1 \multimap_{f} A \multimapinv \Gamma_2$.
Indeed, a multicategory is fibred over the terminal one iff it is closed.
Putting these two results together, a multicategory is bifibred over $\one$ iff it is the underlying multicategory of a monoidal closed category.
\end{example}

\begin{example}
From the example above we get a notion of fibration of monoidal categories.
Monoidal bifibrations have been studied in \cite{Shulman2007}.
While monoidal opfibrations correspond to opfibrations of representable multicategories, this is not the case for monoidal fibrations.
This is because in monoidal fibrations, $\otimes$ should preserves the in-cartesian morphisms.
However, since the tensor is defined as a pushforward it only preserves opcartesian morphisms.
\end{example}

It is however possible to recover the notion of monoidal fibration using pull-fibrations of polycategories.

\begin{example}
Given a functor of multicategories $p \colon \cE \to \cB$, one can define a functor of polycategories $p^\op \colon \cE^\op \to \cB^\op$ where $\cE^\op$ is the polycategory with only unary polymaps $A \to \Gamma$ corresponding to multimaps $\Gamma \to A$ in \cE, and similarly for \cB.
If the multicategories \cE and \cB are representable then their tensor products define par products in $\cE^\op$ and $\cB^\op$ and $p$ is a monoidal fibration iff $p^\op$ is a pull-fibration.
\end{example}

We proved that an opfibration of multicategories $p \colon \cE \to \one$ defines a monoidal category.
Furthermore, a functor $\one \to \cE$ is a monoid internal to $\cE$.
This let us define a monoid internal to a monoidal category as a section of an opfibration into $\one$, the free multicategory containing a monoid.
This recipe can be extended to other algebraic structures by:
\begin{itemize}
\item considering the free multicategory/category containing such an structure
\item an opfibration/bifibration over it will be a categorified version of the algebraic structure
\item a section of it will be an an internal version of it
\end{itemize}
An another example that we encounter is for Frobenius monoids internal to $\ast$-autonomous categories.

\begin{example}
Recall that \Act{} is the free multicategory containing a monoid action.
It has two objects $\ast, \star$, a family of multimaps making $\ast$ a monoid and a family of multimaps defining an action of $\ast$ on $\star$.

Consider a functor of multicategories $p \colon \cE \to \Act$.
The polyfibre over the monoid $\ast$, $p^{-1}(\ast)$ defines a multicategory.
It consists of all the objects in $\cE$ sent to $\ast$ and all the multimaps in \cE sent to the multiplication of $\ast$.
The fibre over $\star$ $p^{-1}(\star)$ defines a category.
If $p$ is an opfibration then $p ^{-1}(\ast)$ is representable.
Furthermore, by pushing objects of $p^{-1}(\ast)$ over the action in $\Act$, we get an action $p^{-1}(\ast)$ on $p^{-1}(\star)$, namely an actegory.
This defines a correspondence between actegories and opfibred multicategories over \Act.
A section of $p$ is then a monoid action internal to an actegory.
This is a monoid internal to the monoidal category acting on an object of the category acted on in a coherent way.

If instead of an opfibration $p \colon \cE \to \Act$ we only have pushforwards over the actions (but not the multiplication) this defines a multicategory acting on a category. 
\end{example}

\begin{rk}
The fact that opfibrations/bifibrations give a categorified version of the algebraic structures is explained by the B\'{e}nabou-Grothendieck correspondence that we will explore at the end of this thesis.
For example, an opfibration of multicategories $\cE \to \cB$ correspond to a pseudofunctor $\cB \to \Cat$.
When $\cB$ is the free multicategory on an algebraic structure, this gives a non-strict version of this algebraic structure in \Cat.
\end{rk}

\begin{rk}
Instead of \Act{} it should be possible to define \PolyAct{}, the free polycategory containing a monoid $\ast$ acting on a Frobenius monoid $\star$.
Then by considering some fibred structures on a functor of polycategories $p \colon \cE \to \PolyAct$ we should recover the polyactegories and linear actegories of \cite{CockettPastro2007}.
\end{rk}

\section{Forgetful functor from Banach spaces}
\label{sec:Banach2}

We will use proposition \ref{prop:compcartfunc} to derive the birepresentability of the polycategory $\FBanc$ defined in \ref{sec:Banach1}.
In order to do that we consider the  forgetful functor $\FBanc \to \FVect$.
We want to characterise the polymaps that admit cartesian liftings.

In the following, given a finite list of Banach spaces $(A_i, \|-\|_{A_i})$ with $\Gamma := A_1,\dots,A_n$ we will write $\|-\|_{\Gamma} := \|-\|_{A_1}\dots\|-\|_{A_n}$ in equations.
We will also write $\|-\|_{\Gamma} := \|-\|_{A_1},\dots,\|-\|_{A_n}$ for the list of the norms.

Given a polylinear map $g \colon \Gamma_1, A, \Gamma_2 \to \Delta$ and norms $\|-\|_{\Gamma_i}, \|-\|_{\Delta}$ we want to define a norm on $A$ $\|-\|_g$ that is a pullback.
In particular, it should make $g$ contractive.
That is, if $\|x\|_g \leq 1$ and given $\overrightarrow{a}^k, \overrightarrow{\varphi}$ composed of subunits, i.e., vectors of norm lesser than 1, we should have $|(\overrightarrow{\varphi})g(\overrightarrow{a}^1,x,\overrightarrow{a}^2)| \leq 1$.
To ensure that we will ask that it is true for their supremum.

\begin{definition}
  Given $g \colon \Gamma_1, A, \Gamma_2 \to \Delta$ a polylinear map with norms $\|-\|_\Gamma, \|-\|_{\Delta_i}$, we define a function $\|-\|_g \colon A \to \mathbb{K}$ by \[\|x\|_g :=  \sup\limits_{\substack{\overrightarrow{a}^k \in \Gamma^k, \overrightarrow{\varphi} \in \Delta^\ast\\ \|a_i^k\|_{A_i^k}\leq 1, \|\varphi_j\|_{B_j^\ast} \leq 1}}|(\overrightarrow{\varphi})g(\overrightarrow{a}^1,x,\overrightarrow{a}^2)|\]
\end{definition}

This does not define a norm in general.
Indeed, take $g \colon A \to B$ for simplicity.
Then $\|x\|_g := \sup_{\|\varphi\|_{B^\ast} \leq 1} |(\varphi)g(x)|$.
Now $\|x\|_g = 0$ iff for any $\|\varphi\|_{A^\ast} \leq 1$, we have $\varphi(g(x)) = 0$, which is true iff $g(x) = 0$.
Then if $g$ is not injective, there an element $x \neq 0$ such that $\|x\|_g = 0$.

\begin{prop}
  $\|-\|_g$ defines a pseudonorm on $A$.
\end{prop}
\begin{proof}
   This follows from linearity of $g$ and the properties of the norms.
\end{proof}

We want to characterise the polymaps for which this is a norm.

\begin{definition}
  $g$ is \defin{injective in $A$} - or \defin{$A$-injective} - if \[(\forall \overrightarrow{a^i},\forall \overrightarrow{\varphi},\ (\overrightarrow{\varphi})g(\overrightarrow{a^1},x,\overrightarrow{a^2}) = 0) \Rightarrow x = 0\]
\end{definition}

\begin{definition}
  The \defin{$A$-kernel} of $g$ is the set \[\Ker_{A}(g) := \set{x \in A}{(\overrightarrow{\varphi})g(\overrightarrow{a^1},x,\overrightarrow{a^2}) = 0\ \forall \overrightarrow{a^i},\forall\overrightarrow{\varphi}}\]
\end{definition}

The $A$-kernel of $g$ forms a vector space.
$g$ is $A$-injective iff its $A$-kernel is trivial.

\begin{remark}
  A polylinear map $g \colon A \to B$ is $A$-injective if it is injective as a linear map.
\end{remark}

\begin{prop}
  $\|-\|_g$ is a norm iff $g$ is $A$-injective.
\end{prop}

It is worth noticing that the fact that it is a norm only depends on $g$ and not on any properties of the norms on $\Gamma_i$ and $\Delta$.

\begin{prop}
  For $g$ $A$-injective and norms on $\Gamma$ and $\Delta_i$, the norm $\|-\|_g$ makes $g$ contractive.
\end{prop}

This norm defines a pullback in $\FBanc$.

\begin{prop}
  Given a $A$-injective polylinear map $g \colon \Gamma_1,A,\Gamma_2 \to \Delta$ and norms $\|-\|_{\Gamma_i},\|-\|_\Delta$, the pullback is given by \[\pull{g}{(\Gamma_1,\|-\|_{\Gamma_1})}{(\Gamma_2,\|-\|_{\Gamma_2})}{(\Delta,\|-\|_\Delta)} = (A,\|-\|_g)\]
\end{prop}
\begin{proof}
Consider a polylinear map $f \colon \Gamma \to \Delta_1, A, \Delta_2$ such that $g\circ f$ is contractive.
We want to prove that $f$ is contractive when $A$ is equipped with $\|-\|_g$.

First by contractivity of $g \circ f$ we have that 
\[ |(\overrightarrow{\varphi_1},\overrightarrow{\varphi},\overrightarrow{\varphi_2})g\circ f(\overrightarrow{a_1},\overrightarrow{a},\overrightarrow{a_2}) | \leq 1 \]
for any subunits.
By definition of composition this means that
\begin{equation}
\label{eqn:pullback1}
 |(\overrightarrow{\varphi_1},(\overrightarrow{\varphi})g(\overrightarrow{a_1},-,\overrightarrow{a_2}),\overrightarrow{\varphi_2})f(\overrightarrow{a}) | \leq 1
\end{equation}

What we want to prove is that $f$ is contractive, i.e.
\[ |(\overrightarrow{\varphi_1}, \xi, \overrightarrow{\varphi_2})f(\overrightarrow{a})| \leq 1 \]
for any subunits.
First, notice that since $(\overrightarrow{\varphi_1},-,\overrightarrow{\varphi_2})f(\overrightarrow{a})$ is a functional on $\rdual{A}$ and $A \rightarrow A^{\ast\ast}$ is an isomorphism since $A$ is finite dimensional, there exists a unique $u \in A$ such that for any $\xi \in \rdual{A}$
\[ (\overrightarrow{\varphi_1},\xi,\overrightarrow{\varphi_2})f(\overrightarrow{a})= \xi(u)\]
So contractivity of $f$ is equivalent to ask that for any subunit $\xi$
\[ |\xi(u)| \leq 1\]
Now since $g \circ f$ is contractive we have for any subunits
\[ (\overrightarrow{\varphi}g(\overrightarrow{a_1},u,\overrightarrow{a_2})| = |(\overrightarrow{\varphi_1},(\overrightarrow{\varphi})g(\overrightarrow{a_1},-,\overrightarrow{a_2}),\overrightarrow{\varphi_2})f(\overrightarrow{a})| = |(\overrightarrow{\varphi_1},\overrightarrow{\varphi},\overrightarrow{\varphi_2})g\circ f(\overrightarrow{a_1},\overrightarrow{a},\overrightarrow{a_2})| \leq 1 \]
So $\|u\|_g \leq 1$.
This means by definition that for any subunit $\xi \in \rdual{A}$, $||\xi(u)| \leq 1$, and $f$ is contractive.
\end{proof}

So we have in-cartesian liftings of any polylinear map that is injective in the input considered.
The injectivity condition is only needed for $\|-\|_f$ to be a norm, otherwise it is still a seminorm, i.e., $\|x\|_f \ge 0$ for all $x$ and $\|0\|_f = 0$, but $\|x\|_f = 0$ does not imply $x = 0$.

\begin{cor}
  There is a polycategory $\mathbf{FBan_1^{ps}}$ of finite dimensional complete seminormed vector spaces and contractive polylinear maps that comes with a forgetful functor that is pull-fibred.
\end{cor}

Now we want to determine which polylinear maps have out-cartesian liftings.

Given a polylinear map $f \colon \Gamma \to \Delta_1, A, \Delta_2$ and norms of $\Gamma, \Delta_i$, we want to define a norm $\|-\|^f$ on $A$.
Remember that $(\overrightarrow{\varphi_1},-,\overrightarrow{\varphi_2})f(\overrightarrow{a})$ defines an element of $A$.
We would like that \[\|(\overrightarrow{\varphi_1},-,\overrightarrow{\varphi_2})f(\overrightarrow{a})\|^f = \|\overrightarrow{\varphi_1}\|_{(\Delta_1)^\ast}\|\overrightarrow{\varphi_2}\|_{(\Delta_2)^\ast}\|\overrightarrow{a}\|_{\Gamma}\]
Then if all the vectors considers are subunits, this norm is less than 1.
Furthermore, since $A$ is a vector space we can consider linear combinaison of elements such as above.
In this case we would like to define the norm as the linear combination of the norms.
However this is not well-defined since a vector of $A$ could potentially be decomposed as a linear combination of images of $f$ in multiple ways.
So we take the infimum on all these decomposition.

\begin{definition}
  For $f : \Gamma \to \Delta_1, A, \Delta_2$ and families of norms $\|-\|_{\Gamma},\|-\|_{\Delta_1},\|-\|_{\Delta_2}$, we define a function $\|-\|^f : A \to \bar{\mathbb{K}}$ where $\bar{\mathbb{K}}$ is the completion of $\mathbb{K}$, i.e., we add a point at infinity. It is given by
  $\|y\|^f := \inf\limits_{\check{y} = \sum\limits_i (\vect{\varphi}_{1,i},-,\vect{\varphi}_{2,i})f(\vect{a}_i)} \sum\limits_i \|\vect{\varphi}_{1,i}\|\|\vect{\varphi}_{2,i}\|\|\vect{a}_i\|$
  where the inf is over all the decompositions of $\check{y}$, the functional on $\rdual{A}$ given by evaluation at $y$.
\end{definition}

\begin{prop}
  $\|-\|^f$ is an extended norm, i.e., a norm with value in $\bar{\mathbb{K}}$.
\end{prop}
\begin{proof}
This is an extended norm since the decomposition may not exist in which case we take the infimum of an empty set.
The properties of an extended norm follows from linearity of $f$.
\end{proof}

\begin{definition}
  $f : \Gamma \to \Delta_1, A, \Delta_2$ is \defin{$A$-surjective} if \[\forall y \in A, \exists \vect{\varphi}_{1,i}, \vect{\varphi}_{2,i}, \vect{a}_i,\ y = \sum\limits_i (\vect{\varphi}_{1,i},-,\vect{\varphi}_{2,i})f(\vect{a}_i)\]
  The $A$-image of $f$ is the set $\Img_A(f) := \{ \sum\limits_i (\vect{\varphi}_{1,i},-,\vect{\varphi}_{2,i})f(\vect{a}_i) \}$.
\end{definition}

\begin{prop}
  $\Img_A(f)$ forms a vector space.
  $f$ is $A$-surjective iff $\Img_A(f) = A$.
\end{prop}

\begin{remark}
  A linear map $f \colon A \to B$ is $B$-surjective iff it is surjective.
  Indeed if for $y \in B$ there are $x_i$ such that $y = \sum\limits_if(x_i)$ then by linearity $y = f(\sum\limits_ix_i)$.
\end{remark}

\begin{prop}
  For $f$ and families of norms $\|-\|_{\Gamma},\|-\|_{\Delta_1},\|-\|_{\Delta_2}$, $\|-\|^f$ is a norm iff $f$ is $A$-surjective.
\end{prop}

\begin{prop}
  For $f$ $A$-surjective and families of norms as usual, $\|-\|^f$ makes $f$ contractive.
\end{prop}

\begin{proof}
We want to prove that for any subunits, \[|(\vect{\varphi}_1,\varphi,\vect{\varphi}_2)f(\vect{a})| \leq 1\]
This is equivalent to prove
\[
	|\varphi((\vect{\varphi}_1,-,\vect{\varphi}_2)f(\vect{a}))| \leq 1
\]

Now since $(\|\varphi\|^f)' \leq 1$ we have that $|\varphi(x)| \leq 1$ for any $\|x\|^f \leq 1$.
But then \[ \|(\vect{\varphi}_1,-,\vect{\varphi}_2)f(\vect{a})\|^f \leq \|\vect{\varphi}_1\|_{\Delta_1^\ast}\|\vect{\varphi}_2\|_{\Delta_2^\ast}\|\vect{a}\|_\Gamma \leq 1\]
because all of the vectors considered are subunits.

 \end{proof}

This norm defines a pushforward on $\FBanc$.

\begin{prop}
  For $f : \Gamma \to \Delta_1, A, \Delta_2$ an $A$-surjective polylinear map and the usual families of norms, we get the pushforward $\push{f}{\Gamma}{\Delta_1}{\Delta_2} = (A,\|-\|^f)$.
\end{prop}

\begin{proof}
Suppose that $g \colon \Gamma_1, A, \Gamma_2 \to \Delta$ is such that $g\circ f$ is contractive, i.e. \[|(\vect{\varphi}_1, \vect{\varphi},\vect{\varphi}_2)g \circ f(\vect{a}_1, \vect{a},\vect{a}_2)| \leq 1\]
for any subunits.
We have to prove that then, $g$ is contractive with $\|-\|^f$ on $A$.
That is, for any subunits \[|(\vect{\varphi})g(\vect{a}_1, x, \vect{a}_2)| \leq 1\]
Now since \[\|x\|^f = \inf\limits_{x = \sum_i (\vect{\psi}_{1,i},-,\vect{\psi}_{2,i})f(\vect{b}_i)} \sum_i \|\vect{\psi}_{1,i}\|\|\vect{\psi}_{2,i}\|\|\vect{a}_i\| \leq 1\]
Given any decomposition \[x =  \sum_i (\vect{\psi}_{1,i},-,\vect{\psi}_{2,i})f(\vect{b}_i)\]
we have 
\begin{align*}
|(\vect{\varphi})g(\vect{a}_1, x, \vect{a}_2)| &= |(\vect{\varphi})g(\vect{a}_1, \sum_i (\vect{\psi}_{1,i},-,\vect{\psi}_{2,i})f(\vect{b}_i), \vect{a}_2)|\\
&\leq \sum_i |(\vect{\varphi})g(\vect{a}_1, (\vect{\psi}_{1,i},-,\vect{\psi}_{2,i})f(\vect{b}_i), \vect{a}_2)|\\
&= \sum_i |(\vect{\psi}_{1,i},\vect{\varphi},\vect{\psi}_{2,i})g \circ f(\vect{a}_1, \vect{b}_i,\vect{a}_2)|\\
&= \sum_i \|\vect{\psi}_{1,i}\|\|\vect{\psi}_{2,i}\|\|\vect{b}_i\||(\vect{\psi}_{1,i}',\vect{\varphi},\vect{\psi}_{2,i}')g \circ f(\vect{a}_1, \vect{b}_i',\vect{a}_2)|
\end{align*}
where $\vect{\psi}_{1,i}'$ is the normalisation of $\vect{\psi}_{1,i}$, i.e., we divide each vector by its norm to make it a unit.
Then, since $g \circ f$ is contractive:
\begin{align*}
|(\vect{\varphi})g(\vect{a}_1, x, \vect{a}_2)| & \leq \sum_i \|\vect{\psi}_{1,i}\| |\vect{\psi}_{2,i}\| \|\vect{b}_i\| |(\vect{\psi}_{1,i}',\vect{\varphi},\vect{\psi}_{2,i}')g \circ f(\vect{a}_1, \vect{b}_i',\vect{a}_2)|\\
&\leq \sum_i \|\vect{\psi}_{1,i}\| \|\vect{\psi}_{2,i}\|\|\vect{b}_i\|
\end{align*}
Since this inequality holds for any decomposition it holds for their infimum which is $\|x\|^f$.
And since $x$ is a subunits we get
\[|(\vect{\varphi})g(\vect{a}_1, x, \vect{a}_2)| \leq 1\]
\end{proof}

So we can take the out-cartesian lifting of any polymap that is surjective in the considered output.

\begin{cor}
  There are polycategories $\mathbf{FBan_1^{ex}}$ and $\mathbf{FBan_1^{ex,ps}}$ of f.d. extended normed/seminormed vector spaces and polylinear maps with forgetful functors that are push-fibred and bifibred respectively. 
\end{cor}

When considering $\FBanc$ even without semi-/extended norms, there are still enough cartesian polymaps to lift the $\ast$-representability of $\FVect$.

\begin{prop}
  In $\FVect$, the universal polylinear maps $m_{A,B} : A,B \to A \otimes B$, $w_{A,B} : A \otimes B \to A,B$ and $rcap_A : A^\ast, A \to \cdot$ are $A \otimes B$-surjective, $A \otimes B$-injective and $A^\ast$-injective, respectively.
\end{prop}

\begin{proof}
   By definition of the tensor product any $u \in A \otimes B$ is a linear combination of elements from $A$ and $B$, $u = \sum\limits_ia_i \otimes b_i$.
   So $m_{A,B}$ is $A\otimes B$-surjective.
   $A\otimes B$-injectivity of $w_{A,B}$ is trivial.
   Finally given $\varphi \in A^\ast$ if, for all $a \in A$, $\varphi(a) = 0$ then $\varphi = 0$.
\end{proof}

\begin{cor}
  $\FBanc$ is birepresentable.
\end{cor}

\begin{remark}
  We get the projective, injective and dual norm using the norms above:  $\|-\|_{A \otimes B} = \|-\|_{m_{A,B}}$, $\|-\|_{A \parr B} = \|-\|^{w_{A,B}}$ and $\|-\|_{A^\ast} = \|-\|^{rcap_A}$.
  The fact that the projective and injective crossnorms are extremal follows directly from the factorisation properties of the cartesian polymaps $m_{A,B}$ and $w_{A,B}$.
\end{remark}

\chapter{Pushfibrations of virtual~double~categories}

The notion of 2-pushfibration of bicategories that we will present here is different from the notion of 2-(op)fibration developed by Hermida and also studied in \cite{Bakovic2011} and \cite{Buckley2014} for the purpose of a bicategorical Grothendieck correspondence.
In these works, a 2-fibration comes equipped with pullbacks or pushforwards both for morphisms and for 2-morphisms.
Furthermore, the pullback or pushforward of a cartesian 2-morphism is asked to be cartesian.

In order to talk about the Bénabou-Grothendieck correspondence, we will only need the existence of pushforward along 2-morphisms.
Hence, our notion of 2-pushfibration will have only lifting of 2-morphisms.
Also, our main example $\cU \colon \Dist_\ast \to \Dist$ is a strict functor.
This makes it easier to define the notion of pushforward of 2-cell and it lets us make use of the vertical diagrammatic representation of functors.
So we will required for a 2-pushfibration to be a strict 2-functor.

\section{Virtual double pushfibration}

In the following, we will fix a functor of vdc $p \colon \vdE \to \vdB$.
We will represent it diagrammatically vertically.

\begin{definition}
A cell $A$ in \cE with image $\alpha = p(A)$ in \cB

% https://q.uiver.app/?q=WzAsMTYsWzAsMCwiUl8wIl0sWzIsMCwiUl8xIl0sWzMsMCwiXFxkb3RzIl0sWzQsMCwiUl97bi0xfSJdLFs2LDAsIlJfbiJdLFswLDIsIlJfMCJdLFs2LDIsIlJfbiJdLFswLDQsIkFfMCJdLFsyLDQsIkFfMSJdLFs0LDQsIkFfe24tMX0iXSxbNiw0LCJBX24iXSxbMyw0LCJcXGRvdHMiXSxbMCw2LCJBXzAiXSxbNiw2LCJBX24iXSxbOCwxLCJcXHZkRSJdLFs4LDUsIlxcdmRCIl0sWzAsMSwiXFxwaV8xIiwxXSxbMyw0LCJcXHBpX24iLDFdLFs1LDYsIlxcc2lnbWEiLDFdLFs0LDYsIiIsMSx7ImxldmVsIjoyLCJzdHlsZSI6eyJoZWFkIjp7Im5hbWUiOiJub25lIn19fV0sWzAsNSwiIiwxLHsibGV2ZWwiOjIsInN0eWxlIjp7ImhlYWQiOnsibmFtZSI6Im5vbmUifX19XSxbNyw4LCJwXzEiLDFdLFs5LDEwLCJwX24iLDFdLFsxMiwxMywicyIsMV0sWzcsMTIsIiIsMSx7ImxldmVsIjoyLCJzdHlsZSI6eyJoZWFkIjp7Im5hbWUiOiJub25lIn19fV0sWzEwLDEzLCIiLDEseyJsZXZlbCI6Miwic3R5bGUiOnsiaGVhZCI6eyJuYW1lIjoibm9uZSJ9fX1dLFsxNCwxNSwicCIsMV0sWzIsMTgsIlxcbWF0aHJte0F9IiwxLHsic2hvcnRlbiI6eyJ0YXJnZXQiOjIwfX1dLFsxMSwyMywiXFxhbHBoYSIsMSx7InNob3J0ZW4iOnsidGFyZ2V0IjoyMH19XV0=
\begin{tikzcd}[ampersand replacement=\&]
	{R_0} \&\& {R_1} \& \dots \& {R_{n-1}} \&\& {R_n} \\
	\&\&\&\&\&\&\&\& \vdE \\
	{R_0} \&\&\&\&\&\& {R_n} \\
	\\
	{A_0} \&\& {A_1} \& \dots \& {A_{n-1}} \&\& {A_n} \\
	\&\&\&\&\&\&\&\& \vdB \\
	{A_0} \&\&\&\&\&\& {A_n}
	\arrow["{\pi_1}"{description}, from=1-1, to=1-3]
	\arrow["{\pi_n}"{description}, from=1-5, to=1-7]
	\arrow[""{name=0, anchor=center, inner sep=0}, "\sigma"{description}, from=3-1, to=3-7]
	\arrow[Rightarrow, no head, from=1-7, to=3-7]
	\arrow[Rightarrow, no head, from=1-1, to=3-1]
	\arrow["{p_1}"{description}, from=5-1, to=5-3]
	\arrow["{p_n}"{description}, from=5-5, to=5-7]
	\arrow[""{name=1, anchor=center, inner sep=0}, "s"{description}, from=7-1, to=7-7]
	\arrow[Rightarrow, no head, from=5-1, to=7-1]
	\arrow[Rightarrow, no head, from=5-7, to=7-7]
	\arrow["p"{description}, from=2-9, to=6-9]
	\arrow["{\mathrm{A}}"{description}, shorten >=7pt, Rightarrow, from=1-4, to=0]
	\arrow["\alpha"{description}, shorten >=7pt, Rightarrow, from=5-4, to=1]
\end{tikzcd}

is \defin{opcartesian} if any cell $\Lambda$ lying over a factorisation through $\alpha$,

\resizebox{\hsize}{!}{
\begin{tikzcd}[ampersand replacement=\&]
	{S_0} \&\& {S_1} \& \dots \& {S_{m-1}} \&\& {R_0} \&\& {R_1} \& \dots \& {R_{n-1}} \&\& {R_n} \&\& {S_1'} \& \dots \& {S_{m'-1}'} \&\& {S_{m'}'} \\
	\\
	{S_0} \&\&\&\&\&\&\&\&\&\&\&\&\&\&\&\&\&\& {S_{m'}'} \&\& \vdE \\
	\\
	{T_0} \&\&\&\&\&\&\&\&\&\&\&\&\&\&\&\&\&\& {T_1} \\
	\\
	{B_0} \&\& {B_1} \& \dots \& {B_{m-1}} \&\& {A_0} \&\& {A_1} \& \dots \& {A_{n-1}} \&\& {A_n} \&\& {B_1'} \& \dots \& {B_{m'-1}'} \&\& {B_{m'}'} \\
	\\
	{B_0} \&\& {B_1} \& \dots \& {B_{m-1}} \&\& {A_0} \&\&\&\&\&\& {A_n} \&\& {B_1'} \& \dots \& {B_{m'-1}'} \&\& {B_{m'}'} \&\& \vdB \\
	\\
	{C_0} \&\&\&\&\&\&\&\&\&\&\&\&\&\&\&\&\&\& {C_1}
	\arrow["{\pi_1}"{description}, from=1-7, to=1-9]
	\arrow["{\pi_n}"{description}, from=1-11, to=1-13]
	\arrow["{p_1}"{description}, from=7-7, to=7-9]
	\arrow["{p_n}"{description}, from=7-11, to=7-13]
	\arrow[""{name=0, anchor=center, inner sep=0}, "s"{description}, from=9-7, to=9-13]
	\arrow[Rightarrow, no head, from=7-7, to=9-7]
	\arrow[Rightarrow, no head, from=7-13, to=9-13]
	\arrow["p"{description}, from=3-21, to=9-21]
	\arrow["{\rho_m}"{description}, from=1-5, to=1-7]
	\arrow["{\rho_1}"{description}, from=1-1, to=1-3]
	\arrow["{\rho_1'}"{description}, from=1-13, to=1-15]
	\arrow["{\rho_{m'}'}"{description}, from=1-17, to=1-19]
	\arrow[Rightarrow, no head, from=1-1, to=3-1]
	\arrow["\varphi"{description}, from=3-1, to=5-1]
	\arrow[Rightarrow, no head, from=1-19, to=3-19]
	\arrow["\psi"{description}, from=3-19, to=5-19]
	\arrow[""{name=1, anchor=center, inner sep=0}, "\tau"{description}, from=5-1, to=5-19]
	\arrow["{r_1}"{description}, from=7-1, to=7-3]
	\arrow["{r_m}"{description}, from=7-5, to=7-7]
	\arrow["{r_1'}"{description}, from=7-13, to=7-15]
	\arrow["{r_{m'}'}"{description}, from=7-17, to=7-19]
	\arrow[Rightarrow, no head, from=7-1, to=9-1]
	\arrow["{r_1}"{description}, from=9-1, to=9-3]
	\arrow[Rightarrow, no head, from=7-3, to=9-3]
	\arrow[Rightarrow, no head, from=7-5, to=9-5]
	\arrow["{r_m}"{description}, from=9-5, to=9-7]
	\arrow["{r_1'}"{description}, from=9-13, to=9-15]
	\arrow[Rightarrow, no head, from=7-15, to=9-15]
	\arrow[Rightarrow, no head, from=7-17, to=9-17]
	\arrow[Rightarrow, no head, from=7-19, to=9-19]
	\arrow["{r_{m'}'}"{description}, from=9-17, to=9-19]
	\arrow[""{name=2, anchor=center, inner sep=0}, "t"{description}, from=11-1, to=11-19]
	\arrow["f"{description}, from=9-1, to=11-1]
	\arrow["g"{description}, from=9-19, to=11-19]
	\arrow["\alpha"{description}, shorten >=7pt, Rightarrow, from=7-10, to=0]
	\arrow["\beta"{description}, shorten <=9pt, shorten >=9pt, Rightarrow, from=0, to=2]
	\arrow["\Lambda"{description}, shorten >=16pt, Rightarrow, from=1-10, to=1]
\end{tikzcd}
}

 can be uniquely factored through $\mathrm{A}$:

\resizebox{\hsize}{!}{
\begin{tikzcd}[ampersand replacement=\&]
	{S_0} \&\& {S_1} \& \dots \& {S_{m-1}} \&\& {R_0} \&\& {R_1} \& \dots \& {R_{n-1}} \&\& {R_n} \&\& {S_1'} \& \dots \& {S_{m'-1}'} \&\& {S_{m'}'} \\
	\\
	{S_0} \&\& {S_1} \& \dots \& {S_{m-1}} \&\& {R_0} \&\&\&\&\&\& {R_n} \&\& {S_1'} \& \dots \& {S_{m'-1}'} \&\& {S_{m'}'} \&\& \vdE \\
	\\
	{T_0} \&\&\&\&\&\&\&\&\&\&\&\&\&\&\&\&\&\& {T_1} \\
	\\
	{B_0} \&\& {B_1} \& \dots \& {B_{m-1}} \&\& {A_0} \&\& {A_1} \& \dots \& {A_{n-1}} \&\& {A_n} \&\& {B_1'} \& \dots \& {B_{m'-1}'} \&\& {B_{m'}'} \\
	\\
	{B_0} \&\& {B_1} \& \dots \& {B_{m-1}} \&\& {A_0} \&\&\&\&\&\& {A_n} \&\& {B_1'} \& \dots \& {B_{m'-1}'} \&\& {B_{m'}'} \&\& \vdB \\
	\\
	{C_0} \&\&\&\&\&\&\&\&\&\&\&\&\&\&\&\&\&\& {C_1}
	\arrow["{\pi_1}"{description}, from=1-7, to=1-9]
	\arrow["{\pi_n}"{description}, from=1-11, to=1-13]
	\arrow[Rightarrow, no head, from=1-13, to=3-13]
	\arrow[Rightarrow, no head, from=1-7, to=3-7]
	\arrow["{p_1}"{description}, from=7-7, to=7-9]
	\arrow["{p_n}"{description}, from=7-11, to=7-13]
	\arrow[""{name=0, anchor=center, inner sep=0}, "s"{description}, from=9-7, to=9-13]
	\arrow[Rightarrow, no head, from=7-7, to=9-7]
	\arrow[Rightarrow, no head, from=7-13, to=9-13]
	\arrow["p"{description}, from=3-21, to=9-21]
	\arrow["{\rho_m}"{description}, from=1-5, to=1-7]
	\arrow["{\rho_1}"{description}, from=1-1, to=1-3]
	\arrow["{\rho_1'}"{description}, from=1-13, to=1-15]
	\arrow["{\rho_{m'}'}"{description}, from=1-17, to=1-19]
	\arrow[Rightarrow, no head, from=1-1, to=3-1]
	\arrow["\varphi"{description}, from=3-1, to=5-1]
	\arrow[Rightarrow, no head, from=1-3, to=3-3]
	\arrow["{\rho_1}"{description}, from=3-1, to=3-3]
	\arrow[Rightarrow, no head, from=1-5, to=3-5]
	\arrow["{\rho_m}"{description}, from=3-5, to=3-7]
	\arrow[Rightarrow, no head, from=1-15, to=3-15]
	\arrow[""{name=1, anchor=center, inner sep=0}, "\sigma"{description}, from=3-7, to=3-13]
	\arrow["{\rho_1'}"{description}, from=3-13, to=3-15]
	\arrow[Rightarrow, no head, from=1-17, to=3-17]
	\arrow["{\rho_{m'}'}"{description}, from=3-17, to=3-19]
	\arrow[Rightarrow, no head, from=1-19, to=3-19]
	\arrow["\psi"{description}, from=3-19, to=5-19]
	\arrow[""{name=2, anchor=center, inner sep=0}, "\tau"{description}, from=5-1, to=5-19]
	\arrow["{r_1}"{description}, from=7-1, to=7-3]
	\arrow["{r_m}"{description}, from=7-5, to=7-7]
	\arrow["{r_1'}"{description}, from=7-13, to=7-15]
	\arrow["{r_{m'}'}"{description}, from=7-17, to=7-19]
	\arrow[Rightarrow, no head, from=7-1, to=9-1]
	\arrow["{r_1}"{description}, from=9-1, to=9-3]
	\arrow[Rightarrow, no head, from=7-3, to=9-3]
	\arrow[Rightarrow, no head, from=7-5, to=9-5]
	\arrow["{r_m}"{description}, from=9-5, to=9-7]
	\arrow["{r_1'}"{description}, from=9-13, to=9-15]
	\arrow[Rightarrow, no head, from=7-15, to=9-15]
	\arrow[Rightarrow, no head, from=7-17, to=9-17]
	\arrow[Rightarrow, no head, from=7-19, to=9-19]
	\arrow["{r_{m'}'}"{description}, from=9-17, to=9-19]
	\arrow[""{name=3, anchor=center, inner sep=0}, "t"{description}, from=11-1, to=11-19]
	\arrow["f"{description}, from=9-1, to=11-1]
	\arrow["g"{description}, from=9-19, to=11-19]
	\arrow["\alpha"{description}, shorten >=7pt, Rightarrow, from=7-10, to=0]
	\arrow["{\mathrm{A}}"{description}, shorten >=7pt, Rightarrow, from=1-10, to=1]
	\arrow["{\Lambda/\mathrm{A}}"{description}, shorten <=9pt, shorten >=9pt, Rightarrow, dashed, from=1, to=2]
	\arrow["\beta"{description}, shorten <=9pt, shorten >=9pt, Rightarrow, from=0, to=3]
\end{tikzcd}
}

We call $\sigma$ the pushforward of $\pi_1,...,\pi_n$ along $\alpha$ and we write $\pushone{\alpha}{\pi_1,\dots,\pi_n}$.
\end{definition}

It makes sense to talk about \emph{the} pushforward since it is unique up to unique vertical invertible cell, where a vertical cell is one that lies over the identity cell.

\begin{prop}
The pushforward of a chain of horizontal morphisms along a cell is unique up to unique vertical invertible cell.
\end{prop}
\begin{proof}
It is similar to the proof of unicity of the composite or of unicity of pushforward in categories.
Given two opcartesian cells with the same domain and lying over the same cell, we can factor both through the other.
Then the fact that the cells given by factorisation are invertible come from the unicity of the factorisation.
\end{proof}

\begin{prop}
Opcartesian cells compose, i.e. if all the $\mathrm{A}_i$ and $\mathrm{B}$ are opcartesian then $\mathrm{B}(\mathrm{A}_1,\dots,\mathrm{A}_n)$ is opcartesian.
\end{prop}
\begin{proof}
Similar to the proof of compositionality of universal cells: we first rearrange the composite so that we only have one cell per vertical layer and then we use the factorisation property.
It is summed up in the following diagram.

\rotatebox{90}{
\resizebox{\vsize}{!}{
\begin{tikzcd}[ampersand replacement=\&]
	\bullet \&\& \bullet \& \dots \& \bullet \&\& \bullet \&\& \bullet \& \dots \& \bullet \&\& \bullet \& \dots \& \bullet \&\& \bullet \& \dots \& \bullet \&\& \bullet \&\& \bullet \& \dots \& \bullet \&\& \bullet \\
	\\
	\bullet \&\& \bullet \& \dots \& \bullet \&\& \bullet \&\&\&\&\&\& \bullet \& \dots \& \bullet \&\& \bullet \& \dots \& \bullet \&\& \bullet \&\& \bullet \& \dots \& \bullet \&\& \bullet \\
	\& \vdots \&\& \vdots \&\& \vdots \&\&\&\& \vdots \&\&\&\& \vdots \&\& \vdots \&\& \vdots \&\& \vdots \&\& \vdots \&\& \vdots \&\& \vdots \\
	\bullet \&\& \bullet \& \dots \& \bullet \&\& \bullet \&\&\&\&\&\& \bullet \& \dots \& \bullet \&\& \bullet \& \dots \& \bullet \&\& \bullet \&\& \bullet \& \dots \& \bullet \&\& \bullet \\
	\\
	\bullet \&\& \bullet \& \dots \& \bullet \&\& \bullet \&\&\&\&\&\& \bullet \& \dots \& \bullet \&\&\&\&\&\& \bullet \&\& \bullet \& \dots \& \bullet \&\& \bullet \\
	\\
	\bullet \&\& \bullet \& \dots \& \bullet \&\& \bullet \&\&\&\&\&\&\&\&\&\&\&\&\&\& \bullet \&\& \bullet \& \dots \& \bullet \&\& \bullet \\
	\\
	\bullet \&\&\&\&\&\&\&\&\&\&\&\&\&\&\&\&\&\&\&\&\&\&\&\&\&\& \bullet \\
	\\
	\bullet \&\& \bullet \& \dots \& \bullet \&\& \bullet \&\& \bullet \& \dots \& \bullet \&\& \bullet \& \dots \& \bullet \&\& \bullet \& \dots \& \bullet \&\& \bullet \&\& \bullet \& \dots \& \bullet \&\& \bullet \\
	\\
	\bullet \&\& \bullet \& \dots \& \bullet \&\& \bullet \&\&\&\&\&\& \bullet \& \dots \& \bullet \&\& \bullet \& \dots \& \bullet \&\& \bullet \&\& \bullet \& \dots \& \bullet \&\& \bullet \\
	\& \vdots \&\& \vdots \&\& \vdots \&\&\&\& \vdots \&\&\&\& \vdots \&\& \vdots \&\& \vdots \&\& \vdots \&\& \vdots \&\& \vdots \&\& \vdots \\
	\bullet \&\& \bullet \& \dots \& \bullet \&\& \bullet \&\&\&\&\&\& \bullet \& \dots \& \bullet \&\& \bullet \& \dots \& \bullet \&\& \bullet \&\& \bullet \& \dots \& \bullet \&\& \bullet \\
	\\
	\bullet \&\& \bullet \& \dots \& \bullet \&\& \bullet \&\&\&\&\&\& \bullet \& \dots \& \bullet \&\&\&\&\&\& \bullet \&\& \bullet \& \dots \& \bullet \&\& \bullet \\
	\\
	\bullet \&\& \bullet \& \dots \& \bullet \&\& \bullet \&\&\&\&\&\&\&\&\&\&\&\&\&\& \bullet \&\& \bullet \& \dots \& \bullet \&\& \bullet \\
	\\
	\bullet \&\&\&\&\&\&\&\&\&\&\&\&\&\&\&\&\&\&\&\&\&\&\&\&\&\& \bullet
	\arrow["{\rho_1}"{description}, from=1-1, to=1-3]
	\arrow["{\pi_{1,1}}"{description}, color={rgb,255:red,54;green,130;blue,252}, from=1-7, to=1-9]
	\arrow["{\pi_{1,m_1}}"{description}, color={rgb,255:red,54;green,130;blue,252}, from=1-11, to=1-13]
	\arrow["{\pi_{n,1}}"{description}, from=1-15, to=1-17]
	\arrow["{\pi_{n,m_n}}"{description}, from=1-19, to=1-21]
	\arrow["{\rho_{k'}'}"{description}, from=1-25, to=1-27]
	\arrow[from=3-1, to=3-3]
	\arrow[""{name=0, anchor=center, inner sep=0}, "{\pushone{\alpha_1}{\pi_{1,1},\dots\pi_{1,m_1}}}"{description}, color={rgb,255:red,54;green,130;blue,252}, from=3-7, to=3-13]
	\arrow[from=3-15, to=3-17]
	\arrow[from=3-19, to=3-21]
	\arrow["{\rho_1'}"{description}, from=1-21, to=1-23]
	\arrow[from=3-21, to=3-23]
	\arrow["{\rho_k}"{description}, from=1-5, to=1-7]
	\arrow[from=3-5, to=3-7]
	\arrow[from=3-25, to=3-27]
	\arrow[Rightarrow, no head, from=1-1, to=3-1]
	\arrow[Rightarrow, no head, from=1-3, to=3-3]
	\arrow[Rightarrow, no head, from=1-5, to=3-5]
	\arrow[color={rgb,255:red,54;green,130;blue,252}, Rightarrow, no head, from=1-7, to=3-7]
	\arrow[color={rgb,255:red,54;green,130;blue,252}, Rightarrow, no head, from=1-13, to=3-13]
	\arrow[Rightarrow, no head, from=1-15, to=3-15]
	\arrow[Rightarrow, no head, from=1-17, to=3-17]
	\arrow[Rightarrow, no head, from=1-19, to=3-19]
	\arrow[Rightarrow, no head, from=1-21, to=3-21]
	\arrow[Rightarrow, no head, from=1-23, to=3-23]
	\arrow[Rightarrow, no head, from=1-25, to=3-25]
	\arrow[Rightarrow, no head, from=1-27, to=3-27]
	\arrow[Rightarrow, no head, from=3-1, to=5-1]
	\arrow[Rightarrow, no head, from=3-3, to=5-3]
	\arrow[from=5-1, to=5-3]
	\arrow[Rightarrow, no head, from=3-5, to=5-5]
	\arrow[Rightarrow, no head, from=3-7, to=5-7]
	\arrow[from=5-5, to=5-7]
	\arrow[color={rgb,255:red,54;green,130;blue,252}, Rightarrow, no head, from=3-13, to=5-13]
	\arrow[from=5-7, to=5-13]
	\arrow[color={rgb,255:red,54;green,130;blue,252}, Rightarrow, no head, from=3-15, to=5-15]
	\arrow[Rightarrow, no head, from=3-17, to=5-17]
	\arrow[color={rgb,255:red,54;green,130;blue,252}, from=5-15, to=5-17]
	\arrow[Rightarrow, no head, from=3-19, to=5-19]
	\arrow[color={rgb,255:red,54;green,130;blue,252}, from=5-19, to=5-21]
	\arrow[from=5-21, to=5-23]
	\arrow[Rightarrow, no head, from=3-21, to=5-21]
	\arrow[Rightarrow, no head, from=3-23, to=5-23]
	\arrow[Rightarrow, no head, from=3-25, to=5-25]
	\arrow[Rightarrow, no head, from=3-27, to=5-27]
	\arrow[from=5-25, to=5-27]
	\arrow[Rightarrow, no head, from=5-1, to=7-1]
	\arrow[Rightarrow, no head, from=5-3, to=7-3]
	\arrow[from=7-1, to=7-3]
	\arrow[Rightarrow, no head, from=5-5, to=7-5]
	\arrow[Rightarrow, no head, from=5-7, to=7-7]
	\arrow[from=7-5, to=7-7]
	\arrow[color={rgb,255:red,54;green,130;blue,252}, from=7-7, to=7-13]
	\arrow[Rightarrow, no head, from=5-13, to=7-13]
	\arrow[color={rgb,255:red,54;green,130;blue,252}, Rightarrow, no head, from=5-15, to=7-15]
	\arrow[""{name=1, anchor=center, inner sep=0}, "{\pushone{\alpha_n}{\pi_{n,1},\dots,\pi_{n,m_n}}}"{description}, color={rgb,255:red,54;green,130;blue,252}, from=7-15, to=7-21]
	\arrow[color={rgb,255:red,54;green,130;blue,252}, Rightarrow, no head, from=5-21, to=7-21]
	\arrow[Rightarrow, no head, from=5-23, to=7-23]
	\arrow[from=7-21, to=7-23]
	\arrow[Rightarrow, no head, from=5-25, to=7-25]
	\arrow[Rightarrow, no head, from=5-27, to=7-27]
	\arrow[from=7-25, to=7-27]
	\arrow[from=9-1, to=9-3]
	\arrow[from=9-5, to=9-7]
	\arrow[""{name=2, anchor=center, inner sep=0}, "{\pushone{\beta}{\pushone{\alpha_1}{\pi_{1,1},\dots,\pi_{1,m_1}},\dots,\pushone{\alpha_n}{\pi_{n,1},\dots,\pi_{n,m_n}}}}"{description}, color={rgb,255:red,54;green,130;blue,252}, from=9-7, to=9-21]
	\arrow[from=9-21, to=9-23]
	\arrow[from=9-25, to=9-27]
	\arrow[Rightarrow, no head, from=7-1, to=9-1]
	\arrow[Rightarrow, no head, from=7-3, to=9-3]
	\arrow[Rightarrow, no head, from=7-5, to=9-5]
	\arrow[color={rgb,255:red,54;green,130;blue,252}, Rightarrow, no head, from=7-7, to=9-7]
	\arrow[color={rgb,255:red,54;green,130;blue,252}, Rightarrow, no head, from=7-21, to=9-21]
	\arrow[Rightarrow, no head, from=7-23, to=9-23]
	\arrow[Rightarrow, no head, from=7-25, to=9-25]
	\arrow[Rightarrow, no head, from=7-27, to=9-27]
	\arrow[""{name=3, anchor=center, inner sep=0}, "\xi"{description}, from=11-1, to=11-27]
	\arrow["\varphi"{description}, from=9-1, to=11-1]
	\arrow["\psi"{description}, from=9-27, to=11-27]
	\arrow["{r_1}"{description}, from=13-1, to=13-3]
	\arrow["{r_k}"{description}, from=13-5, to=13-7]
	\arrow["{p_{1,1}}"{description}, from=13-7, to=13-9]
	\arrow["{p_{1,m_1}}"{description}, from=13-11, to=13-13]
	\arrow["{p_{n,1}}"{description}, from=13-15, to=13-17]
	\arrow["{p_{n,m_n}}"{description}, from=13-19, to=13-21]
	\arrow["{r_1'}"{description}, from=13-21, to=13-23]
	\arrow["{r_{k'}'}"{description}, from=13-25, to=13-27]
	\arrow[from=15-1, to=15-3]
	\arrow[from=15-5, to=15-7]
	\arrow[""{name=4, anchor=center, inner sep=0}, "{s_1}"{description}, from=15-7, to=15-13]
	\arrow[from=15-15, to=15-17]
	\arrow[from=15-19, to=15-21]
	\arrow[from=15-21, to=15-23]
	\arrow[from=15-25, to=15-27]
	\arrow[Rightarrow, no head, from=13-1, to=15-1]
	\arrow[Rightarrow, no head, from=13-3, to=15-3]
	\arrow[Rightarrow, no head, from=13-5, to=15-5]
	\arrow[Rightarrow, no head, from=13-13, to=15-13]
	\arrow[Rightarrow, no head, from=13-7, to=15-7]
	\arrow[Rightarrow, no head, from=13-15, to=15-15]
	\arrow[Rightarrow, no head, from=13-17, to=15-17]
	\arrow[Rightarrow, no head, from=13-21, to=15-21]
	\arrow[Rightarrow, no head, from=13-23, to=15-23]
	\arrow[Rightarrow, no head, from=13-25, to=15-25]
	\arrow[Rightarrow, no head, from=13-27, to=15-27]
	\arrow[Rightarrow, no head, from=15-1, to=17-1]
	\arrow[from=17-1, to=17-3]
	\arrow[Rightarrow, no head, from=15-3, to=17-3]
	\arrow[Rightarrow, no head, from=15-5, to=17-5]
	\arrow[Rightarrow, no head, from=15-7, to=17-7]
	\arrow[from=17-5, to=17-7]
	\arrow[from=17-7, to=17-13]
	\arrow[Rightarrow, no head, from=15-13, to=17-13]
	\arrow[Rightarrow, no head, from=15-15, to=17-15]
	\arrow[from=17-15, to=17-17]
	\arrow[Rightarrow, no head, from=15-17, to=17-17]
	\arrow[Rightarrow, no head, from=13-19, to=15-19]
	\arrow[Rightarrow, no head, from=15-19, to=17-19]
	\arrow[Rightarrow, no head, from=15-21, to=17-21]
	\arrow[from=17-19, to=17-21]
	\arrow[Rightarrow, from=15-23, to=17-23]
	\arrow[from=17-21, to=17-23]
	\arrow[Rightarrow, no head, from=15-25, to=17-25]
	\arrow[Rightarrow, no head, from=15-27, to=17-27]
	\arrow[from=17-25, to=17-27]
	\arrow[from=19-1, to=19-3]
	\arrow[from=19-5, to=19-7]
	\arrow[from=19-7, to=19-13]
	\arrow[""{name=5, anchor=center, inner sep=0}, "{s_n}"{description}, from=19-15, to=19-21]
	\arrow[from=19-21, to=19-23]
	\arrow[from=19-25, to=19-27]
	\arrow[Rightarrow, no head, from=17-1, to=19-1]
	\arrow[Rightarrow, no head, from=17-3, to=19-3]
	\arrow[Rightarrow, no head, from=17-5, to=19-5]
	\arrow[Rightarrow, no head, from=17-7, to=19-7]
	\arrow[Rightarrow, no head, from=17-13, to=19-13]
	\arrow[Rightarrow, no head, from=17-15, to=19-15]
	\arrow[Rightarrow, no head, from=17-21, to=19-21]
	\arrow[Rightarrow, no head, from=17-23, to=19-23]
	\arrow[Rightarrow, no head, from=17-25, to=19-25]
	\arrow[Rightarrow, no head, from=17-27, to=19-27]
	\arrow[from=21-1, to=21-3]
	\arrow[from=21-5, to=21-7]
	\arrow[from=21-25, to=21-27]
	\arrow[""{name=6, anchor=center, inner sep=0}, "t"{description}, from=21-7, to=21-21]
	\arrow[from=21-21, to=21-23]
	\arrow[Rightarrow, no head, from=19-1, to=21-1]
	\arrow[Rightarrow, no head, from=19-3, to=21-3]
	\arrow[Rightarrow, no head, from=19-5, to=21-5]
	\arrow[Rightarrow, no head, from=19-7, to=21-7]
	\arrow[Rightarrow, no head, from=19-21, to=21-21]
	\arrow[Rightarrow, no head, from=19-23, to=21-23]
	\arrow[Rightarrow, no head, from=19-25, to=21-25]
	\arrow[Rightarrow, no head, from=19-27, to=21-27]
	\arrow["f"{description}, from=21-1, to=23-1]
	\arrow[""{name=7, anchor=center, inner sep=0}, "x"{description}, from=23-1, to=23-27]
	\arrow["g"{description}, from=21-27, to=23-27]
	\arrow["{\mathrm{A}_1}"{description}, color={rgb,255:red,54;green,130;blue,252}, shorten >=7pt, Rightarrow, from=1-10, to=0]
	\arrow["{\mathrm{A}_n}"{description}, color={rgb,255:red,54;green,130;blue,252}, shorten >=7pt, Rightarrow, from=5-18, to=1]
	\arrow["{\mathrm{B}}"{description}, color={rgb,255:red,54;green,130;blue,252}, shorten >=7pt, Rightarrow, from=7-14, to=2]
	\arrow["{(((\Lambda/\mathrm{A}_1)/\dots)/\mathrm{A}_n)/\mathrm{B}}"{description}, shorten <=9pt, shorten >=9pt, Rightarrow, from=2, to=3]
	\arrow["{\alpha_1}"{description}, shorten >=7pt, Rightarrow, from=13-10, to=4]
	\arrow["{\alpha_n}"{description}, shorten >=7pt, Rightarrow, from=17-18, to=5]
	\arrow["\beta"{description}, shorten >=7pt, Rightarrow, from=19-14, to=6]
	\arrow["\gamma"{description}, shorten <=9pt, shorten >=9pt, Rightarrow, from=6, to=7]
\end{tikzcd}
}
}
\end{proof}

\begin{definition}
A pushfibration of vdcs is a functor of vdcs $p \colon \vdC \to \vdD$ such that for any chain of horizontal morphisms $\pi_i \colon R_{i-1} \to R_i$ and any cell

% https://q.uiver.app/?q=WzAsNyxbMCwwLCJwKFJfMCkiXSxbMiwwLCJwKFJfMSkiXSxbMywwLCJcXGRvdHMiXSxbNCwwLCJwKFJfe24tMX0pIl0sWzYsMCwicChSX24pIl0sWzAsMiwicChSXzApIl0sWzYsMiwicChSX24pIl0sWzAsMSwicChcXHBpXzEpIiwxXSxbMyw0LCJwKFxccGlfbikiLDFdLFs1LDYsInMiLDFdLFswLDUsIiIsMSx7ImxldmVsIjoyLCJzdHlsZSI6eyJoZWFkIjp7Im5hbWUiOiJub25lIn19fV0sWzQsNiwiIiwxLHsibGV2ZWwiOjIsInN0eWxlIjp7ImhlYWQiOnsibmFtZSI6Im5vbmUifX19XSxbMiw5LCJcXGFscGhhIiwxLHsic2hvcnRlbiI6eyJ0YXJnZXQiOjIwfX1dXQ==
\begin{tikzcd}[ampersand replacement=\&]
	{p(R_0)} \&\& {p(R_1)} \& \dots \& {p(R_{n-1})} \&\& {p(R_n)} \\
	\\
	{p(R_0)} \&\&\&\&\&\& {p(R_n)}
	\arrow["{p(\pi_1)}"{description}, from=1-1, to=1-3]
	\arrow["{p(\pi_n)}"{description}, from=1-5, to=1-7]
	\arrow[""{name=0, anchor=center, inner sep=0}, "s"{description}, from=3-1, to=3-7]
	\arrow[Rightarrow, no head, from=1-1, to=3-1]
	\arrow[Rightarrow, no head, from=1-7, to=3-7]
	\arrow["\alpha"{description}, shorten >=7pt, Rightarrow, from=1-4, to=0]
\end{tikzcd}

there is a horizontal morphism $\pushone{\alpha}{\pi_1,\dots,\pi_n}$ with an opcartesian cell lying over $\alpha$:

% https://q.uiver.app/?q=WzAsNyxbMCwwLCJSXzAiXSxbMiwwLCJSXzEiXSxbMywwLCJcXGRvdHMiXSxbNCwwLCJSX3tuLTF9Il0sWzYsMCwiUl9uIl0sWzAsMiwiUl8wIl0sWzYsMiwiUl9uIl0sWzAsMSwiXFxwaV8xIiwxXSxbMyw0LCJcXHBpX24iLDFdLFs1LDYsIlxccHVzaG9uZXtcXGFscGhhfXtcXHBpXzEsXFxkb3RzLFxccGlfbn0iLDFdLFswLDUsIiIsMSx7ImxldmVsIjoyLCJzdHlsZSI6eyJoZWFkIjp7Im5hbWUiOiJub25lIn19fV0sWzQsNiwiIiwxLHsibGV2ZWwiOjIsInN0eWxlIjp7ImhlYWQiOnsibmFtZSI6Im5vbmUifX19XSxbMiw5LCIiLDEseyJzaG9ydGVuIjp7InRhcmdldCI6MjB9fV1d
\begin{tikzcd}[ampersand replacement=\&]
	{R_0} \&\& {R_1} \& \dots \& {R_{n-1}} \&\& {R_n} \\
	\\
	{R_0} \&\&\&\&\&\& {R_n}
	\arrow["{\pi_1}"{description}, from=1-1, to=1-3]
	\arrow["{\pi_n}"{description}, from=1-5, to=1-7]
	\arrow[""{name=0, anchor=center, inner sep=0}, "{\pushone{\alpha}{\pi_1,\dots,\pi_n}}"{description}, from=3-1, to=3-7]
	\arrow[Rightarrow, no head, from=1-1, to=3-1]
	\arrow[Rightarrow, no head, from=1-7, to=3-7]
	\arrow[shorten >=7pt, Rightarrow, from=1-4, to=0]
\end{tikzcd}
\end{definition}

\begin{example}
A functor between the delooping of multicategories is a pushfibration of vdcs iff the corresponding functor between multicategories is a pushfibration of multicategories.
\end{example}

\begin{example}
We will see later that there is a vdc $\Dist_\ast$ of pointed distributors and that the forgetful functor $\Dist_\ast \to \Dist$ is a pushfibration of vdcs.
\end{example}

\begin{example}
Given a vdc \vdC and \vdX consisting of some of the objects, vertical morphisms, horizontal morphisms and cells of \vdC, not  necessarily forming a vdc.
We define the vdc $\vdC^{\vdX}$ with:
\begin{itemize}
\item objects are vertical morphisms $\varphi \colon R \to A$ in \vdX
\item vertical morphisms $f \colon (\varphi \colon R \to A) \to (\psi \colon S \to B)$ are vertical morphisms $f \colon A \to B$ in \vdC such that $\psi = f \circ \varphi$
\item horizontal morphisms $p \colon (\varphi_0 \colon R_0 \to A_0) \to (\varphi_1 \colon R_1 \to A_1)$ are cells in \vdX of the shape:
% https://q.uiver.app/?q=WzAsNyxbMCwwLCJSXzAiXSxbMiwwLCJYXzEiXSxbMywwLCJcXGRvdHMiXSxbNCwwLCJYX3tuLTF9Il0sWzYsMCwiUl8xIl0sWzAsMiwiQV8wIl0sWzYsMiwiQV8xIl0sWzAsMSwiXFxwaV8xIiwxXSxbMyw0LCJcXHBpX24iLDFdLFswLDUsIlxcdmFycGhpXzAiLDFdLFs1LDYsInBfMSIsMV0sWzQsNiwiXFx2YXJwaGlfMSIsMV0sWzIsMTAsIlxcUGlfMSIsMSx7InNob3J0ZW4iOnsidGFyZ2V0IjoyMH19XV0=
\[\begin{tikzcd}[ampersand replacement=\&]
	{R_0} \&\& {X_1} \& \dots \& {X_{n-1}} \&\& {R_1} \\
	\\
	{A_0} \&\&\&\&\&\& {A_1}
	\arrow["{\pi_1}"{description}, from=1-1, to=1-3]
	\arrow["{\pi_n}"{description}, from=1-5, to=1-7]
	\arrow["{\varphi_0}"{description}, from=1-1, to=3-1]
	\arrow[""{name=0, anchor=center, inner sep=0}, "{p_1}"{description}, from=3-1, to=3-7]
	\arrow["{\varphi_1}"{description}, from=1-7, to=3-7]
	\arrow["{\Pi_1}"{description}, shorten >=7pt, Rightarrow, from=1-4, to=0]
\end{tikzcd}\]
\item cells 
% https://q.uiver.app/?q=WzAsNyxbMCwwLCJcXHZhcnBoaV8wIl0sWzIsMCwiXFx2YXJwaGlfMSJdLFszLDAsIlxcZG90cyJdLFs0LDAsIlxcdmFycGhpX3tuLTF9Il0sWzYsMCwiXFx2YXJwaGlfbiJdLFswLDIsIlxccHNpXzAiXSxbNiwyLCJcXHBzaV8xIl0sWzAsMSwiXFxQaV8xIiwxXSxbMyw0LCJcXFBpX24iLDFdLFswLDUsImZfMCIsMV0sWzUsNiwiXFxTaWdtYSIsMV0sWzQsNiwiZl8xIiwxXSxbMiwxMCwiXFxhbHBoYSIsMSx7InNob3J0ZW4iOnsidGFyZ2V0IjoyMH19XV0=
\[\begin{tikzcd}[ampersand replacement=\&]
	{\varphi_0} \&\& {\varphi_1} \& \dots \& {\varphi_{n-1}} \&\& {\varphi_n} \\
	\\
	{\psi_0} \&\&\&\&\&\& {\psi_1}
	\arrow["{\Pi_1}"{description}, from=1-1, to=1-3]
	\arrow["{\Pi_n}"{description}, from=1-5, to=1-7]
	\arrow["{f_0}"{description}, from=1-1, to=3-1]
	\arrow[""{name=0, anchor=center, inner sep=0}, "\Sigma"{description}, from=3-1, to=3-7]
	\arrow["{f_1}"{description}, from=1-7, to=3-7]
	\arrow["\alpha"{description}, shorten >=7pt, Rightarrow, from=1-4, to=0]
\end{tikzcd}\]
are cells $\alpha$ in \vdC
% https://q.uiver.app/?q=WzAsMTcsWzAsMCwiUl8wIl0sWzEsMCwiXFxkb3RzIl0sWzIsMCwiUl8xIl0sWzAsMiwiQV8wIl0sWzIsMiwiQV8xIl0sWzMsMCwiXFxkb3RzIl0sWzMsMiwiXFxkb3RzIl0sWzQsMCwiUl97bi0xfSJdLFs1LDAsIlxcZG90cyJdLFs2LDAsIlJfbiJdLFs0LDIsIkFfe24tMX0iXSxbNiwyLCJBX24iXSxbMCw0LCJCXzAiXSxbNiw0LCJCXzEiXSxbMCw2LCJTXzAiXSxbNiw2LCJTXzEiXSxbMyw2LCJcXGRvdHMiXSxbMCwzLCJcXHZhcnBoaV8wIiwxXSxbMyw0LCJwXzEiLDFdLFsyLDQsIlxcdmFycGhpXzEiLDFdLFswLDFdLFsxLDJdLFs3LDhdLFs4LDldLFs3LDEwLCJcXHZhcnBoaV97bi0xfSIsMV0sWzksMTEsIlxcdmFycGhpX24iLDFdLFsxMCwxMSwicF9uIiwxXSxbMywxMiwiZl8wIiwxXSxbMTIsMTMsInEiLDFdLFsxMSwxMywiZl8xIiwxXSxbMTQsMTIsIlxccHNpXzAiLDFdLFsxNSwxMywiXFxwc2lfMSIsMV0sWzE0LDE2XSxbMTYsMTVdLFsxLDE4LCJcXFBpXzEiLDEseyJzaG9ydGVuIjp7InRhcmdldCI6MjB9fV0sWzgsMjYsIlxcUGlfbiIsMSx7InNob3J0ZW4iOnsidGFyZ2V0IjoyMH19XSxbNiwyOCwiXFxhbHBoYSIsMSx7InNob3J0ZW4iOnsidGFyZ2V0IjoyMH19XSxbMTYsMjgsIlxcU2lnbWEiLDEseyJzaG9ydGVuIjp7InRhcmdldCI6MjB9fV1d
\[\begin{tikzcd}[ampersand replacement=\&]
	{R_0} \& \dots \& {R_1} \& \dots \& {R_{n-1}} \& \dots \& {R_n} \\
	\\
	{A_0} \&\& {A_1} \& \dots \& {A_{n-1}} \&\& {A_n} \\
	\\
	{B_0} \&\&\&\&\&\& {B_1} \\
	\\
	{S_0} \&\&\& \dots \&\&\& {S_1}
	\arrow["{\varphi_0}"{description}, from=1-1, to=3-1]
	\arrow[""{name=0, anchor=center, inner sep=0}, "{p_1}"{description}, from=3-1, to=3-3]
	\arrow["{\varphi_1}"{description}, from=1-3, to=3-3]
	\arrow[from=1-1, to=1-2]
	\arrow[from=1-2, to=1-3]
	\arrow[from=1-5, to=1-6]
	\arrow[from=1-6, to=1-7]
	\arrow["{\varphi_{n-1}}"{description}, from=1-5, to=3-5]
	\arrow["{\varphi_n}"{description}, from=1-7, to=3-7]
	\arrow[""{name=1, anchor=center, inner sep=0}, "{p_n}"{description}, from=3-5, to=3-7]
	\arrow["{f_0}"{description}, from=3-1, to=5-1]
	\arrow[""{name=2, anchor=center, inner sep=0}, "q"{description}, from=5-1, to=5-7]
	\arrow["{f_1}"{description}, from=3-7, to=5-7]
	\arrow["{\psi_0}"{description}, from=7-1, to=5-1]
	\arrow["{\psi_1}"{description}, from=7-7, to=5-7]
	\arrow[from=7-1, to=7-4]
	\arrow[from=7-4, to=7-7]
	\arrow["{\Pi_1}"{description}, shorten >=7pt, Rightarrow, from=1-2, to=0]
	\arrow["{\Pi_n}"{description}, shorten >=7pt, Rightarrow, from=1-6, to=1]
	\arrow["\alpha"{description}, shorten >=7pt, Rightarrow, from=3-4, to=2]
	\arrow["\Sigma"{description}, shorten >=7pt, Rightarrow, from=7-4, to=2]
\end{tikzcd}\]
s.t. $\Sigma = \alpha(\Pi_1,\dots,P_n)$, in particular the domain should also agree
\end{itemize}
Since the vertical morphisms and the cells are those of \vdC satisfying some preservation property, it suffices to check that this property is closed under identity and composition.
It is the case and $\vdC^{\vdX}$ is a vdc.
There is a functor of vdcs $\cU \colon \vdC^{\vdX} \to \vdC$ that takes the objects and horizontal morphisms of $\vdC^{\vdX}$ to their codomain and is the identity on the vertical morphisms and cells.
Given some $\Pi_1,\dots, \Pi_n$ horizontal morphisms in $\vdC^{\vdX}$ and a cell $\alpha$ in \vdC whose vertical parts are identities, $\pushone{\alpha}{\Pi_1,\dots,\Pi_n}$ exists iff $\alpha(\Pi_1,\dots,\Pi_n) \in \vdX$ in which case it is the pushforward.
In particular, $\cU \colon \vdC^{\vdX} \to \vdC$ is a pushfibration iff the horizontal morphisms of \vdX are closed under postcomposition by arbitrary cells
% https://q.uiver.app/?q=WzAsNyxbMCwwLCJBXzAiXSxbMiwwLCJBXzEiXSxbMywwLCJcXGRvdHMiXSxbNCwwLCJBX3tuLTF9Il0sWzYsMCwiQV9uIl0sWzAsMiwiQV8wIl0sWzYsMiwiQV9uIl0sWzAsMSwicF8xIiwxXSxbMyw0LCJwX24iLDFdLFs1LDYsInEiLDFdLFs0LDYsIiIsMSx7ImxldmVsIjoyLCJzdHlsZSI6eyJoZWFkIjp7Im5hbWUiOiJub25lIn19fV0sWzAsNSwiIiwxLHsibGV2ZWwiOjIsInN0eWxlIjp7ImhlYWQiOnsibmFtZSI6Im5vbmUifX19XSxbMiw5LCJcXGFscGhhIiwxLHsic2hvcnRlbiI6eyJ0YXJnZXQiOjIwfX1dXQ==
\[\begin{tikzcd}[ampersand replacement=\&]
	{A_0} \&\& {A_1} \& \dots \& {A_{n-1}} \&\& {A_n} \\
	\\
	{A_0} \&\&\&\&\&\& {A_n}
	\arrow["{p_1}"{description}, from=1-1, to=1-3]
	\arrow["{p_n}"{description}, from=1-5, to=1-7]
	\arrow[""{name=0, anchor=center, inner sep=0}, "q"{description}, from=3-1, to=3-7]
	\arrow[Rightarrow, no head, from=1-7, to=3-7]
	\arrow[Rightarrow, no head, from=1-1, to=3-1]
	\arrow["\alpha"{description}, shorten >=7pt, Rightarrow, from=1-4, to=0]
\end{tikzcd}\]
\end{example}

The unique factorisation property of a universal cell and an opcartesian one are alike.
The difference resides in the fact that the one for opcartesian cells depends on the existence of the factorisation in the base vdc (i.e. the codomain of the pushfibration).
In turn, if the factorisation in the base is trivial, then we should recover the same notions.

Let $! \colon \vdC \to \one$ be the unique functor from a vdc \vdC to the terminal one.

\begin{prop}
A cell $\alpha$ in \vdC is universal iff it is !-opcartesian.
\end{prop}
\begin{proof}
Consider a cell $\lambda$ in \vdC

\resizebox{\hsize}{!}{
% https://q.uiver.app/?q=WzAsMTUsWzYsMCwiQV8wIl0sWzgsMCwiQV8xIl0sWzksMCwiXFxkb3RzIl0sWzEwLDAsIkFfe24tMX0iXSxbMTIsMCwiQV9uIl0sWzAsMiwiQl8wIl0sWzE4LDIsIkJfMSJdLFs0LDAsIkNfe20tMX0iXSxbMywwLCJcXGRvdHMiXSxbMCwwLCJDXzAiXSxbMiwwLCJDXzEiXSxbMTQsMCwiRF8xIl0sWzE1LDAsIlxcZG90cyJdLFsxNiwwLCJEX3trLTF9Il0sWzE4LDAsIkRfayJdLFswLDEsInBfMSIsMV0sWzMsNCwicF9uIiwxXSxbNSw2LCJxIiwxXSxbNywwLCJyX20iLDFdLFs5LDEwLCJyXzAiLDFdLFs0LDExLCJzXzEiLDFdLFsxMywxNCwic19rIiwxXSxbOSw1LCJmIiwxXSxbMTQsNiwiZyIsMV0sWzIsMTcsIlxcbGFtYmRhIiwxLHsic2hvcnRlbiI6eyJ0YXJnZXQiOjIwfX1dXQ==
\begin{tikzcd}[ampersand replacement=\&]
	{C_0} \&\& {C_1} \& \dots \& {C_{m-1}} \&\& {A_0} \&\& {A_1} \& \dots \& {A_{n-1}} \&\& {A_n} \&\& {D_1} \& \dots \& {D_{k-1}} \&\& {D_k} \\
	\\
	{B_0} \&\&\&\&\&\&\&\&\&\&\&\&\&\&\&\&\&\& {B_1}
	\arrow["{p_1}"{description}, from=1-7, to=1-9]
	\arrow["{p_n}"{description}, from=1-11, to=1-13]
	\arrow[""{name=0, anchor=center, inner sep=0}, "t"{description}, from=3-1, to=3-19]
	\arrow["{r_m}"{description}, from=1-5, to=1-7]
	\arrow["{r_0}"{description}, from=1-1, to=1-3]
	\arrow["{s_1}"{description}, from=1-13, to=1-15]
	\arrow["{s_k}"{description}, from=1-17, to=1-19]
	\arrow["f"{description}, from=1-1, to=3-1]
	\arrow["g"{description}, from=1-19, to=3-19]
	\arrow["\lambda"{description}, shorten >=7pt, Rightarrow, from=1-10, to=0]
\end{tikzcd}
}

We necessarily have that $!(\alpha) = \underline{n}$ and $!(\lambda) = \underline{m+n+k}$ since there is only one cell of each arity.
But for the same reason, $\underline{m+n+k} = \underline{m+1+k}(\id,\dots,\id,\underline{n},\id,\dots,\id)$
So $\lambda$ lies over a factorisation by the image of $\alpha$.

Now suppose that $\alpha$ is opcartesian, then we can uniquely factor $\lambda$ through it, so $\alpha$ is universal.
Conversely, if $\alpha$ is universal, we can uniquely factor $\lambda$ through it, and the factorisation necessarily lies other $\underline{m+1+k}$, so $\alpha$ is opcartesian.
\end{proof}

\begin{cor}
A vdc is representable iff it is pushfibred over $\one$.
\end{cor}

Now let us consider two functors of vdcs $p \colon \vdP \to \vdE$ and $q \colon \vdE \to \vdB$.

\begin{prop}
If a cell $\alpha$ in \vdP is $p$-opcartesian and $p(\alpha)$ is $q$-opcartesian, then $\alpha$ is $(q\circ p)$-opcartesian. 
\end{prop}
\begin{proof}
Let's consider a cell $\lambda$ in \vdP such that its image $(q\circ p)(\lambda) = q(p(\lambda))$ factors through $(q \circ p)(\alpha) = q(p(\alpha))$.
Then since $p(\alpha)$ is $q$-opcartesian we can factor $p(\lambda)$ through it.
But since $p(\lambda)$ factors through $p(\alpha)$ and $\alpha$ is $p$-opcartesian we can factor $\lambda$ through $\alpha$.
\end{proof}

So pushfibrations compose.

\begin{cor}
If $p$ and $q$ are pushfibrations then $q\circ p$ is a pushfibration.
\end{cor}
\begin{proof}
Given a chain of horizontal morphisms $p_1,...,p_n$ in \vdP and a cell $\alpha$ in \vdB, since $q$ is a pushfibration there is a $q$-opcartesian cell from $p(p_1),\dots,p(p_n)$ to $\pushone{\alpha}{p(p_1),\dots,p(p_n))}$ which lies over $\alpha$.
But then we can push along this cell to get a $p$-opcartesian cell from $p_1,\dots,p_n$ to $\pushone{\pushone{\alpha}{p(p_1),\dots,p(p_n)}}{p_1,\dots,p_n}$.
But then since it is $p$-opcartesian over a $q$-opcartesian cell, it is $(q\circ p)$-opcartesian.
\end{proof}

In particular, a pushfibration over a representable vdc is representable.

\begin{cor}
If $p \colon \vdE \to \vdB$ is a pushfibration and \vdB is representable, then \vdE is representable.
\end{cor}
\begin{proof}
Since \vdB is representable $!_{\vdB} \colon \vdB \to \one$ is a pushfibration.
But then since $p$ is a pushfibration we have that $!_{\vdB}\circ p = !_{\vdE} \colon \vdE \to \one$ is also.
So $\vdE$ is representable.
\end{proof}

The identities and composition in $\vdE$ are given by pushing along their universal cells in $\vdB$.
In fact, one does not need for $p$ to be a pushfibration, but only that pushing along universal cells exist.

\begin{prop}
Consider a functor $\vdE \to \vdB$ such that $\vdB$ is representable and all opcartesian liftings of universal cells exists.
Then, it is a pushfibration iff for any $\pi$ in \vdE lying over $p$ and any unary cell $\alpha$ out of $p$, there is a cell $\mathrm{A}$ out $\pi$ lying over it:

% https://q.uiver.app/?q=WzAsOCxbMCw0LCJBXzAiXSxbMiw0LCJBXzEiXSxbMCw2LCJBXzAiXSxbMiw2LCJBXzEiXSxbMCwwLCJSXzAiXSxbMiwwLCJSXzEiXSxbMCwyLCJSXzAiXSxbMiwyLCJSXzEiXSxbMCwxLCJwIiwxXSxbMiwzLCJzIiwxXSxbMCwyLCIiLDEseyJsZXZlbCI6Miwic3R5bGUiOnsiaGVhZCI6eyJuYW1lIjoibm9uZSJ9fX1dLFsxLDMsIiIsMSx7ImxldmVsIjoyLCJzdHlsZSI6eyJoZWFkIjp7Im5hbWUiOiJub25lIn19fV0sWzQsNSwiXFxwaSIsMV0sWzQsNiwiIiwxLHsibGV2ZWwiOjIsInN0eWxlIjp7ImhlYWQiOnsibmFtZSI6Im5vbmUifX19XSxbNSw3LCIiLDEseyJsZXZlbCI6Miwic3R5bGUiOnsiaGVhZCI6eyJuYW1lIjoibm9uZSJ9fX1dLFs2LDcsIlxccHVzaG9uZXtcXGFscGhhfXtcXHBpfSIsMV0sWzgsOSwiXFxhbHBoYSIsMSx7InNob3J0ZW4iOnsic291cmNlIjoyMCwidGFyZ2V0IjoyMH19XSxbMTIsMTUsIlxcbWF0aHJte0F9IiwxLHsic2hvcnRlbiI6eyJzb3VyY2UiOjIwLCJ0YXJnZXQiOjIwfX1dXQ==
\begin{tikzcd}[ampersand replacement=\&]
	{R_0} \&\& {R_1} \\
	\\
	{R_0} \&\& {R_1} \\
	\\
	{A_0} \&\& {A_1} \\
	\\
	{A_0} \&\& {A_1}
	\arrow[""{name=0, anchor=center, inner sep=0}, "p"{description}, from=5-1, to=5-3]
	\arrow[""{name=1, anchor=center, inner sep=0}, "s"{description}, from=7-1, to=7-3]
	\arrow[Rightarrow, no head, from=5-1, to=7-1]
	\arrow[Rightarrow, no head, from=5-3, to=7-3]
	\arrow[""{name=2, anchor=center, inner sep=0}, "\pi"{description}, from=1-1, to=1-3]
	\arrow[Rightarrow, no head, from=1-1, to=3-1]
	\arrow[Rightarrow, no head, from=1-3, to=3-3]
	\arrow[""{name=3, anchor=center, inner sep=0}, "{\pushone{\alpha}{\pi}}"{description}, from=3-1, to=3-3]
	\arrow["\alpha"{description}, shorten <=9pt, shorten >=9pt, Rightarrow, from=0, to=1]
	\arrow["{\mathrm{A}}"{description}, shorten <=9pt, shorten >=9pt, Rightarrow, from=2, to=3]
\end{tikzcd}

such that for any cell $\Lambda$

% https://q.uiver.app/?q=WzAsMTYsWzIsNiwiQV8wIl0sWzQsNiwiQV8xIl0sWzIsOCwiQV8wIl0sWzQsOCwiQV8xIl0sWzAsNiwiQl8wIl0sWzYsNiwiQl8xIl0sWzAsOCwiQl8wIl0sWzYsOCwiQl8xIl0sWzAsMTAsIkNfMCJdLFs2LDEwLCJDXzEiXSxbMCwwLCJTXzAiXSxbMiwwLCJSXzAiXSxbNCwwLCJSXzEiXSxbNiwwLCJTXzEiXSxbMCw0LCJUXzAiXSxbNiw0LCJUXzEiXSxbMCwxLCJwIiwxXSxbMiwzLCJzIiwxXSxbMCwyLCIiLDEseyJsZXZlbCI6Miwic3R5bGUiOnsiaGVhZCI6eyJuYW1lIjoibm9uZSJ9fX1dLFsxLDMsIiIsMSx7ImxldmVsIjoyLCJzdHlsZSI6eyJoZWFkIjp7Im5hbWUiOiJub25lIn19fV0sWzQsMCwiciIsMV0sWzEsNSwiciciLDFdLFs2LDIsInIiLDFdLFs0LDYsIiIsMSx7ImxldmVsIjoyLCJzdHlsZSI6eyJoZWFkIjp7Im5hbWUiOiJub25lIn19fV0sWzUsNywiIiwxLHsibGV2ZWwiOjIsInN0eWxlIjp7ImhlYWQiOnsibmFtZSI6Im5vbmUifX19XSxbMyw3LCJyJyIsMV0sWzgsOSwidCIsMV0sWzYsOCwiZiIsMV0sWzcsOSwiZyIsMV0sWzEwLDExLCJcXHJobyIsMV0sWzExLDEyLCJcXHBpIiwxXSxbMTIsMTMsIlxccmhvJyIsMV0sWzEzLDE1LCJcXHBzaSIsMV0sWzE0LDE1LCJcXHRhdSIsMV0sWzEwLDE0LCJcXHZhcnBoaSIsMV0sWzE3LDI2LCJcXGJldGEiLDEseyJzaG9ydGVuIjp7InNvdXJjZSI6MjAsInRhcmdldCI6MjB9fV0sWzE2LDE3LCJcXGFscGhhIiwxLHsic2hvcnRlbiI6eyJzb3VyY2UiOjIwLCJ0YXJnZXQiOjIwfX1dLFszMCwzMywiXFxMYW1iZGEiLDEseyJzaG9ydGVuIjp7InNvdXJjZSI6MjAsInRhcmdldCI6MjB9fV1d
\begin{tikzcd}[ampersand replacement=\&]
	{S_0} \&\& {R_0} \&\& {R_1} \&\& {S_1} \\
	\\
	\\
	\\
	{T_0} \&\&\&\&\&\& {T_1} \\
	\\
	{B_0} \&\& {A_0} \&\& {A_1} \&\& {B_1} \\
	\\
	{B_0} \&\& {A_0} \&\& {A_1} \&\& {B_1} \\
	\\
	{C_0} \&\&\&\&\&\& {C_1}
	\arrow[""{name=0, anchor=center, inner sep=0}, "p"{description}, from=7-3, to=7-5]
	\arrow[""{name=1, anchor=center, inner sep=0}, "s"{description}, from=9-3, to=9-5]
	\arrow[Rightarrow, no head, from=7-3, to=9-3]
	\arrow[Rightarrow, no head, from=7-5, to=9-5]
	\arrow["r"{description}, from=7-1, to=7-3]
	\arrow["{r'}"{description}, from=7-5, to=7-7]
	\arrow["r"{description}, from=9-1, to=9-3]
	\arrow[Rightarrow, no head, from=7-1, to=9-1]
	\arrow[Rightarrow, no head, from=7-7, to=9-7]
	\arrow["{r'}"{description}, from=9-5, to=9-7]
	\arrow[""{name=2, anchor=center, inner sep=0}, "t"{description}, from=11-1, to=11-7]
	\arrow["f"{description}, from=9-1, to=11-1]
	\arrow["g"{description}, from=9-7, to=11-7]
	\arrow["\rho"{description}, from=1-1, to=1-3]
	\arrow[""{name=3, anchor=center, inner sep=0}, "\pi"{description}, from=1-3, to=1-5]
	\arrow["{\rho'}"{description}, from=1-5, to=1-7]
	\arrow["\psi"{description}, from=1-7, to=5-7]
	\arrow[""{name=4, anchor=center, inner sep=0}, "\tau"{description}, from=5-1, to=5-7]
	\arrow["\varphi"{description}, from=1-1, to=5-1]
	\arrow["\beta"{description}, shorten <=9pt, shorten >=9pt, Rightarrow, from=1, to=2]
	\arrow["\alpha"{description}, shorten <=9pt, shorten >=9pt, Rightarrow, from=0, to=1]
	\arrow["\Lambda"{description}, shorten <=17pt, shorten >=17pt, Rightarrow, from=3, to=4]
\end{tikzcd}

there is a unique factorisation through $\mathrm{A}$:

\begin{tikzcd}[ampersand replacement=\&]
	{S_0} \&\& {R_0} \&\& {R_1} \&\& {S_1} \\
	\\
	{S_0} \&\& {R_0} \&\& {R_1} \&\& {S_1} \\
	\\
	{T_0} \&\&\&\&\&\& {T_1} \\
	\\
	{B_0} \&\& {A_0} \&\& {A_1} \&\& {B_1} \\
	\\
	{B_0} \&\& {A_0} \&\& {A_1} \&\& {B_1} \\
	\\
	{C_0} \&\&\&\&\&\& {C_1}
	\arrow[""{name=0, anchor=center, inner sep=0}, "p"{description}, from=7-3, to=7-5]
	\arrow[""{name=1, anchor=center, inner sep=0}, "s"{description}, from=9-3, to=9-5]
	\arrow[Rightarrow, no head, from=7-3, to=9-3]
	\arrow[Rightarrow, no head, from=7-5, to=9-5]
	\arrow["r"{description}, from=7-1, to=7-3]
	\arrow["{r'}"{description}, from=7-5, to=7-7]
	\arrow["r"{description}, from=9-1, to=9-3]
	\arrow[Rightarrow, no head, from=7-1, to=9-1]
	\arrow[Rightarrow, no head, from=7-7, to=9-7]
	\arrow["{r'}"{description}, from=9-5, to=9-7]
	\arrow[""{name=2, anchor=center, inner sep=0}, "t"{description}, from=11-1, to=11-7]
	\arrow["f"{description}, from=9-1, to=11-1]
	\arrow["g"{description}, from=9-7, to=11-7]
	\arrow["\rho"{description}, from=1-1, to=1-3]
	\arrow[""{name=3, anchor=center, inner sep=0}, "\pi"{description}, from=1-3, to=1-5]
	\arrow["{\rho'}"{description}, from=1-5, to=1-7]
	\arrow[""{name=4, anchor=center, inner sep=0}, "\tau"{description}, from=5-1, to=5-7]
	\arrow[Rightarrow, no head, from=1-3, to=3-3]
	\arrow[Rightarrow, no head, from=1-5, to=3-5]
	\arrow[""{name=5, anchor=center, inner sep=0}, "{\pushone{\alpha}{\pi}}"{description}, from=3-3, to=3-5]
	\arrow["\rho"{description}, from=3-1, to=3-3]
	\arrow["{\rho'}"{description}, from=3-5, to=3-7]
	\arrow[Rightarrow, no head, from=1-1, to=3-1]
	\arrow[Rightarrow, no head, from=1-7, to=3-7]
	\arrow["\varphi"{description}, from=3-1, to=5-1]
	\arrow["\psi"{description}, from=3-7, to=5-7]
	\arrow["\beta"{description}, shorten <=9pt, shorten >=9pt, Rightarrow, from=1, to=2]
	\arrow["\alpha"{description}, shorten <=9pt, shorten >=9pt, Rightarrow, from=0, to=1]
	\arrow["{\mathrm{A}}"{description}, shorten <=9pt, shorten >=9pt, Rightarrow, from=3, to=5]
	\arrow["{\Lambda/\mathrm{A}}"{description}, shorten <=9pt, shorten >=9pt, Rightarrow, from=5, to=4]
\end{tikzcd}
\end{prop}
\begin{proof}
If it is a pushfibration then this property follows directly from the existence of opcartesian morphisms and their universal property.

So let us suppose that this is true.
First, since the functor has all opcartesian liftings of universal cells, \vdE is representable with the composite in \vdE lying over the composite in \vdB (i.e. it corresponds to a strict functor of double categories).

We will write the universal cells $-\bullet-$:

% https://q.uiver.app/?q=WzAsMTQsWzAsMCwiXFxidWxsZXQiXSxbMiwwLCJcXGJ1bGxldCJdLFszLDAsIlxcZG90cyJdLFs0LDAsIlxcYnVsbGV0Il0sWzYsMCwiXFxidWxsZXQiXSxbMCwyLCJcXGJ1bGxldCJdLFs2LDIsIlxcYnVsbGV0Il0sWzAsNCwiXFxidWxsZXQiXSxbMiw0LCJcXGJ1bGxldCJdLFszLDQsIlxcZG90cyJdLFs0LDQsIlxcYnVsbGV0Il0sWzYsNCwiXFxidWxsZXQiXSxbMCw2LCJcXGJ1bGxldCJdLFs2LDYsIlxcYnVsbGV0Il0sWzAsMSwiXFxwaV8xIiwxXSxbMyw0LCJcXHBpX24iLDFdLFs1LDYsIlxccGlfbiBcXGJ1bGxldFxcZG90c1xcYnVsbGV0XFxwaV8xIiwxXSxbMCw1LCIiLDEseyJsZXZlbCI6Miwic3R5bGUiOnsiaGVhZCI6eyJuYW1lIjoibm9uZSJ9fX1dLFs0LDYsIiIsMSx7ImxldmVsIjoyLCJzdHlsZSI6eyJoZWFkIjp7Im5hbWUiOiJub25lIn19fV0sWzcsOCwicF8xIiwxXSxbMTAsMTEsInBfbiIsMV0sWzEyLDEzLCJwX25cXGJ1bGxldFxcZG90c1xcYnVsbGV0IHBfMSIsMV0sWzcsMTIsIiIsMSx7ImxldmVsIjoyLCJzdHlsZSI6eyJoZWFkIjp7Im5hbWUiOiJub25lIn19fV0sWzExLDEzLCIiLDEseyJsZXZlbCI6Miwic3R5bGUiOnsiaGVhZCI6eyJuYW1lIjoibm9uZSJ9fX1dLFsyLDE2LCItXFxidWxsZXQtIiwxLHsic2hvcnRlbiI6eyJ0YXJnZXQiOjIwfX1dLFs5LDIxLCItXFxidWxsZXQtIiwxLHsic2hvcnRlbiI6eyJ0YXJnZXQiOjIwfX1dXQ==
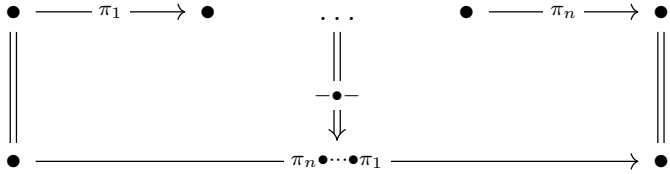
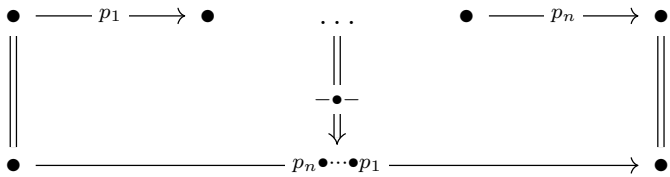
\begin{tikzcd}[ampersand replacement=\&]
	\bullet \&\& \bullet \& \dots \& \bullet \&\& \bullet \\
	\\
	\bullet \&\&\&\&\&\& \bullet \\
	\\
	\bullet \&\& \bullet \& \dots \& \bullet \&\& \bullet \\
	\\
	\bullet \&\&\&\&\&\& \bullet
	\arrow["{\pi_1}"{description}, from=1-1, to=1-3]
	\arrow["{\pi_n}"{description}, from=1-5, to=1-7]
	\arrow[""{name=0, anchor=center, inner sep=0}, "{\pi_n \bullet\dots\bullet\pi_1}"{description}, from=3-1, to=3-7]
	\arrow[Rightarrow, no head, from=1-1, to=3-1]
	\arrow[Rightarrow, no head, from=1-7, to=3-7]
	\arrow["{p_1}"{description}, from=5-1, to=5-3]
	\arrow["{p_n}"{description}, from=5-5, to=5-7]
	\arrow[""{name=1, anchor=center, inner sep=0}, "{p_n\bullet\dots\bullet p_1}"{description}, from=7-1, to=7-7]
	\arrow[Rightarrow, no head, from=5-1, to=7-1]
	\arrow[Rightarrow, no head, from=5-7, to=7-7]
	\arrow["{-\bullet-}"{description}, shorten >=7pt, Rightarrow, from=1-4, to=0]
	\arrow["{-\bullet-}"{description}, shorten >=7pt, Rightarrow, from=5-4, to=1]
\end{tikzcd}

For any chain of horizontal morphisms $pi_i \colon R_{i-1} \to R_i$ and any cell $\alpha$ we want to prove that the following cell is opcartesian:

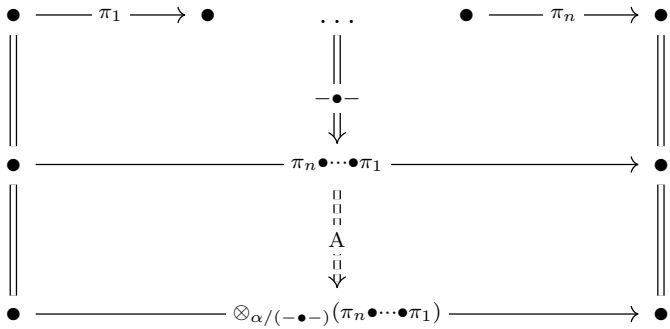
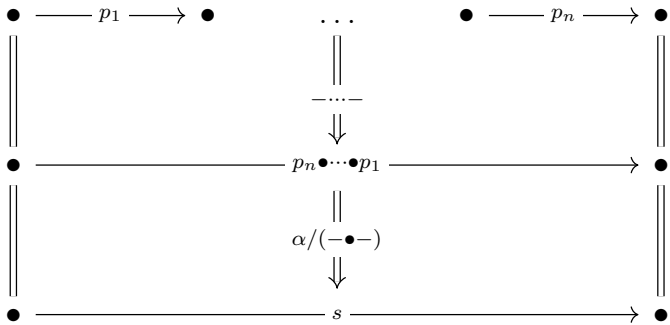
\begin{tikzcd}[ampersand replacement=\&]
	\bullet \&\& \bullet \& \dots \& \bullet \&\& \bullet \\
	\\
	\bullet \&\&\&\&\&\& \bullet \\
	\\
	\bullet \&\&\&\&\&\& \bullet \\
	\\
	\bullet \&\& \bullet \& \dots \& \bullet \&\& \bullet \\
	\\
	\bullet \&\&\&\&\&\& \bullet \\
	\\
	\bullet \&\&\&\&\&\& \bullet
	\arrow["{\pi_1}"{description}, from=1-1, to=1-3]
	\arrow["{\pi_n}"{description}, from=1-5, to=1-7]
	\arrow[""{name=0, anchor=center, inner sep=0}, "{\pi_n\bullet\dots\bullet\pi_1}"{description}, from=3-1, to=3-7]
	\arrow[Rightarrow, no head, from=1-1, to=3-1]
	\arrow[Rightarrow, no head, from=1-7, to=3-7]
	\arrow[""{name=1, anchor=center, inner sep=0}, "{\pushone{\alpha/(-\bullet-)}{\pi_n\bullet\dots\bullet\pi_1}}"{description}, from=5-1, to=5-7]
	\arrow[Rightarrow, no head, from=3-1, to=5-1]
	\arrow[Rightarrow, no head, from=3-7, to=5-7]
	\arrow["{p_1}"{description}, from=7-1, to=7-3]
	\arrow["{p_n}"{description}, from=7-5, to=7-7]
	\arrow[""{name=2, anchor=center, inner sep=0}, "{p_n\bullet\dots\bullet p_1}"{description}, from=9-1, to=9-7]
	\arrow[""{name=3, anchor=center, inner sep=0}, "s"{description}, from=11-1, to=11-7]
	\arrow[Rightarrow, no head, from=7-1, to=9-1]
	\arrow[Rightarrow, no head, from=7-7, to=9-7]
	\arrow[Rightarrow, no head, from=9-1, to=11-1]
	\arrow[Rightarrow, no head, from=9-7, to=11-7]
	\arrow["{-\bullet-}"{description}, shorten >=7pt, Rightarrow, from=1-4, to=0]
	\arrow["{\mathrm{A}}"{description}, shorten <=9pt, shorten >=9pt, Rightarrow, dashed, from=0, to=1]
	\arrow["{-\dots-}"{description}, shorten >=7pt, Rightarrow, from=7-4, to=2]
	\arrow["{\alpha/(-\bullet-)}"{description}, shorten <=9pt, shorten >=9pt, Rightarrow, from=2, to=3]
\end{tikzcd}

So let's consider a cell $\Lambda$ lying over a composition by $\alpha$:

\resizebox{\hsize}{!}{
\begin{tikzcd}[ampersand replacement=\&]
	{S_0} \&\& {S_1} \& \dots \& {S_{m-1}} \&\& {R_0} \&\& {R_1} \& \dots \& {R_{n-1}} \&\& {R_n} \&\& {S_1'} \& \dots \& {S_{m'-1}'} \&\& {S_{m'}'} \\
	\\
	{S_0} \&\&\&\&\&\&\&\&\&\&\&\&\&\&\&\&\&\& {S_{m'}'} \\
	\\
	{T_0} \&\&\&\&\&\&\&\&\&\&\&\&\&\&\&\&\&\& {T_1} \\
	\\
	{B_0} \&\& {B_1} \& \dots \& {B_{m-1}} \&\& {A_0} \&\& {A_1} \& \dots \& {A_{n-1}} \&\& {A_n} \&\& {B_1'} \& \dots \& {B_{m'-1}'} \&\& {B_{m'}'} \\
	\\
	{B_0} \&\& {B_1} \& \dots \& {B_{m-1}} \&\& {A_0} \&\&\&\&\&\& {A_n} \&\& {B_1'} \& \dots \& {B_{m'-1}'} \&\& {B_{m'}'} \\
	\\
	{C_0} \&\&\&\&\&\&\&\&\&\&\&\&\&\&\&\&\&\& {C_1}
	\arrow["{\pi_1}"{description}, from=1-7, to=1-9]
	\arrow["{\pi_n}"{description}, from=1-11, to=1-13]
	\arrow["{p_1}"{description}, from=7-7, to=7-9]
	\arrow["{p_n}"{description}, from=7-11, to=7-13]
	\arrow[""{name=0, anchor=center, inner sep=0}, "s"{description}, from=9-7, to=9-13]
	\arrow[Rightarrow, no head, from=7-7, to=9-7]
	\arrow[Rightarrow, no head, from=7-13, to=9-13]
	\arrow["{\rho_m}"{description}, from=1-5, to=1-7]
	\arrow["{\rho_1}"{description}, from=1-1, to=1-3]
	\arrow["{\rho_1'}"{description}, from=1-13, to=1-15]
	\arrow["{\rho_{m'}'}"{description}, from=1-17, to=1-19]
	\arrow[Rightarrow, no head, from=1-1, to=3-1]
	\arrow["\varphi"{description}, from=3-1, to=5-1]
	\arrow[Rightarrow, no head, from=1-19, to=3-19]
	\arrow["\psi"{description}, from=3-19, to=5-19]
	\arrow[""{name=1, anchor=center, inner sep=0}, "\tau"{description}, from=5-1, to=5-19]
	\arrow["{r_1}"{description}, from=7-1, to=7-3]
	\arrow["{r_m}"{description}, from=7-5, to=7-7]
	\arrow["{r_1'}"{description}, from=7-13, to=7-15]
	\arrow["{r_{m'}'}"{description}, from=7-17, to=7-19]
	\arrow[Rightarrow, no head, from=7-1, to=9-1]
	\arrow["{r_1}"{description}, from=9-1, to=9-3]
	\arrow[Rightarrow, no head, from=7-3, to=9-3]
	\arrow[Rightarrow, no head, from=7-5, to=9-5]
	\arrow["{r_m}"{description}, from=9-5, to=9-7]
	\arrow["{r_1'}"{description}, from=9-13, to=9-15]
	\arrow[Rightarrow, no head, from=7-15, to=9-15]
	\arrow[Rightarrow, no head, from=7-17, to=9-17]
	\arrow[Rightarrow, no head, from=7-19, to=9-19]
	\arrow["{r_{m'}'}"{description}, from=9-17, to=9-19]
	\arrow[""{name=2, anchor=center, inner sep=0}, "t"{description}, from=11-1, to=11-19]
	\arrow["f"{description}, from=9-1, to=11-1]
	\arrow["g"{description}, from=9-19, to=11-19]
	\arrow["\alpha"{description}, shorten >=7pt, Rightarrow, from=7-10, to=0]
	\arrow["\beta"{description}, shorten <=9pt, shorten >=9pt, Rightarrow, from=0, to=2]
	\arrow["\Lambda"{description}, shorten >=16pt, Rightarrow, from=1-10, to=1]
\end{tikzcd}
}

Then we can rewrite the bottom cell using the universal properties of $\bullet$:

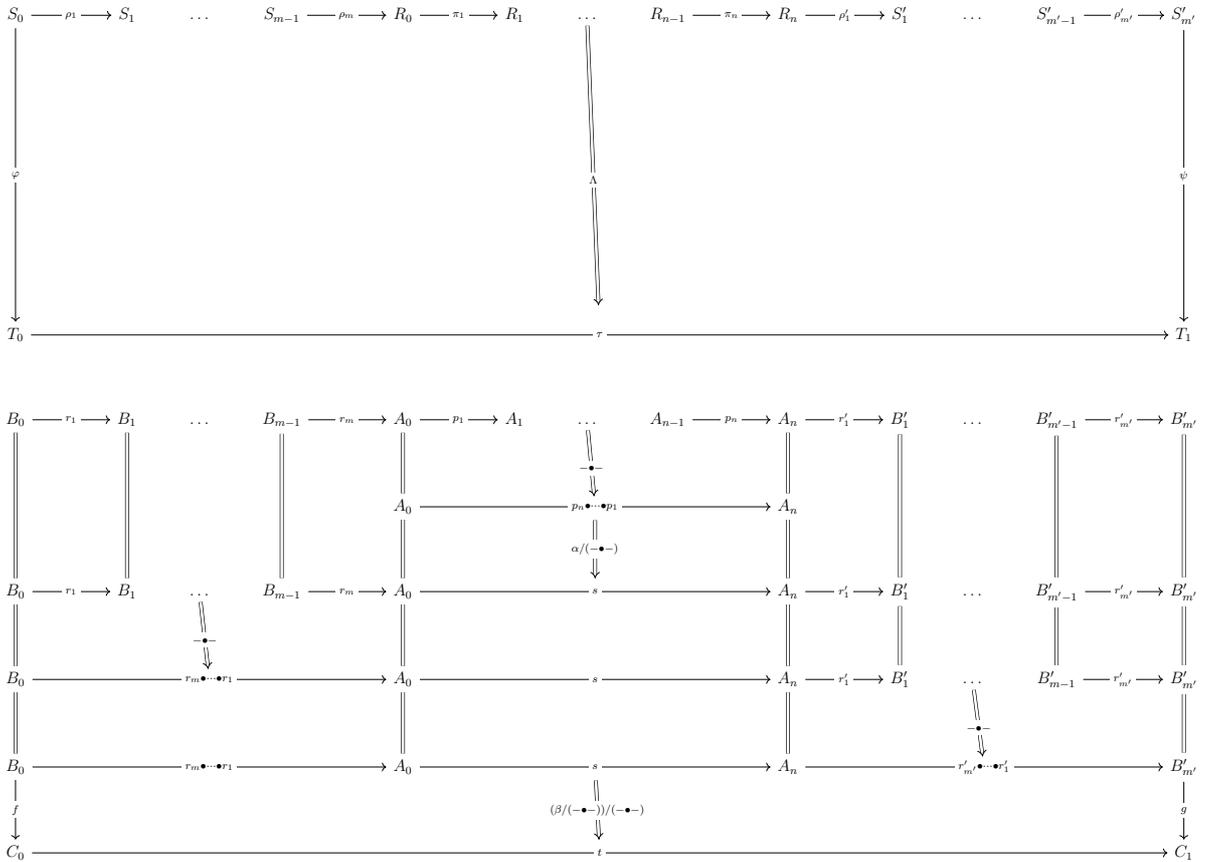
\resizebox{\hsize}{!}{
\begin{tikzcd}[ampersand replacement=\&]
	{S_0} \&\& {S_1} \& \dots \& {S_{m-1}} \&\& {R_0} \&\& {R_1} \& \dots \& {R_{n-1}} \&\& {R_n} \&\& {S_1'} \& \dots \& {S_{m'-1}'} \&\& {S_{m'}'} \\
	\\
	\\
	\\
	\\
	\\
	\\
	\\
	\\
	\\
	{T_0} \&\&\&\&\&\&\&\&\&\&\&\&\&\&\&\&\&\& {T_1} \\
	\\
	{B_0} \&\& {B_1} \& \dots \& {B_{m-1}} \&\& {A_0} \&\& {A_1} \& \dots \& {A_{n-1}} \&\& {A_n} \&\& {B_1'} \& \dots \& {B_{m'-1}'} \&\& {B_{m'}'} \\
	\\
	\&\&\&\&\&\& {A_0} \&\&\&\&\&\& {A_n} \\
	\\
	{B_0} \&\& {B_1} \& \dots \& {B_{m-1}} \&\& {A_0} \&\&\&\&\&\& {A_n} \&\& {B_1'} \& \dots \& {B_{m'-1}'} \&\& {B_{m'}'} \\
	\\
	{B_0} \&\&\&\&\&\& {A_0} \&\&\&\&\&\& {A_n} \&\& {B_1'} \& \dots \& {B_{m-1}'} \&\& {B_{m'}'} \\
	\\
	{B_0} \&\&\&\&\&\& {A_0} \&\&\&\&\&\& {A_n} \&\&\&\&\&\& {B_{m'}'} \\
	\\
	{C_0} \&\&\&\&\&\&\&\&\&\&\&\&\&\&\&\&\&\& {C_1}
	\arrow["{\pi_1}"{description}, from=1-7, to=1-9]
	\arrow["{\pi_n}"{description}, from=1-11, to=1-13]
	\arrow["{p_1}"{description}, from=13-7, to=13-9]
	\arrow["{p_n}"{description}, from=13-11, to=13-13]
	\arrow[""{name=0, anchor=center, inner sep=0}, "s"{description}, from=17-7, to=17-13]
	\arrow["{\rho_m}"{description}, from=1-5, to=1-7]
	\arrow["{\rho_1}"{description}, from=1-1, to=1-3]
	\arrow["{\rho_1'}"{description}, from=1-13, to=1-15]
	\arrow["{\rho_{m'}'}"{description}, from=1-17, to=1-19]
	\arrow[""{name=1, anchor=center, inner sep=0}, "\tau"{description}, from=11-1, to=11-19]
	\arrow["{r_1}"{description}, from=13-1, to=13-3]
	\arrow["{r_m}"{description}, from=13-5, to=13-7]
	\arrow["{r_1'}"{description}, from=13-13, to=13-15]
	\arrow["{r_{m'}'}"{description}, from=13-17, to=13-19]
	\arrow["{r_1}"{description}, from=17-1, to=17-3]
	\arrow["{r_m}"{description}, from=17-5, to=17-7]
	\arrow["{r_1'}"{description}, from=17-13, to=17-15]
	\arrow[Rightarrow, no head, from=13-19, to=17-19]
	\arrow["{r_{m'}'}"{description}, from=17-17, to=17-19]
	\arrow[""{name=2, anchor=center, inner sep=0}, "t"{description}, from=23-1, to=23-19]
	\arrow[""{name=3, anchor=center, inner sep=0}, "{p_n\bullet\dots\bullet p_1}"{description}, from=15-7, to=15-13]
	\arrow[Rightarrow, no head, from=13-7, to=15-7]
	\arrow[Rightarrow, no head, from=13-13, to=15-13]
	\arrow[Rightarrow, no head, from=13-17, to=17-17]
	\arrow[Rightarrow, no head, from=13-15, to=17-15]
	\arrow[Rightarrow, no head, from=13-5, to=17-5]
	\arrow[Rightarrow, no head, from=13-3, to=17-3]
	\arrow[Rightarrow, no head, from=13-1, to=17-1]
	\arrow[Rightarrow, no head, from=15-13, to=17-13]
	\arrow[Rightarrow, no head, from=15-7, to=17-7]
	\arrow[""{name=4, anchor=center, inner sep=0}, "{r_m\bullet\dots\bullet r_1}"{description}, from=19-1, to=19-7]
	\arrow["s"{description}, from=19-7, to=19-13]
	\arrow["{r_1'}"{description}, from=19-13, to=19-15]
	\arrow["{r_{m'}'}"{description}, from=19-17, to=19-19]
	\arrow[Rightarrow, no head, from=17-1, to=19-1]
	\arrow[Rightarrow, no head, from=17-7, to=19-7]
	\arrow[Rightarrow, no head, from=17-13, to=19-13]
	\arrow[Rightarrow, no head, from=17-15, to=19-15]
	\arrow[Rightarrow, no head, from=17-17, to=19-17]
	\arrow[Rightarrow, no head, from=17-19, to=19-19]
	\arrow["{r_m\bullet\dots\bullet r_1}"{description}, from=21-1, to=21-7]
	\arrow[""{name=5, anchor=center, inner sep=0}, "s"{description}, from=21-7, to=21-13]
	\arrow[""{name=6, anchor=center, inner sep=0}, "{r_{m'}'\bullet\dots\bullet r_1'}"{description}, from=21-13, to=21-19]
	\arrow[Rightarrow, no head, from=19-19, to=21-19]
	\arrow[Rightarrow, no head, from=19-1, to=21-1]
	\arrow["f"{description}, from=21-1, to=23-1]
	\arrow["g"{description}, from=21-19, to=23-19]
	\arrow[Rightarrow, no head, from=19-7, to=21-7]
	\arrow[Rightarrow, no head, from=19-13, to=21-13]
	\arrow["\varphi"{description}, from=1-1, to=11-1]
	\arrow["\psi"{description}, from=1-19, to=11-19]
	\arrow["\Lambda"{description}, shorten >=21pt, Rightarrow, from=1-10, to=1]
	\arrow["{-\bullet-}"{description}, shorten >=7pt, Rightarrow, from=13-10, to=3]
	\arrow["{\alpha/(-\bullet-)}"{description}, shorten <=9pt, shorten >=9pt, Rightarrow, from=3, to=0]
	\arrow["{-\bullet-}"{description}, shorten >=7pt, Rightarrow, from=17-4, to=4]
	\arrow["{-\bullet-}"{description}, shorten >=7pt, Rightarrow, from=19-16, to=6]
	\arrow["{(\beta/(-\bullet-))/(-\bullet-)}"{description}, shorten <=9pt, shorten >=9pt, Rightarrow, from=5, to=2]
\end{tikzcd}
}

Then we can rearrange the inner cells in the bottom diagram to factor $\Lambda$ in the top diagram:

\resizebox{\hsize}{!}{
\begin{tikzcd}[ampersand replacement=\&]
	{S_0} \&\& {S_1} \& \dots \& {S_{m-1}} \&\& {R_0} \&\& {R_1} \& \dots \& {R_{n-1}} \&\& {R_n} \&\& {S_1'} \& \dots \& {S_{m'-1}'} \&\& {S_{m'}'} \\
	\\
	{S_0} \&\& {S_1} \& \dots \& {S_{m-1}} \&\& {R_0} \&\&\&\&\&\& {R_n} \&\& {S_1'} \& \dots \& {S_{m'-1}'} \&\& {S_{m'}'} \\
	\\
	{S_0} \&\&\&\&\&\& {R_0} \&\&\&\&\&\& {R_n} \&\& {S_1'} \& \dots \& {S_{m'-1}'} \&\& {S_{m'}'} \\
	\\
	{S_0} \&\&\&\&\&\& {R_0} \&\&\&\&\&\& {R_n} \&\&\&\&\&\& {S_{m'}'} \\
	\\
	{S_0} \&\&\&\&\&\& {R_0} \&\&\&\&\&\& {R_n} \&\&\&\&\&\& {S_{m'}'} \\
	\\
	{T_0} \&\&\&\&\&\&\&\&\&\&\&\&\&\&\&\&\&\& {T_1} \\
	\\
	{B_0} \&\& {B_1} \& \dots \& {B_{m-1}} \&\& {A_0} \&\& {A_1} \& \dots \& {A_{n-1}} \&\& {A_n} \&\& {B_1'} \& \dots \& {B_{m'-1}'} \&\& {B_{m'}'} \\
	\\
	{B_0} \&\& {B_1} \& \dots \& {B_{m-1}} \&\& {A_0} \&\&\&\&\&\& {A_n} \&\& {B_1'} \& \dots \& {B_{m'-1}'} \&\& {B_{m'}'} \\
	\\
	{B_0} \&\&\&\&\&\& {A_0} \&\&\&\&\&\& {A_n} \&\& {B_1'} \& \dots \& {B_{m'-1}'} \&\& {B_{m'}'} \\
	\\
	{B_0} \&\&\&\&\&\& {A_0} \&\&\&\&\&\& {A_n} \&\&\&\&\&\& {B_{m'}'} \\
	\\
	{B_0} \&\&\&\&\&\& {A_0} \&\&\&\&\&\& {A_n} \&\&\&\&\&\& {B_{m'}'} \\
	\\
	{C_0} \&\&\&\&\&\&\&\&\&\&\&\&\&\&\&\&\&\& {C_1}
	\arrow["{\pi_1}"{description}, from=1-7, to=1-9]
	\arrow["{\pi_n}"{description}, from=1-11, to=1-13]
	\arrow["{p_1}"{description}, from=13-7, to=13-9]
	\arrow["{p_n}"{description}, from=13-11, to=13-13]
	\arrow["{\rho_m}"{description}, from=1-5, to=1-7]
	\arrow["{\rho_1}"{description}, from=1-1, to=1-3]
	\arrow["{\rho_1'}"{description}, from=1-13, to=1-15]
	\arrow["{\rho_{m'}'}"{description}, from=1-17, to=1-19]
	\arrow[""{name=0, anchor=center, inner sep=0}, "\tau"{description}, from=11-1, to=11-19]
	\arrow["{r_1}"{description}, from=13-1, to=13-3]
	\arrow["{r_m}"{description}, from=13-5, to=13-7]
	\arrow["{r_1'}"{description}, from=13-13, to=13-15]
	\arrow["{r_{m'}'}"{description}, from=13-17, to=13-19]
	\arrow["{r_1'}"{description}, from=17-13, to=17-15]
	\arrow["{r_{m'}'}"{description}, from=17-17, to=17-19]
	\arrow[""{name=1, anchor=center, inner sep=0}, "t"{description}, from=23-1, to=23-19]
	\arrow[""{name=2, anchor=center, inner sep=0}, "{p_n\bullet\dots\bullet p_1}"{description}, from=15-7, to=15-13]
	\arrow[Rightarrow, no head, from=13-7, to=15-7]
	\arrow[Rightarrow, no head, from=13-13, to=15-13]
	\arrow[Rightarrow, no head, from=15-13, to=17-13]
	\arrow[Rightarrow, no head, from=15-7, to=17-7]
	\arrow["{r_m\bullet\dots\bullet r_1}"{description}, from=19-1, to=19-7]
	\arrow[""{name=3, anchor=center, inner sep=0}, "{p_n\bullet\dots\bullet p_1}"{description}, from=19-7, to=19-13]
	\arrow[Rightarrow, no head, from=17-1, to=19-1]
	\arrow[Rightarrow, no head, from=17-7, to=19-7]
	\arrow[Rightarrow, no head, from=17-13, to=19-13]
	\arrow[Rightarrow, no head, from=17-19, to=19-19]
	\arrow["{r_m\bullet\dots\bullet r_1}"{description}, from=21-1, to=21-7]
	\arrow[""{name=4, anchor=center, inner sep=0}, "s"{description}, from=21-7, to=21-13]
	\arrow["{r_{m'}'\bullet\dots\bullet r_1'}"{description}, from=21-13, to=21-19]
	\arrow[Rightarrow, no head, from=19-19, to=21-19]
	\arrow[Rightarrow, no head, from=19-1, to=21-1]
	\arrow["f"{description}, from=21-1, to=23-1]
	\arrow["g"{description}, from=21-19, to=23-19]
	\arrow[Rightarrow, no head, from=19-7, to=21-7]
	\arrow[Rightarrow, no head, from=19-13, to=21-13]
	\arrow[from=15-1, to=15-3]
	\arrow[from=15-5, to=15-7]
	\arrow[from=15-13, to=15-15]
	\arrow[from=15-17, to=15-19]
	\arrow[Rightarrow, no head, from=13-19, to=15-19]
	\arrow[Rightarrow, no head, from=15-19, to=17-19]
	\arrow[Rightarrow, no head, from=13-17, to=15-17]
	\arrow[Rightarrow, no head, from=15-17, to=17-17]
	\arrow[Rightarrow, no head, from=13-15, to=15-15]
	\arrow[Rightarrow, no head, from=15-15, to=17-15]
	\arrow[Rightarrow, no head, from=13-5, to=15-5]
	\arrow[Rightarrow, no head, from=13-3, to=15-3]
	\arrow[Rightarrow, no head, from=13-1, to=15-1]
	\arrow[Rightarrow, no head, from=15-1, to=17-1]
	\arrow[Rightarrow, no head, from=1-1, to=3-1]
	\arrow[Rightarrow, no head, from=1-3, to=3-3]
	\arrow[Rightarrow, no head, from=1-5, to=3-5]
	\arrow[color={rgb,255:red,5;green,87;blue,240}, Rightarrow, no head, from=1-7, to=3-7]
	\arrow[color={rgb,255:red,5;green,87;blue,240}, Rightarrow, no head, from=1-13, to=3-13]
	\arrow[Rightarrow, no head, from=1-15, to=3-15]
	\arrow[Rightarrow, no head, from=1-17, to=3-17]
	\arrow[Rightarrow, no head, from=1-19, to=3-19]
	\arrow[""{name=5, anchor=center, inner sep=0}, "{\pi_n\bullet\dots\bullet\pi_1}"{description}, color={rgb,255:red,5;green,87;blue,240}, from=3-7, to=3-13]
	\arrow["{\rho_1}"{description}, from=3-1, to=3-3]
	\arrow["{\rho_m}"{description}, from=3-5, to=3-7]
	\arrow["{\rho_1'}"{description}, from=3-13, to=3-15]
	\arrow["{p_n\bullet\dots\bullet p_1}"{description}, from=17-7, to=17-13]
	\arrow[""{name=6, anchor=center, inner sep=0}, from=17-1, to=17-7]
	\arrow[""{name=7, anchor=center, inner sep=0}, from=19-13, to=19-19]
	\arrow[""{name=8, anchor=center, inner sep=0}, "{\rho_m\bullet\dots\bullet\rho_1}"{description}, color={rgb,255:red,5;green,87;blue,240}, from=5-1, to=5-7]
	\arrow["{\rho_1'}"{description}, from=5-13, to=5-15]
	\arrow["{\rho_{m'}'}"{description}, from=5-17, to=5-19]
	\arrow[color={rgb,255:red,5;green,87;blue,240}, Rightarrow, no head, from=3-7, to=5-7]
	\arrow["{\pi_n\bullet\dots\bullet\pi_1}"{description}, from=5-7, to=5-13]
	\arrow[color={rgb,255:red,5;green,87;blue,240}, Rightarrow, no head, from=3-1, to=5-1]
	\arrow[Rightarrow, no head, from=3-13, to=5-13]
	\arrow[Rightarrow, no head, from=3-15, to=5-15]
	\arrow[Rightarrow, no head, from=3-17, to=5-17]
	\arrow["{\rho_{m'}'}"{description}, from=3-17, to=3-19]
	\arrow[Rightarrow, no head, from=3-19, to=5-19]
	\arrow["{\rho_m\bullet\dots\bullet\rho_1}"{description}, from=7-1, to=7-7]
	\arrow[""{name=9, anchor=center, inner sep=0}, "{\pi_n\bullet\dots\bullet\pi_1}"{description}, color={rgb,255:red,42;green,178;blue,55}, from=7-7, to=7-13]
	\arrow[""{name=10, anchor=center, inner sep=0}, "{\rho_{m'}'\bullet\dots\bullet\rho_1'}"{description}, color={rgb,255:red,5;green,87;blue,240}, from=7-13, to=7-19]
	\arrow[color={rgb,255:red,5;green,87;blue,240}, Rightarrow, no head, from=5-19, to=7-19]
	\arrow[color={rgb,255:red,5;green,87;blue,240}, Rightarrow, no head, from=5-13, to=7-13]
	\arrow[Rightarrow, no head, from=5-7, to=7-7]
	\arrow[Rightarrow, no head, from=5-1, to=7-1]
	\arrow["{\rho_m\bullet\dots\bullet\rho_1}"{description}, from=9-1, to=9-7]
	\arrow[""{name=11, anchor=center, inner sep=0}, "{\pushone{\alpha/(-\bullet-)}{\pi_n\bullet\dots\bullet\pi_1}}"{description}, color={rgb,255:red,42;green,178;blue,55}, from=9-7, to=9-13]
	\arrow["{\rho_{m'}'\bullet\dots\rho_1'}"{description}, from=9-13, to=9-19]
	\arrow[Rightarrow, no head, from=7-1, to=9-1]
	\arrow[color={rgb,255:red,42;green,178;blue,55}, Rightarrow, no head, from=7-7, to=9-7]
	\arrow[color={rgb,255:red,42;green,178;blue,55}, Rightarrow, no head, from=7-13, to=9-13]
	\arrow[Rightarrow, no head, from=7-19, to=9-19]
	\arrow["\varphi"{description}, from=9-1, to=11-1]
	\arrow["\psi"{description}, from=9-19, to=11-19]
	\arrow["{-\bullet-}"{description}, shorten >=7pt, Rightarrow, from=13-10, to=2]
	\arrow["{(\beta/(-\bullet-))/(-\bullet-)}"{description}, shorten <=9pt, shorten >=9pt, Rightarrow, from=4, to=1]
	\arrow["{-\bullet-}"{description}, color={rgb,255:red,5;green,87;blue,240}, shorten >=7pt, Rightarrow, from=1-10, to=5]
	\arrow["{-\bullet-}"{description}, shorten >=7pt, Rightarrow, from=15-4, to=6]
	\arrow["{-\bullet-}"{description}, shorten >=7pt, Rightarrow, from=17-16, to=7]
	\arrow["{\alpha/(-\bullet-)}"{description}, shorten <=9pt, shorten >=9pt, Rightarrow, from=3, to=4]
	\arrow["{-\bullet-}"{description}, color={rgb,255:red,5;green,87;blue,240}, shorten >=7pt, Rightarrow, from=3-4, to=8]
	\arrow["{-\bullet-}"{description}, color={rgb,255:red,5;green,87;blue,240}, shorten >=7pt, Rightarrow, from=5-16, to=10]
	\arrow["A"{description}, shorten <=9pt, shorten >=9pt, Rightarrow, from=9, to=11]
	\arrow["{(\Lambda/(-\bullet-)^3)/A}"{description}, shorten <=9pt, shorten >=9pt, Rightarrow, from=11, to=0]
\end{tikzcd}
}

where the blue lifting comes from the fact that $-\bullet-$ is an opcartesian lifting and the green lifting comes from our assumption.

Then we can rewrite the diagrams to get a cell into $\pushone{\alpha/(-\bullet-)}{\pi_n\bullet\dots\bullet\pi_1}$ lying over $\alpha$ followed by a cell lying over $\beta$:

\resizebox{\hsize}{!}{
\begin{tikzcd}[ampersand replacement=\&]
	{S_0} \&\& {S_1} \& \dots \& {S_{m-1}} \&\& {R_0} \&\& {R_1} \& \dots \& {R_{n-1}} \&\& {R_n} \&\& {S_1'} \& \dots \& {S_{m'-1}'} \&\& {S_{m'}'} \\
	\\
	{S_0} \&\& {S_1} \&\& {S_{m-1}} \&\& {R_0} \&\&\&\&\&\& {R_n} \&\& {S_1'} \& \dots \& {S_{m'-1}'} \&\& {S_{m'}'} \\
	\\
	{S_0} \&\& {S_1} \& \dots \& {S_{m-1}} \&\& {R_0} \&\&\&\&\&\& {R_n} \&\& {S_1'} \& \dots \& {S_{m'-1}'} \&\& {S_{m'}'} \\
	\\
	{S_0} \&\&\&\&\&\& {R_0} \&\&\&\&\&\& {R_n} \&\& {S_1'} \& \dots \& {S_{m'-1}'} \&\& {S_{m'}'} \\
	\\
	{S_0} \&\&\&\&\&\& {R_0} \&\&\&\&\&\& {R_n} \&\&\&\&\&\& {S_{m'}'} \\
	\\
	{T_0} \&\&\&\&\&\&\&\&\&\&\&\&\&\&\&\&\&\& {T_1} \\
	\\
	{B_0} \&\& {B_1} \& \dots \& {B_{m-1}} \&\& {A_0} \&\& {A_1} \& \dots \& {A_{n-1}} \&\& {A_n} \&\& {B_1'} \& \dots \& {B_{m'-1}'} \&\& {B_{m'}'} \\
	\\
	{B_0} \&\& {B_1} \&\& {B_{m-1}} \&\& {A_0} \&\&\&\&\&\& {A_n} \&\& {B_1'} \& \dots \& {B_{m'-1}'} \&\& {B_{m'}'} \\
	\\
	{B_0} \&\& {B_1} \& \dots \& {B_{m-1}} \&\& {A_0} \&\&\&\&\&\& {A_n} \&\& {B_1'} \&\& {B_{m'-1}'} \&\& {B_{m'}'} \\
	\\
	\\
	\\
	\\
	\\
	{C_0} \&\&\&\&\&\&\&\&\&\&\&\&\&\&\&\&\&\& {C_1}
	\arrow["{\pi_1}"{description}, from=1-7, to=1-9]
	\arrow["{\pi_n}"{description}, from=1-11, to=1-13]
	\arrow["{p_1}"{description}, from=13-7, to=13-9]
	\arrow["{p_n}"{description}, from=13-11, to=13-13]
	\arrow["{\rho_m}"{description}, from=1-5, to=1-7]
	\arrow["{\rho_1}"{description}, from=1-1, to=1-3]
	\arrow["{\rho_1'}"{description}, from=1-13, to=1-15]
	\arrow["{\rho_{m'}'}"{description}, from=1-17, to=1-19]
	\arrow[""{name=0, anchor=center, inner sep=0}, "\tau"{description}, from=11-1, to=11-19]
	\arrow["{r_1}"{description}, from=13-1, to=13-3]
	\arrow["{r_m}"{description}, from=13-5, to=13-7]
	\arrow["{r_1'}"{description}, from=13-13, to=13-15]
	\arrow["{r_{m'}'}"{description}, from=13-17, to=13-19]
	\arrow["{r_1'}"{description}, from=17-13, to=17-15]
	\arrow["{r_{m'}'}"{description}, from=17-17, to=17-19]
	\arrow[""{name=1, anchor=center, inner sep=0}, "t"{description}, from=23-1, to=23-19]
	\arrow[Rightarrow, no head, from=13-7, to=15-7]
	\arrow[Rightarrow, no head, from=13-13, to=15-13]
	\arrow[Rightarrow, no head, from=15-13, to=17-13]
	\arrow[Rightarrow, no head, from=15-7, to=17-7]
	\arrow[from=15-1, to=15-3]
	\arrow[from=15-5, to=15-7]
	\arrow[from=15-13, to=15-15]
	\arrow[from=15-17, to=15-19]
	\arrow[Rightarrow, no head, from=13-19, to=15-19]
	\arrow[Rightarrow, no head, from=15-19, to=17-19]
	\arrow[Rightarrow, no head, from=13-17, to=15-17]
	\arrow[Rightarrow, no head, from=15-17, to=17-17]
	\arrow[Rightarrow, no head, from=13-15, to=15-15]
	\arrow[Rightarrow, no head, from=15-15, to=17-15]
	\arrow[Rightarrow, no head, from=13-5, to=15-5]
	\arrow[Rightarrow, no head, from=13-3, to=15-3]
	\arrow[Rightarrow, no head, from=13-1, to=15-1]
	\arrow[Rightarrow, no head, from=15-1, to=17-1]
	\arrow[Rightarrow, no head, from=1-1, to=3-1]
	\arrow[Rightarrow, no head, from=1-3, to=3-3]
	\arrow[Rightarrow, no head, from=1-5, to=3-5]
	\arrow[color={rgb,255:red,5;green,87;blue,240}, Rightarrow, no head, from=1-7, to=3-7]
	\arrow[color={rgb,255:red,5;green,87;blue,240}, Rightarrow, no head, from=1-13, to=3-13]
	\arrow[Rightarrow, no head, from=1-15, to=3-15]
	\arrow[Rightarrow, no head, from=1-17, to=3-17]
	\arrow[Rightarrow, no head, from=1-19, to=3-19]
	\arrow[""{name=2, anchor=center, inner sep=0}, "{\pi_n\bullet\dots\bullet\pi_1}"{description}, color={rgb,255:red,5;green,87;blue,240}, from=3-7, to=3-13]
	\arrow["{\rho_1}"{description}, from=3-1, to=3-3]
	\arrow["{\rho_m}"{description}, from=3-5, to=3-7]
	\arrow["{\rho_1'}"{description}, from=3-13, to=3-15]
	\arrow[""{name=3, anchor=center, inner sep=0}, "s"{description}, from=17-7, to=17-13]
	\arrow["{\rho_1'}"{description}, from=5-13, to=5-15]
	\arrow["{\rho_{m'}'}"{description}, from=5-17, to=5-19]
	\arrow[color={rgb,255:red,42;green,178;blue,55}, Rightarrow, no head, from=3-7, to=5-7]
	\arrow[Rightarrow, no head, from=3-1, to=5-1]
	\arrow[color={rgb,255:red,42;green,178;blue,55}, Rightarrow, no head, from=3-13, to=5-13]
	\arrow[Rightarrow, no head, from=3-15, to=5-15]
	\arrow[Rightarrow, no head, from=3-17, to=5-17]
	\arrow["{\rho_{m'}'}"{description}, from=3-17, to=3-19]
	\arrow[Rightarrow, no head, from=3-19, to=5-19]
	\arrow[""{name=4, anchor=center, inner sep=0}, "{\rho_m\bullet\dots\bullet\rho_1}"{description}, color={rgb,255:red,5;green,87;blue,240}, from=7-1, to=7-7]
	\arrow[Rightarrow, no head, from=5-19, to=7-19]
	\arrow[Rightarrow, no head, from=5-13, to=7-13]
	\arrow[color={rgb,255:red,5;green,87;blue,240}, Rightarrow, no head, from=5-7, to=7-7]
	\arrow[color={rgb,255:red,5;green,87;blue,240}, Rightarrow, no head, from=5-1, to=7-1]
	\arrow["{\rho_m\bullet\dots\bullet\rho_1}"{description}, from=9-1, to=9-7]
	\arrow[""{name=5, anchor=center, inner sep=0}, "{\rho_{m'}'\bullet\dots\rho_1'}"{description}, color={rgb,255:red,5;green,87;blue,240}, from=9-13, to=9-19]
	\arrow[Rightarrow, no head, from=7-1, to=9-1]
	\arrow[color={rgb,255:red,5;green,87;blue,240}, Rightarrow, no head, from=7-19, to=9-19]
	\arrow["\varphi"{description}, from=9-1, to=11-1]
	\arrow["\psi"{description}, from=9-19, to=11-19]
	\arrow[from=17-1, to=17-3]
	\arrow[from=17-5, to=17-7]
	\arrow[Rightarrow, no head, from=15-5, to=17-5]
	\arrow[no head, from=15-3, to=17-3]
	\arrow["f"{description}, from=17-1, to=23-1]
	\arrow["g"{description}, from=17-19, to=23-19]
	\arrow[color={rgb,255:red,5;green,87;blue,240}, Rightarrow, no head, from=7-13, to=9-13]
	\arrow[Rightarrow, no head, from=7-7, to=9-7]
	\arrow[""{name=6, anchor=center, inner sep=0}, "{\pushone{\alpha/(-\bullet-)}{\pi_n\bullet\dots\bullet\pi_1}}"{description}, from=9-7, to=9-13]
	\arrow[""{name=7, anchor=center, inner sep=0}, "{\pushone{\alpha/(-\bullet-)}{\pi_n\bullet\dots\bullet\pi_1}}"{description}, color={rgb,255:red,42;green,178;blue,55}, from=5-7, to=5-13]
	\arrow["{\pushone{\alpha/(-\bullet-)}{\pi_n\bullet\dots\bullet\pi_1}}"{description}, from=7-7, to=7-13]
	\arrow[Rightarrow, no head, from=5-15, to=7-15]
	\arrow["{\rho_1'}"{description}, from=7-13, to=7-15]
	\arrow[Rightarrow, no head, from=5-17, to=7-17]
	\arrow["{\rho_{m'}'}"{description}, from=7-17, to=7-19]
	\arrow[Rightarrow, no head, from=3-3, to=5-3]
	\arrow["{\rho_1}"{description}, from=5-1, to=5-3]
	\arrow[Rightarrow, no head, from=3-5, to=5-5]
	\arrow["{\rho_m}"{description}, from=5-5, to=5-7]
	\arrow["{-\bullet-}"{description}, color={rgb,255:red,5;green,87;blue,240}, shorten >=7pt, Rightarrow, from=1-10, to=2]
	\arrow["\beta"{description}, shorten <=26pt, shorten >=26pt, Rightarrow, from=3, to=1]
	\arrow["\alpha"{description}, shorten >=16pt, Rightarrow, from=13-10, to=3]
	\arrow["{(\Lambda/(-\bullet-)^3)/A}"{description}, shorten <=9pt, shorten >=9pt, Rightarrow, from=6, to=0]
	\arrow["A"{description}, color={rgb,255:red,42;green,178;blue,55}, shorten <=9pt, shorten >=9pt, Rightarrow, from=2, to=7]
	\arrow["{-\bullet-}"{description}, color={rgb,255:red,5;green,87;blue,240}, shorten >=7pt, Rightarrow, from=7-16, to=5]
	\arrow["{-\bullet-}"{description}, color={rgb,255:red,5;green,87;blue,240}, shorten >=7pt, Rightarrow, from=5-4, to=4]
\end{tikzcd}
}

Furthermore, it is unique by uniqueness of the factorisations we used.
This shows that $\mathrm{A}(-\bullet-)$ is opcartesian.

\end{proof}

From the notion of pushfibration of vdcs we can derive a notion of pushfibration of double categories and of 2-pushfibration of bicategories.

\begin{definition}
Given a strict functor of double categories $p \colon \vdE \to \vdB$, a cell $\mathrm{A}$

% https://q.uiver.app/?q=WzAsMTAsWzAsMCwiUl8wIl0sWzIsMCwiUl8xIl0sWzAsMiwiUl8wIl0sWzIsMiwiUl8xIl0sWzAsMywiQV8wIl0sWzIsMywiQV8xIl0sWzAsNSwiQV8wIl0sWzIsNSwiQV8xIl0sWzMsMSwiXFx2ZEUiXSxbMyw0LCJcXHZkQiJdLFswLDEsIlxccGkiLDFdLFsyLDMsIlxcc2lnbWEiLDFdLFswLDIsIiIsMSx7ImxldmVsIjoyLCJzdHlsZSI6eyJoZWFkIjp7Im5hbWUiOiJub25lIn19fV0sWzEsMywiIiwxLHsibGV2ZWwiOjIsInN0eWxlIjp7ImhlYWQiOnsibmFtZSI6Im5vbmUifX19XSxbNCw1LCJwIiwxXSxbNiw3LCJzIiwxXSxbNCw2LCIiLDEseyJsZXZlbCI6Miwic3R5bGUiOnsiaGVhZCI6eyJuYW1lIjoibm9uZSJ9fX1dLFs1LDcsIiIsMSx7ImxldmVsIjoyLCJzdHlsZSI6eyJoZWFkIjp7Im5hbWUiOiJub25lIn19fV0sWzgsOSwicCIsMV0sWzEwLDExLCJcXG1hdGhybXtBfSIsMSx7InNob3J0ZW4iOnsic291cmNlIjoyMCwidGFyZ2V0IjoyMH19XSxbMTQsMTUsIlxcYWxwaGEiLDEseyJzaG9ydGVuIjp7InNvdXJjZSI6MjAsInRhcmdldCI6MjB9fV1d
\begin{tikzcd}[ampersand replacement=\&]
	{R_0} \&\& {R_1} \\
	\&\&\& \vdE \\
	{R_0} \&\& {R_1} \\
	{A_0} \&\& {A_1} \\
	\&\&\& \vdB \\
	{A_0} \&\& {A_1}
	\arrow[""{name=0, anchor=center, inner sep=0}, "\pi"{description}, from=1-1, to=1-3]
	\arrow[""{name=1, anchor=center, inner sep=0}, "\sigma"{description}, from=3-1, to=3-3]
	\arrow[Rightarrow, no head, from=1-1, to=3-1]
	\arrow[Rightarrow, no head, from=1-3, to=3-3]
	\arrow[""{name=2, anchor=center, inner sep=0}, "p"{description}, from=4-1, to=4-3]
	\arrow[""{name=3, anchor=center, inner sep=0}, "s"{description}, from=6-1, to=6-3]
	\arrow[Rightarrow, no head, from=4-1, to=6-1]
	\arrow[Rightarrow, no head, from=4-3, to=6-3]
	\arrow["p"{description}, from=2-4, to=5-4]
	\arrow["{\mathrm{A}}"{description}, shorten <=9pt, shorten >=9pt, Rightarrow, from=0, to=1]
	\arrow["\alpha"{description}, shorten <=9pt, shorten >=9pt, Rightarrow, from=2, to=3]
\end{tikzcd}

is \defin{opcartesian} if for any cell $\Lambda$ lying over a composite

% https://q.uiver.app/?q=WzAsMTgsWzIsMCwiUl8wIl0sWzQsMCwiUl8xIl0sWzIsNiwiQV8wIl0sWzQsNiwiQV8xIl0sWzIsOCwiQV8wIl0sWzQsOCwiQV8xIl0sWzcsMiwiXFx2ZEUiXSxbNyw4LCJcXHZkQiJdLFs2LDYsIkJfMSJdLFswLDYsIkJfMCJdLFswLDgsIkJfMCJdLFs2LDgsIkJfMSJdLFswLDEwLCJDXzAiXSxbNiwxMCwiQ18xIl0sWzAsMCwiU18wIl0sWzYsMCwiU18xIl0sWzAsNCwiVF8wIl0sWzYsNCwiVF8xIl0sWzIsMywicCIsMV0sWzQsNSwicyIsMV0sWzIsNCwiIiwxLHsibGV2ZWwiOjIsInN0eWxlIjp7ImhlYWQiOnsibmFtZSI6Im5vbmUifX19XSxbMyw1LCIiLDEseyJsZXZlbCI6Miwic3R5bGUiOnsiaGVhZCI6eyJuYW1lIjoibm9uZSJ9fX1dLFs2LDcsInAiLDFdLFszLDgsInJfMSIsMV0sWzksMiwicl8wIiwxXSxbMTAsNF0sWzUsMTFdLFs5LDEwLCIiLDEseyJsZXZlbCI6Miwic3R5bGUiOnsiaGVhZCI6eyJuYW1lIjoibm9uZSJ9fX1dLFs4LDExLCIiLDEseyJsZXZlbCI6Miwic3R5bGUiOnsiaGVhZCI6eyJuYW1lIjoibm9uZSJ9fX1dLFsxMiwxMywidCIsMV0sWzEwLDEyLCJmIiwxXSxbMTEsMTMsImciLDFdLFswLDEsIlxccGkiLDFdLFsxNCwwLCJcXHJob18wIiwxXSxbMSwxNSwiXFxyaG9fMSIsMV0sWzE2LDE3LCJcXHRhdSIsMV0sWzE0LDE2LCJcXHZhcnBoaSIsMV0sWzE1LDE3LCJcXHBzaSIsMV0sWzE4LDE5LCJcXGFscGhhIiwxLHsic2hvcnRlbiI6eyJzb3VyY2UiOjIwLCJ0YXJnZXQiOjIwfX1dLFsxOSwyOSwiXFxiZXRhIiwxLHsic2hvcnRlbiI6eyJzb3VyY2UiOjIwLCJ0YXJnZXQiOjIwfX1dLFszMiwzNSwiXFxMYW1iZGEiLDEseyJzaG9ydGVuIjp7InNvdXJjZSI6MjAsInRhcmdldCI6MjB9fV1d
\begin{tikzcd}[ampersand replacement=\&]
	{S_0} \&\& {R_0} \&\& {R_1} \&\& {S_1} \\
	\\
	\&\&\&\&\&\&\& \vdE \\
	\\
	{T_0} \&\&\&\&\&\& {T_1} \\
	\\
	{B_0} \&\& {A_0} \&\& {A_1} \&\& {B_1} \\
	\\
	{B_0} \&\& {A_0} \&\& {A_1} \&\& {B_1} \& \vdB \\
	\\
	{C_0} \&\&\&\&\&\& {C_1}
	\arrow[""{name=0, anchor=center, inner sep=0}, "p"{description}, from=7-3, to=7-5]
	\arrow[""{name=1, anchor=center, inner sep=0}, "s"{description}, from=9-3, to=9-5]
	\arrow[Rightarrow, no head, from=7-3, to=9-3]
	\arrow[Rightarrow, no head, from=7-5, to=9-5]
	\arrow["p"{description}, from=3-8, to=9-8]
	\arrow["{r_1}"{description}, from=7-5, to=7-7]
	\arrow["{r_0}"{description}, from=7-1, to=7-3]
	\arrow[from=9-1, to=9-3]
	\arrow[from=9-5, to=9-7]
	\arrow[Rightarrow, no head, from=7-1, to=9-1]
	\arrow[Rightarrow, no head, from=7-7, to=9-7]
	\arrow[""{name=2, anchor=center, inner sep=0}, "t"{description}, from=11-1, to=11-7]
	\arrow["f"{description}, from=9-1, to=11-1]
	\arrow["g"{description}, from=9-7, to=11-7]
	\arrow[""{name=3, anchor=center, inner sep=0}, "\pi"{description}, from=1-3, to=1-5]
	\arrow["{\rho_0}"{description}, from=1-1, to=1-3]
	\arrow["{\rho_1}"{description}, from=1-5, to=1-7]
	\arrow[""{name=4, anchor=center, inner sep=0}, "\tau"{description}, from=5-1, to=5-7]
	\arrow["\varphi"{description}, from=1-1, to=5-1]
	\arrow["\psi"{description}, from=1-7, to=5-7]
	\arrow["\alpha"{description}, shorten <=9pt, shorten >=9pt, Rightarrow, from=0, to=1]
	\arrow["\beta"{description}, shorten <=9pt, shorten >=9pt, Rightarrow, from=1, to=2]
	\arrow["\Lambda"{description}, shorten <=17pt, shorten >=17pt, Rightarrow, from=3, to=4]
\end{tikzcd}

there is a unique factorisation through $\mathrm{A}$

\begin{tikzcd}[ampersand replacement=\&]
	{S_0} \&\& {R_0} \&\& {R_1} \&\& {S_1} \\
	\\
	{S_0} \&\& {R_0} \&\& {R_1} \&\& {S_1} \& \vdE \\
	\\
	{T_0} \&\&\&\&\&\& {T_1} \\
	\\
	{B_0} \&\& {A_0} \&\& {A_1} \&\& {B_1} \\
	\\
	{B_0} \&\& {A_0} \&\& {A_1} \&\& {B_1} \& \vdB \\
	\\
	{C_0} \&\&\&\&\&\& {C_1}
	\arrow[""{name=0, anchor=center, inner sep=0}, "p"{description}, from=7-3, to=7-5]
	\arrow[""{name=1, anchor=center, inner sep=0}, "s"{description}, from=9-3, to=9-5]
	\arrow[Rightarrow, no head, from=7-3, to=9-3]
	\arrow[Rightarrow, no head, from=7-5, to=9-5]
	\arrow["p"{description}, from=3-8, to=9-8]
	\arrow["{r_1}"{description}, from=7-5, to=7-7]
	\arrow["{r_0}"{description}, from=7-1, to=7-3]
	\arrow[from=9-1, to=9-3]
	\arrow[from=9-5, to=9-7]
	\arrow[Rightarrow, no head, from=7-1, to=9-1]
	\arrow[Rightarrow, no head, from=7-7, to=9-7]
	\arrow[""{name=2, anchor=center, inner sep=0}, "t"{description}, from=11-1, to=11-7]
	\arrow["f"{description}, from=9-1, to=11-1]
	\arrow["g"{description}, from=9-7, to=11-7]
	\arrow[""{name=3, anchor=center, inner sep=0}, "\pi"{description}, from=1-3, to=1-5]
	\arrow["{\rho_0}"{description}, from=1-1, to=1-3]
	\arrow["{\rho_1}"{description}, from=1-5, to=1-7]
	\arrow[""{name=4, anchor=center, inner sep=0}, "\tau"{description}, from=5-1, to=5-7]
	\arrow["\varphi"{description}, from=3-1, to=5-1]
	\arrow["\psi"{description}, from=3-7, to=5-7]
	\arrow["{\rho_0}"{description}, from=3-1, to=3-3]
	\arrow[""{name=5, anchor=center, inner sep=0}, "\sigma"{description}, from=3-3, to=3-5]
	\arrow["{\rho_1}"{description}, from=3-5, to=3-7]
	\arrow[Rightarrow, no head, from=1-1, to=3-1]
	\arrow[Rightarrow, no head, from=1-3, to=3-3]
	\arrow[Rightarrow, no head, from=1-5, to=3-5]
	\arrow[Rightarrow, no head, from=1-7, to=3-7]
	\arrow["\alpha"{description}, shorten <=9pt, shorten >=9pt, Rightarrow, from=0, to=1]
	\arrow["\beta"{description}, shorten <=9pt, shorten >=9pt, Rightarrow, from=1, to=2]
	\arrow["{\mathrm{A}}"{description}, shorten <=9pt, shorten >=9pt, Rightarrow, from=3, to=5]
	\arrow["{\Lambda/\mathrm{A}}"{description}, shorten <=9pt, shorten >=9pt, Rightarrow, dashed, from=5, to=4]
\end{tikzcd}
\end{definition}

\begin{definition}
A \defin{pushfibration} of double categories is a functor of double categories $p \colon \vdE \to \vdB$ that is strict and such that for any horizontal morphism $\pi$ in \vdE and any cell $\alpha$ in \vdB with domain $p(\pi)$ (and vertical morphisms identities) there is an opcartesian cell $\pi \Rightarrow \pushone{\alpha}{\pi}$.
\end{definition}

Finally, by taking the vertical category to be discrete we get the notion of 2-pushfibration of bicategories.

\begin{definition}
Given a strict functor $p \colon \cE \to \cB$ of bicategories, a 2-morphism $\mathrm{A} \colon \varphi \Rightarrow \psi$ is opcartesian if any $\Lambda \colon \rho_1 \circ \varphi \circ \rho_0 \Rightarrow \tau$ lying over a factorisation by the image of $\varphi$ factors through it:

% https://q.uiver.app/?q=WzAsMTAsWzIsMywiQV8wIl0sWzUsMywiQV8xIl0sWzgsMCwiXFx2ZEUiXSxbOCwzLCJcXHZkQiJdLFswLDMsIkJfMCJdLFs3LDMsIkJfMSJdLFswLDAsIlNfMCJdLFs3LDAsIlNfMSJdLFsyLDAsIlJfMCJdLFs1LDAsIlJfMSJdLFswLDEsImciLDFdLFsyLDMsInAiLDFdLFs0LDAsInJfMCIsMV0sWzEsNSwicl8xIiwxXSxbOCw5LCJcXHBzaSIsMV0sWzksNywiXFxyaG9fMSIsMV0sWzgsOSwiXFx2YXJwaGkiLDEseyJjdXJ2ZSI6LTV9XSxbNiw4LCJcXHJob18wIiwxXSxbNiw3LCJcXHRhdSIsMSx7ImN1cnZlIjo1fV0sWzAsMSwiZiIsMSx7ImN1cnZlIjotNX1dLFs0LDUsInQiLDEseyJjdXJ2ZSI6NX1dLFsxNiwxNCwiXFxtYXRocm17QX0iLDAseyJzaG9ydGVuIjp7InNvdXJjZSI6MjAsInRhcmdldCI6MjB9fV0sWzE0LDE4LCJcXExhbWJkYS9cXG1hdGhybXtBfSIsMCx7InNob3J0ZW4iOnsic291cmNlIjoyMCwidGFyZ2V0IjoyMH0sInN0eWxlIjp7ImJvZHkiOnsibmFtZSI6ImRhc2hlZCJ9fX1dLFsxOSwxMCwiXFxhbHBoYSIsMCx7InNob3J0ZW4iOnsic291cmNlIjoyMCwidGFyZ2V0IjoyMH19XSxbMTAsMjAsIlxcYmV0YSIsMCx7InNob3J0ZW4iOnsic291cmNlIjoyMCwidGFyZ2V0IjoyMH19XV0=
\begin{tikzcd}[ampersand replacement=\&]
	{S_0} \&\& {R_0} \&\&\& {R_1} \&\& {S_1} \& \cE \\
	\\
	\\
	{B_0} \&\& {A_0} \&\&\& {A_1} \&\& {B_1} \& \cB
	\arrow[""{name=0, anchor=center, inner sep=0}, "g"{description}, from=4-3, to=4-6]
	\arrow["p"{description}, from=1-9, to=4-9]
	\arrow["{r_0}"{description}, from=4-1, to=4-3]
	\arrow["{r_1}"{description}, from=4-6, to=4-8]
	\arrow[""{name=1, anchor=center, inner sep=0}, "\psi"{description}, from=1-3, to=1-6]
	\arrow["{\rho_1}"{description}, from=1-6, to=1-8]
	\arrow[""{name=2, anchor=center, inner sep=0}, "\varphi"{description}, curve={height=-30pt}, from=1-3, to=1-6]
	\arrow["{\rho_0}"{description}, from=1-1, to=1-3]
	\arrow[""{name=3, anchor=center, inner sep=0}, "\tau"{description}, curve={height=30pt}, from=1-1, to=1-8]
	\arrow[""{name=4, anchor=center, inner sep=0}, "f"{description}, curve={height=-30pt}, from=4-3, to=4-6]
	\arrow[""{name=5, anchor=center, inner sep=0}, "t"{description}, curve={height=30pt}, from=4-1, to=4-8]
	\arrow["{\mathrm{A}}", shorten <=4pt, shorten >=4pt, Rightarrow, from=2, to=1]
	\arrow["{\Lambda/\mathrm{A}}", shorten <=4pt, shorten >=4pt, Rightarrow, dashed, from=1, to=3]
	\arrow["\alpha", shorten <=4pt, shorten >=4pt, Rightarrow, from=4, to=0]
	\arrow["\beta", shorten <=4pt, shorten >=4pt, Rightarrow, from=0, to=5]
\end{tikzcd}

\end{definition}

\begin{definition}
A \defin{2-pushfibration} of bicategories is a strict 2-functor $p \colon \cE \to \cB$ such that for any morphism $\varphi$ in \cE and any 2-morphism $\alpha \colon p(\varphi) \Rightarrow g$ in \cB there is an opcartesian 2-morphism $\varphi \Rightarrow \pushone{\alpha}{\varphi}$ lying over $\alpha$.
\end{definition}

\begin{prop}
In a 2-pushfibration, pushforwards preserve composition, i.e.
\begin{itemize}
\item for any $\varphi_1 \colon R_0 \to R_1$, $\varphi_2 \colon R_1 \to R_2$, $\alpha_1 \colon p(\varphi_1) \Rightarrow g_1$, and $\alpha_2 \colon p(\varphi_2) \Rightarrow g_2$, we have \[\pushone{\alpha_2 \bullet \alpha_1}{\varphi_2 \bullet \varphi_1} \simeq \pushone{\alpha_2}{\varphi_2} \circ \pushone{\alpha_1}{\varphi_1}\]
\item for any $\varphi \colon R_0 \to R_1$, $\alpha \colon p(\varphi) \Rightarrow g$, and $\beta \colon g \Rightarrow h$, we have \[\pushone{\beta \circ \alpha}{\varphi} \simeq \pushone{\beta}{\pushone{\alpha}{\varphi}}\]
\end{itemize}
\end{prop}
\begin{proof}
We use that opcartesian cells compose in a vdc together with the fact that composition of morphisms in the bicategory is given by a pushforward in the underlying vdc.
\end{proof}

\section{Pullback in the category of vdcs}

Let $\mathbf{VDC}$ denote the category of virtual double categories and functors between them.

Consider two functors $F \colon \vdA \to \vdC$ and $P \colon \vdB \to \vdC$.
We will write $\vdA \times_{\vdC} \vdB$ for the virtual double category consisting of pairs of objects, morphisms and cells in $\vdA$ and $\vdB$ whose image coincide.
We will write diagrams in \vdB that are sent to ones in \vdC has lying over them and those in \vdA sent to ones in \vdC as lying on their left:

\resizebox{\hsize}{!}{
% https://q.uiver.app/?q=WzAsMjUsWzAsNCwiQV8wIl0sWzIsNCwiQV8xIl0sWzMsNCwiXFxkb3RzIl0sWzQsNCwiQV97bi0xfSJdLFs2LDQsIkFfbiJdLFswLDYsIkFfMCciXSxbNiw2LCJBXzEnIl0sWzksNCwiQ18wIl0sWzExLDQsIkNfMSJdLFsxMiw0LCJcXGRvdHMiXSxbMTMsNCwiQ197bi0xfSJdLFsxNSw0LCJDX24iXSxbOSw2LCJDXzAnIl0sWzE1LDYsIkNfMSciXSxbOSwwLCJCXzAiXSxbMTEsMCwiQl8xIl0sWzEzLDAsIkJfe24tMX0iXSxbMTUsMCwiQl9uIl0sWzksMiwiQl8wJyJdLFsxNSwyLCJCXzEnIl0sWzEyLDAsIlxcZG90cyJdLFszLDcsIlxcdmRBIl0sWzEyLDcsIlxcdmRDIl0sWzE3LDUsIlxcdmRDIl0sWzE3LDEsIlxcdmRCIl0sWzAsMSwicF8xIiwxXSxbMyw0LCJwX24iLDFdLFswLDUsImZfMCIsMV0sWzQsNiwiZl8xIiwxXSxbNSw2LCJwJyIsMV0sWzcsOCwicl8xIiwxXSxbMTAsMTEsInJfbiIsMV0sWzEyLDEzLCJyJyIsMV0sWzcsMTIsImhfMCIsMV0sWzExLDEzLCJoXzEiLDFdLFsxNCwxNSwicV8xIiwxXSxbMTYsMTcsInFfbiIsMV0sWzE4LDE5LCJxJyIsMV0sWzE0LDE4LCJnXzAiLDFdLFsxNywxOSwiZ18xIiwxXSxbMjQsMjMsInAiLDFdLFsyMSwyMiwiRiIsMV0sWzIsMjksIlxcYWxwaGEiLDEseyJzaG9ydGVuIjp7InRhcmdldCI6MjB9fV0sWzksMzIsIlxcZ2FtbWEiLDEseyJzaG9ydGVuIjp7InRhcmdldCI6MjB9fV0sWzIwLDM3LCJcXGJldGEiLDEseyJzaG9ydGVuIjp7InRhcmdldCI6MjB9fV1d
\begin{tikzcd}[ampersand replacement=\&]
	\&\&\&\&\&\&\&\&\& {B_0} \&\& {B_1} \& \dots \& {B_{n-1}} \&\& {B_n} \\
	\&\&\&\&\&\&\&\&\&\&\&\&\&\&\&\&\& \vdB \\
	\&\&\&\&\&\&\&\&\& {B_0'} \&\&\&\&\&\& {B_1'} \\
	\\
	{A_0} \&\& {A_1} \& \dots \& {A_{n-1}} \&\& {A_n} \&\&\& {C_0} \&\& {C_1} \& \dots \& {C_{n-1}} \&\& {C_n} \\
	\&\&\&\&\&\&\&\&\&\&\&\&\&\&\&\&\& \vdC \\
	{A_0'} \&\&\&\&\&\& {A_1'} \&\&\& {C_0'} \&\&\&\&\&\& {C_1'} \\
	\&\&\& \vdA \&\&\&\&\&\&\&\&\& \vdC
	\arrow["{p_1}"{description}, from=5-1, to=5-3]
	\arrow["{p_n}"{description}, from=5-5, to=5-7]
	\arrow["{f_0}"{description}, from=5-1, to=7-1]
	\arrow["{f_1}"{description}, from=5-7, to=7-7]
	\arrow[""{name=0, anchor=center, inner sep=0}, "{p'}"{description}, from=7-1, to=7-7]
	\arrow["{r_1}"{description}, from=5-10, to=5-12]
	\arrow["{r_n}"{description}, from=5-14, to=5-16]
	\arrow[""{name=1, anchor=center, inner sep=0}, "{r'}"{description}, from=7-10, to=7-16]
	\arrow["{h_0}"{description}, from=5-10, to=7-10]
	\arrow["{h_1}"{description}, from=5-16, to=7-16]
	\arrow["{q_1}"{description}, from=1-10, to=1-12]
	\arrow["{q_n}"{description}, from=1-14, to=1-16]
	\arrow[""{name=2, anchor=center, inner sep=0}, "{q'}"{description}, from=3-10, to=3-16]
	\arrow["{g_0}"{description}, from=1-10, to=3-10]
	\arrow["{g_1}"{description}, from=1-16, to=3-16]
	\arrow["P"{description}, from=2-18, to=6-18]
	\arrow["F"{description}, from=8-4, to=8-13]
	\arrow["\alpha"{description}, shorten >=7pt, Rightarrow, from=5-4, to=0]
	\arrow["\gamma"{description}, shorten >=7pt, Rightarrow, from=5-13, to=1]
	\arrow["\beta"{description}, shorten >=7pt, Rightarrow, from=1-13, to=2]
\end{tikzcd}
}

\begin{prop}
The following diagram is a pullback, where the functors out of $\vdA \times_\vdC \vdB$ are the obvious projections.
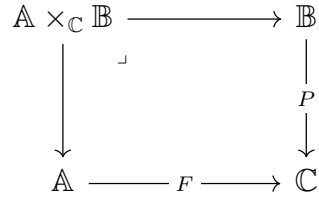
\[
% https://q.uiver.app/?q=WzAsNCxbMCwwLCJcXHZkQSBcXHRpbWVzX1xcdmRDIFxcdmRCIl0sWzIsMCwiXFx2ZEIiXSxbMCwyLCJcXHZkQSJdLFsyLDIsIlxcdmRDIl0sWzIsMywiRiIsMV0sWzEsMywicCIsMV0sWzAsMl0sWzAsMV0sWzAsMywiIiwxLHsic3R5bGUiOnsibmFtZSI6ImNvcm5lciJ9fV1d
\begin{tikzcd}[ampersand replacement=\&]
	{\vdA \times_\vdC \vdB} \&\& \vdB \\
	\\
	\vdA \&\& \vdC
	\arrow["F"{description}, from=3-1, to=3-3]
	\arrow["P"{description}, from=1-3, to=3-3]
	\arrow[from=1-1, to=3-1]
	\arrow[from=1-1, to=1-3]
	\arrow["\lrcorner"{anchor=center, pos=0.125}, draw=none, from=1-1, to=3-3]
\end{tikzcd}
\]
\end{prop}

Now suppose that $\vdA, \vdB, \vdC$ are representable.
It doesn't follow necessarily that $\vdA \times_\vdC \vdB$ is.
Indeed, representability only assures that there are cells $F(p_n) \bullet \dots \bullet F(p_1) \Rightarrow F(p_n\bullet \dots \bullet p_1) = P(q_n, \dots, q_1) \Leftarrow P(q_n)\bullet\dots\bullet P(q_1)$.
So the obvious choice of composing pointwise does not work since the image in \vdC does not coincide.
One possibility to make it work would be to ask for both functors to be strict.
Instead, we will keep one functor lax and ask the second one to be a pushfibration, which in particular is strict.
In the following we will take $P$ to be a pushfibration.
Furthermore, we will denote $F_\bullet := F(\bullet)/\bullet \colon F(p_n) \bullet \dots \bullet F(p_1) \to F(p_n \bullet \dots \bullet p_1)$.
The idea is to define composition by pushing the composite in \vdB along $F_\bullet$:
\[
(p_n,q_n)\bullet\dots\bullet (p_1,q_1) := (p_n \bullet\dots \bullet p_1, \pushone{F_\bullet}{q_n \bullet \dots \bullet q_1)})
\]
The construction is summarised in the following diagram:

\resizebox{\hsize}{!}{
\begin{tikzcd}[ampersand replacement=\&]
	\&\&\&\&\&\&\&\&\& {B_0} \&\& {B_1} \& \dots \& {B_{n-1}} \&\& {B_n} \\
	\\
	\&\&\&\&\&\&\&\&\& {B_0} \&\&\&\&\&\& {B_n} \\
	\\
	\&\&\&\&\&\&\&\&\& {B_0} \&\&\&\&\&\& {B_n} \\
	\\
	{A_0} \&\& {A_1} \& \dots \& {A_{n-1}} \&\& {A_n} \&\&\& {F(A_0)} \&\& {F(A_1)} \& \dots \& {F(A_{m-1})} \&\& {F(A_n)} \\
	\\
	\&\&\&\&\&\&\&\&\& {F(A_0)} \&\&\&\&\&\& {F(A_n)} \\
	\\
	{A_0} \&\&\&\&\&\& {A_n} \&\&\& {F(A_0)} \&\&\&\&\&\& {F(A_n)}
	\arrow["{F(p_1)}"{description}, from=7-10, to=7-12]
	\arrow["{F(p_n)}"{description}, from=7-14, to=7-16]
	\arrow[Rightarrow, no head, from=7-10, to=9-10]
	\arrow[Rightarrow, no head, from=7-16, to=9-16]
	\arrow[""{name=0, anchor=center, inner sep=0}, "{F(p_n)\bullet\dots\bullet F(p_1)}"{description}, from=9-10, to=9-16]
	\arrow[""{name=1, anchor=center, inner sep=0}, "{F(p_n \bullet \dots \bullet p_1)}"{description}, from=11-10, to=11-16]
	\arrow[Rightarrow, no head, from=9-10, to=11-10]
	\arrow[Rightarrow, no head, from=9-16, to=11-16]
	\arrow["{p_n}"{description}, from=7-5, to=7-7]
	\arrow["{p_1}"{description}, from=7-1, to=7-3]
	\arrow[""{name=2, anchor=center, inner sep=0}, "{p_n\bullet\dots\bullet p_1}"{description}, from=11-1, to=11-7]
	\arrow[Rightarrow, no head, from=7-7, to=11-7]
	\arrow[Rightarrow, no head, from=7-1, to=11-1]
	\arrow[""{name=3, anchor=center, inner sep=0}, "{\pushone{F_\bullet}{q_n \bullet\dots\bullet q_1}}"{description}, from=5-10, to=5-16]
	\arrow[""{name=4, anchor=center, inner sep=0}, "{q_n \bullet \dots \bullet q_1}"{description}, from=3-10, to=3-16]
	\arrow["{q_1}"{description}, from=1-10, to=1-12]
	\arrow["{q_n}"{description}, from=1-14, to=1-16]
	\arrow[Rightarrow, no head, from=1-10, to=3-10]
	\arrow[Rightarrow, no head, from=3-10, to=5-10]
	\arrow[Rightarrow, no head, from=1-16, to=3-16]
	\arrow[Rightarrow, no head, from=3-16, to=5-16]
	\arrow["{F_\bullet}"{description}, shorten <=9pt, shorten >=9pt, Rightarrow, from=0, to=1]
	\arrow["\bullet"{description}, shorten >=7pt, Rightarrow, from=7-13, to=0]
	\arrow["\bullet"{description}, shorten >=16pt, Rightarrow, from=7-4, to=2]
	\arrow["\bullet"{description}, shorten >=7pt, Rightarrow, from=1-13, to=4]
	\arrow[shorten <=9pt, shorten >=9pt, Rightarrow, from=4, to=3]
\end{tikzcd}
}

By definition we have that $P(\pushone{F_\bullet}{q_n\bullet\dots\bullet q_1} = F(p_n \bullet\dots\bullet p_1)$ and we use the fact that $P$ is strict so $P(q_n \bullet\dots\bullet q_1) = P(q_n)\bullet\dots\bullet P(q_1) = F(p_n)\bullet\dots\bullet F(p_1)$

\begin{prop}
The cell $(p_1,q_1),\dots,(p_n,q_n) \Rightarrow (p_n \bullet\dots\bullet p_1, \pushone{F_\bullet}{q_n \bullet\dots\bullet q_1})$ defined above is universal in $\vdA \times_\vdC \vdB$.
\end{prop}
\begin{proof}
Let $F_\bullet^\ast \colon q_n \bullet\dots\bullet q_1 \Rightarrow \pushone{F_\bullet}{q_n\bullet\dots\bullet q_1}$ denote the opcartesian cell in \vdB that we got from pushing along $F_\bullet$.

Now let's suppose that we have \[(\Lambda, Mu) \colon (r_1,s_1),\dots,(r_k,s_k),(p_1,q_1),\dots,(p_n,q_n),(r_1',s_1'),\dots,(r_{k'}',s_{k'}') \Rightarrow (t,u)\] in $\vdA \times_\vdC \vdB$.
We want to show that it can be uniquely factored through $(\bullet,F_\bullet^\ast \circ \bullet)$.

Since $\bullet$ is universal, we can factor $\Lambda$ through it to get a unique cell $\Lambda/\bullet$.
Now we have 
\begin{align*}
P(\mathrm{M}) &= F(\Lambda)\\
&= F(\Lambda/\bullet \circ \bullet)\\
&= F(\Lambda/\bullet) \circ F(\bullet)\\ 
&= F(\Lambda/\bullet) \circ F_\bullet \circ \bullet
\end{align*}

So $\mathrm{M}$ lies over a factorisation by $F_\bullet \circ \bullet$.
But since $P$ is strict and opcartesian cell compose we can factor $\mathrm{M}$ through the opcartesian cell $F_\bullet^\ast \circ \bullet$ to get a unique cell $(M/\bullet)/F_\bullet^\ast$ such that $P((M/\bullet)/F_\bullet^\ast) = F(\Lambda/\bullet)$.

This is summarised in figure \ref{fig:vdc_pull_comp}.

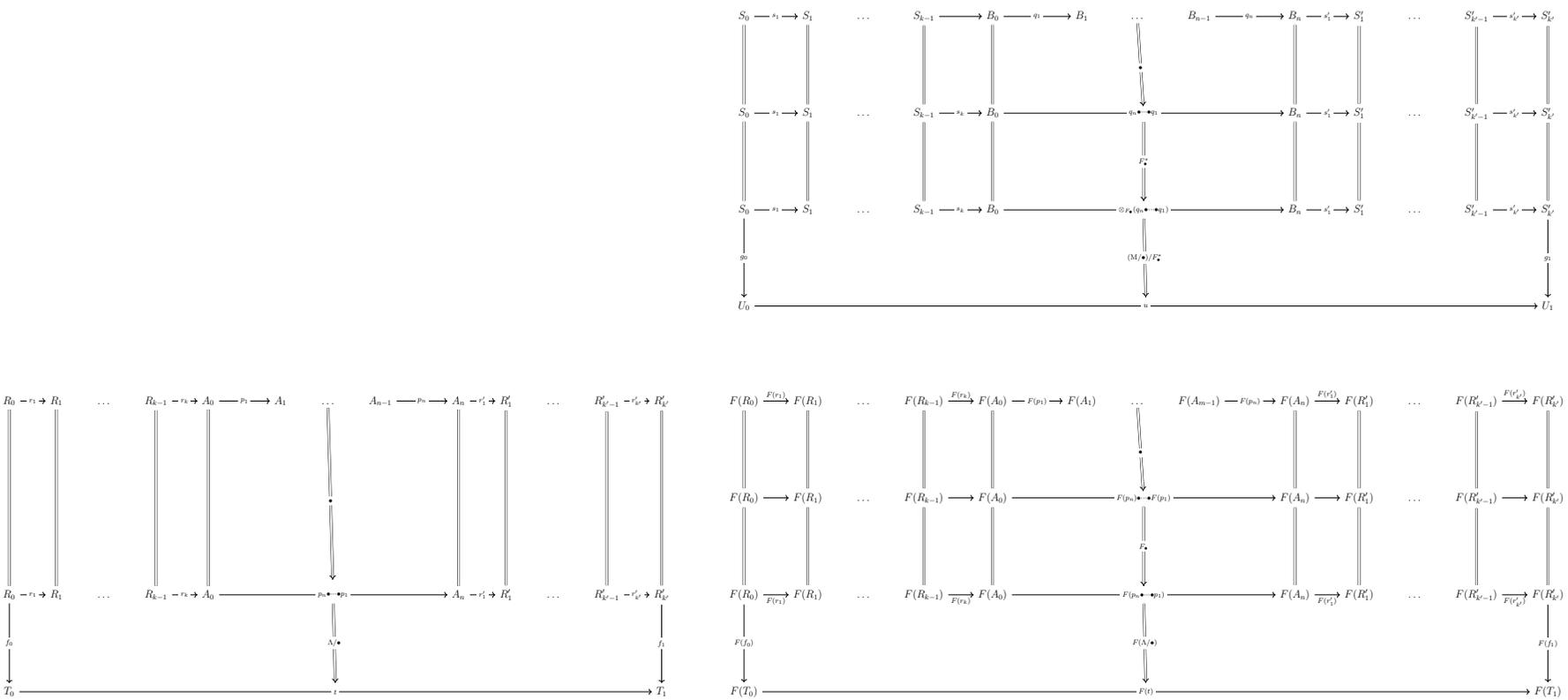
\begin{figure}
\rotatebox{90}{
\resizebox{\vsize}{!}{
\begin{tikzcd}[row sep= huge,ampersand replacement=\&]
	\&\&\&\&\&\&\&\&\&\&\&\&\&\&\&\& {S_0} \& {S_1} \& \dots \& {S_{k-1}} \& {B_0} \&\& {B_1} \& \dots \& {B_{n-1}} \&\& {B_n} \& {S_1'} \& \dots \& {S_{k'-1}'} \& {S_{k'}'} \\
	\\
	\&\&\&\&\&\&\&\&\&\&\&\&\&\&\&\& {S_0} \& {S_1} \& \dots \& {S_{k-1}} \& {B_0} \&\&\&\&\&\& {B_n} \& {S_1'} \& \dots \& {S_{k'-1}'} \& {S_{k'}'} \\
	\\
	\&\&\&\&\&\&\&\&\&\&\&\&\&\&\&\& {S_0} \& {S_1} \& \dots \& {S_{k-1}} \& {B_0} \&\&\&\&\&\& {B_n} \& {S_1'} \& \dots \& {S_{k'-1}'} \& {S_{k'}'} \\
	\\
	\&\&\&\&\&\&\&\&\&\&\&\&\&\&\&\& {U_0} \&\&\&\&\&\&\&\&\&\&\&\&\&\& {U_1} \\
	\\
	{R_0} \& {R_1} \& \dots \& {R_{k-1}} \& {A_0} \&\& {A_1} \& \dots \& {A_{n-1}} \&\& {A_n} \& {R_1'} \& \dots \& {R_{k'-1}'} \& {R_{k'}'} \&\& {F(R_0)} \& {F(R_1)} \& \dots \& {F(R_{k-1})} \& {F(A_0)} \&\& {F(A_1)} \& \dots \& {F(A_{m-1})} \&\& {F(A_n)} \& {F(R_1')} \& \dots \& {F(R_{k'-1}')} \& {F(R_{k'}')} \\
	\\
	\&\&\&\&\&\&\&\&\&\&\&\&\&\&\&\& {F(R_0)} \& {F(R_1)} \& \dots \& {F(R_{k-1})} \& {F(A_0)} \&\&\&\&\&\& {F(A_n)} \& {F(R_1')} \& \dots \& {F(R_{k'-1}')} \& {F(R_{k'}')} \\
	\\
	{R_0} \& {R_1} \& \dots \& {R_{k-1}} \& {A_0} \&\&\&\&\&\& {A_n} \& {R_1'} \& \dots \& {R_{k'-1}'} \& {R_{k'}'} \&\& {F(R_0)} \& {F(R_1)} \& \dots \& {F(R_{k-1})} \& {F(A_0)} \&\&\&\&\&\& {F(A_n)} \& {F(R_1')} \& \dots \& {F(R_{k'-1}')} \& {F(R_{k'}')} \\
	\\
	{T_0} \&\&\&\&\&\&\&\&\&\&\&\&\&\& {T_1} \&\& {F(T_0)} \&\&\&\&\&\&\&\&\&\&\&\&\&\& {F(T_1)}
	\arrow["{F(p_1)}"{description}, from=9-21, to=9-23]
	\arrow["{F(p_n)}"{description}, from=9-25, to=9-27]
	\arrow[Rightarrow, no head, from=9-21, to=11-21]
	\arrow[Rightarrow, no head, from=9-27, to=11-27]
	\arrow[""{name=0, anchor=center, inner sep=0}, "{F(p_n)\bullet\dots\bullet F(p_1)}"{description}, from=11-21, to=11-27]
	\arrow[""{name=1, anchor=center, inner sep=0}, "{F(p_n \bullet \dots \bullet p_1)}"{description}, from=13-21, to=13-27]
	\arrow[Rightarrow, no head, from=11-21, to=13-21]
	\arrow[Rightarrow, no head, from=11-27, to=13-27]
	\arrow["{p_n}"{description}, from=9-9, to=9-11]
	\arrow["{p_1}"{description}, from=9-5, to=9-7]
	\arrow[""{name=2, anchor=center, inner sep=0}, "{p_n\bullet\dots\bullet p_1}"{description}, from=13-5, to=13-11]
	\arrow[Rightarrow, no head, from=9-11, to=13-11]
	\arrow[Rightarrow, no head, from=9-5, to=13-5]
	\arrow[""{name=3, anchor=center, inner sep=0}, "{\pushone{F_\bullet}{q_n \bullet\dots\bullet q_1}}"{description}, from=5-21, to=5-27]
	\arrow[""{name=4, anchor=center, inner sep=0}, "{q_n \bullet \dots \bullet q_1}"{description}, from=3-21, to=3-27]
	\arrow["{q_1}"{description}, from=1-21, to=1-23]
	\arrow["{q_n}"{description}, from=1-25, to=1-27]
	\arrow[Rightarrow, no head, from=1-21, to=3-21]
	\arrow[Rightarrow, no head, from=3-21, to=5-21]
	\arrow[Rightarrow, no head, from=1-27, to=3-27]
	\arrow[Rightarrow, no head, from=3-27, to=5-27]
	\arrow["{r_k}"{description}, from=9-4, to=9-5]
	\arrow["{r_1}"{description}, from=9-1, to=9-2]
	\arrow["{r_1'}"{description}, from=9-11, to=9-12]
	\arrow["{r_{k'}'}"{description}, from=9-14, to=9-15]
	\arrow[""{name=5, anchor=center, inner sep=0}, "t"{description}, from=15-1, to=15-15]
	\arrow["{r_1}"{description}, from=13-1, to=13-2]
	\arrow["{r_k}"{description}, from=13-4, to=13-5]
	\arrow["{r_1'}"{description}, from=13-11, to=13-12]
	\arrow["{r_{k'}'}"{description}, from=13-14, to=13-15]
	\arrow[Rightarrow, no head, from=9-15, to=13-15]
	\arrow[Rightarrow, no head, from=9-14, to=13-14]
	\arrow[Rightarrow, no head, from=9-12, to=13-12]
	\arrow[Rightarrow, no head, from=9-4, to=13-4]
	\arrow[Rightarrow, no head, from=9-2, to=13-2]
	\arrow[Rightarrow, no head, from=9-1, to=13-1]
	\arrow["{f_1}"{description}, from=13-15, to=15-15]
	\arrow["{f_0}"{description}, from=13-1, to=15-1]
	\arrow["{F(r_k)}", from=9-20, to=9-21]
	\arrow["{F(r_1)}", from=9-17, to=9-18]
	\arrow["{F(r_1')}", from=9-27, to=9-28]
	\arrow["{F(r_{k'}')}", from=9-30, to=9-31]
	\arrow["{F(f_0)}"{description}, from=13-17, to=15-17]
	\arrow[""{name=6, anchor=center, inner sep=0}, "{F(t)}"{description}, from=15-17, to=15-31]
	\arrow["{F(f_1)}"{description}, from=13-31, to=15-31]
	\arrow["{F(r_1)}"', from=13-17, to=13-18]
	\arrow["{F(r_k)}"', from=13-20, to=13-21]
	\arrow["{F(r_1')}"', from=13-27, to=13-28]
	\arrow["{F(r_{k'}')}"', from=13-30, to=13-31]
	\arrow[from=1-20, to=1-21]
	\arrow["{s_1}"{description}, from=1-17, to=1-18]
	\arrow["{s_1'}"{description}, from=1-27, to=1-28]
	\arrow["{s_{k'}'}"{description}, from=1-30, to=1-31]
	\arrow[Rightarrow, no head, from=1-17, to=3-17]
	\arrow["{s_1}"{description}, from=3-17, to=3-18]
	\arrow[Rightarrow, no head, from=1-18, to=3-18]
	\arrow[Rightarrow, no head, from=1-20, to=3-20]
	\arrow["{s_k}"{description}, from=3-20, to=3-21]
	\arrow[Rightarrow, no head, from=1-28, to=3-28]
	\arrow["{s_1'}"{description}, from=3-27, to=3-28]
	\arrow[Rightarrow, no head, from=1-31, to=3-31]
	\arrow["{s_{k'}'}"{description}, from=3-30, to=3-31]
	\arrow[Rightarrow, no head, from=1-30, to=3-30]
	\arrow[Rightarrow, no head, from=3-17, to=5-17]
	\arrow["{s_1}"{description}, from=5-17, to=5-18]
	\arrow[Rightarrow, no head, from=3-18, to=5-18]
	\arrow[Rightarrow, no head, from=3-20, to=5-20]
	\arrow["{s_k}"{description}, from=5-20, to=5-21]
	\arrow["{s_1'}"{description}, from=5-27, to=5-28]
	\arrow[Rightarrow, no head, from=3-28, to=5-28]
	\arrow[Rightarrow, no head, from=3-31, to=5-31]
	\arrow["{s_{k'}'}"{description}, from=5-30, to=5-31]
	\arrow[Rightarrow, no head, from=3-30, to=5-30]
	\arrow["{g_0}"{description}, from=5-17, to=7-17]
	\arrow["{g_1}"{description}, from=5-31, to=7-31]
	\arrow[""{name=7, anchor=center, inner sep=0}, "u"{description}, from=7-17, to=7-31]
	\arrow[Rightarrow, no head, from=9-17, to=11-17]
	\arrow[from=11-17, to=11-18]
	\arrow[Rightarrow, no head, from=9-18, to=11-18]
	\arrow[Rightarrow, no head, from=9-20, to=11-20]
	\arrow[from=11-20, to=11-21]
	\arrow[from=11-27, to=11-28]
	\arrow[Rightarrow, no head, from=9-28, to=11-28]
	\arrow[Rightarrow, no head, from=9-31, to=11-31]
	\arrow[from=11-30, to=11-31]
	\arrow[Rightarrow, no head, from=9-30, to=11-30]
	\arrow[Rightarrow, no head, from=11-17, to=13-17]
	\arrow[Rightarrow, no head, from=11-18, to=13-18]
	\arrow[Rightarrow, no head, from=11-20, to=13-20]
	\arrow[Rightarrow, no head, from=11-31, to=13-31]
	\arrow[Rightarrow, no head, from=11-30, to=13-30]
	\arrow[Rightarrow, no head, from=11-28, to=13-28]
	\arrow["{F_\bullet}"{description}, shorten <=9pt, shorten >=9pt, Rightarrow, from=0, to=1]
	\arrow["\bullet"{description}, shorten >=7pt, Rightarrow, from=9-24, to=0]
	\arrow["\bullet"{description}, shorten >=16pt, Rightarrow, from=9-8, to=2]
	\arrow["\bullet"{description}, shorten >=7pt, Rightarrow, from=1-24, to=4]
	\arrow["{F_\bullet^\ast}"{description}, shorten <=9pt, shorten >=9pt, Rightarrow, from=4, to=3]
	\arrow["{\Lambda/\bullet}"{description}, shorten <=9pt, shorten >=9pt, Rightarrow, from=2, to=5]
	\arrow["{F(\Lambda/\bullet)}"{description}, shorten <=9pt, shorten >=9pt, Rightarrow, from=1, to=6]
	\arrow["{(\mathrm{M}/\bullet)/F_\bullet^\ast}"{description}, shorten <=9pt, shorten >=9pt, Rightarrow, from=3, to=7]
\end{tikzcd}
}}
\caption{Universality of composition in the pullback of vdcs}
\label{fig:vdc_pull_comp}
\end{figure}
\end{proof}

Furthermore, since composition on the left component in $\vdA \times_\vdC \vdB$ is defined as composition in \vdA, the projection is strict.
In general, this is not the case for the projection on \vdB.
However, it will have the same properties as $F$.

\begin{prop}
If $F$ is normal/pseudo/strict then the projection $\vdA \times_\vdC \vdB \to \vdB$ is normal/pseudo/strict.
\end{prop}
\begin{proof}
The identity is opcartesian, if $F_\bullet$ is an identity we can choose the identity for $F_\bullet^\ast$ and then the composition in $\vdA \times_\vdC \vdB$ becomes $(\bullet,\bullet)$.
If $F_\bullet$ is invertible, then since $F_\bullet^{-1} \circ F_\bullet = \id_{F(p_n)\bullet\dots\bullet F(p_1)}$ we can factor the identity cell through $F_\bullet^\ast$: \[\id_{q_n \circ q_1} = (\id_{q_n \circ q_1}/F_\bullet^\ast) \circ F_\bullet^\ast\]
Then by postcomposing by $F_\bullet^\ast$ we get \[F_\bullet^\ast \circ (\id_{q_n \circ q_1}/F_\bullet^\ast) \circ F_\bullet^\ast = F_\bullet^\ast = \id_{\pushone{F_\bullet}{q_n\bullet\dots\bullet q_1}} \circ F_\bullet^\ast\]
But by unicity of the factorisation through a pushforward we get \[F_\bullet^\ast \circ (\id_{q_n \circ q_1}/F_\bullet^\ast) = \id_{\pushone{F_\bullet}{q_n\bullet\dots\bullet q_1}}\]
So $F_\bullet^\ast$ is invertible.
\end{proof}

Let us sum up what we just proved.

\begin{prop}
Let \vdA, \vdB, \vdC be representable vdcs, $F \colon \vdA \to \vdC$ a functor and $P \colon \vdB \to \vdC$ a pushfibration, then $\vdA \times_\vdC \vdB$ is representable, with projection on $\vdA$ strict and projection on $\vdB$ normal/pseudo/strict when $F$ is.
\[
% https://q.uiver.app/?q=WzAsNCxbMCwwLCJcXHZkQSBcXHRpbWVzX1xcdmRDIFxcdmRCIl0sWzMsMCwiXFx2ZEIiXSxbMCwzLCJcXHZkQSJdLFszLDMsIlxcdmRDIl0sWzIsMywiXFx0ZXh0e2xheC9wcy4vbm9yLi9zdHIufSIsMV0sWzAsMSwiXFx0ZXh0e2xheC9wcy4vbm9yLi9zdHIufSIsMV0sWzEsMywiXFx0ZXh0e3B1c2hmaWIufSIsMV0sWzAsMiwiXFx0ZXh0e3N0cn0iLDFdLFswLDMsIiIsMSx7InN0eWxlIjp7Im5hbWUiOiJjb3JuZXIifX1dXQ==
\begin{tikzcd}[ampersand replacement=\&]
	{\vdA \times_\vdC \vdB} \&\&\& \vdB \\
	\\
	\\
	\vdA \&\&\& \vdC
	\arrow["{\text{lax/ps./nor./str.}}"{description}, from=4-1, to=4-4]
	\arrow["{\text{lax/ps./nor./str.}}"{description}, from=1-1, to=1-4]
	\arrow["{\text{pushfib.}}"{description}, from=1-4, to=4-4]
	\arrow["{\text{str}}"{description}, from=1-1, to=4-1]
	\arrow["\lrcorner"{anchor=center, pos=0.125}, draw=none, from=1-1, to=4-4]
\end{tikzcd}
\]
\end{prop}

We can first note that it is not necessary to assume that \vdB is representable since it follows from the fact that it is pushfibred over a representable vdc.
One could also ask when is the projection to \vdA pushfibred.
In fact, it is always the case.

\begin{prop}
Consider a functor of vdcs $F \colon \vdA \to \vdC$ and a pushfibration $P \colon \vdB \to \vdC$.
Then, the projection $\vdA \times_\vdC \vdB \to \vdA$ is a pushfibration with pushforwards:
\[
\pushone{\alpha}{(p_1,q_1),\dots,(p_n,q_n)} := (\alpha, \pushone{F(\alpha)}{q_1,...,q_n})
\]
\end{prop}
\begin{proof}
First, it is always strict since it is a projection.

Assume that we have $(p_i,q_i)$ horizontal morphisms in $\vdA \times_\vdC \vdB$ and $\alpha \colon p_1,\dots, p_n \Rightarrow q$ in \vdA.
Then, since $P$ is a pushfibration there is an opcartesian cell $F(\alpha)^\ast$ over $F(\alpha)$.
\resizebox{\hsize}{!}{
% https://q.uiver.app/?q=WzAsMjEsWzAsNCwiQV8wIl0sWzIsNCwiQV8xIl0sWzMsNCwiXFxkb3RzIl0sWzQsNCwiQV97bi0xfSJdLFs2LDQsIkFfbiJdLFswLDYsIkFfMCJdLFs2LDYsIkFfbiJdLFs4LDQsIkYoQV8wKSJdLFsxMCw0LCJGKEFfMSkiXSxbMTEsNCwiXFxkb3RzIl0sWzEyLDQsIkYoQV97bi0xfSkiXSxbMTQsNCwiRihBX24pIl0sWzgsNiwiRihBXzApIl0sWzE0LDYsIkYoQV9uKSJdLFs4LDAsIkJfMCJdLFsxMCwwLCJCXzEiXSxbMTEsMCwiXFxkb3RzIl0sWzEyLDAsIkJfe24tMX0iXSxbMTQsMCwiQl9uIl0sWzgsMiwiQl8wIl0sWzE0LDIsIkJfbiJdLFswLDEsInBfMSIsMV0sWzMsNCwicF9uIiwxXSxbNSw2LCJxIiwxXSxbMCw1LCIiLDEseyJsZXZlbCI6Miwic3R5bGUiOnsiaGVhZCI6eyJuYW1lIjoibm9uZSJ9fX1dLFs0LDYsIiIsMSx7ImxldmVsIjoyLCJzdHlsZSI6eyJoZWFkIjp7Im5hbWUiOiJub25lIn19fV0sWzcsOCwiRihwXzEpIiwxXSxbMTAsMTEsIkYocF9uKSIsMV0sWzEyLDEzLCJGKHEpIiwxXSxbNywxMiwiIiwxLHsibGV2ZWwiOjIsInN0eWxlIjp7ImhlYWQiOnsibmFtZSI6Im5vbmUifX19XSxbMTEsMTMsIiIsMSx7ImxldmVsIjoyLCJzdHlsZSI6eyJoZWFkIjp7Im5hbWUiOiJub25lIn19fV0sWzE0LDE1LCJyXzAiLDFdLFsxNywxOCwicl9uIiwxXSxbMTksMjAsIlxccHVzaG9uZXtGKFxcYWxwaGEpfXtyXzAsXFxkb3RzLHJfbn0iLDFdLFsxNCwxOSwiIiwxLHsibGV2ZWwiOjIsInN0eWxlIjp7ImhlYWQiOnsibmFtZSI6Im5vbmUifX19XSxbMTgsMjAsIiIsMSx7ImxldmVsIjoyLCJzdHlsZSI6eyJoZWFkIjp7Im5hbWUiOiJub25lIn19fV0sWzIsMjMsIlxcYWxwaGEiLDEseyJzaG9ydGVuIjp7InRhcmdldCI6MjB9fV0sWzksMjgsIkYoXFxhbHBoYSkiLDEseyJzaG9ydGVuIjp7InRhcmdldCI6MjB9fV0sWzE2LDMzLCJGKFxcYWxwaGEpXlxcYXN0IiwxLHsic2hvcnRlbiI6eyJ0YXJnZXQiOjIwfX1dXQ==
\begin{tikzcd}[ampersand replacement=\&]
	\&\&\&\&\&\&\&\& {B_0} \&\& {B_1} \& \dots \& {B_{n-1}} \&\& {B_n} \\
	\\
	\&\&\&\&\&\&\&\& {B_0} \&\&\&\&\&\& {B_n} \\
	\\
	{A_0} \&\& {A_1} \& \dots \& {A_{n-1}} \&\& {A_n} \&\& {F(A_0)} \&\& {F(A_1)} \& \dots \& {F(A_{n-1})} \&\& {F(A_n)} \\
	\\
	{A_0} \&\&\&\&\&\& {A_n} \&\& {F(A_0)} \&\&\&\&\&\& {F(A_n)}
	\arrow["{p_1}"{description}, from=5-1, to=5-3]
	\arrow["{p_n}"{description}, from=5-5, to=5-7]
	\arrow[""{name=0, anchor=center, inner sep=0}, "q"{description}, from=7-1, to=7-7]
	\arrow[Rightarrow, no head, from=5-1, to=7-1]
	\arrow[Rightarrow, no head, from=5-7, to=7-7]
	\arrow["{F(p_1)}"{description}, from=5-9, to=5-11]
	\arrow["{F(p_n)}"{description}, from=5-13, to=5-15]
	\arrow[""{name=1, anchor=center, inner sep=0}, "{F(q)}"{description}, from=7-9, to=7-15]
	\arrow[Rightarrow, no head, from=5-9, to=7-9]
	\arrow[Rightarrow, no head, from=5-15, to=7-15]
	\arrow["{r_0}"{description}, from=1-9, to=1-11]
	\arrow["{r_n}"{description}, from=1-13, to=1-15]
	\arrow[""{name=2, anchor=center, inner sep=0}, "{\pushone{F(\alpha)}{r_0,\dots,r_n}}"{description}, from=3-9, to=3-15]
	\arrow[Rightarrow, no head, from=1-9, to=3-9]
	\arrow[Rightarrow, no head, from=1-15, to=3-15]
	\arrow["\alpha"{description}, shorten >=7pt, Rightarrow, from=5-4, to=0]
	\arrow["{F(\alpha)}"{description}, shorten >=7pt, Rightarrow, from=5-12, to=1]
	\arrow["{F(\alpha)^\ast}"{description}, shorten >=7pt, Rightarrow, from=1-12, to=2]
\end{tikzcd}
}

We want to prove that $(\alpha, F(\alpha)^\ast)$ is an opcartesian cell relative to the projection functor.

Suppose that we have a cell in $\vdA \times_\vdC \vdB$ that lies over $\beta \circ \alpha$ in \vdA.
That means that it is of the form $(\beta \circ \alpha, \mathrm{M})$ with \[F(\beta \circ \alpha) = F(\beta) \circ F(\alpha) = P(\mathrm{M})\]
But since $F(\alpha^\ast)$ is opcartesian, we can uniquely factor $P(\mathrm{M})$ through it followed by a cell $\mathrm{M}/F(\alpha)^\ast$ lying over $F(\beta)$.
So $(\beta \circ \alpha, \mathrm{M}) = (\beta, P(\mathrm{M}/F(\alpha)^\ast) \circ (\alpha, \mathrm{M})$.
\end{proof}

In particular, when both \vdA and \vdC are representable, we get that \vdB and $\vdA \times_\vdC \vdB$ are with composites given by pushing along the universal cells. 
This recovers what we have done above.

Now, we want to find out the condition for which $\vdA \times_\vdC \vdB$ is (the underlying vdc of) a bicategory.
If \vdA is then the vertical morphisms in $\vdA \times_\vdC \vdB$ are of the form $(\id_{A},g)$ where $P(g) = \id_A$.
For it to be a bicategory, we will want that the only vertical morphisms in \vdB that lie over identities are identities.
Given an object $C \in \vdC$ we will call the vertical fibre over $C$ and write $P^v_C$ for the (categorical) fibre over the functor between the vertical categories $P \colon \vdV(\vdB) \to \vdV(\vdC)$.
In other words, it is the category consisting of objects above $C$ and vertical morphisms above $\id_C$.

\begin{definition}
For a functor of vdc $F \colon \vdA \to \vdB$ and an object $B$ in \vdB, the \defin{vertical fibre} $F_B^v$ is the fibre of the vertical functor associated to $F$ over $B$.
\end{definition}

\begin{prop}
\label{prop:vdc-pull-bicat}
If \vdA is (the underlying vdc of) a bicategory, \vdC is representable, $P$ is a pushfibration and all its vertical fibres are discrete, then $\vdA \times_\vdC \vdB$ is a bicategory.
\end{prop}

Now suppose that \vdA is a category, so its only cells are identities.
Then the cells in $\vdA \times_\vdC \vdB$ are of the form $(\id_p,\beta)$ where $\beta$ is a cell that lies over the identity.

\begin{definition}
Given a functor of vdc $F \colon \vdA \to \vdB$ and a horizontal morphism $p \colon B_0 \to B_1$ in \vdB, the \defin{fibre over $p$} is the category $F_p$ whose:
\begin{itemize}
\item objects are horizontal morphisms $q \colon A_0 \to A_1$ such that $F(q) = F(p)$
\item morphisms are cells such that $F(\beta) = \id_p$
% https://q.uiver.app/?q=WzAsNCxbMCwwLCJBXzAiXSxbMCwyLCJBXzAnIl0sWzIsMCwiQV8xIl0sWzIsMiwiQl8xJyJdLFswLDEsImZfMCIsMV0sWzIsMywiZl8xIiwxXSxbMCwyLCJxIiwxXSxbMSwzLCJxJyIsMV0sWzYsNywiXFxiZXRhIiwxLHsic2hvcnRlbiI6eyJzb3VyY2UiOjIwLCJ0YXJnZXQiOjIwfX1dXQ==
\[\begin{tikzcd}[ampersand replacement=\&]
	{A_0} \&\& {A_1} \\
	\\
	{A_0'} \&\& {B_1'}
	\arrow["{f_0}"{description}, from=1-1, to=3-1]
	\arrow["{f_1}"{description}, from=1-3, to=3-3]
	\arrow[""{name=0, anchor=center, inner sep=0}, "q"{description}, from=1-1, to=1-3]
	\arrow[""{name=1, anchor=center, inner sep=0}, "{q'}"{description}, from=3-1, to=3-3]
	\arrow["\beta"{description}, shorten <=9pt, shorten >=9pt, Rightarrow, from=0, to=1]
\end{tikzcd}\]
\end{itemize}
\end{definition}

\begin{prop}
\label{prop:vdc-pull-cat}
If \vdA is (the underlying horizontal vdc of) a category, \vdC is representable, $P$ is a pushfibration and all its vertical fibres and its fibres are discrete, then $\vdA \times_\vdC \vdB$ is a category.
\end{prop}

In the next chapter we will look at the Bénabou-Grothendieck construction from this perspective, where the pushfibrations $\vdB \to \vdC$ will be $\Dist_\ast \to \Dist$, $\Span_\ast \to \Span$, $\Cat_\ast \to \Cat$, $\Cat^\op_\ast \to \Cat^\op$ and $\Adj_\ast \to \Adj$ and the functors $F$ will be lax/normal/pseudo.

In the meanwhile, let us look at the example $\vdC^{\vdX}$.

\begin{example}
Given a vdc $\vdC$ and $\vdX \subseteq \vdC$, not necessarily a vdc, if $\vdC$ is representable and the horizontal morphisms of \vdX are closed under postcomposition by the composition cells of \vdC then $\vdC^{\vdX}$ is representable.
In the following we will assume that \vdC is representable and the horizontal morphisms of \vdX are closed under postcomposition by cells of \vdC where the vertical morphisms are identities.
So $\cU \colon \vdC^{\vdX} \to \vdC$ is a pushfibration and $\vdC^{\vdX}$ is representable.
Furthermore, since the vertical morphisms and cells of $\vdC^{\vdX}$ are those in \vdC, its vertical fibres and its fibres are discrete.
So for a category \vdA and a functor $F \colon \vdA \to \vdC$, $\vdA \times_{\vdC} \vdC^{\vdX}$ is a category.
\end{example}

\part{Bénabou-Grothendieck correspondences}

In addition to generalising the inverse image, a reason why fibred categories are really useful is that they correspond to indexed categories.
To understand this, let us first take a look at how to define an indexed set.

Let us fix a set $I$ of indices.
An indexed set $(X_i)_{i \in I}$ can be described as a function $f_X \colon I \to \Set$ where $\Set$ is the (large) set of sets. $f_X$ assigns to each index $i \in I$ the set $f(i) := X_i$.
There are some inconveniences with this presentation.
First, we have to deal with some size issues to be able to talk about a set of sets.
One way to deal with this is by considering the index set as a functor $f_X \colon I \to \Set$ instead.
Here we implicitly consider the set $I$ as a discrete category.
But, this requires to go up one level in the hierarchy of $n$-categories.
Instead, one could define an indexed set $(X_i)_{i \in I}$ by staying at a set-theoretic level.
The idea is to build the set $X := \coprod_{i \in I} X_i$.
Then to keep track of the different sets, we define a function $p_X \colon X \to I$ that sends an object $x \in X$ to the index $i$ for which $x \in X_i$.
That way, we can recover $(X_i)_i := (p_X^{-1}(i))_i$.
Conversely, any function $p \colon X \to I$ between arbitrary sets $X,I$ defines an indexed set by $X_i := p^{-1}(i)$.
This gives an equivalence between functions $p \colon X \to I$ and functors $F \colon I \to \Set$.
 
This can be categorified.
Given a category $\cB$, by a category indexed over $\cB$ we mean a family of categories $(X_b)_{b \in \cB}$ such that any map $f \colon b \to b'$ in $\cB$ induces a reindexing functor $f^\ast \colon X_{b'} \to X_b$.
Furthermore, we want the reindexing to be pseudofunctorial.
The fact that we ask for the reindexing functor to go in the opposite direction can be justified by the fact that we want to collect indexed sets into a set-indexed category.
This is called the family of sets in \cite{Jacobs2001}. 
Let $\cB$ be $\Set$. 
To each set $I \in \Set$ we associate the category $X_I$ with objects $I$-indexed sets $(X_i)_{i \in I}$ and morphisms $(f_i)_{i \in I} \colon (X_i)_{i \in I} \to (Y_i)_{i \in I}$.
Now given a reindexing function $r \colon I \to J$ we get a functor $r^\ast \colon X_J \to X_I$ by sending $(X_j)_{j \in J}$ to $(X_{r_i})_{i \in I}$ and similarly for morphims.

So, a $\cB$-indexed category is a pseudofunctor $F \colon \cB^\op \to \Cat$.
Once again, to define indexed categories, we had to move one level up to 2-categories.
But, this can be turned into a categorical definition.
Given an indexed category, we can define a category $\int F$ that is obtained by adding all the categories $(F(b))_{b \in B}$ together, like we did for the indexed set.
This $\int F$ is called the Grothendieck construction.
It comes equipped with a functor $\int F \to \cB$.
However, not all functors arise in this way.
In fact, those that do are the fibrations, where the pullback functor correspond to the reindexing functor.
This gives a correspondence between fibrations $p \colon \cE \to \cB$ and pseudofunctor $F \colon \cB^\op \to \Cat$ that we will call the Grothendieck correspondence.
If we had chosen to define indexed categories such that the reindexing functor arises covariantly instead of contravariantly, this would have given us a pseudofunctor $F \colon \cB \to \Cat$.
These are in correspondence with opfibred categories.
When considering mere functors $p \colon \cE \to \cB$, we still get a family of categories $(p^{-1}(b))_{b \in \cB}$.
But instead of getting reindexing functor from a map $r \colon b \to b'$ we get a reindexing distributor $\delta(r) \colon (p^{-1}(b')) \xto p^{-1}(b)$.
Furthermore, this reindexing is not pseudofunctorial but lax normal functorial, i.e., the identity is preserves up to isomorphism but the associativity only in a lax way.
So we get a correspondence between functors $p \colon \cE \to \cB$ and lax normal functors $F \colon \cB \to \Dist$ that we call the B\'{e}nabou-Grothendieck correspondence since it generalises the Grothendieck correspondence and was originally described by B\'{e}nabou, see \cite{Benabou2000}.
We get back the covariant (resp. contravariant) Grothendieck correspondence by noticing that a functor is a fibration (resp. bifibration) iff the reindexing distributors $\delta(r) \colon (p^{-1}(b')) \xto p^{-1}(b)$ are representable by a functor $r^\ast$ (resp. corepresentable by a functor $r_\ast$) in a pseudofunctorial way.
Notice that we can put both Grothendieck correspondences together to get that a functor is a bifibration iff $\delta(r)$ is both representable by $r^\ast$ and $r_\ast$.
This means that $r_\ast \dashv r^\ast$ and so we get another correspondence, this time between bifibrations $p \colon \cE \to \cB$ and pseudofunctors $F \colon \cB \to \Adj$.

In chapter \ref{ch:bg-cat}, we first give an abstraction description of the Grothendieck construction making use of the notion of pushfibration of virtual double categories.
This can be used to describe variations of the B\'{e}nabou-Grothendieck correspondences and we illustrate it on some examples.
Then we give a polycategorical version of the B\'{e}nabou-Grothendieck correspondences.
More specifically, we show that functors of polycategories $p \colon \cE \to \cB$ correspond to lax normal functors of weak 2-polycategories $F \colon \cB \to \Dist$.
Then, we show that bifibrations of polycategories correspond to pseudofunctors $F \colon \cB \to \MVar$ where $\MVar$ is the 2-polycategory of categories and multivariable adjunctions.
Notice that we directly jumped from functors to bifibrations without characterising fibrations and opfibrations first.
This is because fibrations (resp. opfibrations) will correspond to functors such that for any polymap in the base $r \colon b_1,\dots, b_m \to b_1', \dots, b_n'$ the reindexing distributor $\delta(r) \colon p^{-1}(b_1'), \dots, p^{-1}(b_n') \xto p^{-1}(b_1),\dots, p^{-1}(b_m)$ is representable in each of its outputs (resp. each of its inputs).
But these do not compose.
For example, consider distributors $d \colon A \xto B$ and $d' \colon B, B' \xto C$ both representable in their inputs.
For $d'\circ d \colon A,B' \xto C$ to be representable in $B'$ we would need a functor $A^\op, C \to B'$.
Since $d'$ is representable in $B'$ we have a functor $B^\op, C \to B$.
But $d$ is only representable in $A$, given us a functor $B \to A$ to represent it but not one $A \to B$.
So we cannot get a sub-2-polycategory of \Dist{} by restricting to the distributors representable only in the inputs or only in the outputs.

Lastly, we should emphasise that the results in chapter \ref{ch:bg-poly}, and in particular the polycategorical Grothendieck correspondences, are conditioned on having a theory of weak 2-polycategories.
To the extent of our knowledge such a theory has not been carefully worked out yet.
We leave it as future work to craft this theory.
Meanwhile we will describe the properties we assume to hold for weak 2-polycategories after recalling the usual notion of categorical Grothendieck correspondences. 

The Grothendieck construction has been introduced by Grothendieck in \cite{Grothendieck1960} and its extension mentioned above by B\'{e}nabou (see \cite{Benabou2000}).
The fact that the Grothendieck construction can be obtained as a strict 2-pullback of bicategories via the 2-functor $\Cat_\ast \to \Cat$ seems to be a folklore result.
It is for example mentioned on the nlab page about the Grothendieck construction.
We couldn't find a reference in the literature.
Considering a pullback of virtual double categories instead to encompass versions of the Bénabou-Grothendieck construction is new.
The chapter on the polycategorical Bénabou-Grothendieck construction is also original work.

\chapter{The B\'{e}nabou-Grothendieck construction}
\label{ch:bg-cat}
%\section{The Bénabou-Grothendieck correspondence}
%
%By ``Grothendieck correspondence'', we refer to the equivalence between fibrations and indexed categories, as well as a range of several other similar correspondences (the first was originally described by B\'{e}nabou \cite{Benabou2000}):
%\begin{itemize}
%\item Functor $\mathcal{E} \to \mathcal{B}$ $\longleftrightarrow$ lax normal functor $\mathcal{B}^\op \to \Dist$
%\item Fibration $\mathcal{E} \to \mathcal{B}$ $\longleftrightarrow$ pseudofunctor $\mathcal{B}^\op \to \Cat$
%\item Opfibration $\mathcal{E} \to \mathcal{B}$ $\longleftrightarrow$ pseudofunctor $\mathcal{B}^\op \to \Catop$ (or equivalently $\mathcal{B} \to \Cat$)
%\item Bifibration $\mathcal{E} \to \mathcal{B}$ $\longleftrightarrow$ pseudofunctor $\mathcal{B}^\op \to \Adj$
%\end{itemize}
%Here $\Dist$ is the bicategory whose objects are small categories and 1-cells $A \xslashedrightarrow{} B$ are \emph{distributors}, i.e., functors $A \times B^\op \to \Set$, with 2-cells given by natural transformations.
%While $\Cat$ (respectively $\Adj$) is the (strict) 2-category whose objects are small categories, 1-cells are functors (resp.~adjunctions), and 2-cells are natural transformations.
%Notice that $\Cat$, $\Catop$ and $\Adj$ are all subbicategories of $\Dist$ corresponding to distributors $A \xslashedrightarrow{} B$ that are representable in $A$, in $B$, and in both, respectively.
%
%\section{The Bénabou-Grothendieck correspondence, a 2-equivalence}

\section{The B-G construction as a pullback}

In this section we want to recover the Bénabou construction as a pullback:
% https://q.uiver.app/?q=WzAsNCxbMCwwLCJcXGludCBGIl0sWzIsMCwiXFxEaXN0X1xcYXN0Il0sWzIsMiwiXFxEaXN0Il0sWzAsMiwiXFxjQiJdLFswLDNdLFszLDIsIlxcY0YiLDJdLFsxLDIsIlxcY1UiXSxbMCwxXSxbMCwyLCIiLDEseyJzdHlsZSI6eyJuYW1lIjoiY29ybmVyIn19XV0=
\[\begin{tikzcd}
	{\int F} && {\Dist_\ast} \\
	\\
	\cB && \Dist
	\arrow[from=1-1, to=3-1]
	\arrow["\cF"', from=3-1, to=3-3]
	\arrow["\cU", from=1-3, to=3-3]
	\arrow[from=1-1, to=1-3]
	\arrow["\lrcorner"{anchor=center, pos=0.125}, draw=none, from=1-1, to=3-3]
\end{tikzcd}\]

First, let define the virtual double categories involved.

\begin{definition}
\Dist{} is the vdc whose:
\begin{itemize}
\item objects are categories
\item vertical morphisms are functors,
\item horizontal morphisms are distributors $p \colon A \to B$, i.e. functors $p \colon A^\op \times B \to \Set$
\item a cell 
% https://q.uiver.app/?q=WzAsNyxbMCwwLCJBXzAiXSxbMiwwLCJBXzEiXSxbNCwwLCJBX3tuLTF9Il0sWzYsMCwiQV9uIl0sWzAsMiwiQl8wIl0sWzYsMiwiQl8xIl0sWzMsMCwiXFxkb3RzIl0sWzAsMSwicF8xIiwxXSxbMiwzLCJwX24iLDFdLFs0LDUsInEiLDFdLFswLDQsImZfMCIsMV0sWzMsNSwiZl8xIiwxXSxbNiw5LCJcXGFscGhhIiwxLHsic2hvcnRlbiI6eyJ0YXJnZXQiOjIwfX1dXQ==
\[\begin{tikzcd}[ampersand replacement=\&]
	{A_0} \&\& {A_1} \& \dots \& {A_{n-1}} \&\& {A_n} \\
	\\
	{B_0} \&\&\&\&\&\& {B_1}
	\arrow["{p_1}"{description}, from=1-1, to=1-3]
	\arrow["{p_n}"{description}, from=1-5, to=1-7]
	\arrow[""{name=0, anchor=center, inner sep=0}, "q"{description}, from=3-1, to=3-7]
	\arrow["{F_0}"{description}, from=1-1, to=3-1]
	\arrow["{F_1}"{description}, from=1-7, to=3-7]
	\arrow["\alpha"{description}, shorten >=7pt, Rightarrow, from=1-4, to=0]
\end{tikzcd}\]
is a natural transformation \[\alpha \colon \int^{a_i \in A_i} p_1(-,a_1) \times p_2(a_1,a_2) \times \dots \times p_{n-1}(a_{n-2},a_{n-1}) \times p_n(a_{n-1},=) \to q(F_0(-),F_1(=))\]
\end{itemize}
\end{definition}

\begin{rk}
Alternatively, we could define a cell as a family of maps $p_1(a_0,a_1) \times \dots \times p_n(a_{n-1},a_n) \to q(F_0(a_0),F_1(a_n))$ which are natural in $a_0,a_n$ and dinatural in the other $a_i$.
It would have the benefice of not requiring to take the coend of $p_i$ which means less coherence to check.
It also makes for a more natural definition of cells in $\Dist_\ast$.
\end{rk}

We will sometimes write distributors with a dashed arrow $p \colon A \xto B$ to mark the difference with functors.
Any functor $F \colon A \to B$ gives rise to two distributors one contravariantly and the other covariantly: $B(F(-),=) \colon A \xto B$ and $B(-,F(=)) \colon B \xto A$.
Such a functor is said to represent the distributor.
A distributor $p \colon A \xto B$ is said to be representable in $A$ if there exists a functor $F \colon A \to B$ such that $p(-,=) \simeq B(F(-),=)$ naturally and representable in $B$ if there is a functor $G \colon B \to A$ such that $p(-,=) \simeq A(-,G(=))$ naturally.
In particular, if $p$ is representable both in $A$ and $B$ it is represented by a pair of adjoint functor $F \dashv G$.
There is a distributor $A(-,=) \colon A \xto A$, it is an horizontal identity in \vdC.

For one-object categories considered as the delooping of monoids, a distributor $p \colon \cB(A) \xto \cB(B)$ is a set equipped with a actions by $A$ and $B$, so it is a $(A,B)$-bimodule.
We can think of a distributor as generalising this by assigning to each pair of objects a set that is acted upon by composition of the morphisms in $A$ and $B$.
A cell is given by a natural transformation \[\alpha : \int^{a_i \in A_i} p_1(-,a_1) \times p_2(a_1,a_2)\times\dots\times p_n(a_{n-1},=) \to q(F_0(-),F_1(=))\]that is, a family of functions \[\alpha_{a_0,a_n} \colon \int^{a_i \in A_i} p_1(a_0,a_1) \times p_2(a_1,a_2)\times\dots\times p_n(a_{n-1},a_n) \to q(F_0(a_0),F_1(a_n))\] naturally.
In other words, they are functions such that \[\alpha(\varphi_1,\dots, \varphi_{i-1} \cdot f, \varphi_i, \dots, \varphi_n) = \alpha(\varphi_1,\dots, \varphi_{i-1}, f \cdot \varphi_i, \dots, \varphi_n)\] where $\cdot f$ is defined as $p_{i_1}(\id,f)$ and $f \cdot$ as $p_i(f,\id)$.
So the coend is there to equate the left and right action, similar to what the reflexive coequaliser does for bimodules.
It can be constructed explictly as a coequaliser of a coproduct, e.g. in the binary case:
% https://q.uiver.app/?q=WzAsMyxbMCwwLCJcXGNvcHJvZFxcbGltaXRzX3tmIFxcY29sb24gYiBcXHRvIGInfSBwKGEsYilcXHRpbWVzIHAoYicsYykiXSxbMywwLCJcXGNvcHJvZFxcbGltaXRzX3tiXFxpbiBCfSBwKGEsYikgXFx0aW1lcyBwKGIsYykiXSxbNSwwLCJcXGludF57YiBcXGluIEJ9cChhLGIpIFxcdGltZXMgcChiLGMpIl0sWzAsMSwiXFxpb3RhX2InKFxcaWRfe3AoYSxiKX0gXFx0aW1lcyBwKGYsXFxpZF9jKSkiLDAseyJvZmZzZXQiOi0yfV0sWzAsMSwiXFxpb3RhX2IocChcXGlkX2EsZilcXHRpbWVzXFxpZF97cChiJyxjKX0iLDIseyJvZmZzZXQiOjJ9XSxbMSwyXV0=
\[\begin{tikzcd}[ampersand replacement=\&]
	{\coprod\limits_{f \colon b \to b'} p(a,b)\times p(b',c)} \&\&\& {\coprod\limits_{b\in B} p(a,b) \times p(b,c)} \&\& {\int^{b \in B}p(a,b) \times p(b,c)}
	\arrow["{\iota_b'(\id_{p(a,b)} \times p(f,\id_c))}", shift left=2, from=1-1, to=1-4]
	\arrow["{\iota_b(p(\id_a,f)\times\id_{p(b',c)}}"', shift right=2, from=1-1, to=1-4]
	\arrow[from=1-4, to=1-6]
\end{tikzcd}\]
where $\iota_b$ is the embedding into the $b$-th component of the coproduct.
The vdc \Dist{} is representable with composite given by the coend formula above.

\begin{definition}
$\Dist_\ast$ is the vdc whose:
\begin{itemize}
\item objects are pointed categories: pairs $(A,a)$ of a category $A$ and an object $a \in \Ob{A}$
\item vertical morphisms are functors that preserves the points
\item horizontal morphisms $p \colon (A,a) \to (B,b)$ are pointed distributors, i.e. pairs $(p,\varphi)$ of a distributor $p \colon A \to B$ and an element $\varphi \in p(a,b)$
\item cells are natural transformations:
% https://q.uiver.app/?q=WzAsNyxbMCwwLCIoQV8wLGFfMCkiXSxbMiwwLCIoQV8xLGFfMSkiXSxbMywwLCJcXGRvdHMiXSxbNCwwLCIoQV97bi0xfSxhX3tuLTF9KSJdLFs2LDAsIihBX24sYV9uKSJdLFswLDIsIihCXzAsYl8wKSJdLFs2LDIsIihCXzEsYl8xKSJdLFswLDEsIihwXzEsXFx2YXJwaGlfMSkiLDFdLFszLDQsIihwX24sXFx2YXJwaGlfbikiLDFdLFswLDUsImZfMCIsMV0sWzQsNiwiZl8xIiwxXSxbNSw2LCIocSxcXHBzaSkiLDFdLFsyLDExLCJcXGFscGhhIiwxLHsic2hvcnRlbiI6eyJ0YXJnZXQiOjIwfX1dXQ==
\[\begin{tikzcd}[ampersand replacement=\&]
	{(A_0,a_0)} \&\& {(A_1,a_1)} \& \dots \& {(A_{n-1},a_{n-1})} \&\& {(A_n,a_n)} \\
	\\
	{(B_0,b_0)} \&\&\&\&\&\& {(B_1,b_1)}
	\arrow["{(p_1,\varphi_1)}"{description}, from=1-1, to=1-3]
	\arrow["{(p_n,\varphi_n)}"{description}, from=1-5, to=1-7]
	\arrow["{f_0}"{description}, from=1-1, to=3-1]
	\arrow["{f_1}"{description}, from=1-7, to=3-7]
	\arrow[""{name=0, anchor=center, inner sep=0}, "{(q,\psi)}"{description}, from=3-1, to=3-7]
	\arrow["\alpha"{description}, shorten >=8pt, Rightarrow, from=1-4, to=0]
\end{tikzcd}\]
such that $\alpha_{a_0,a_n}(\varphi_n \bullet \dots \bullet \varphi_1) = \psi$ where $\varphi_n, \bullet \dots \bullet \varphi_1$ is the image of $(\varphi_1,\dots,\varphi_n)$ through the embedding $p_1(a_0,a_1) \times \dots \times p_n(a_{n-1},a_n) \hookrightarrow \int^{x_i} p_1(a_0,x_1) \times \dots \times p_{n-1}(x_{n-1},a_n)$.
\end{itemize}
\end{definition}

$\Dist_\ast$ is representable with composite given by \[(\int^{x_i} p_1(a_0,x_1) \times \dots \times p_n(x_{n-1},a_n), \varphi_n \bullet \dots \bullet  \varphi_1)\]

\begin{prop}
The forgetful functor $\Dist_\ast \to \Dist$ is a pushfibration.
\end{prop}
\begin{proof}
Since it only forgets structure it is strict.

Furthermore, since both vdcs are representable, we just need to prove it for unary cells.

So consider the following data of a cell in \Dist{} together with points of its domain:
% https://q.uiver.app/?q=WzAsNCxbMCwwLCIoQSxhKSJdLFsyLDAsIihCLGIpIl0sWzAsMiwiQV8wIl0sWzIsMiwiQiJdLFswLDEsIihwLFxcdmFycGhpKSIsMV0sWzIsMywicSIsMV0sWzAsMiwiIiwxLHsibGV2ZWwiOjIsInN0eWxlIjp7ImhlYWQiOnsibmFtZSI6Im5vbmUifX19XSxbMSwzLCIiLDEseyJsZXZlbCI6Miwic3R5bGUiOnsiaGVhZCI6eyJuYW1lIjoibm9uZSJ9fX1dLFs0LDUsIlxcYWxwaGEiLDEseyJzaG9ydGVuIjp7InNvdXJjZSI6MjAsInRhcmdldCI6MjB9fV1d
\[\begin{tikzcd}[ampersand replacement=\&]
	{(A,a)} \&\& {(B,b)} \\
	\\
	{A} \&\& B
	\arrow[""{name=0, anchor=center, inner sep=0}, "{(p,\varphi)}"{description}, from=1-1, to=1-3]
	\arrow[""{name=1, anchor=center, inner sep=0}, "q"{description}, from=3-1, to=3-3]
	\arrow[Rightarrow, no head, from=1-1, to=3-1]
	\arrow[Rightarrow, no head, from=1-3, to=3-3]
	\arrow["\alpha"{description}, shorten <=9pt, shorten >=9pt, Rightarrow, from=0, to=1]
\end{tikzcd}\]

We can complete it into a cell in $\Dist_\ast$:
% https://q.uiver.app/?q=WzAsNCxbMCwwLCIoQSxhKSJdLFsyLDAsIihCLGIpIl0sWzAsMiwiKEEsYSkiXSxbMiwyLCIoQixiKSJdLFswLDEsIihwLFxcdmFycGhpKSIsMV0sWzIsMywiKHEsXFxhbHBoYV97YSxifShcXHZhcnBoaSkpIiwxXSxbMCwyLCIiLDEseyJsZXZlbCI6Miwic3R5bGUiOnsiaGVhZCI6eyJuYW1lIjoibm9uZSJ9fX1dLFsxLDMsIiIsMSx7ImxldmVsIjoyLCJzdHlsZSI6eyJoZWFkIjp7Im5hbWUiOiJub25lIn19fV0sWzQsNSwiXFxhbHBoYSIsMSx7InNob3J0ZW4iOnsic291cmNlIjoyMCwidGFyZ2V0IjoyMH19XV0=
\[\begin{tikzcd}[ampersand replacement=\&]
	{(A,a)} \&\& {(B,b)} \\
	\\
	{(A,a)} \&\& {(B,b)}
	\arrow[""{name=0, anchor=center, inner sep=0}, "{(p,\varphi)}"{description}, from=1-1, to=1-3]
	\arrow[""{name=1, anchor=center, inner sep=0}, "{(q,\alpha_{a,b}(\varphi))}"{description}, from=3-1, to=3-3]
	\arrow[Rightarrow, no head, from=1-1, to=3-1]
	\arrow[Rightarrow, no head, from=1-3, to=3-3]
	\arrow["\alpha"{description}, shorten <=9pt, shorten >=9pt, Rightarrow, from=0, to=1]
\end{tikzcd}\]

Now let prove that it is opcartesian.

Given a cell:
% https://q.uiver.app/?q=WzAsMTAsWzIsMCwiKEEsYSkiXSxbNCwwLCIoQixiKSJdLFsyLDIsIihBLGEpIl0sWzQsMiwiKEIsYikiXSxbMCwwLCIoQV8wLGFfMCkiXSxbNiwwLCIoQl8xLGJfMSkiXSxbMCw0LCIoQ18wLGNfMCkiXSxbNiw0LCIoQ18xLGNfMSkiXSxbMCwyLCIoQV8wLGFfMCkiXSxbNiwyLCIoQl8xLGJfMSkiXSxbMCwxLCIocCxcXHZhcnBoaSkiLDFdLFsyLDMsIihxLFxcYWxwaGFfe2EsYn0oXFx2YXJwaGkpKSIsMV0sWzAsMiwiIiwxLHsibGV2ZWwiOjIsInN0eWxlIjp7ImhlYWQiOnsibmFtZSI6Im5vbmUifX19XSxbMSwzLCIiLDEseyJsZXZlbCI6Miwic3R5bGUiOnsiaGVhZCI6eyJuYW1lIjoibm9uZSJ9fX1dLFs0LDAsIihwXzAsXFx2YXJwaGlfMCkiLDFdLFsxLDUsIihwXzEsXFx2YXJwaGlfMSkiLDFdLFs2LDcsIihyLFxccmhvKSIsMV0sWzQsOCwiIiwxLHsibGV2ZWwiOjIsInN0eWxlIjp7ImhlYWQiOnsibmFtZSI6Im5vbmUifX19XSxbOCwyLCIocF8wLFxcdmFycGhpXzApIiwxXSxbMyw5LCIocF8xLFxcdmFycGhpXzEpIiwxXSxbNSw5LCIiLDEseyJsZXZlbCI6Miwic3R5bGUiOnsiaGVhZCI6eyJuYW1lIjoibm9uZSJ9fX1dLFs4LDYsIkZfMCIsMV0sWzksNywiRl8xIiwxXSxbMTAsMTEsIlxcYWxwaGEiLDEseyJzaG9ydGVuIjp7InNvdXJjZSI6MjAsInRhcmdldCI6MjB9fV0sWzExLDE2LCJcXGJldGEiLDEseyJzaG9ydGVuIjp7InNvdXJjZSI6MjAsInRhcmdldCI6MjB9fV1d
\[\begin{tikzcd}[ampersand replacement=\&]
	{(A_0,a_0)} \&\& {(A,a)} \&\& {(B,b)} \&\& {(B_1,b_1)} \\
	\\
	{(A_0,a_0)} \&\& {(A,a)} \&\& {(B,b)} \&\& {(B_1,b_1)} \\
	\\
	{(C_0,c_0)} \&\&\&\&\&\& {(C_1,c_1)}
	\arrow[""{name=0, anchor=center, inner sep=0}, "{(p,\varphi)}"{description}, from=1-3, to=1-5]
	\arrow[""{name=1, anchor=center, inner sep=0}, "{(q,\alpha_{a,b}(\varphi))}"{description}, from=3-3, to=3-5]
	\arrow[Rightarrow, no head, from=1-3, to=3-3]
	\arrow[Rightarrow, no head, from=1-5, to=3-5]
	\arrow["{(p_0,\varphi_0)}"{description}, from=1-1, to=1-3]
	\arrow["{(p_1,\varphi_1)}"{description}, from=1-5, to=1-7]
	\arrow[""{name=2, anchor=center, inner sep=0}, "{(r,\rho)}"{description}, from=5-1, to=5-7]
	\arrow[Rightarrow, no head, from=1-1, to=3-1]
	\arrow["{(p_0,\varphi_0)}"{description}, from=3-1, to=3-3]
	\arrow["{(p_1,\varphi_1)}"{description}, from=3-5, to=3-7]
	\arrow[Rightarrow, no head, from=1-7, to=3-7]
	\arrow["{F_0}"{description}, from=3-1, to=5-1]
	\arrow["{F_1}"{description}, from=3-7, to=5-7]
	\arrow["\alpha"{description}, shorten <=9pt, shorten >=9pt, Rightarrow, from=0, to=1]
	\arrow["\beta"{description}, shorten <=9pt, shorten >=9pt, Rightarrow, from=1, to=2]
\end{tikzcd}\]
such that $\beta \circ \alpha$ is in $\Dist_\ast$, we want to prove that $\beta$ is in $\Dist_\ast$.

First $F_0(a_0) = c_0$ and $F_1(b_1) = c_1$ since $\beta \circ \alpha$ is in $\Dist_\ast$.

Now we want to prove that $\beta_{a_0,b_1}(\alpha_{a,b}(\varphi)) = \rho$.
But since $\beta \circ \in \Dist_\ast$ we have $(\beta \circ \alpha)_{a_0,b_1}(\varphi)  = \rho$ and then we use the definition of composition.
\end{proof}

Since the vertical morphisms and the cells of $\Dist_\ast$ consists of vertical morphisms and cells in \Dist, the vertical fibres and the fibres of the forgetful functor are discrete.

Now let consider a functor of vdc $F \colon \vdB \to \Dist$.
We can form the pullback $\vdB \times_\Dist \Dist_\ast$.
It is given by:
\begin{itemize}
\item Objects $(B\in \Ob{\vdB}, b \in F(B))$
\item vertical morphisms $f \colon (B,b) \to (B',b')$ are vertical morphisms $f \colon B \to B'$ in \vdB such that $F(f)(b) = b'$
\item horizontal morphisms $(p, \varphi) \colon (B_0,b_0) \to (B_1,b_1)$ are given by pairs of a horizontal morphism $p \colon B_0 \to B_1$ in \vdB and an element $\varphi \in F(p)(b_0,b_1)$
\item cells the ones in \vdB such that $F(\alpha)_{b_0,b_n}(\varphi_n \bullet \dots \bullet \varphi_1) = \varphi'$.
\end{itemize}

\begin{definition}
We call the vdc $\vdB \times_\Dist \Dist_\ast$ the \defin{vdc of elements} of $F \colon \vdB \to \Dist$, or its \defin{Bénabou-Grothendieck construction} (B-G construction for short) and we denote it $\int F$.
\end{definition}

\begin{prop}
For any $F \colon \vdB \to \Dist$, $\int F \to \vdB$ is a pushfibration.
\end{prop}
\begin{proof}
Since $\Dist_\ast \to \Dist$ is a pushfibration.
\end{proof}

\begin{prop}
For any $F \colon \cB \to \Dist$ functor of bicategory, $\int F$ is a bicategory and $\int F \to \cB$ is a 2-pushfibration.
\end{prop}
\begin{proof}
By prop \ref{prop:vdc-pull-bicat}, since $\Dist_\ast \to \Dist$ is a pushfibration and its vertical fibres are discrete we get a bicategory $\int F$ and the projection is a 2-pushfibration of bicategories.
Furthermore, since we don't use the vertical morphisms of $\Dist$, we can forget them and consider the bicategory of categories and distributors.
\end{proof}

\begin{prop}
For any (lax) functor $F \colon \cB \to \Dist$ where \cB is a category, $\int F$ is a category and $\int F \to \cB$ a functor.
\end{prop}
\begin{proof}
By \ref{prop:vdc-pull-cat}, since the forgetful functor is a pushfibration and its vertical fibres and fibres are discrete.
\end{proof}

Explicitly, the category $\int F$ is given by:
\begin{itemize}
\item its objects are $(B \in \Ob{\cC}, b \in F(B))$
\item its morphisms are $(f \colon B \to B', \varphi \in F(p)(b,b'))$
\item its identities are $(\id_B, \eta_{b,b}(\id_b))$ where $\eta \colon F(A)(-,--) \Rightarrow F(\id_A)(-,--)$ is the unit of the lax functor $F$
\item its compositions are $F(f')\bullet F(f), \mu_{f,f'}(\varphi'\bullet\varphi) \in F(f' \circ f)(b,b''))$ where $F(f') \bullet F(f)$ is the composition of distributors given by the coend formula, $\varphi' \bullet \varphi$ is the image of $(\varphi, \varphi')$ by the inclusion in $(F(f')\bullet F(f))(b,b'')$ and $\mu \colon (F(f') \bullet F(f)) \Rightarrow F(f' \circ f)$ comes from the laxity of $F$
\end{itemize}

Note, however that this construction does not provide an inverse to the fibre construction.
This is because the fibre construction is normal.
If we restrict yourselves to lax normal functor into $\Dist$ the we get the usual Bénabou-Grothendieck construction.

\begin{rk}
Consider $\vdX \subseteq \Dist$ with cells all the cells in \Dist of the shape:
% https://q.uiver.app/?q=WzAsMyxbMSwwLCJcXG1hdGhiZnsxfSJdLFswLDEsIkEiXSxbMiwxLCJCIl0sWzAsMSwiYSIsMV0sWzEsMiwicCIsMV0sWzAsMiwiYiIsMV0sWzAsNCwiXFx2YXJwaGkiLDEseyJzaG9ydGVuIjp7InRhcmdldCI6MjB9fV1d
\[\begin{tikzcd}[ampersand replacement=\&]
	\& {\mathbf{1}} \\
	A \&\& B
	\arrow["a"{description}, from=1-2, to=2-1]
	\arrow[""{name=0, anchor=center, inner sep=0}, "p"{description}, from=2-1, to=2-3]
	\arrow["b"{description}, from=1-2, to=2-3]
	\arrow["\varphi"{description}, shorten >=3pt, Rightarrow, from=1-2, to=0]
\end{tikzcd}\]
A vertical morphism $a \colon 1 \to A$ in \vdX is an object in $A$ and a cell such as above is an element $\varphi \in p(a,b)$.
Then, $\Dist^{\vdX}$ is $\Dist_\ast$.
So what we did here follow could have been deduced from the more general case of $\vdC^{\vdX} \to \vdC$.
However, we found it more instructive to spell out the proof for this particular case. 

\end{rk}

\section{Variants of the B-G construction}

In the above, instead taking the pullback of $\Dist_\ast \to \Dist$ one can consider other functors.

\subsection{Discrete B-G correspondence}

Let $\Rel\ast$ be the vdc of pointed relations:
\begin{itemize}
\item its objects are pairs $(A,a)$ of a set and an element of it
\item its vertical morphisms are functions preserving the elements
\item its horizontal morphisms $R \colon (A,a) \to (B,b)$ are relations such that $a R b$
\item its cells are cells in \Rel
\end{itemize}

$\Rel_\ast$ is representable with horizontal composition given by the one of \Rel.

\begin{prop}
The forgetful functor $\Rel_\ast \to \Rel$ is a pushfibration.
\end{prop}
\begin{proof}
We start with the following data.
% https://q.uiver.app/?q=WzAsNyxbMCwyLCJBXzAiXSxbNiwyLCJBX24iXSxbMCwwLCIoQV8wLGFfMCkiXSxbMiwwLCIoQV8xLGFfMSkiXSxbNCwwLCIoQV97bi0xfSxhX3tuLTF9KSJdLFs2LDAsIihBX24sYV9uKSJdLFszLDAsIlxcZG90cyJdLFswLDEsIlMiLDFdLFsyLDMsIlJfMSIsMV0sWzQsNSwiUl9uIiwxXSxbMiwwLCIiLDEseyJsZXZlbCI6Miwic3R5bGUiOnsiaGVhZCI6eyJuYW1lIjoibm9uZSJ9fX1dLFs1LDEsIiIsMSx7ImxldmVsIjoyLCJzdHlsZSI6eyJoZWFkIjp7Im5hbWUiOiJub25lIn19fV0sWzYsNywiXFxhbHBoYSIsMSx7InNob3J0ZW4iOnsidGFyZ2V0IjoyMH19XV0=
\[\begin{tikzcd}[ampersand replacement=\&]
	{(A_0,a_0)} \&\& {(A_1,a_1)} \& \dots \& {(A_{n-1},a_{n-1})} \&\& {(A_n,a_n)} \\
	\\
	{A_0} \&\&\&\&\&\& {A_n}
	\arrow[""{name=0, anchor=center, inner sep=0}, "S"{description}, from=3-1, to=3-7]
	\arrow["{R_1}"{description}, from=1-1, to=1-3]
	\arrow["{R_n}"{description}, from=1-5, to=1-7]
	\arrow[Rightarrow, no head, from=1-1, to=3-1]
	\arrow[Rightarrow, no head, from=1-7, to=3-7]
	\arrow["\alpha"{description}, shorten >=8pt, Rightarrow, from=1-4, to=0]
\end{tikzcd}\]
Since there is at most one cell in $\Rel_\ast$ over a cell in \Rel, we just need to prove that the codomain lift to $\Rel_\ast$, i.e. that $a_0 S a_n$ which is true because of the definition of cells in \Rel.
The fact that it is opcartesian follows from a similar argument.
\end{proof}

So we can define B-G construction for functors $F \colon \vdB \to \Rel$.
In the categorical case it gives a correspondence between functors whose fibres are posetal and functors $F \colon \vdB \to \Rel$.
Like in the usual B-G construction, pullbacks and pushforwards correspond to representability or the relations, i.e. that the relations are the graph of functions.
We call fibrations and opfibrations that whose fibres are discrete, discrete fibrations and discrete opfibrations.
They correspond under the discrete Grothendieck correspondence to presheaves and copresheaves.

\subsection{B-G in Span}

Let \cC be a category with finite limits.
There is a vdc $\Span_\ast(\cC)$ of pointed spans.
The objects, morphisms and cells are summed up in the following diagram:
% https://q.uiver.app/?q=WzAsMTEsWzAsMCwiKEFfMCxhXzApIl0sWzIsMCwiKFBfMSxwXzEpIl0sWzQsMCwiKEFfMSxhXzEpIl0sWzUsMCwiXFxkb3RzIl0sWzgsMCwiKFBfbixwX24pIl0sWzYsMCwiKEFfe24tMX0sYV97bl8xfSkiXSxbMTAsMCwiKEFfbixhX24pIl0sWzAsMywiKEJfMSxiXzEpIl0sWzUsMywiKFEscSkiXSxbMTAsMywiKEJfMSxiXzEpIl0sWzUsMSwiKFBfMSBcXHRpbWVzX3tBXzF9IFxcZG90cyBcXHRpbWVzX3tBX3tuLTF9fVBfbiwocF8xLFxcZG90cyxwX24pKSJdLFsxLDBdLFsxLDJdLFs0LDVdLFs0LDZdLFswLDcsImZfMCIsMV0sWzgsN10sWzgsOV0sWzYsOSwiZl8xIiwxXSxbMTAsMV0sWzEwLDRdLFsxMCw4LCJcXGFscGhhIiwxXSxbMTAsMywiIiwxLHsic3R5bGUiOnsibmFtZSI6ImNvcm5lciJ9fV1d

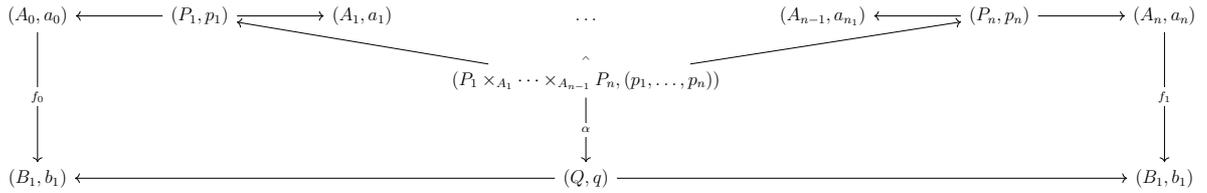
\resizebox{\hsize}{!}{
\begin{tikzcd}[ampersand replacement=\&]
	{(A_0,a_0)} \&\& {(P_1,p_1)} \&\& {(A_1,a_1)} \& \dots \& {(A_{n-1},a_{n_1})} \&\& {(P_n,p_n)} \&\& {(A_n,a_n)} \\
	\&\&\&\&\& {(P_1 \times_{A_1} \dots \times_{A_{n-1}}P_n,(p_1,\dots,p_n))} \\
	\\
	{(B_1,b_1)} \&\&\&\&\& {(Q,q)} \&\&\&\&\& {(B_1,b_1)}
	\arrow[from=1-3, to=1-1]
	\arrow[from=1-3, to=1-5]
	\arrow[from=1-9, to=1-7]
	\arrow[from=1-9, to=1-11]
	\arrow["{f_0}"{description}, from=1-1, to=4-1]
	\arrow[from=4-6, to=4-1]
	\arrow[from=4-6, to=4-11]
	\arrow["{f_1}"{description}, from=1-11, to=4-11]
	\arrow[from=2-6, to=1-3]
	\arrow[from=2-6, to=1-9]
	\arrow["\alpha"{description}, from=2-6, to=4-6]
	\arrow["\lrcorner"{anchor=center, pos=0.125, rotate=135}, draw=none, from=2-6, to=1-6]
\end{tikzcd}
}

where pairs consists of objects of \cC and (global) elements of that object, e.g. $(A,a)$ is an object $A$ together with a map $a \colon 1 \to A$, and all the morphisms are asked to preserve the object, i.e. a morphism $f \colon (A,a) \to (B,b)$ should make the following diagram commute:
% https://q.uiver.app/?q=WzAsMyxbMCwwLCIxIl0sWzAsMSwiQSJdLFsxLDEsIkIiXSxbMCwxLCJhIiwxXSxbMSwyLCJmIiwxXSxbMCwyLCJiIiwxXV0=
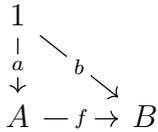
\begin{tikzcd}[ampersand replacement=\&]
	1 \\
	A \& B
	\arrow["a"{description}, from=1-1, to=2-1]
	\arrow["f"{description}, from=2-1, to=2-2]
	\arrow["b"{description}, from=1-1, to=2-2]
\end{tikzcd}

Furthermore the element $(p_1,\dots, p_n) \in P_1 \times \dots \times P_n$ is really $\langle p_1, \dots, p_n \rangle$:
% https://q.uiver.app/?q=WzAsNyxbMCwzLCIoUF8xLHBfMSkiXSxbMiwzLCIoQV8xLGFfMSkiXSxbMywzLCJcXGRvdHMiXSxbNiwzLCIoUF9uLHBfbikiXSxbNCwzLCIoQV97bi0xfSxhX3tuXzF9KSJdLFszLDIsIihQXzEgXFx0aW1lc197QV8xfSBcXGRvdHMgXFx0aW1lc197QV97bi0xfX1QX24sKHBfMSxcXGRvdHMscF9uKSkiXSxbMywwLCIxIl0sWzAsMV0sWzMsNF0sWzUsMF0sWzUsM10sWzUsMiwiIiwxLHsic3R5bGUiOnsibmFtZSI6ImNvcm5lciJ9fV0sWzYsMCwicF8xIiwxXSxbNiwzLCJwX24iLDFdLFs2LDUsIlxcbGFuZ2xlIHBfMSxcXGRvdHMscF9uXFxyYW5nbGUiLDEseyJzdHlsZSI6eyJib2R5Ijp7Im5hbWUiOiJkYXNoZWQifX19XV0=

\resizebox{\hsize}{!}{\begin{tikzcd}[ampersand replacement=\&]
	\&\&\& 1 \\
	\\
	\&\&\& {(P_1 \times_{A_1} \dots \times_{A_{n-1}}P_n,(p_1,\dots,p_n))} \\
	{(P_1,p_1)} \&\& {(A_1,a_1)} \& \dots \& {(A_{n-1},a_{n_1})} \&\& {(P_n,p_n)}
	\arrow[from=4-1, to=4-3]
	\arrow[from=4-7, to=4-5]
	\arrow[from=3-4, to=4-1]
	\arrow[from=3-4, to=4-7]
	\arrow["\lrcorner"{anchor=center, pos=0.125, rotate=-45}, draw=none, from=3-4, to=4-4]
	\arrow["{p_1}"{description}, from=1-4, to=4-1]
	\arrow["{p_n}"{description}, from=1-4, to=4-7]
	\arrow["{\langle p_1,\dots,p_n\rangle}"{description}, dashed, from=1-4, to=3-4]
\end{tikzcd}}

\begin{prop}
The forgetful functor $\Span_\ast(\cC) \to \Span(\cC)$ is a pushfibration.
\end{prop}
\begin{proof}
We are given the following data:
% https://q.uiver.app/?q=WzAsMTEsWzIsMCwiKFBfMSxwXzEpIl0sWzQsMCwiKEFfMSxhXzEpIl0sWzUsMCwiXFxkb3RzIl0sWzgsMCwiKFBfbixwX24pIl0sWzYsMCwiKEFfe24tMX0sYV97bl8xfSkiXSxbNSwxLCIoUF8xIFxcdGltZXNfe0FfMX0gXFxkb3RzIFxcdGltZXNfe0Ffe24tMX19UF9uLChwXzEsXFxkb3RzLHBfbikpIl0sWzEwLDAsIihBX24sYV9uKSJdLFsxMCwzLCIoQV9uLGFfbikiXSxbNSwzLCJRIl0sWzAsMywiKEFfMCxhXzApIl0sWzAsMCwiKEFfMCxhXzApIl0sWzAsMV0sWzMsNF0sWzUsMF0sWzUsM10sWzUsMiwiIiwxLHsic3R5bGUiOnsibmFtZSI6ImNvcm5lciJ9fV0sWzMsNl0sWzYsNywiIiwxLHsibGV2ZWwiOjIsInN0eWxlIjp7ImhlYWQiOnsibmFtZSI6Im5vbmUifX19XSxbOCw3XSxbNSw4LCJcXGFscGhhIiwxXSxbOCw5XSxbMCwxMF0sWzEwLDksIiIsMSx7ImxldmVsIjoyLCJzdHlsZSI6eyJoZWFkIjp7Im5hbWUiOiJub25lIn19fV1d

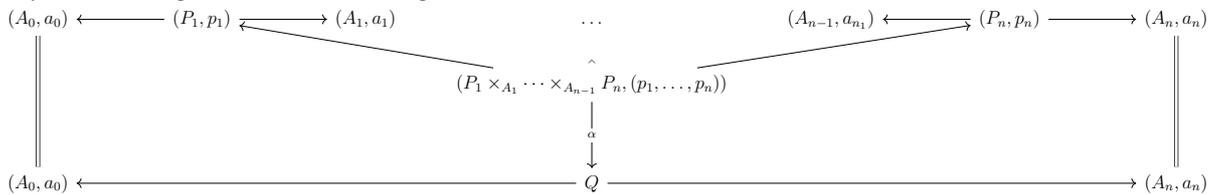
\resizebox{\hsize}{!}{\begin{tikzcd}[ampersand replacement=\&]
	{(A_0,a_0)} \&\& {(P_1,p_1)} \&\& {(A_1,a_1)} \& \dots \& {(A_{n-1},a_{n_1})} \&\& {(P_n,p_n)} \&\& {(A_n,a_n)} \\
	\&\&\&\&\& {(P_1 \times_{A_1} \dots \times_{A_{n-1}}P_n,(p_1,\dots,p_n))} \\
	\\
	{(A_0,a_0)} \&\&\&\&\& Q \&\&\&\&\& {(A_n,a_n)}
	\arrow[from=1-3, to=1-5]
	\arrow[from=1-9, to=1-7]
	\arrow[from=2-6, to=1-3]
	\arrow[from=2-6, to=1-9]
	\arrow["\lrcorner"{anchor=center, pos=0.125, rotate=135}, draw=none, from=2-6, to=1-6]
	\arrow[from=1-9, to=1-11]
	\arrow[Rightarrow, no head, from=1-11, to=4-11]
	\arrow[from=4-6, to=4-11]
	\arrow["\alpha"{description}, from=2-6, to=4-6]
	\arrow[from=4-6, to=4-1]
	\arrow[from=1-3, to=1-1]
	\arrow[Rightarrow, no head, from=1-1, to=4-1]
\end{tikzcd}}
and we want an element of $Q$ that is the image of $(p_1,\dots,p_n)$ by $\alpha$, so we just have to take $\alpha(p_1,\dots,p_n)$.
That it is opcartesian follows directly from the definition.
\end{proof}

Then, for any functor $F \colon \vdB \to \Span(\cC)$ we can form the pullback along the forget functor $\Span_\ast(\cC) \to \Span(\cC)$.
It is given by:
\begin{itemize}
\item objects, pairs $(B \in \vdB, b \colon 1 \to F(B))$
\item vertical morphisms, $f \colon B \to B'$ vertical morphisms in \vdB such that $F(f) \circ b = b'$
\item horizontal morphisms, pairs $(P, p) \colon (B,b) \to (B',b')$ of an horizontal morphism $P \colon B \to B'$ in \vdB and a global element of $F(P)$ that is send to $b$ and $b'$ by the left and right leg of the span respectively
\item cells, cells in \vdB such that their image preserves the points
\end{itemize}

Similarly to before, all the conditions descend to bicategories, and even strict 2-categories.

For any lax functor $F \colon \cB \to \Span(\cC)$ from a category to the 2-category $\Span(\cC)$, one get its B-G construction, or category of elements:
\begin{itemize}
\item objects are pairs $(B \in \cB, b \colon 1 \to F(B))$
\item morphisms are pairs $(f,\varphi)$ where $f \colon B \to B'$ is a morphism in \cB and $\varphi$ is an element of $F(f)$ sent by left and right leg to $b$ and $b'$
\item the identity is $(\id_B,\eta(b))$ where $\eta$ comes from the laxity of $F$ and $\eta(b) := \eta \circ b \colon 1 \to F(\id_B)$
\item composition is $(f' \circ f, \mu(\varphi,\varphi'))$
\end{itemize}

In the following we will focus on the case $\cC = \Set$ and we will write $\Span$ for $\Span(\Set)$
Let consider a functor of categories $P \colon \cE \to \cB$ and define the lax functor $\partial p \colon \cB \to \Span$ by:
\begin{itemize}
\item for $b\in \Ob{\cB}$, $\partial p(b) := p^{-1}(b) = \set{e \in \Ob{\cE}}{p(e) = b}$
\item for $f \colon b \to b'$, $\partial p(f) := p^{-1}(b) \xleftarrow{s} \set{\varphi \colon e \to e'}{p(\varphi)=f} \xrightarrow{t} p^{-1}(b')$ where $s$ and $t$ are the source and target
\item $\eta \colon \id_{\partial p} \Rightarrow \partial p(\id)$ with components % https://q.uiver.app/?q=WzAsNCxbNCwyLCJwXnstMX0oYikiXSxbMCwyLCJwXnstMX0oYikiXSxbMiwyLCJcXHtcXHZhcnBoaSBcXGNvbG9uIGUgXFx0byBlJyB8IHAoXFx2YXJwaGkpID0gXFxpZF9iXFx9Il0sWzIsMCwicF57LTF9KGIpIl0sWzIsMSwicyIsMV0sWzIsMCwidCIsMV0sWzMsMSwiIiwxLHsibGV2ZWwiOjIsInN0eWxlIjp7ImhlYWQiOnsibmFtZSI6Im5vbmUifX19XSxbMywwLCIiLDEseyJsZXZlbCI6Miwic3R5bGUiOnsiaGVhZCI6eyJuYW1lIjoibm9uZSJ9fX1dLFszLDIsIlxcZXRhX2IiLDFdXQ==
\[\begin{tikzcd}[ampersand replacement=\&]
	\&\& {p^{-1}(b)} \\
	\\
	{p^{-1}(b)} \&\& {\{\varphi \colon e \to e' | p(\varphi) = \id_b\}} \&\& {p^{-1}(b)}
	\arrow["s"{description}, from=3-3, to=3-1]
	\arrow["t"{description}, from=3-3, to=3-5]
	\arrow[Rightarrow, no head, from=1-3, to=3-1]
	\arrow[Rightarrow, no head, from=1-3, to=3-5]
	\arrow["{\eta_b}"{description}, from=1-3, to=3-3]
\end{tikzcd}\]
is given by $\eta_b(e) := \id_e$
\item for $f \colon b \to b'$ and $g \colon b' \to b''$,
% https://q.uiver.app/?q=WzAsNixbNCwwLCJwXnstMX0oYicpIl0sWzAsMCwicF57LTF9KGIpIl0sWzIsMCwiXFx7XFx2YXJwaGkgXFxjb2xvbiBlIFxcdG8gZScgfnx+IHAoXFx2YXJwaGkpID0gZlxcfSJdLFs2LDAsIlxce1xccHNpIFxcY29sb24gZScgXFx0byBlJycgfnx+IHAoXFxwc2kpID0gYicnXFx9Il0sWzgsMCwicF57LTF9KGInJykiXSxbNCwyLCJcXHtcXHhpIFxcY29sb24gZSBcXHRvIGUnJyB+fH4gcChcXHhpKSA9IGdcXGNpcmMgZlxcfSJdLFsyLDFdLFsyLDBdLFszLDBdLFszLDRdLFs1LDFdLFs1LDRdLFswLDUsIlxcbXVfe2YsZ30iLDFdXQ==

\resizebox{\hsize}{!}{
\begin{tikzcd}[ampersand replacement=\&]
	{p^{-1}(b)} \&\& {\{\varphi \colon e \to e' ~|~ p(\varphi) = f\}} \&\& {p^{-1}(f)\times_{p^{-1}(b')}p^{-1}(g)} \&\& {\{\psi \colon e' \to e'' ~|~ p(\psi) = g\}} \&\& {p^{-1}(b'')} \\
	\\
	\&\&\&\& {\{\xi \colon e \to e'' ~|~ p(\xi) = g\circ f\}}
	\arrow[from=1-3, to=1-1]
	\arrow[from=1-3, to=1-5]
	\arrow[from=1-7, to=1-5]
	\arrow[from=1-7, to=1-9]
	\arrow[from=3-5, to=1-1]
	\arrow[from=3-5, to=1-9]
	\arrow["{\mu_{f,g}}"{description}, from=1-5, to=3-5]
\end{tikzcd}
}
is given by $\mu_{f,g}(\varphi,\psi) = \psi \circ \varphi$
\end{itemize}

\begin{prop}
For $p \colon \cE \to \cB$, $\partial p \colon \cB \to \Span$ is a lax functor and $\partial$ is inverse to the B-G construction $\int$.
\end{prop}

So we get a B-G correspondence for $\Span$.
Like for $\Dist$, it extends to a 2-equivalence of bicategories and the Grothendieck correspondences for fibrations can be recovered by factorising through $\Cat \to \Span$ by sending a functor to its graph.

\chapter{The polycategorical Bénabou-Grothendieck correspondence}
\label{ch:bg-poly}

\section{2-polycategories}

\subsection{About 2-polycategories}

As stated above, to express these correspondences we need some theory of weak 2-polycategory, where by 2-polycategory we mean that the 1-cells can have multiple inputs and outputs but not the 2-cells.
In the following we only assume that there are weak 2-polycategories $\Dist$ and $\MVar$ and that lax functors and pseudofunctors behave in the expected way.
More generally we suspect that weak 2-polycategories and $\ast$-autonomous bicategories will be connected in a way such that the results of this thesis can be relaxed to this setting.
In particular, any compact closed bicategory -- as defined by Mike Stay in \cite{Stay2016} -- should be a $\ast$-autonomous bicategory, and by extension a weak 2-polycategory.
This would entail that $\Dist$ is weak 2-polycategory.

% Finally, it is worth noticing that everything considered in the correspondence between bifibrations and pseudofunctors into $\MVar$ is strict.
% Since strict 2-polycategories have been defined before this last part does not require any assumption.
% This can be found in \cite{Shulman2020chu}

\subsection{Distributors and multivariable adjunctions}

In this section we introduce the weak 2-polycategories $\Dist$ and $\MVar$.
We prove that a multivariable adjunction can be understood as a representable distributor.

\begin{definition}
  $\Dist$ is the weak 2-polycategory that has as objects categories, that has as polymaps $f : A_1,...,A_m \to B_1,...,B_n$ distributors $f : A_1 \times ... \times A_m \xslashedrightarrow[]{} B_1\times...\times B_n$ and that has as 2-cells natural transformations.
\end{definition}

\begin{definition}
  Given categories $A_1,...,A_m,B_1,...,B_n$, a \defin{$(m,n)$-adjunction} or \defin{multivariable adjunction} $(F_l)_{1\leq l \leq n} \dashv (G_k)_{1 \leq k \leq m} : A_1,...,A_m \to B_1,...,B_n$ consists of the following data:
  \begin{itemize}
  \item functors $F_l : \prod\limits_i A_i \times \prod\limits_{j \neq l} B_j^\op \to B_l$ for each $l$
  \item functors $G_k : \prod\limits_{i\neq k} A_i^{\op} \times \prod\limits_j B_j \to A_k$ for each $k$
  \item natural isomorphisms $B_l(F_l(a_1,...,a_m,b_1,...,b_n),b_l) \simeq A_k(a_k,G_k(a_1,...,a_m,b_1,...,b_n))$ for any $k,l$
    \end{itemize}
\end{definition}

\begin{example}
  A (1,1)-adjunction between $A,B$ is a pair of functor $F : A \to B$ and $G : B \to A$ such that $B(F(a),b) = A(a,G(b))$. It is just a usual adjunction.
\end{example}

\begin{example}
  Let $(\mathcal{C},\otimes,I)$ be a biclosed monoidal category.
  By definition $(A \otimes -)$ has a right adjoint $A \multimap -$ and $(- \otimes B)$ has a right adjoint $ - \multimapinv B$.
  We get three functors $\otimes : \mathcal{C} \times \mathcal{C} \to \mathcal{C}$, $\multimap : \mathcal{C}^{\op} \times \mathcal{C} \to \mathcal{C}$ and $\multimapinv : \mathcal{C}^{\op} \times \mathcal{C} \to \mathcal{C}$ such that \[\mathcal{C}(A \otimes B, C) \simeq \mathcal{C}(B, A \multimap C) \simeq \mathcal{C}(A, C \multimapinv B)\] This is an example of a $(2,1)$-adjunction $(\otimes) \dashv (\multimap,\multimapinv)$.
\end{example}

\begin{prop}
  A $(m,n)$-adjunction $(F_l)_{1\leq l \leq n} \dashv (G_k)_{1 \leq k \leq m} : A_1,...,A_m \to B_1,...,B_n$  is the same thing as a distributor $P : A_1 \times ... \times A_m \xslashedrightarrow{} B_1 \times ... \times B_n$ that is representable in each of its variables.
\end{prop}
\begin{proof}
  From any of the $F_l$ we can define a distributor $P_l : A_1 \times ... \times A_m \xslashedrightarrow[]{} B_1 \times ... \times B_n$ representable in $B_l$ by $P_l(-,-) := B_l(F_l(-),-)$.
  Similarly we can get distributors representable in $A_k$ from the functors $G_k$ by $P^k(-,-) := A_k(-,G_k(-))$.
  But all of these distributors are naturally isomorphic by definition of a multivariable adjunction.

  Conversely given a distributor $P : A_1 \times ... \times A_m \xslashedrightarrow{} B_1 \times ... \times B_n$, representability in the $A_k$ and $B_l$ produce functors $G_k$ with natural isomorphisms $P(-,-) \simeq  A_k(-,G_k(-))$ and functors $F_l$ with natural isomorphisms $P(-,-) \simeq B_l(F_l(-),-)$. 
\end{proof}

Multivariable adjunctions compose giving a subpolycategory of \Dist{} that we will denote \MVar{}.
To see that they compose let consider the case of composition a $(1,2)$-adjunction with a $(2,1)$-adjunction to get a $(2,2)$-adjunction.
Let $P \colon A \xslashedrightarrow B,C$ and $Q \colon C, A' \xslashedrightarrow B'$ be distributors represented by $P_A, P_B, P_C, Q_C, Q_{A'}, Q_{B'}$.
Then $Q \circ P \colon A,A' \xslashedrightarrow B,B'$ is representable in each variable via the functors:
\begin{itemize}
\item $(Q \circ P)_A := P_A \circ Q_C$
\item $(Q \circ P)_B := P_B \circ Q_C^\op$
\item $(Q \circ P)_{A'} := Q_{A'} \circ P_c^\op$
\item $Q \circ P_{B'} := Q_{B'} \circ P_C$
\end{itemize}
where we have omitted the symmetry isomorphisms used to rearrange the inputs.
It can be checked that these functors represent $Q\circ P$.

\section{Polycategorical B\'{e}nabou-Grothendieck correspondence}
\label{ch:poly-Groth}

\subsection{Polycategorical Grothendieck correspondences}

We want to extend the categorical B\'{e}nabou-Grothendieck correspondences to polycategories as follows:
\begin{itemize}
\item Poly-refinement systems $\mathcal{E} \to \mathcal{B}$ $\longleftrightarrow$ lax normal functors $\mathcal{B}^\op \to \Dist$
\item Bifibrations $\mathcal{E} \to \mathcal{B}$ $\longleftrightarrow$ Pseudofunctors $\mathcal{B}^\op \to \MVar$ 
\end{itemize}
where $\Dist$ is the weak 2-polycategory of sets and multivariable distributors and $\MVar$ is the weak 2-polycategory of sets and multivariable adjunctions.
Like in the categorical case, $\MVar$ is a sub-2-polycategory of $\Dist$ consisting of distributors that are representable in each of their variables.
It is worth noting that for us $(0,0)$-adjunctions will be sets, in contrast to the original definition in \cite{Shulman2020} where they are taken to be trivial.
Shulman discusses both possibilities, but chooses the latter to turn $\MVar$ into a strict 2-polycategory, whereas the former fits more naturally in our framework at the price of $\MVar$ being a weak 2-polycategory.

\subsection{Fibres of a poly-refinement system and distributors between them}

In the following we fix a poly-refinement system $p : \mathcal{E} \to \mathcal{B}$.
We define a lax normal functor $\partial p : \mathcal{B}^\op \to \Dist$ by considering the fibres of $p$ like in the categorical case.
We will use the convention that for any $\Gamma = A_1,...,A_n$,  $p^{-1}(\Gamma) := p^{-1}(A_1) \times ... \times p^{-1}(A_m)$.
$\partial p$ assigns to each object its fibre $\partial p(B) := p^{-1}(B)$ which is a category.
The objects of $p(B)$ consist of objects in \cE whose image by $p$ is $B$.
Its morphisms are unary polymaps in \cE whose image by $p$ is $\id_B$.
To a polymap $f : \Delta \to \Gamma$ in $\mathcal{B}^\op$, that we will equivalently consider as a polymap $f : \Gamma \to \Delta$ in $\mathcal{B}$, is assigned a distributor between the fibres $\partial p(f) : p^{-1}(\Delta) \times p^{-1}(\Gamma)^\op \to \Set$.
This distributor consists of the set of polymaps lying over $f$ acted on by pre- and post-composition.

More precisely, given lists of objects in the fibres $\Pi = (R_1,...,R_m) \refs \Gamma = (A_1,...,A_m), \Sigma = (S_1,...,S_n) \refs \Delta = (B_1,...,B_n)$ we define the action of the distributor $\partial p(f)$ on these objects by $\partial p(f)(\Sigma, \Pi) := \set{\varphi : \Pi \to \Sigma}{p(\varphi) = f}$.
And given lists of polymaps in the fibre $\vect{\psi} = (\psi_i : R_i' \seq{\id_{A_i}} R_i)_{1 \leq i \leq m}$ and $\vect{\xi} = (\xi_j : S_j \seq{\id_{B_j}} S_j')_{1 \leq j \leq n}$ we get $\partial p(f)(\vect{\xi},\vect{\psi}) := \vect{\xi} \circ - \circ \vect{\psi}$.

% \begin{align*}
%   \partial p :\ \mathcal{B}^\op &\to \Dist\\
%   B &\mapsto p^{-1}(B):= \{ R \in \mathcal{E}\ |\ p(R) = B \}\\
%   f : \Gamma \to \Delta &\mapsto \begin{aligned}
%     \partial p(f) : p^{-1}(\Gamma)^{\op} \times p^{-1}(\Delta) &\to \Set\\
%     (\Pi, \Sigma) &\mapsto \{ \varphi : \Pi \to \Sigma \ |\ p(\varphi) = f \}\\
%     ((\psi_i)_{1\leq i \leq m}, (\xi_j)_{1 \leq j \leq n}) &\mapsto \xi_1 \circ ... \circ \xi_n \circ - \circ \psi_1 \circ ... \circ \psi_m
%     \end{aligned}
% \end{align*}

% Here $\partial p(f)$ define an action of the family of morphisms $\psi_i : R_i' \to R_i$, $\xi_j : S_j \to S_j'$ by precomposition and postcomposition.
% The order in which we do these composition does not matter due to the interchange and associativity laws.
% Graphically

This can be represented graphically.

\begin{center}\scalebox{0.7}{\input{ppco.tikz}}\end{center}

It can be noted that the polymaps in the fibre have one-object domain and codomain. This is because all the polymaps in the fibre lies over the identity polymap in the base.

This is summarized in the following definition.

\begin{definition}
  For a poly-refinement system $p : \mathcal{E} \to \mathcal{B}$ we define the lax normal functor $\partial p : \mathcal{B}^\op \to \Dist$ by:
  \begin{itemize}
  \item For any $B \in \mathcal{B}$, $\partial p(B) := \set{S \in \mathcal{E}}{p(S)=B}$
  \item For any $f : \Gamma \to \Delta$ in $\mathcal{B}$, $\partial p(f) : p^{-1}(\Delta) \xslashedrightarrow{} p^{-1}(\Gamma)$ defined by:
    \begin{itemize}
    \item For any $\Pi \refs \Gamma, \Sigma \refs \Delta$, $\partial p(f)(\Sigma, \Pi) := \set{\varphi : \Pi \to \Sigma}{p(\varphi)=f}$
    \item For any $\vect{\psi} = (\psi_i : R_i' \seq{\id_{A_i}} R_i)_{1 \leq i \leq m}$ and $\vect{\xi} = (\xi_j : S_j \seq{\id_{B_j}} S_j')_{1 \leq j \leq n}$, $\partial p(f)(\vect{\xi},\vect{\psi})= \vect{\xi} \circ - \circ \vect{\psi}$
    \end{itemize}
  \end{itemize}
\end{definition}

The proof that this defines a lax normal functor is similar to the categorical one with some extra bookkeeping because of the presence of contexts of inputs and outputs.
% For example, given polymaps $f : \Gamma \to \Delta_1, A, \Delta_2$ and $g : \Gamma_1', A, \Gamma_2' \to \Delta'$ we define a natural transformation $\partial p (g) \circ_{\partial p(A)} \partial p (f) \Rightarrow \partial p (g \circ_A f)$ by the following family of multimaps:
% \begin{align*}
%   (\partial p (g) \circ_{\partial p(A)} \partial p (f))(\Sigma_1,\Sigma', \Sigma_2,\Pi_1',\Pi,\Pi_2') &= \int^R \partial p(g)(\Sigma',\Pi_1',R,\Pi_2') \times \partial p(f)(\Sigma_1,R,\Sigma_2, \Pi)\\
%                                                                                                       &= \int^R\{ \varphi' : \Pi_1', R, \Pi_2' \to \Sigma'\ |\ p(\varphi') = g \} \times \{ \varphi : \Pi \to \Sigma_1, R, \Sigma_2\ |\ p(\varphi) = f \}\\
%                                                                                                       &\to \{ \varphi' : \Pi_1', R, \Pi_2' \to \Sigma'\ |\ p(\varphi') = g \} \times \{ \varphi : \Pi \to \Sigma_1, R, \Sigma_2\ |\ p(\varphi) = f \}\\
%   &\to \{ \varphi'' : \Pi_1', \Pi, \Pi_2' \to \Sigma_1, \Sigma', \Sigma_2' \ |\ p(\varphi'') = g \circ_A f \}
% \end{align*}

Now suppose that $p$ is a bifibration and consider a polymap $f : \Gamma_1, A, \Gamma_2 \to \Delta$.
We define a functor $\pullop{f}_A : p^{-1}(\Delta) \times p^{-1}(\Gamma_1)^\op \times p^{-1}(\Gamma_2)^\op \to p^{-1}(A)$ by sending any $\Sigma \refs \Delta, \Pi_i \refs \Gamma_i$ to $\pull{f}{\Pi_1}{\Pi_2}{\Sigma}$.
From lists of polymaps $\vect{\psi_1}, \vect{\psi_2}$ and $\vect{\xi}$ we get a polymap $\pull{f}{\Pi_1}{\Pi_2}{\Sigma} \to \pull{f}{\Pi_1'}{\Pi_2'}{\Sigma'}$ by using the factorisation property of $\pull{f}{\Pi_1'}{\Pi_2'}{\Sigma'}$.
It is represented in Figure \ref{fig:pullpull} where the two big blank boxes are the in-cartesian polymaps associated to the pullbacks.

\begin{figure}
  \centering
  \scalebox{0.6}{\input{pullfunct.tikz}}
  \caption{Polymap $\pull{f}{\Pi_1}{\Pi_2}{\Sigma} \to \pull{f}{\Pi_1'}{\Pi_2'}{\Sigma'}$}
  \label{fig:pullpull}
\end{figure}
% \begin{center}\scalebox{0.6}{\tikzfig{pullfunct}}\end{center}

By the universal property of the pullback we can link $\partial p(f)$ and $\pullop{f}_A$ in the following way.
\begin{align*}
  \partial p(f)(\Sigma,\Pi_1,-,\Pi_2 )&= \set{\varphi : \Pi_1, -, \Pi_2 \to \Sigma}{p(\varphi) = f}\\
                                      &= \set{\psi : - \to \pull{f}{\Pi_1}{\Pi_2}{\Sigma}}{p(\psi) = \id_A}\\
                                      &= Hom_{p^{-1}(A)}(-,\pullop{f}_A(\Sigma,\Pi_1,\Pi_2))
\end{align*}

This makes $\partial p(f)$ representable in $A$.
Since by definition of a bifibration we get such a pull-functor for each of the inputs of $f$ and some similar push-functors for the outputs this makes $\partial p(f)$ a multivariable adjunction.
Since it is true for any polymap $f$ in $\mathcal{B}^\op$ we get that $\partial p$ factors through $\MVar$.
Finally the fact that cartesian polymaps compose makes $\partial p : \mathcal{B}^\op \to \MVar$ a pseudofunctor.

% Like in the categorical case, for $\partial p(g)$ to be representable in $p^{-1}(A)$ we need to define a functor\\ $p_\ast^A(g) : p^{-1}(\Gamma_1')^{\op} \times p^{-1}(\Gamma_2')^{\op} \times p^{-1}(\Delta') \to p^{-1}(A)$ establishing a correspondence:
% \AXC{$\Pi_1', R ,\Pi_2' \xRightarrow[g]{\psi} \Sigma'$}
% \UIC{$R \xRightarrow[\id_A]{\tilde{\psi}} p_\ast^A(g)(\Pi_1',\Pi_2',\Sigma')$}
% \DP
% and such that $\psi = p_\ast^A(g)(\id_R)\circ \tilde{\psi}$.
% This is the property of a weak in-cartesian polymap.
% So if we want $\partial p(g)$ to be representable in all its inputs we get a prefibration.
% Similarly if we want representability in all variables we get a prebifibration.
% Finally asking for a pseudofunctor and not just a lax normal functor is the same as asking that weak in-cartesian/opcartesian polymaps compose, so for a bifibration.
% Having all the representability condition for the distributor is the same as asking for a multivariable adjunction.
% So we get that any bifibration gives rise to a pseudofunctor in $\MVar$.

\subsection{Polycategorical Grothendieck-B\'{e}nabou construction}

Conversely, given a lax normal functor $F : \mathcal{B}^\op \to \Dist$ we construct its polycategory of elements $\int F$.

\begin{definition}
  The \defin{polycategory of elements} $\int F$ has:
  \begin{itemize}
  \item for objects, pairs $(A,R)$ with $A \in \mathcal{B}$ and $R \in F(A)$
  \item for polymaps $(f,\varphi) : (\Gamma, \Pi) \to (\Delta, \Sigma)$, pairs of a polymap $f : \Gamma \to \Delta$ in $\mathcal{B}$ and an element $\varphi \in F(f)(\Sigma,\Pi)$
  \item for identities $(\id_A,\id_R)$
  \item for composition $(g, \psi) \circ_{(A,R)} (f,\varphi) = (g \circ_A f, \mu(\widetilde{(\varphi,\psi)}))$ where:
    \begin{itemize}
    \item $\widetilde{(\varphi,\psi)} \in (F(g) \circ_{F(A)} F(f))(\Sigma_1, \Sigma', \Sigma_2, \Pi_1',\Pi,\Pi_2')$ is the canonical element induced by the elements $\varphi \in F(f)(\Sigma_1, R, \Sigma_2, \Pi)$ and $\psi \in F(g)(\Sigma', \Pi_1', R, \Pi_2')$
    \item $\mu : F(g) \circ_{F(A)} F(f) \Rightarrow F(g \circ_A f)$ is the natural transformation giving lax functoriality of $F$
    \end{itemize}
  \end{itemize}
\end{definition}

The fact that this is a polycategory follows from the coherence laws of $F$.
Furthermore it can be proven that these constructions are inverse to each other using the same arguments as for the categorical constructions.

\section{Frobenius pseudomonoids and Classical Linear Logic}

Like in Section \ref{sect:frob} there are different ways to define a Frobenius pseudomonoid.
The most convenient in our case will be to think of those as (the image of) a pseudofunctor out of $\one$.

\begin{definition}
  A \defin{Frobenius pseudomonoid} in a 2-polycategory $\mathcal{C}$ is a pseudofunctor $F : \one \to \mathcal{C}$.
\end{definition}

Using the polycategorical Grothendieck correspondence we recover the result recently announced by Shulman that Frobenius pseudomonoids in $\MVar$ are equivalent to $\ast$-autonomous categories.

\begin{theorem}[Shulman \cite{Shulman2019}]
  There is a correspondence between Frobenius pseudomonoid and $\ast$-autonomous categories.
\end{theorem}
\begin{proof}
  Using the polycategorical Grothendieck correspondence, pseudofunctors $\one \to \MVar$ correspond to bifibrations $p : \mathcal{E} \to \one$. Then using theorem \ref{th:bif1} these correspond to representable $\ast$-polycategories.
\end{proof}

\begin{remark}
  Given a Frobenius monoid $(A, \overline{(m,n)_A})$ in $\mathcal{B}$ and a lax normal functor $F : \mathcal{B} \to \Dist$.
  The polyfiber of $A$ relatively to the functor $\int F$ such as defined in section \ref{sect:frob} is given by the image of $A$ (and the polymaps $\overline{(m,n)_A}$) by $F$.
  If $F$ is pseudo on these polymaps this forms a Frobenius pseudomonoid in $\Dist$.
  When the images of these polymaps are representable in all their variables this factors through $\MVar$ giving a $\ast$-representable polycategory.
  This is the another way of understanding the result \ref{cor:frob}. 
\end{remark}

\part{Conclusion}

\chapter*{Conclusion}
\label{ch:conc}

\section*{Conclusion and Further work}

In this thesis we developed the theory of bifibred polycategories and provided examples.
We started by considering the concept of representability.
Representability is a tool used to express some categorical structures as universal properties by changing the shape of the morphisms.
We considered three cases: multicategories, where the inputs of a morphism is a finite list of objects, polycategories where both the inputs and outputs of a morphisms are finite list of objects and virtual double categories where the morphisms, called cells, go from a list of horizontal maps to another horizontal map along some vertical maps.
There is a systematic way of producing such category-like structures called generalised multicategories.
The interested reader can find the theory of generalised multicategories in Leinster in \cite{Leinster2004} and \cite{CruttwellShulman2009}.
The main idea is to consider morphisms with shape $T(A) \to A$ for $T$ a monad on some category.
Categories are then given by taking $T$ to be the identity monad on \Set{}, multicategories by taking $T$ to be the list monad or free monoid monad on \Set{} and virtual double categories by taking $T$ to be the free-category monad on graphs \footnote{actually, Leinster use the term fc-multicategory instead of virtual double category for this reason}.
In fact, this can be abstracted further by letting $T$ be a monad on a virtual double category, which is what is done in \cite{CruttwellShulman2009}.
Then the generalised multicategories corresponding to a monad $T$ on a virtual double category \vdC, arrange themselves into a virtual double category.
This gives a formal way of dealing with generalised multicategories, but also functors between them.
This also provides a huge pool of examples: symmetric multicategories, cartesian multicategories, topological spaces, monads on \Set{} and functors are all generalised multicategories.
Representability with regards to the output can then be defined in this generalising setting.
For the examples above one gets, symmetric monoidal categories, cartesian categories, ordered compact Hausdorff spaces, categories with small products and fibrations.
Notice however that polycategories are not examples.
Indeed, for generalised multicategories morphisms are of the form $T(A) \to A$ where the monad $T$ where $T$ encode the shape of the input.
However there is always one output.
Polycategories have polymaps with multiple outputs.
To generalise them one would need for morphisms with shape $T(A) \to T'(A)$ where both the inputs and the outputs are parametrised by monads.
A systematic treatment of generalised polycategories is still lacking, although some basis for the theory has been proposed in first in \cite{Koslowski2005} and then in \cite{Garner2008}.
It would be interesting to develop more thoroughly the theory of generalised polycategories, in particular to find interesting examples and to consider representability in this setting.

In addition to spelling out representability for multicategories and polycategories, we have connected them to monoidal categories and $\ast$-autonomous categories, which are categorical models of IMLL and MLL respectively.
In fact, we have suggested in the introduction a more direct interpretation of the sequent calculi in multicategories and polycategories.
The idea being that the shape of the morphisms in multicategories and polycategories correspond to that of the sequents of IMLL and MLL, so they encode precisely the structure of the sequent calculi.
Then the connectives and their rules are interpreted by objects and maps having universal properties.
This idea can be extend to sequent calculi for other logics by considering generalised multicategories or generalised polycategories that encode the structure of the logic considered.
In \cite{HylandTasson2020}, the authors have constructed a 2-monad on \Cat{} that they call the linear-non-linear substitution monad.
The generalised multicategories for this monad should correspond to models of the intuitionistic linear logic, or more precisely of the linear-non-linear sequent calculus presented in \cite{Benton1994}.
Their motivation is to use this to study models of Differential Linear Logic (for more on the theory of DiLL and their models see \cite{Ehrhard2018,Bluteetal2020}
In parallel, in this article \cite{Shulman2021} LNL-polycategories are introduced.
Representability is introduced and it is shown that it captures all the connectives of linear logic.
When a systematic treatment of generalised polycategories is worked out, it should provide an interesting example.
Furthermore, extending the case study of Banach spaces to LNL-polycategories could help understand better LL and DiLL, for example by studying models coming from functional analysis such as those considered in Kerjean's work, see \cite{KerjeanTasson2015,DabrowskiKerjean2017,Kerjean2018,KerjeanLemay2019}
.

In the second part of this thesis, we introduced bifibrations of polycategories and pushfibrations of virtual double categories.
We show how to think of fibrationality as a parametrised version of representability.
It should be possible to extend pushfibrations to generalised multicategories where an element of $t \in T(A)$ is pushed along a morphism $F(t) \to s$.
However, extending fibrations and bifibrations of multicategories to generalised ones would be more tricky.
The issue is that when defining a fibration of multicategories, one has to specify an object from the domain in which to pull.
It is not clear how this idea of picking an element in a list could be generalised to other monads.
The same limitation applies to defining any fibrational property in the case of generalised polycategories.
The investigation of fibrations of generalised multicategories or polycategories is left to further work.

In the context of virtual double categories, pushfibrations only have pushforward of cells along identity maps.
This was enough for the applications that we had in mind, namely describing the B\'{e}nabou-Grothendieck construction as a pullback.
However, we could ask for more fibrational structures.
First, by considering pushing and pulling objects along vertical or horizontal maps: these should be a generalisation of 2-fibrations, introduced in \cite{Hermida1999} (although with one condition missing) and studied in \cite{Bakovic2011,Buckley2014}.
Second, one could also ask for the possibility of pulling along a cell in some specific horizontal morphism.
Just as a pullback in a multicategory is a parametrised version of an internal hom-object, a pullback along a cell in a virtual double category should be a parametrised version of an internal hom-morphism, i.e. a horizontal morphism that would internalise cells.
We believe that there should be a notion of fibration of virtual double categories with enough structure to do formal category theory fibrationally.
These could provide parametrised models for the formal logic studied recently in \cite{LicataNew2022}.

Finally, in the last part we studied the B\'{e}nabou-Grothendieck correspondence and extended it to polycategories.
For the polycategorical one we assumed the existence of weak 2-polycategories.
We think this is a subject worth investigating.
In addition to letting us put the B-G correspondence on firm ground, weak 2-polycategories would be interesting as models of MLL.
Indeed, the usual categorical semantics for logic provides models that are quotiented by $\beta\eta$-equivalence.
In this context, the computational aspect of proof theory and type theory is lost.
To address this issue, some research has been done in providing bicategorical models of logic where $\beta$-reduction and $\eta$-expension correspond to 2-morphisms.
We believe that weak 2-polycategories would provide a framework for extending such studies to classical MLL.
An idea that we would like to explore in that direction is the following.
Consider \Gph, the category of graphs and graphs morphisms.
The forgetful functor $\Cat \to \Gph$ has a left adjoint that builds the free-category on a graph.
This defines a monad on $\Gph$, the free-category monad.
Algebras for this monad correspond to categories.
A similar construction works for $\Cat$-$\Gph$ the category of $\Cat$-enriched graphs and their morphisms.
It as a free-2-category monad whose algebra corresponds to (strict) 2-categories.
Now there is a 2-category \Cat{}-\Gph{} and the free-2-category monad extends to a 2-monad on it.
Algebras for this 2-monad are still strict 2-categories, while pseudoalgebras are bicategories.
Now, we can consider a category \MultiGph{} of multigraphs.
It is possible to define the free-multicategory on a multigraph.
So we could consider the monad it induces on \MultiGph, enriched the setting over $\Cat$ and extend the monad to a bicategorical one and then consider its pseudoalgebras to get a notion of weak 2-multicategory.
There it is interesting to notice that a graph/multigraph is a span $C_0 \leftarrow C_1 \rightarrow C_0$/$\cL(C_0) \leftarrow C_1 \rightarrow C_0$ with $\cL$ the list monad on set.
The free-category and free-multicategory construction on \Gph{} and \MultiGph{} can be obtained from the free monoid construction on the vdc of $T$-spans.
So when this free monoid construction exists, it should be possible to adapt the strategy above to define weak generalised 2-multicategories.
The case of weak 2-polycategories could be done in an ad hoc way by describing the free polycategory construction on a polygraph or by considering the abstract setting of generalised polycategories.

Another line of research would be to use this work on concrete models for developing tools in applied category theory.
For example the operational reading that we gave in the example of Banach spaces can be used in other cases.
One can consider a polycategory \cP whose objects model systems and whoses polymaps model processes, for example by considering the underlying polycategory of a compact closed category and using the literature on ACT.
In such a polycategory, nullary polymaps $\cdot \to A$ are understood as states or values of $A$, co-nullary polymaps $A \to \cdot$ as measurements or tests on $A$ and nullary-co-nullary polymaps $\cdot \to \cdot$ as scalars or truth values.
Furthermore, a specific subset of scalars is fixed that correspond to accepted scalars.
Let $\cM \subseteq \cP(\cdot,\cdot)$ be the subset of accepted scalars.
Then, one can built a polycategory where an object is an object $A$ in \cP equipped with a choice of substates of $A$, $c_A \subseteq \cP(\cdot, A)$.
This can be regarded as considering only states with a certain property.
To each such choice corresponds a set $c_A^\ast := \set{\varphi \colon A \to \cdot}{\forall a \in A,\ \varphi(A) \in \cM}$ that tests for this property.
Then the polymaps $f\colon A_1,\dots,A_m \to B_1, \dots, B_n$ are asked to respect the properties, meaning that for all $a_i \in c_{A_i}$ and all $\varphi_j \in c_{B_j}^\ast$ we should have that \[(\varphi_1,\dots,\varphi_n)f(a_1,\dots,a_m) \in \cM\]
The forgetful functor - that only keeps the object and forgets the property - defines a bifibration.
It would be interesting to look at how to capture certain examples by adding some more structures/properties on this construction.
Banach spaces and contractive polymaps would be an example, where the choice of substates is defined by the unit ball of the norm.
The work by Kissinger and Uijlen on causal structure \cite{KissingerUijlen2017lics} is another.
In both cases, the category of systems and processes is compact closed while the category of substates is $\ast$-autonomous with $\otimes \neq \parr$.
Other interesting examples could be considered.
An application could be to use the connection between bifibrations and Hoare logic to adapt tools from Hoare logic to study logical properties of systems modeled by bifibrations of polycategories.

\backmatter

\nocite{*}

% Bibliography
\bibliographystyle{plainurl}
\bibliography{thesis}

\end{document}